%% file: ArXiv.tex
\documentclass[12pt,a4paper,twoside,leqno,oldfontcommands]{memoir}
\usepackage[T1]{fontenc}

\usepackage{amsmath,amsthm,amssymb,amsfonts,mathtools,mathrsfs}   
\usepackage{dsfont} 

\usepackage{tensor} 

\usepackage{enumerate}

\usepackage[dvipsnames]{xcolor}
\usepackage[hide links]{hyperref}

\usepackage{subfiles}

\usepackage{graphicx} \graphicspath{{Images/}} 
\usepackage{placeins}
\usepackage{tikz}

\usepackage{hyperref}
\usepackage{theoremref}
\usepackage{booktabs}

\usepackage{layout} 
\usepackage[lmargin=4cm, rmargin=2cm, headheight=15pt]{geometry} 
\OnehalfSpacing

\makechapterstyle{Jack}{%
\chapterstyle{demo2}

\setlength{\beforechapskip}{1.5\baselineskip}

}

\chapterstyle{Jack}

\usepackage{fancyhdr}
\renewcommand{\sectionmark}[1]{\markright{#1}}
\renewcommand{\subsectionmark}[1]{}

\newtheorem{thm}{Theorem}[section]
\newtheorem{cor}[thm]{Corollary} 
\newtheorem{lem}[thm]{Lemma} 
\newtheorem{prop}[thm]{Proposition} 
 
\newtheorem*{claim}{Claim}

\theoremstyle{definition} 
\newtheorem{defn}[thm]{Definition}

\theoremstyle{remark}
\newtheorem{remarkx}[thm]{Remark}
\newenvironment{remark}
  {\pushQED{\qed}\remarkx}
  {\popQED\endremarkx}

\numberwithin{equation}{chapter}
\numberwithin{thm}{chapter}

\newcommand{\R}{\mathbb{R}}
 
\newcommand{\Z}{\mathbb{Z}}
\newcommand{\N}{\mathbb{N}}

\newcommand{\Sph}{\mathbb{S}}

\newcommand{\abs}[1]{ \lvert #1 \rvert}

\newcommand{\diam}{\operatorname{diam}}
\newcommand{\dist}{\operatorname{dist}}
\newcommand{\supp}{\operatorname{supp}}
\renewcommand{\div}{\operatorname{div}}

\newcommand{\PV}{\mathrm{P.V.}}
\DeclareMathOperator*{\esssup}{ess\,sup}
\DeclareMathOperator*{\essinf}{ess\,inf}
\newcommand{\RE}{\operatorname{Re}}  
\newcommand{\dd}{\mathop{}\!d} 

\renewcommand{\H}{\mathcal H}

\def\XXint#1#2#3{{\setbox0=\hbox{$#1{#2#3}{\int}$}
\vcenter{\hbox{$#2#3$}}\kern-.5\wd0}}

\newenvironment{PDE}
	{ \left \{
	\begin{array}{r@{ \ }l @{\quad \: \;} l}
	}
	{
	\end{array} \right . 
}

\newcommand{\Anorm}[1]{\| #1 \|_{\mathscr{A}_s(\R^n)}}

\renewcommand{\leq}{\leqslant}
\renewcommand{\geq}{\geqslant}
\renewcommand{\epsilon}{\varepsilon}

\usepackage[
backend=biber,
style=alphabetic,
maxbibnames=99, 
maxcitenames=99,
]{biblatex}
\newcommand{\pa}{\partial}
\newcommand{\ol}{\overline}
\newcommand{\Om}{\Omega}
\newcommand{\Ga}{\Gamma}
\newcommand{\na}{\nabla}
\newcommand{\al}{\alpha}

\newcommand{\ga}{\gamma} 
\newcommand{\chf}{\mathds{1}}
\newcommand{\de}{\delta}
\newcommand{\ka}{\kappa}
\newcommand{\la}{\lambda}
\newcommand{\La}{\Lambda}
\newcommand{\ul}{\underline}
\newcommand{\sS}{\mathcal{S}}

\newcommand{\e}{\varepsilon}
\newcommand{\HH}{\mathcal{H}^{n-1}}

\newcommand{\sph}{\Sph}
\newcommand{\fourier}{\mathscr{F}}

\usepackage{lipsum}
\usepackage{newpxtext}
\usepackage[varbb]{newpxmath} 
\renewcommand\mathcal[1]{\text{\usefont{OMS}{cmsy}{m}{n}#1}}

\begin{document}

\frontmatter

\include{Frontmatter/Titlepage}


\include{Frontmatter/Acknowledgements}

\cleardoublepage
\setcounter{tocdepth}{2}
\tableofcontents

\mainmatter

\include{Frontmatter/Introduction}

\part{Stability results for nonlocal Serrin-type overdetermined problems} \label{Part1}

\include{Part1/Intro-overdetermined}

\include{Part1/Role-Antisymmetric-functions-thesis}

\include{Part1/Parallel-surface-stability-thesis}

\include{Part1/Nonlocal-Serrin-Stability-thesis}

\part{Nonlocal Harnack inequalities for antisymmetric functions} \label{Part2}

\include{Part2/Intro-antisym-Harnack} 

\include{Part2/Antisymmetric-Harnack-final}

\include{Part2/Harnack-Bochner-thesisV2}

\part{ Nonlocal geometric identities} \label{Part3}

\include{Part3/Intro-geometric-identities} 

\include{Part3/Simons-identity-thesis}

\include{Part3/Density-estimate}

\appendix

\backmatter

\printbibliography

\end{document}

%% file: Frontmatter/Titlepage.tex
\begin{titlingpage} 
\thispagestyle{titlingpage}

\centering
	\includegraphics[scale=0.4]{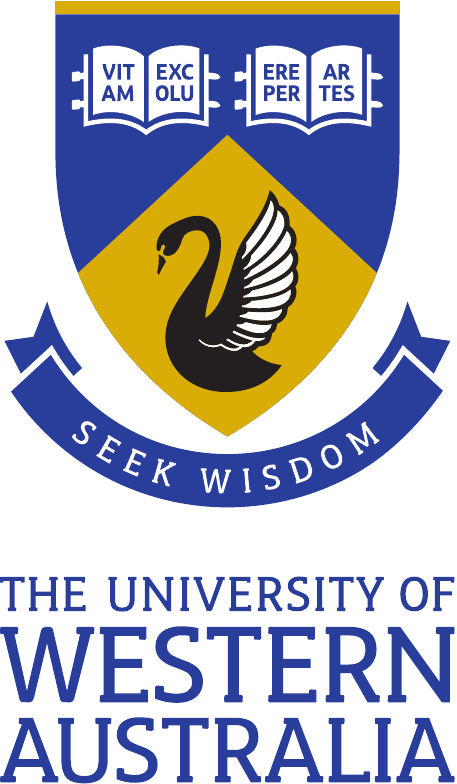}\par\vspace{1cm}
	{
	}
	{\huge\scshape Stability results for nonlocal Serrin-type problems, antisymmetric Harnack inequalities, and geometric estimates \par} 										
	\vspace{2cm}
	{\Large\scshape Jack Thompson}

	\vfill

	{\scshape  This thesis is presented for the degree of Doctor of Philosophy of The University of Western Australia \par
School of Physics, Mathematics and Computing \par 
Mathematics
 \par \par  \large 2024\par} 								

\end{titlingpage}

%% file: Frontmatter/Acknowledgements.tex


\chapter*{Acknowledgements\markright{Acknowledgements}} \addcontentsline{toc}{chapter}{Acknowledgements}

This research is supported by an Australian Government Research Training Program Scholarship. 

First and foremost, I am deeply indebted to my supervisors, Serena Dipierro and Enrico Valdinoci. Throughout my PhD, you've both taught me so much, not to mention given me so many opportunities, always going above and beyond to support my career. I would also like to thank to my close collaborator, Giorgio Poggesi, whose knowledge and insight helped me write several of the papers that are in this thesis. Moreover, I am extremely grateful to Xavier Ros-Oton for the support and hospitality he gave me when I visited Barcelona and for his help with my research. I would also like to thank Xavier Cabré, Marvin, and everyone else I have met during conferences and my travels. 

Next, I have to thank all my family and friends who have been with me over the years. To Giovanni Giacomin, when I look back on these past three and a half years, the first thing I will always think of is all the fun and silly memories you helped create, whether it's your bursts of chaotic energy (particularly in the presence of music that my grandparents listen to), questionable medical conclusions about common and curable diseases, or the solidarity we had from ignoring our responsibilities together. And, of course, I can't forget your epic pastas (and not-so-epic parmigiana). A special thanks also to Francesco De Pas, Tommaso Di Ubaldo, and David Perella for all the coffee trips, beers we've shared, and interesting (or, particularly in Francesco's case, comically confusing) conversations we've had. I’d also like to thank my colleagues at UWA: Caterina, David Pfefferlé, Edoardo, Jiwen, Jo{\~a}o, Kooper, Kurt, Mary, Riccardo, and Sandra. 

To my friends from Perth, thank you for all the experiences we've shared and the memories we've made. Thank you also to my friends from Brisbane---in particular, Ben, Kate, Kelsey, Harry, and Izzy---who I can always rely on to bring the best of times whenever I come home. 

To Adam, thank you for everything---all the adventures we've had, our long conversations about maths, and the whisky we've drunk. I also can't forget Monty who might not be the best at ignoring possums in the middle of the night but is the best when it comes to cheering me up.

To Mum and Dad, I will be forever grateful for the opportunities you've given me in my life and the support you provide me. You continue to be the basis on which I model myself and I love you both very much (you too Adam and Monty).

%% file: Frontmatter/Introduction.tex
\chapter*{Introduction\markright{Introduction}}  \addcontentsline{toc}{chapter}{Introduction}

In this thesis, we explore several related topics broadly regarding the symmetry and geometric properties of nonlocal partial differential equations (\textsc{PDE}). Nonlocal \textsc{PDE}, also known as integro-differential equations or fractional \textsc{PDE}, have become increasingly popular over the past few decades due to their ability to accurately model real-world phenomena. The prototypical example of a nonlocal operator is the \emph{fractional Laplacian}, defined according to the following singular integral: \begin{align*}
    (-\Delta)^s u(x) = c_{n,s} \PV \int_{\R^n} \frac{u(x)-u(y)}{\vert x - y \vert^{n+2s}} \dd y . 
\end{align*} Here \(s\in (0,1)\) is the fractional parameter, \(\PV\) stands for principle value, and \(c_{n,s}>0\) is an appropriate normalisation constant. Nonlocal \textsc{PDE} arise naturally in the study of stochastic processes and can be viewed as generalisations of partial differential equations (PDE)---as \(s\to 1^-\), \((-\Delta)^s \to -\Delta\) where \(\Delta\) is the usual Laplacian---which allow the possibility of long-range interactions or forces. Some examples where nonlocal \textsc{PDE} arise in the applied sciences are: 
crystal dislocation  \cite{MR3469920}; 
elasticity \cite{Kroner1967-elasticity,Eringen1972-nonlocal-elasticity};
ecology \cite{Viswanathan-etal1996-wandering-albatrosses,ReynoldsRhodes2009-levy-random-search,Humphries-etal2010-levy-brownian-marine-predators,MR3579567}; 
mathematical finance \cite{Nolan1999-fitting-data,Pham1997-optimal-stopping-american-option,Levendorskii2004-pricing-american-put,Merton1976-option-pricing-discontinuous}; 
the surface quasi-geostrophic equation in fluid mechanics\footnote{In fact, the most fundamental equation in fluid mechanics, the incompressible Navier-Stokes equation, are an example of a nonlocal \textsc{PDE} in disguise. This is due to the presence of the pressure term which depends on the velocity field in a nonlocal way.} \cite{Constantin2006-Euler-NavierStokes-turbulence}; 
image processing \cite{GilboaOsher2008-image-processing}; 
the physics of plasmas and flames \cite{MancinelliVergniVulpiani2003-front-propagation-anomalous-diffusion,MetzlerKlafter2000-fractional-dynamics};
magnetised plasmas \cite{Blazevski2013-shear-magnetic-fields}; 
and quantum mechanics \cite{Laskin2000-quantum-mechanics,Laskin2002-fractional-schrodinger}.

This thesis is split into three parts. In the first part, we study two overdetermined problems, namely Serrin's problem and the parallel surface problem, driven by the fractional Laplacian.

In Chapter~\ref{1XMNCzrH}, we use a Hopf-type lemma for antisymmetric super-solutions to the Dirichlet problem for the fractional Laplacian with zero-th order terms,
in combination with the method of moving planes, to prove symmetry for the \emph{semilinear fractional parallel surface problem}. That is, we prove that non-negative solutions to semilinear Dirichlet problems for the fractional Laplacian in a bounded open set $\Omega \subset \R^n$ must be radially symmetric if one of their level surfaces is parallel to the boundary of $\Omega$; in turn, $\Omega$ must be a ball.

Furthermore, we discuss maximum principles and the Harnack inequality for antisymmetric functions in the fractional setting and provide counter-examples to these theorems when only `local' assumptions are imposed on the solutions. 
The construction of these counter-examples relies on an approximation result that states that
`all antisymmetric functions are locally antisymmetric and \(s\)-harmonic up to a small error'.

In Chapter~\ref{Fw81drHU}, we analyze the stability of the parallel surface problem for semilinear equations driven by the fractional Laplacian. 
We prove a quantitative stability result that goes beyond the one previously obtained in \cite{MR4577340}. 

Moreover, we discuss in detail several techniques and challenges in obtaining the optimal exponent in this stability result. In particular, this includes an upper bound on the exponent via an explicit computation involving a family of ellipsoids. 
We also sharply investigate a technique that was proposed in \cite{MR3836150} to obtain the optimal stability exponent in the quantitative estimate for the nonlocal Alexandrov's soap bubble theorem,
obtaining accurate estimates to be compared with a new, explicit example.

In Chapter~\ref{yP1bfxWn}, we establish quantitative stability for the nonlocal Serrin overdetermined problem, via the method of the moving planes.
Interestingly, our stability estimate is even better than those obtained so far in the classical setting (i.e., for the classical Laplacian) via the method of the moving planes.

A crucial ingredient is the construction of a new antisymmetric barrier, which allows a unified treatment of the moving planes method. This strategy allows us to establish a new general quantitative nonlocal maximum principle for antisymmetric functions, leading to new quantitative nonlocal versions of both the Hopf lemma and the Serrin corner point lemma.

All these tools -- i.e., the new antisymmetric barrier, the general quantitative nonlocal maximum principle, and the quantitative nonlocal versions of both the Hopf lemma and the Serrin corner point lemma -- are of independent interest. 

In the second part, we study the Harnack inequality for solutions to nonlocal \textsc{PDE} which are antisymmetric, that is, the have an odd symmetry with respect to reflections across some hyperplane. This topic has a strong motivation coming from proving quantitative stability estimates for nonlocal overdetermined problems which we explain in more detail in the introduction to this part.  

In Chapter~\ref{RYsUIqSr}, we prove the Harnack inequality for antisymmetric \(s\)-harmonic functions, and more generally for
solutions of fractional equations with zero-th order terms, in a general domain. This may be used in conjunction with the method of moving planes to obtain quantitative stability results for symmetry and overdetermined problems for semilinear equations driven by the fractional Laplacian.

The proof is split into two parts: an interior Harnack inequality away from the plane of symmetry, and a boundary Harnack inequality close to the plane of symmetry. We prove these results by first
establishing the weak Harnack inequality for super-solutions and local boundedness for sub-solutions in both the interior and boundary case.

En passant, we also obtain a new mean value formula for antisymmetric $s$-harmonic functions.

In Chapter~\ref{OXXm8GiN}, we revisit the Harnack inequality for antisymmetric functions that has been established in Chapter~\ref{RYsUIqSr} for the fractional Laplacian and we extend it to more general nonlocal elliptic operators. 

The new approach to deal with these problems that we propose in this chapter leverages Bochner’s relation, allowing one to relate a one-dimensional Fourier transform of an odd function with a three-dimensional Fourier transform of a radial function.

In this way, Harnack inequalities for odd functions, which are essentially Harnack inequalities of boundary type, are reduced to interior Harnack inequalities.

In the third part, we prove several geometric identities and inequalities involving the fractional mean curvature. 

In Chapter~\ref{oyZ0zmYT}, inspired by a classical identity proved by James Simons, we establish a new geometric formula
in a nonlocal, possibly fractional, setting.

Our formula also recovers the classical case in the limit, thus providing an approach to Simons' work
that does not heavily rely on differential geometry. 

In Chapter~\ref{1KsVju8j}, we prove that measurable sets \(E\subset \R^n\) with locally finite perimeter and zero \(s\)-mean curvature satisfy the surface density estimates: \begin{align*}
    \operatorname{Per} (E; B_R(x)) \geq CR^{n-1}
\end{align*} for all \(R>0\), \(x\in \partial^\ast E\). The constant \(C\) depends only on \(n\) and \(s\), and remains bounded as \(s\to 1^-\). As an application, we prove that the fractional Sobolev inequality holds on the boundary of sets with zero \(s\)-mean curvature.

%% file: Part1/Intro-overdetermined.tex
\chapter{Introduction to Part I}  

The first part of this thesis is on the rigidity and stability of two related overdetermined problems, namely Serrin's problem and the parallel surface problem, in the nonlocal context. Overdetermined problems are boundary value problems for \textsc{PDE} where `too many' boundary conditions are specified. For general regions, such a problem is ill-posed, that is, there is no solution to the set of equations, so the objective of overdetermined problems is to identify regions for which a solution does exist. In applications, overdetermined problems often appear as a well-posed boundary value problem, modelling a given physical situation, accompanied by an additional boundary condition, often referred to as the overdetermined condition, which encodes the `optimal configuration' in the physical situation. As such, overdetermined problems have a close relationship with the calculus of variations and shape optimisation.

\section{Rigidity in local and nonlocal Serrin-type overdetermined problems} 
\subsection{The local problem}
Let \(\Omega\subset \R^n\) be a bounded domain\footnote{Here a domain is an open connected set.}. In the celebrated paper \cite{MR333220}, Serrin considered the following well-posed PDE
\begin{align}
    \begin{PDE}
-\Delta u &= 1, &\text{in } \Omega \\
u&=0, &\text{on } \partial \Omega
    \end{PDE} \label{6gDv6VCY}
\end{align} 
paired with the overdetermined condition\footnote{Note that if \(u\in C^2(\Omega) \cap C^1(\overline \Omega)\)  satisfies~\eqref{6gDv6VCY} and \(-\partial_\nu u= c\) on \(\partial \Omega\) for some \(c\in \R\) then integration by parts implies that \begin{align*}
    \vert \Omega \vert = - \int_\Omega \Delta u \dd x &= - \int_{\partial \Omega} \partial_\nu u \dd \mathcal H^{n-1} = c\vert \partial \Omega \vert,
\end{align*} so, necessarily, \(c= \frac{\vert \Omega \vert}{\vert \partial \Omega \vert} \).}
\begin{align}
    -\partial_\nu u = \text{const.} \qquad \text{on } \partial \Omega . \label{Gg6QbqFP}
\end{align} It is clear that if \(\Omega = B_R(x_0)\) for some \(x_0\in \R^n\) and \(R>0\) then~\eqref{6gDv6VCY}-\eqref{Gg6QbqFP} admit a solution, namely, \(u(x) = \frac1 {2n} \big(R^2 -\vert x-x_0\vert^2\big ) \). In \cite{MR333220}, Serrin addressed the reverse problem: \begin{align}
    \text{If there exists }u\in C^2(\Omega) \cap C^1(\overline \Omega) \text{ satisfying~\eqref{6gDv6VCY}-\eqref{Gg6QbqFP}, is } \Omega \text{ a ball?} \label{kLWIZTg3}
\end{align} Serrin proved that, assuming \(\Omega\) has \(C^2\) boundary, the answer to~\eqref{kLWIZTg3} is yes, that is, there exists a solution to~\eqref{6gDv6VCY}-\eqref{Gg6QbqFP} if and only if \(\Omega\) is a ball. In fact, more generally, Serrin proved in \cite{MR333220} that if \(f\in C^{0,1}_{\mathrm{loc}}(\R)\) and \(u\) satisfies \begin{align}
    \begin{PDE}
-\Delta u &= f(u), &\text{in } \Omega \\
u&=0, &\text{on } \partial \Omega\\
u&>0, &\text{in } \Omega
    \end{PDE}\label{E6TedKUX}
\end{align} and~\eqref{Gg6QbqFP} then \(\Omega\) is a ball. We also mention that not long after Serrin's paper, an alternative proof of the affirmative answer to~\eqref{kLWIZTg3} was given by Weinberger \cite{MR333221} that made use of a special auxiliary function (known as the \(P\)-function) and integral identities.  

Serrin's motivation for the question~\eqref{kLWIZTg3} came from a mathematical model for finding the shape of a pipe, containing a flowing fluid, which maximises the pipe's time to failure based solely on shape (i.e. independent of material, etc). Indeed, let the pipe be modelled by \(\Omega \times \R\) where \(\Omega\subset \R^2\) is the cross-section of the pipe and assume that the fluid is viscous and incompressible, and the flow of the fluid is laminar. In this physical situation, it is known that the so-called flow velocity function \(u\) depends only on \(x_1\) and \(x_2\) (not \(x_3\) nor time \(t\)) and \(u\) satisfies, up to normalisation constants,~\eqref{6gDv6VCY}. A non-uniform load on the boundary of the pipe will introduce a bending stress, so if the pipe is subject to a consistent non-uniform load then it could fail sooner than if the load were uniform. Since a non-uniform load induces a non-uniform tangential stress on the pipe (and vice-versa), it is reasonable to postulate that the pipe that maximises the time to failure should be the one on which the fluid exerts a constant tangential stress per unit area on the boundary of the pipe. Furthermore, the tangential stress per unit area on the pipe is given mathematically by \(-\partial_\nu u\) (again up to a multiplicative constant), so assuming a uniform load amounts to assumption~\eqref{Gg6QbqFP}. Thus, Serrin's result gives a mathematical justification that cylindrical pipes (i.e. pipes whose cross-section is a disk) maximise the pipe's time to failure. 

From a purely mathematical viewpoint, one of the reasons the paper~\cite{MR333220} has become so influential is due to the technique Serrin employed to prove his theorem, known today as the \emph{method of moving planes}. This technique is a refinement of a reflection principle conceived by Alexandrov in \cite{MR0150710} to prove the so-called \emph{soap bubble theorem}, which states that the only connected closed hypersurfaces embedded in a space form with everywhere constant mean curvature are spheres. The method of moving planes was also famously employed by Gidas, Ni, and Nirenberg to prove several symmetry results for semilinear elliptic equations, see \cite{MR634248,MR544879}. See \cite{MR3932952} for a survey on the many applications of the method of moving planes.

A problem closely related to Serrin's overdetermined problem is the so-called parallel surface problem. Suppose that \(G\) is an open bounded subset of~\(\R^n\) and  let~\(\Omega  = G+B_R\) for some~\(R>0\), where~\(A+B\) is the Minkowski sum of sets defined by \begin{align*}
A+B = \{ a+b \text{ s.t. } a\in A, b \in B\}. 
\end{align*}  Then, the parallel surface problem asks: if there exists a function \(u\) that satisfies the equation~\eqref{E6TedKUX} 
 as well as the overdetermined condition \begin{align}
u = \text{const.} \qquad \text{on } \partial G, \label{JVG7Q5lS}
\end{align} then is \(\Omega\) necessarily a ball? This problem was introduced in \cite{MR2629887} in the context of invariant isothermic surfaces of a nonlinear non-degenerate fast diffusion equation, see also \cite{MR3420522,MR2916825}, and gets its name since the overdetermined condition~\eqref{JVG7Q5lS} fixes the value of \(u\) on the (hyper)surface \(\partial G\) which is parallel to the boundary of \(\Omega\). Another motivation for considering this problem is it can be viewed as a `discrete' analogue of Serrin's problem. Indeed, suppose that \(\{\Gamma_k\}\) is a countably infinite family of hypersurfaces that are parallel to \(\partial\Omega\) with \(\operatorname{dist}(\Gamma_k,\partial\Omega)=\frac 1 k\) and let \(\{c_k\}\) be such that \(c_k>0\) and \(\lim_{k\to+\infty}kc_k<+\infty\). Assuming \(u\in C^1(\overline \Omega)\) if \(u\) satisfies~\eqref{E6TedKUX} as well as~\(u= c_k>0\) on each \(\Gamma_k\) then, for each \(x\in \partial\Omega\), \begin{align*}
    -\partial_\nu u(x) = \lim_{h\to 0^+} \frac{u(x-h\nu(x))}h = \lim_{k\to +\infty } k c_k = \text{const.},
\end{align*} so \(u\) satisfies~\eqref{Gg6QbqFP}. Hence, if \(u\) is prescribed on an infinite number of parallel surfaces then \(\Omega\) is a ball by Serrin's result; in the parallel surface problem, we only consider a single parallel surface, but this still turns out to be enough to conclude \(\Omega\) is a ball. Indeed, it was proven in \cite{MR2629887,MR3420522,MR2916825} that if \(u\in C^2(\Omega)\) satisfies~\eqref{E6TedKUX} and the overdetermined condition~\eqref{JVG7Q5lS} then \(\Omega\) is a ball.

\subsection{The nonlocal problem}

There have been many directions in which Serrin's problem has been generalised. We have already mentioned that the parallel surface problem may be viewed as a discrete version of Serrin's problem; furthermore,  some of the contexts in which Serrin's problem has been explored include: alternative/incomplete overdetermined conditions~\cite{MR1616562, MR2436831, MR3231971, MR4230553}, nonlinear elliptic equations~\cite{MR980297, MR2293958, MR2366129, MR2448319, MR2764863, MR3040677, MR3385189, AlessandriniGarofalo1989symmetry, reichel1996radial, MR1674355, MR3977217}, and on manifolds~\cite{MR3959271,MR3663321}. A natural generalisation of Serrin's problem is to replace the Laplacian with a nonlocal operator, in particular, the fractional Laplacian. Let \(s\in (0,1)\), and consider the Dirichlet problem \begin{align}
    \begin{PDE}
(-\Delta)^su &=f(u), &\text{in }\Omega \\
u&=0, &\text{in } \R^n \setminus \Omega \\
u&>0,&\text{in }\Omega.
    \end{PDE} \label{AIgLTPmL}
\end{align} Recall, that due to the nonlocal nature of the fractional Laplacian, one must prescribe `boundary data' in the complement of \(\Omega\) to obtain a well-posed boundary value problem.

To our knowledge, the first paper on an overdetermined problem for the fractional Laplacian was \cite{MR3395749} where the authors considered solutions to~\eqref{AIgLTPmL} with an overdetermined condition analogous~\eqref{Gg6QbqFP} in Serrin's problem. To understand their overdetermined condition, it is important to understand that the boundary regularity for equations involving the fractional Laplacian differs from the classical case. Indeed, the solution to \begin{align*}
    \begin{PDE}
(-\Delta)^s u &= 1, &\text{in } B_R \\
u&=0, &\text{in }\R^n \setminus B_R
    \end{PDE}
\end{align*} is given  by \begin{align*}
    u(x) = \gamma_{n,s} \big(R^2-\vert x \vert^2 \big )^s
\end{align*} for some \(\gamma_{n,s}>0\) and \(u\) is no better than \(C^s(\overline B_R)\) even though all the `given data' is smooth (recall, one obtains \(u\) is smooth up to the boundary for the corresponding situation in the classical case), see, for example~\cite{MR3168912}. As such,~\eqref{Gg6QbqFP} no longer makes sense in this context since, in general, we expect \(\vert \partial_\nu u \vert=+\infty\). Instead, one considers \begin{align*}
    \partial_\nu^s u (x) &:=-\lim_{h \to 0^+} \frac{u(x-h\nu(x))-u(x)}{h^s} \qquad \text{ for }x\in \partial\Omega.
\end{align*} This fractional analogue of the Neumann derivative also appears naturally in the context of the nonlocal Pohozaev identity, see \cite{MR3211861}. Then, the main result in \cite{MR3395749} is as follows. 

\begin{thm}[{\cite[Theorem 1.2]{MR3395749}}]
Let \(c \in \R\) and \(\Omega \subset  \R^n\) be a bounded open set with \(C^2\) boundary. Furthermore, let \(f\in C^{0,1}_{\mathrm{loc}}(\R)\) and assume that there exists a function \(u \in C^s(\overline \Omega)\) satisfying~\eqref{AIgLTPmL} and the overdetermined condition \begin{align*}
    \partial_\nu^s u = c \qquad \text{on }\partial \Omega. 
\end{align*} Then \(\Omega\) is a ball.
\end{thm}

An interesting remark here is that \cite[Theorem 1.2]{MR3395749} does not require \(\Omega\) to be connected unlike in the local case since the nonlocality of the fractional Laplacian allows it to `see' disconnected components. The proof of \cite[Theorem 1.2]{MR3395749} is similar in principle to that of Serrin's proof---the authors make use of the method of moving planes, and require a Hopf-type corner lemma,---however, several nontrivial adjustments must be made in the nonlocal situation, particularly to the maximum principle, to make the method of moving planes work. We explain this in more detail in the introduction to Part II. 

For the nonlocal analogue of the parallel surface problem, the setting is analogous to the classical situation, with the Laplacian replaced with the fractional Laplacian. Indeed, one now considers solutions to the fractional Dirichlet problem~\eqref{AIgLTPmL} where \(\Omega = G+B_R\) for some open bounded set \(G\) and enforces the same overdetermined condition as in the classical setting~\eqref{JVG7Q5lS}. The first result in this direction was in \cite{ciraolo2021symmetry} where the authors considered the particular case \(f=1\). Here the authors recover the same result as in the previous examples, that the only \(\Omega\) for which~\eqref{AIgLTPmL} (with \(f=1\)) and~\eqref{JVG7Q5lS} admit a solution is a ball. They also establish a stability result for this problem which we will discuss in the following section. In Chapter~\ref{yP1bfxWn}, we generalise the result of~\cite{ciraolo2021symmetry} to the case \(f\in C^{0,1}_{\mathrm{loc}}(\R)\).

\section{Stability in local and nonlocal Serrin-type overdetermined problems} 

\subsection{The local problem}

So far, we have only discussed the \emph{rigidity} of Serrin's problem and the parallel surface, that is, the validity of the statement ``if \(u\) satisfies~\eqref{E6TedKUX} or~\eqref{AIgLTPmL} and the respective overdetermined condition then \(\Omega\) is a ball". An interesting follow-up problem in the contemporary literature is regarding the \emph{stability} of overdetermined problems, that is, heuristically, if \(u\) satisfies~\eqref{E6TedKUX} or~\eqref{AIgLTPmL} and `almost' satisfies the overdetermined condition then is \(\Omega\) `almost' a ball? Of course, to make this statement rigorous, one needs to specify what is meant by `almost'. One strategy for proving stability estimates for overdetermined problems is to apply the method of moving planes, as in the proof of rigidity, but replace any `qualitative arguments` with `quantitative estimates'. For example, the method of moving planes relies heavily on the strong maximum principle which suffices to prove rigidity---for stability, one replaces the strong maximum principle with its quantitative analogue, the Harnack inequality. In this thesis, we will focus primarily on this strategy; however, it is important to mention that another common strategy for proving rigidity and stability results for overdetermined problems is via integral identities, see~\cite{payne1989duality, MR980297, MR2448319, cianchi2009overdetermined, MR3663321, MR3959271, MR4054869, MR4124125, CavallinaPoggesi2024}. Both strategies have their pros and cons. Indeed, for a given overdetermined problem, using integral identities requires one to first find such an identity. This can be, in principle, very challenging while the method of moving planes is more versatile in its application. However, the method of moving planes relies heavily on the linear structure of \(\R^n\), so it doesn't generalise well to manifolds while integral identities, provided they can be found, naturally generalise to manifolds. Moreover, the method of moving planes lends itself to stability estimates in terms of \(L^\infty\) and Hölder norms since the argument is pointwise in nature while using integral identities better lends itself to estimates in terms of \(L^p\) norms.

The first paper to address the stability of overdetermined problems was \cite{MR1729395} which proved a stability result for Serrin's overdetermined problem. Here the authors define \begin{align*}
    \rho (\Omega) = \inf\{ R- r \text{ s.t. } B_r(x) \subset \Omega \subset B_R(x) \text{ for some } x\in \Omega \} 
\end{align*} which measures how close \(\Omega\) is to a ball in an `$L^\infty$ sense`. The main result of \cite{MR1729395} is as follows.

\begin{thm}[{\cite[Theorem 1]{MR1729395}}] Let \(\Omega\subset \R^n\) be a bounded \(C^{2,\alpha}\) domain and \(f\in C_{\mathrm{loc}}^{0,1}(\R)\) with \(f(0)\geq 0\). If \(u\in C^2(\overline \Omega)\) satisfies~\eqref{E6TedKUX} with \(\vert \partial_\nu u \vert \geq d_0>0\) on \(\partial\Omega\) then there exists \(\varepsilon\in (0,1)\) such that \begin{align*}
    \| \partial_\nu u - c \|_{C^1(\partial\Omega)} < \varepsilon
\end{align*} implies that \begin{align*}
    \rho(\Omega) \leq C \vert \log \| \partial_\nu u - c \|_{C^1(\partial\Omega)} \vert^{-\frac 1 n }.
\end{align*} The constants \(\varepsilon\) and \(C\) depend only on the \(C^{2,\alpha}\) regularity of \(\Omega\), and on upper bounds for \(\diam \Omega\), \(\|u\|_{L^\infty(\Omega)}\), \(\|f\|_{C^{0,1}([0,\|u\|_{L^\infty(\Omega)}])}\), and \(d_0^{-1}\). 
\end{thm}

For the special case that \(u\) satisfies~\eqref{6gDv6VCY} instead of~\eqref{E6TedKUX} (i.e. the case \(f=1\)), \cite[Theorem 1]{MR1729395} has been strengthened by making use of integral estimates \`a la Weinberger's proof of Serrin's problem. Moreover, the logarithmic modulus of continuity appearing in \cite[Theorem 1]{MR1729395} was improved in~\cite{MR3522349} to a Hölder modulus of continuity under the assumption \(\Omega\) is convex. 

For the parallel surface problem, the stability was established in \cite{MR3481178}. In this paper, the following results were proven 

\begin{thm}[{\cite[Theorem 1.1]{MR3481178}}]
Let \(G\) be a bounded \(C^{2,\alpha}\) domain in \(\R^n\) and let \(\Omega=G+B_R\) for some \(R > 0\). Let \(u \in C^2(\Omega) \cap C(\overline \Omega)\) satisfy~\eqref{E6TedKUX} with \(-\partial_\nu u \geq d_0>0\) on \(\partial\Omega\). If \(R < \frac 12 d_0 \| u\|_{C^2(\overline \Omega)}^{-1}\) then there exists \(\varepsilon>0\) such that \begin{align*}
    [u]_{\partial G}:= \sup_{x,y\in \partial G} \frac{\vert u(x) - u(y) \vert}{\vert x - y \vert} < \varepsilon
\end{align*} implies \begin{align*}
    \rho(\Omega) \leq C [u]_{\partial G}.
\end{align*} The constants \(\varepsilon\) and \(C\) depend only on \(n\), \(R\), the \(C^{2,\alpha}\)-regularity of \(\partial G\), and on upper bounds for \(\diam G\), \(\|u\|_{L^\infty(\Omega)}\), and  \(\|f\|_{C^{0,1}([0,\|u\|_{L^\infty(\Omega)}])}\).
\end{thm}

Observe that \cite[Theorem 1]{MR1729395} imposes stricter conditions with respect to the convergence back to the corresponding rigidity problem compared to \cite[Theorem 1.1]{MR3481178} in the sense that \cite[Theorem 1]{MR1729395} requires \(\partial_\nu u \vert_{\partial\Omega}\) to be close to a constant in \(C^1\) while \cite[Theorem 1.1]{MR3481178} only requires \(u \vert_{\partial G}\) to be close a constant in \(C^1\). Moreover, the modulus of continuity in \cite[Theorem 1]{MR1729395} ( \(t\mapsto \vert \log t \vert^{-\frac1n}\)) decays much slower than the linear modulus of continuity present in \cite[Theorem 1.1]{MR3481178}. The reason for these differences is that \cite[Theorem 1.1]{MR3481178} is an interior problem while \cite[Theorem 1]{MR1729395} is a boundary problem; as such, the estimates required for  \cite[Theorem 1]{MR1729395} are more delicate than the ones required for \cite[Theorem 1.1]{MR3481178}. 

Furthermore, the result of \cite[Theorem 1.1]{MR3481178} is sharp in general since if one takes \(\Omega_\varepsilon\) to be the 1-parameter family of ellipsoids \( \{(1+\varepsilon)^{-2}x_1^2 +x_2^2 +\dots+x_n^2 <1\}\), \(G_\varepsilon\) to be such that \(\Omega_\varepsilon = G_\varepsilon + B_{1/2}\), and \(u_\varepsilon\) to be the corresponding solution to~\eqref{E6TedKUX} when \(f=1\) then a direct computation shows that $$ \lim_{\varepsilon \to 0^+} \frac{[u_\varepsilon]_{\partial G_\varepsilon}}{\rho(\Omega_\varepsilon)}=C(n)>0,$$ see Section~\ref{z1jzKMXt} in Chapter~\ref{Fw81drHU}. By a direct computation, for the same 1-parameter family if \(\nu_\varepsilon\) is the outward point unit normal to \(\partial\Omega_\varepsilon\) and\footnote{Note, the choice of evaluating \(\partial_{\nu_\varepsilon} u_\varepsilon\) at \((1+\varepsilon,0,\dots,0)\) in the definition of \(c\) is not significant since the symmetry present in the 1-parameter family forces \(\lim_{\varepsilon \to 0^+} \partial_{\nu_\varepsilon} u_\varepsilon(x)\) to be independent of \(x\).} \(c=\lim_{\varepsilon \to 0^+} \partial_{\nu_\varepsilon} u_\varepsilon (1+\varepsilon,0,\dots,0)\) then \begin{align*}
  \lim_{\varepsilon \to 0^+}  \frac{\| \partial_{\nu_\varepsilon} u_\varepsilon - c \|_{C^1(\partial \Omega_\varepsilon)}}{\rho(\Omega_\varepsilon)} = C(n)>0,
\end{align*} so this family does not prove that \cite[Theorem 1]{MR1729395} is sharp. In fact, to the author's knowledge, it is still an open problem to establish whether or not the estimate in \cite[Theorem 1]{MR1729395} is optimal.

\subsection{The nonlocal problem}

The subject of Part I of this thesis is establishing rigidity and stability estimates for the nonlocal Serrin's problem and parallel surface problem. Before our results, we were only aware of one paper in this direction, \cite{ciraolo2021symmetry}, which considered the stability of the nonlocal parallel surface problem in the particular case \(f=1\). Here the authors prove the following result. 

\begin{thm}[{\cite[Theorem 1.2]{ciraolo2021symmetry}}] \thlabel{pQsFqMH1}
Let \(G\) be an open and bounded set of \(\R^n\) with \(C^1\) boundary, and let \(\Omega = G + B_R\).
Furthermore, assume that the boundary of \(\Omega\) is \(C^2\) and \(u\in C^2(\Omega) \cap C(\R^n)\) satisfies~\eqref{AIgLTPmL} with \(f=1\).  Then \begin{align*}
    \rho(\Omega) \leq C [u]_{\partial G}^{\frac 1{s+2}} . 
\end{align*} The constant \(C\) depends only on \(n\), \(s\), \(R\), and \(\diam \Omega\). 
\end{thm}

In Chapter~\ref{tH4AETfY}, we extend \cite[Theorem 1.2]{ciraolo2021symmetry} to the general case \(f\in C^{0,1}_{\mathrm{loc}}(\R)\) and weaken some of the generality assumptions on the domains: 

\begin{thm}[{\thref{oAZAv7vy} in Chapter~\ref{tH4AETfY}}] \label{sge2H8Ql}
    Let \(G\) be an open bounded subset of~\( \R^n\) and \(\Omega = G+B_R\) for some \(R>0\). Furthermore, let \(\Omega\) and \(G\) have \(C^1\) boundary,
%
%
and let~\(f \in C^{0,1}_{\mathrm{loc}}(\R)\) be such that \(f(0)\geqslant 0\). Suppose that \(u\) satisfies~\eqref{AIgLTPmL} in the weak sense. Then \begin{align*}
\rho(\Omega) &\leqslant C [u]_{\partial G}^{\frac 1 {s+2}} 
\end{align*} The constant \(C\) depends only on\footnote{Recall \begin{align*}
    \|u\|_{L_s(\R^n)} = \int_{\R^n} \frac{\vert u(x)\vert}{1+\vert x \vert^{n+2s}} \dd x. 
\end{align*}} \(n\), \(s\), \(R\), \([f]_{C^{0,1}(\overline \Omega \times [0,\|u\|_{L^\infty(\Omega)}])}\), \(f(0)\), \(\|u\|_{L_s(\R^n)}\), \(\diam \Omega\), and \(\vert \Omega\vert\).
\end{thm}

Furthermore, in Chapter~\ref{yP1bfxWn}, we prove the corresponding theorem for the nonlocal Serrin's problem: 

\begin{thm}[{\thref{thm:Main Theorem} in Chapter~\ref{yP1bfxWn}}] \label{T7e81SDc}
    Let~\( \Omega \) be an open bounded subset of~$ \R^n$ with~$C^2$ boundary satisfying the uniform interior sphere condition with radius~$ r_\Omega >0$.
Let~$f \in C^{0,1}_{\mathrm{loc}}(\R)$ be such that~$f(0)\geq 0$.
Let~\(u\) be a weak solution of~\eqref{AIgLTPmL} such that \(\dist(\cdot, \R^n \setminus \Omega)^{-s}u \in C^1(\overline \Omega)\). Then
\begin{equation*}
\rho(\Omega) \leq C [ \pa_\nu^s u]_{\partial \Om}^{\frac 1 {s+2}} 
\end{equation*} with \begin{align*}
    [ \pa_\nu^s u]_{\partial \Om} &= \sup_{\substack{x,y\in \pa \Om \\ x\neq y}}\frac{\vert \pa_\nu^s u(x) - \pa_\nu^s u(y) \vert}{\vert x - y \vert }.
\end{align*}
The constant \(C\) depends only on \(n\), \(s\), \(r_\Omega\), \([f]_{C^{0,1}([0,\|u\|_{L^\infty(\Omega)}])}\), \(f(0)\), \(\|u\|_{L_s(\R^n)}\), and \(\diam \Omega\).
\end{thm}

See also \cite{MR4546078} where similar results are obtained for nonlocal overdetermined problems in exterior and annular domains.

The proofs of Theorems~\ref{pQsFqMH1}-\ref{T7e81SDc} rely on the method of moving planes and quantitative versions of nonlocal maximum principles/Hopf-type lemmas. Though one must adapt each proof to the relevant overdetermined condition, the overarching strategy for each of the three results is the same which justifies why we obtain the same modulus of continuity (\(t\mapsto t^{\frac 1 {s+2}}\)) in each result. However, setting \(s=1\) in Theorems \ref{pQsFqMH1}-\ref{sge2H8Ql} (formally, since a careful analysis of the proofs reveals that constant blows up as \(s\to 1^-\)), we see that we don't recover the (sharp) modulus of continuity that was obtained in \cite[Theorem 1.1]{MR3481178}. Conceptually, the reason for this is that the techniques used in the proof are fundamentally nonlocal, having no corresponding local analogue `in the limit'. Since the interior regularity for fractional Laplacian is essentially the same as in the local case and the nonlocal parallel surface problem is an interior problem, we expect that the sharp modulus of continuity should be linear. If we consider the 1-parameter family of ellipsoids described at the end of the previous subsection, we obtain a linear modulus of continuity which also supports this claim, see Section~\ref{z1jzKMXt} in Chapter~\ref{Fw81drHU}. Proving this, however, is still an open problem. 

Interestingly, setting \(s=1\) in Theorem~\ref{T7e81SDc} (again, formally), we obtain a better modulus of continuity compared to the results obtained in \cite[Theorem 1]{MR1729395} and in~\cite{MR3522349} (even though~\cite{MR3522349} imposes that \(\Omega\) is convex which we do not require). Furthermore, leveraging a fine analysis performed in Section~\ref{c8w8u7Hn} of Chapter~\ref{Fw81drHU}, we are able to improve the moduli of continuity in Theorem~\ref{sge2H8Ql} and Theorem~\ref{T7e81SDc} under the assumption that the boundary of \(\Omega\) is\footnote{Here, provided that \(\alpha\in (0,+\infty) \setminus \Z\), we use the notation \(C^\alpha = C^{k,\beta}\) where \(k = [\alpha]\) is the integer part of \(\alpha\), \(\beta = \alpha - [\alpha] \in (0,1)\).}  `uniformly \(C^\alpha\)' for \(\alpha\in (0,+\infty) \setminus \Z\), which we rigorously define in Chapter~\ref{tH4AETfY}. Indeed, under this extra boundary assumption, we obtain \begin{align*}
    \rho(\Omega) \leq C [u]_{\partial G}^{\frac{\alpha}{1+\alpha(s+1)}} \text{ and }  \rho(\Omega) \leq C [ \pa_\nu^s u]_{\partial \Om}^{\frac{\alpha}{1+\alpha(s+1)}} 
\end{align*} in Theorem~\ref{sge2H8Ql} and Theorem~\ref{T7e81SDc} respectively. Though these estimates are an improvement, when we formally send \(\alpha\to+\infty\) and \(s\to 1^-\), we still do not recover the sharp exponents from the local case, so we do not expect that the above estimates are sharp.  The analysis of the optimal exponents in stability estimates for local and nonlocal overdetermined problems remains an active and ongoing area of research.

%% file: Part1/Role-Antisymmetric-functions-thesis.tex
\chapter{The role of antisymmetric functions in nonlocal equations}  \label{1XMNCzrH}

We use a Hopf-type lemma for antisymmetric super-solutions to the Dirichlet problem for the fractional Laplacian with zero-th order terms,
in combination with the method of moving planes, to prove symmetry for the \emph{semilinear fractional parallel surface problem}. That is, we prove that non-negative solutions to semilinear Dirichlet problems for the fractional Laplacian in a bounded open set $\Omega \subset \R^n$ must be radially symmetric if one of their level surfaces is parallel to the boundary of $\Omega$; in turn, $\Omega$ must be a ball.

Furthermore, we discuss maximum principles and the Harnack inequality for antisymmetric functions in the fractional setting and provide counter-examples to these theorems when only `local' assumptions are imposed on the solutions. 
The construction of these counter-examples relies on an approximation result that states that
`all antisymmetric functions are locally antisymmetric and \(s\)-harmonic up to a small error'.

\section{Introduction and results}
In this paper, we are concerned with maximum principles for antisymmetric solutions to linear, nonlocal partial differential equations (\textsc{PDE}) where the nonlocal term is given by the fractional Laplacian and with the application of such maximum principles to the method of moving planes.

Since its introduction by Alexandrov \cite{MR0150709,MR143162} and the
subsequent refinement by Serrin~\cite{MR333220}, the method of moving planes has become a standard technique for the analysis of overdetermined \textsc{PDE}, and for proving symmetry and monotonicity results for solutions, see \cite{MR544879,MR634248,MR2293958,MR2366129,MR3040677,MR3189604,MR3385189,MR3454619}. In recent years, nonlocal \textsc{PDE}, particularly the ones involving the fractional Laplacian, have received a lot of attention, due to their consolidated ability to model complex physical phenomena, as well as their rich analytic theory. The method of moving planes has been also adapted in several works so that it may be applied to nonlocal problems, see
\cite{MR3395749,MR3827344,MR3881478,MR3836150,MR4313576,ciraolo2021symmetry}.

The method of moving planes typically leads to 
considering the difference between a given solution and its reflection, and such a difference is an antisymmetric function. However, in the local case, `standard' maximum principles
can be directly applied to the moving plane method,
without specifically relying on the notion of
antisymmetric functions, since the ingredients required
are all of `local' nature and can disregard the `global' geometric
properties of the objects considered. Instead,
in the nonlocal world, new maximum principles need to be developed which take into account the extra symmetry, to
compensate for global control of the function under consideration,
for example see \cite{MR3395749}. 

We will now state our main results.
Let \(s\in (0,1)\) and~\(n\geqslant 1\) be an integer. The fractional Laplacian \((-\Delta)^s\) is defined by \begin{align*}
(-\Delta)^s u(x) &= c_{n,s}\PV \int_{\R^n} \frac{u(x)-u(y)}{\vert x -y \vert^{n+2s}} \dd y
\end{align*} where \begin{align}
c_{n,s} = \frac{s 4^s \Gamma \big (\frac{n+2s}{2} \big )}{\Gamma(1-s) } \label{tK7sb}
\end{align}  is a positive normalisation constant and \(\PV\) refers to the Cauchy principle value. Moreover, let~\(\R^n_+ = \{ x\in \R^n \text{ s.t. } x_1>0\}\) and \(B_1^+ = B_1 \cap \R^n_+\). 

Our first result proves that `all antisymmetric functions are locally antisymmetric and \(s\)-harmonic up to a small error'. This is the analogue of \cite[Theorem 1]{MR3626547} for antisymmetric functions. 

\begin{thm} \label{u7dPW} Suppose that \(k \in \N\) and that \(f \in C^k(\overline{B_1})\) is \(x_1\)-antisymmetric. For all \(\varepsilon>0\), there exist a smooth \(x_1\)-antisymmetric function \(u:\R^n \to \R\) and a real number \(R>1\) such that \(u\) is \(s\)-harmonic in \(B_1\), \(u=0\) in \(\R^n\setminus B_R\), and \begin{align*}
\| u - f \|_{C^k(B_1)}<\varepsilon .
\end{align*}
\end{thm}

When we say that \(u\) is \(x_1\)-antisymmetric we mean that \( u(-x_1,\dots,x_n)=-u(x_1,\dots,x_n) \) for all~\(x\in \R^n\). In \S\ref{9Hmeh} we will give a more general definition of antisymmetry for an arbitrary plane. As an application of Theorem \ref{u7dPW}, we can construct some counter-examples to the local Harnack inequality and strong maximum principle for antisymmetric \(s\)-harmonic functions, as follows:

\begin{cor}\label{X96XH}
There exist \(\Omega \subset \R_+^n\) and \(\Omega'\subset \subset \Omega\) such that, for all \(C>0\), there exists an \(x_1\)-antisymmetric function \(u:\R^n \to \R\) that is non-negative and \(s\)-harmonic in \(\Omega\), and \begin{align*}
\sup_{\Omega'} u > C \inf_{\Omega'} u .
\end{align*}
\end{cor}

\begin{cor} \label{KTBle} There exist \(\Omega \subset \R^n_+\) and an \(x_1\)-antisymmetric function \(u:\R^n \to \R\) that is non-negative and \(s\)-harmonic in \(\Omega\), is equal to zero at a point in \(\Omega\), but is not identically zero in \(\Omega\).
\end{cor}

For our second result, we use
a Hopf-type lemma (see the forthcoming Proposition~\ref{lem:FLIFu}) in combination with the method of moving planes to establish symmetry for the \emph{semilinear fractional parallel surface problem}, which is described in what follows. 
%
%
Suppose that \(G\subset \R^n \) is open and bounded, and let~\(B_R\) denote the ball of radius \(R>0\) centered at 0. Let \(\Omega = G+B_R\), where, given sets \(A,B \subset \R^n\), \(A+B\) denotes the `Minkowski sum' of \(A\) and \(B\), defined as \begin{align*}
A+B = \{a+b \text{ s.t. } a\in A, b\in B\}. 
\end{align*} Consider the semilinear fractional equation \begin{align}
\begin{PDE}
(-\Delta)^s u &= f(u) &\text{in }\Omega, \\
u&=0 &\text{in }\R^n \setminus \Omega,\\
u&\geqslant 0 &\text{in } \Omega ,
\end{PDE} \label{rzZmb}
\end{align} with the overdetermined condition \begin{align}
u = c_0 \qquad \text{on } \partial G, \label{lnARf}
\end{align} for some (given) \(c_0\geqslant 0\). The \emph{semilinear fractional parallel surface problem} asks the following question:
for which~\(G\) does the problem~\eqref{rzZmb}-\eqref{lnARf} admit a non-trivial solution? The answer is that \(G\) must be a ball. More specifically, we have the following
result:

\begin{thm} \label{CccFw} Suppose that \(G\) is a bounded open set in \(\R^n\) with \(C^1\) boundary. Let~\(\Omega=G+B_R\), \(f:\R \to \R\) be locally Lipschitz, and \(c_0 \geqslant 0\). Furthermore, assume that there exists a non-negative function \(u \in C^s(\R^n)\) that is not identically zero and satisfies, in the pointwise sense,\begin{align} 
\begin{PDE}
(-\Delta)^s u &= f(u) &\text{in } \Omega,\\ u&=0 &\text{in } \R^n \setminus \Omega,  \\ u &= c_0 & \text{on } \partial G .
\end{PDE} \label{zcm6a}
\end{align} Then \(u\) is radially symmetric, \(u>0\) in \(\Omega\), and \(\Omega\) (and hence \(G\)) is a ball. 
\end{thm}

In the local setting (i.e., \(s=1\)), symmetry and
%
%
stability results for \eqref{zcm6a} were obtained in \cite{MR3420522} and \cite{MR3481178}.
%
%
Such analysis was motivated by the study of invariant isothermic surfaces of a nonlinear nondegenerate fast diffusion equation (see \cite{MR2629887}). 

More recently in \cite{ciraolo2021symmetry}, the nonlocal case \(s\in (0,1)\) was addressed 
and the analogue to Theorem \ref{CccFw}, as well as its stability generalization, were proved in the particular case \(f \equiv 1\); those results were obtained by making use of the method of moving planes as well as a (quantitative) Hopf-type lemma (see \cite[formula~(3.3)]{ciraolo2021symmetry}), which could be obtained as an application of the boundary Harnack inequality for antisymmetric $s$-harmonic functions proved in \cite[Lemma~2.1]{ciraolo2021symmetry}. The arguments used in~\cite{ciraolo2021symmetry} to establish such a boundary Harnack inequality rely on the explicit knowledge of the fractional Poisson kernel for the ball. However, due to the general nonlinear term~\(f\) in \eqref{zcm6a}, here the method of moving planes leads to considering a linear equation involving the fractional Laplacian but with zero-th order terms for which no Poisson formula is available. To overcome this conceptual difficulty, we provide the following Hopf-type lemma which allows zero-th order terms.

\begin{prop} \label{lem:FLIFu}  Suppose that \(c \in L^\infty(B_1^+)\),  \(u \in  H^s(\R^n) \cap C(B_1)\) is \(x_1\)-antisymmetric, and satisfies \begin{align}
\begin{PDE}
(-\Delta)^su +cu &\geqslant 0 &\text{in } B_1^+, \\
u&\geqslant 0 &\text{in } \R^n_+ ,\\
u&>0 &\text{in } B_1^+. 
\end{PDE} \label{YEL36}
\end{align} Then \begin{align}
\liminf_{h\to 0} \frac{u(he_1)} h > 0 . \label{FLBHT}
\end{align} 
\end{prop}

This result has previously been obtained in \cite[Proposition 2.2]{MR3937999} though the proof uses a different barrier. 

Proposition \ref{lem:FLIFu} establishes that antisymmetric solutions to \eqref{YEL36} must be bounded from below by a positive constant multiple of \(x_1\) close to the origin. At a first glance this is surprising as solutions to \eqref{YEL36} are in general only \(s\)-H\"older continuous up to the boundary, as proven in \cite{MR3168912}. Hence, it would appear more natural to 
consider \(\liminf_{h\to 0} \frac{u(he_1)} {h^s}\) instead of \eqref{FLBHT}, see for example \cite[Proposition 3.3]{MR3395749}. However, the (anti)symmetry of \(u\) means that Proposition~\ref{lem:FLIFu} is better understood as an interior estimate rather than a boundary estimate. 
Indeed, we stress that, differently from the classical Hopf lemma,
Proposition~\ref{lem:FLIFu} is not concerned with the growth of a solution from a boundary point, but mostly with the growth from a reflection point in the antisymmetric setting (notice specifically that~\eqref{FLBHT} provides a linear growth from the origin, which is
an interior point). In this sense, the combination of the antisymmetric geometry
of the solution and the fractional nature of the equation
leads to the two structural differences between Proposition~\ref{lem:FLIFu} and several other Hopf-type lemmata available in the literature for the nonlocal
case, namely the linear (instead of \(s\)-H\"older) growth
and the interior (instead of boundary) type of statement.

A similar result to Proposition~\ref{lem:FLIFu} was obtained in \cite[Theorem 1]{MR3910421} for the entire half-space instead of \(B_1^+\). This entails that~\cite[Theorem 1]{MR3910421} was only applicable (via the method of moving planes) to symmetry problems posed in all of \(\R^n\) whereas our result can be applied to problems in bounded domains. Moreover, \cite[Theorem 1]{MR3910421} is proven by contradiction while the proof of Proposition~\ref{lem:FLIFu}, in a similar vein to the original Hopf lemma, relies on the construction of an appropriate barrier from below, see Lemma \ref{SaBD4}. Hence, using a barrier approach also allows us to obtain a quantitative version of Proposition~\ref{lem:FLIFu} leading to quantitative estimates for the stability of Theorem \ref{CccFw}, which we plan to address in an upcoming paper.

It is also natural ask whether the converse to Theorem~\ref{CccFw} holds. We recall that if $\Omega$ is known to be a ball, i.e., $\Omega:=B_\rho$, then the radial symmetry of a bounded positive solution $u$ to
\begin{equation}\label{eq:Classical Dirichlet}
\begin{PDE}
-\Delta u &= f(u) & \text{in }  B_\rho\\
u&=0 &\text{on } \partial B_\rho
\end{PDE}
\end{equation}
is guaranteed for a large class of nonlinearities $f$. For instance, this is the case for $f$ locally Lipschitz, by the classical result obtained via the method of moving planes by Gidas-Ni-Nirenberg \cite{MR634248,MR544879}. Many extensions of Gidas-Ni-Nirenberg result can be found in the literature, including generalizations to the fractional setting (see, e.g., \cite{MR2114412}).
We also recall that an alternative method pioneered by P.-L. Lions in \cite{MR653200} provides symmetry of nonnegative bounded solutions to \eqref{eq:Classical Dirichlet} for (possibly discontinuous) nonnegative nonlinearities; we refer to \cite{MR3003296,MR1382205,MR2019179,MR4380032} for several generalizations.

\medskip

The paper is organised as follows. In Section \ref{eBxsh} we recall some standard notation, as well as provide some alternate proofs for the weak and strong maximum principle for antisymmetric functions. Moreover, we will prove Theorem \ref{u7dPW} and subsequently prove Corollary \ref{X96XH} and Corollary \ref{KTBle}. In Section \ref{R9GxY}, we will prove Proposition~\ref{lem:FLIFu} and in Section \ref{Svlif} we will prove Theorem \ref{CccFw}. In Appendix \ref{ltalz} we give some technical lemmas needed in the paper. 

\section{Maximum principles and counter-examples} \label{eBxsh}

\subsection{Definitions and notation} \label{9Hmeh}
Let \(n \geqslant 1\) be an integer, \(s\in (0,1)\), and \(\Omega\subset \R^n\) be a bounded open set.  The fractional Sobolev space \(H^s(\R^n)\) is defined as \begin{align*}
H^s(\R^n) = \bigg \{ u \in L^2(\R^n) \text{ s.t. } [ u ]_{H^s(\R^n)} < \infty \bigg \}
\end{align*} where \([ u ]_{H^s(\R^n)}\) is the Gagliardo semi-norm
\begin{align*}
[ u ]_{H^s(\R^n)} &= \bigg ( \int_{\R^n} \int_{\R^n} \frac{\vert u(x)-u(y)\vert^2}{\vert x -y \vert^{n+2s}} \dd x \dd y \bigg )^{\frac12}. 
\end{align*} As usual we identify functions that are equal except on a set of measure zero. It will also be convenient to introduce the space \begin{align*}
\mathcal H^s_0(\Omega) = \{ u \in H^s (\R^n) \text{ s.t. } u = 0 \text{ in } \R^n \setminus \Omega \}. 
\end{align*} Suppose that \(c:\Omega \to \R\) is a measurable function and that \(c\in L^\infty(\Omega)\).  A function \(u \in H^s(\R^n)\) is a \emph{weak} or \emph{generalised} solution of \((-\Delta)^s u +cu =0\) (resp. \(\geqslant 0, \leqslant 0\)) in \(\Omega\) if \begin{align*}
\mathfrak L (u,v) := \frac{c_{n,s}}2 \iint_{\R^n \times \R^n} \frac{(u(x)-u(y))(v(x)-v(y))}{\abs{x-y}^{n+2s}} \dd x \dd y  \\
\qquad + \int_{\Omega} c(x) u(x) v(x) \dd x =0 \, (\geqslant 0, \leqslant 0)
\end{align*} for all \(v \in \mathcal H^s_0(\Omega) \), \(v \geqslant 0\). Also, it will be convenient to use the notation \begin{align*}
\mathcal E (u,v) = \frac{c_{n,s}}2 \iint_{\R^n \times \R^n} \frac{(u(x)-u(y))(v(x)-v(y))}{\abs{x-y}^{n+2s}} \dd x \dd y
\end{align*} for each \(u,v \in H^s(\R^n)\). 

A function \(u :\R^n \to \R\) is \emph{antisymmetric} with respect to a plane \(T\) if \begin{align*}
u(Q_T(x)) = -u(x) \qquad \text{for all } x\in \R^n
\end{align*} where \(Q_T:\R^n \to \R^n\) is the reflection of \(x\) across \(T\). A function \(u\) is \emph{\(x_1\)-antisymmetric} if it is antisymmetric with respect to the plane \(T=\{x_1=0\}\). In the case \(T=\{x_1=0\}\), \(Q_T\) is given explicitly by \(Q_T(x) = x-2x_1e_1\). When it is clear from context what \(T\) is, we will also write~\(Q(x)\) or~\(x_\ast\) to mean~\(Q_T(x)\). 

We will also make use of standard notation: the upper and lower half-planes are given by \(\R^n_+= \{ x \in \R^n \text{ s.t } x_1>0\}\) and \(\R^n_-= \{ x \in \R^n \text{ s.t } x_1<0\}\) respectively, and, for all \(r>0\), the half-ball of radius \(r\) is given by \(B_r^+ = \R^n_+ \cap B_r\). We also denote the positive and negative part of a function~\(u\) by~\(u^+(x) = \max \{u(x),0\}\)
and~\(u^-(x) = \max \{-u(x),0\}\). Furthermore, the characteristic function of a set \(A\) is given by \begin{align*}
\chi_A(x) &= \begin{cases}
1, &\text{if } x\in A, \\
0,&\text{if } x\not\in A. 
\end{cases}
\end{align*}

\subsection{Maximum principles}

We now present some results of maximum principle type for antisymmetric functions.

\begin{prop}[A weak maximum principle] \label{aKht4} Suppose that \(\Omega\) is a strict open subset of \(\R^n_+\), and~\(c :\Omega \to \R\) ia a measurable and non-negative function. Let \(u \in H^s(\R^n)\) be \(x_1\)-antisymmetric. If \(u\) satisfies \((-\Delta)^su+cu\leqslant0\) in \(\Omega\) then \begin{align*}
\sup_{\R^n_+} u \leqslant \sup_{\R^n_+\setminus \Omega} u^+ . 
\end{align*} Likewise, if \(u\) satisfies \((-\Delta)^su+cu\geqslant0\) in \(\Omega\) then \begin{align*}
\inf_{\R^n_+} u \geqslant - \sup_{\R^n_+ \setminus \Omega } u^- . 
\end{align*}
\end{prop}

In our proof, we use that $c \ge 0$. We refer, e.g., to \cite[Proposition 3.1]{MR3395749} for a weak maximum principle where the non-negativity of $c$ is not required.

\begin{proof}[Proof of Proposition \ref{aKht4}]
It is enough to prove the case \((-\Delta)^su+cu\leqslant0\) in \(\Omega\). Suppose that \(\ell := \sup_{\R^n_+ \setminus \Omega} u^+ <+\infty\), otherwise we are done.

For each \(v\in H^s_0(\Omega)\), \(v \geqslant 0\),  rearranging  \(\mathfrak L (u,v) \leqslant 0\) gives \begin{align}
\mathcal E(u,v) \leqslant - \int_\Omega c(x) u(x) v(x) \dd x \leqslant 0  \label{FQBj4}
\end{align} provided that \(uv \geqslant 0\) in \(\Omega\). 

Let \(w(x):=u(x)-\ell\) and \(v (x) := (u(x)-\ell)^+\chi_{\R^n_+}(x)=w^+(x)\chi_{\R^n_+}(x)\). Note that for all \(x\in \Omega\),\begin{align*}
u(x)v(x) &=  ((u(x)-\ell)^+)^2 + \ell (u(x)-\ell)^+ \geqslant 0. 
\end{align*} Here we used that \((u-\ell)^+(u-\ell)^-=0\). To prove the proposition, it is enough to show that \(v=0\) in \(\R^n\). As \(u(x)-u(y)=w(x)-w(y)\), expanding then using that \(w^+ (x) w^-(x)=0\) gives  \begin{align*}
(u(x)-u(y))(v(x)-v(y)) &= ((v(x)-v(y))^2 +v(x)(v(y)-w(y)) +v(y) (v(x)-w(x))
\end{align*} for all \(x,y \in \R^n\). Hence, \begin{align*}
\mathcal E (u,v) &= \mathcal E (v,v) + c_{n,s} \iint_{\R^n\times \R^n} \frac{v(x)(v(y)-w(y))}{\vert x- y \vert^{n+2s} } \dd x \dd y.
\end{align*} Observe that\begin{align*}
\iint_{\R^n\times \R^n} \frac{v(x)(v(y)-w(y))}{\vert x- y \vert^{n+2s} } \dd x \dd y &= \iint_{\R^n_+\times \R^n_+} \frac{w^+(x)w^-(y)}{\vert x- y \vert^{n+2s} } \dd x \dd y\\
&\qquad -\iint_{\R^n_-\times \R^n_+} \frac{w^+(x)w(y)}{\vert x- y \vert^{n+2s} } \dd x \dd y.
\end{align*} Making the change of variables \(y \to y_\ast\), where \(y_\ast\) denotes the reflection of \(y\) across the plane \(\{x_1=0\}\), then writing \(w=w^+-w^-\) gives \begin{align*}
&\hspace{-2em}\iint_{\R^n\times \R^n} \frac{v(x)(v(y)-w(y))}{\vert x- y \vert^{n+2s} } \dd x \dd y \\
&=\iint_{\R^n_+\times \R^n_+} \frac{w^+(x)w^-(y)}{\vert x- y \vert^{n+2s} } \dd x \dd y+\iint_{\R^n_+\times \R^n_+} \frac{w^+(x)(w(y)+2\ell)}{\vert x- y_\ast \vert^{n+2s} } \dd x \dd y \\
&= \iint_{\R^n_+\times \R^n_+} w^+(x)w^-(y) \bigg ( \frac1{\vert x- y \vert^{n+2s}}-\frac1{\vert x- y_\ast \vert^{n+2s}} \bigg )\dd x \dd y \\
& \qquad + \iint_{\R^n_+\times \R^n_+} \frac{w^+(x)(w^+(y)+2\ell)}{\vert x- y_\ast \vert^{n+2s} } \dd x \dd y \\
&\geqslant 0. 
\end{align*} Thus, \begin{align}
\mathcal E(u,v) \geqslant \mathcal E  (v,v) \geqslant 0 . \label{vndMm}
\end{align} Combining \eqref{FQBj4} and \eqref{vndMm} gives that \(\mathcal E(v,v) =0\) which implies that \(v\) is a constant in \(\R^n\). Since \(v=0\) in \(\R^n_- \), we must have that \(v=0\) as required. 
\end{proof}

\begin{remark}
It follows directly from Proposition \ref{aKht4} that if \(u\) satisfies \((-\Delta)^s u +cu =0\) in \(\Omega\) then \begin{align*}
\sup_{\R^n_+} \vert u \vert \leqslant \sup_{\R^n_+\setminus \Omega} \vert  u \vert . 
\end{align*}
\end{remark}

Next, we prove the strong maximum principle for antisymmetric functions. This result is not new in the literature, see \cite[Corollary 3.4]{MR3395749}; however, Proposition \ref{fyoaW} provides an alternate elementary proof. 

\begin{prop}[A strong maximum principle] \label{fyoaW} Let \(\Omega\subset \R^n_+\) and \(c :\Omega \to \R\). Suppose that \(u :\R^n \to \R\) is \(x_1\)-antisymmetric and \((-\Delta)^su\) is defined pointwise in \(\Omega\). If \(u\) satisfies \((-\Delta)^su+cu \geqslant 0\) in \(\Omega\), and \(u\geqslant 0\) in \(\R^n_+\) then either \(u>0\) in \(\Omega\) or \(u\) is zero almost everywhere in \(\R^n\). 
\end{prop}

\begin{remark}Note that in Proposition~\ref{fyoaW}, as in \cite[Corollary 3.4]{MR3395749}, no sign assumption is required on \(c\). The reason Proposition~\ref{fyoaW} holds without this assumption is due to the sign assumption on \(u\) which is not usually present in the statement of the strong maximum principle.
\end{remark}

\begin{proof}[Proof of Proposition~\ref{fyoaW}]
Due to the antisymmetry of \(u\), by a change of variables we may write \begin{align}
\begin{aligned}
(-\Delta)^s u(x) &= c_{n,s} \PV \int_{\R^n_+} \bigg ( \frac 1 {\vert x - y \vert^{n+2s}} - \frac 1 {\vert x_\ast - y \vert^{n+2s}} \bigg ) (u(x)-u(y)) \dd y  \\
&\qquad \qquad + 2c_{n,s}\int_{\R^n_+} \frac{u(x)}{\vert x_\ast - y \vert^{n+2s}} \dd y
\end{aligned} \label{u9I7E}
\end{align} where \(x_\ast=x-2x_1e_1\) is the reflection of \(x\) across \(\{x_1=0\}\). Suppose that there exists \(x_\star\in \Omega\) such that \(u(x_\star)=0\). On one hand,  \eqref{u9I7E} implies that \begin{align}
(-\Delta)^s u(x_\star) &= -c_{n,s} \PV \int_{\R^n_+} \bigg ( \frac 1 {\vert x_\star - y \vert^{n+2s}} - \frac 1 {\vert (x_\star)_\ast - y \vert^{n+2s}} \bigg )u(y) \dd y \label{XXP2e}
\end{align} where we used that \(u\) is antisymmetric. Since \begin{align*}
\frac 1 {\vert x_\ast - y \vert^{n+2s}} < \frac 1 {\vert x - y \vert^{n+2s}}, \qquad \text{for all } x,y\in \R^n_+, \, x\neq y,
\end{align*} equation \eqref{XXP2e} implies that \((-\Delta)^s u(x_\star) \leqslant 0\) with equality if and only if \(u\) is identically zero almost everywhere in \(\R^n_+\). On the other hand, \((-\Delta)^su(x_\star) = (-\Delta)^s u(x_\star) +c(x_\star) u(x_\star) \geqslant 0 \). Hence, \((-\Delta)^su(x_\star)=0\) so \(u\) is identically zero a.e. in \(\R^n_+\).
\end{proof}

\subsection{Counter-examples}
The purpose of this subsection is to provide counter-examples to the classical Harnack inequality and the strong maximum principle for antisymmetric \(s\)-harmonic functions. A useful tool to construct such functions is the following antisymmetric analogue of Theorem 1 in \cite{MR3626547}. This proves that `all antisymmetric functions are locally antisymmetric and \(s\)-harmonic up to a small error'. 

\begin{proof}[Proof of Theorem \ref{u7dPW}]
Due to \cite[Theorem 1.1]{MR3626547} there exist \(R>1\) and \(v\in H^s (\R^n) \cap C^s(\R^n)\) such that \((-\Delta)^sv=0\) in \(B_1\), \(v=0\) in \(\R^n \setminus B_R\), and \( \| v - f\|_{C^k(B_1)} < \varepsilon \). In fact, by a mollification argument, we may take \(v\) to be smooth, see Remark \ref{BQkXm} here below.

Let \(x_\ast\) denote the reflection of \(x\) across \( \{ x_1=0 \}\), and \begin{align*}
u(x) = \frac12 \big ( v(x) - v(x_\ast) \big ) . 
\end{align*} It is easy to verify that \((-\Delta)^s u = 0\) in \(B_1\) (and of course \(u=0\) in \(\R^n \setminus B_R\)). By writing \begin{align*}
u(x) -f(x) &= \frac12 \big (v(x) - f(x) \big ) + \frac12 \big ( f(x_\ast) - v(x_\ast) \big ) \qquad\text{for all } x \in B_1
\end{align*} we also obtain \begin{equation*}
\| u - f\|_{C^k(B_1)} \leqslant \| v - f\|_{C^k(B_1)} < \varepsilon . \qedhere
\end{equation*}
\end{proof}

\begin{remark} \label{BQkXm}
Fix \(\varepsilon>0\). If \(f \in C^k (\overline{B_1} )\) there exists \(\mu >0\) such that \(f\in C^k(B_{1+2\mu})\). Rescaling then applying \cite[Theorem 1.1]{MR3626547}, there exist \(\tilde R>1+\mu\) and \(\tilde v\in H^s (\R^n) \cap C^s(\R^n)\) such that \((-\Delta)^s \tilde v=0\) in \(B_{1+\mu}\), \(\tilde v=0\) in \(\R^n \setminus B_{\tilde R }\), and \( \| \tilde v - f\|_{C^k(B_{1+\mu})} < \varepsilon/2 \).

Let \(\eta \in C^\infty_0(\R^n)\) be the standard mollifier \begin{align*}
\eta (x) &= \begin{cases}
C e^{- \frac 1 {\vert x \vert^2-1}} , &\text{if } x \in B_1 ,\\
0, &\text{if } x \in \R^n \setminus B_1,
\end{cases}
\end{align*} with \(C>0\) chosen so that \(\int_{\R^n} \eta(x) \dd x =1\). For each \(\delta>0\), let \(\eta_\delta(x) = \delta^{-n} \eta (x/\delta))\) and  \begin{align*}
\tilde{v}^{(\delta)}(x) := (\tilde v \ast \eta_\delta)(x) = \int_{\R^n} \eta_\delta (x-y) \tilde v (y) \dd y. 
\end{align*} By the properties of mollifiers, \begin{align*}
\| \tilde v^{(\delta)} - f \|_{C^k(B_1)} \leqslant \| \tilde v^{(\delta)} - \tilde v \|_{C^k(B_1)} + \| \tilde v - f \|_{C^k(B_1)} < \varepsilon
\end{align*} provided \(\delta\) is sufficiently small. Since \(\tilde v \in H^s(\R^n) \subset L^2(\R^n)\), it follows from \cite[Proposition 4.18]{MR2759829} that \(\supp \tilde{v}^{(\delta)}(x) \subset  \overline{B_{\tilde{R}+\delta}}\). Moreover, via the Fourier transform, \((-\Delta)^s \tilde{v}^{(\delta)} = \big ( (-\Delta)^s \tilde{v} \big ) \ast \eta_\delta\), so \cite[Proposition 4.18]{MR2759829} again implies that \( \supp \big ( (-\Delta)^s \tilde{v}^{(\delta)}\big ) \subset \R^n \setminus B_{1+\mu-\delta}\).

Hence, setting \(R=\tilde R+\delta\), \(v=\tilde v^{(\delta)}\), we have constructed a function \(v  \in C^\infty_0(\R^n)\) such that \(v\) is \(s\)-harmonic in \(B_1\), \(v=0\) outside \(B_R\) and \(\| v - f \|_{C^k(B_1)}<\varepsilon\). Moreover, 
$$ v(x_\ast) = \int_{\R^n} \eta_\delta (y) \tilde v (x_\ast- y) \dd y =-\int_{\R^n} \eta_\delta (y) \tilde v (x- y_\ast) \dd y =-v(x)$$ via the change of variables \(z=y_\ast\) and using that \(\eta_\delta(y_\ast) = \eta_\delta (y)\). 
\end{remark}

It is well known that the classical Harnack inequality fails for \(s\)-harmonic functions, see \cite{Kassmann2007clas} for a counter-example. The counter-example provided in \cite{Kassmann2007clas} is not antisymmetric; however, Corollary \ref{X96XH} proves that even with this extra symmetry the classical Harnack inequality does not hold for \(s\)-harmonic functions. Moreover, the construction of the counter-example in Corollary \ref{X96XH} is entirely different to the one in \cite{Kassmann2007clas}. 
 
 \begin{proof}[Proof of Corollary \ref{X96XH}] We will begin by proving the case \(n=1\). Suppose that \(\varepsilon\in (0,1)\), \(I'=(1,2)\),  and \(I=(1/2,5/2)\). Let \(\{ f^{(\varepsilon)}\} \subset C^\infty (\R)\) be a family of odd functions that depend smoothly on the parameter \(\varepsilon\). Moreover, suppose that \begin{align}
 \sup_{I'} f^{(\varepsilon)} = 4 \label{UJd6j}
\end{align}  and \begin{align}
\inf_{I} f^{(\varepsilon)} =\inf_{I'} f^{(\varepsilon)} =2\varepsilon . \label{BR8WJ}
\end{align}  For example, such a family of functions is \begin{align*}
 f^{(\varepsilon)}(x) =  ax +bx^3+cx^5 +dx^7 
 \end{align*} where \begin{align*}
 a = \frac 5{54} (64+5\varepsilon), \quad b= -\frac 1{72}(128+73\varepsilon), \quad c=\frac 1 {36} (-8+23\varepsilon), \quad d=\frac 1 {216}(16-19\varepsilon).
 \end{align*} The functions \(f^{(\varepsilon)}\) are plotted for several values of \(\varepsilon\) in Figure \ref{2FAHv}. 
 
 After a rescaling, it follows from Theorem \ref{u7dPW} that there exists a family \(v^{(\varepsilon)}\) of odd functions that are \(s\)-harmonic in \((-5/2,5/2)\supset I \) and satisfy \begin{align}
 \| v^{(\varepsilon)} - f^{(\varepsilon)} \|_{C((-5/2,5/2))} < \varepsilon . \label{pA7Na}
 \end{align} It follows from \eqref{BR8WJ} and \eqref{pA7Na} that \begin{align*}
\inf_{I} v^{\varepsilon}  \geqslant \varepsilon>0 . 
\end{align*}  Moreover, from \eqref{UJd6j}-\eqref{BR8WJ} and \eqref{pA7Na} , we have that \begin{align*}
\sup_{I'} v^{(\varepsilon)}\geqslant 4-\varepsilon \qquad \text{and} \qquad \inf_{ I'} v^{(\varepsilon)}\leqslant   3\varepsilon . 
\end{align*} Hence, \begin{align*}
\frac{\sup_{I'} v^{(\varepsilon)}}{\inf_{I'} v^{(\varepsilon)}} \geqslant \frac{4-\varepsilon}{3\varepsilon} \to + \infty
\end{align*} as \(\varepsilon \to 0^+\). Setting \(\Omega=I\) and \(\Omega'=I'\) proves the statement in the case \(n=1\).

To obtain the case \(n >1\), set \(\Omega'=I'\times (-1,1)^{n-1}\), \(\Omega=I \times (-2,2)^{n-1}\), and \(u^{(\varepsilon)} (x) = v^{(\varepsilon)} (x_1) \). Using that \begin{align*}
(-\Delta)^s u^{(\varepsilon)} (x) = C (-\Delta)^s_{\R} v^{(\varepsilon)} (x_1) \qquad \text{for all } x\in \R^n 
\end{align*}where \((-\Delta)^s_{\R} \) denotes the fractional Laplacian in one dimension and \(C\) is some constant, see \cite[Lemma 2.1]{MR3536990}, all of the properties of \(v^{(\varepsilon)} \) carry directly over to \(u^{(\varepsilon)} \). 
 \end{proof}

\begin{figure}[ht]
\centering
\includegraphics[scale=1]{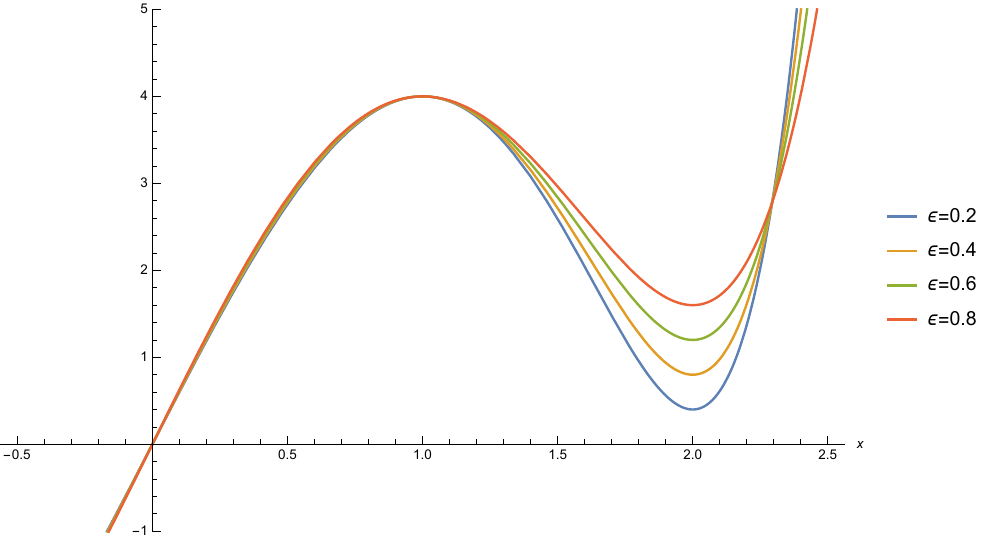} 
\caption{Plot of \(f^{(\varepsilon)}(x) \) for \(\varepsilon=0.2,0.4,0.6,0.8\).}
\label{2FAHv}
\end{figure}

Observe that the family of functions in the proof of Corollary \ref{X96XH} does not violate the classical strong maximum principle. In our next result, Corollary \ref{KTBle}, we use Theorem \ref{u7dPW} to provide a counter-example to the classical strong maximum principle for antisymmetric \(s\)-harmonic functions. This is slightly more delicate to construct than the counter-example in Corollary \ref{X96XH}. Indeed, Theorem \ref{u7dPW} only gives us an \(s\)-harmonic function \(u\) that is \(\varepsilon\)-close to a given function \(f\), so how can we guarantee that \(u=0\) at some point but \(u\) is not negative? 

The idea is to begin with an antisymmetric function \(f\) that has a non-zero minimum away from zero, see Figure \ref{73Iyk}, use Theorem \ref{u7dPW} to get an antisymmetric \(s\)-harmonic function with analogous properties then minus off a known antisymmetric \(s\)-harmonic function until the function touches zero. The proof of Corollary \ref{KTBle} relies on two technical results, Lemma \ref{IyXUO} and Lemma \ref{oujju}, both of which are included in Appendix \ref{ltalz}.

\begin{proof}[Proof of Corollary \ref{KTBle}] By the same argument at the end of the proof of Corollary~\ref{X96XH}, it is enough to take \(n=1\). Let \(\Omega=(0,3)\) and \(f :\R \to \R\) be a smooth, odd function such that \begin{align}
f(x) &\geqslant 1, \text{ for all } x\in [1,3];  \label{cv7Fh}\\
f(x) &\geqslant 3x,\text{ for all } x\in [0,1];  \label{iHz0V}\\
f(2) & =1; \text{ and }  \label{QEb8m} \\
f(3)&= 5. \label{ugncP}
\end{align} For example, such a function is  \begin{align*}
f(x) = -\frac{371 }{43200}x^9+\frac{167 }{1440}x^7-\frac{2681 }{14400}x^5-\frac{4193 }{2160}x^3+\frac{301 }{50} x. 
\end{align*} See Figure \ref{73Iyk} for a plot of \(f\).

By Theorem \ref{u7dPW} with \(\varepsilon =1\), there exists an odd function \(v \in C^\infty_0(\R)\) such that \(v\) is \(s\)-harmonic in~\((-4,4)\) and \begin{align}
\| v-f\|_{C^1((-4,4))} < 1 . \label{wjbkE}
\end{align} Let \(c_0,\zeta_R\) be as in Lemma \ref{oujju} and set \(\zeta := \frac 1 {c_0} \zeta_R. \) By choosing \(R>3\) sufficiently large, we have that  \begin{align}
\frac 3 4 x \leqslant \zeta (x) \leqslant \frac 5 4 x \qquad \text{in }(0,4) . \label{rU1A4}
\end{align} Define \begin{align*}
\phi_t(x) = v(x) - t \zeta (x), \qquad {\mbox{for all~$x\in(-3,3)$ and~$t >0$.}}
\end{align*} We have that, for all \(t>0\), \begin{align*}
(-\Delta)^s \phi_t(x) = 0  \qquad \text{in } (-3,3)
\end{align*} where \((-\Delta)^s=(-\Delta)^s_x \) is the fractional Laplacian with respect to \(x\), and \begin{align*}
x \mapsto \phi_t(x) \text{ is odd.} 
\end{align*} As in Lemma \ref{IyXUO}, let \begin{align*}
m(t) = \min_{x \in [1,3]} \phi_t(x) . 
\end{align*}  From \eqref{cv7Fh} and \eqref{wjbkE} we have that \( \phi_0(x) = v(x) > 0\) in \([1,3]\). Hence, \begin{align*}
m(0) >0. 
\end{align*} Moreover, by \eqref{QEb8m} and \eqref{wjbkE}, \(v(2)<2\). It follows from \eqref{rU1A4} that \begin{align*}
m(4/3) \leqslant \phi_{4/3}(2)=v(2) - \frac 4 3 \zeta (2) <0. 
\end{align*} Since \(m\) is continuous, as shown in Lemma \ref{IyXUO}, the intermediate value theorem implies the existence of \(t_\star \in (0,4/3)\) such that \begin{align*}
m(t_\star) = 0. 
\end{align*} Let \(u := \phi_{t_\star}\). By construction \(u \geqslant 0\) in \([1,3]\). Moreover, since \(u\) is continuous, there exists some \(x_\star \in [1,3]\) such that \(u(x_\star)=0\). In fact, we have that \(x_\star \neq 3\). Indeed, \eqref{ugncP} implies that \(v(3)>4\) and, since \(t_\star < 4/3\) and \(\zeta (3) <9/4\), we obtain \(u(3) >0\). 

All that is left to be shown is that \(u \geqslant 0\) in \((0,1)\). By \eqref{wjbkE}, we have that \begin{align*}
v'(x) > f'(x) - 1, \qquad {\mbox{for all }} x \in (-4,4). 
\end{align*} Since \(v(0)=f(0)=0\), it follows from the fundamental theorem of calculus that \begin{equation}\begin{split}
u(x) &= \int_0^x v'(\tau ) \dd \tau -t_\star \zeta (x) \\
&\geqslant  \int_0^x \big ( f'(\tau)-1 \big ) \dd \tau  - t_\star \zeta (x) \\
&=f(x)-x - t_\star \zeta (x) \label{A0ZNc}
\end{split}\end{equation} provided that~\(x\in (0,3)\). By \eqref{A0ZNc}, \eqref{iHz0V} and \eqref{rU1A4}, \begin{align*}
u(x) \geqslant  \frac 1 3 x \geqslant 0 \qquad \text{in } [0,1] ,
\end{align*}
as desired.
\end{proof}

\begin{figure}[ht]
\centering
\includegraphics[scale=1]{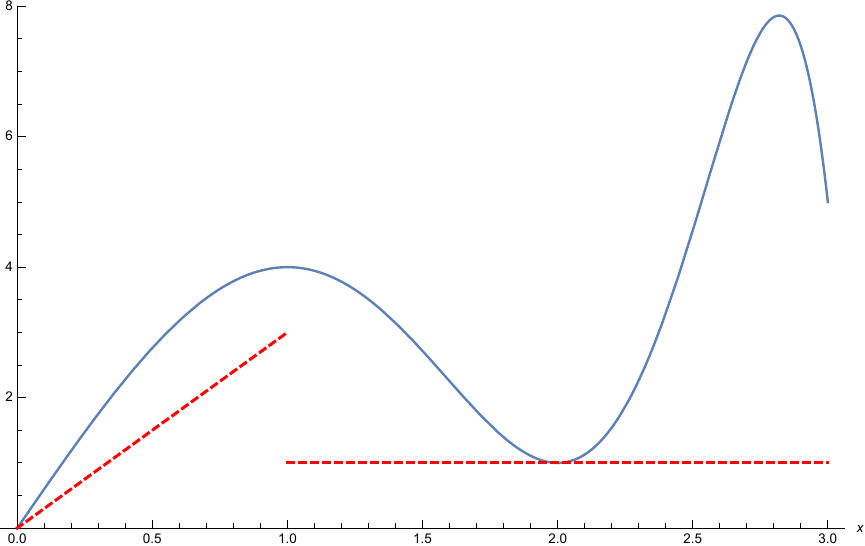} 
\caption{Plot of \(f \) as in the proof of Corollary \ref{KTBle}.}\label{73Iyk}
\end{figure}

\section{A Hopf lemma} \label{R9GxY}
In this section, we prove Proposition~\ref{lem:FLIFu}. The main step in the proof is the construction of the following barrier. 

\begin{lem} \label{SaBD4}
Suppose that \(c \in L^\infty (B_2^+)\). Then there exists an \(x_1\)-antisymmetric function \(\varphi \in C^\infty(\R^n)\) such that \begin{align*}
\begin{PDE}
(-\Delta)^s \varphi +c \varphi &\leqslant 0 &\text{in }B_2^+\setminus B_{1/2}(e_1), \\
\varphi &=0 &\text{in } \R^n_+ \setminus B_2^+ ,\\
\varphi &\leqslant 1 &\text{in } B_{1/2}(e_1),\\
\partial_1 \varphi (0) &>0 . 
\end{PDE}
\end{align*} 
\end{lem}

\begin{proof}
Let \(\zeta\) be a smooth \(x_1\)-antisymmetric cut-off function such that the support of \(\zeta\) is contained in \(B_2\),  \(\zeta \geqslant 0\) in \(B_2^+\), and \(\partial_1\zeta (0)>0\). For example, such a function is \( x_1\eta(x)\) where \(\eta\) is the standard mollifier defined in Remark \ref{BQkXm}. Since \(\zeta\) is smooth with compact support, we have that~\((-\Delta)^s \zeta \in C^\infty (\R^n)\). 

Moreover, \((-\Delta)^s \zeta \) is \(x_1\)-antisymmetric since \(\zeta\) is \(x_1\)-antisymmetric, so it follows that there exists~\(C>0\) such that \begin{align}
(-\Delta)^s \zeta (x) +c \zeta (x) \leqslant C(1+\| c \|_{L^\infty(B_2^+)}) x_1 \qquad \text{in } B_2^+  \label{I6ftM}
\end{align} using that \(c\in L^\infty(B_2^+)\).

Next, let \(\tilde \zeta \) be a smooth \(x_1\)-antisymmetric function such that \(\tilde \zeta  \equiv 1\) in \(B_{1/4}(e_1)\), \(\tilde \zeta  \equiv 0\) in \(\R^n_+ \setminus B_{3/8}(e_1)\), and \(0\leqslant \tilde \zeta  \leqslant 1\) in \(\R^n_+\). Recall that, given an \(x_1\)-antisymmetric function \(u\), the fractional Laplacian of \(u\) can be written as \begin{align*}
(-\Delta)^su(x) &= c_{n,s} \int_{\R^n_+} \bigg ( \frac 1 {\vert x - y\vert^{n+2s}} -  \frac 1 {\vert x_\ast -y\vert^{n+2s}} \bigg ) (u(x)-u(y) ) \dd y \\
&\qquad + 2 c_{n,s} u(x) \int_{\R^n_+} \frac{\dd y } {\vert x_\ast - y \vert^{n+2s}}.
\end{align*} Hence, for each \(x\in B_2^+ \setminus B_{1/2}(e_1)\),  \begin{align*}
(-\Delta)^s \tilde \zeta (x)&=-c_{n,s}  \int_{B_{3/8}(e_1)} \bigg ( \frac 1 {\vert x - y \vert^{n+2s}} - \frac 1 {\vert x_\ast - y \vert^{n+2s}} \bigg ) \tilde \zeta (y) \dd y.
\end{align*} By the fundamental theorem of calculus, for all \(x\in B_2^+ \setminus B_{1/2}(e_1)\) and~\(y\in B_{3/8}(e_1)\),\begin{align*}
\frac 1 {\vert x - y \vert^{n+2s}} - \frac 1 {\vert x_\ast - y \vert^{n+2s}} &= \frac{n+2s}{2} \int_{\vert x - y \vert^2}^{\vert x_\ast - y \vert^2} \frac{ \dd \tau } {\tau^{\frac{n+2s+2}{2}}}\\
&\geqslant \frac{Cx_1y_1}{\vert x_\ast - y \vert^{n+2s+2}} \\
&\geqslant Cx_1  
\end{align*} with \(C\) depending only on \(n\) and \(s\). Hence, \( (-\Delta)^s\tilde \zeta (x) \leqslant -Cx_1\) in \( B_2^+ \setminus B_{1/2}(e_1)\). Then the required function is given by \(\varphi(x) := \zeta (x) + \alpha \tilde \zeta (x)\) for all \(x\in \R^n\) with \(\alpha>0\) to be chosen later. Indeed, from~\eqref{I6ftM}, we have that \begin{align*}
(-\Delta)^s \varphi +c \varphi \leqslant C (1-\alpha+ \| c \|_{L^\infty(B_2^+)})x_1 \leqslant 0
\qquad\text{in } B_2^+ \setminus B_{1/2}(e_1)
\end{align*} provided that \(\alpha \) is large enough. 
\end{proof}

{F}rom Lemma \ref{SaBD4} the proof of Proposition~\ref{lem:FLIFu} follows easily:

\begin{proof}[Proof of Proposition~\ref{lem:FLIFu}]
Since \(u >0\) in \(B_1^+\), we have that,
for all \(v\in \mathcal H^s_0(B_1^+)\) with~\(v\geqslant0\), \begin{align*}
0\leqslant \mathcal E(u,v) + \int_{B_1^+} c(x) u(x) v(x) \dd x \leqslant\mathcal E(u,v) + \int_{B_1^+} c^+(x) u(x) v(x) \dd x
\end{align*} and thereofore \begin{align*}
(-\Delta)^su +c^+u \geqslant 0 \qquad \text{in } B_1^+.
\end{align*} Hence, it suffices to prove Proposition~\ref{lem:FLIFu} with \(c \geqslant 0\). Let \(\rho>0\) be such that \(B_{2\rho} \subset B_1\) and  \(\varphi_\rho (x) = \varphi(x/\rho)\) where \(\varphi\) is as in Lemma \ref{SaBD4}. Provided that \(\varepsilon\) is sufficiently small, we have \begin{align*}
(-\Delta)^s(u-\varepsilon \varphi_\rho) +c(u-\varepsilon \varphi_\rho)\geqslant 0  \qquad \text{in }  B_{2\rho} \setminus B_{\rho/2}(\rho e_1)
\end{align*} and \(u-\varepsilon \varphi_\rho\geqslant 0\) in \((\R^n_+\setminus B_{2\rho})\cup B_{\rho/2}(\rho e_1)\). It follows from Proposition~\ref{aKht4} that \begin{align*}
u \geqslant \varepsilon \varphi_\rho \qquad \text{in } \R^n_+
\end{align*} where we used that \(c\geqslant0\). Since \(u(0)=\varphi_\rho(0)=0\), we conclude that \begin{align*}
\liminf_{h\to 0} \frac{u(he_1)}{h} \geqslant \varepsilon \partial_1 \varphi_\rho(0) >0,
\end{align*}
as desired.
\end{proof}

\section{Symmetry for the semilinear fractional parallel surface problem}\label{Svlif}

In this section, we will give the proof of Theorem \ref{CccFw}. For simplicity and the convenience of the reader, we will first state the particular case of Proposition~3.1 in \cite{MR3395749}, which we make use of several times in the proof of Theorem \ref{CccFw}. Note that \(c\) in Proposition \ref{QawmgWdG} corresponds to \(-c\) in \cite{MR3395749}---we made this change so that the notation of Proposition \ref{QawmgWdG} would agree with the notation in the proof of Theorem \ref{CccFw}.

\begin{prop}[Proposition~3.1 in \cite{MR3395749}] \label{QawmgWdG} Let \(H\) be a halfspace, \(\Omega \subset H\) be any open, bounded set, and \(c\in L^\infty (\Omega)\) be such that \(-c \leqslant c_\infty<\lambda_1(\Omega)\) in \(\Omega\) for some \(c_\infty \geqslant 0\),
where~\(\lambda_1(\Omega)\) is the first Dirichlet eigenvalue of \((-\Delta)^s\) in \(\Omega\). If \(u \in H^s(\R^n)\) satisfies \((-\Delta)^s u +cu \geqslant 0\) in \(\Omega\) and \(u\) is antisymmetric in \(\R^n\) then \(u \geqslant 0\) almost everywhere in \(\Omega\). 
\end{prop}

Now we will prove Theorem \ref{CccFw} in the case \(n\geqslant2\). We prove the case \(n=1\) later in the section.

\begin{proof}[Proof of Theorem \ref{CccFw} for \(n\geqslant2\)] 
Fix a direction \( e \in \Sph^{n-1}\). Without loss of generality, we may assume that~\(e=e_1\). Let \(T_\lambda = \{x\in \R^n \text{ s.t. } x_1 = \lambda\}\)---this will be our `moving plane' which we will vary by decreasing the value of \(\lambda\). Since \(\Omega\) is bounded, we may let \(M = \sup_{x\in \Omega}x_1\) which is the first value of \(\lambda\) for which \(T_\lambda\) intersects \(\overline{\Omega}\). Moreover, let \(H_\lambda = \{ x \in \R^n \text{ s.t. } x_1 >\lambda\}\), \(\Omega_\lambda = H_\lambda \cap \Omega\), and \(Q_\lambda:\R^n \to \R^n\) be given by \(x \mapsto x- 2x_1+2\lambda e_1\). Geometrically, \(Q_\lambda(x)\) is the reflection of \(x\) across the hyperplane~\(T_\lambda\). 

Since \(\partial \Omega \) is \(C^1\) there exists some \(\mu <M\) such that \(\Omega_\lambda' := Q_\lambda(\Omega_\lambda) \subset \Omega\) for all \(\lambda\in (\mu,M)\), see for example \cite{MR1751289}. Let \(m\) be the smallest such \(\mu\), that is, let \begin{align*}
m = \inf \{ \mu<M \text{ s.t. } \Omega_\lambda' \subset \Omega \text{ for all } \lambda \in  (\mu,M)\}.
\end{align*} Note that \(\Omega_m' \subset \Omega\). Indeed, if this were not the case then there would exist some \(\varepsilon>0\) such that \(\Omega_{m+\varepsilon}' \not\subset \Omega\), which would contradict the definition of~\(m\). 

As is standard in the method of moving planes, we will consider the function \begin{align*}
v_\lambda (x) &= u(x) - u(Q_\lambda(x)), \qquad x\in \R^n
\end{align*}  for each \(\lambda \in [m,M)\). It follows that \(v_\lambda\) is antisymmetric and satisfies \begin{align*}
\begin{PDE}
(-\Delta)^s v_\lambda +c_\lambda v_\lambda &= 0 &\text{in } \Omega_\lambda' ,\\
v_\lambda &\geqslant 0 &\text{in } H_\lambda'\setminus \Omega_\lambda' ,
\end{PDE}
\end{align*} where \begin{align*}
c_\lambda (x) &= \begin{cases}
\frac{f((u(x))-f(u(Q_\lambda(x)))}{u(x) - u(Q_\lambda(x))}, &\text{if } u(x) \neq u(Q_\lambda(x)), \\
0,& \text{if } u(x) = u(Q_\lambda(x)),
\end{cases}
\end{align*} and \(H_\lambda' := Q_\lambda (H_\lambda )\). We claim that \begin{align}
 v_m \equiv 0 \qquad \text{in } \R^n. \label{VxU8s}
\end{align} 

Before we show \eqref{VxU8s}, let us first prove the weaker statement \begin{align}
v_m \geqslant 0 \qquad \text{in } H_m' . \label{g3zkO}
\end{align}  Since \(f \in C^{0,1}_{\textrm{loc}} (\R)\), it follows that \(c_\lambda \in L^\infty (\Omega_\lambda')\) and that \begin{align*}
\|c_\lambda \|_{L^\infty (\Omega_\lambda')} \leqslant  [ f ]_{C^{0,1}([0,\| u \|_{L^\infty(\Omega)}])} .
\end{align*} Here, as usual, \begin{align*}
[ f ]_{C^{0,1}([0,a])} &= \sup_{\substack{x,y\in [0,a]\\x\neq y} } \frac{\vert f(x)-f(y)\vert }{\vert x - y \vert}.
\end{align*}  We cannot directly apply Proposition~\ref{QawmgWdG} as \([ f ]_{C^{0,1}([0,\| u \|_{L^\infty(\Omega)}])}\) might be large. However,
%
%
by Proposition~\ref{ayQzh}
we have that \begin{align*}
\lambda_1 (\Omega_\lambda' ) \geqslant C \vert \Omega_\lambda' \vert^{-\frac{2s}n} \to \infty
\quad \text{ as } \, \lambda \to M^- ,
\end{align*} 
%
%
and since \(\|c_\lambda \|_{L^\infty (\Omega_\lambda')}\) is uniformly bounded with respect to \(\lambda\), Proposition~\ref{QawmgWdG} implies that there exists some \(\mu \in [m,M)\) such that \(v_\lambda \geqslant 0\) in \(\Omega_\lambda'\) for all \(\lambda \in [\mu,M)\). In fact, since \(u\) is not identically zero, after possibly increasing the value of \(\mu\) (still with \(\mu<M\)) we claim that \(v_\lambda > 0\) in \(\Omega_\lambda'\) for all \(\lambda \in [\mu,M)\). Indeed, there exists \(x_0\in \Omega\) such that \(u(x_0)>0\) so, provided \(\mu\) is close to \(M\), \(Q_\lambda (x_0) \not\in \Omega\). Then, for all \(\lambda \in [\mu,M)\), \begin{align*}
v_\lambda (x_0) = u(x_0) >0
\end{align*} so it follows from the strong maximum principle Proposition~\ref{fyoaW} that \(v_\lambda>0\) in \(\Omega_\lambda'\).

This allows us to define \begin{align*}
\tilde{m} =\inf \{ \mu\in [m,M) \text{ s.t. } v_\lambda > 0 \text{ in } \Omega_\lambda' \text{ for all } \lambda \in [\mu,M)\}.
\end{align*} We claim that \(\tilde{m} = m\). For the sake of contradiction, suppose that we have \(\tilde{m} >m\). By continuity in \(\lambda\), \(v_{\tilde{m}} \geqslant 0\) in \(H_{\tilde{m}}'\) and then, as above, Proposition~\ref{fyoaW} implies that \(v_{\tilde{m}} >0\) in \(\Omega_{\tilde{m}}'\). Due to the definition of \(\tilde{m}\), for all \(0<\varepsilon < \tilde{m}-m\), the set \(\{ v_{\tilde m-\varepsilon} \leqslant 0\} \cap \Omega_{\tilde m-\varepsilon}' \) is non-empty. Let \(\Pi_\varepsilon \subset\Omega_{\tilde m-\varepsilon}'\) be an open set such that \( \{ v_{\tilde m-\varepsilon} \leqslant 0\} \cap \Omega_{\tilde m-\varepsilon}'  \subset \Pi_\varepsilon\). By making \(\varepsilon\) smaller we may choose \(\Pi_\varepsilon\) such that \(\vert \Pi_\varepsilon\vert \) is arbitrarily close to zero. Hence, applying Proposition~\ref{QawmgWdG} then Proposition~\ref{fyoaW} gives that \(v_{\tilde m -\varepsilon} >0\) in \(\Pi_\varepsilon\) which is a contradiction. This proves \eqref{g3zkO}. 

Since \(v_m \geqslant 0\) in \(H_m'\), Proposition~\ref{fyoaW} implies that either \eqref{VxU8s} holds or \(v_m>0\) in \(\Omega_m'\). For the sake of contradiction, let us suppose that \begin{align}
v_m>0 \qquad \text{in } \Omega_m'. \label{JQNtW}
\end{align} By definition, \(m\) is the minimum value of \(\lambda\) for which \(\Omega_\lambda' \subset \Omega\). There are only two possible cases that can occur at this point. 

\begin{quotation}
Case 1: There exists \(p \in ( \Omega_m' \cap \partial \Omega) \setminus T_m\).\\
Case 2: There exists \(p \in T_m \cap \partial \Omega\) such that \(e_1\) is tangent to \(\partial \Omega\) at \(p\).
\end{quotation} 

If we are in Case 1 then there is a corresponding point \(q\in ( G_m' \cap \partial G) \setminus T_m \subset \Omega_m'\) since \(\partial G\) is parallel to \(\partial \Omega\). But \(u\) is constant on \(\partial G\), so we have \begin{align*}
v_m(q) = u(q) - u(Q_m (q)) = 0
\end{align*} which contradicts \eqref{JQNtW}. 

If we are in Case 2 then there exists \(q \in T_m \cap \partial G \) such that \(e_1\) is tangent to \(\partial G\) at \(q\). Since \(u\) is a constant on \(\partial G\), the gradient of \(u\) is perpendicular to \(\partial G\) and so \(\partial_1 u(q) = 0. \) Moreover, by the chain rule, \begin{align*}
\partial_1 (u\circ Q_m)(q) = - \partial_1u(q) = 0. 
\end{align*} Hence, \begin{align*}
\partial_1 v_m (q) = 0. 
\end{align*} However, this contradicts Proposition~\ref{lem:FLIFu}. Thus, in both Case 1 and Case 2 we have shown that the assumption \eqref{JQNtW} leads to a contradiction, so we conclude that \eqref{VxU8s} is true. 

This concludes the main part of the proof. The final two steps are to show that \(u\) is radially symmetric and that \(\supp u = \overline{\Omega}\). Due to the previous arguments we know that for all \(e\in \Sph^{n-1}\) and \(\lambda \in [m(e), M(e))\), \begin{align*}
u(x) - u(Q_{\lambda,e}(x))\geqslant 0 \qquad \text{in } H_{\lambda,e}'
\end{align*} where \(Q_{\lambda,e} =Q_{\lambda}\), \(H'_{\lambda,e} =H_{\lambda}' \) as before but we've included the subscript \(e\) to emphasise that the result is true for each \(e\). It follows that for every halfspace we either have \(u(x) - u(Q_{\partial H}(x)) \geqslant 0\) for all \(x\in H\) or \(u(x) - u(Q_{\partial H}(x)) \leqslant 0\) for all \(x\in H\). By \cite[Proposition 2.3]{MR2722502}, we can conclude that there exists some \(z\in \R^n\) such that \(x \mapsto u(x-z)\) is radially symmetric. Moreover, since \(u\geqslant0\) in \(\Omega\) and not identically zero, \cite[Proposition 2.3]{MR2722502} also implies that \(x \mapsto u(x-z)\) is non-increasing in the radial direction. 

It follows that \(\supp u\) is a closed ball. (Note that we are using the convention \\\(\supp u = \overline{\{x \in \R^n \text{ s.t. } u(x)\neq 0 \}}\)). We claim that \(\supp u = \overline{\Omega}\). For the sake of contradiction, suppose that this is not the case. Then there exists a direction \(e\in \Sph^{n-1}\) and \(\lambda \in (m(e),M(e))\) such that \(\Omega'_\lambda \cap \supp u = \varnothing \). It follows that \(v_\lambda \equiv 0\) in \(\Omega_\lambda'\). This is a contradiction since we previously showed that \(v_\lambda>0\) in \( \Omega_\lambda'\) for all \(\lambda \in (m(e),M(e))\).
\end{proof}

For $n=1$, Theorem \ref{CccFw} reads as follows.

\begin{thm}  Suppose that \(G\) is a bounded open set in \(\R\),  \(\Omega=G+B_R\), \(f:\R \to \R\) is locally Lipschitz, and \(c_0\in \R\). Furthermore, assume that there exists a non-negative function \(u \in C^s(\R)\) that is not identically zero and satisfies\begin{align} 
\begin{PDE}
(-\Delta)^s u &= f(u) &\text{in } \Omega , \\ u&=0 &\text{in } \R \setminus \Omega , \\ u &= c_0 & \text{on } \partial G .
\end{PDE} 
\end{align} Then, up to a translation, \(u\) is even, \(u>0\) in \(\Omega\), and \(G=(a,b)\) for some \(a<b\). 
\end{thm}

\begin{proof}
In one dimension the moving plane \(T_\lambda\) is just a point \(\lambda\) and there are only two directions \(T_\lambda\) can move--from right to left or from left to right. We will begin by considering the case \(T_\lambda\) is moving from right to left. Let \((a-R, b +R)\), \(a<b\), be the connected component of \(\Omega\) such that \(\sup \Omega=b+R\). Using the same notation as in the proof of Theorem \ref{CccFw}, it is clear that \(M =b+R\) and \(m = \frac{a+b} 2 \). For all \(\lambda \in ( \frac{a+b} 2, b+R)\), let  \begin{align*}
v_\lambda (x) &= u(x) - u(-x+2 \lambda ), \qquad x\in \R. 
\end{align*} Arguing as in the proof of Theorem \ref{CccFw}, we obtain \begin{align}
v_\lambda &> 0 \qquad \text{in } (-\infty , \lambda) \label{8e2sD}
\end{align} for all \(\lambda \in  ( \frac{a+b} 2, b+R)\) and \begin{align*}
v_{\frac{a+b}2} &\geqslant 0 \qquad \text{in } (-\infty , (a+b)/2).
\end{align*} Then the overdetermined condition \(u(a)=u(b)=c_0\) implies that \begin{align*}
v_{\frac{a+b}2} (b)&= u(b) - u(a) =0.
\end{align*} Hence, Proposition~\ref{fyoaW} gives that \begin{align}
v_{\frac{a+b}2} &\equiv 0 \qquad \text{in } \R, \label{os0w4}
\end{align}that is, \(u\) is an even function about the point \(\frac{a+b}2\). Moreover, \eqref{os0w4} implies that \(u\equiv 0\) in \(\R\setminus (a-R,b+R)\). 

Now suppose that we move \(T_\lambda\) from left to right. If \(\Omega\) is made up of at least two connected components then repeating the argument above we must also have that \(u \equiv 0\) in \((a-R,b+R)\). However, \(u\) is not identically zero, so \(\Omega\) can only have one connected component. Finally, \eqref{8e2sD} implies that \(u\) is strictly monotone in \((\frac{a+b} 2 , b+R)\) which further implies the positivity of \(u\).
\end{proof}

\section{Appendix A: Technical lemmas} \label{ltalz}

In this appendix, we list several lemmas that are used throughout the paper.

\begin{lem} \label{IyXUO}
Let \(I \subset \R\) and \(\Omega \subset \R^n\) be open and bounded sets in \(\R\) and \( \R^n\) respectively. Suppose that \(\phi_t(x): \overline{\Omega}\times \overline{I} \to \R\) is continuous in \(  \overline{\Omega} \times \overline{I} \) and define \begin{align*}
m(t) = \min_{x \in \overline{\Omega}} \phi_t(x) . 
\end{align*} Then \(m \in C(\overline{I})\). 
\end{lem}

\begin{proof}
Since \(  \overline{\Omega} \times  \overline{I}\) is compact, \(\phi\) is uniformly continuous. Fix \(\varepsilon >0\) and let \((\bar{x},\bar{t}) \in \overline{\Omega} \times  \overline{I}\) be arbitrary. There exists some \(\delta>0\) (indepedent of \(\bar{t}\) and \(\bar{x}\)) such that if \((x,t) \in \overline{\Omega} \times  \overline{I}\) and \( \abs{(x,t)-(\bar{x},\bar{t})} < \delta \) then \begin{align*}
\abs{\phi_t(x)-\phi_{\bar{t}}(\bar{x})} < \varepsilon .
\end{align*} In particular, we may take \(x=\bar{x}\) to conclude that if \(t \in \overline{I}\) and \(\abs{t-\bar{t}} < \delta\) then \begin{align*}
\abs{\phi_t(\bar{x})-\phi_{\bar{t}}(\bar{x})} < \varepsilon 
\end{align*} for any \(\bar{x} \in \overline{\Omega}\). Consequently, \begin{align*}
\phi_t(\bar{x}) > \phi_{\bar{t}}(\bar{x}) - \varepsilon \geqslant m(\bar{t}) -\varepsilon \qquad \text{for all } \bar{x} \in \overline{\Omega}
\end{align*} Minimising over \(\bar{x}\) we obtain \begin{align*}
m(t) \geqslant  m(\bar{t}) -\varepsilon .
\end{align*} Similarly, we also have that \begin{align*}
m(\bar{t}) \geqslant  m(t) -\varepsilon . 
\end{align*} Thus, we have shown that if \(t \in \overline{I}\) and \(\abs{t-\bar{t} } < \delta\) then \begin{equation*}
\abs{m(t) -m(\bar{t} )} \leqslant \varepsilon . \qedhere
\end{equation*}
\end{proof}

\begin{lem} \label{oujju}
Let \(R>0\) and \(\zeta_R: \R \to \R \) be the solution to \begin{align}
\begin{PDE}
(-\Delta)^s \zeta_R &=0 &\text{in } (-R,R) ,\\
\zeta_R &= g_R &\text{in } \R \setminus (-R,R),
\end{PDE} \label{acLUt}
\end{align} where \begin{align*}
g_R(x) &= \begin{cases}
R, &\text{if }x >R ,\\
-R, &\text{if } x<-R . 
\end{cases}
\end{align*} Then \begin{align}
x \mapsto \frac{\zeta_R(x)}{x} \qquad \text{is  defined at }x=0 \label{JQM6p}
\end{align} and there exists a constant \(c_0=c_0(s)>0\) such that, as \(R\to \infty\), \begin{align*}
\frac{\zeta_R(x)}{x} \to c_0  \qquad \text{in } C_{\textrm{loc}} (\R).
\end{align*} 
\end{lem}

\begin{proof} By the scale-invariance property of the fractional Laplacian, we may write \begin{align}
\zeta_R(x) &= R \zeta_1(x/R) \label{DkRUV}
\end{align} where \(\zeta_1\) is the solution to \eqref{acLUt} with \(R=1\). 
The function \(\zeta_1\) is given explicitly via the Poisson kernel, (for more details see \cite[Section 15]{MR3916700}) \begin{align*}
\zeta_1 (x) &= a_s\big ( 1 -x^2 \big)^s \int_{\R \setminus (-1,1)} \frac{g_1(t)}{(t^2-1)^s\vert x-t \vert} \dd t, \qquad x \in (-1,1)  .
\end{align*} where \(a_s\) is a positive normalisation constant. Using the definition of \(g_1\) and a change of variables, we obtain  \begin{align}
\zeta_1( x) = 2a_s x\big (1-x^2 \big )^s \int_1^\infty \frac{\dd t }{(t^2-x^2)(t^2-1)^s} . \label{XLPcN}
\end{align} Provided \(\vert x\vert <1/2\), \begin{align*}
\frac{1}{(t^2-x^2)(t^2-1)^s} \leqslant \frac 1{(t^2-(1/2)^2)(t^2-1)^s} \in L^1((1,\infty))
\end{align*} so by \eqref{XLPcN} and the dominated convergence theorem \begin{align*}
\lim_{x\to 0} \frac{\zeta_R(x)}{x} &= R\lim_{x\to 0} \frac{\zeta_1(x/R)}{x}= 2a_s  \int_1^\infty \frac{\dd t }{t^2(t^2-1)^s} =:c_0.
\end{align*} This proves \eqref{JQM6p}. 

Moreover, \begin{align}
\bigg \vert \frac{\zeta_R(x)}{x} -  c_0\bigg \vert &= 2a_s\bigg \vert    \int_1^\infty \bigg ( \frac {\big (1-(x/R)^2 \big )^s} {t^2-(x/R)^2} - \frac 1 {t^2} \bigg ) \frac{\dd t }{(t^2-1)^s} \bigg \vert \nonumber \\
&\leqslant 2a_s   \int_1^\infty \bigg \vert  \frac {t^2\big ( (1-(x/R)^2  )^s-1\big ) +(x/R)^2 \big )} {t^2\big (t^2-(x/R)^2 \big )}  \bigg \vert \frac{\dd t }{(t^2-1)^s} \nonumber \\
&\leqslant 2a_s \int_1^\infty \bigg \vert \frac{\big (1-(x/R)^2\big )^s -1}{t^2-(x/R)^2} \bigg \vert \frac{\dd t }{(t^2-1)^s} \nonumber\\
&\qquad + \frac{2a_s\vert x \vert^2}{R^2}\int_1^\infty \frac{\dd t}{t^2 \vert t^2-(x/R)^2 \vert (t^2-1)^s} . \label{ZRXKr}
\end{align} Suppose that \(x \in  \Omega' \subset \subset \R\). Then for \(R>0\) sufficiently large, \begin{align}
\frac 1 {\vert t^2-(x/R)^2  \vert } \leqslant C \qquad \text{for all } t>1 \label{RZqcI}
\end{align} and, by Bernoulli's inequality, \begin{align}
\vert \big (1-(x/R)^2\big )^s -1 \vert &\leqslant \frac C {R^2}. \label{zHcKY}
\end{align} Combining \eqref{ZRXKr}, \eqref{RZqcI}, and \eqref{zHcKY}, we conclude that \begin{align*}
\bigg \vert \frac{\zeta_R(x)}{x} - c_0\bigg \vert &\leqslant \frac C {R^2}\bigg ( \int_1^\infty  \frac{\dd t }{(t^2-1)^s}
+ \int_1^\infty \frac{\dd t}{t^2(t^2-1)^s} \bigg ) \to 0
\end{align*} as \(R \to \infty\).
\end{proof}

\begin{prop} \label{ayQzh}
Suppose that \(\Omega\) is a bounded open subset of \(\R^n\) and that \(\lambda(\Omega)\) is a Dirichlet eigenvalue of \((-\Delta)^s\) in \(\Omega\). Then \begin{align*}
\lambda(\Omega)\geqslant \frac{n}{2s} \vert B_1 \vert^{1+2s/n} c_{n,s}\vert \Omega \vert^{- \frac{2s}n } .
\end{align*} where \(c_{n,s}\) is the defined in \eqref{tK7sb}.
\end{prop}

The result in Proposition \ref{ayQzh} is not new (see \cite{MR3824213,MR3063552}); however, to the authors' knowledge, the proof is original and simple.

\begin{proof}[Proof of Proposition \ref{ayQzh}]
Let \(u \in L^2(\Omega)\) satisfy \begin{align*}
\begin{PDE}
(-\Delta)^s u &= \lambda(\Omega) u &\text{in } \Omega, \\
u&= 0 &\text{in }\R^n \setminus \Omega, \\
\| u \|_{L^2(\Omega)}&=1. 
\end{PDE}
\end{align*} The existence of such a function can be proved via semi-group theory, see for example \cite[Chapter 4]{MR2569321}.

Using the integration by parts formula for the fractional Laplacian, see \cite[Lemma 3.3]{Dipierro2017Neum}, we have
\begin{align*}
\lambda(\Omega) = \frac{c_{n,s}}{2} \int_{\R^{2n} \setminus (\Omega^c)^2} \frac{\vert u(x)-u(y) \vert^2}{\vert x- y \vert^{n+2s}} \dd y \dd x \geqslant c_{n,s} \int_\Omega \int_{ \R^n \setminus \Omega } \frac{\vert u(x)\vert^2}{\vert x- y \vert^{n+2s}} \dd y \dd x. 
\end{align*}For \(x\in \Omega\), if \(B_r(x)\) is the ball such that \(\vert B_r(x)\vert = \vert \Omega\vert \) then by~\cite[Lemma 6.1]{MR2944369}, \begin{align*}
\int_{ \R^n \setminus \Omega } \frac{ \dd y }{\vert x- y \vert^{n+2s}} \geqslant \int_{ \R^n \setminus B_r(x)} \frac{ \dd y }{\vert x- y \vert^{n+2s}} =\frac{n}{2s} \vert B_1 \vert^{1+2s/n} \vert \Omega\vert^{-\frac{2s}{n}}.
\end{align*} Hence, \begin{equation*}
\lambda(\Omega) \geqslant \frac{n}{2s} \vert B_1 \vert^{1+2s/n}c_{n,s}\vert \Omega\vert^{-\frac{2s}{n}} \| u \|_{L^2(\Omega)}^2 =  \frac{n}{2s} \vert B_1 \vert^{1+2s/n}c_{n,s} \vert \Omega\vert^{-\frac{2s}{n}}.\qedhere
\end{equation*}
\end{proof}

\section*{Acknowledgments}

All the authors are members of AustMS. SD is
supported by the Australian Research Council DECRA DE180100957 “PDEs, free boundaries and
applications”.
GP is supported by the Australian Research Council (ARC) Discovery Early Career Researcher Award (DECRA) DE230100954 “Partial Differential Equations: geometric aspects and applications”, and is member of INdAM/GNAMPA.
GP and EV are supported by the Australian Laureate Fellowship FL190100081 “Minimal surfaces, free boundaries and partial differential equations”.
JT is supported by an Australian Government Research Training Program Scholarship.

%% file: Part1/Parallel-surface-stability-thesis.tex
\chapter{Quantitative stability for overdetermined nonlocal problems with parallel surfaces and investigation of the stability exponents} \label{Fw81drHU}

In this article, we analyze the stability of the parallel surface problem for semilinear equations driven by the fractional Laplacian. 
We prove a quantitative stability result that goes beyond the one previously obtained in \cite{MR4577340}. 

Moreover, we discuss in detail several techniques and challenges in obtaining the optimal exponent in this stability result. In particular, this includes an upper bound on the exponent via an explicit computation involving a family of ellipsoids. 
We also sharply investigate a technique that was proposed in \cite{MR3836150} to obtain the optimal stability exponent in the quantitative estimate for the nonlocal Alexandrov's soap bubble theorem,
obtaining accurate estimates to be compared with a new, explicit example.

\section{Introduction and main results} \label{tH4AETfY}

\subsection{The long-standing tradition of overdetermined problems}
Overdetermined problems are a broad class of partial differential equations (\textrm{PDEs}) where `too many conditions' are imposed on the solution. Since not every region will admit a solution which satisfies all the conditions, the objective in the study of overdetermined problems is to classify those regions which do admit solutions. 

Often, overdetermined problems arise naturally in applications as a combination of a well-posed partial differential equation (\textsc{PDE}) which describes the dynamics of a given physical system, as well as an extra condition, often referred to as the overdetermined condition, which describes a property or an optimal quality you would like the solution to possess. As such, overdetermined problems have a close relationship with optimization, free boundary problems and calculus of variations, particularly shape optimization, as well as many other applications including fluid mechanics, solid mechanics, thermodynamics, and electrostatics, see \cite{MR333220,MR2436831,MR3791463,MR4230553}.

The study of overdetermined problems began in the early 1970's with the celebrated paper of Serrin~\cite{MR333220}. In this influential paper, Serrin proved that, given a bounded domain \(\Omega \subset \R^n\) with~\(C^2\) boundary and \(f\in C^{0,1}_{\text{loc}}(\R)\), if there exists a positive solution \(u\in C^2(\overline \Omega )\) that satisfies the Dirichlet boundary value problem \begin{align}
\begin{PDE} 
-\Delta u &= f(u) &\text{in } \Omega, \\
u&=0 &\text{on } \partial \Omega,
\end{PDE} \label{570WVdsY}
\end{align} as well as the overdetermined condition \begin{align}
\partial_\nu u &= c \qquad  \text{on } \partial \Omega \label{gdBJQCnF}
\end{align} for some constant \(c\in \R \setminus \{0\}\), then~\(\Omega\) must be a ball. Here~\(\nu\) is the unit outward pointing normal to~\(\partial\Omega\) and~\(\partial_\nu u = \nabla u \cdot \nu\).
The proof relies on a powerful technique, now known as the \emph{method of moving planes}, a refinement of a reflection principle conceived by Alexandrov in \cite{MR0150710} to prove the so-called \emph{soap bubble theorem}, which states that the only connected closed hypersurfaces embedded in a space form with everywhere constant mean curvature are spheres.

Since the work of Alexandrov and Serrin, the analysis of overdetermined problems and the method of moving planes has seen intense research activity. Some of the literature on the method of moving planes, overdetermined problems, and symmetry in \textsc{PDE} include: \begin{itemize}
\item \emph{Alternate/incomplete overdetermined conditions:} One can consider elliptic equations such as~\eqref{570WVdsY} or otherwise with an alternate overdetermined condition to~\eqref{gdBJQCnF},
see~\cite{MR1616562, MR2436831, MR3231971, MR4230553}. The parallel surface problem fits into this category, see below for references. 
\item \emph{Integral identities:} Integral identities such as the Pohozaev identity are often employed in the analysis of overdetermined problems, providing an alternative approach to the method of moving planes. Such an alternative approach was pioneered in \cite{MR333221} and then further developed in, e.g., \cite{payne1989duality, MR980297, MR2448319, cianchi2009overdetermined, MR3663321, MR3959271, MR4054869, MR4124125, CavallinaPoggesi2024}. 
Also, it was successfully employed to study problems in the mixed boundary value setting, see \cite{MR2776063,MR4109828,MR4054862,MR4053600,PoggesiSB2022,MR4010636,MR4682805}.
\item \emph{Nonlinear elliptic equations:} The original paper of Serrin \cite{MR333220} dealt with uniformly elliptic quasilinear equations including equations of mean curvature type. Other (possibly degenerate) quasilinear and fully nonlinear equations have been analyzed \cite{MR980297, MR2293958, MR2366129, MR2448319, MR2764863, MR3040677, MR3385189, AlessandriniGarofalo1989symmetry, reichel1996radial, MR1674355, MR3977217}. 
\item \emph{Symmetry:} The method of moving planes was famously used to prove symmetry results for solutions to semilinear \textsc{PDE} in domains with symmetry in \cite{MR544879, MR634248} (see also~\cite{MR1159383, MR1190345, MR1025886}). For this type of results, an alternative approach which combines the Pohozaev identity and isoperimetric-type inequalities was pioneered in \cite{MR653200} and then further developed in \cite{MR1382205, MR3003296, MR4380032}. 
\item \emph{Nonlocal operators:} For overdetermined problems and symmetry for nonlocal equations, see \cite{MR3395749,MR3881478,MR3836150,MR3937999,RoleAntisym2022,MR4577340,MR3189604,MR3827344}.
\end{itemize}

We also mention that various symmetry-breaking results were established and can be found e.g. in~\cite{MR4046014,MR3954957,MR4451480,Ruiz2022,CavallinaPoggesi2024,CiraoloPacellaPolvara2023}.

\subsection{The nonlocal parallel surface problem and the
main results of this paper}
In this paper, we are concerned with the nonlocal parallel surface problem. The context of this problem is as follows: Let \(n\) be a positive integer and \(s\in (0,1)\). Suppose that \(G\) is an open bounded subset of~\(\R^n\) and 
let~\(\Omega  = G+B_R\) for some~\(R>0\), where~\(A+B\) is the Minkowski sum of sets defined by\footnote{Here and throughout this article the abbreviation `\(\textrm{s.t.}\)' always stands for `such that'. Though this notation is unconventional, we believe it leads to less ambiguity compared to the two standard symbols for `such that': the colon (:) and the vertical line (\(\vert\)).} \begin{align*}
A+B = \{ a+b \text{ s.t. } a\in A, b \in B\}. 
\end{align*}  Moreover, let \((-\Delta)^s\) denote the fractional Laplacian defined by \begin{align*}
(-\Delta)^s u(x) &= c_{n,s}\PV \int_{\R^n} \frac{u(x)-u(y)}{\vert x - y \vert^{n+2s}} \dd y 
\end{align*} where \(c_{n,s}\) is a positive normalization constant and \(\PV\) denotes the Cauchy principle value. Also, assume that \(f:\R \to \R\) is locally Lipschitz and satisfies \(f(0)\geqslant 0\).  Then, the nonlocal parallel surface problem asks: if there exists a function \(u\) that satisfies (in an appropriate sense) the equation \begin{align}
\begin{PDE}\label{problem00}
(-\Delta)^s u &=f(u) &\text{in } \Omega ,\\
u&=0 &\text{in } \R^n \setminus \Omega, \\
u&\geqslant 0 &\text{in } \R^n ,
\end{PDE}
\end{align} as well as the overdetermined condition \begin{align}
u = \text{constant} \qquad \text{on } \partial G, \label{k2Zr4V52}
\end{align} then is \(\Omega\) necessarily a ball? 

This question is referred to as the rigidity problem for the parallel surface problem. Furthermore, one can ask about the stability of this problem, that is, heuristically, if \(u\) satisfies~\eqref{problem00} and `almost' satisfies~\eqref{k2Zr4V52} then is \(\Omega\) `almost' a ball? This is the subject of the current article. Of course, one must be precise by what one means by `almost'---this will be made clear in the proceeding paragraphs as we describe the literature and our main results. 

The local analogue of the parallel surface problem (i.e. the case \(s=1\)) was introduced in \cite{MR2629887} in the context of invariant isothermic surfaces of a nonlinear non-degenerate fast diffusion equation, and in \cite{MR2916825} as a discrete analogue of the original Serrin's problem. Moreover, it also appears in~\cite{MR3420522}.

To see why the parallel surface problem can be
viewed as a discrete analogue of Serrin's problem, consider that~\(u\in C^2(\overline \Omega)\) satisfies~\eqref{570WVdsY} as well as~\(u= c_k>0\) on a countably infinite family of parallel surfaces~\(\Gamma_k\) that are a distance~\(1/k\) from the boundary of~\(\Omega\). Then~\(kc_k\) necessarily converges to some~\(c\) by regularity assumptions on~\(u\) and~\(u\)  satisfies~\eqref{gdBJQCnF}. Consequently, \(\Omega\) must be a ball.

To reiterate, in the parallel surface problem there is only a single parallel surface (not a family as just described); regardless, this was enough to prove in \cite{MR2916825,MR3420522} the rigidity result: if there exists a solution to~\eqref{570WVdsY} satisfying~\eqref{k2Zr4V52} then \(\Omega\) must be a ball. 

Subsequently, in \cite{MR3481178}, the stability of~\eqref{problem00} for the local problem was addressed. In that article, the authors used the shape functional \begin{align*}
\rho(\Omega) = \inf \big\{  R-r \text{ s.t. } B_r(x) \subset \Omega \subset B_R(x) \text{ for some } x\in \Omega \big\} 
\end{align*} to quantify how close \(\Omega\) is to being a ball and the semi-norm \begin{align*}
[u]_{\partial G} = \sup_{\substack{x,y\in \partial G \\ x\neq y}} \bigg \{ \frac{\vert u(x)-u(y) \vert}{\vert x - y \vert } \bigg \} 
\end{align*}
to quantify how close \(u\) is to being constant on \(\partial G \). They showed, under reasonable assumptions on~\(\Omega\), that \begin{align}
\rho(\Omega) \leqslant C [u]_{\partial G} . \label{dKQORoXN}
\end{align}

Moreover, for the nonlocal parallel surface problem, the rigidity problem was answered affirmatively in \cite{MR4577340} for the case \(f=1\) and \cite{RoleAntisym2022} for general \(f\). Moreover, \cite{MR4577340} also addressed the stability problem (still for the case \(f=1\)) and showed that \begin{align}
\rho(\Omega) \leqslant C [u]_{\partial G}^{\frac 1 {s+2}} . \label{svfQTEHq}
\end{align} It is interesting to observe that~\eqref{svfQTEHq} is sub-optimal in the sense that it does not recover the estimate~\eqref{dKQORoXN} when \(s=1\). This is due to the nonlocality of the fractional Laplacian which caused contributions from mass `far away' to have a significant effect on the analysis. 

This leads us to our main result. 
%
%
\begin{thm} \thlabel{oAZAv7vy}
Let \(G\) be an open bounded subset of~\( \R^n\) and \(\Omega = G+B_R\) for some \(R>0\). Furthermore, let \(\Omega\) and \(G\) have \(C^1\) boundary,
%
%
and let~\(f \in C^{0,1}_{\mathrm{loc}}(\R)\) be such that \(f(0)\geqslant 0\). Suppose that \(u\) satisfies~\eqref{problem00} in the weak sense. 

Then, \begin{align}
\rho(\Omega) &\leqslant C [u]_{\partial G}^{\frac 1 {s+2}} \label{srHxBYlS}
\end{align} with  \begin{align*}
&C:= \\ &C(n,s) \bigg [ \frac{(\max\{1, \diam \Omega \} )^{n+2s}\big(1+(\diam \Omega)^{2s} [f]_{C^{0,1}(\overline \Omega \times [0,\|u\|_{L^\infty(\Omega)}])} \big)(\diam \Omega)^{n+2s+2} } {R_0^{2s}R^s\big( f(0)+\|u\|_{L_s(\R^n)}\big)}  \\
&\hspace{20em} + (\diam \Omega)^{n-1} + \frac{\vert \Omega \vert} R  \bigg) \bigg ] \frac{(\diam \Omega)}{\vert \Omega \vert }
\end{align*} and \begin{align}
R_0=\min \big \{ R, [f]_{C^{0,1}([0,\|u\|_{L^\infty(\Omega)}])}^{-\frac 1{2s} } \big \} \label{K57sa7Sk}
\end{align}


\end{thm}

For the precise definition of a solution satisfying~\eqref{problem00} in the weak sense, see Section~\ref{xxtYVHnv}. 

\thref{oAZAv7vy} is a direct extension of \cite{MR4577340} to the case \(f \neq 1\). Moreover, \thref{oAZAv7vy} relaxes some of the regularity assumptions on \(\Omega\) appearing in \cite{MR4577340}.
It is clear that the dependence on the volume $| \Omega |$ in the constant of Theorem \ref{oAZAv7vy}  can be
removed by means of the following bounds
$$
|B_1| R^n \le | \Omega | \le |B_1| ( \diam \Omega)^n
,
\text{ where } | B_1 | \text{ is the volume of the unit ball in } \R^n
,
$$
which hold true in light of the monotonicity of the volume with respect to inclusion.

Currently, an important open problem for the nonlocal parallel surface problem is to understand the optimality of \thref{oAZAv7vy}. Indeed, let \(\overline\beta(s)\) be the optimal exponent in~\eqref{srHxBYlS} defined as the supremum over \(\beta \in \R\) such that if \(u=u_f\) is a weak solution of~\eqref{problem00} for some \(f\) satisfying the assumptions of \thref{oAZAv7vy} then \(\rho(\Omega) \leqslant C [u]_{\partial G}^{\beta}\). In this framework, an explicit expression for~\(\overline\beta\) as a function of~\(s\) is still unknown. \thref{oAZAv7vy} implies that~\(\overline\beta(s) \geqslant \frac 1 {s+2}\) and \cite{MR3481178} establishes that~\(\overline\beta(1)=1\),
therefore we believe it is an interesting problem to detect optimal stability exponents in the nonlocal setting, also to recover, whenever possible, the classical exponent of the local cases in the limit.

In Section~\ref{z1jzKMXt}, we explicitly construct a family of domains \(G_\varepsilon\) that are small perturbations of a ball and corresponding solutions \(u_\varepsilon\) to \eqref{problem00} when \(f=1\) which satisfy \([u_\varepsilon]_{\partial G_\varepsilon} \simeq C \rho(\Omega_\varepsilon)\) as \(\varepsilon \to 0^+\). This entails the following result: 

\begin{thm} \thlabel{O4zCF33o}
Let \(G\) be an open bounded subset of \(\R^n\) and \(\Omega = G+B_R\) for some \(R>0\). Furthermore, let \(\Omega\)
and \(G\) have \(C^1\) boundary.

Then, we have that~\(\overline\beta(s) \leqslant 1\). 
\end{thm}
 
As far as the authors are aware, these are the only known estimates for~\(\overline\beta(s)\). Furthermore, by considering the case in which~\(\Omega = \Omega_\varepsilon\) is a small perturbation of a ball and exploting
interior regularity for the fractional Laplacian, one would expect that~\(\rho(\Omega_\varepsilon) \simeq [u_\varepsilon]_{\partial G_\varepsilon} \simeq \varepsilon\) (up to a constant) as~\(\varepsilon \to 0^+\) which suggests that~\(\overline\beta(s)=1\) for all~\(s\in (0,1]\).  In Section~\ref{fEsBEcuv}, we give a broad discussion on some of the challenges that
the nonlocality of the fractional Laplacian presents in obtaining this result. In particular, by way of an example via the Poisson representation formula, we show that estimates for a singular integral involving the reciprocal of the distance to the boundary function play a key role in obtaining the anticipated optimal result. This suggests, surprisingly, that fine geometric estimates for the distance function close to the boundary are required to obtain the optimal exponent. 

\subsection{Scrutiny of the stability exponent} In the recent literature,
some inventive methods have been introduced to improve
the stability results obtained via the moving plane method.
A remarkable one was put forth in~\cite{MR3836150}.
Roughly speaking, the setting considered in~\cite[Proposition~3.1(b)]{MR3836150} focused
on the critical hyperplane~$\pi_\lambda=\{x_n=\lambda\}$ for the moving
plane technique for a set~$\Omega$ which contains the ball~$B_r$ and is contained in the ball~$B_R$, looked at the symmetric difference between $\Omega$ and its reflection~$\Omega'$ and obtained a bound on the measure of the set\begin{equation}\label{pijwdled-2}
\big\{ x\in \Omega \triangle \Omega ' \text{ s.t. } \dist(x,\pi_\lambda) \leqslant \gamma \big\} \end{equation}
which was linear in both~$\gamma$ and~$R-r$.

This constituted a fundamental ingredient in~\cite{MR3836150} to achieve an optimal stability exponent.

Unfortunately, we believe that, in the very broad generality in which the result is stated in~\cite{MR3836150}, this statement may not be true, and we present an explicit counter-example.

Nevertheless, in our opinion, a weaker version of~\cite[Proposition 3.1(b)]{MR3836150} does hold true. We state and prove this new result and check its optimality against the counter-example mentioned above. This construction plays a decisive role in improving the stability exponent.

Though we refer the reader to Section~\ref{c8w8u7Hn} for full details
of this strategy, let us anticipate here some important details.

More specifically, in the forthcoming Theorem~\ref{m9q9X2hE}, under the additional assumption that the set~$\Omega$ is of class~$C^\alpha$, with~$\alpha>1$, we will bound the measure of the set in~\eqref{pijwdled-2}, up to constants,
by\begin{equation}\label{piqfk0uwefhiovwegui}
\gamma (R-r)^{1-\frac 1 \alpha }.\end{equation}
We stress that this bound formally recovers exactly the one stated in~\cite[Proposition 3.1(b)]{MR3836150} when~$\alpha=+\infty$.

However, we believe that the bound in~\eqref{piqfk0uwefhiovwegui} is optimal and cannot be improved (in particular, while the dependence in~$\gamma$ of the estimate above is linear, the dependence in~$R-r$, surprisingly, is not!). This will be shown by an explicit counter-example put forth in Theorem~\ref{example61}---while the analytical details
of this counter-example are very delicate, the foundational idea behind
it is sketched in Figure~\ref{Fig1} (roughly speaking, the example
is obtained by a very small and localized modification of a ball  to induce
a critical situation at the maximal location allowed by the regularity of the set; the construction is technically demanding since the constraint for the set of ``being between two balls'' induces two different scales, in the horizontal and in the vertical directions, which in principle could provide different contributions).

As a consequence of the bound~\eqref{piqfk0uwefhiovwegui} for the measure of the set in~\eqref{pijwdled-2}, as given in \thref{m9q9X2hE}, the stability exponent~$1/(s+2)$ of \thref{oAZAv7vy} can be improved to~$\alpha/( 1 + \alpha (s+1))$ provided that~$\Omega$ is of class~$C^\alpha$ for~$\alpha >1$: we refer to Section~\ref{subsec:new improvement exponent} and \thref{thm:improvement} for details.

\subsection{Organization of paper}

The paper is organized as follows. In Section~\ref{xxtYVHnv}, we summarize the notation and basic definitions used throughout the article. In Section~\ref{S7DGIjUf}, we give several quantitative maximum principles---in both the non-antisymmetric and antisymmetric situations---that are required in the proof of \thref{oAZAv7vy} and, in Section~\ref{sec:stabest}, we give the proof of \thref{oAZAv7vy}.

The remaining sections are broadly focused on the techniques and challenges in obtaining the optimal stability exponent.
In Section~\ref{fEsBEcuv}, we discuss the surprising role that
fine geometric estimates for the distance function close to the boundary play in the attainment of the optimal exponent. In Section~\ref{c8w8u7Hn}, we accurately discuss the possibility of obtaining the optimal exponent
and comment about some criticalities in the existing literature,
and, in Section~\ref{z1jzKMXt}, we construct an explicit family of solutions which implies \thref{O4zCF33o}.

\section*{Acknowledgements} 

All the authors are members of AustMS.
GP is supported by
the Australian Research Council (ARC) Discovery Early Career Researcher Award (DECRA) DE230100954 ``Partial Differential Equations: geometric aspects and applications'' and is member of the Gruppo Nazionale Analisi Matematica Probabilit\`a e Applicazioni (GNAMPA) of the Istituto Nazionale di Alta Matematica (INdAM).
JT is supported by an Australian Government Research Training Program Scholarship.
EV is supported by the Australian Laureate Fellowship FL190100081
``Minimal surfaces, free boundaries and partial differential equations''.

\section{Preliminaries and notation} \label{xxtYVHnv}

In this section, we fix the notation that
we will use throughout the article and give some relevant definitions. Let~\(n\geqslant 1\), ~\(s\in (0,1)\). The fractional Sobolev space~\(H^s(\R^n)\) is defined as \begin{align*}
H^s(\R^n)  = \big\{ u \in L^2(\R^n) \text{ such that } [u]_{H^s(\R^n)} <+\infty \big\} 
\end{align*} where \([\cdot ]_{H^s(\R^n)}\) is the Gagliardo semi-norm given by \begin{align*}
[u ]_{H^s(\R^n)} = \frac{c_{n,s}} 2 \int_{\R^n} \int_{\R^n} \frac{\vert u(x) - u(y) \vert^2}{\vert x-y\vert^{n+2s}} \dd y \dd x 
\end{align*} and \(c_{n,s}\) is the same constant appearing in the definition of the fractional Laplacian. Moreover, given an open, bounded set~\(\Omega \subset \R^n\), define \begin{align*}
\mathcal Q(\Omega) = \big ( \R^n \times \R^n \big ) \setminus \big ( \Omega^c \times \Omega^c\big )
\end{align*} with \(\Omega^c = \R^n \setminus \Omega\) and \begin{align*}
\mathcal E (u,v)= \frac{c_{n,s}}{2} \iint_{\mathcal Q(\Omega) }\frac{( u(x) - u(y) )(v(x)- v(y)) }{\vert x-y\vert^{n+2s}} \dd y \dd x .
\end{align*} 

Now, let \(c\in L^\infty(\Omega)\) and \(g\in L^2(\Omega)\). We say that a function~\(u\in L^1_{\mathrm{loc} } (\R^n) \) such that \(\mathcal E(u,u)<+\infty\) satisfies~\((-\Delta)^s u + c  u \geqslant g\) (respectively, \(\leqslant\)) in the \emph{weak sense} if \begin{align*}
\mathcal E(u,v) + \int_\Omega c uv \dd x \geqslant \int_\Omega g v \dd x  \qquad (\text{respectively, } \leqslant )
\end{align*} for all \(v\in H^s(\R^n)\) with \(v\geqslant 0\) in \(\R^n\) and \(v=0\) in \(\R^n\setminus \Omega\). In this case, we also say that~\(u\) is a supersolution (respectively,
subsolution) of~\((-\Delta)^s u + c  u = g\) in the weak sense. Moreover, we say a function~\(u\in L^1_{\mathrm{loc} } (\R^n) \) such that \(\mathcal E(u,u)<+\infty\) satisfies~\((-\Delta)^s u + c  u = g\) in the \emph{weak sense} if \begin{align*}
\mathcal E(u,v) + \int_\Omega c uv \dd x = \int_\Omega g v \dd x
\end{align*} for all \(v\in H^s(\R^n)\) with \(v=0\) in \(\R^n\setminus \Omega\). This is equivalent to~\(u\) being both a weak supersolution and a weak subsolution. 

We also define the following weighted \(L^1\) norm via \begin{align*}
\| u\|_{L_s(\R^n)}  = \int_{\R^n} \frac{\vert u(x) \vert}{1+\vert x \vert^{n+2s} } \dd x
\end{align*} and the space \(L_s(\R^n)\) as \begin{align*}
L_s(\R^n)  = \big\{ L^1_{\mathrm{loc}} (\R^n) \text{ such that } \| u\|_{L_s(\R^n)} <+\infty \big\}.
\end{align*}

\medskip

Next, we describe some notation regarding antisymmetric functions. This is closely related to the notation of the method of moving planes; however, we will defer our explanation of the method of moving planes until Section~\ref{sec:stabest} in the interests of simplicity. 

A function \(v: \R^n \to \R\) is said to be \emph{antisymmetric} with respect to a plane \(T\) if \begin{align*}
v(Q_T(x)) = -v(x) \qquad \text{for all } x\in \R^n
\end{align*} where \(Q_T : \R^n \to \R^n\) is the function that reflects~\(x\) across~\(T\). Often it will suffice to consider the case~\(T=\{x_1=0\}\), in which case~\(Q_T(x) = x-2x_1e_1\). 

For simplicity, we will refer to~\(v\) as \emph{antisymmetric} if it is antisymmetric with respect to the plane~\(\{x_1=0\}\). Moreover, we sometimes write~\(x'\) to denote~\(Q_T(x)\) when it is clear from context what~\(T\) is.

\medskip

Finally, some other notation we will employ through the article is:
given a set~\(A\subset \R^n\), the function~\(\chi_A : \R^n \to \R \) denotes the characteristic function of~\(A\), given by \begin{align*}
\chi_A(x) &= \begin{cases}
1, &\text{if } x\in A ,\\
0, &\text{if } x\not \in A,
\end{cases}
\end{align*} and \(\delta_A : \R^n \to [0,+\infty]\) denotes the distance function to \(A\), given by \begin{align*}
\delta_A(x) = \inf_{y \in A} \vert x-y\vert. 
\end{align*}

We will denote by~$\R^n_+:=\{x=(x_1,\dots,x_n)\in\R^n\;{\mbox{s.t.}}\; x_1>0\}$ and, given~$A\subseteq\R^n$, we will denote by~$A^+:=
A\cap\R^n_+$.

Moreover, if \(A\) is open and bounded with sufficiently regular boundary (we only use the case in which~\(A\) has smooth boundary), we will denote by \(\psi_A\) the (unique) function in \(C^\infty(A) \cap C^s(\R^n)\) such that \(\psi_A\) satisfies \begin{align} \label{bNuXHq1g}
\begin{PDE}
(-\Delta)^s \psi_A &=1 &\text{in } A ,\\
\psi_A&=0 &\text{in } \R^n \setminus A. 
\end{PDE}
\end{align}
When \(A=B_r(x_0)\), \(\psi_A\) is known explicitly and is given by \begin{align}
\psi_{B_r(x_0)}(x) &= C \big ( r^2 - \vert x- x_0\vert^2 \big )^s_+, \qquad \text{for all } x\in \R^n \label{v2FT4ISn}
\end{align} with \(C\) a positive normalisation constant depending only on \(n\) and \(s\). Also, we will use \(\lambda_1(A)\) to denote the first Dirichlet eigenvalue of the fractional Laplacian. 

Throughout this article, positive constants will be denoted by \(C\) or \(C(a_1,\dots,a_k)\) when we want to emphasise that \(C\) depends only on \(a_1,\dots,a_k\). Moreover, the explicit value of \(C\) or \(C(a_1,\dots,a_k)\) may change from line to line.

\section{Quantitative maximum principles up to the boundary} \label{S7DGIjUf}

The basic idea of the proof of \thref{oAZAv7vy} is to apply the method of moving planes, as in the proof of the analogous rigidity result \cite[Theorem 1.4]{RoleAntisym2022}, but replace `qualitative' maximum principles, such as the strong maximum principles, with `quantitative' maximum principles, such as the Harnack inequality.

The purpose of this section is to prove several such quantitative maximum principles both in the non-antisymmetric and the antisymmetric setting. In particular, we require that these maximum principles hold up to the boundary of regions with very little boundary regularity. Moreover, we are careful to keep track of precisely how constants depend on relevant quantities since we believe that this may be useful in future analyses of the nonlocal parallel surface problem. 

The section is split into two parts: maximum principles for equations without any antisymmetry assumptions and maximum principles for equations with antisymmetry assumptions. 

\subsection{Equations without antisymmetry}

In this subsection, we will give several maximum principles for linear equations with zero-th order terms without any antisymmetry assumptions. This will culminate in a quantitative analogue of the Hopf lemma for non-negative supersolutions of general semilinear equations (see \thref{zAPw0npM} below). 

Our first result is as follows: 

\begin{prop} \thlabel{YCl8lL0I}
Let \(\Omega\) be an bounded open subset of \(\R^n\), \(c\in L^\infty(\Omega)\), and \(g\in L^2(\Omega)\) with \(g\geqslant 0\). Suppose that \(u\in L_s(\R^n)\) satisfies \begin{align*}
\begin{PDE}
(-\Delta)^s u+cu &\geqslant g &\text{in } \Omega, \\
u&\geqslant 0 &\text{in } \R^n
\end{PDE}
\end{align*} in the weak sense. 

Then,
\begin{align*}
u(x) &\geqslant  C\Big(\essinf_\Omega g + \| u\|_{L_s(\R^n)}\Big) \delta_{\partial \Omega}^{2s}(x) \qquad \text{for a.e. } x\in \Omega
\end{align*} with \(C:=C(n,s)(\max\{1, \diam \Omega \} )^{-n-2s}\big(1+(\diam \Omega)^{2s}\|c^+\|_{L^\infty(\Omega)}\big)^{-1}\).
\end{prop}

The proof of \thref{YCl8lL0I} essentially follows by, at each point \(x\in \Omega\), `touching' the solution~\(u\) from below by~\(\psi_B\) (recall the notation in~\eqref{bNuXHq1g}) where~\(B\) is the largest ball contained in~\(\Omega\) and centred at~\(x\). This is similar to the proof of \cite[Theorem 2.2]{MR4023466} where a nonlinear interior analogue of \thref{YCl8lL0I} was proven. To prove \thref{YCl8lL0I}, however, some care must be taken since the touching point may occur on the boundary of \(\Omega\) (unlike in
the interior result of~\cite[Theorem~2.2]{MR4023466}) where the PDE does not necessarily hold. Moreover, \thref{YCl8lL0I} is stated in the context of weak solutions where pointwise techniques no longer make sense, so a simple mollification argument needs to be made. We require the following lemma.

\begin{lem} \thlabel{Voueel7o}
Let \(\rho>0 \), \(c\in L^\infty(B_\rho)\), and \(g\in C^\infty_0(\R^n)\) be such that \(g\geqslant 0\) in \(B_\rho\). Suppose that~\(u\in C^\infty_0(\R^n)\) satisfies \begin{align*}
\begin{PDE}
(-\Delta)^s u +cu &\geqslant g &\text{in } B_\rho, \\
u&\geqslant 0 &\text{in } \R^n .
\end{PDE}
\end{align*}

Then, \begin{align*}
u(x) \geqslant C\Big(\inf_{B_\rho} g + \| u\|_{L_s(\R^n)}\Big)\psi_{B_\rho}(x) \qquad \text{for all } x\in B_\rho
\end{align*} with \(C:=C(n,s)(\max\{1, \rho\} )^{-n-2s}\big(1+\rho^{2s}\|c^+\|_{L^\infty(B_\rho)}\big)^{-1}\).
\end{lem}

We observe that Lemma~\ref{Voueel7o} can be seen as a quantitative version of a strong maximum principle (which can be proved directly):
in particular, it entails that if~$u(x_0)=0$ for some $x_0\in\Omega$, then~$u$ vanishes identically.

\begin{proof}[Proof of Lemma~\ref{Voueel7o}]
We will begin by proving the case \(\rho=1\). Fix \(\varepsilon>0\) small and let \(\tau \geqslant 0\) be the largest value such that \(u\geqslant \tau \psi_{B_{1-\varepsilon}}\) in \(\R^n\) (recall \(\psi_{B_{1-\varepsilon}}\) is given by~\eqref{v2FT4ISn}). Now, by the definition of \(\tau\), there exists \(x_0\in B_{1-\varepsilon}\) such that \(u(x_0)=\tau \psi_{B_{1-\varepsilon}}(x_0)\). On one hand, since \( \psi_{B_{1-\varepsilon}}(x_0) = C(n,s)(1-\varepsilon)^{2s}\leqslant C(n,s)\), \begin{align*}
(-\Delta)^s(u-\tau \psi_{B_{1-\varepsilon}}) (x_0)+c(x_0) (u-\tau \psi_{B_{1-\varepsilon}}) (x_0) &\geqslant \inf_{B_1} g  - \tau \big (  1 +c(x_0) \psi_{B_{1-\varepsilon}} (x_0) \big ) \\
&\geqslant \inf_{B_1} g -C(n,s)\big(1+ \|c^+\|_{L^\infty(B_1)}\big) \tau . 
\end{align*} On the other hand, since \(\vert x_0 - y \vert^{-n-2s} \geqslant C(n,s)(1+\vert y \vert^{n+2s} )^{-1}\), we have that  \begin{align*}
(-\Delta)^s(u-\tau\psi_{B_{1-\varepsilon}}) (x_0)+c(x_0) (u-\tau \psi_{B_{1-\varepsilon}}) (x_0) &= - \int_{\R^n} \frac{(u-\tau \psi_{B_{1-\varepsilon}}) (y)}{\vert x_0 - y \vert^{n+2s}} \dd y \\
&\leqslant  -C(n,s) \int_{\R^n} \frac{(u-\tau \psi_{B_{1-\varepsilon}}) (y)}{1+\vert y \vert^{n+2s}} \dd y \\
&\leqslant - C(n,s)  \big( \|u\|_{L_s(\R^n)} - \tau \big) . 
\end{align*} Rearranging gives \(\tau \geqslant C(n,s) \big(1+\|c^+\|_{L^\infty(B_1)}\big)^{-1}  \big(\inf_{B_1} g + \| u\|_{L_s(\R^n)}\big)\), which implies that \begin{align*}
u(x) &\geqslant  C(n,s) \big(1+\|c^+\|_{L^\infty(B_1)}\big)^{-1} \Big(\inf_{B_1} g + \|u\|_{L_s(\R^n)}\Big) \psi_{B_{1-\varepsilon}}.
\end{align*} Sending \(\varepsilon \to 0^+\) gives the result when \(\rho=1\). 

Now, for the case that \(\rho\) is not necessarily 1, let \(u_\rho(x):= u(\rho x)\). Then \begin{align*}
(-\Delta)^su_\rho + \rho^{2s} c_\rho &\geqslant \rho^{2s} g_\rho \qquad \text{in }B_1
\end{align*} where \(c_\rho(x):=c(\rho x)\) and \(g_\rho(x):=g(\rho x)\), so, for all \(x\in B_\rho\), we have that \begin{align*}
u(x) &= u_\rho(x/\rho) \\
&\geqslant C(n,s) \big(1+\rho^{2s}\|c^+_\rho\|_{L^\infty(B_1)}\big)^{-1} \Big(\rho^{2s}\inf_{B_1} g_\rho + \|u_\rho\|_{L_s(\R^n)}\Big) \psi_{B_1}(x/\rho) \\
&= C(n,s) \rho^{-2s}\big(1+\rho^{2s}\|c^+\|_{L^\infty(B_\rho)}\big)^{-1} \Big(\rho^{2s}\inf_{B_\rho} g + \|u_\rho\|_{L_s(\R^n)}\Big) \psi_{B_\rho}(x).
\end{align*} Finally, \begin{align*}
\|u_\rho\|_{L_s(\R^n)} = \int_{\R^n} \frac{\vert u(\rho x) \vert }{1+\vert x \vert ^{n+2s} } \dd x &= \rho^{2s}\int_{\R^n} \frac{ \vert u(x) \vert }{\rho^{n+2s}+\vert x \vert ^{n+2s} } \dd x \\ &\geqslant \rho^{2s}(\max\{1, \rho\} )^{-n-2s} \| u\|_{L_s(\R^n)}
\end{align*} which completes the proof. 
\end{proof}

We can now give the proof of \thref{YCl8lL0I}. 

\begin{proof}[Proof of \thref{YCl8lL0I}]
If \(u,g\in C^\infty_0(\R^n)\) then the proof follows immediately from \thref{Voueel7o}. Indeed, let \(x\in \Omega\) be arbitrary. Then, applying \thref{Voueel7o} in the ball \(B_\rho(x)\) with \(\rho := \delta_{\partial \Omega}(x)\), we obtain that \begin{align*}
u(y) &\geqslant C  \Big(\essinf_{\Omega} g + \| u\|_{L_s(\R^n)}\Big)\psi_{B_\rho}(y) \qquad \text{for all } y \in B_\rho(x) 
\end{align*} with \(C=C(n,s)(\max\{1, \diam \Omega \} )^{-n-2s}\big(1+(\diam \Omega)^{2s}\|c^+\|_{L^\infty(\Omega)}\big)^{-1}\). Substituting \(y=x\) completes the proof.

In the general case that \(u\in H^s(\R^n)\), if \(u_\varepsilon\) and \(g_\varepsilon\) are mollifications of \(u\) and \(g\) respectively then \((-\Delta)^s u_\varepsilon +\|c^+\|_{L^\infty(\Omega) } u_\varepsilon \geqslant g_\varepsilon\) in \(\Omega_\varepsilon := \Omega \cap \{ \delta_{\partial \Omega} >\varepsilon\}\). Then, applying the result for smooth compactly supported functions, we have that \(u_\varepsilon \geqslant C \delta^{2s}_{\partial\Omega_\varepsilon}\) with \(C\) as in the statement of Proposition~\ref{YCl8lL0I} (and, in particular, uniformly bounded in \(\varepsilon\)). Moreover, given \(x\in \Omega_\varepsilon\), if \(y\in \partial \Omega_\varepsilon\) is the closest point to \(x\) on \( \partial \Omega_\varepsilon\) and \(z\in \partial \Omega\) is the closest point to \(y\) on \(\partial \Omega\) then \begin{align*}
\delta_{\partial \Omega}(x) \leqslant \vert x - z\vert \leqslant \vert x-y \vert + \vert y-z\vert \leqslant \delta_{\partial\Omega_\varepsilon}(x)+\varepsilon,
\end{align*}so \(u_\varepsilon \geqslant C (\delta_{\partial\Omega}-\varepsilon)^{2s}\) in \(\Omega_\varepsilon\). Then \(u_\varepsilon \to u\) a.e., so sending \(\varepsilon \to 0^+\) gives the result. 
\end{proof}

Next, we obtain a refined version of \thref{YCl8lL0I} when \(\Omega\) satisfies the uniform interior ball condition. 
As customary, we say that a bounded domain $\Omega \subset \R^n$ satisfies the {\it uniform interior ball condition
with radius} $r_\Omega>0$ if for every point $x_0 \in \partial \Omega$ there exists a ball $B \subset \Omega$ of radius
$r_\Omega$ such that its closure intersects $\partial \Omega$ only at $x_0$.

\begin{prop} \thlabel{1kapKUbs}
Let \(\Omega\) be an open bounded subset of \(\R^n\) with $C^1$ boundary and satisfying the uniform interior ball condition with radius \(r_\Omega>0\), \(c\in L^\infty(\Omega)\), and \(g\in L^2(\Omega)\) with \(g\geqslant 0\). Suppose that \(u\in L_s(\R^n)\) satisfies \begin{align*}
\begin{PDE}
(-\Delta)^s u+cu &\geqslant g &\text{in } \Omega, \\
u&\geqslant 0 &\text{in } \R^n
\end{PDE}
\end{align*} in the weak sense. 

Then,  \begin{align*}
u(x) &\geqslant  C\Big(\essinf_\Omega g + \| u\|_{L_s(\R^n)}\Big) \delta_{\partial \Omega}^{s}(x) \qquad \text{for a.e. } x\in \Omega
\end{align*} with \(C:=C(n,s)r_\Omega^s (\max\{1, \diam \Omega \})^{-n-2s}\big(1+(\diam \Omega)^{2s}\|c^+\|_{L^\infty(\Omega)}\big)^{-1}\).
\end{prop}

\begin{proof}
By an analogous argument to the one in the proof of \thref{YCl8lL0I}, we may assume that \(u,g\in C^\infty_0(\R^n)\). Let \(x\in \Omega\) be arbitrary. If \(\delta_{\partial \Omega}(x) \geqslant r_\Omega\) then we are done by \thref{YCl8lL0I}, so let \(\delta_{\partial \Omega}(x)<r_\Omega\). If \(\bar x \) denotes the closest point to \(x\) on \(\partial \Omega\) then there exists a ball~\(B=B_{r_\Omega}(x_0)\) such that~\(x\in B\), \(B \subset \Omega\), and~\(B\) touches~\(\partial \Omega\) at \(\bar x\). Then, applying \thref{Voueel7o} in \(B\), we have that \(u(y) \geqslant C\big(\essinf_\Omega g + \| u\|_{L_s(\R^n)}\big)\psi_{B_{r_\Omega}}(y)\)
with \begin{align*}
    C=C(n,s)(\max\{1, \diam \Omega \} )^{-n-2s}\big(1+(\diam \Omega)^{2s}\|c^+\|_{L^\infty(\Omega)}\big)^{-1}.
\end{align*} Substituting \(y=x\), we obtain \begin{align*}
u(x) &\geqslant C\Big(\essinf_\Omega g + \| u\|_{L_s(\R^n)}\Big)(r_\Omega^2-\vert x-x_0\vert^2)^s \\
&= C\Big(\essinf_\Omega g + \| u\|_{L_s(\R^n)}\Big)(r_\Omega+\vert x-x_0\vert)^s(r_\Omega-\vert x-x_0\vert)^s \\
&\geqslant Cr_\Omega^s\Big(\essinf_\Omega g + \| u\|_{L_s(\R^n)}\Big)(r_\Omega-\vert x-x_0\vert)^s.
\end{align*} Observing that \(r_\Omega-\vert x-x_0\vert = \delta_{\partial \Omega}(x)\) completes the proof. 
\end{proof}

From \thref{1kapKUbs}, we immediately obtain the following corollary.

\begin{cor} \thlabel{zAPw0npM}
Let \(\Omega\) be an open bounded subset of \(\R^n\) with $C^1$ boundary and satisfying the uniform interior ball condition with radius \(r_\Omega>0\). Let \(f \in C^{0,1}_{\mathrm{loc}}(\overline \Omega \times \mathbb R)\) satisfy~\(f_0:=\inf_{x \in \Omega} f(x,0) \geqslant 0\). Suppose that~\(u \in L_s(\R^n)\) satisfies \begin{align*}
\begin{PDE}
(-\Delta)^s u &\geqslant f(x,u) &\text{in } \Omega ,\\
u&\geqslant 0 &\text{in } \R^n 
\end{PDE}
\end{align*} in the weak sense. 

Then, \begin{align*}
u(x) &\geqslant  C_\ast\big(f_0 + \| u\|_{L_s(\R^n)}\big) \delta_{\partial \Omega}^{s}(x) \qquad \text{for a.e. } x\in \Omega
\end{align*} with \(C_\ast:=C(n,s)r_\Omega^s (\max\{1, \diam \Omega \} )^{-n-2s}\big(1+(\diam \Omega)^{2s} [f]_{C^{0,1}(\overline \Omega \times [0,\|u\|_{L^\infty(\Omega)}])} \big)^{-1}\).
\end{cor}

\begin{proof}
Since \(f=f(x,z)\) is locally lipschitz, \(\partial_z f\) exists a.e., so we may define \begin{align*}
c(x):= -\int_0^1 \partial_z f (x, t u(x)) \dd t .
\end{align*} Observe that \(\| c\|_{L^\infty(\Omega)} \leqslant [f]_{C^{0,1}(\overline \Omega \times [0,\| u\|_{L^\infty(\Omega)}])}\). Then, \((-\Delta)^s u+c u \geqslant f(\cdot, 0)\) in \(\Omega\) in the weak sense, so the result follows by \thref{1kapKUbs}. 
\end{proof}

\begin{remark}
One could also obtain an analogous result to \thref{zAPw0npM}
without assuming that $\Omega$ has $C^1$ boundary and satisfies the uniform interior ball condition; this can be achieved by applying \thref{YCl8lL0I} instead of \thref{1kapKUbs}. 
\end{remark}

\subsection{Equations with antisymmetry}

The purpose of this subsection is to prove the following proposition. 

\begin{prop} \thlabel{TV1cTSyn}
Let \(H\subset \R^n\) be a halfspace, \(U \) be an open subset of \(H\), and~\(a\in H\) and~\(\rho>0\) such that~\(B_\rho(a)\cap H \subset U\). Moreover, let \(c\in L^\infty(U)\) be such that \begin{equation}\label{ed0395v04ncrhrhfwejfvwejhfvqjhkfvqejhk}
\|c^+\|_{L^\infty(U)} < \lambda_1(B_\rho(a)\cap H).\end{equation} 
Suppose that \(v\) is antisymmetric with respect to \(\partial H\) and satisfies \begin{align*}
\begin{PDE}
(-\Delta)^s v +c v &\geqslant 0 &\text{in } U  ,\\
v&\geqslant 0 & \text{in } H 
\end{PDE}
\end{align*} in the weak sense. 

If \(K\subset H\) is a non-empty open set that is disjoint from \(B_\rho(a)\) and 
$$ \inf_{{x\in K}\atop{y\in B_\rho(a)\cap H}} \vert Q_{\partial H}(x)-y\vert^{-1} \geqslant M\geqslant 0 $$ then \begin{align}
v(x) &\geqslant C \| \delta_{\partial H} v\|_{L^1(K)} \delta_{\partial H}(x) \qquad \text{for a.e. } x\in B_{\rho/2}(a)\cap H \label{tQIcIMb5}
\end{align}  with \begin{align*}
C := C(n,s) \rho^{2s} M^{n+2s+2} \Big(1+\rho^{2s}\|c^+\|_{L^\infty(U)} \Big)^{-1}.
\end{align*}
\end{prop}

\thref{TV1cTSyn} can be viewed as a quantitative version of Proposition 3.3 in \cite{MR3395749}. A similar result was also obtained in \cite{MR4577340} in the case \(c=0\). One advantage that
\thref{TV1cTSyn} has over both of the results of \cite{MR3395749} and \cite{MR4577340} is that it allows the ball~\(B_\rho(a)\) to
go right up to and, indeed, overlap the plane of symmetry \(\partial H\). To allow for this possibility, it is required to construct an antisymmetric barrier as given in \thref{gmjWA2gy} below. 

However, a disadvantage of \thref{TV1cTSyn} is that it is not a boundary estimate, in the sense that~\eqref{tQIcIMb5} holds in \(B_{\rho/2}(a) \cap H\) and not \(B_\rho(a) \cap H\), so it does not give any information up to \(\partial U\setminus \partial H\). This is also a by-product of the barrier given in \thref{gmjWA2gy}.
In~\cite{dipierro2023quantitative}, we prove an improved version of \thref{TV1cTSyn} which holds up to the boundary. This allows us to establish quantitative stability estimates for the nonlocal Serrin overdetermined problem.

\begin{remark} \label{uElLogjj}
Notice that~\eqref{ed0395v04ncrhrhfwejfvwejhfvqjhkfvqejhk} is always satisfied provided that \(\rho\) is small enough. Indeed, the Faber-Krahn inequality for the fractional Laplacian gives that \(\lambda_1(A) \geqslant C(n,s) \vert A \vert^{-\frac{2s}n}\) for all open, bounded sets \(A\subset \R^n\), see \cite{MR3395749,yildirim2013estimates} (see also~\cite{brasco2020quantitative}). Since \(\vert B_\rho(a)\cap H \vert \leqslant \frac 12 \vert B_1 \vert \rho^n\), it follows that~\eqref{ed0395v04ncrhrhfwejfvwejhfvqjhkfvqejhk} is satisfied whenever \begin{align*}
\rho < C(n,s) \|c^+\|_{L^\infty(U)}^{-\frac 1{2s}}
\end{align*} with \(C(n,s)\) chosen sufficiently small. 
\end{remark}

To prove \thref{TV1cTSyn}, we require two lemmata.

\begin{lem} \thlabel{xwrqefNE}
Let \(v \in C^\infty_0(\R^n)\) be antisymmetric with respect to \(\{x_1=0\}\). 

Then, \((-\Delta)^sv\) is a smooth function in \(\R^n\) that is antisymmetric with respect to \(\{x_1=0\}\). Furthermore, if \(w(x):=v(x)/x_1\) then \begin{align}
\vert (-\Delta)^s v(x)\vert &\leqslant C \Big( \|w \|_{L^\infty(\R^n)}+\|x_1^{-1} \partial_{1} w \|_{L^\infty(\R^n)}+\|D^2w \|_{L^\infty(\R^n)} \Big)x_1 \label{v5SMVDf3}
\end{align} for all \(x\in \R^n_+\). The constant \(C\) depends only on \(n\) and \(s\). 
\end{lem}

\begin{proof}
The proof that \((-\Delta)^sv\) is a smooth function and antisymmetric follows from standard properties of the fractional Laplacian, so we will omit it and focus on the proof of~\eqref{v5SMVDf3}. For all \(\tilde x \in \R^{n+2}\), define \(\tilde w (\tilde x ) := w \big(\sqrt{\tilde x_1^2+\tilde x_2^2+\tilde x_3^2 }, \tilde x_4, \dots , \tilde x_{n+2}\big)\). By Bochner's relation, we have that \begin{align}
(-\Delta)^sv(x) &= x_1 (-\Delta)^s_{\R^{n+2}} \tilde w (x_1,0,0,x_2, \dots ,x_n) \label{0I9upPuf}
\end{align} where \((-\Delta)^s_{\R^{n+2}} \) is the fractional Laplacian in \(\R^{n+2}\) (and \((-\Delta)^s \) still refers to the fractional Laplacian in \(\R^n\)). For more details regarding Bochner's relation and a proof of~\eqref{0I9upPuf}, we refer the interested reader to the upcoming note~\cite{MR4411363}. 

Thus, applying standard estimates for the fractional Laplacian, we have that, \begin{align*}
\vert (-\Delta)^s v(x)\vert &\leqslant C\Big( \| D^2 \tilde w \|_{L^\infty(\R^{n+2})} + \|  \tilde w \|_{L^\infty(\R^{n+2})}  \Big)x_1 
\end{align*}for all \(x\in \R^n_+\) and \(C>0\) depending on~$n$ and~$s$. Clearly, \(\|  \tilde w \|_{L^\infty(\R^{n+2})} \leqslant \| w \|_{L^\infty(\R^n)}\). Moreover, by a direct computation, one can check that
\begin{eqnarray*}
\partial_{ij} \tilde w( \tilde x) &=&\partial_{11}w\Big(\sqrt{\tilde x_1^2+\tilde x_2^2+\tilde x_3^2 }, \tilde x_4, \dots , \tilde x_{n+2}\Big)
\frac{\tilde x_i\tilde x_j}{\tilde x_1^2+\tilde x_2^2+\tilde x_3^2}\\&&+
\partial_1 w\Big(\sqrt{\tilde x_1^2+\tilde x_2^2+\tilde x_3^2 }, \tilde x_4, \dots , \tilde x_{n+2}\Big)\left[\frac{\delta_{ij}}{\sqrt{\tilde x_1^2+\tilde x_2^2+\tilde x_3^2 }}-\frac{\tilde x_i\tilde x_j}{(\tilde x_1^2+\tilde x_2^2+\tilde x_3^2 )^{3/2}}
\right] \\
&&\qquad{\mbox{ if }} i,j\in\{1,2,3\},\\
\partial_{ij} \tilde w( \tilde x) &=&\partial_{1j-2}w\Big(\sqrt{\tilde x_1^2+\tilde x_2^2+\tilde x_3^2 }, \tilde x_4, \dots , \tilde x_{n+2}\Big)
\frac{\tilde x_i}{\sqrt{\tilde x_1^2+\tilde x_2^2+\tilde x_3^2 }}\\
&&\qquad {\mbox{ if }} i\in\{1,2,3\} \;{\mbox{ and }}\; j\in\{4, \dots, n+2\},\\
\partial_{ij} \tilde w( \tilde x) &=&\partial_{i-2j-2}w\Big(\sqrt{\tilde x_1^2+\tilde x_2^2+\tilde x_3^2 }, \tilde x_4, \dots , \tilde x_{n+2}\Big)
\qquad {\mbox{ if }} i,j\in\{4,\dots, n+2\},
\end{eqnarray*}
and therefore
\begin{align*}
\| D^2 \tilde w\|_{L^\infty(\R^{n+2})} &\leqslant C \Big(\|x_1^{-1} \partial_1 w \|_{L^\infty(\R^n)}+\|D^2w \|_{L^\infty(\R^n)} \Big)
\end{align*} for some universal constant \(C>0\), which implies the desired result. 
\end{proof}

We now construct the barrier that will be essential to allow \(B_\rho(a)\) in \thref{TV1cTSyn} to come up to \(\partial H\). 

\begin{lem} \thlabel{gmjWA2gy} Let \(a \in \overline{\R^n_+}\) and \(\rho>0\). There exists a function \(\varphi \in C^\infty_0(B_\rho(a)\cup B_\rho(a'))\) such that \(\varphi\) is antisymmetric with respect to \(\partial \R^n_+\), \(\rho^{2s}\chi_{B_{\rho/2}(a)}(x)x_1 \leqslant \varphi(x) \leqslant C\rho^{2s} x_1\) in \(\R^n_+\), and \begin{equation}\label{formula}
\vert (-\Delta)^s \varphi(x) \vert \leqslant C  x_1 \qquad \text{in } B_\rho^+(a).
\end{equation} The constant \(C\) depends only on \(n\) and \(s\). 
\end{lem}

Recall that~$ B_\rho^+(a)=B_\rho(a)\cap \R^n_+$, and notice that~$
B_\rho^+(a)$ in Lemma~\ref{gmjWA2gy} coincides with~$B_\rho(a)$
if~$\rho\in(0,a_1)$ with~$a=(a_1,\dots,a_n)$.

\begin{proof}[Proof of Lemma~\ref{gmjWA2gy}]
Let \(\eta \in C^\infty_0(B_1)\) be a radial function such that~\(\eta =1\) in~\(B_{1/2}\) and~\(0\leqslant \eta \leqslant 1\) in~\(\R^n\). Define \begin{align*}
\varphi(x) &:= \rho^{2s} x_1 \bigg ( \eta \bigg ( \frac{x-a} \rho \bigg ) + \eta \bigg ( \frac{x-a'} \rho \bigg ) \bigg ) .
\end{align*} Then \(\varphi \in C^\infty_0(B_\rho(a)\cup B_\rho(a'))\), it is antisymmetric, and satisfies \(\rho^{2s}\chi_{B_{\rho/2}(a)}(x)x_1\leqslant \varphi \leqslant 2 \rho^{2s} x_1\) in~\(\R^n_+\). 

Moreover, let 
$$\bar \varphi(x):= x_1 \left( \eta \left(x - \frac{a}\rho\right) + \eta \left(x-
\frac{a'}\rho\right) \right) $$ so that \(\varphi(x) = \rho^{2s+1}\bar \varphi(x/\rho)\). From \thref{xwrqefNE}, it follows that \( \vert (-\Delta)^s \bar \varphi \vert \leqslant C x_1\) for some \(C>0\) depending on~$n$ and~$s$, which implies that \begin{align*}
\vert (-\Delta)^s \varphi(x) \vert &= \rho \vert (-\Delta)^s \bar \varphi(x/\rho ) \vert \leqslant  C\rho \bigg ( \frac{x}{\rho} \bigg )_1 = C x_1. \qedhere
\end{align*}
\end{proof}

Now, we will give the proof of \thref{TV1cTSyn}.

\begin{proof}[Proof of \thref{TV1cTSyn}]
Without loss of generality, we may assume that \(H = \R^n_+\). 
We also denote by~$B=B_\rho(a)$.
Recall that, given \(A\subset \R^n\), we use the notation \(A^+ = A \cap \R^n_+\). Let \(\tau \geqslant 0\) be a constant to be chosen later and \(w := \tau\varphi +  (\chi_K+\chi_{K'}) v\) where \(\varphi\) is as in \thref{gmjWA2gy} and \(K':=Q_{\partial \R^n_+}(K)\). Furthermore, let~\(\xi \in \mathcal H^s_0(B^+)\) with~\(\xi \geqslant 0\). By formula~\eqref{formula} in \thref{gmjWA2gy}, we have that \begin{align*}
\mathcal E (w, \xi )& = \tau \mathcal E (\varphi, \xi )+ \mathcal E (\chi_K v , \xi )+ \mathcal E (\chi_{K'}v , \xi ) \\
&\leqslant C  \tau \int_{B^+} x_1 \xi(x) \dd x -c_{n,s}  \int_{B^+ } \int_{K} \frac{\xi(x) v(y)}{\vert x-y\vert^{n+2s}} \dd y \dd x \\
&\qquad -c_{n,s}  \int_{B^+ } \int_{K'} \frac{\xi(x) v(y)}{\vert x-y\vert^{n+2s}} \dd y \dd x \\
&= \int_{B^+ } \bigg [ C\tau x_1 - c_{n,s}  \int_{K}  \bigg ( \frac 1{\vert x-y\vert^{n+2s}} - \frac 1 {\vert x'- y \vert^{n+2s}} \bigg ) v(y) \dd y \bigg ] \xi(x) \dd x 
\end{align*} where \(x':= Q_{\partial H}(x)\).

Since, for all \(x\in B^+\) and \(y\in K\), we have that \begin{align*}
\frac 1{\vert x-y\vert^{n+2s}} - \frac 1 {\vert x'- y \vert^{n+2s}}  &= \frac{n+2s}2 \int_{\vert x- y \vert^2}^{\vert x'- y \vert^2} t^{-\frac{n+2s+2}2} \dd t \\
&\geqslant \frac{n+2s}2  \;\frac{
\vert x' - y \vert^2 - \vert x- y \vert^2 }{ \vert x'- y \vert^{n+2s+2}} \\
&= 2(n+2s)  \frac{x_1y_1}{\vert x ' -y \vert^{n+2s+2} } \\
&\geqslant 2(n+2s) M^{n+2s+2} x_1y_1,
\end{align*} it follows that \begin{align*}
\mathcal E(w,\xi) &\leqslant C \Big ( \tau- \tilde C M^{n+2s+2} \| y_1 v \|_{L^1(K)} \Big)\int_{B^+ }   x_1 \xi(x) \dd x 
\end{align*} with \(C\) and \(\tilde C\) depending only on \(n\) and \(s\).

Hence, using that \(w = \tau \varphi \leqslant C \rho^{2s} \tau x_1\) in \(B^+\) by \thref{gmjWA2gy}, we have that \begin{align*}
\mathcal E(w,\xi) + &\int_{B^+} c(x) w(x) \xi(x) \dd x  \\ &\leqslant C \Big[ \tau\big(1  + \rho^{2s}\|c^+\|_{L^\infty(U)}\big)  - \tilde C M^{n+2s+2}  \| y_1 v \|_{L^1(K)} \Big]\int_{B^+ }   x_1 \xi(x) \dd x 
\end{align*} (after possibly relabeling \(C\) and \(\tilde C\), but still depending only on \(n\) and \(s\)). Choosing \begin{align*}
\tau &:=  \frac12 \tilde C M^{n+2s+2}  \| y_1 v \|_{L^1(K)}\big(
1  + \rho^{2s}\|c^+\|_{L^\infty(U)} \big)^{-1}
\end{align*} we obtain that \begin{align*}
(-\Delta)^s w + cw &\leqslant 0 \qquad \text{in } B^+. 
\end{align*} Moreover, in \(\R^n_+ \setminus B^+\), we have that~\(w = v \chi_K \leqslant  v\), so,
recalling also~\eqref{ed0395v04ncrhrhfwejfvwejhfvqjhkfvqejhk}, \cite[Proposition 3.1]{MR3395749} 
implies that,
in~\(B^+\), \begin{align*}
v(x) &\geqslant w(x)  =  \frac12 \tilde C M^{n+2s+2} \big(
1+\rho^{2s}\|c^+\|_{L^\infty(U)} \big)^{-1} \| y_1 v \|_{L^1(K)} \varphi(x). 
\end{align*} Recalling that \(\varphi \geqslant \rho^{2s} x_1 \chi_{B_{\rho/2}^+(a)}\) 
(by Lemma~\ref{gmjWA2gy}),
we obtain the final result. 
\end{proof}

\section{The stability estimate and proof of Theorem~\ref{oAZAv7vy} }
\label{sec:stabest}

The proof of \thref{oAZAv7vy} makes use of the method of moving planes. Before we begin our discussion of this technique and give the proof of \thref{oAZAv7vy}, we must fix some notation. Let~\(\mu \in \R\),  \(e\in \Sph^{n-1}\), and~\(A \subset \R^n\). Then we have the following standard definitions: 
\begin{align*}
\pi_\mu&=\{ x\in \R^n \text{ s.t. } x\cdot e =\mu \} && \text{a hyperplane orthogonal to }e \\
H_\mu&=\{x\in \R^n  \text{ s.t. } x\cdot e>\mu \} &&\text{the right-hand half space with respect to } \pi_\mu  \\
H_\mu'&=\{x\in \R^n  \text{ s.t. } x\cdot e<\mu \} &&\text{the left-hand half space with respect to } \pi_\mu\\
A_\mu &= A \cap H_\mu &&\text{the portion of }A \text{ on the right-hand side of } \pi_\mu \\
x_\mu' &= x-2(x\cdot e -\mu) e && \text{the reflection of } x \text{ across } \pi_\mu \\ 
A_\mu' &= \{x\in \R^n \text{ s.t. } x'_\mu\in A_\mu \} && \text{the reflection of } A_\mu \text{ across } \pi_\mu
\end{align*} Note that in some articles such as \cite{MR3836150} \(A'_\mu\) is used to denote the reflection of~\(A\) (instead of~\(A_\mu\)) across~\(\pi_\mu\).

The method of moving planes works as follows. Fix a direction \(e\in \Sph^{n-1}\) and suppose that \(\Omega\) is a bounded open subset of \(\R^n\) with \(C^1\) boundary. Since \(\Omega\) is bounded, for \(\mu\) sufficiently large the hyperplane \(\pi_\mu\) does not intersect \(\Omega\). Furthermore, by decreasing the value of \(\mu\), at some point \(\pi_\mu\) will intersect \(\overline \Omega\). We denote the value of \(\mu\) at this point by \begin{align*}
\Lambda = \Lambda_e :=\sup\{x \cdot e \text{ s.t. } x\in \Omega \}.
\end{align*} From here, we continue to decrease the value of \(\mu\). Initially, since \(\partial \Omega\) is \(C^1\), the reflection of \(\Omega_\mu\) across \(\pi_\mu\) will be contained within \(\Omega\), that is \(\Omega_\mu'\subset \Omega\) for \(\mu < \Lambda\) but with \(\mu\) sufficiently close to \(\Lambda\).
Eventually, as we continue to make \(\mu\) smaller, there will come a point when this is no longer the case.
More precisely, there exists \(\lambda =\lambda_e\in \R\) such that for all \(\mu \in [\lambda,\Lambda)\), it occurs that~\(\Omega'_\mu \subset \Omega\), but~\(\Omega'_\mu \not\subset \Omega\) for~\(\mu<\lambda\). We may write~\(\lambda\) more explicitly as \begin{align}
\lambda := \inf \Big\{ \tilde \mu \in \R \text{ s.t. } \Omega'_\mu \subset \Omega \text{ for all } \mu \in (\tilde \mu , \Lambda) \Big\} . \label{IAhmMsFq}
\end{align} When \(\mu = \lambda\), geometrically speaking, there are two possibilities that can occur: \label{pagemov}

\emph{Case 1:} The boundary of \(\Omega'_\lambda\) is internally tangent to~\(\partial\Omega\) at some point not on~\(\pi_\lambda\), that is, there exists~\(p\in (\partial \Omega \cap \partial \Omega'_\lambda) \setminus \pi_\lambda \); or 

\emph{Case 2:} The critical hyperplane \(\pi_\lambda\) is orthogonal to the~\(\partial \Omega\) at some point, that is, there exists~\(p\in \partial\Omega \cap \pi_\mu\) such that the normal of~\(\partial \Omega\) at~\(p\) is contained in the plane~\(\pi_\mu\). 

We emphasise that, when we refer to the critical hyperplane \(\pi_\lambda\) in the proceeding discussion and results, \(\pi_\lambda\) may satisfy Case 1 or Case 2 (or both) i.e. the notation \(\pi_\lambda\) is not specific to Case 1 nor is it specific to Case 2, but allows for either/both situations to occur. At this stage, let us introduce the function \begin{align}
v_\mu(x) := u(x) - u( x_\mu' ), \qquad \text{for all } \mu \in \R \text{ and }  x\in \R^n . \label{4LML9uGd}
\end{align} It follows that, in \(\Omega_\mu'\), \begin{align*}
(-\Delta)^s v_\mu (x) &= f(u(x)) - f(u(x_\mu')) = -c_\mu(x) v_\mu (x) 
\end{align*} where \begin{align*}
c_\mu(x) &= \int_0^1 f'((1-t)u(x) +tu(x_\mu')) \dd t .
\end{align*} Note that \begin{align}
\| c_\mu \|_{L^\infty(\Omega_\mu ')} \leqslant [f]_{C^{0,1}([0,\| u\|_{L^\infty(\Omega)}])}. \label{j5AoQqGi}
\end{align} Hence, \(v_\mu\) is an antisymmetric function that satisfies \begin{align*}
\begin{PDE}
(-\Delta)^s v_\mu  +c_\mu v_\mu &= 0 &\text{in } \Omega_\mu ' ,\\
v_\mu &= u &\text{in } (\Omega\cap H_\mu') \setminus \Omega_\mu' ,\\
v_\mu &= 0 &\text{in } H_\mu' \setminus \Omega ,
\end{PDE}
\end{align*} with \(c_\mu \in L^\infty (\Omega_\mu ')\).

In the situations where one expects that~\(\Omega\) should be a ball, the goal of the method of moving planes is to prove that~\(v_\lambda \equiv 0\) i.e. \(u\) is even with respect to reflections across the critical hyperplane~\(\pi_\lambda\). Since the direction~\(e\) was arbitrary, one can then deduce that~\(u\) must be radial with respect to some point. The proof that~\(v_\lambda \equiv 0\) is achieved through repeated applications of the maximum principle.

\begin{remark} \thlabel{VbKq6OVB}
From the preceding exposition, it is clear that in several instances we will need to evaluate \(u\) and \(v_\mu\) at a single point. This is technically an issue since, in \thref{oAZAv7vy}, \(u\) is only assumed to be in \(H^s(\R^n)\cap L^\infty(\R^n)\). However, by standard regularity theory, we have that\footnote{Here we are using the notation that for \(\alpha>0\) with \(\alpha\) not an integer, \(C^\alpha(\Omega):=C^{k,\beta}(\Omega)\) where \(k\) is the integer part of \(\alpha\) and \(\beta=\alpha-k\in (0,1)\).} \(u \in C^{2s+1-\varepsilon}(\Omega)\) for all \(\varepsilon>0\) such that \(2s+1-\varepsilon\) is not an integer. In particular, this implies that~\(u\in C^1(\Omega)\) which will be essential for the proof of the theorem. Indeed, using that \(f\) is locally Lipschitz, we have that \((-\Delta)^s u = -cu+f(0)\) with \(c(x) = - \int_0^1 f'(tu(x))\dd t\), so \((-\Delta)^su \in L^\infty(\Omega)\). Hence, it follows that \(u\in C^{2s-\varepsilon}(\Omega)\) for all \(\varepsilon\in (0,2s)\), \(2s-\varepsilon \not\in \Z\), see \cite[Proposition 2.3]{MR3168912}. Then, it follows that~\((-\Delta)^su = f(u) \in C^{\min \{1-\varepsilon,2s-\varepsilon \}} (\Omega)\), so by \cite[Proposition~2.2]{MR3168912} and a bootstrapping argument (if necessary), we obtain that~\(u \in C^{2s+1-\varepsilon}(\Omega)\). 
\end{remark}

\subsection{Uniform stability in each direction}

In this subsection, we will use the maximum principle of Section~\ref{S7DGIjUf} to prove uniform stability for each direction~\(e\in \Sph^{n-1}\), that is, we will show, for each~\(e\in \Sph^{n-1}\), that~\(\Omega\) is almost symmetric with respect to~\(e\). This is stated precisely in \thref{7TQmUHhl}. We will repeatedly use that fact that~\(v_\lambda\) is~\(C^1\) which follows from \thref{VbKq6OVB}. Before proving \thref{7TQmUHhl}, we have two lemmata. 

\begin{lem} \thlabel{JFhLgBx4}
Let \(\Omega\) be a bounded open set with \(C^1\) boundary, \(e\in \Sph^{n-1}\), and \(v_\mu\) as in~\eqref{4LML9uGd}. 

Then, for all \(\mu \in [\lambda , \Lambda]\), we have that~\(v_\mu \geqslant 0\) in \(\Omega_\mu'\).   
\end{lem}

The proof of \thref{JFhLgBx4}
follows from an application of the moving plane method in the nonlocal framework and is given as part of the proof of Theorem 1.4 in \cite{RoleAntisym2022}, so we will not include it again here. However, we would like to emphasise that, even though \cite[Theorem 1.4]{RoleAntisym2022} assumes the solution \(u\) of~\eqref{problem00} is constant on \(\partial G\) (i.e. \([u]_{\partial G}=0\)), this assumption was unnecessary to obtain the much weaker result of \thref{JFhLgBx4}.

We now give the second lemma.

\begin{lem} \thlabel{goNy8kYt}
Let \(\Omega\) and~$G$ be open bounded sets with~\(C^1\) boundary
such that~$\Omega=G+B_R$
for some~$R>0$. 
Let~\(f \in C^{0,1}_{\mathrm{loc}}(\R)\) be such that \(f(0)\geqslant 0\)
and~\(u\in  H^s (\R^n)\) be a solution of~\eqref{problem00}.

Then, for each \(e\in \Sph^{n-1}\), we have that \begin{align}
 \int_{(\Omega\cap H'_\lambda) \setminus \Omega_\lambda'} \delta_{\pi_\lambda}(x) u (x) \dd x \leqslant C  [u]_{\partial G} \label{HqEzyGyW},
\end{align} where \begin{align}
C := C(n,s)  R_0^{-2s} (\diam \Omega)^{n+2s+2} \label{KW9pZNdx}
\end{align} and \(R_0\) is given by~\eqref{K57sa7Sk}.
\end{lem}

\begin{proof}
Without loss of generality, take \(e=-e_1\) and \(\lambda =0\). 
We now apply the method of moving planes to~$G$.
First, suppose that we are in the first case,
namely, the boundary of \(G'_\lambda\) is internally tangent to~\(\partial G\) at some point not on~\(\{x_1=0\}\),
and let~\(p\in (\partial G \cap \partial G_\lambda') \setminus \{x_1=0\}\).

By \thref{TV1cTSyn} with~$U:=\Omega'_\lambda$, \(K := (\Omega \cap H_\lambda')\setminus \Omega_\lambda'\) and~\(B_\rho(a):=B_{\varepsilon R_0}(p)\)
(notice that, according to Remark~\ref{uElLogjj}, the condition~\eqref{ed0395v04ncrhrhfwejfvwejhfvqjhkfvqejhk}
is satisfied for \(\varepsilon=\varepsilon(n,s)\) sufficiently small), we have that \begin{align}
\frac{v_\lambda(p)}{p_1} &\geqslant C \int_{(\Omega\cap H_\lambda') \setminus \Omega_\lambda '} y_1 v_\lambda (y) \dd y = C \int_{(\Omega\cap H_\lambda') \setminus \Omega_\lambda '} y_1 u(y) \dd y. \label{bYXi8urM}
\end{align}
Note that, since \(x_\lambda'\) belongs to the reflection of~$\Omega$ across~\(\{x_1=0\}\)
for each \(x\in(\Omega\cap H_\lambda') \setminus \Omega_\lambda'\), we have that  \begin{align*}
\inf_{{x\in( \Omega\cap H'_\lambda) \setminus \Omega_\lambda'}\atop{y\in B_{R/2}^+(p)}} \vert x_\lambda'- y \vert^{-1} &\geqslant (\diam \Omega)^{-1}, 
\end{align*} so \thref{TV1cTSyn} implies that the constant in~\eqref{bYXi8urM} is given by \begin{align*}
C &= C(n,s) R_0^{2s} (\diam \Omega)^{-n-2s-2} \big(1+R_0^{2s}\|c^+\|_{L^\infty(\Omega_\lambda')} \big)^{-1} \geqslant C(n,s) R_0^{2s} (\diam \Omega)^{-n-2s-2} 
\end{align*} using~\eqref{j5AoQqGi} and~\eqref{K57sa7Sk}. Moreover, we have that \begin{align*}
\frac{v_\lambda(p)}{p_1} = \frac{u(p)-u(p_\lambda')}{p_1} = \frac{2(u(p)-u(p_\lambda'))}{\vert p_1-(p_\lambda')_1\vert} \leqslant 2 [u]_{\partial G},
\end{align*} which, along with~\eqref{bYXi8urM}, gives~\eqref{HqEzyGyW}. 

Now, let us suppose that we are in the second case and let~\(p\in \partial G \cap \{x_1=0\}\) be
such that the normal of~\(\partial G\) at~\(p\) is contained in~\(\{x_1=0\}\).
Proceeding in a similar fashion as the first
case, we apply \thref{TV1cTSyn} with~$U:=\Omega'_\lambda$,
\(K := (\Omega\cap H_\lambda') \setminus \Omega_\lambda'\) and~\(B_\rho(a):=B_{\varepsilon R}(p+he_1)\) with \(\varepsilon>0\) as above and~\(h>0\) very small,
to obtain that \begin{align*}
\int_{(\Omega\cap H_\lambda') \setminus \Omega_\lambda'} x_1 u (x) \dd x \leqslant C \frac{v_\lambda(p+he_1)}h
\end{align*} with \(C\) in the same form as in~\eqref{KW9pZNdx} and, in particular, independent of \(h\). Sending \(h\to 0^+\), we obtain that \begin{align*}
\int_{(\Omega\cap H_\lambda') \setminus \Omega_\lambda'} x_1 u (x) \dd x \leqslant C \partial_1 v_\lambda(p) \leqslant C [u]_{\partial G} 
\end{align*} which gives~\eqref{HqEzyGyW} in this case as well. 
\end{proof}

We are now able to obtain uniform stability for each direction in \thref{7TQmUHhl} below. 

\begin{prop}\thlabel{7TQmUHhl}
Let \(\Omega\) be an open bounded set with~\(C^1\) boundary and satisfying the uniform interior ball condition with radius~\(r_\Omega >0\)
and~$G$ be an open bounded set with~\(C^1\) boundary
such that~$\Omega=G+B_R$
for some~$R>0$. Let~\(f \in C^{0,1}_{\mathrm{loc}}(\R)\) be such that \(f(0)\geqslant 0\)
and~\(u\) be a weak solution of~\eqref{problem00}.

For \(e\in \Sph^{n-1}\), let \(\Omega '\) denote the reflection of \(\Omega\) with respect to the critical hyperplane \(\pi_\lambda\).

Then, \begin{align}
\vert \Omega \triangle \Omega' \vert \leqslant C_\star [u]_{\partial G}^{\frac1{s+2}}, \label{D8UMJ0i3}
\end{align} where \begin{align*}
&C_\star \\&:= C(n,s) \Big( C_\ast^{-1}
\big( f(0)+\|u\|_{L_s(\R^n)}\big)^{-1} R_0^{-2s} (\diam \Omega)^{n+2s+2}
+ (\diam \Omega)^{n-1} + r_\Omega^{-1}\vert \Omega \vert \Big) \Big ),
\end{align*} the constant \(C_\ast\) as in \thref{zAPw0npM}, and \(R_0\) given by~\eqref{K57sa7Sk}.
\end{prop}

\begin{proof}
Without loss of generality, take \(e=-e_1\) and \(\lambda =0\). By \thref{zAPw0npM}, we have that \begin{align}
\int_{(\Omega\cap H_\lambda') \setminus \Omega_\lambda'} x_1 u (x) \dd x &\geqslant C_\ast \big( f(0)+\|u\|_{L_s(\R^n)}\big)
 \int_{(\Omega\cap H_\lambda') \setminus \Omega_\lambda'} x_1 \delta_{\partial \Omega}^s(x) \dd x. \label{JfjtK29b}
\end{align}
Fix \(\gamma >0\). By Chebyshev's inequality, \eqref{JfjtK29b} and \thref{goNy8kYt}, we have that  \begin{equation}\begin{split}\label{ido4985vb6tdfghwafrikewytoiw}
\Big| \big\{ x \in (\Omega\cap H_\lambda') \setminus \Omega_\lambda' \text{ s.t. } x_1 \delta_{\partial \Omega}^s(x) >\gamma \big\} \Big| \leqslant\;& \frac1{\gamma}
\int_{(\Omega\cap H_\lambda') \setminus \Omega_\lambda'} x_1 \delta_{\partial \Omega}^s (x) \dd x \\ \leqslant\;& \frac {CC_\ast^{-1}\big( f(0)+\|u\|_{L_s(\R^n)}\big)^{-1}} \gamma [u]_{\partial G}.
\end{split}\end{equation}  Moreover, \begin{eqnarray*}&&
\Big| \big\{ x \in (\Omega\cap H_\lambda') \setminus \Omega_\lambda' \text{ s.t. } x_1 \delta_{\partial \Omega}^s(x) \leqslant \gamma \big\} \Big| \\&=&
\Big| \big\{ x \in (\Omega\cap H_\lambda') \setminus \Omega_\lambda' \text{ s.t. } x_1 \delta_{\partial \Omega}^s(x) \leqslant \gamma, x_1< \gamma^{\frac 1{s+1}} \big\} \Big| \\&&\qquad + \Big| \big\{ x \in (\Omega \cap H_\lambda')
\setminus \Omega_\lambda' \text{ s.t. } x_1 \delta_{\partial \Omega}^s(x) \leqslant \gamma , x_1 \geqslant \gamma^{\frac 1{s+1}} \big\} \Big| \\
&\leqslant& \Big| \big\{ x \in \Omega^+ \text{ s.t. } x_1< \gamma^{\frac 1{s+1}} \big\} \Big| 
+  \Big| \big\{ x \in \Omega \text{ s.t. } \delta_{\partial \Omega}(x) \leqslant \gamma^{\frac 1{s+1}} \big\} \Big| . 
\end{eqnarray*}

Furthermore, we have the estimate \begin{align*}
 \Big| \big\{ x \in \Omega^+ \text{ s.t. } x_1< \gamma^{\frac 1{s+1}} \big\} \Big| &\leqslant (\diam \Omega)^{n-1} \gamma^{\frac 1{s+1}}
\end{align*} and, by \cite[Lemma 5.2]{MR4577340} in the case that \(\partial \Omega\) is \(C^2\) and more generally in \cite{MR483992}, we have that \begin{align*}
 \Big| \big\{ x \in \Omega \text{ s.t. } \delta_{\partial \Omega}(x) \leqslant \gamma^{\frac 1{s+1}} \big\} \Big|
 &\leqslant \frac{2n \vert \Omega \vert }{r_\Omega} \gamma^{\frac 1{s+1}}.
\end{align*} Thus,
\begin{align}\label{eq:NUOVAPERIMPROVEMENT}
\Big| \big\{ x \in (\Omega\cap H_\lambda') \setminus \Omega_\lambda' \text{ s.t. }& x_1 \delta_{\partial \Omega}^s(x) \leqslant \gamma \big\} \Big| 
\\
&\le \left[ (\diam \Omega)^{n-1} +\frac{2n \vert \Omega \vert }{r_\Omega} \right] \gamma^{\frac 1{s+1}}. \nonumber
\end{align}
{F}rom this and~\eqref{ido4985vb6tdfghwafrikewytoiw}, we deduce that
\begin{eqnarray*}
\vert (\Omega\cap H_\lambda') \setminus \Omega'_\lambda \vert \leqslant
\frac {CC_\ast^{-1}\big( f(0)+\|u\|_{L_s(\R^n)}\big)^{-1}} \gamma [u]_{\partial G}
+\left[ (\diam \Omega)^{n-1} +\frac{2n \vert \Omega \vert }{r_\Omega} \right] \gamma^{\frac 1{s+1}}.
\end{eqnarray*}
Hence,
 \begin{align*}
\vert \Omega \triangle \Omega '\vert &= 2 \vert (\Omega\cap H_\lambda') \setminus \Omega'_\lambda \vert \\ &\leqslant \frac {2CC_\ast^{-1}\big( f(0)+\|u\|_{L_s(\R^n)}\big)^{-1}} {\gamma} [u]_{\partial G}
+2\left[ (\diam \Omega)^{n-1} +\frac{2n \vert \Omega \vert }{r_\Omega} \right] \gamma^{\frac 1{s+1}}
\end{align*} for all \(\gamma >0\). Choosing \(\gamma = [u]_{\partial G}^{\frac{s+1}{s+2}}\),
we obtain that
$$
\vert \Omega \triangle \Omega '\vert \leqslant \left[ 2CC_\ast^{-1}\big( f(0)+\|u\|_{L_s(\R^n)}\big)^{-1}
+2\left( (\diam \Omega)^{n-1} +\frac{2n \vert \Omega \vert }{r_\Omega} \right) \right][u]_{\partial G}^{\frac 1{s+2}},$$
which gives~\eqref{D8UMJ0i3}. 
\end{proof}

Note, setting \([u]_{\partial G}=0\) in \thref{7TQmUHhl} implies that \(\Omega\) (and, therefore, \(G\)) must be symmetric with respect to the critical hyperplane for every direction. This implies that they are both balls, thereby recovering the main result of \cite{RoleAntisym2022}. 

\subsection{Proof of Theorem~\ref{oAZAv7vy}} 

We will now use the results of the previous subsection to prove \thref{oAZAv7vy}. In fact, \thref{oAZAv7vy} follows almost immediately from the following result, \thref{k4VRS3Fx}. The idea of \thref{k4VRS3Fx} is to choose a centre for \(\Omega\) by looking at the intersection of \(n\) orthogonal critical hyperplanes, then to prove that every other critical hyperplane is quantifiably close to this centre in terms of the semi-norm \([u]_{\partial G}\). The precise statement is as follows. 

\begin{prop} \thlabel{k4VRS3Fx} Let \(\Omega\) be an open bounded set with \(C^1\) boundary and satisfying the uniform interior ball condition with radius \(r_\Omega >0\) and suppose that the critical planes \(\pi_{e_i}\) with respect to the coordinate directions \(e_i\) coincide with \(\{x_i=0\}\) for every \(i=1,\dots, n\). Also, given \(e\in \Sph^{n-1}\), denote by \(\lambda_e\) the critical value associated with~\(e\) as in~\eqref{IAhmMsFq}. 

Assume that \begin{align}
[u]^{\frac 1{s+2} }_{\partial G} \leqslant \frac {\vert \Omega \vert } {n C_\star }   \label{MCxhJzBW}
\end{align} where \(C_\star\) is the constant in \thref{7TQmUHhl}.

Then,
\begin{align}\label{eq:NEWNEW OLD lambda estimate}
\vert \lambda_e \vert \leqslant C [u]_{\partial G}^{\frac 1 {s+2} }
\end{align}
for all \(e\in \Sph^{n-1}\) with \begin{eqnarray*}
C &:=& \frac{C(n) C_\star  (\diam \Omega)}{\vert \Omega \vert }\\
&=&C(n,s) \Big( C_\ast^{-1}
\big( f(0)+\|u\|_{L_s(\R^n)}\big)^{-1} R_0^{-2s} (\diam \Omega)^{n+2s+2}
\\
&&\hspace{10em}+ (\diam \Omega)^{n-1} + r_\Omega^{-1}\vert \Omega \vert \Big)  \frac{(\diam \Omega)}{\vert \Omega \vert },
\end{eqnarray*} the constant \(C_\ast\) as in \thref{zAPw0npM}, and \(R_0\) given by~\eqref{K57sa7Sk}. 
\end{prop}

\begin{proof}
Define \(\Omega^{\mathbf 0} := \{-x \text{ s.t. } x\in \Omega\}\). Moreover, let \(Q_i:=Q_{\pi_{e_i}}\) be the reflections across each critical plane \(\pi_{e_i}\) and define \(\Omega^{\mathbf 0}\) recursively via \(\Omega^{\mathbf 0}_{i+1}  := Q_{i+1}(\Omega^{\mathbf 0}_i)\) for \(i=1,\dots n-1\) with \(\Omega^{\mathbf 0}_1:=\Omega\). Observe that \( \Omega^{\mathbf 0} = \Omega^{\mathbf 0}_n\). Via the triangle inequality for  symmetric difference, it follows that \begin{align*}
\vert \Omega \triangle \Omega^{\mathbf 0} \vert &\leqslant \vert \Omega \triangle Q_n(\Omega ) \vert  + \vert  Q_n(\Omega ) \triangle  Q_n(\Omega^{\mathbf 0}_{n-1}  )\vert 
\end{align*} Since \( \vert  Q_n(\Omega ) \triangle  Q_n(\Omega^{\mathbf 0}_{n-1}  )\vert = \vert  Q_n(\Omega  \triangle  \Omega^{\mathbf 0}_{n-1}  )\vert = \vert  \Omega  \triangle  \Omega^{\mathbf 0}_{n-1} \vert \), we have that \begin{align*}
\vert \Omega \triangle \Omega^{\mathbf 0} \vert &\leqslant \vert \Omega \triangle Q_n(\Omega ) \vert  + \vert  \Omega  \triangle  \Omega^{\mathbf 0}_{n-1} \vert .
\end{align*} Iterating, we obtain \begin{align}
\vert \Omega \triangle \Omega^{\mathbf 0} \vert &\leqslant \sum_{i=1}^n \vert \Omega \triangle Q_i(\Omega ) \vert \leqslant n C_\star [u]_{\partial G}^{\frac1 {s+2}} \label{1wuBMACC}
\end{align} by~\thref{7TQmUHhl}. 

Next, let us assume that \(\lambda_e>0\) (the case \(\lambda_e<0\) is analogous). If \(\Lambda_e > \diam \Omega\) then \(x\cdot e \geqslant \Lambda_e - \diam \Omega  \geqslant 0\) for all \(x\in \Omega\), so \(\vert \Omega \triangle \Omega^{\mathbf 0} \vert = 2\vert \Omega \vert \). However, this is in contradiction with~\eqref{MCxhJzBW} and~\eqref{1wuBMACC}, so we must have that \(\Lambda_e \leqslant \diam \Omega\). Arguing as in Lemma~4.1 in~\cite{MR3836150} using~\thref{7TQmUHhl} instead of \cite[Proposition 3.1 (a)]{MR3836150} , we find that \begin{align*}
\vert \Omega_{\lambda_e} \vert \lambda_e \leqslant (n+3) C_\star (\diam \Omega)  [u]_{\partial G}^{\frac 1{s+2}} . 
\end{align*} Now, recalling the notation of Section~\ref{sec:stabest}, we have that \begin{align*}
\vert \Omega \triangle \Omega^{\mathbf 0} \vert &= 2 \vert (\Omega \cap H_{\lambda_e}' )\setminus \Omega_{\lambda_e}' \vert = 2 \big ( \vert \Omega \vert - 2 \vert \Omega_{\lambda_e} \vert \big ),
\end{align*} and therefore \begin{align*}
\vert \Omega_{\lambda_e} \vert  =\frac{\vert \Omega \vert } 2 - \frac14 \vert \Omega \triangle \Omega^{\mathbf 0} \vert  \geqslant \frac{\vert \Omega \vert } 2 -\frac14 C_\star [u]_{\partial G}^{\frac1{s+2}} \geqslant \frac{\vert \Omega \vert } 4 
\end{align*} by \thref{7TQmUHhl}  and~\eqref{MCxhJzBW}. Thus, we conclude that \begin{align*}
\vert \lambda_e\vert &\leqslant  \frac{C(n) C_\star  (\diam \Omega)}{\vert \Omega \vert }  [u]_{\partial G}^{\frac 1{s+2}} ,
\end{align*}
as desired.
\end{proof}

Now the proof of \thref{oAZAv7vy} follows almost immediately from \thref{k4VRS3Fx}. 

\begin{proof}[Proof of \thref{oAZAv7vy}] 
If \([u]_{\partial G}^{\frac 1 {s+2}}> \frac{\vert \Omega \vert}{nC_\star}\) then the result is trivial since \[\rho(\Omega) \leqslant  \diam \Omega \leqslant  \frac{nC_\star \diam \Omega	}{\vert \Omega \vert }  [u]^{\frac 1 {s+2}}_{\partial G}\leqslant C [u]^{\frac 1 {s+2}}_{\partial G}\] with \(C\) as in the statement of the theorem.

If~\eqref{MCxhJzBW} holds then the result follows by reasoning as in the proof of \cite[Theorem~1.2]{MR3836150}, but using \thref{k4VRS3Fx} instead of \cite[Lemma 4.1]{MR3836150}.
Notice that the dependence on $r_\Omega$ appearing in \thref{k4VRS3Fx} can be removed (and, in fact, does not appear in \thref{oAZAv7vy}).
In fact, by the definition of $\Omega := G + B_R(0)$, we have that $\Omega$ automatically satisfies the uniform interior sphere condition and we can take, e.g., $r_\Omega := R/2$.
\end{proof}

\section{The role of boundary estimates in the attainment of the optimal exponent} \label{fEsBEcuv}

In this section, we give a broad discussion on some of the challenges the nonlocality of the fractional Laplacian presents in obtaining the optimal exponent. By way of an example, via the Poisson representation formula, we show that estimates for a singular integral involving the reciprocal of the distance to the boundary function play a key role in obtaining the anticipated optimal result. This suggests, surprisingly, that fine geometric estimates for the distance function close to the boundary are required to obtain the optimal exponent. 

Recall the notation at the beginning of Section~\ref{sec:stabest} and also
\begin{eqnarray*}
&&\R^n_+:=\{x=(x_1,\dots,x_n)\in\R^n \;{\mbox{ s.t. }} x_1>0\},\\
&&\R^n_-:=\{x=(x_1,\dots,x_n)\in\R^n \;{\mbox{ s.t. }} x_1<0\},\\
&&B_1^+:=B_1\cap\R^n_+, \qquad B_1^-:=B_1\cap\R^n_-,\\
&&\Omega^+:=\Omega\cap \R^n_+ \qquad
{\mbox{and}} \qquad \Omega^-:=\Omega\cap \R^n_-.
\end{eqnarray*}

Let~$\Omega$ and $G$ be bounded and smooth sets.
For the purposes of this discussion, let us assume that \(\Omega = G+B_{1/2}\) where~\(G\) is such that \(\R^n_- \cap G = B_{1/2}^-\), \(B_{1/2}^+ \subset \R^n_+ \cap \Omega\), and the critical plane~\(\pi_\lambda\) corresponding to running the method of moving planes with~\(e=-e_1\) is equal to~\(\{x_1=0\}\), see Figure~\ref{Fig2}.

\begin{figure}[ht]
\centering 
\includegraphics{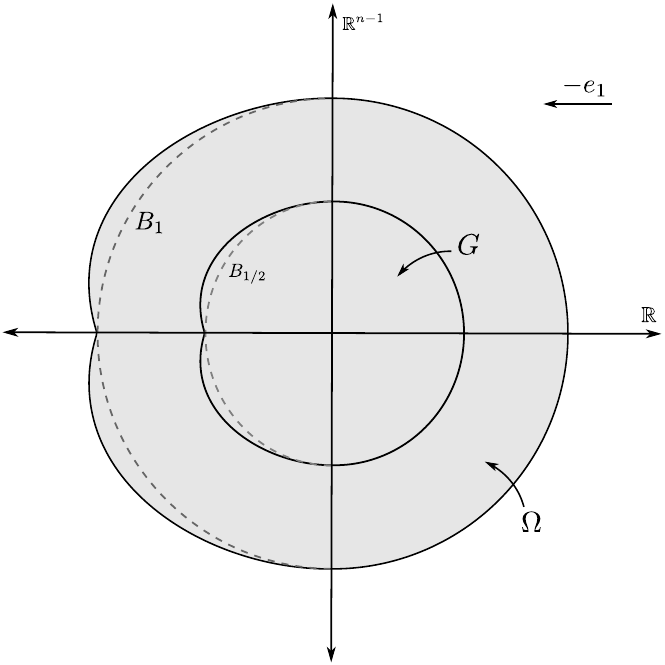}
\caption{Geometry of \(\Omega\) as described in Section~\ref{fEsBEcuv}.}
\label{Fig2}
\end{figure}

Suppose that \(u \in C^2(\Omega )\cap L^\infty(\R^n)\) satisfies the torsion problem \begin{align*}
\begin{PDE}
(-\Delta)^s u &= 1&\text{in } \Omega ,\\
u &=0 &\text{in } \R^n \setminus \Omega . 
\end{PDE}
\end{align*} In this case, if \(v(x) := u(x) - u(-x_1,x_2,\dots,x_n)\), then \(v\) is antisymmetric in \(\R^n\) with respect to~$\{x_1=0\}$
and \(s\)-harmonic in~\(B_1\). Hence, by the Poisson kernel representation, (up to normalization constants) \begin{align*}
v(x) &=  \int_{\R^n \setminus B_1 } \bigg ( \frac{1-\vert x \vert^2}{\vert y \vert^2-1} \bigg )^s \frac{v(y)}{\vert x - y \vert^n} \dd y .
\end{align*} Using the antisymmetry of \(v\), we may rewrite this as \begin{align*}
v(x) &= \int_{\R^n_+ \setminus B_1^+ } \bigg ( \frac{1-\vert x \vert^2}{\vert y \vert^2-1} \bigg )^s \bigg (\frac1{\vert x - y \vert^n} - \frac1{\vert x' - y \vert^n}\bigg )  v(y) \dd y \\
&=  \int_{\Omega^+ \setminus B_1^+ } \bigg ( \frac{1-\vert x \vert^2}{\vert y \vert^2-1} \bigg )^s \bigg (\frac1{\vert x - y \vert^n} - \frac1{\vert x' - y \vert^n}\bigg ) u(y) \dd y,
\end{align*}
where~$x'$ denotes the reflection of the point~$x$ with respect to~$\{x_1=0\}$.

Now, by \thref{oAZAv7vy}, when \([u]_{\partial G}\) is sufficiently small, then~\(\Omega\) is uniformly close to~\(B_1\), so we can suppose that~\(G_\lambda' \subset B_{3/4}\). Therefore, for all \(x\in \partial G\) and \(y \in \R^n_+\setminus \Omega\),  \begin{align*}
(1-\vert x \vert^2)^s \bigg (\frac1{\vert x - y \vert^n} - \frac1{\vert x' - y \vert^n}\bigg ) &\geqslant C x_1y_1.
\end{align*} 
Also, by Corollary~\ref{zAPw0npM},
we have that~\(u \geqslant C \delta^s_{\partial \Omega}\)
(with~$C>0$ not depending on~$u$),
and thus
\begin{align*}
v(x) &\geqslant C x_1 \int_{\Omega^+ \setminus B_1^+ }  \frac {y_1 u(y)}{(\vert y \vert^2-1)^s}  \dd y \geqslant Cx_1 \int_{\Omega^+ \setminus B_1^+ }  \frac {y_1\delta_{\partial \Omega}^s(y)}{\delta_{\partial B_1}^s(y)}  \dd y. 
\end{align*}

Recall the discussion on the moving plane method on page~\pageref{pagemov}
and, for simplicity, assume that we are in the first case (the second
case can be treated similarly), so there exists \(p\in (\partial G \cap \partial G_\lambda')\setminus \{x_1=0\}\). Hence, \begin{align} 
\int_{\Omega^+ \setminus B_1^+ }  y_1\bigg ( \frac {\delta_{\partial \Omega}(y)}{\delta_{\partial B_1}(y)} \bigg )^s  \dd y &\leqslant C \frac{v(p)}{p_1} \leqslant C [u]_{\partial G}. \label{K98q9WSJ}
\end{align}

Now the point of obtaining~\eqref{K98q9WSJ} is that the left hand side is geometric (does not depend on~\(u\)), and~\eqref{K98q9WSJ} is sharp in the sense that the only terms that have been `thrown away'
are bounded away from zero as \([u]_{\partial G} \to 0\). This suggests that if \([u]_{\partial G}\) is of order \(\varepsilon\) then the behaviour of \begin{align*}
\int_{\Omega^+ \setminus B_1^+ }  y_1\bigg ( \frac {\delta_{\partial \Omega}(y)}{\delta_{\partial B_1}(y)} \bigg )^s  \dd y
\end{align*} as a function of \(\varepsilon\) will entirely determine the optimal exponent~\(\overline\beta(s)\) (recall that
the definition of~\(\overline\beta(s)\) is given after the statement of \thref{oAZAv7vy} in Section~\ref{tH4AETfY}). This is surprising, since in the local case, the problem is entirely an interior one in that the proof relies only on interior estimates such as the Harnack inequality while, it appears that in the nonlocal case, the geometry of the boundary may have a significant effect on the value of~\(\overline\beta(s)\). If one hopes to obtain that~\(\overline\beta(s)=1\) from~\eqref{K98q9WSJ} then it would be necessary to show that \begin{align*}
\int_{\Omega^+ \setminus B_1^+ }  y_1\bigg ( \frac {\delta_{\partial \Omega}(y)}{\delta_{\partial B_1}(y)} \bigg )^s  \dd y \geqslant C \vert \Omega \setminus B_1^+ \vert. 
\end{align*} If it were true that \(y_1\delta^s_{\partial \Omega} \geqslant C \delta^s_{\partial B_1}\) then this would follow immediately; however, this is not the case as seen by sending \(y \to \partial \R^n_+ \cap B_{1/2}\). 

Moreover, even though we made several major assumptions on the geometry of \(\Omega\) to obtain~\eqref{K98q9WSJ}, we believe that the inequality is indicative of the more general situation. Indeed, one would expect that, under reasonable assumptions on \(\partial \Omega\), similar estimates to the ones employed above would hold for Poisson kernels in general domains, so one may suspect that an inequality of the form \begin{align*}
\int_{\Omega \triangle \Omega ' } \delta_{\pi_\lambda}(y)  \bigg ( \frac {\delta_{\partial \Omega}(y)}{\delta_{\partial \Omega' }(y)} \bigg )^s  \dd y &\leqslant C [u]_{\partial G}
\end{align*} should hold. It would be interesting in future articles to further explore this methodology to try obtain improvements on the exponent in \thref{oAZAv7vy}.

\section{Sharp investigation of a technique to improve the stability exponent}\label{c8w8u7Hn}

In the papers \cite{MR3881478, MR3836150}, it was proven that the only open bounded sets\footnote{For the precise statements and the regularity assumptions on the boundary of the set, see \cite{MR3836150,MR3881478}.} whose boundary has constant nonlocal mean curvature are balls.
Moreover, in \cite{MR3836150}, the method of moving planes was employed to obtain a stability estimate for sets whose boundary has
\emph{almost} constant nonlocal mean curvature. The stability estimate that
the authors proved also seems to achieve the optimal exponent of~\(1\) which is however
obtained using the following technical result,
see~\cite[Proposition 3.1(b)]{MR3836150}:

\begin{lem} \thlabel{Z08Q7917}
Let \(\Omega \subset \R^n\) be an open bounded set with \(C^{2,\alpha}\) boundary for~\(\alpha >2s\) and~\(s\in (0,1/2)\). Let~\(\pi_\lambda\) be the
critical hyperplane. Suppose that~\(\dist(0,\pi_\lambda) \leqslant 1/ 8\) and~\(B_r \subset \Omega \subset B_R\) with~\(1/2 \leqslant r\leqslant R \leqslant 2\).

Then, \begin{align}
\Big| \big\{ x\in \Omega \triangle \Omega ' \text{ s.t. } \dist(x,\pi_\lambda) \leqslant \gamma \big\} \Big| \leqslant C(n) \gamma (R-r) \label{IAMJKSyF}
\end{align} for all \(0<\gamma \leqslant 1/4\). 
\end{lem} Without this result, the exponent in the stability estimate in~\cite{MR3836150} would have been \(1/2\), so \thref{Z08Q7917} seems to play a major role in~\cite{MR3836150} in the attempt of ``doubling the exponent'' that they obtained in previous estimates. This feature is extremely relevant to our main result, \thref{oAZAv7vy}, since we hoped that a similar argument to \thref{Z08Q7917} would lead to an exponent that was twice the one that we obtained. Unfortunately, we believe that, in the very broad generality in which the result is stated in~\cite{MR3836150}, \thref{Z08Q7917} may not be true, and we present a counter-example
(which may also impact some of the statements in~\cite{MR3836150}) as well as a corrected statement of \thref{Z08Q7917}. 

\subsection{A geometric lemma and a counter-example to \protect\thref{Z08Q7917}} 
\label{6LQMifK3}

The main result of this section is a geometric lemma which may be viewed as a corrected version of \thref{Z08Q7917}. Moreover, we will give an explicit family of domains which demonstrate that our geometric lemma is sharp. This family will also serve as a counter-example to \thref{Z08Q7917}. 

An essential component to the geometric lemma that is not present in the assumptions of \thref{Z08Q7917} is a uniform bound on the boundary regularity of \(\partial \Omega\). To our knowledge, there is no `standard' definition of such a bound, so we will begin by concretely specifying what is meant by this.

Let \(\Omega\) be an open subset of \(\R^n\) with \(C^1\) boundary. For each \(x\in \partial\Omega\), let 
$$ \Pi_x(y):= (y-x) - ((y-x)\cdot \nu(x))\nu(x),$$
which is the projection of \(\R^n\) onto \(T_x\partial \Omega\), that is, the tangent plane of \(\partial \Omega\) at \(x\). We have made a translation in the definition of \(\Pi_x\), so that \(\Pi_x(x)=0\). For simplicity, we will often identify  \(T_x\partial \Omega\) with \(\R^{n-1}\). 

\begin{defn} \thlabel{deZqwlaB}
Let \(\Omega\) be an open subset of \(\R^n\) with \(C^\alpha\) boundary, with~\(\alpha>1\), 
and let~\(\rho\), \(M>0\). We say that \(\partial \Omega \in C^\alpha_{M,\rho}\) if, for all \(x\in \partial \Omega\), there exists \(\psi^{(x)} : B_\rho^{n-1}\to\R\) such that \(\psi^{(x)}\in C^\alpha (B_\rho^{n-1} )\), \(\| \psi^{(x)}\|_{ C^\alpha (B_\rho^{n-1})  } \leqslant M \), and \begin{align*}
B_\rho (x) \cap \Omega = \big\{ y \in B_\rho(x) \text{ s.t. }  y\cdot \nu(x) < \psi^{(x)} (\Pi_x (y))  \big\} .
\end{align*}
\end{defn}

We remark that, if~$\alpha\in \Z$ in Definition~\ref{deZqwlaB},
the notation~$C^\alpha_{M,\rho}$ means~$C^{\alpha-1,1}_{M,\rho}$.

We now give the geometric lemma, the main result of this section.

\begin{thm} \thlabel{m9q9X2hE}
Let \(\Omega \subset \R^n\) be an open bounded set with \(\partial \Omega \in C^\alpha_{M,\rho}\), with~\(\alpha>1\), 
for some~\(M>0\) and~\(\rho \in (0,1/4)\). Moreover, suppose that \(B_r \subset \Omega \subset B_R\) with~\(1/2 \leqslant r\leqslant R \leqslant 2\).
Let~\(e\in \Sph^{n-1}\), denote by~\(\pi_\lambda\) the critical hyperplane with respect to~\(e\), and suppose that~\(\dist(0,\pi_\lambda) \leqslant 1/ 8\).

Then, \begin{align}
\Big| \big\{ x\in \Omega \triangle \Omega ' \text{ s.t. } \dist(x,\pi_\lambda) \leqslant \gamma \big\} \Big| \leqslant C\gamma (R-r)^{1-\frac 1 \alpha } \label{mKBpmgxs}
\end{align} for all \(0<\gamma \leqslant 1/4\). The constant \(C\) depends only on \(n\), \(M\), \(\rho\), and \(\alpha\). 
\end{thm}

In order to prove Theorem~\ref{m9q9X2hE},
we require two preliminary lemmata. The first lemma gives an elementary estimate of the left hand side of~\eqref{mKBpmgxs} in terms of~\(R-r\) and an error term involving~\(\dist (0,\pi_\lambda)\).

\begin{lem} \thlabel{o1pUKCHX}
Let \(\Omega \subset \R^n\) be an open bounded set with \(C^1\) boundary, and \(B_r \subset \Omega \subset B_R\) with~\(1/2 \leqslant r\leqslant R \leqslant 2\). 
Let~\(e\in \Sph^{n-1}\), denote by~\(\pi_\lambda\) the critical hyperplane with respect to~\(e\), and suppose that~\(\dist(0,\pi_\lambda) \leqslant 1/ 8\).

Then, \begin{align*}
\Big| \big\{ x\in \Omega \triangle \Omega ' \text{ s.t. } \vert x_1 -\lambda \vert \leqslant \gamma \big\} \Big| \leqslant C \gamma\big ( R-r + \gamma \vert \lambda \vert \big ) 
\end{align*} for all \(0<\gamma \leqslant 1/4\). The constant \(C\) depends only on \(n\). 
\end{lem}

\begin{proof}
First, observe that \begin{align}
\Omega \triangle \Omega ' \subset ( B_R \cup B_R' ) \setminus (B_r \cap B_r') \label{M3UmmPsi}
\end{align} where \(B_\rho '\) is the reflection of~$B_\rho$ across the critical hyperplane.
Indeed, if \(x\in \Omega \cup \Omega'\) then \(x\in \Omega \subset B_R\) or \(x\in \Omega'\subset B_R'\), so \(x\in B_R\cup B_R'\). Moreover, if \(x\in B_r\cap B_r'\) then \(x\in B_r\subset \Omega\) and \(x\in B_r'\subset \Omega'\), so~\(B_r\cap B_r'\subset \Omega \cap \Omega'\). Then~\eqref{M3UmmPsi} follows from the the fact that~\(\Omega \triangle \Omega' = (\Omega \cup \Omega' ) \setminus (\Omega \cap \Omega' )\). 

From~\eqref{M3UmmPsi}, it immediately follows that \begin{align*}
\Big| \big\{ x\in \Omega \triangle \Omega ' \text{ s.t. } \vert x_1-\lambda \vert \leqslant \gamma \big\} \Big| &\leqslant \Big| \big\{ x\in  ( B_R \cup B_R' ) \setminus (B_r \cap B_r')  \text{ s.t. } \vert x_1-\lambda \vert \leqslant \gamma \big\} \Big| .
\end{align*} Without loss of generality, we may assume that \(\lambda >0\). Then \begin{align*}
\Big| \big\{ x\in  ( B_R \cup B_R' ) &\setminus (B_r \cap B_r')  \text{ s.t. } \vert x_1-\lambda \vert \leqslant \gamma \big\} \Big|  \\
&= 2 \Big| \big\{ x\in   B_R\setminus B_r'  \text{ s.t. } \lambda-\gamma < x_1<\lambda \big\} \Big| \\
&= 2 \Big ( \Big| B_R \cap \{\lambda-\gamma < x_1<\lambda\big\}  \Big| - \Big| B_r' \cap \big\{\lambda-\gamma < x_1<\lambda\big\}  \Big|  \Big ) .
\end{align*}
Hence, via the co-area formula, we obtain that \begin{align*}
\Big| \big\{ x\in  ( B_R \cup B_R' ) &\setminus (B_r \cap B_r')  \text{ s.t. } \vert x_1-\lambda \vert \leqslant \gamma \big\} \Big| \\
&= 2 \int_{\lambda-\gamma}^\lambda \left [ \mathcal H^{n-1} \left( B_{ \sqrt{R^2-t^2}}^{n-1} \right) -  \mathcal H^{n-1} \left( B_{\sqrt{r^2-(2\lambda - t)^2}}^{n-1} \right)  \right] \dd t \\
&= 2 \omega_{n-1} \int_{\lambda-\gamma}^\lambda \left[ 
\big (R^2 - t^2 \big )^{\frac{n-1}2 } - \big (r^2 - (2\lambda - t )^2 \big )^{\frac{n-1}2 }\right] \dd t . 
\end{align*} 

Next, if \(0<a<\tau < b\) then \begin{align*}
\bigg \vert \frac{\dd }{\dd \tau } \tau^{\frac{n-1}2} \bigg \vert &= \frac{n-1} 2 \tau^{\frac{n-3}2 } \leqslant \frac{n-1} 2\begin{cases}
a^{-1/2}, &\text{if } n=2 \\
b^{\frac{n-3}2 }, &\text{if } n \geqslant 3. 
\end{cases}
\end{align*} Since \(\lambda>0\) and \(2\lambda - t \in (\lambda,\lambda+\gamma) \subset (0,3/8)\), we have that \begin{align*}
C^{-1} \leqslant r^2 - (2\lambda - t )^2 \leqslant R^2 - t^2 \leqslant C \qquad \text{ for all } t\in (\lambda-\gamma , \lambda)
\end{align*} with \(C>0\) a universal constant. Hence, \begin{align*}
\Big \vert \big (R^2 - t^2 \big )^{\frac{n-1}2 } - \big (r^2 - (2\lambda - t )^2 \big )^{\frac{n-1}2 }  \Big \vert &\leqslant C \big \vert \big (R^2 - t^2\big ) - \big ( r^2- (2\lambda - t )^2  \big ) \big \vert\\
&\leqslant C \big ( R-r +  \lambda (\lambda - t ) \big ). 
\end{align*} 

Gathering these pieces of information, we conclude that \begin{align*}
\Big| \big\{ x\in \Omega \triangle \Omega ' \text{ s.t. } \vert x_1-\lambda \vert \leqslant \gamma \big\} \Big| &\leqslant C \int_{\lambda-\gamma}^\lambda  \big(R-r +  \lambda (\lambda - t ) \big) \dd t \leqslant C \gamma\big ( R-r + \gamma \lambda \big ),
\end{align*} as required. 
\end{proof}

The purpose of the second lemma will be to reduce the proof of \thref{m9q9X2hE} to the case of graphs of functions. 

\begin{lem} \thlabel{GNZAnsDn}
Let \(\Omega \subset \R^n\) be an open bounded set with \(\partial \Omega \in C^\alpha_{M,\rho}\), with~\(\alpha>1\), 
for some~\(M>0\) and~\(\rho \in (0,1/4)\).

Then, there exists~$\epsilon_0\in(0,1)$ such that, for all~$\epsilon\in(0,\epsilon_0]$, it holds that if~\(B_{1-\varepsilon/2} \subset \Omega \subset B_{1+\varepsilon/2}\) then
there exists \(\psi_\epsilon:=\psi : B_{3/4}^{n-1} \to (0,+\infty)\)
such that~\(\psi \in C^\alpha (B_{3/4}^{n-1})\), \begin{align}
\big \|\psi-\sqrt{1-\vert x' \vert^2} \big \|_{C^1(B_{3/4}^{n-1})} \leqslant C \varepsilon^{1-\frac1\alpha}, \label{WR8AxQVQ}
\end{align}  and \begin{align}
\Omega \cap \big ( B_{3/4}^{n-1} \times (0,+\infty) \big ) = \{ x\in \R^n \text{ s.t. } 0<x_n < \psi(x') \}  . \label{Gug035Xo}
\end{align} The constant \(C\) depends only on \(n\), \(\alpha\), \(\rho\), and \(M\). 
\end{lem}

\begin{proof}
Let \(x_\star \in \partial \Omega \cap \big ( B_{3/4}^{n-1} \times (0,+\infty) \big ) \) and~\(\psi^{(x_\star)}\) be the function given in \thref{deZqwlaB} corresponding to the point~\(x_\star\). Next, let~\(A\) be a rigid motion such that~\(Ax_\star = 0\),  \(T_{x_\star}\partial\Omega\) is mapped to~ \(\R^{n-1}\times \{0\}\), and~\(\nu(x_\star) \) is mapped to~\(e_n\). To avoid any confusion, we will always use the variable~\(x\) to denote points in the original (unrotated) coordinates and~\(y\) to denote points in the new rotated coordinates, i.e. \(y=Ax\). By \thref{deZqwlaB}, in the \(y\) coordinates, \begin{align*}
y = (y' , \psi^{(x_\star)} (y') ) \qquad \text{ for all } y\in B_\rho \cap \partial \Omega
\end{align*} where \(y'=(y_1,\dots,y_{n-1})\).

Next, if \(\partial \Omega\) is contained in the closure of \(B_{1+\varepsilon/2} \setminus B_{1-\varepsilon/2}\), we have that \begin{align*}
\bigg ( 1-\frac \varepsilon 2\bigg )^2-\vert y ' \vert^2\leqslant \big( \psi^{(x_\star)} (y') \big )^2 \leqslant \bigg ( 1+\frac \varepsilon 2\bigg )^2-\vert y ' \vert^2
\end{align*} for all \(y' \in B^{n-1}_\rho \). 

Additionally, by Bernoulli's inequality, we have that \begin{align*}
0\leqslant \sqrt{\bigg ( 1+\frac \varepsilon 2\bigg )^2-\vert y ' \vert^2} -\sqrt{1-\vert y ' \vert^2} &\leqslant \frac 1 {2\sqrt{1-\vert y ' \vert^2}} \bigg ( \varepsilon + \frac{\varepsilon^2}4 \bigg ) \leqslant C \varepsilon
\end{align*} with \(C=C(\rho)>0\). Similarly, we also have that \begin{align*}
-C\varepsilon \leqslant \sqrt{\bigg ( 1-\frac \varepsilon 2\bigg )^2-\vert y ' \vert^2} -\sqrt{1-\vert y ' \vert^2} \leqslant 0 . 
\end{align*} Hence, it follows that \begin{align*}
\big \vert | \psi^{(x_\star)} (y')|  -\sqrt{1-\vert y ' \vert^2}  \big \vert &\leqslant C \varepsilon \qquad \text{for all } y' \in B^{n-1}_\rho . 
\end{align*} 

Thus, by interpolation, we have that \begin{align*}
\big \| &|\psi^{(x_\star)} (y')|  -\sqrt{1-\vert y ' \vert^2}  \big \|_{C^1( B^{n-1}_\rho)} \\
 &\leqslant C \big \| |\psi^{(x_\star)} (y')|  -\sqrt{1-\vert y ' \vert^2}  \big \|_{C^\alpha( B^{n-1}_\rho)}\big \| |\psi^{(x_\star)} (y')|  -\sqrt{1-\vert y ' \vert^2}  \big \|_{L^\infty( B^{n-1}_\rho)}^{1-\frac 1 \alpha } \\
 &\leqslant C \varepsilon^{1-\frac1{\alpha}}
\end{align*} using also that \(\partial \Omega \in C^\alpha_{M,\rho}\). 

Note that we have left the \(y\) variable in the equations above to emphasise that we are still using the rotated \(y\) coordinates. Now, returning to the original \(x\) coordinates, we observe that, by the above computation, we have that \(\partial \Omega\) is uniformly close to \(\partial B_1\) in the \(C^1\) sense, so \begin{align*}
\nu(x) \cdot e_n \geqslant C >0 \qquad \text{ for all } x\in \partial \Omega \cap \big ( B_{3/4}^{n-1} \times (0,+\infty) \big )  
\end{align*} provided that \(\varepsilon\) is sufficiently small. Thus, it follows that \(\partial \Omega\) is given by a graph with respect to the~\(e_n\) direction, that is, the claim in~\eqref{Gug035Xo} holds for some \(\psi : B_{3/4}^{n-1} \to (0,+\infty)\). We can see that~\(\psi \in C^\alpha (B_{3/4}^{n-1})\) since~\(\partial \Omega\) is~\(C^\alpha\) and we obtain the claim in~\eqref{WR8AxQVQ} by an identical interpolation argument to the one above. 
\end{proof}

We may now give the proof of \thref{m9q9X2hE}.

\begin{proof}[Proof of \thref{m9q9X2hE}] Without loss of generality, we may assume that \(e=e_1\) and \(\lambda>0\). By~\thref{o1pUKCHX}, it is enough to prove that \begin{align}
\lambda\leqslant C(n,M,\rho,\alpha)(R-r)^{1-\frac 1 \alpha } . \label{EtWD36ww}
\end{align}
 Since \(\lambda \leqslant 2\),  if there exists \(C=C(n,M,\rho,\alpha)>0\) such that \(R-r\geqslant C \) then we are done, so we may assume that \(R-r = \varepsilon\) with \(\varepsilon\) arbitrarily small. Moreover, by rescaling, we may further assume that~\(r=1-\varepsilon/2\) and~\(R=1+\varepsilon/2\). Furthermore, let~\(\psi :B_{3/4}^{n-1}\) be the function given by \thref{GNZAnsDn}. 
 
First, let us consider Case 1 of the method of moving planes. In this case, we obtain a point~\(p=(p_1,\dots,p_n)\in (\partial \Omega \cap \partial \Omega'_\lambda) \setminus \pi_\lambda \) (recall that, in this notation, \(\Omega'_\lambda = \Omega' \cap H_\lambda\), so \(p_1<\lambda\)). Hence,
we have that~\( r^2 \leqslant \vert p \vert^2 \leqslant (r+\varepsilon)^2\) and~\( r^2 \leqslant \vert Q_{\pi_\lambda} (p) \vert^2 \leqslant (r+\varepsilon)^2\) (recall that~$Q_{\pi_\lambda}$ reflects a point across~\(\pi_\lambda\)), from which it follows that \begin{align*}
\lambda (\lambda - p_1 ) \leqslant  \varepsilon (2r+\varepsilon) \leqslant C \varepsilon . 
\end{align*} If \(\lambda - p_1 \geqslant 1/4\) then we are done, so we may assume that \(\lambda - p_1 < 1/4\). 

Moreover, by rotating with respect to~\((x_2,\dots,x_n)\), we may assume without loss of generality that~\(p_2=\dots=p_{n-1} =0\) and~\(p_n>0\). In particular, this implies that \(p'=(p_1,\dots,p_{n-1} ) \in B_{3/4}^{n-1}\). 

Now, on one hand, if \(\psi_\lambda(x') := \psi(x') - \psi (2\lambda-x_1,x_2,\dots,x_{n-1}) \) then \(\psi_\lambda \geqslant 0\) for \(x'\in B_{3/4}^{n-1} \cap \{x_1<\lambda \}\) and \(\psi_\lambda(p')=0\), so \(\partial_1\psi_\lambda (p')=0\). On the other hand, by \thref{GNZAnsDn}, if \(x'' = (x_2,\dots,x_{n-1})\), \begin{align*}
\partial_1 \psi_\lambda(p') &= \partial_1 \big( \psi_\lambda - \sqrt{1-\vert x ' \vert^2}+ \sqrt{1-(2\lambda-x_1)^2 -\vert x''\vert^2 } \big ) (p') \\
&\qquad + \partial_1 \big( \sqrt{1-\vert x ' \vert^2}- \sqrt{1-(2\lambda-x_1)^2 -\vert x''\vert^2 } \big ) (p') \\
&\leqslant C \varepsilon^{1-\frac 1 \alpha}  - \bigg ( \frac{p_1}{\sqrt{1-p_1^2}} + \frac{2\lambda-p_1}{\sqrt{1-(2\lambda-p_1)^2}} \bigg ).
\end{align*} Hence, we have that \begin{equation}\label{elemcalc0}
\frac{p_1}{\sqrt{1-p_1^2}} + \frac{2\lambda-p_1}{\sqrt{1-(2\lambda-p_1)^2}}  \leqslant C \varepsilon^{1-\frac 1 \alpha} .
\end{equation} 

Moreover, we claim that \begin{equation}\label{elemcalc}
\frac{p_1}{\sqrt{1-p_1^2}} + \frac{2\lambda-p_1}{\sqrt{1-(2\lambda-p_1)^2}}  \geqslant\frac{2 \lambda}{\sqrt{1-\lambda^2}} .
\end{equation} 
To prove this, we consider the function
$$ f(t):=\frac{\lambda-t}{\sqrt{1-(\lambda-t)^2}} + \frac{\lambda+t}{\sqrt{1-(\lambda+t)^2}}$$
for~$t\in\left(-\frac12,\frac12\right)$. Notice that~$f$ is even, and therefore we restrict our analysis to the interval~$\left[0,\frac12\right)$.

By a direct calculation,
\begin{eqnarray*}
f'(t)&=& 
-\frac{(\lambda - t)^2}{(1 - (\lambda - t)^2)^{3/2}} - \frac{1}{\sqrt{1 - (\lambda - t)^2}} + \frac1{\sqrt{1 - (\lambda + t)^2}} 
+ \frac{(\lambda + t)^2}{(1 - (\lambda + t)^2)^{3/2}}\\
&=&-\frac{1}{(1 - (\lambda - t)^2)^{3/2} }+\frac{1}{(1 - (\lambda + t)^2)^{3/2}} .
\end{eqnarray*}

Now we define~$\phi(\tau):=(1-\tau)^{-3/2}$. Since~$(\lambda-t)^2\le(\lambda+t)^2\le17/32< 1$, we can write
\begin{eqnarray*}
f'(t)=\int_{ (\lambda - t)^2}^{(\lambda + t)^2} \phi'(\tau)\,d\tau=
\frac{3}{2} \int_{ (\lambda - t)^2}^{(\lambda + t)^2} (1-\tau)^{-5/2}\,d\tau\ge0.
\end{eqnarray*}
Therefore, for all~$t\in\left(0,\frac12\right)$ (and thus for all~$t\in\left(-\frac12,\frac12\right)$), we have that
$$ f(t)\ge f(0)=\frac{2\lambda}{\sqrt{1-\lambda^2}}.$$
Hence, setting~$t=\lambda-p_1\in\left(-\frac12,\frac12\right)$, we obtain the desired inequality in~\eqref{elemcalc}.

In particular, from \eqref{elemcalc} we have that
$$ \frac{p_1}{\sqrt{1-p_1^2}} + \frac{2\lambda-p_1}{\sqrt{1-(2\lambda-p_1)^2}}  \geqslant C \lambda,$$
for some~$C>0$.
Putting together this and~\eqref{elemcalc0}, we deduce that
the claim in~\eqref{EtWD36ww} holds. 

For Case 2 of the method of moving planes, we obtain a point~\(p\in \partial\Omega \cap \pi_\lambda\) at which~\(\partial \Omega\) is orthogonal to~\(\pi_\lambda\). In this case, we have that \(\partial_1\psi_\lambda(p') \geqslant 0\). {F}rom here the proof is identical to Case~1. 
\end{proof}

We will now give an explicit family of domains which demonstrate that the result obtained in \thref{m9q9X2hE}
is sharp. This also serves as a counter-example to \thref{Z08Q7917}. 

\begin{thm}\label{example61}
There exists a smooth \(1\)-parameter family \(\{\Omega_\varepsilon\}\) of open bounded subsets of \(\R^2\) such that \begin{itemize}
\item \(\partial \Omega_\varepsilon \in C^\alpha_{M,1/8}\), with \(\alpha>1\), for some \(M>0\) independent of \(\varepsilon\);
\item \(B_{1-C_1\varepsilon} \subset \Omega_\varepsilon \subset B_{1+C_1\varepsilon}\) for a universal constant \(C_1>0\);
\item if \(\pi_\lambda\) is
the critical hyperplane with respect to~\(e_1\) and~\(\gamma \in (0,1/4)\), then  \begin{align}
\Big| \big\{ x\in \Omega_\varepsilon \triangle \Omega_\varepsilon ' \text{ s.t. } \dist(x,\pi_\lambda) \leqslant \gamma\big\} \Big| \geqslant C_2\varepsilon^{1-\frac 1 \alpha } \label{EtWD36ab}
\end{align} as \(\varepsilon \to 0^+\). The constant \(C_2>0\) depends on \(n\) and \(\gamma\). 
\end{itemize} 
\end{thm}

\begin{figure}[ht]
\centering 
\includegraphics{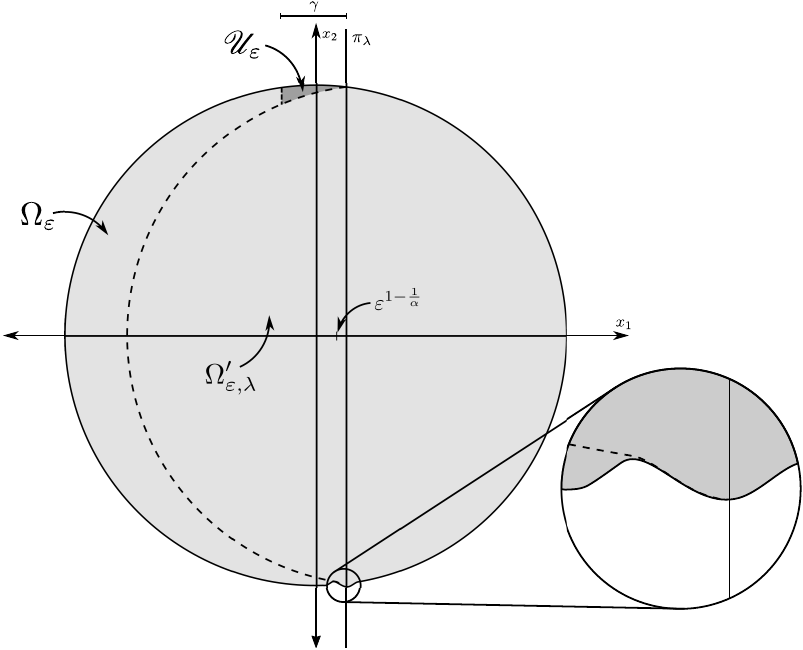}
\caption{Diagram of \(\Omega_\varepsilon\) in Theorem~\ref{example61}.}
\label{Fig1}
\end{figure}

\begin{proof}
We will define \(\Omega_\varepsilon\) by specifying its boundary, see Figure~\ref{Fig1}. 
More precisely, in the region~\(\R^2 \setminus \big ( (0,1/2) \times (-3/2,-1/2 ) \big )  \), let~\(\partial \Omega_\varepsilon = \partial B_1\).
For the definition of~\(\partial \Omega_\varepsilon\) in~\( (0,1/2) \times (-3/2,-1/2 ) \), let~\(\eta~\in C^\infty_0(\R)\) be such that \begin{align*}
\begin{cases}
\eta(\tau) = 2 \tau &\text{in }  ( - 1/4, 1/4 ), \\
\eta =0 &\text{in } \R \setminus (-1,1), \\
\eta(-\tau) =-\eta(\tau) &\text{in } \R, \\
\vert \eta \vert \leqslant 1  &\text{in } \R. 
\end{cases}
\end{align*} 
Let also
\begin{eqnarray*}
&&\eta_\varepsilon(\tau) :=
\varepsilon\eta  \left(\frac{\tau -\varepsilon^{1-\frac 1 \alpha}  } {\varepsilon^{1/\alpha}}  \right) 
\qquad
{\mbox{and}} \qquad
\psi_\varepsilon(\tau):= -\sqrt{1-\tau^2} - \eta_\varepsilon(\tau)
\end{eqnarray*}
and define the remaining portion of \(\partial \Omega_\varepsilon \) by \begin{align*}
\partial \Omega_\varepsilon \cap \Big ( (0,1/2) \times (-3/2,-1/2 ) \Big ) = \big\{ (x,\psi_\varepsilon(x) )  \text{ s.t. } x\in (0,1/2) \big\} .
\end{align*} 

We observe that \(\partial \Omega_\varepsilon\) is smooth and that \begin{align*}
\tau^2 + (\psi_\varepsilon(\tau))^2  &= 1+ 2   \eta_\varepsilon(\tau)  \sqrt{1-\tau^2}  +( \eta_\varepsilon(\tau) )^2 \leqslant 1+ C \varepsilon.
\end{align*} Hence, by Bernoulli's inequality,\begin{align*}
\sqrt{\tau^2 + \psi_\varepsilon(\tau)^2}-1  &\leqslant C\varepsilon ,
\end{align*} and similarly, we can show that \( \sqrt{\tau^2 + \psi_\varepsilon(\tau)^2}-1  \geqslant -C\varepsilon\), so \(B_{1-C\varepsilon} \subset \Omega_\varepsilon \subset  B_{1+C\varepsilon}\). Moreover, we have that \(\| \psi_\varepsilon\|_{C^\alpha((0,1/2))}\leqslant M\) for some \(M>0\) sufficiently large and independent of \(\varepsilon\), so \(\partial \Omega_\varepsilon \in C^\alpha_{M,1/8}\).

Now all that is left to be shown is~\eqref{EtWD36ab}.
First, we claim that the critical parameter~\(\lambda_\varepsilon\) satisfies \begin{align}
\lambda_\varepsilon \geqslant \varepsilon^{1-\frac1\alpha}.  \label{b6gP569L}
\end{align}
To prove this, for all \(\mu \in \R\), we define~\(\psi_{\varepsilon,\mu}(\tau) := \psi_\varepsilon(\tau) - \psi_\varepsilon(2\mu-\tau)\). 
Furthermore, we
consider
$$ \tau_\varepsilon:= \varepsilon^{1-\frac1 \alpha}-\frac14\varepsilon^{\frac1\alpha} \qquad
{\mbox{and}} \qquad \mu_\varepsilon := \varepsilon^{1-\frac1\alpha}.$$ We have that
\begin{eqnarray*}
\psi_{\varepsilon,\mu_\varepsilon}(\tau_\varepsilon) &=& 
- \Big ( \sqrt{1-\tau_\varepsilon^2}  - \sqrt{1-(2\mu_\varepsilon-\tau_\varepsilon) ^2}  \Big ) \\
&&\qquad - \varepsilon \left( 
\eta\left(\frac{\tau_\varepsilon-\varepsilon^{1-\frac1\alpha}}{\varepsilon^{1/\alpha}}\right)-\eta\left(
\frac{2\mu_\varepsilon-\tau_\varepsilon-\varepsilon^{1-\frac1\alpha}}{\varepsilon^{1/\alpha}}\right)\right)
\\&=&- \Big ( \sqrt{1-\tau_\varepsilon^2}  - \sqrt{1-(2\mu_\varepsilon-\tau_\varepsilon) ^2}  \Big ) - 2\varepsilon \bigg ( \frac{\tau_\varepsilon-\varepsilon^{1-\frac1\alpha}}{\varepsilon^{1/\alpha}}-\frac{2\mu_\varepsilon-\tau_\varepsilon-\varepsilon^{1-\frac1\alpha}}{\varepsilon^{1/\alpha}}\bigg ) \\
&=& - \frac{4\mu_\varepsilon(\mu_\varepsilon-\tau_\varepsilon)}{ \sqrt{1-\tau_\varepsilon^2}  +\sqrt{1-(2\mu_\varepsilon-\tau_\varepsilon) ^2}} 
+4\varepsilon^{1-\frac1\alpha} \big (  \mu_\varepsilon- \tau_\varepsilon  \big )  \\
&=&- 4\mu_\varepsilon(\mu_\varepsilon-\tau_\varepsilon)\left(\frac12+o(1)\right)
+4\varepsilon^{1-\frac1\alpha} \big (  \mu_\varepsilon- \tau_\varepsilon  \big ) \\
&=&2\varepsilon^{1-\frac1\alpha} \big (  \mu_\varepsilon- \tau_\varepsilon  \big ) +o\Big(\varepsilon^{1-\frac1\alpha} \big (  \mu_\varepsilon- \tau_\varepsilon  \big )\Big)\\
&=&\frac12\varepsilon+o(\varepsilon)\\
&\ge&\frac14\varepsilon ,  
\end{eqnarray*}
as soon as~$\varepsilon$ is sufficiently small.

Hence, we have that~\(\psi_{\varepsilon,\mu_\varepsilon}(\tau_\varepsilon) >0\), which implies that the reflected region of \(\Omega_{\varepsilon,\mu}\) must have left the region~\(\Omega_\varepsilon\) by the time we have reached~\(\mu=\mu_\varepsilon\), so we conclude that the critical time~\(\lambda_\varepsilon\) satisfies~\(\lambda_\varepsilon > \mu_\varepsilon\),
which gives~\eqref{b6gP569L}.

In fact, since \(\eta\) is zero outside of \((-1,1)\), it follows that \begin{align*}
\lambda_\varepsilon \leqslant C \varepsilon^{1-\frac 1\alpha},
\end{align*}
so, in particular, \(\lambda_\varepsilon \to 0\) as \(\varepsilon \to 0^+\). 

Finally, to complete the proof of~\eqref{EtWD36ab}, we will show that \begin{equation}\label{9u43c7bv9843t69}
\Big| \big\{ x\in \Omega_\varepsilon \triangle \Omega_\varepsilon ' \text{ s.t. } \dist(x,\pi_\lambda) \leqslant \gamma\big\} \Big| \geqslant C \lambda_\varepsilon. 
\end{equation} Indeed, let \(\mathscr U_\varepsilon :=  \big ( \big ( \Omega_\varepsilon \cap \{ \lambda-\gamma <x_1<\lambda \} \big ) \setminus \Omega_{\varepsilon,\lambda}'  \big ) \cap \{x_2>0\}\), see Figure~\ref{Fig1}. Observe that \begin{align*}
\Big| \big ( \Omega_\varepsilon \cap \{ \lambda-\gamma <x_1<\lambda \} \big ) \setminus \Omega_{\varepsilon,\lambda}' \Big|  \geqslant \vert \mathscr U_\varepsilon \vert . 
\end{align*} Moreover, using that \(\lambda_\varepsilon \to 0\), we have that  \begin{eqnarray*}
\vert \mathscr U_\varepsilon \vert &=& \int_{\lambda_\varepsilon-\gamma}^{\lambda_\varepsilon} \Big(\sqrt{1-\tau^2} - \sqrt{1-(2\lambda_\varepsilon-\tau)^2}\Big) \dd \tau \\
&=& \int_{\lambda_\varepsilon-\gamma}^{\lambda_\varepsilon} 
\frac{4\lambda_\varepsilon(\lambda_\varepsilon-\gamma)}{
\sqrt{1-\tau^2} + \sqrt{1-(2\lambda_\varepsilon-\tau)^2}} \dd \tau \\
&\ge&
4\lambda_\varepsilon
\int_{\lambda_\varepsilon-\gamma}^{\lambda_\varepsilon-\frac\gamma2} 
\frac{\lambda_\varepsilon-\gamma}{2} \dd \tau \\
&\ge&\frac{\gamma^2}2  \lambda_\varepsilon ,
\end{eqnarray*} which establishes~\eqref{9u43c7bv9843t69}, and thus completes the proof of Theorem~\ref{example61}.
\end{proof}


\subsection{An application of \thref{m9q9X2hE}: improvement of the exponent in \thref{oAZAv7vy}}\label{subsec:new improvement exponent}

We will now show how to use \thref{m9q9X2hE} to improve the stability exponent in \thref{m9q9X2hE}.
The key step is to use \thref{m9q9X2hE} to achieve the following modified version of \thref{7TQmUHhl}.

\begin{prop}\thlabel{Prop:new per improvement}
Let \(\Omega\) be an open bounded set with~\(C^1\) boundary and satisfying the uniform interior ball condition with radius~\(r_\Omega >0\)
and~$G$ be an open bounded set with~\(C^1\) boundary
such that~$\Omega=G+B_R$
for some~$R>0$. Let~\(f \in C^{0,1}_{\mathrm{loc}}(\R)\) be such that \(f(0)\geqslant 0\)
and \(u\) satisfies~\eqref{problem00} in the weak sense.

For \(e\in \Sph^{n-1}\), let \(\Omega '\) denote the reflection of \(\Omega\) with respect to the critical hyperplane \(\pi_\lambda\).
	
In addition, suppose that \(\partial \Omega \in C^\alpha_{M,\rho}\), with~\(\alpha>1\), 
	for some~\(M>0\) and~\(\rho \in (0,1/4)\), and that 
	\begin{equation}\label{eq:hp per improvement con r e R}
	B_{\rho_i} \subset \Omega \subset B_{\rho_e} \quad  \text{ with } \quad \frac12 \leqslant \rho_i \leqslant \rho_e \leqslant 2    
    \end{equation}
	and
\begin{equation}\label{eq:hp NEWNEW per improvement con r e R}
\rho_e - \rho_i \ge 4^{\frac{ \alpha (s+2)}{ (\alpha -1)(s+1) }} [u]_{\partial G}^{\frac{ \alpha }{ (\alpha -1)(s+1) }} .
\end{equation}	
	
	Then, 
	\begin{align*}
		\vert \Omega \triangle \Omega' \vert \leqslant \widetilde{C}_\star  (\rho_e - \rho_i)^{\frac{\alpha-1}{ \alpha (s+2)} }  \,[u]_{\partial G}^{\frac{1}{s+2}}, 
	\end{align*}
	where $\widetilde{C}_\star$ is some explicit constant depending on~$n$, $s$,
	$\alpha$, $M$, $\rho$, $\diam \Omega$, $R$, $[f]_{C^{0,1}([0,\|u\|_{L^\infty(\Omega)}])}$, and~$\|u\|_{L_s(\R^n)}$.
\end{prop}

\begin{proof}
The proof of Proposition~\ref{Prop:new per improvement} 
is a suitable modification of the one of \thref{7TQmUHhl}. More precisely, we recall~\eqref{ido4985vb6tdfghwafrikewytoiw}
and we use that~$f(0) \ge 0$
to obtain that
\begin{equation}\label{d9i43765v943675y90378659403869yuhgijfjhgflkdjvk7yt8ur}\begin{split}
\Big| \big\{ x \in (\Omega\cap H_\lambda') \setminus \Omega_\lambda' \text{ s.t. } x_1 \delta_{\partial \Omega}^s(x) >\gamma \big\} \Big|
\leqslant\;& \frac {CC_\ast^{-1}\big( f(0)+\|u\|_{L_s(\R^n)}\big)^{-1}} \gamma [u]_{\partial G}\\
\leqslant\;& \frac {CC_\ast^{-1} \|u\|_{L_s(\R^n)}^{-1}} \gamma [u]_{\partial G},
\end{split}\end{equation}
where \(C_\ast\) is as in \thref{zAPw0npM}.

Furthermore, given~$\gamma$, $\beta\in\left(0,\frac14\right]$, we see that
	\begin{eqnarray*}&&
		\Big| \big\{ x \in (\Omega\cap H_\lambda') \setminus \Omega_\lambda' \text{ s.t. } x_1 \delta_{\partial \Omega}^s(x) \leqslant \gamma \big\} \Big| \\&=&
		\Big| \big\{ x \in (\Omega\cap H_\lambda') \setminus \Omega_\lambda' \text{ s.t. } x_1 \delta_{\partial \Omega}^s(x) \leqslant \gamma, x_1< \beta \big\} \Big| \\&&\qquad + \Big| \big\{ x \in (\Omega \cap H_\lambda')
		\setminus \Omega_\lambda' \text{ s.t. } x_1 \delta_{\partial \Omega}^s(x) \leqslant \gamma , x_1 \geqslant \beta \big\} \Big| \\
		&\leqslant& \Big| \big\{ x \in \Omega^+ \text{ s.t. } x_1 < \beta \big\} \Big| +  \left| \left\{ x \in \Omega \text{ s.t. } \delta_{\partial \Omega}(x) \leqslant \left( \frac{\gamma}{\beta} \right)^{\frac 1{s}} \right\} \right| . 
	\end{eqnarray*}
Now, using \thref{m9q9X2hE} we have that
	\begin{align*}
		\Big| \big\{ x \in \Omega^+ \text{ s.t. } x_1< \beta \big\} \Big| &\leqslant 	C \,   (\rho_e - \rho_i)^{1 - \frac{1}{\alpha}}  \beta  ,
	\end{align*}
	for some~$C>0$ depending only on~$n$, $M$, $\rho$ and~$\alpha$.
	
Also, using \cite[Lemma 5.2]{MR4577340} and \cite{MR483992}, we obtain that \begin{align*}
		\left| \left\{ x \in \Omega \text{ s.t. } \delta_{\partial \Omega}(x) \leqslant \left( \frac{\gamma}{\beta} \right)^{\frac{1}{s}} \right\} \right|  &\leqslant
		\frac{2n \vert \Omega \vert }{r_\Omega} 
		\left( \frac{\gamma}{\beta} \right)^{\frac 1{s}}.
	\end{align*} 
Gathering these pieces of information, we obtain that
		\begin{eqnarray*}
	\Big| \big\{ x \in (\Omega\cap H_\lambda') \setminus \Omega_\lambda' \text{ s.t. } x_1 \delta_{\partial \Omega}^s(x) \leqslant \gamma \big\} \Big| \le
	C  (\rho_e - \rho_i)^{1 - \frac{1}{\alpha}} \beta +\frac{2n \vert \Omega \vert }{r_\Omega} \left( \frac{\gamma}{\beta}\right)^{\frac 1{s}}.
	\end{eqnarray*}	
	
{F}rom this and~\eqref{d9i43765v943675y90378659403869yuhgijfjhgflkdjvk7yt8ur}, we now deduce that
	\begin{eqnarray*}
	\vert \Omega \triangle \Omega '\vert &=& 2 \vert (\Omega\cap H_\lambda') \setminus \Omega'_\lambda \vert
	\\&\leqslant& 
		 CC_\ast^{-1} \|u\|_{L_s(\R^n)}^{-1}  \frac{[u]_{\partial G}}{\gamma} 
		 +  C(\rho_e - \rho_i)^{1 - \frac{1}{\alpha}} \beta + \frac{2n \vert \Omega \vert }{r_\Omega}\left( \frac{\gamma}{\beta}\right)^{\frac 1{s}} \\
		  &\leqslant& 
		 C\left[ \frac{[u]_{\partial G}}{\gamma} 
		 +  (\rho_e - \rho_i)^{1 - \frac{1}{\alpha}} \beta + \left( \frac{\gamma}{\beta}\right)^{\frac 1{s}}\right].
	\end{eqnarray*}
	Minimizing the expression in the last line in the variables $(\beta,\gamma)$ gives 
	$$
	\gamma \beta = \frac{[u]_{\partial G}}{(\rho_e - \rho_i)^{1 - \frac{1}{\alpha}}} 
		\qquad {\mbox{ and }}\qquad
	\beta =   \frac{[u]_{\partial G}^{\frac{1}{s+2}} }{(\rho_e - \rho_i)^{ {\frac{(\alpha -1 )(s+1)}{\alpha (s+2) }}  }},
	$$
	that is,
	$$ \gamma = \frac{[u]_{\partial G}^{\frac{s+1}{s+2}}}{(\rho_e - \rho_i)^{\frac{\alpha-1}{\alpha(s+2)} } } 
		\qquad {\mbox{ and }}\qquad
	\beta =   \frac{[u]_{\partial G}^{\frac{1}{s+2}} }{(\rho_e - \rho_i)^{ {\frac{(\alpha -1 )(s+1)}{\alpha (s+2) }}  }}.
	$$
	Notice that the assumption in~\eqref{eq:hp NEWNEW per improvement con r e R} guarantees that, with these choices,
	$\gamma$, $\beta\in\left(0,\frac14\right]$.
	
Thus, we conclude that	$$
	\vert \Omega \triangle \Omega '\vert \leqslant C \, (\rho_e - \rho_i)^{\frac{\alpha-1}{ \alpha (s+2)} }  \,[u]_{\partial G}^{\frac{1}{s+2}}.$$
This completes the proof of Proposition~\ref{Prop:new per improvement}. 
\end{proof}

Proposition~\ref{Prop:new per improvement} leads to the following statement, which is the counterpart
of Proposition~\ref{k4VRS3Fx}. 

\begin{prop} \thlabel{k4VRS3FxBIS} Let \(\Omega\) be an open bounded set with \(C^1\) boundary and satisfying the uniform interior ball condition with radius \(r_\Omega >0\) and suppose that the critical planes \(\pi_{e_i}\) with respect to the coordinate directions \(e_i\) coincide with \(\{x_i=0\}\) for every \(i=1,\dots, n\). Also, given \(e\in \Sph^{n-1}\), denote by \(\lambda_e\) the critical value associated with~\(e\) as in~\eqref{IAhmMsFq}. 

In addition, suppose that \(\partial \Omega \in C^\alpha_{M,\rho}\), with~\(\alpha>1\), 
	for some~\(M>0\) and~\(\rho \in (0,1/4)\),
	and that the assumptions in~\eqref{eq:hp per improvement con r e R} and~\eqref{eq:hp NEWNEW per improvement con r e R}
are satisfied.
	
Assume that \begin{align*}
[u]^{\frac 1{s+2} }_{\partial G} \leqslant \frac {\vert \Omega \vert } {n \widetilde{C}_\star }  
\end{align*}
where \(\widetilde{C}_\star\) is the constant in Proposition~\ref{Prop:new per improvement}.

Then,
\begin{align*}
\vert \lambda_e \vert \leqslant C \, (\rho_e - \rho_i)^{\frac{\alpha-1}{ \alpha (s+2)} } [u]_{\partial G}^{\frac 1 {s+2} }
\end{align*}
for all \(e\in \Sph^{n-1}\),
where~$C$
is some explicit constant depending on~$n$, $s$,
	$\alpha$, $M$, $\rho$, $\diam \Omega$, $R$, $[f]_{C^{0,1}([0,\|u\|_{L^\infty(\Omega)}])}$, and~$\|u\|_{L_s(\R^n)}$.
\end{prop}

We omit the proof of Proposition~\ref{k4VRS3FxBIS}, since it follows the same line as that of Proposition~\ref{k4VRS3Fx}.

Now, up to a translation, we can suppose that the critical planes \(\pi_{e_i}\) with respect to the coordinate directions \(e_i\) coincide with \(\{x_i=0\}\) for every \(i=1,\dots, n\). Hence, following \cite[Proof of Theorem 1.2]{MR3836150}, we can define
\begin{equation}\label{def:new rhoe and rhoi}
\rho_e := \max_{\partial \Omega}|x|  \qquad {\mbox{ and }}\qquad \rho_i := \min_{\partial \Omega}|x| ,
\end{equation}
and work with $\rho_e -\rho_i$ (which clearly gives an upper bound for $\rho(\Omega)$) to achieve the desired stability estimates.

In the following result, we also assume that
\begin{equation}\label{eq:scaling condition mmm...}
\rho_e=1.
\end{equation}
Notice that such an assumption is always satisfied, up to a dilation.

\begin{thm} \thlabel{thm:improvement}
Let \(G\) be an open bounded subset of~\( \R^n\) and~\(\Omega = G+B_R\) for some~\(R>0\). Furthermore, let~\(\Omega\) and~\(G\) have~\(C^1\) boundary,
and let~\(f \in C^{0,1}_{\mathrm{loc}}(\R)\) be such that~\(f(0)\geqslant 0\). Suppose that~\(u\) satisfies~\eqref{problem00} in the weak sense.

In addition, let \(\partial \Omega \in C^\alpha_{M,\rho}\), with~\(\alpha>1\), 
	for some~\(M>0\) and~\(\rho \in (0,1/4)\). 
	
	Let $\rho_e$ and $\rho_i$ be as in~\eqref{def:new rhoe and rhoi} and~\eqref{eq:scaling condition mmm...}.
	
	Then, 
	\begin{align*}
		\rho(\Omega) \leqslant \rho_e -\rho_i &\leqslant \widetilde{C} [u]_{\partial G}^{\frac{\alpha}{1+\alpha(s+1)}} ,
	\end{align*} 
	where $\widetilde{C}$ is some explicit constant depending 
	on~$n$, $s$,
	$\alpha$, $M$, $\rho$, $\diam \Omega$, $R$, $[f]_{C^{0,1}([0,\|u\|_{L^\infty(\Omega)}])}$, and~$\|u\|_{L_s(\R^n)}$.
	\end{thm}

\begin{proof}
Without loss of generality, we can assume that \eqref{eq:hp NEWNEW per improvement con r e R} is satisfied, as otherwise \thref{thm:improvement} trivially holds.

An inspection of the proof of \thref{oAZAv7vy} shows that $\rho_e$ and $\rho_i$ defined in \eqref{def:new rhoe and rhoi} satisfy
\begin{align*}
 	\rho_e - \rho_i &\leqslant C [u]_{\partial G}^{\frac 1 {s+2}} .
\end{align*}
Thus, if~$[u]_{\partial G} \le (2 C)^{-(s+2)} $,
we obtain that~$\rho_e -\rho_i \le 1/2$,
and therefore, in light of the assumption in~\eqref{eq:scaling condition mmm...},
we have that~$1/2\le\rho_i\le\rho_e\le2$, namely~\eqref{eq:hp per improvement con r e R} is satisfied.

Hence, we can employ Proposition~\ref{k4VRS3FxBIS} and obtain that
\begin{equation}
\label{eq:finalstepperimprovement}
\vert \lambda_e \vert \leqslant C \, (\rho_e - \rho_i)^{\frac{\alpha-1}{ \alpha (s+2)} } [u]_{\partial G}^{\frac 1 {s+2} }.
\end{equation}
The desired result follows by reasoning as in the proof of
\thref{oAZAv7vy}, but using \eqref{eq:finalstepperimprovement} instead of~\eqref{eq:NEWNEW OLD lambda estimate}.
\end{proof}


\section{The exponent for a family of ellipsoids} \label{z1jzKMXt}

Throughout this section, we will denote a point \(x=(x_1,\dots ,x_n) \in \R^n\) by \(x=(x_1,x')\) where\footnote{Not to be confused with the notation of the method of moving planes where an apostrophe referred to a reflection across a hyperplane.} \(x'= (x_2,\dots,x_n)\in \R^{n-1}\).

\begin{prop} \thlabel{0LlyfQlx}
Suppose that \(n\geqslant 2\) is an integer, \(s\in (0,1)\), and \(\varepsilon>0\). Define \begin{align}\label{sjiwo555ruwytui333tyeiurtfr}
\Omega_\varepsilon :=  \bigg \{ (x_1,x')\in \R^n \text{ s.t. } \frac{x_1^2}{(1+\varepsilon)^2} + \vert x ' \vert^2 <1 \bigg \}.
\end{align} 

Then, for \(0<\varepsilon<1/4\), there exists a 1-parameter family \(G_\varepsilon \subset \Omega_\varepsilon\) with smooth boundary such that \begin{align*}
G_\varepsilon + B_{1/2} = \Omega_\varepsilon.
\end{align*}

Moreover, let \(u_\varepsilon \in C^2(\Omega_\varepsilon) \cap C^s(\R^n)\) be the unique solution to \begin{align}
\begin{PDE}
(-\Delta)^s u_\varepsilon &=1 &\text{in } \Omega_\varepsilon, \\
u_\varepsilon &=0 &\text{in } \R^n \setminus \Omega_\varepsilon .
\end{PDE} \label{kaPuQECN}
\end{align} 

Then, \begin{align}
\lim_{\varepsilon \to 0^+} \frac{[u_\varepsilon]_{\partial G_{\varepsilon}}}{\rho(\Omega_\varepsilon)} &= s \gamma_{n,s} \bigg ( \frac34 \bigg )^{s-1}  , \label{83Wl82Op}
\end{align}
where
\begin{align}\label{gammaenneesse}
\gamma_{n,s} := \frac{2^{-2s} \Gamma \big ( \frac n2 \big )} {\Gamma \left(\frac{n+2s}2 \right) \Gamma (1+s)}.
\end{align} 
\end{prop}

The existence of the 1-parameter family \(G_\varepsilon\) as in \thref{0LlyfQlx} is an easy corollary of the following geometric lemma.

\begin{lem} \thlabel{jUJs9Upe}
Let \(n\geqslant 2\) be an integer and suppose that \(\Omega\) is an open bounded subset of \(\R^n\) that satisfies the uniform interior ball condition with radius \(r_\Omega>0\). If \(\delta:\Omega \to \R\) is the distance to the boundary of \(\Omega\) and \(\Omega^{\rho} := \{ x\in \Omega \text{ s.t. } \delta(x)>\rho\}\) then for all \(\rho \in (0,r_\Omega)\), \begin{align*}
\Omega^\rho+B_\rho = \Omega.
\end{align*} Moreover, if the boundary of \(\Omega\) is \(C^k\) for some integer \( k \ge 2 \) then,
for all \(\rho \in (0,r_\Omega)\), the boundary of \( \Omega^{\rho}\) is also \(C^k\). 
\end{lem}

\begin{proof}
We will begin by showing that \(\Omega^{\rho}+B_\rho \subset \Omega\). If \(z\in \Omega^{\rho}+B_\rho \) then \(z=x+y \) for some \(x\in \Omega^{\rho}\) and \(y \in B_\rho\). Let \(R>0\) be the largest value possible such that \(B_R(x)\subset \Omega\). Then \(R=\delta(x)\), so \(R >\rho\). Hence, \begin{align*}
\vert z - x\vert = \vert y \vert <\rho < R,
\end{align*} so \(z\in B_R(x) \subset \Omega\). 

Now, let us show that \(\Omega \subset \Omega^{\rho}+B_\rho\). Let \(z\in \Omega\). If \(z \in \Omega^{\rho}\) then \(z=z+0\in \Omega^{\rho}+B_\rho\) and we are done, so we may assume that \(z\in \Omega \setminus \Omega^{\rho}\). Now, let \(R>0\) be the largest value possible such that~\(B_R(z)\subset \Omega\) and let \(p\in \partial \Omega\) be a point at which \(B_R(z)\) touches \(\partial \Omega\). Define the unit vector \begin{align*}
\nu (p) := \frac{p-z}{\vert p - z \vert}. 
\end{align*} Next, since \(B_R(z)\) is an interior touching ball, there exists some~\(\bar z \in \Omega \) such that~\(B_{r_\Omega}(\bar z)\) also touches~\(\partial \Omega\) at~\(p\) and \(p\), \(z\), and \(\bar z\) are collinear. Now, let \(\mu \in (0, \min \{ r_\Omega - \rho , R\})\) and define \begin{align*}
x := z - (\rho - R+\mu ) \nu (p) .
\end{align*} Since \( z = \bar z +(r_\Omega - R) \nu (p)\), it follows that \( \vert x-\bar z\vert = \vert r_\Omega-\rho-\mu \vert < r_\Omega\), so \(x\in B_{r_\Omega}(\bar z)\). Hence, \(x \in \Omega\). Moreover, \(\vert x -p\vert = \rho +\mu > \rho\), so \(x\in \Omega^{\rho}\) and \(\vert z-x\vert = \rho -R +\mu <\rho\), so \(z-x\in B_\rho\). Thus, we have that \(z =x +(z-x)\in \Omega^{\rho} +B_\rho\) as required. This completes the first part of the proof.

The fact that the boundary of \( \Omega^{\rho}\) is \(C^k\) for any \(\rho \in (0,r_\Omega)\) is a simple consequence of \cite[Lemma 14.16]{MR1814364}. Indeed, following the proof of \cite[Lemma 14.16]{MR1814364}, we see that \(\delta\in C^k(\Gamma^{r_\Omega})\) where 
$$ \Gamma^{r_\Omega} := \{ x\in \Omega \text{ s.t. } \delta(x) < r_\Omega\}.$$ Then the boundary of~\(\Omega^{\rho}\) is simply the level set~\(\Omega \cap \{\delta=\rho\}\) which is contained in~\(\Gamma^{r_\Omega}\), so~\(\partial \Omega^{\rho}\) is~\(C^k\). 
\end{proof}

From \thref{jUJs9Upe}, we immediately obtain the proof of the first part of \thref{0LlyfQlx}: 

\begin{cor} \thlabel{FOHX1qV4}
Suppose that \(n\geqslant 2\) is an integer, \(s\in (0,1)\), \(\varepsilon>0\),
and~$\Omega_\varepsilon$ be as in~\eqref{sjiwo555ruwytui333tyeiurtfr}. 

Then, for~\(0<\varepsilon<1/4\), there exists a 1-parameter family \(G_\varepsilon \subset \Omega_\varepsilon\) with smooth boundary such that \begin{align*}
G_\varepsilon + B_{1/2} =  \Omega_\varepsilon.
\end{align*}
\end{cor}

\begin{proof}
It is easy to check that \(\Omega_\varepsilon\) satisfies the uniform interior ball condition with uniform radius given by \(r_{\Omega_\varepsilon}=(1+\varepsilon)^{-1}\). Hence, since \(\varepsilon<1/4\), we have that~\(r_{\Omega_\varepsilon}>1/2\), so if~\(G_\varepsilon := \Omega_\varepsilon^\rho\) with~\(\rho = 1/2\) (i.e. \(G_\varepsilon := \{ x\in \Omega_\varepsilon \text{ s.t. } \dist (x, \partial \Omega_\varepsilon)>1/2\}\)) then \thref{jUJs9Upe} implies that~\(G_\varepsilon\) satisfies all the required properties. 
\end{proof}

The remainder of this section will be spent proving the second part of \thref{0LlyfQlx}. In theory, this is relatively simple: since \(\Omega_\varepsilon\) is an ellipsoid, the solution to the torsion problem~\eqref{kaPuQECN} is known explicitly, see \cite{MR4181195}, and is given by \begin{align*}
u_\varepsilon(x):= \gamma_{n,s,\varepsilon}  \bigg (1 - \frac{x_1^2}{(1+\varepsilon)^2} - \vert x ' \vert^2 \bigg )^s_+ 
\end{align*} where \begin{align*}
\gamma_{n,s,\varepsilon} := \frac{\gamma_{n,s}}{(1+\varepsilon)  {}_2F_1 \big ( \frac{n+2s}2 , \frac 12 ; \frac n2  ; 1-(1+\varepsilon)^2 \big)}
\end{align*} and~$\gamma_{n,s}$ is given by~\eqref{gammaenneesse}.

Here above \({}_2F_1 \) is the Hypergeometric function, see~\cite{MR1225604} for more details. Note that
\({}_2F_1 (a,b;c;0)=1\), so~\(\gamma_{n,s,0}=\gamma_{n,s}\), and~\(\gamma_{n,s}\) is precisely the constant for which~\(u_0 (x) := \gamma_{n,s}(1-\vert x \vert^2)^s_+\) satisfies~\((-\Delta )^su_0=1\) in \(B_1\).

Now, it is easy to check that \(\rho(\Omega_\varepsilon)=\varepsilon\), so to prove the second part of \thref{0LlyfQlx}, one simply has to find the first order expansion of \([u_{\varepsilon}]_{\partial G_\varepsilon}\) in \(\varepsilon\). However, in practice, this is not trivial for a few reasons: one, parametrizations of \(G_\varepsilon\) are algebraically complicated; two, \([u_{\varepsilon}]_{\partial G_\varepsilon}\) is defined in terms of a supremum, which, in general, does not commute well with limits; and three, the supremum is over a quotient whose numerator and denominator both depend on \(\varepsilon\) and whose denominator is not bounded away from zero.

We will now state a technical result that will allow us to postpone addressing these difficulties and to proceed directly to the proof of \thref{0LlyfQlx}. 

\begin{lem} \thlabel{M2fW7tQX} Let \(\varepsilon \in[0,1)\) and define \(a_\varepsilon,b_\varepsilon :[0,+\infty) \to \R\) and \(\phi_\varepsilon : B_1^{n-1} \to \R^n\) as follows \begin{equation}\label{defperf333}\begin{split}
a_\varepsilon (\tau):=\;& 1+\varepsilon - \frac 1 {2 \sqrt{1+ \big ( (1+\varepsilon)^2-1\big ) \tau^2 }} ,\\
b_\varepsilon (\tau):=\;&  1- \frac {1+\varepsilon} {2 \sqrt{1+ \big ( (1+\varepsilon)^2-1\big ) \tau^2 }} ,\\
{\mbox{and}} \qquad \phi_\varepsilon (r):=\;& \Big(a_\varepsilon(\vert r \vert) \sqrt{1-\vert r\vert^2} , b_\varepsilon(\vert r \vert ) r \Big). 
\end{split}\end{equation} 

Then, \begin{align*}
\bigg \vert &\frac{\vert (u_\varepsilon \circ \phi_\varepsilon)(r)-(u_\varepsilon\circ \phi_\varepsilon)(\tilde r ) \vert }{\varepsilon \vert \phi_\varepsilon(r) - \phi_\varepsilon(\tilde r) \vert } - \frac12 s \gamma_{n,s}\left(\frac34\right)^{s-1} \frac{\big \vert  \vert r \vert^2 - \vert \tilde r \vert^2 \big \vert }{\vert \phi_0(r) - \phi_0(\tilde r)\vert}\bigg \vert \leqslant C \varepsilon .
\end{align*} The constant \(C>0\) depends only on \(n\) and \(s\).  
\end{lem}

In the statement of~\thref{M2fW7tQX} and in what follows, we use the notation~$B^{n-1}_1$ for the unit ball in~$\R^{n-1}$.

We will withhold the proof of \thref{M2fW7tQX} until the end of the section. Given \thref{M2fW7tQX}, we can now complete the proof of \thref{0LlyfQlx}.

\begin{proof}[Proof \thref{0LlyfQlx}] The proof of the first statement regarding the existence of \(G_\varepsilon\) is the subject of \thref{FOHX1qV4}, so we will focus on proving~\eqref{83Wl82Op}. Since the ellipsoid is symmetric with respect to reflections across the plane \(\{x_1=0\}\), it follows that \(G_\varepsilon\) is too. Hence, we have that \begin{align*}
[u_\varepsilon]_{\partial G_\varepsilon} &= \sup_{\substack{x,y\in \partial G_\varepsilon^+ \\ x\neq y } } \bigg \{ \frac{\vert u_\varepsilon(x)-u_\varepsilon(y) \vert }{\vert x - y \vert } \bigg \}.
\end{align*} Here, we are using the notation that, given \(A\subset \R^n\), \(A^+ := A \cap \{x_1>0\}\).

Next, the upper half ellipsoid \(\Omega_\varepsilon^+\) can be parametrized by \(\bar \phi_\varepsilon : B_1^{n-1} \to \Omega_\varepsilon^+\) defined
by~\(\bar \phi_\varepsilon (r) := ((1+\varepsilon)\sqrt{1-\vert r \vert^2},r)\). Moreover, the outward pointing normal to~\(B_1\) is given by
\begin{align*}
\nu_\varepsilon(x) &:= \frac1{\sqrt{\frac{x_1^2}{(1+\varepsilon)^4} + \vert x' \vert^2}} \bigg ( \frac{x_1}{(1+\varepsilon)^2} , x' \bigg ),
\end{align*} so a parametrization of \(\partial G_\varepsilon^+\) is \(\phi_\varepsilon :B_1^{n-1} \to \R^n\) defined by~\(\phi_\varepsilon := \bar \phi_\varepsilon-\frac12 (\nu_\varepsilon \circ \bar \phi_\varepsilon) \), that is, \begin{align*}
\phi_\varepsilon(r) &= \big ( (1+\varepsilon)\sqrt{1-\vert r \vert^2}, r \big ) - \frac{1+\varepsilon}{2\sqrt{1+((1+\varepsilon)^2-1)\vert r \vert^2}} \bigg ( \frac{\sqrt{1-\vert r \vert^2}}{1+\varepsilon} ,r  \bigg )\\ &= \big(a_\varepsilon(\vert r\vert ) \sqrt{1-\vert r \vert^2} , b_\varepsilon(\vert r \vert) r\big) 
\end{align*} in the notation of \thref{M2fW7tQX}.

Hence, \begin{align*}
[u_\varepsilon]_{\partial G_\varepsilon} &= \sup_{\substack{r,\tilde r\in B_1^{n-1} \\ r\neq \tilde r } } \left\{  \frac{\vert (u_\varepsilon \circ \phi_\varepsilon)(r)-(u_\varepsilon\circ \phi_\varepsilon)(\tilde r ) \vert }{\vert \phi_\varepsilon(r) - \phi_\varepsilon(\tilde r) \vert } \right\}.
\end{align*} Next, using \thref{M2fW7tQX} and the second triangle inequality in the \(\sup\) norm (over \(r,\tilde r\in B_1^{n-1}\)), we obtain that  \begin{align*}
\bigg \vert \frac{[u_\varepsilon]_{\partial G_\varepsilon}}{\varepsilon} -&\frac12 s \gamma_{n,s}\left(\frac34\right)^{s-1}  \sup_{\substack{r,\tilde r\in B_1^{n-1} \\ r\neq \tilde r } } \bigg \{ \frac{\big \vert  \vert r \vert^2 - \vert \tilde r \vert^2 \big \vert }{\vert \phi_0(r) - \phi_0(\tilde r)\vert} \bigg \}\bigg \vert \\
&\leqslant \sup_{\substack{r,\tilde r\in B_1^{n-1} \\ r\neq \tilde r } } \bigg \vert \frac{\vert (u_\varepsilon \circ \phi_\varepsilon)(r)-(u_\varepsilon\circ \phi_\varepsilon)(\tilde r ) \vert }{\varepsilon \vert \phi_\varepsilon(r) - \phi_\varepsilon(\tilde r) \vert } - \frac12 s \gamma_{n,s}\left(\frac34\right)^{s-1}  \frac{\big \vert  \vert r \vert^2 - \vert \tilde r \vert^2 \big \vert }{\vert \phi_0(r) - \phi_0(\tilde r)\vert}\bigg \vert  \\
&\leqslant C\varepsilon,
\end{align*} so, by sending \(\varepsilon \to 0^+\), we find that \begin{align}
 \lim_{\varepsilon \to 0^+} \frac{[u_\varepsilon]_{\partial G_\varepsilon}}{\varepsilon} &= \frac12 s \gamma_{n,s}\left(\frac34\right)^{s-1} \sup_{\substack{r,\tilde r\in B_1^{n-1} \\ r\neq \tilde r } } \bigg \{ \frac{\big \vert  \vert r \vert^2 - \vert \tilde r \vert^2 \big \vert }{\vert \phi_0(r) - \phi_0(\tilde r)\vert} \bigg \} .\label{vTr8Kp6E}
\end{align}

Finally, we claim that \begin{align}
\sup_{\substack{r,\tilde r\in B_1^{n-1} \\ r\neq \tilde r } } \bigg \{ \frac{\big \vert  \vert r \vert^2 - \vert \tilde r \vert^2 \big \vert }{\vert \phi_0(r) - \phi_0(\tilde r)\vert} \bigg \} = 2. \label{GPAI62Da}
\end{align} Indeed, if \begin{align*}
Q:=\frac{\big (  \vert r \vert^2 - \vert \tilde r \vert^2 \big )^2 }{\vert \phi_0(r) - \phi_0(\tilde r)\vert^2}
\end{align*} then \begin{align*}
Q=   \frac{4\big (  \vert r \vert^2 - \vert \tilde r \vert^2 \big )^2}{\big (\sqrt{1-\vert r \vert^2}-\sqrt{1-\vert \tilde r \vert^2} \big )^2+\vert r - \tilde r \vert^2} \leqslant  \frac{4\big (  \vert r \vert^2 - \vert \tilde r \vert^2 \big )^2 }{(\sqrt{1-\vert r \vert^2}-\sqrt{1-\vert \tilde r \vert^2})^2+\big ( \vert r\vert  - \vert \tilde r \vert \big )^2}.
\end{align*} Next, let \(t,\tilde t \in (0,\pi/2)\) be such that \(\vert r\vert = \cos t\) and \(\vert \tilde r\vert = \cos \tilde t\). Then \begin{align*}
Q &\leqslant  \frac{ 4(  \cos^2t - \cos^2 \tilde t  )^2 }{(\sin t -\sin \tilde t)^2+ ( \cos t   - \cos \tilde t )^2}.
\end{align*} Then, via elementary trigonometric identities, \begin{eqnarray*}&&
(  \cos^2t - \cos^2 \tilde t  )^2= 4 \sin^2  ( t+\tilde t  ) \sin^2 \bigg ( \frac{t-\tilde t}2  \bigg )  \cos^2 \bigg ( \frac{t-\tilde t}2  \bigg )  \\ \text{ and} &&
(\sin t -\sin \tilde t)^2+ ( \cos t   - \cos \tilde t )^2 = 4 \sin^2 \bigg ( \frac{t-\tilde t}2 \bigg ),
\end{eqnarray*} so \begin{align*}
Q \leqslant 4 \sin^2  ( t+\tilde t  )  \cos^2 \bigg ( \frac{t-\tilde t}2  \bigg ) \leqslant 4,
\end{align*} which shows that the left-hand side of~\eqref{GPAI62Da} is less than or equal to the right-hand side of~\eqref{GPAI62Da}. To show that the reverse inequality is also true, consider \(r:=e_1/\sqrt{2}\) and \(\tilde r := ( 1/\sqrt 2 +h)e_1\). Then \begin{align*}
\frac{\big \vert  \vert r \vert^2 - \vert \tilde r \vert^2 \big \vert }{\vert \phi_0(r) - \phi_0(\tilde r)\vert} &= \frac{2\vert h \vert \left(h+\sqrt{2}\right)}{\sqrt{1-h\sqrt 2-\sqrt{1-2 \sqrt{2} h-2 h^2}}} \to 2
\end{align*} as \(h \to 0\) which completes the proof of~\eqref{GPAI62Da}. Thus, recalling that \(\rho(\Omega_\varepsilon) = \varepsilon\),~\eqref{83Wl82Op} follows from~\eqref{vTr8Kp6E} and~\eqref{GPAI62Da}. 
\end{proof}

We will now give the proof of \thref{M2fW7tQX}. We will do this over two lemmata.

\begin{lem} \thlabel{c3Ge7XSV}
Let \(\varepsilon \in[ 0,1)\) and  \(a_\varepsilon\), \(b_\varepsilon\), and \(\phi_\varepsilon\) be as in \thref{M2fW7tQX}. 

Then, for all~\(r,\tilde r\in B_1^{n-1}\) with~\(r\neq \tilde r \), \begin{align}
\bigg \vert \frac{\vert\phi_0(r)-\phi_0(\tilde r)\vert }{\vert \phi_\varepsilon(r)-\phi_\varepsilon(\tilde r)\vert} -1 \bigg \vert \leqslant C \varepsilon .\label{Wy0X8YzO}
\end{align} The constant \(C>0\) depends only on \(n\). 
\end{lem}

\begin{proof}
Observe that \begin{equation}\label{JqHdxdks}\begin{split}
\bigg \vert \frac{\vert\phi_0(r)-\phi_0(\tilde r)\vert }{\vert \phi_\varepsilon(r)-\phi_\varepsilon(\tilde r)\vert} -1 \bigg \vert  & = \bigg \vert \frac{\vert\phi_0(r)-\phi_0(\tilde r)\vert-\vert \phi_\varepsilon(r)-\phi_\varepsilon(\tilde r)\vert}{\vert \phi_\varepsilon(r)-\phi_\varepsilon(\tilde r)\vert}\bigg \vert  \\
&\leqslant \frac{\vert (\phi_0-\phi_\varepsilon)(r)- (\phi_0-\phi_\varepsilon)(\tilde r) \vert}{\vert \phi_\varepsilon(r)-\phi_\varepsilon(\tilde r)\vert}. 
\end{split}\end{equation}
Next, observe that \begin{align*}
\phi_0(r) &= A_\varepsilon(\vert r \vert ) \phi_\varepsilon (r)
\end{align*} where, for each \(\tau\in \R\), \(A_\varepsilon(\tau)\) is the \(n\times n\) matrix given by \begin{align*}
A_\varepsilon(\tau) &:= \begin{pmatrix}
\frac1{2a_\varepsilon(\tau)} & 0 \\ 0 & \frac1{2b_\varepsilon(\tau)} I_{n-1}
\end{pmatrix}.
\end{align*}  Here \(I_k\) denotes the \(k\times k\) identity matrix and we are thinking of \(\phi_\varepsilon\) and~\(\phi_0\) as column vectors.

It will be useful in the future to note that \(a_\varepsilon\) and \(b_\varepsilon\) are strictly increasing for \(\tau\geqslant 0\) (see~\eqref{mAVkC45x} below for \(a_\varepsilon\), the computation for \(b_\varepsilon\) is analogous), so, for all \(\tau \in [0,1)\), \begin{equation}
\begin{split}&
\frac12 +\varepsilon \leqslant a_\varepsilon(\tau) \leqslant 1 +\varepsilon - \frac 1{2(1+\varepsilon)}\leqslant \frac12 +C\varepsilon \\{\mbox{and }}\quad&
\frac12(1-\varepsilon) \leqslant b_\varepsilon(\tau)  \leqslant \frac 12 .
\end{split} \label{ijAZGImZ}
\end{equation} In particular, when \(\varepsilon\) is small, \(A_\varepsilon(\tau)\) is well-defined since \(a_\varepsilon,b_\varepsilon>0\). Hence, it follow from~\eqref{JqHdxdks} that \begin{equation}\label{bgfHZN8C}\begin{split}
&\bigg \vert \frac{\vert\phi_0(r)-\phi_0(\tilde r)\vert }{\vert \phi_\varepsilon(r)-\phi_\varepsilon(\tilde r)\vert} -1 \bigg \vert  \\
\leqslant\;& \frac{\vert (I_n-A_\varepsilon(\vert r \vert )) \phi_\varepsilon(r)- (I_n-A_\varepsilon(\vert \tilde r \vert ) )\phi_\varepsilon(\tilde r) \vert}{\vert \phi_\varepsilon(r)-\phi_\varepsilon(\tilde r)\vert}\\
\leqslant\;& \frac{\vert (I_n-A_\varepsilon(\vert r \vert )) \phi_\varepsilon(r)- (I_n-A_\varepsilon(\vert  r \vert ) )\phi_\varepsilon(\tilde r) \vert}{\vert \phi_\varepsilon(r)-\phi_\varepsilon(\tilde r)\vert} 
+\frac{\vert (A_\varepsilon(\vert \tilde r \vert )-A_\varepsilon(\vert r \vert )) \phi_\varepsilon(\tilde r) \vert}{\vert \phi_\varepsilon(r)-\phi_\varepsilon(\tilde r)\vert}   \\
\leqslant\;& \|I_n-A_\varepsilon(\vert r \vert )\|_{\textrm{OP}} + \frac{\| A_\varepsilon(\vert \tilde r \vert )-A_\varepsilon(\vert r \vert )\|_{\mathrm{OP}}  }{\vert \phi_\varepsilon(r)-\phi_\varepsilon(\tilde r)\vert}
\vert \phi_\varepsilon(\tilde r) \vert,
\end{split}\end{equation}
where \(\| \cdot \|_{\mathrm{OP}}\) denotes the matrix operator norm.

It follows from~\eqref{ijAZGImZ} that \begin{align}
\|I_n-A_\varepsilon(\tau )\|_{\textrm{OP}} &= \max \bigg \{\bigg \vert 1 - \frac 1{2a_\varepsilon(\tau)}\bigg \vert ,  \bigg \vert 1 - \frac 1{2b_\varepsilon(\tau)}\bigg \vert \bigg \}\leqslant C\varepsilon . \label{ZiPmtbEz}
\end{align} Moreover, if \(\| \cdot \|_{\mathrm F}\) denotes the Frobenius norm, then \begin{align*}
\| A_\varepsilon(\vert \tilde r \vert )-A_\varepsilon(\vert r \vert )\|_{\mathrm{OP}}^2 &\leqslant C \| A_\varepsilon(\vert \tilde r \vert )-A_\varepsilon(\vert r \vert )\|_{\mathrm{F}}^2 \\
&= C \left[\left( \frac 1 {a_\varepsilon(\vert r \vert) } - \frac 1 {a_\varepsilon(\vert \tilde r \vert) } \right)^2 +   \left( \frac 1 {b_\varepsilon(\vert r \vert) } - \frac 1 {b_\varepsilon(\vert \tilde r \vert) } \right)^2 \right]\\
&= C \left( \frac{\big (a_\varepsilon(\vert r \vert) - a_\varepsilon(\vert \tilde r \vert) \big )^2}{a_\varepsilon(\vert r \vert)^2 a_\varepsilon(\vert \tilde r \vert)^2} + \frac{\big (b_\varepsilon(\vert r \vert) - b_\varepsilon(\vert \tilde r \vert) \big )^2}{b_\varepsilon(\vert r \vert)^2 b_\varepsilon(\vert \tilde r \vert)^2} \right)\\
&\leqslant  C \Big ( \big (a_\varepsilon(\vert r \vert) - a_\varepsilon(\vert \tilde r \vert) \big )^2 + \big (b_\varepsilon(\vert r \vert) - b_\varepsilon(\vert \tilde r \vert) \big )^2 \Big ), 
\end{align*} where we used~\eqref{ijAZGImZ} to obtain the final inequality.

We also have that \begin{align*}
\vert \phi_\varepsilon(r)-\phi_\varepsilon(\tilde r)\vert^2 &= \Big ( a_\varepsilon(\vert r \vert ) \sqrt{1-\vert r \vert^2}-a_\varepsilon(\vert \tilde r \vert ) \sqrt{1-\vert \tilde r \vert^2} \Big )^2+\vert b_\varepsilon (\vert r \vert ) r - b_\varepsilon(\vert \tilde r \vert ) \tilde r \vert^2 \\
&\geqslant \Big ( a_\varepsilon(\vert r \vert ) \sqrt{1-\vert r \vert^2}-a_\varepsilon(\vert \tilde r \vert ) \sqrt{1-\vert \tilde r \vert^2} \Big )^2+\big ( b_\varepsilon (\vert r \vert ) \vert r \vert  - b_\varepsilon(\vert \tilde r \vert ) \vert  \tilde r \vert \big )^2.
\end{align*}

We claim that \begin{equation}
\begin{split}&
\big \vert a_\varepsilon(\vert r \vert) - a_\varepsilon(\vert \tilde r \vert) \big \vert \leqslant C \varepsilon \big \vert a_\varepsilon(\vert r \vert ) \sqrt{1-\vert r \vert^2}-a_\varepsilon(\vert \tilde r \vert ) \sqrt{1-\vert \tilde r \vert^2} \big \vert  \\ \text{and}\qquad &
\big \vert b_\varepsilon(\vert r \vert) - b_\varepsilon(\vert \tilde r \vert) \big \vert \leqslant C \varepsilon \big \vert  b_\varepsilon (\vert r \vert ) \vert r \vert  - b_\varepsilon(\vert \tilde r \vert ) \vert  \tilde r \vert \big \vert.
\end{split} \label{fERN80H1}
\end{equation} We will prove the inequality in~\eqref{fERN80H1} involving \(a_\varepsilon\), the proof of the inequality involving \(b_\varepsilon\) is entirely analogous. Let \(\bar a_\varepsilon(\tau ) := a_\varepsilon (\tau ) \sqrt{1-\tau^2}\). Then, by Cauchy's Mean Value Theorem, \begin{align}
\bigg \vert \frac{ a_\varepsilon(\vert r \vert) - a_\varepsilon(\vert \tilde r \vert) }{\bar a_\varepsilon(\vert r \vert) - \bar a_\varepsilon(\vert \tilde r \vert) }\bigg \vert &\leqslant \bigg \| \frac{a_\varepsilon'}{\bar a_\varepsilon'} \bigg \|_{L^\infty((0,1))}. \label{w3PCZfMC}
\end{align} On one hand,  \begin{align*}
a'_\varepsilon(\tau) &= \frac12 \varepsilon(\varepsilon+2) \tau \big ( 1 + \big ( (1+\varepsilon)^2-1 \big ) \tau^2  \big )^{-3/2},
\end{align*}so, for all \(\tau \in [0,1)\), \begin{align}
C^{-1} \varepsilon \tau   \leqslant a'_\varepsilon(\tau) \leqslant C \varepsilon \tau. \label{mAVkC45x}
\end{align} On the other hand, using~\eqref{ijAZGImZ} and~\eqref{mAVkC45x}, we have that \begin{align*}
\bar a_\varepsilon '(\tau ) &= a_\varepsilon '(\tau ) \sqrt{1-\tau^2} - \frac{\tau a_\varepsilon(\tau)}{\sqrt{1-\tau^2}} \leqslant C(\varepsilon -1)\tau  <0 . 
\end{align*}Hence, \(\vert \bar a_\varepsilon '(\tau ) \vert =  -a_\varepsilon'(\tau) \geqslant C \tau\), so \begin{align*}
 \bigg \vert \frac{a_\varepsilon'(\tau)}{\bar a_\varepsilon'(\tau)} \bigg \vert &=   \frac{a_\varepsilon'(\tau)}{\vert \bar a_\varepsilon'(\tau) \vert } \leqslant C \varepsilon, \qquad \text{for all } \tau \in (0,1),
\end{align*} which, along with~\eqref{w3PCZfMC}, implies the inequality for \(a_\varepsilon\) in~\eqref{fERN80H1}. 

Gathering these pieces of information, we obtain that
\begin{equation}\begin{split}
&\frac{\| A_\varepsilon(\vert \tilde r \vert )-A_\varepsilon(\vert r \vert )\|_{\mathrm{OP}}}{\vert \phi_\varepsilon(r)-\phi_\varepsilon(\tilde r)\vert}
\leqslant\\ &\frac{
C \sqrt{  \big ((a_\varepsilon(\vert r \vert) - a_\varepsilon(\vert \tilde r \vert) \big )^2 + \big (b_\varepsilon(\vert r \vert) - b_\varepsilon(\vert \tilde r \vert) \big )^2 }
}{ \sqrt{\Big ( a_\varepsilon(\vert r \vert ) \sqrt{1-\vert r \vert^2}-a_\varepsilon(\vert \tilde r \vert ) \sqrt{1-\vert \tilde r \vert^2} \Big )^2+\big ( b_\varepsilon (\vert r \vert ) \vert r \vert  - b_\varepsilon(\vert \tilde r \vert ) \vert  \tilde r \vert \big )^2}}
\\
\leqslant\;& C \varepsilon 
.\label{R7PfHtVD}
\end{split}\end{equation}

Finally, since \(G_\varepsilon \subset \Omega_\varepsilon \subset B_{1+\varepsilon}\), we have that
\begin{align}
\vert \phi_\varepsilon(\tilde r) \vert &\leqslant 1+\varepsilon \leqslant C. \label{5nLBFgsK}
\end{align} Thus,~\eqref{Wy0X8YzO} follows from~\eqref{bgfHZN8C},~\eqref{ZiPmtbEz},~\eqref{R7PfHtVD}, and~\eqref{5nLBFgsK}.
\end{proof}

\begin{lem} \thlabel{F3lTRgfZ} 
Let \(\varepsilon \in[ 0,1)\) and  \(a_\varepsilon\), \(b_\varepsilon\), and \(\phi_\varepsilon\) be as in \thref{M2fW7tQX}. 

Then, for all~\(r,\tilde r\in B_1^{n-1}\) with~\(r\neq \tilde r \),  \begin{align*}
\bigg \vert \frac{\vert (u_\varepsilon \circ \phi_\varepsilon)(r)- (u_\varepsilon \circ \phi_\varepsilon)(\tilde r)\vert}{\varepsilon\vert \phi_0(r) - \phi_0(\tilde r)\vert} -
\frac12 s \gamma_{n,s}\left(\frac34\right)^{s-1} 
\frac{\big \vert  \vert r \vert^2 - \vert \tilde r \vert^2 \big \vert }{\vert \phi_0(r) - \phi_0(\tilde r)\vert} \bigg \vert \leqslant C \varepsilon.
\end{align*}
The constant \(C>0\) depends only on \(n\) and \(s\). 
\end{lem}

\begin{proof}
Let \begin{align*}
v_\varepsilon(x) &:= \frac 1 {\gamma_{n,s,\varepsilon}} u_\varepsilon(x) = \left ( 1 - \frac{x_1^2}{(1+\varepsilon)^2} - \vert x'\vert^2 \right )^s 
\end{align*} 
and define
$$ \psi_\epsilon(\tau):=\frac{\displaystyle \left ( 1 - \frac{(a_\epsilon(\tau))^2(1-\tau^2)}{(1+\varepsilon)^2} - (b_\epsilon(\tau))^2\tau^2 \right )^s
-\left(\frac34\right)^s}{\epsilon}
+\frac12 s\left(\frac34\right)^{s-1}(1-\tau^2).$$
Exploiting the expressions of~$a_\epsilon$ and $b_\epsilon$ in~\eqref{defperf333}, we compute the derivative of~$\psi_\epsilon$ with respect to~$\tau$ as
\begin{align*}
&\frac{d}{d\tau}\psi_\epsilon(\tau)\\
&=
s \left[    -\frac{\tau ^3 (1+\epsilon) \left(2\epsilon + \epsilon^2 \right) \left( 1-\frac{1+\epsilon}{2 \sqrt{1+\tau ^2 \left(2\epsilon +\epsilon^2 \right)}}\right)}{
\left(1+\tau ^2 \left(2\epsilon+ \epsilon^2 \right)\right)^{3/2}}
+\frac{2 \tau  \left(1+\epsilon-\frac{1}{2 \sqrt{1+\tau ^2 \left(2\epsilon + \epsilon^2 \right)}} \right)^2}{(1+\epsilon)^2}\right.
\\&\quad\left. 
-2 \tau  \left(1- \frac{1+\epsilon}{2 \sqrt{1+\tau ^2 \left(2\epsilon + \epsilon^2 \right)}}  \right)^2
-\frac{ \left(1-\tau ^2\right) \tau  \left(2\epsilon + \epsilon^2 \right) \left(1+\epsilon-\frac{1}{2 \sqrt{1+\tau ^2 \left(2\epsilon + \epsilon^2 \right)}} \right)}{
(1+\epsilon)^2 \left(1+\tau ^2 \left(2\epsilon + \epsilon^2 \right)\right)^{3/2}}\right]
\\& \times
\left[1-\frac{\left(1-\tau ^2\right) \left(1+\epsilon-\frac{1}{2 \sqrt{1+\tau ^2 \left(2\epsilon + \epsilon^2 \right)}}\right)^2}{(1+\epsilon)^2}-\tau ^2 \left(1-\frac{1+\epsilon }{2 \sqrt{1+
\tau ^2 \left(2\epsilon + \epsilon^2 \right)}}\right)^2\right]^{s-1}.
\end{align*}

We claim that, for~$\epsilon$ sufficiently small,
\begin{equation}\label{di4o375bv96c34567834567578}
\left\|\frac{d}{d\tau}\psi_\epsilon\right\|_{L^\infty((0,1))}\le C\epsilon,
\end{equation}
for some~$C>0$ depending only on~$s$. Not to interrupt the flow of the argument, we defer the proof of~\eqref{di4o375bv96c34567834567578}
to Appendix~\ref{apptedcalc}.

Now, we notice that
$$ \frac{(v_\varepsilon  \circ \phi_\varepsilon)(r)-\left(\frac34\right)^s}{\epsilon}
+\frac12 s\left(\frac34\right)^{s-1}(1-|r|^2)=\psi_\epsilon(|r|).
$$
Therefore,
\begin{eqnarray*}&&
\bigg \vert \frac{\vert (v_\varepsilon \circ \phi_\varepsilon)(r)- (v_\varepsilon \circ \phi_\varepsilon)(\tilde r)\vert}{\varepsilon\vert \phi_0(r) - \phi_0(\tilde r)\vert} -\frac12 s\left(\frac34\right)^{s-1} \frac{\big \vert  \vert r \vert^2 - \vert \tilde r \vert^2 \big \vert}{\vert \phi_0(r) - \phi_0(\tilde r)\vert} \bigg \vert \\
&\leqslant&  \bigg \vert  \frac{  \varepsilon^{-1}\left[(v_\varepsilon \circ \phi_\varepsilon)(r)-\left(\frac34\right)^s\right]
-\frac12 s\left(\frac34\right)^{s-1} \vert r \vert^2}{\vert \phi_0(r) - \phi_0(\tilde r)\vert} \\
&&\hspace{5em}-
\frac{\varepsilon^{-1}\left[(v_\varepsilon \circ \phi_\varepsilon)(\tilde r)-\left(\frac34\right)^s\right]
-\frac12 s\left(\frac34\right)^{s-1} \vert \tilde r \vert^2}{\vert \phi_0(r) - \phi_0(\tilde r)\vert} \bigg \vert   \\
&=&  \bigg \vert  \frac{  \varepsilon^{-1}\left[(v_\varepsilon \circ \phi_\varepsilon)(r)-\left(\frac34\right)^s\right]
+\frac12 s\left(\frac34\right)^{s-1} (1-\vert r \vert^2)}{\vert \phi_0(r) - \phi_0(\tilde r)\vert} \\
&&\hspace{5em}-
\frac{\varepsilon^{-1}\left[(v_\varepsilon \circ \phi_\varepsilon)(\tilde r)-\left(\frac34\right)^s\right]
+\frac12 s\left(\frac34\right)^{s-1}(1- \vert \tilde r \vert^2)}{\vert \phi_0(r) - \phi_0(\tilde r)\vert} \bigg \vert   \\
&=& \frac{\vert \psi_\varepsilon ( \vert r \vert ) - \psi_\varepsilon( \vert \tilde r \vert ) \vert}{\vert \phi_0(r) - \phi_0(\tilde r)\vert} .
\end{eqnarray*} Moreover, \begin{align*}
\vert \phi_0( r  ) - \phi_0( \tilde r  ) \vert^2  &= \frac14 \Big ( \big ( \sqrt{1-\vert r \vert^2}-\sqrt{1-\vert \tilde r \vert^2} \big )^2 + \vert r - \tilde r \vert^2 \Big ) \geqslant C \big \vert \vert r \vert - \vert \tilde r \vert \big \vert^2 ,  
\end{align*} so, recalling~\eqref{di4o375bv96c34567834567578}, we conclude that
\begin{equation}\label{OAWOSaAl}\begin{split}&
\bigg \vert \frac{\vert (v_\varepsilon \circ \phi_\varepsilon)(r)- (v_\varepsilon \circ \phi_\varepsilon)(\tilde r)\vert}{\varepsilon\vert \phi_0(r) - \phi_0(\tilde r)\vert} - \frac12 s\left(\frac34\right)^{s-1}  \frac{\big \vert  \vert r \vert^2 - \vert \tilde r \vert^2 \big \vert}{\vert \phi_0(r) - \phi_0(\tilde r)\vert} \bigg \vert \leqslant
C \frac{\vert \psi_\varepsilon ( \vert r \vert ) - \psi_\varepsilon( \vert \tilde r \vert ) \vert}{\big\vert \vert r \vert - \vert \tilde r \vert \big \vert}\\
&\qquad\qquad\qquad
\leqslant  C \left\|\frac{d}{d\tau}\psi_\epsilon\right\|_{L^\infty((0,1))} 
\le  C\epsilon.
\end{split}\end{equation}

Finally, note that \(\gamma_{n,s,\varepsilon}\) depends
smoothly on \(\varepsilon\) for \(\varepsilon\) small and \(\gamma_{n,s,\varepsilon} \to \gamma_{n,s}\) as \(\varepsilon \to 0^+\), which implies that \begin{align}
\vert \gamma_{n,s,\varepsilon} - \gamma_{n,s} \vert &\leqslant C \varepsilon . \label{YH4k6yy6}
\end{align} Combing~\eqref{YH4k6yy6} with~\eqref{OAWOSaAl} gives the final result. 
\end{proof}

We can now prove \thref{M2fW7tQX}. 

\begin{proof}[Proof of \thref{M2fW7tQX}]
By Lemmata~\ref{c3Ge7XSV} and~\ref{F3lTRgfZ}, for all \(r,\tilde r \in B_1^{n-1}\) with \(r \neq \tilde r\), \begin{align*}
\bigg \vert &\frac{\vert (u_\varepsilon \circ \phi_\varepsilon)(r)-(u_\varepsilon\circ \phi_\varepsilon)(\tilde r ) \vert }{\varepsilon \vert \phi_\varepsilon(r) - \phi_\varepsilon(\tilde r) \vert } -
\frac12 s \gamma_{n,s}\left(\frac34\right)^{s-1} \frac{\big \vert  \vert r \vert^2 - \vert \tilde r \vert^2 \big \vert }{\vert \phi_0(r) - \phi_0(\tilde r)\vert}\bigg \vert \\
&\leqslant  \bigg \vert \frac{\vert (u_\varepsilon \circ \phi_\varepsilon)(r)-(u_\varepsilon\circ \phi_\varepsilon)(\tilde r ) \vert }{\varepsilon \vert \phi_0(r) - \phi_0(\tilde r) \vert } \bigg \vert \cdot \bigg \vert  \frac{\vert\phi_0(r)-\phi_0(\tilde r)\vert }{\vert \phi_\varepsilon(r)-\phi_\varepsilon(\tilde r)\vert} -1 \bigg \vert   \\ 
& \qquad +\bigg \vert \frac{\vert (u_\varepsilon \circ \phi_\varepsilon)(r)-(u_\varepsilon\circ \phi_\varepsilon)(\tilde r ) \vert }{\varepsilon \vert \phi_0(r) - \phi_0(\tilde r) \vert } -\frac12 s \gamma_{n,s}\left(\frac34\right)^{s-1}  \frac{\big \vert  \vert r \vert^2 - \vert \tilde r \vert^2 \big \vert }{\vert \phi_0(r) - \phi_0(\tilde r)\vert}\bigg \vert \\
&\leqslant C\varepsilon \bigg (  \bigg \vert \frac{\vert (u_\varepsilon \circ \phi_\varepsilon)(r)-(u_\varepsilon\circ \phi_\varepsilon)(\tilde r ) \vert }{\varepsilon \vert \phi_0(r) - \phi_0(\tilde r) \vert } \bigg \vert   +1 \bigg ).
\end{align*} Moreover, by \thref{F3lTRgfZ}, for \(\varepsilon\) small, \begin{align*}
 \bigg \vert \frac{\vert (u_\varepsilon \circ \phi_\varepsilon)(r)-(u_\varepsilon\circ \phi_\varepsilon)(\tilde r ) \vert }{\varepsilon \vert \phi_0(r) - \phi_0(\tilde r) \vert } \bigg \vert &\leqslant C \bigg (  \frac{\big \vert  \vert r \vert^2 - \vert \tilde r \vert^2 \big \vert }{\vert \phi_0(r) - \phi_0(\tilde r)\vert} +1 \bigg ) \leqslant C
\end{align*} using~\eqref{GPAI62Da}. This completes the proof. 
\end{proof}

\section{Appendix A: Proof of the claim in \eqref{di4o375bv96c34567834567578}}\label{apptedcalc}

We write
\begin{equation}\label{e032jjjjtb58403t4734t4iuytr}
\frac{d}{d\tau}\psi_\epsilon(\tau)=s g_\epsilon(\tau) (h_\epsilon(\tau))^{s-1},\end{equation}
where \begin{align*}
    g_\epsilon(\tau)&:= -\frac{\tau ^3 (1+\epsilon) \left(2\epsilon + \epsilon^2 \right) 
\left( 1-\frac{1+\epsilon }{2 \sqrt{1+\tau ^2 \left(2\epsilon + \epsilon^2 \right)}}\right)}{
\left(1+\tau ^2 \left(2\epsilon + \epsilon^2 \right)\right)^{3/2}}
+\frac{2 \tau  \left(1+\epsilon-\frac{1}{2 \sqrt{1+\tau ^2 \left(2\epsilon + \epsilon^2 \right)}} \right)^2}{(1+\epsilon)^2}
\\&\quad
-2 \tau  \left(1- \frac{1+\epsilon}{2 \sqrt{1+\tau ^2 \left(2\epsilon +\epsilon^2 \right)}}  \right)^2
-\frac{ \left(1-\tau ^2\right) \tau  \left(2\epsilon + \epsilon^2 \right) \left(1+\epsilon-\frac{1}{2 \sqrt{1+\tau ^2 \left(2\epsilon + \epsilon^2 \right)}} \right)}{
(1+\epsilon)^2 \left(1+\tau ^2 \left(2\epsilon + \epsilon^2 \right)\right)^{3/2}}
\end{align*} and
\begin{align*}
h_\epsilon(\tau)&:=
1-\frac{\left(1-\tau ^2\right) \left(1+\epsilon-\frac{1}{2 \sqrt{1+\tau ^2 \left(2\epsilon + \epsilon^2 \right)}}\right)^2}{(1+\epsilon)^2}-\tau ^2 \left(1-\frac{1+\epsilon }{2 \sqrt{1+
\tau ^2 \left(2\epsilon + \epsilon^2 \right)}}\right)^2.
\end{align*}

We first prove that
\begin{equation}\label{prima000} |g_\epsilon(\tau)|\le C\epsilon,\end{equation}
for all~$\tau\in(0,1)$, for some constant~$C>0$.

For this, we observe that
\begin{equation}\label{prima111}\begin{split}&
\left| -\frac{\tau ^3 (1+\epsilon) \left(2\epsilon +\epsilon^2 \right) \left( 1-\frac{1+\epsilon}{2 \sqrt{1+\tau ^2 \left(2\epsilon + \epsilon^2 \right)}}\right)}{
\left(1+\tau ^2 \left(2\epsilon + \epsilon^2 \right)\right)^{3/2}}\right|\\
&\qquad\qquad\le
\left|\tau ^3 (1+\epsilon) \left(2\epsilon +\epsilon^2 \right) \left( 1-\frac{1+\epsilon }{2 \sqrt{1+\tau ^2 \left(2\epsilon + \epsilon^2 \right)}}\right)\right|
\\&\qquad\qquad\le 6\epsilon\left|  1-\frac{1+\epsilon}{2 \sqrt{1+\tau ^2 \left(2\epsilon+ \epsilon^2 \right)}} \right|\le 12\epsilon.
\end{split}\end{equation}
Similarly,
\begin{equation}\label{prima222}\begin{split}&
\left|-\frac{ \left(1-\tau ^2\right) \tau  \left(2\epsilon + \epsilon^2 \right) \left(1+\epsilon-\frac{1}{2 \sqrt{1+\tau ^2 \left(2\epsilon + \epsilon^2 \right)}} \right)}{
(1+\epsilon)^2 \left(1+\tau ^2 \left(2\epsilon + \epsilon^2 \right)\right)^{3/2}}\right|\\
&\qquad\le
\left|\left(1-\tau ^2\right) \tau  \left(2\epsilon + \epsilon^2 \right) \left(1+\epsilon-\frac{1}{2 \sqrt{1+\tau ^2 \left(2\epsilon + \epsilon^2 \right)}} \right)\right|
\\&\qquad \le 3\epsilon \left|1+\epsilon-\frac{1}{2 \sqrt{1+\tau ^2 \left(2\epsilon + \epsilon^2 \right)}}\right|\le 9\epsilon.
\end{split}\end{equation}

Additionally,
\begin{equation*}\begin{split}&
2\tau\left[
\frac{ \left(1+\epsilon-\frac{1}{2 \sqrt{1+\tau ^2 \left(2\epsilon + \epsilon^2 \right)}} \right)^2}{(1+\epsilon)^2}
-\left(1- \frac{1+\epsilon}{2 \sqrt{1+\tau ^2 \left(2\epsilon +\epsilon^2 \right)}}  \right)^2\right]\\
=\;&
2\tau\left[\left(1+\epsilon-\frac{1}{2 \sqrt{1+\tau ^2 \left(2\epsilon + \epsilon^2 \right)}} \right)^2
-\left(1- \frac{1+\epsilon}{2 \sqrt{1+\tau ^2 \left(2\epsilon +\epsilon^2 \right)}}  \right)^2
\right]\\&\qquad
+2\tau\left(1+\epsilon-\frac{1}{2 \sqrt{1+\tau ^2 \left(2\epsilon + \epsilon^2 \right)}} \right)^2\left[\frac{1}{(1+\epsilon)^2}-1\right]\\
=\;&
2\tau\left[\left(1+\epsilon-\frac{1}{2 \sqrt{1+\tau ^2 \left(2\epsilon + \epsilon^2 \right)}} \right)^2
-\left(1- \frac{1+\epsilon}{2 \sqrt{1+\tau ^2 \left(2\epsilon +\epsilon^2 \right)}}  \right)^2
\right]\\&\qquad
-2\tau\left(1+\epsilon-\frac{1}{2 \sqrt{1+\tau ^2 \left(2\epsilon + \epsilon^2 \right)}} \right)^2\frac{2\epsilon+\epsilon^2}{(1+\epsilon)^2}.
\end{split}\end{equation*}
We observe that 
\begin{equation*} \begin{split}&
\left(1+\epsilon-\frac{1}{2 \sqrt{1+\tau ^2 \left(2\epsilon + \epsilon^2 \right)}} \right)^2
-\left(1- \frac{1+\epsilon}{2 \sqrt{1+\tau ^2 \left(2\epsilon +\epsilon^2 \right)}}  \right)^2
\\
=\;&\left[\left(1+\epsilon-\frac{1}{2 \sqrt{1+\tau ^2 \left(2\epsilon + \epsilon^2 \right)}}\right)+\left(
1- \frac{1+\epsilon}{2 \sqrt{1+\tau ^2 \left(2\epsilon +\epsilon^2 \right)}} \right)\right]\\
&\qquad\times \left[\left(1+\epsilon-\frac{1}{2 \sqrt{1+\tau ^2 \left(2\epsilon + \epsilon^2 \right)}}\right)-
\left(
1- \frac{1+\epsilon}{2 \sqrt{1+\tau ^2 \left(2\epsilon +\epsilon^2 \right)}} \right)
\right]\\=\;&
\left(2+\epsilon- \frac{2+\epsilon}{2 \sqrt{1+\tau ^2 \left(2\epsilon +\epsilon^2 \right)}}\right)
\left(\epsilon+\frac{\epsilon}{2 \sqrt{1+\tau ^2 \left(2\epsilon + \epsilon^2 \right)}}\right)\\
=\;&\epsilon(2+\epsilon)
\left(1- \frac{1}{2 \sqrt{1+\tau ^2 \left(2\epsilon +\epsilon^2 \right)}}\right)
\left(1+\frac{1}{2 \sqrt{1+\tau ^2 \left(2\epsilon + \epsilon^2 \right)}}\right)\\
=\;&\epsilon(2+\epsilon)
\left(1- \frac{1}{4\big(1+\tau ^2 \left(2\epsilon +\epsilon^2 \right)\big)}\right),
\end{split}\end{equation*}
and therefore
\begin{equation*}\begin{split}&
\left|2\tau \left[\left(1+\epsilon-\frac{1}{2 \sqrt{1+\tau ^2 \left(2\epsilon + \epsilon^2 \right)}} \right)^2
-\left(1- \frac{1+\epsilon}{2 \sqrt{1+\tau ^2 \left(2\epsilon +\epsilon^2 \right)}}  \right)^2\right]\right|
\\ &\qquad \le 6\epsilon
\left|1- \frac{1}{4\big(1+\tau ^2 \left(2\epsilon +\epsilon^2 \right)\big)}\right|\le 12\epsilon .
\end{split}\end{equation*}
Furthermore,
\begin{eqnarray*}
\left|-2\tau\left(1+\epsilon-\frac{1}{2 \sqrt{1+\tau ^2 \left(2\epsilon + \epsilon^2 \right)}} \right)^2\frac{2\epsilon+\epsilon^2}{(1+\epsilon)^2}\right|
&&\le 6\epsilon\left|\left(1+\epsilon-\frac{1}{2 \sqrt{1+\tau ^2 \left(2\epsilon + \epsilon^2 \right)}} \right)^2\right|\\&&\le 54\epsilon.
\end{eqnarray*}
Gathering these pieces of information, we obtain that
\begin{eqnarray*}
\left|2\tau\left[
\frac{ \left(1+\epsilon-\frac{1}{2 \sqrt{1+\tau ^2 \left(2\epsilon + \epsilon^2 \right)}} \right)^2}{(1+\epsilon)^2}
-\left(1- \frac{1+\epsilon}{2 \sqrt{1+\tau ^2 \left(2\epsilon +\epsilon^2 \right)}}  \right)^2\right]\right|\le 66\epsilon.
\end{eqnarray*}

{F}rom this, \eqref{prima111} and~\eqref{prima222}
we obtain the claim in~\eqref{prima000}.

We now show that, for~$\epsilon$ sufficiently small,
\begin{equation}\label{djiwecy4i83yci437t539kkkk}
h_\epsilon(\tau)\ge c,
\end{equation}
for all~$\tau\in(0,1)$, for some~$c>0$.

For this, we observe that
\begin{equation*}\begin{split}&1-
\frac{\left(1-\tau ^2\right) \left(1+\epsilon-\frac{1}{2 \sqrt{1+\tau ^2 \left(2\epsilon + \epsilon^2 \right)}}\right)^2}{(1+\epsilon)^2}
-\tau ^2 \left(1-\frac{1+\epsilon }{2 \sqrt{1+
\tau ^2 \left(2\epsilon + \epsilon^2 \right)}}\right)^2\\ =\;&
1-\frac{ \left(1+\epsilon-\frac{1}{2 \sqrt{1+\tau ^2 \left(2\epsilon + \epsilon^2 \right)}}\right)^2}{(1+\epsilon)^2}
\\
&+\tau^2\left[\frac{ \left(1+\epsilon-\frac{1}{2 \sqrt{1+\tau ^2 \left(2\epsilon + \epsilon^2 \right)}}\right)^2}{(1+\epsilon)^2}-
\left(1-\frac{1+\epsilon }{2 \sqrt{1+
\tau ^2 \left(2\epsilon + \epsilon^2 \right)}}\right)^2\right]\\
=\;&\frac{ 1+\epsilon }{(1+\epsilon)^2\sqrt{1+\tau ^2 \left(2\epsilon + \epsilon^2 \right)}}-
\frac{ 1}{4(1+\epsilon)^2 \big({1+\tau ^2 \left(2\epsilon + \epsilon^2 \right)}\big)}
\\&\qquad
+\tau^2\left[\frac{ \left(1+\epsilon-\frac{1}{2 \sqrt{1+\tau ^2 \left(2\epsilon + \epsilon^2 \right)}}\right)^2}{(1+\epsilon)^2}-
\left(1-\frac{1+\epsilon }{2 \sqrt{1+
\tau ^2 \left(2\epsilon + \epsilon^2 \right)}}\right)^2\right]\\
=\;&\frac{ 1  }{(1+\epsilon)\sqrt{1+\tau ^2 \left(2\epsilon + \epsilon^2 \right)}}-
\frac{ 1}{4(1+\epsilon)^2 \big({1+\tau ^2 \left(2\epsilon + \epsilon^2 \right)}\big)}
\\ &\quad
+\tau^2\epsilon(2+\epsilon)\left[
1- \frac{1}{4\big(1+\tau ^2 \left(2\epsilon +\epsilon^2 \right)\big)}
-\left(1+\epsilon-\frac{1}{2 \sqrt{1+\tau ^2 \left(2\epsilon + \epsilon^2 \right)}} \right)^2\frac{1}{(1+\epsilon)^2}
\right]\\
=\;&\frac{ 1}{(1+\epsilon)\sqrt{1+\tau ^2 \left(2\epsilon + \epsilon^2 \right)}}-
\frac{ 1}{4(1+\epsilon)^2 \big({1+\tau ^2 \left(2\epsilon + \epsilon^2 \right)}\big)}
\\ &\qquad
+\tau^2\frac{\epsilon(2+\epsilon)}{(1+\epsilon)^2}\left[\frac{1+\epsilon}{\sqrt{1+\tau ^2 \left(2\epsilon + \epsilon^2 \right)} }
-\frac{1+(1+\epsilon)^2}{4\big({1+\tau ^2 \left(2\epsilon + \epsilon^2 \right)}\big)}
\right].
\end{split}\end{equation*}
For~$\epsilon$ sufficiently small, it follows that \begin{align*}
    &\frac{\left(1-\tau ^2\right) \left(1+\epsilon-\frac{1}{2 \sqrt{1+\tau ^2 \left(2\epsilon + \epsilon^2 \right)}}\right)^2}{(1+\epsilon)^2}
-\tau ^2 \left(1-\frac{1+\epsilon }{2 \sqrt{1+
\tau ^2 \left(2\epsilon + \epsilon^2 \right)}}\right)^2 \\
&\ge\frac{1}{(1+\epsilon)^2}-\frac{1}{4(1+\epsilon)^2}-\frac{\epsilon(2+\epsilon)}{(1+\epsilon)^2}\left|\frac{1+\epsilon}{\sqrt{1+\tau ^2 \left(2\epsilon + \epsilon^2 \right)} }
-\frac{1+(1+\epsilon)^2}{4\big({1+\tau ^2 \left(2\epsilon + \epsilon^2 \right)}\big)}
\right|\\
&\ge\frac1{(1+\epsilon)^2}\left(1-\frac14-\frac{\epsilon(2+\epsilon)(4+4\epsilon+\epsilon^2)}{2}\right)\\
&\ge\frac12.
\end{align*}  This establishes~\eqref{djiwecy4i83yci437t539kkkk}.

{F}rom~\eqref{e032jjjjtb58403t4734t4iuytr}, \eqref{prima000}
and~\eqref{djiwecy4i83yci437t539kkkk}, we obtain the desired claim 
in~\eqref{di4o375bv96c34567834567578}.

%% file: Part1/Nonlocal-Serrin-Stability-thesis.tex
\chapter{Quantitative stability for the nonlocal overdetermined Serrin problem}\label{yP1bfxWn}

We establish quantitative stability for the nonlocal Serrin overdetermined problem, via the method of the moving planes.
Interestingly, our stability estimate is even better than those obtained so far in the classical setting (i.e., for the classical Laplacian) via the method of the moving planes.

A crucial ingredient is the construction of a new antisymmetric barrier, which allows a unified treatment of the moving planes method. This strategy allows us to establish a new general quantitative nonlocal maximum principle for antisymmetric functions, leading to new quantitative nonlocal versions of both the Hopf lemma and the Serrin corner point lemma.

All these tools -- i.e., the new antisymmetric barrier, the general quantitative nonlocal maximum principle, and the quantitative nonlocal versions of both the Hopf lemma and the Serrin corner point lemma -- are of independent interest.

\section{Introduction}

The classical Serrin problem~\cite{MR333220} is the paradigmatic example of an ``overdetermined'' problem in which the existence of a solution satisfying both a Dirichlet and a Neumann boundary condition necessarily forces the domain to be a ball. 

To efficiently utilize classification results of this type in concrete applications it is often desirable to have a ``quantitative'' version, for instance stating that if the Neumann condition is ``almost'' satisfied, then the domain is necessarily ``close'' to a ball.

The recent literature has also intensively investigated the Serrin problem in the fractional setting. This line of research has profitably developed many aspects of the theory of nonlocal equations, including boundary regularity, Pohozaev identities, maximum and comparison principles, various versions of boundary and external conditions, a specific analysis of antisymmetric functions, etc.

However, despite the broad literature on the fractional Serrin problem, a quantitative version of this result is still missing, due to the intrinsic difficulties in adapting the classical techniques to the nonlocal case.

This paper aims to fill this gap in the literature, finally obtaining a quantitative stability result for the fractional Serrin problem. The result itself will rely on several auxiliary results of independent interest, such as
a new general quantitative nonlocal maximum principle for antisymmetric functions, which will lead to new quantitative nonlocal versions of both the Hopf lemma and the Serrin corner point lemma.
Also, a very fine analysis for a barrier function will be put forth, relying on suitable asymptotics of special functions and bespoke geometric estimates.

Specifically, the setting that we consider here is as follows.

\medskip

Let~$n$ be a positive integer and~$ \Om $ be an open bounded subset of~$\R^n$ with smooth boundary.
For~$s \in (0,1)$, let~$(-\Delta)^s$ denote the fractional Laplacian defined by 
\begin{equation}\label{eq:def fractional Laplacian}
(-\Delta)^s u(x) = c_{n,s} \, \PV \int_{\R^n} \frac{u(x)-u(y)}{\vert x - y \vert^{n+2s}} \dd y ,
\end{equation}
where~$\PV$ denotes the Cauchy principal value and~$c_{n,s}$ is a positive normalization constant.

Let~\(u\) be a solution (in an appropriate sense) of the equation
\begin{equation}\label{eq:Dirichlet problem}
\begin{cases}
(-\Delta)^s u =f(u) \quad & \text{ in } \Omega ,\\
u=0 \quad & \text{ in } \R^n \setminus \Omega, \\
u \geq 0 \quad & \text{ in } \R^n ,
\end{cases}
\end{equation}
where~$f:\R \to \R$ is a locally Lipschitz function satisfying~$f(0)\geq 0$.

Setting~$\de(x)$ to be a smooth function in~$\Om$ that coincides with~$\mathrm{dist}(x, \R^n \setminus \Om)$ in a neighborhood of~$\pa\Om$, we will assume that the solution~$u$ to~\eqref{eq:Dirichlet problem} is such that
\begin{equation}\label{eq:higher boundary regularity}
\frac{u}{\de^s} \in C^1 (\ol{\Om}).
\end{equation}
Such an assumption is automatically satisfied whenever~$\Om$ and~$f$ are smooth enough: indeed, \cite{grubb2014local, grubb2015fractional} prove that if~$\Om$ is smooth and~$f \in C^{\ga}$ then~$\frac{u}{\de^s} \in C^{s+\ga} (\ol{\Om})$ for all~$\ga>0$ (see also~\cite{ros2016boundary} where the~$C^{s+\ga} (\ol{\Om})$ regularity is achieved for~$C^{2,\ga}$ domains, for~$\ga$ small).

We stress that, for~$s=1$, \eqref{eq:higher boundary regularity} recovers the assumption~$u\in C^2(\ol{\Om})$ required in the classical proof established by Serrin via the method of moving planes: in particular, that assumption was needed to obtain a contradiction to Serrin corner lemma. In the same way, we believe that, in the nonlocal setting~$s\in(0,1)$, the assumption in~\eqref{eq:higher boundary regularity} is needed to obtain a contradiction to the nonlocal version of Serrin corner lemma and hence the rigidity result (see the second case in the proof of the
forthcoming \thref{lem:symmetric difference}).

For~$x\in\pa\Om$, we define the fractional normal derivative of~$u$ at~$x\in\pa\Om$ as
\begin{equation}\label{eq:def fracnormal derivative}
\pa_\nu^s u (x):= \lim_{t \to 0^+ } \frac{u(x) - u(x - t \nu(x))}{t^s},
\end{equation}
where~$\nu$ denotes the exterior unit normal to~$\pa\Om$.

The main result in~\cite{MR3395749} states that, if
the overdetermined problem given by~\eqref{eq:Dirichlet problem} and the Serrin-type condition
\begin{align}
\pa_\nu^s u = \text{constant} \quad \text{ on } \partial \Om \label{eq:overdetermination}
\end{align}
admits a solution, then~\(\Omega\) is necessarily a ball.
In this paper, we obtain quantitative stability for such a rigidity result.

We measure how close~\(\Omega\) is to being a ball with the quantity
\begin{equation*}
\rho(\Omega) := \inf \big\{  \rho_e - \rho_i \text{ : } B_{\rho_i} (x) \subset \Omega \subset B_{\rho_e} (x) \text{ for some } x \in \Omega \big\},
\end{equation*}
and how close~\( \pa_\nu^s u\) is to being constant on~\(\partial \Omega\) with the Lipschitz semi-norm 
\begin{equation}\label{eq:def Lipschitz seminorm}
[ \pa_\nu^s u]_{\pa \Om} := \sup_{\substack{x,y\in \pa \Om \\ x\neq y}}\frac{\vert \pa_\nu^s u(x) - \pa_\nu^s u(y) \vert}{\vert x - y \vert }.
\end{equation}

Moreover, we consider the following weighted~$L^1$-norm
\begin{equation}\label{semiliweigh}
\| u \|_{L_s(\R^n)}  = \int_{\R^n} \frac{ \vert u(x) \vert }{ 1+ \vert x \vert^{n+2s} } \dd x .\end{equation}

Then, our main result is the following:

\begin{thm} \thlabel{thm:Main Theorem}
Let~\( \Om \) be an open bounded subset of~$ \R^n$ of class~$C^2$ satisfying the uniform interior sphere condition\footnote{As usual, we say that~$\Om$ satisfies the \emph{uniform interior sphere condition with radius~$r_\Om$} if for any point~$x\in\pa\Om$ there exists a ball~$B_{r_\Om}\subset\Om$ of radius~$r_\Om$ such that~$\pa\Om \cap \pa B_{r_\Om} = \left\lbrace x \right\rbrace$.} with radius~$ r_\Om >0$.
Let~$f \in C^{0,1}_{\mathrm{loc}}(\R)$ be such that~$f(0)\geq 0$.
Let~\(u\) be a weak solution of~\eqref{eq:Dirichlet problem} satisfying~\eqref{eq:higher boundary regularity}. 

Define
\begin{equation}\begin{split}\label{eq:def R}
&R:= \min \left\lbrace r_\Om,  \frac{  \ka_{n,s}^{\frac{1}{2s}} |B_1|^{-\frac{1}{n}} }{[f]_{C^{0,1}([0,\|u\|_{L^\infty(\Omega)}])}^{\frac{1}{2s}} } \right\rbrace \\
{\mbox{with }} \qquad
& \ka_{n,s} := 
\frac{n}{2^{1-2s}} |B_1|^{1+2s/n} (1-s) \pi^{-n/2} \frac{\Ga(\frac{n}{2} +s )}{ \Ga(2+s) }
.\end{split}\end{equation}

Then,
\begin{equation}\label{eq:main stability estimate}
\rho(\Omega) \leq C [ \pa_\nu^s u]_{\partial \Om}^{\frac 1 {s+2}} 
\end{equation}
with 
\begin{equation*}
C := C(n,s) \left[1+
\frac{ 1 }{ r_\Om^{1-s}  \left( f(0)+\|u\|_{L_s(\R^n)} \right) } \right]
\left( \frac{\diam \Omega}{R} \right)^{2n+3+6s} 
R^{n+2-s} .
\end{equation*}
\end{thm}

It is worth noting that the stability estimate in~\eqref{eq:main stability estimate} is
even better than those obtained so far in the classical setting (i.e., for~$s=1$) via the method of moving planes. Indeed, \eqref{eq:main stability estimate} gives H\"older-type stability, which is better than the logarithmic-type stability achieved in~\cite{MR1729395}. We mention that, for~$s=1$, the estimate in~\cite{MR1729395} was improved to H\"older stability in~\cite{MR3522349} at the cost of restricting the analysis to a particular class of domains (which includes convex domains). We stress that, in contrast with~\cite{MR3522349}, \eqref{eq:main stability estimate} holds without restricting the class of domains considered; moreover, in the formal limit for~$s \to 1$, the stability exponent in~\eqref{eq:main stability estimate} is better than that obtained in~\cite{MR3522349}.

We mention that in the particular case~$s=1$ and~$f \equiv const.$ other stability results have been established with alternative approaches to the method of moving planes; these are mainly inspired by the approach towards Serrin's symmetry result pioneered by Weinberger in~\cite{MR333221} and can be found, e.g., in~\cite{brandolini2008stability, feldman2018stability, MR4124125, MR4054869, MP2023interpolating}.
\medskip

The quantitative version of the fractional Serrin problem has remained completely untouched, though several related problems have been considered under various perspectives: these include the cases in which the Neumann condition in the fractional Serrin problem is replaced by a parallel surface condition,
see~\cite{MR4577340, RoleAntisym2022, DPTVParallelStability},
and some of the methodologies share some features with quantitative versions of fractional geometric problems, such as~\cite{MR3836150}.

We stress that our result is new even in the case~$f \equiv const.$

\begin{remark}
Leveraging the fine analysis provided in \cite[Sections 6 and 1.3]{DPTVParallelStability} -- in particular, reasoning as in \cite[Theorem 6.9]{DPTVParallelStability} -- the stability exponent $1/(s+2)$ of \thref{thm:Main Theorem} may be improved to $\al / ( 1 +\al(s+1) )$, provided that $\Om$ is of class $C^\al$ for $\al>1$. In this case, the inequality
$$
\rho(\Omega) \leq C [ \pa_\nu^s u]_{\partial \Om}^{\al / ( 1 +\al(s+1) )}
$$
holds true with a constant $C$ depending on the same parameters of the constant appearing in \thref{thm:Main Theorem}
%
%
as well as on the $C^\al$ regularity of $\Om$.
\end{remark}

\smallskip

A crucial ingredient to establish \thref{thm:Main Theorem} is the following result, which provides a new powerful antisymmetric barrier that allows a unified treatment of the method of the moving planes. Such a barrier will allow us to establish a new general nonlocal maximum principle (\thref{Quantitative Maximum Principle}) that leads to quantitative versions (Theorems~\ref{thm:quantitative Hopf} and~\ref{thm:quantitative corner lemma}) of both the nonlocal Hopf-type lemma and the nonlocal Serrin corner point lemma established in Proposition~3.3 and Lemma~4.4
of~\cite{MR3395749}.

\begin{thm}\thlabel{prop:newantisymmetricbarrier}
Let~\(s\in (0,1)\), \(n\) be a positive integer, \( a = (a_1 , \dots , a_n) \in \R^n \) with~$a_1>0$, and~\(\rho>0\).
For~$x_0 \in \R^n$, consider
$$
\psi_{B_\rho (x_0)} := \ga_{n,s} \left( \rho^2 - |x-x_0|^2 \right)_+^s \quad \text{ where } \quad
\ga_{n,s}:= \frac{4^{-s} \Ga(\frac{n}{2})}{\Ga(\frac{n}{2}+s)(\Ga(1+s))} ,
$$
and set
\begin{align}\label{eq:def new barrier}
\varphi(x):= x_1  \big (  \psi_{B_\rho(a)} (x)+ \psi_{B_\rho(a_\ast)} (x)  \big ) ,
\end{align}
where~$a_\ast:= a - a_1 e_1$ is the reflection of~$a$ across~$\left\lbrace x_1 = 0 \right\rbrace$.

Then, we have that 
\begin{align}
(-\Delta)^s \varphi (x) \leqslant \frac{n+2s}n  x_1 \qquad \text{in } B_\rho^+(a) \setminus \partial B_\rho(a_\ast). \label{eTrwiP0v}
\end{align} 
\end{thm}

Here above and in the rest of the paper, we denote by~$B_\rho^+(a):=B_\rho(a)\cap \{x_1>0\}$.

\smallskip 

The paper is organised as follows.
In Section~\ref{sec:tools}, we establish \thref{prop:newantisymmetricbarrier}, the new quantitative nonlocal maximum principle (\thref{Quantitative Maximum Principle}), and the quantitative nonlocal versions of Hopf lemma and Serrin corner point lemma (\thref{thm:quantitative Hopf} and \thref{thm:quantitative corner lemma}). These are the crucial tools to achieve \thref{thm:Main Theorem} and are of independent interest. In particular, Section~\ref{subsec:new barrier} is devoted to establish \thref{prop:newantisymmetricbarrier}, which is then used in 
Section~\ref{subsec:Hopf e Serrin Corner} to achieve \thref{Quantitative Maximum Principle}, \thref{thm:quantitative Hopf}, and \thref{thm:quantitative corner lemma}.
Finally, in Section~\ref{sec:main result}, we complete the proof of \thref{thm:Main Theorem}.

\sectionmark{Quantitative nonlocal Hopf-type lemmata}
\section{Quantitative nonlocal versions of Hopf lemma and Serrin corner point lemma} \sectionmark{Quantitative nonlocal Hopf-type lemmata}\label{sec:tools}
\subsection{Notation and setting}
In this section, we fix the notation and provide some relevant definitions.
For~$s\in (0,1)$ and~$n \geq 1$, 
we denote by
\begin{equation*}
[u ]_{H^s(\R^n)} := c_{n,s} \iint_{\R^{2n}} \frac{\vert u(x) - u(y) \vert^2}{\vert x-y\vert^{n+2s}}  \dd x \dd y 
\end{equation*}
the Gagliardo semi-norm of~$u$ and by
\begin{equation*}
H^s(\R^n)  := \big\{ u \in L^2(\R^n) \text{ such that } [u]_{H^s(\R^n)} < +\infty \big\}
\end{equation*}
the fractional Sobolev space.
The positive constant~$c_{n,s}$ is the same appearing in the definition of the fractional Laplacian in~\eqref{eq:def fractional Laplacian}. 

Functions in~$H^s(\R^n)$ that are equal almost everywhere are identified. 

The bilinear form~$\mathcal E : H^s(\R^n) \times H^s(\R^n) \to \R$ associated with~$[\cdot ]_{H^s(\R^n)}$ is given by
\begin{equation*}
\mathcal E(u,v) = \frac{c_{n,s}}{2} \iint_{\R^{2n}} \frac{( u(x) - u(y) )(v(x)- v(y)) }{\vert x-y\vert^{n+2s}} \dd x \dd y .
\end{equation*}
Moreover, recalling~\eqref{semiliweigh}, we define the space
\begin{align*}
L_s(\R^n)  := \big\{ L^1_{\mathrm{loc}} (\R^n) \text{ such that } \| u\|_{L_s(\R^n)} <+\infty \big\}. 
\end{align*}

Let~$ \Omega$ be an open, bounded subset of~$\R^n$ and define the space
\begin{equation*}
\mathcal H^s_0(\Omega) := \big\{ u \in H^s(\R^n) \text{ such that } u =0 \text{ in } \R^n \setminus \Omega \big\}. 
\end{equation*}
Consider~$c\in L^\infty(\Omega)$ and~$g\in L^2(\Omega)$. We say that a function~$u\in L_s(\R^n) \cap H^s(\R^n) $ is a \emph{weak supersolution} of ~$(-\Delta)^s u + c  u = g$ in~$\Omega$, or, equivalently, that~$u$ satisfies
\begin{equation}\label{eq:def supersol 1}
(-\Delta)^s u + c  u \geq g\quad {\mbox{ in }}\Omega
\end{equation} 
in the \emph{weak sense}, if
\begin{equation}\label{eq:def supersol 2}
\mathcal E(u,v) + \int_\Omega c uv \dd x \geq \int_\Omega g v \dd x  ,
\quad \text{ for all } v\in \mathcal H^s_0(\Omega)\text{ with } v\geq 0 .
\end{equation}
Similarly, the notion of \emph{weak subsolution} (respectively, \emph{weak solution}) is understood by replacing the~$\geq$ sign in~\eqref{eq:def supersol 1}-\eqref{eq:def supersol 2} with~$\leq$ (respectively, $=$).
Notice that we are assuming a priori that weak solutions are in~$L_s(\R^n)$.  
\smallskip

Setting~$Q_T : \R^n \to \R^n$ to be the function that reflects~$x$ across the plane~$T$, we say that a function~$v: \R^n \to \R$ is \emph{antisymmetric with respect to~$T$} if
\begin{equation*}
v(x)  = - v(Q_T(x)) \quad \text{ for all } x \in \R^n .
\end{equation*}
To simplify the notation, we sometimes write~\( x_* \) to denote~$Q_T(x)$ when it is clear from context what~$T$ is.

Often, we will only need to consider the case~$T= \left\lbrace x_1=0 \right\rbrace $, in which case, for~$x=(x_1, \dots,x_n) \in \R^n$, we have that~$x_* = Q_T(x) = x - 2 x_1 e_1 $. 
For simplicity, we will refer to~$v$ as \emph{antisymmetric} if it is antisymmetric with respect to~$ \left\lbrace x_1=0 \right\rbrace $. 

We set
$$\R^n_+:= \big\{ x=(x_1,\dots,x_n)\in\R^n \, \text{ : } \, x_1>0 \big\} $$ and, for any~$A\subseteq\R^n$, we define~$A^+:= A\cap\R^n_+$.

As customary, we also denote by~$u^\pm$ the positive and negative parts of a function~$u$, that is
$$ u^+(x):=\max\{ u(x),0\} \qquad{\mbox{and}}\qquad
u^-(x):=\max\{- u(x),0\}.$$

\smallskip

Given~$ A \subset \R^n\), we define the characteristic function~$\chf_A : \R^n \to \R $ of~$A$ and the distance function~$\delta_A : \R^n \to [0,+\infty]$ to~\(A\) as
\begin{equation*}
\chf_A(x) :=
\begin{cases}
1, &\text{ if } x \in A , \\
0, &\text{ if } x \not \in A ,
\end{cases}
\quad \qquad \text{and} \quad \qquad
\delta_A(x) := \inf_{y \in A} \vert x-y\vert. 
\end{equation*}

Given an open bounded smooth set~$A\subset\R^n$, $\psi_A$ will denote the (unique) function in~$C^s(\R^n) \cap C^\infty(A)$ satisfying
\begin{equation}\label{eq:torsion in A}
\begin{cases}
(-\Delta)^s \psi_A =1 \quad & \text{in } A ,\\
\psi_A =0  \quad & \text{in } \R^n \setminus A . 
\end{cases}
\end{equation} 
For~$A=B_{\rho}(x_0)$, \cite{getoor1961first, bogdan2010heat} show that the explicit solution of~\eqref{eq:torsion in A} is
\begin{equation*}
\psi_{B_\rho (x_0)}:= \ga_{n,s} \left( \rho^2 - |x-x_0|^2 \right)_+^s \quad \text{ with } \quad
\ga_{n,s}:= \frac{4^{-s} \Ga(\frac{n}{2})}{\Ga(\frac{n}{2}+s)(\Ga(1+s))} .
\end{equation*}

Finally, \(\lambda_1(A)\) will denote the first Dirichlet eigenvalue of the fractional Laplacian, that is,
\begin{equation*}
\lambda_1(A) = \min_{u \in \mathcal H^s_0 (A)} \frac{\mathcal{E}(u,u)}{\displaystyle \int_{A} u^2 dx} ,
\end{equation*}
and we will use that
\begin{equation}\label{eq:Faber-Krahn}
\lambda_1(A) \ge \ka_{n,s} |A|^{- \frac{2s}{n}} \quad \text{ with } \quad
\ka_{n,s} := 
\frac{n}{2^{1-2s}} |B_1|^{1+2s/n} (1-s) \pi^{-n/2} \frac{\Ga(\frac{n}{2} +s )}{ \Ga(2+s) }
\end{equation}
for any open bounded set~$A$, see, e.g., \cite{MR3395749,yildirim2013estimates}
(see also~\cite{brasco2020quantitative}).

\subsection{A new antisymmetric barrier: proof of \texorpdfstring{\thref{prop:newantisymmetricbarrier}}{Proposition 1.3}}\label{subsec:new barrier}

This section is devoted to the proof of \thref{prop:newantisymmetricbarrier}.
To this aim, we first prove the following three lemmata.

\begin{lem} \thlabel{P7i5W6Bg}
Let~\(s\in (0,1)\), \(n\) be a positive integer, and
$$ \psi(x) := \gamma_{n,s} (1-\vert x \vert^2)^s_+ \qquad {\mbox{ with }}\quad \gamma_{n,s}:=\frac{4^{-s} \Gamma(n/2)}{\Gamma( \frac{n+2s} 2  ) \Gamma(1+s)}.$$ 

Then, \begin{align*}
(-\Delta)^s \psi(x) &= \begin{cases}
1, &\text{if } x\in B_1, \\
{\displaystyle -a_{n,s} \vert x \vert^{-n-2s } {}_2F_1 \bigg ( \frac{n+2s} 2  , s+1 ; \frac{n+2s}2+1 ; \vert x \vert^{-2} \bigg )   } , & \text{if } x\in \R^n \setminus \overline {B_1}, \\
-\infty, &\text{if } x\in \partial B_1 ,
\end{cases}
\end{align*} where 
$$ a_{n,s}:=  \frac { s\Gamma(n/2)  } {\Gamma  ( \frac{n+2s}2 +1   )\Gamma  ( 1-s)} .$$
\end{lem}

\begin{proof}
The identity~\((-\Delta)^s\psi =1 \) in~\(B_1\) is well-known in the literature, see~\cite[Section~2.6]{MR3469920}, \cite[Proposition~13.1]{MR3916700}, so we will focus on the case~\(x\in \R^n \setminus B_1\).

Let~\(x\in \R^n \setminus \overline{B_1}\). We have that~\(\psi(x) =\gamma_{n,s} \psi_0(\vert x \vert)\) with~\(\psi_0(\tau) :=  (1- \tau^2)^s_+\), so, by~\cite[Lemma~7.1]{MR3916700}, it follows that \begin{align*}
(-\Delta)^s \psi(x) &= \gamma_{n,s}\vert x \vert^{-\frac n 2 -2s-1} \int_0^\infty t^{2s+1} J_{\frac n 2 -1} (t) I(t) \dd t 
\end{align*} where \begin{align*}
I(t) &:= \int_0^\infty \tau^{\frac n 2 } \psi_0(\tau) J_{\frac n 2 -1 } (t \vert x \vert^{-1} \tau ) \dd \tau 
\end{align*} provided that both these integrals exist and converge.

To calculate~\(I(t)\) we make the change of variables, \(\tau = \sin \theta\) to obtain \begin{align*}
I(t) &= \int_0^1\tau^{\frac n 2 } (1-\tau^2)^s J_{\frac n 2 -1 } (t \vert x \vert^{-1} \tau ) \dd \tau \\
&= \int_0^{\frac \pi 2 }(\sin \theta)^{\frac n 2 } (\cos \theta )^{2s+1} J_{\frac n 2 -1 } (t \vert x \vert^{-1} \sin \theta ) \dd \theta . 
\end{align*} Next, we make use of the identity  \begin{align*}
\int_0^{\frac \pi 2} J_\mu (z\sin \theta) (\sin \theta)^{\mu+1} (\cos \theta)^{2\nu +1} \dd \theta &= 2^\nu \Gamma(\nu+1) z^{-\nu-1} J_{\mu+\nu+1}(z)
\end{align*} which holds provided~\(\RE \mu>-1\) and~\(\RE \nu>-1\), see~\cite[Eq.~10.22.19]{NIST:DLMF}. From this, we obtain \begin{align*}
I(t)&= 2^s \Gamma(s+1) \bigg ( \frac{\vert x \vert } t \bigg )^{s+1} J_{\frac n 2 +s}( \vert x \vert^{-1}t ).
\end{align*} Hence, \begin{align*}
(-\Delta)^s \psi (x) &= 2^s \gamma_{n,s}\Gamma(s+1) \vert x \vert^{-\frac n 2 -s} \int_0^\infty t^s J_{\frac n 2 -1} (t) J_{\frac n 2 +s}( \vert x \vert^{-1}t ) \dd t .
\end{align*} To compute this integral, we recall the identity \begin{align*}
\int_0^\infty t^{-\lambda} J_\mu(at) J_\nu (bt) \dd 
&= \frac{a^\mu \Gamma \big( \frac 1 2 \nu + \frac 1 2 \mu - \frac 1 2 \lambda + \frac12\big )}{2^\lambda b^{\mu-\lambda+1}\Gamma \big ( \frac 12 \nu - \frac 1 2 \mu + \frac 12 \lambda + \frac 12 \big ) \Gamma(\mu+1)} \\
&\qquad \times {}_2F_1 \bigg ( \frac 12 \big ( \mu + \nu -  \lambda + 1\big )  , \frac12 \big( \mu - \nu -\lambda+1 \big) ; \, \mu+1 ; \, \frac{a^2}{b^2} \bigg )
\end{align*} which holds when~\(0<a<b\) and~\(\RE(\mu+\nu+1)>\RE \lambda >-1\), see~\cite[Eq.~10.22.56]{NIST:DLMF}. 
Using this identity with~$a:=1/|x|$, $b:=$, $\lambda:=-s$, $\mu:=(n+2s)/2$,
and~$\nu:=(n-2)/2$ (and recalling that~\(\vert x \vert >1\)), we obtain that \begin{align*}
(-\Delta)^s \psi (x) &=  \frac {2 \gamma_{n,s} \cdot 4^s \Gamma(s+1)  } {(n+2s)\Gamma  ( -s)}  \vert x \vert^{-n -2s}  {}_2F_1 \bigg ( \frac{n+2s}2  ,s+1; \, \frac{n+2s}2 +1 ; \, \frac1{\vert x \vert^2} \bigg ).
\end{align*} A simple calculation gives that \begin{align*}
 \frac {2 \gamma_{n,s} \cdot 4^s \Gamma(s+1)  } {(n+2s)\Gamma  ( -s)}  &= -\frac { s\Gamma(n/2)  } {\Gamma  ( \frac{n+2s}2 +1   )\Gamma  ( 1-s)} =-a_{n,s}
\end{align*} which completes the case~\(x\in \R^n \setminus \overline{B_1}\). 

Finally, let~\(x\in \partial B_1\). Then, by rotational symmetry, \begin{align*}
(-\Delta)^s \psi(x) &= -\gamma_{n,s} \int_{B_1} \frac{(1-\vert y \vert ^2)^s}{\vert x -y \vert^{n+2s} } \dd y =-\gamma_{n,s} \int_{B_1} \frac{(1-\vert y \vert ^2)^s}{\vert e_1  -y \vert^{n+2s} } \dd y .
\end{align*} Now, let $$
K:= \big\{ y \in \R^n_+ \text{ : } 1/2< y_1 < 1- \vert y' \vert \big\}$$
and notice that~$K\subset B_1$.

In~\(K\), we have that \begin{align*}
\frac{(1-\vert y \vert ^2)^s}{\vert e_1  -y \vert^{n+2s} } &= \frac{(1-y_1^2-\vert y' \vert ^2)^s}{\big ( (1-y_1)^2+\vert y' \vert^2 \big )^{\frac{n+2s}2} } 
\geqslant \frac{(1-y_1^2-(1-y_1)^2)^s}{2^{\frac{n+2s}2}(1-y_1)^{n+2s} } 
\\ &=\frac{\big(2y_1(1-y_1)\big)^s}{2^{\frac{n+2s}2}(1-y_1)^{n+2s} } 
\geqslant  \frac{1}{ 2^{\frac{n+2s}2}(1-y_1)^{n+s}}.
\end{align*} 
Hence,  \begin{align*}
\int_{B_1} \frac{(1-\vert y \vert ^2)^s}{\vert e_1  -y \vert^{n+2s} } \dd y &\geqslant \frac{1}{ 2^{n+2s}} \int_K \frac {\dd y} { (1-y_1)^{n+s}} .
\end{align*} Then the coarea formula gives that \begin{align*}
\int_K \frac {\dd y} { (1-y_1)^{n+s}} &= \int_{1/2}^1 \int_{B^{n-1}_{1-t}} \frac{\dd \mathcal H^{n-1} \dd t }{ (1-t)^{n+s}} \geqslant C \int_{1/2}^1\frac{\dd t}{ (1-t)^{1+s}} = +\infty 
\end{align*} which shows that~\((-\Delta)^s \psi(x)=-\infty\).
\end{proof}

\begin{lem} \thlabel{Fov7scJw}
Let~\(s\in (0,1)\), \(n\) be a positive integer, and 
$$\psi(x) := \gamma_{n,s} (1-\vert x \vert^2)^s_+ \qquad
{\mbox{ with }}\quad \gamma_{n,s}:= 
\frac{4^{-s} \Gamma(n/2)}{\Gamma( \frac{n+2s} 2  ) \Gamma(1+s)}.$$ 
Let~\(p:\R^n \to \R\) be a homogeneous harmonic polynomial of degree~\(\ell\geqslant0\), that is, \(p\) is a polynomial such that \(p(ax)=a^\ell p(x)\) for all~\(a>0\) and~\( x\in \R\) and~\(\Delta p=0\) in~\(\R^n\). 

Then, \begin{align*}
&(-\Delta)^s (p\psi)(x) = \frac{\gamma_{n,s}}{\gamma_{n+2\ell,s}}  p(x)  \\
&\times\begin{cases}
1, &\text{if } x\in B_1, \\
{\displaystyle -a_{n+2\ell,s} \vert x \vert^{-n-2s-2\ell } {}_2F_1 \bigg ( \frac{n+2s} 2+\ell  , s+1 ; \frac{n+2s}2+1+\ell ; \vert x \vert^{-2} \bigg )   } , & \text{if } x\in \R^n \setminus \overline {B_1}, \\
-\infty, &\text{if } x\in \partial B_1 ,
\end{cases}
\end{align*} where $$a_{n,s}:=  \frac { s\Gamma(n/2)  } {\Gamma  ( \frac{n+2s}2 +1   )\Gamma  ( 1-s)} .$$
\end{lem}

\begin{proof}
Let~\(\widetilde \psi :\R^{n+2\ell} \to \R\) be given by \begin{align*}
\widetilde \psi (z) &= \gamma_{n+2\ell,s} (1-\vert z \vert^2 )^s_+ \qquad \text{for all } z\in \R^{n+2\ell}. 
\end{align*} Since~\(\psi\) is a radial function, it follows from~\cite[Proposition~3]{MR3640641} that \begin{align*}
(-\Delta)^s (p\psi)(x) &= \frac{\gamma_{n,s}}{\gamma_{n+2\ell,s}}  p(x) (-\Delta)^s \widetilde \psi (\widetilde x)
\end{align*} where~\(\widetilde x\in \R^{n+2\ell}\) with~\(\vert \widetilde x \vert = \vert x\vert\). Then the result follows from~\thref{P7i5W6Bg} applied in~\(\R^{n+2\ell}\). 
\end{proof}

\begin{lem} \thlabel{TxksrRDI}
Let~\(s\in (0,1)\) and~\(n\) be a positive integer. For all~\(0<\tau <1 \), let \begin{align*}
K(\tau) &:= a_{n,s} \big ( 1- \tau \big )^{-s} \tau^{\frac{n+2s}2 }\\
\text{ and } \quad
F(\tau)&:= {}_2F_1 \bigg (1  , \frac n 2 ; \frac{n+2s}2+1; \tau \bigg )
\end{align*} where $$a_{n,s}:=  \frac { s\Gamma(n/2)  } {\Gamma  ( \frac{n+2s}2 +1   )\Gamma  ( 1-s)} .$$ 
Moreover, let \begin{align*}
f(\tau) &:= 1-K(\tau ) \big ( F(\tau) -1 \big )  \\
\text{ and }\quad
g(\tau )&:= K(\tau) \bigg ( \frac{n+2s}{2s}-F(\tau) \bigg ) -1.
\end{align*} Then~\(f(\tau )\leqslant 1\) and~\(g(\tau) \leqslant 0\), for all~\(0<\tau <1 \). 
\end{lem}

\begin{proof}
The inequality for~\(f\) is trivial since~\(K\geqslant 0\) and, from the definition of the hypergeometric function, we have that~\( F(\tau) \geqslant 1 \). 

To show that~\(g\leqslant 0\), observe first of all that \begin{align*}
 F(1) = \frac{\Gamma \big (\frac{n+2s}2+1 \big ) \Gamma (s )  }{\Gamma \big (\frac{n+2s}2 \big ) \Gamma (s+1)} = \frac{n+2s}{2s}
\end{align*} by~\cite[Eq 15.4.20]{NIST:DLMF}. Hence, by the Fundamental Theorem of Calculus,\begin{align*}
 \frac{n+2s}{2s}-F(\tau) &= \int^1_\tau F'(t) \dd t \\
 &= \frac n {n+2s+2 } \int^1_\tau  {}_2F_1 \bigg (2  , \frac n 2+1 ; \frac{n+2s}2+2; t \bigg ) \dd t\\
 &= \frac n {n+2s+2 } \int^1_\tau  \big(1-t \big )^{s-1} {}_2F_1 \bigg (\frac{n+2s}2  , s+1 ; \frac{n+2s}2+2; t \bigg ) \dd t
\end{align*} using~\cite[Eq 15.5.1]{NIST:DLMF} and~\cite[Eq 15.8.1]{NIST:DLMF}. Now, using~\cite[Eq 15.4.20]{NIST:DLMF} again, we have that \begin{align*}
{}_2F_1 \bigg (\frac{n+2s}2  , s+1 ; \frac{n+2s}2+2; 1 \bigg )  &= \frac{\Gamma \big (\frac{n+2s}2+2 \big ) \Gamma (1-s) }{\Gamma (2) \Gamma \big ( \frac n2 +1 \big )   } = \frac{s(n+2s+2)}{n a_{n,s}} ,
\end{align*}so \begin{align*}
 &\frac{n+2s}{2s}-F(\tau)\\
 &=\frac n {n+2s+2 } \int^1_\tau  \big(1-t \big )^{s-1}\bigg [  {}_2F_1 \bigg (\frac{n+2s}2  , s+1 ; \frac{n+2s}2+2; t \bigg )-\frac{s(n+2s+2)}{n a_{n,s}}  \bigg ]  \dd t \\
 &\qquad + \frac 1 {a_{n,s}}(1-\tau)^s . 
\end{align*} Hence, \begin{align*}
&g(\tau)\\
 =\;& \big ( 1- \tau \big )^{-s} \tau^{\frac{n+2s}2 } \frac {n\cdot a_{n,s} } {n+2s+2 } \\
 &\times \int^1_\tau  \big(1-t \big )^{s-1}\bigg [  {}_2F_1 \bigg (\frac{n+2s}2  , s+1 ; \frac{n+2s}2+2; t \bigg )-\frac{s(n+2s+2)}{n a_{n,s}}  \bigg ]  \dd t \\
 &\qquad +  \tau^{\frac{n+2s}2 } -1.
\end{align*} Finally, using the monotonicity properties of hypergeometric functions, \begin{align*}
{}_2F_1 \bigg (\frac{n+2s}2  , s+1 ; \frac{n+2s}2+2; t \bigg ) \leqslant {}_2F_1 \bigg (\frac{n+2s}2  , s+1 ; \frac{n+2s}2+2; 1 \bigg )=\frac{s(n+2s+2)}{n a_{n,s}}  
\end{align*} which implies that~\(g(\tau ) \leqslant 0\). 
\end{proof}

We now give the proof of \thref{prop:newantisymmetricbarrier}.

\begin{proof}[Proof of \thref{prop:newantisymmetricbarrier}]
Let~\(\psi (x) := \gamma_{n,s}(1-\vert x \vert^2 )^s_+\) and~\(\overline \psi (x) := \gamma_{n,s} x_1 (1-\vert x \vert^2)^s_+\). Since~\(\psi_{B_\rho(a)}(x) = \rho^{2s} \psi( (x-a)/\rho) \), we see that \begin{align*}
x_1 \psi_{B_\rho(a)}(x) &= \rho^{2s+1} \bigg ( \frac{x_1 -a_1}  \rho \bigg ) \psi \bigg ( \frac{x-a} \rho\bigg ) +\rho^{2s} a_1\psi \bigg ( \frac{x-a} \rho\bigg )  \\
&= \rho^{2s+1} \overline \psi \bigg ( \frac{x-a} \rho\bigg ) +\rho^{2s} a_1\psi \bigg ( \frac{x-a} \rho\bigg )  \\ \text{ and }\quad
x_1 \psi_{B_\rho(a_\ast)}(x) &=\rho^{2s+1} \bigg ( \frac{x_1+a_1}\rho \bigg ) \psi \bigg ( \frac{x-a_\ast} \rho\bigg ) -\rho^{2s} a_1\psi \bigg ( \frac{x-a_\ast} \rho\bigg ) \\
&= \rho^{2s+1} \overline  \psi \bigg ( \frac{x-a_\ast} \rho\bigg ) -\rho^{2s} a_1\psi \bigg ( \frac{x-a_\ast} \rho\bigg ).
\end{align*} Hence, for~\(x\in B_\rho(a)\), it follows immediately from Lemmata~\ref{P7i5W6Bg} and~\ref{Fov7scJw} that \begin{equation}\label{VgjyomNF}\begin{split}
(-\Delta)^s (x_1 \psi_{B_\rho(a)}) (x) &=  \rho (-\Delta)^s \overline \psi \bigg ( \frac{x-a} \rho\bigg ) +a_1(-\Delta)^s \psi \bigg ( \frac{x-a} \rho\bigg )  \\
&= \frac{\gamma_{n,s}} {\gamma_{n+2,s}}  ( x_1-a_1) +a_1  \\
&= \frac{n+2s}n x_1 -\frac{2s}n a_1   
\end{split}\end{equation} using also that \begin{align}
\gamma_{n+2,s}=\frac{4^{-s} \Gamma(n/2+1)}{\Gamma( \frac{n+2s} 2+1  ) \Gamma(1+s)} = \frac n {n+2s} \gamma_{n,s} . \label{YVHFpqS4}
\end{align}  Similarly, if~\(x\in B_\rho(a_\ast)\), then  \begin{align*}
(-\Delta)^s ( x_1 \psi_{B_\rho(a_\ast)})(x)  &=  \rho (-\Delta)^s \overline \psi \bigg ( \frac{x-a_\ast} \rho\bigg ) -a_1(-\Delta)^s \psi \bigg ( \frac{x-a_\ast} \rho\bigg ) \\
&= \frac{\gamma_{n,s}} {\gamma_{n+2,s}} (x_1+a_1) -a_1 \\
&= \frac{n+2s} n x_1 +\frac{2s} n a_1 .
\end{align*} This gives that \begin{align*}
(-\Delta)^s \varphi(x) &= \frac{2(n+2s)} n x_1 \qquad \text{in } B_\rho(a) \cap B_\rho(a_\ast)
\end{align*}which, in particular, trivially implies~\eqref{eTrwiP0v}
for points in~$B_\rho(a) \cap B_\rho(a_\ast)$.

Now, towards the proof of~\eqref{eTrwiP0v}, we will give a formula for~\((-\Delta)^s \varphi(x)\) in the case~\(x\in B^+_\rho(a) \setminus \overline{ B_\rho(a_\ast)}\). Equation~\eqref{VgjyomNF} is still valid in this region, so we can focus on computing~\((-\Delta)^s ( x_1 \psi_{B_\rho(a_\ast)})\). Let~\(y=y(x) := \frac{x-a_\ast}{\rho} \) and note that, since~\(x \not \in \overline{ B_\rho(a_\ast)}\), we have that~\(\vert y \vert >1\) which is important for the upcoming formulas to be well-defined. 

By Lemmata~\ref{P7i5W6Bg} and~\ref{Fov7scJw}, we have that \begin{align*}
(-\Delta)^s &( x_1 \psi_{B_\rho(a_\ast)}) \\
&= -\rho \frac{\gamma_{n,s}}{\gamma_{n+2,s}} y_1 a_{n+2,s} \vert y \vert^{-n-2s-2 } {}_2F_1 \bigg ( \frac{n+2s} 2+1  , s+1 ; \frac{n+2s}2+2; \vert y \vert^{-2} \bigg ) \\
&\qquad +a_1a_{n,s} \vert y \vert^{-n-2s } {}_2F_1 \bigg ( \frac{n+2s} 2  , s+1 ; \frac{n+2s}2+1 ; \vert y \vert^{-2} \bigg ) .
\end{align*} Next, observe that \begin{align*}
a_{n+2,s}=  \frac { s\Gamma(n/2+1)  } {\Gamma  ( \frac{n+2s}2 +2   )\Gamma  ( 1-s)} = \frac n {n+2s+2}a_{n,s}, 
\end{align*} so, using also~\eqref{YVHFpqS4}, we have that \begin{align*}
(-\Delta)^s &( x_1 \psi_{B_\rho(a_\ast)}) \\
&= -a_{n,s}\vert y \vert^{-n-2s-2 }  \bigg [ \rho y_1 \bigg ( \frac{n+2s}{n+2s+2} \bigg )   {}_2F_1 \bigg ( \frac{n+2s} 2+1  , s+1 ; \frac{n+2s}2+2; \vert y \vert^{-2} \bigg ) \\
&\qquad -a_1 \vert y \vert^2 {}_2F_1 \bigg ( \frac{n+2s} 2  , s+1 ; \frac{n+2s}2+1 ; \vert y \vert^{-2} \bigg )  \bigg ] .
\end{align*} Then, via the transformation formula, \begin{align*}
{}_2F_1 ( a  ,b ; c; \tau) &= (1-\tau)^{c-a-b} {}_2F_1 ( c-a  ,c-b  ; c; \tau ) 
\end{align*} which, for~\(\tau \in \R\), holds provided that~\(0<\tau <1\), 
see~\cite[Eq~15.8.1]{NIST:DLMF}, we obtain  \begin{align*}
&(-\Delta)^s ( x_1 \psi_{B_\rho(a_\ast)}) \\
=& -a_{n,s} \big ( 1- \vert y \vert^{-2} \big )^{-s} \vert y \vert^{-n-2s-2 }  \bigg [ \rho y_1 \bigg ( \frac{n+2s}{n+2s+2} \bigg )   {}_2F_1 \bigg (1  , \frac n 2 +1 ; \frac{n+2s}2+2; \vert y \vert^{-2} \bigg ) \\
&\qquad -a_1 \vert y \vert^2 {}_2F_1 \bigg ( 1  , \frac n 2  ; \frac{n+2s}2+1 ; \vert y \vert^{-2} \bigg )  \bigg ]   \\
=& -\vert y \vert^{-2 }K \big (  \vert y \vert^{-2 } \big )   \bigg [ \rho y_1 \bigg ( \frac{n+2s}{n+2s+2} \bigg )   {}_2F_1 \bigg (1  , \frac n 2 +1 ; \frac{n+2s}2+2; \vert y \vert^{-2} \bigg )-a_1 \vert y \vert^2 F \big ( \vert y \vert^{-2} \big )   \bigg ] 
\end{align*} using the notation introduced in~\thref{TxksrRDI}. 

Moreover, we apply the following identity between contiguous functions: \begin{align*}
\frac {b\tau} c  {}_2F_1 ( a  ,b+1 ; c+1; \tau) &= {}_2F_1 ( a  ,b ; c; \tau) - {}_2F_1 ( a-1  ,b ; c; \tau),
\end{align*} 
(used here with~$a:=1$, $b:=n/2$ and~$c:=(n+2s+2)/2$)
see~\cite[Eq~15.5.16\textunderscore5]{NIST:DLMF}, and that
\({}_2F_1 ( 0  ,b ; c; \tau) =1 \), to obtain \begin{align*}
 {}_2F_1 \bigg (1  , \frac n 2 +1 ; \frac{n+2s}2+2; \vert y \vert^{-2} \bigg ) &= \frac{n+2s+2}n \vert y \vert^2 \bigg ( {}_2F_1 \bigg (1  , \frac n 2 ; \frac{n+2s}2+1; \vert y \vert^{-2} \bigg ) -1 \bigg ) \\
 &= \frac{n+2s+2}n \vert y \vert^2 \big (  F \big ( \vert y \vert^{-2} \big ) -1 \big ) .
\end{align*} Hence, it follows that \begin{align*}&
(-\Delta)^s ( x_1 \psi_{B_\rho(a_\ast)}) \\=& -K \big (  \vert y \vert^{-2 } \big )  \bigg [ \rho y_1 \bigg ( \frac{n+2s}n \bigg ) \big ( F \big ( \vert y \vert^{-2} \big ) -1  \big )  -a_1 F \big ( \vert y \vert^{-2} \big )  \bigg ] \\
=& -\bigg ( \frac{n+2s}n \bigg )  K \big (  \vert y \vert^{-2 } \big )  \big ( F \big ( \vert y \vert^{-2} \big ) -1  \big ) x_1   + \frac{2s } nK  \big (  \vert y \vert^{-2 } \big )  \bigg (    \frac{n+2s}{2s} -  F \big ( \vert y \vert^{-2} \big )  \bigg )a_1
\end{align*} using that~\(\rho y_1 =x_1+a_1\). Thus, in~\(x\in B^+_\rho(a) \setminus \overline{ B_\rho(a_\ast)}\), \begin{align*}
(-\Delta)^s \varphi(x) &= \bigg ( \frac{n+2s}n \bigg ) \bigg [ 1 - K \big (  \vert y \vert^{-2 } \big )  \big ( F \big ( \vert y \vert^{-2} \big ) -1  \big ) \bigg ]  x_1  \\
&\qquad \qquad  + \frac{2s } n \bigg [ K  \big (  \vert y \vert^{-2 } \big )  \bigg (    \frac{n+2s}{2s} -  F \big ( \vert y \vert^{-2} \big )  \bigg ) -1  \bigg] a_1.
\end{align*} Finally, by Lemma~\ref{TxksrRDI}, we conclude that \begin{align*}
(-\Delta)^s \varphi(x) &\leqslant \frac{n+2s}n  x_1 \qquad \text{in }  B^+_\rho(a) \setminus \overline{ B_\rho(a_\ast)},
\end{align*}
which completes the proof of the desired result in~\eqref{eTrwiP0v}.
\end{proof}

\subsection{A quantitative nonlocal maximum principle and quantitative nonlocal versions of Hopf lemma and Serrin corner point lemma}\label{subsec:Hopf e Serrin Corner}
The following proposition provides a new quantitative nonlocal maximum principle, which enhances~\cite[Proposition~3.6]{DPTVParallelStability} from interior to boundary estimates. Remarkably, this allows a unified treatment for the quantitative analysis of the method of the moving planes, leading to quantitative versions of both the nonlocal Hopf-type lemma and the nonlocal 
Serrin corner point lemma established in~\cite[Proposition~3.3 and Lemma~4.4]{MR3395749}: see Theorems~\ref{thm:quantitative Hopf} and~\ref{thm:quantitative corner lemma} below.

\begin{prop}\thlabel{Quantitative Maximum Principle}
Let~\(H\subset \R^n\) be a halfspace and~\(U \) be an open subset of~$H$. Let~$U$, $a\in \ol{H} $ and~$\rho>0$ be such that~\(B_\rho(a)\cap H \subset U\), and set~$a_*:=Q_{\pa H}(a)$. Also, let~\(c\in L^\infty(U)\) be such that
\begin{equation}\label{eq:condition with eigenvalue}
\|c^+\|_{L^\infty(U)} < \lambda_1(B_\rho(a)\cap H).
\end{equation} 

Let~$K\subset H$ be a non-empty open set that is disjoint from~$B_\rho(a)$ and let~$0 \le \sS <\infty$ be such that
\begin{equation}\label{eq:def Sup vecchio M^-1}
\sup_{\genfrac{}{}{0pt}{2}{x\in K}{y\in B_\rho(a)\cap H}}
\vert Q_{\partial H}(x)-y\vert \leq \sS .
\end{equation}

If~$v$ is antisymmetric with respect to~$\pa H$ and satisfies
\begin{equation*}
\begin{cases}
(-\Delta)^s v +c v \geq 0 \quad & \text{ in } U  ,\\
v \geq 0 \quad & \text{ in } H ,
\end{cases}
\end{equation*}
in the weak sense, then
we have that
\begin{equation}\label{eq:quantitative maximum principle}
v(x) \geq C \| \delta_{\partial H} v\|_{L^1(K)} \, \de_{\pa H}(x) \, \left( \rho^2 - |x-a|^2 \right)_+^s 
\end{equation}
\begin{equation*}
\begin{cases}
(i) \ \ \text{ for a.e. } x \in H \cap  B_{\rho}(a) , & \quad \text{if } \de_{\pa H}(a)=0 ,
\\
(ii) \ \text{ for a.e. } x \in H \cap B_{\rho}(a) , & \quad \text{if } 0 < \de_{\pa H}(a) \le \rho/2 ,
\\
(iii) \text{ for a.e. } x \in B_{\rho}(a) , & \quad \text{if } \de_{\pa H}(a) > \rho ,
\end{cases}
\end{equation*}
with
\begin{equation*}
C := C(n,s) \frac{\rho^{2s}}{\sS^{n+2s+2}} \Big(1+\rho^{2s}\|c^+\|_{L^\infty(U)} \Big)^{-1}.
\end{equation*}
\end{prop}

\begin{remark}\label{rem:remark on Faber}
{\rm
Notice that~\eqref{eq:condition with eigenvalue} is always satisfied provided that~$\rho$ is small enough. More precisely, the property in~\eqref{eq:Faber-Krahn} and the fact that~$|B_\rho (a) \cap H| \le |B_1|  \rho^n$ give that~\eqref{eq:condition with eigenvalue} is satisfied provided that
\begin{equation*}
\rho \le \frac{ \ka_{n,s}^{\frac{1}{2s}} }{|B_1|^{\frac{1}{n}} \|c^+\|_{L^\infty(U)}^{\frac{1}{2s}}} .
\end{equation*}
}
\end{remark}

\begin{remark}
{\rm
The three cases
in the statement of \thref{Quantitative Maximum Principle} may be gathered together in the following unified statement: {\it \eqref{eq:quantitative maximum principle} holds for a.e.~$x \in H \cap  B_{\rho}(a)$ provided that either~$\de_{\pa H}(a) \in \left[ 0, \rho/2 \right]$ or~$\de_{\pa H}(a) > \rho$.
}

Nevertheless, we prefer to keep the statement of \thref{Quantitative Maximum Principle} with the three cases to clarify their roles in their forthcoming applications.
}
\end{remark}
\begin{proof}[Proof of \thref{Quantitative Maximum Principle}]
Without loss of generality, we may assume that
$$
H = \R^n_+ = \big\{ x=(x_1,\dots,x_n)\in\R^n \, \text{ : } \, x_1>0 \big\} ,
$$
and hence use the notations~$x_*:= Q_{\pa H }(x) = Q_{\pa \R^n_+ }(x)=x-2x_1e_1$, for all~$x\in\R^n$, and~$A^+ = A \cap \R^n_+$, for all~$A\subset \R^n$.
 \smallskip
 
We first consider case~$(ii)$ (i.e., $0<\de_{\pa H}(a)=a_1\le \rho/2$), which is the most complicated, and take for granted the result in case~$(i)$, whose (simpler) proof is postponed below. Set~$B:=B_\rho(a)$ and~$B_*:=B_\rho( a_* )$.
Let~$\tau \geq 0$ be a constant to be chosen later and set
$$
w := \tau \varphi +  (\chf_K+\chf_{K_*}) v ,
$$
where~$K_*:=Q_{\pa \R^n_+}(K)$ and~$\varphi$ is as in~\eqref{eq:def new barrier}. By \thref{prop:newantisymmetricbarrier}, for any~$\xi \in \mathcal H^s_0(B^+ \setminus \ol{B_*} )$ with~$\xi \geqslant 0$, we have that
\begin{equation*}
\begin{split}&
\mathcal E (w, \xi )\\
=\;& \tau \mathcal E (\varphi, \xi )+ \mathcal E (\chf_K v , \xi )+ \mathcal E (\chf_{K_*} v , \xi ) 
\\
\leq\;& \frac{n+2s}{n} \,  \tau \int_{ \supp(\xi) } x_1 \xi(x) \dd x \\
&\quad - c_{n,s} \left[ 
\iint_{ \supp(\xi)\times K} \frac{\xi(x) v(y)}{\vert x-y\vert^{n+2s}} \dd x  \dd y+
\iint_{ \supp(\xi)\times K_*} \frac{\xi(x) v(y)}{\vert x-y\vert^{n+2s}} \dd x  \dd y\right]
\\
=\;& \int_{\supp(\xi) } \bigg [ \frac{n+2s}{n} \, \tau x_1 - c_{n,s}  \int_{K}  \bigg ( \frac 1{\vert x-y\vert^{n+2s}} - \frac 1 {\vert x_* - y \vert^{n+2s}} \bigg ) v(y) \dd y \bigg ] \xi(x) \dd x ,
\end{split}
\end{equation*}
where~$\supp(\xi) $ denotes the support of~$\xi$.

Recalling~\eqref{eq:def Sup vecchio M^-1} and noting that, for all~\(x\in B^+\) and~\(y\in K\),
\begin{equation*}
\begin{split}
\frac 1{\vert x-y\vert^{n+2s}} - \frac 1 {\vert x_*- y \vert^{n+2s}}  
& = \frac{n+2s}2 \int_{\vert x- y \vert^2}^{\vert x_*- y \vert^2} t^{-\frac{n+2s+2}2} \dd t
\\
& \geq \frac{n+2s}2  \;\frac{ \vert x_* - y \vert^2 - \vert x- y \vert^2 }{ \vert x_*- y \vert^{n+2s+2}}
\\
&= 2(n+2s)  \frac{x_1y_1}{\vert x_* -y \vert^{n+2s+2} } 
\\
&\geq \frac{2(n+2s)}{\sS^{n+2s+2}} x_1y_1 ,
\end{split}
\end{equation*}
we obtain that
\begin{eqnarray*}
\mathcal E(w,\xi) &\leq& \left(
\frac{n+2s}n\,\tau -\frac{2c_{n,s}(n+2s)}{\sS^{n+2s+2}}\,\|y_1v\|_{L^1(K)} \right)  \int_{\supp(\xi) }   x_1 \xi(x) \dd x 
\\&=&
C \left( \tau - \widetilde C \,  \frac{ \| y_1 v \|_{L^1(K)}}{\sS^{n+2s+2}} \right)  \int_{\supp(\xi) }   x_1 \xi(x) \dd x ,
\end{eqnarray*}
where~\(C\) and~\(\widetilde C\) depend only on~\(n\) and~\(s\).

Hence, using that~$w = \tau \varphi = \tau x_1 \psi_B \leq \ga_{n,s} \rho^{2s} \tau x_1$ in~$B^+ \setminus \ol{B_*} \supset \supp(\xi) $, we have that \begin{align*}
 \mathcal E(w,\xi) + &\int_{\supp(\xi) } c(x) w(x) \xi(x) \dd x  \\
&\leq
C  \left[ \tau \left( 1  + \rho^{2s}\|c^+\|_{L^\infty(U)} \right)  - \widetilde C \, \frac{  \| y_1 v \|_{L^1(K)} }{\sS^{n+2s+2}} \right] \int_{\supp(\xi)}   x_1 \xi(x) \dd x ,   
\end{align*}
where~$C$ and~$\widetilde C$ may have changed from the previous formula but still depend only on~$n$ and~$s$. 

We thus get that
\begin{equation*}
(-\Delta)^s w + cw \leq 0 \quad \text{ in } B^+ \setminus \ol{B_*},\end{equation*}
provided that \begin{equation*}
\tau \le \frac{ \widetilde C\, \| y_1 v \|_{L^1(K)}}{2\sS^{n+2s+2}} \left( 1  + \rho^{2s}\|c^+\|_{L^\infty(U)} \right)^{-1}
. 
\end{equation*}

We now claim that
\begin{equation}\label{deiwty0987654}
w\le v \quad {\mbox{ in }}\R^n_+\setminus\big( B^+ \setminus \ol{B_*}\big).
\end{equation}
To this end, we notice that, in~$\R^n_+ \setminus B^+$, we have 
that~$w = \chf_K v \le  v$. Hence, to complete the proof of~\eqref{deiwty0987654}
it remains to check that
\begin{equation}\label{deiwty098765400}
w \le v\quad {\mbox{ in }}B_*^+,\end{equation}
provided that~$\tau$ is small enough. In order to prove this, we set~$\widetilde{B}:=B_{\sqrt{\rho^2 - a_1^2}}\left( \frac{a+a_*}{2} \right)$ and notice that~$B \cap B_* \subset \widetilde{B} \subset B \cup B_*$.
Thus, an application of item~$(i)$ in~$\widetilde{B}$ gives that 
\begin{equation*}
v \ge \widehat{C} \, \frac{ (\rho^{2} - a_1^2)^s}{\sS^{n+2s+2}} \Big(1+(\rho^{2}-a_1^2)^s \|c^+\|_{L^\infty(U)} \Big)^{-1} \,\| y_1 v \|_{L^1(K)}\, x_1 \psi_{\widetilde{B}}(x)  \quad \text{ in } \widetilde{B}^+ ,
\end{equation*}
where~$\widehat{C}$ is a constant only depending on~$n$ and~$s$.
Hence, since~$ 0 < a_1 \le \rho/2$,
\begin{equation}\label{diweoghuyoiuypoiu98765}
v \ge \left( \frac{3}{4} \right)^s  \frac{\widehat{C}\, \rho^{2s}}{\sS^{n+2s+2}} \Big( 1 + \rho^{2s} \|c^+\|_{L^\infty(U)} \Big)^{-1} \,\| y_1 v \|_{L^1(K)}\, x_1 \psi_{\widetilde{B}}(x)  \quad \text{ in } \widetilde{B}^+ ,
\end{equation}
up to renaming~$\widehat{C}$.

Moreover, we claim that
\begin{equation}\label{topoeugtf}
\psi_{\widetilde{B}} \ge \frac{1}{2} \left( \psi_{B} + \psi_{B_*} \right) \quad \text{ in } B_*^+ .
\end{equation}
To check this, we observe that, for all~$t\in\left(0,\frac12\right)$,
\begin{equation}\label{pluto}
2(1-t)^s-(1-2t)^s \ge 1.\end{equation}
Also, if~$x\in B_*^+$, 
we have that~$\rho^2>|x-a_*|^2=|x-a|^2+4x_1a_1$. 
Hence, setting~$t:=\frac{2a_1x_1}{\rho^2-|x-a|^2}$, we see that~$t\in\left(0,\frac12\right)$ and thus we obtain from~\eqref{pluto} that
\begin{eqnarray*}
1\le 2\left(1-\frac{2a_1x_1}{\rho^2-|x-a|^2}\right)^s-\left(1-\frac{4a_1x_1}{\rho^2-|x-a|^2}\right)^s, 
\end{eqnarray*} 
that is
\begin{equation}\label{oewyjiythjvdnsjfe0w9876PPPP}\begin{split}
\big(\rho^2-|x-a|^2\big)^s\le\;& 2\Big( \rho^2-|x-a|^2 
- 2a_1x_1\Big)^s-\Big( \rho^2-|x-a|^2-4a_1x_1\Big)^s \\
=\;&2\left( \rho^2-a_1^2-\left| x-\frac{a+a_*}2\right|^2
\right)^s-\Big( \rho^2-|x-a_*|^2\Big)^s.
\end{split}\end{equation} 

Since~$x\in B_*^+$, we have that~$x\in B\cap B_*^+\subset \widetilde{B}$,
and therefore~\eqref{oewyjiythjvdnsjfe0w9876PPPP} gives that~$ \psi_{B}(x)\le 2\psi_{\widetilde{B}}(x)- \psi_{B_*}(x) $,
which is the desired result in~\eqref{topoeugtf}.

Putting together~\eqref{diweoghuyoiuypoiu98765} and~\eqref{topoeugtf}, 
we conclude that
\begin{equation*}
v \ge \left( \frac{3}{4} \right)^s \frac{ \widehat{C} \, \rho^{2s}}{2\sS^{n+2s+2}} \Big( 1 + \rho^{2s} \|c^+\|_{L^\infty(U)} \Big)^{-1} \,\| y_1 v \|_{L^1(K)}\,\varphi  \quad \text{ in } {B}_*^+ .
\end{equation*}
As a consequence, if
$$  \tau \le \left( \frac{3}{4} \right)^s \frac{\widehat{C}\,\rho^{2s}\,\| y_1 v \|_{L^1(K)}}{2\sS^{n+2s+2}} \Big(1+\rho^{2s}\|c^+\|_{L^\infty(U)} \Big)^{-1}, $$
we obtain that~$
w= \tau \varphi \le v$ in~$ B_*^+$.
This proves~\eqref{deiwty098765400}, and so~\eqref{deiwty0987654}
is established.

Gathering all these pieces of information, we conclude that, setting
\begin{equation*}
\tau :=  \min \left\lbrace \widetilde C , \left( \frac{3}{4} \right)^s \widehat{C}\rho^{2s} \, \right\rbrace  \, \frac{  \| y_1 v \|_{L^1(K)}}{2\sS^{n+2s+2}} \left( 1  + \rho^{2s}\|c^+\|_{L^\infty(U)} \right)^{-1} ,
\end{equation*} 
we have that
\begin{equation*}
\begin{cases}
(-\Delta)^s w + cw \leq 0 \quad \text{ in } B^+ \setminus \ol{B_*},
\\
v \ge w \quad \text{ in } \R^n_+ \setminus \left( B^+ \setminus \ol{B_*} \right) .
\end{cases}
\end{equation*}
Hence, recalling also~\eqref{eq:condition with eigenvalue}, the comparison principle in~\cite[Proposition~3.1]{MR3395749} 
gives that,
in~\( B^+ \setminus \ol{B_*} \),
\begin{equation*}
v(x) \geq w(x)  =   \min \left\lbrace \widetilde C , \left( \frac{3}{4} \right)^s \widehat{C} \, \rho^{2s}\,\right\rbrace \, \frac{ \| y_1 v \|_{L^1(K)} }{2\sS^{n+2s+2}} 
\,\left( 1+\rho^{2s}\|c^+\|_{L^\infty(U)} \right)^{-1}
\, x_1 \, \psi_B (x) ,
\end{equation*}
from which the desired result (in case~$(ii)$) follows.
\smallskip 

In case~$(i)$, that is the case where~$\de_{\pa H}(a) = a_1=0$ (i.e., $a \in \pa H$), it is easy to check that the desired result follows by repeating the above argument, but simply with~$\varphi := x_1 \psi_B$ in place of the barrier in~\eqref{eq:def new barrier} and~$B^+$ in place of~$B^+ \setminus \ol{B_*}$. The argument significantly simplifies, as, in this case, we do not need to check the inequality~$w \le v$ in~$B_*$, as given by~\eqref{deiwty098765400}, before applying~\cite[Proposition~3.1]{MR3395749}.
\smallskip

In case~$(iii)$, i.e., the case where~$\de_{\pa H}(a) = a_1 > \rho$, the desired result follows by repeating the same argument in item~$(ii)$. Notice that in this case, we have that~$B \setminus \ol{B}_* = B = B^+$, and hence (as in case~$(i)$) the argument significantly simplifies, as we do not need to check the inequality~$w \le v$ in~$B_*$ before applying~\cite[Proposition~3.1]{MR3395749}.
\end{proof}

\begin{thm}[Quantitative nonlocal Hopf Lemma]\thlabel{thm:quantitative Hopf}
Under the assumptions of \thref{Quantitative Maximum Principle}, if~$a\in H$, $v \in C (\ol{B_\rho(a)} \setminus B_\rho(a_*) )$,
and~$v(p)=0$ for some~$p=(p_1,p_2,\dots,p_n) \in  \pa B_\rho(a)\cap H$, then we have that
\begin{equation}\label{eq:Hopf}
\liminf_{t\to 0^+} \frac{ v(p-t \nu(p)) - v(p) }{ t^s \,  \de_{\pa H}(p-t \nu(p)) } \ge \ol{C} \, \| \delta_{\partial H} v\|_{L^1(K)},
\end{equation}
with
\begin{equation}\label{eq:def ol C}
\ol{C} := C(n,s) \frac{\rho^{3s}}{\sS^{n+2s+2}} \Big(1+\rho^{2s}\|c^+\|_{L^\infty(U)} \Big)^{-1}.
\end{equation}
Here, $\nu(p)$ is the exterior unit normal of~$\pa B_\rho(a)$ at~$p$.
\end{thm}

\begin{proof}
The aim is to use the estimate
in~\eqref{eq:quantitative maximum principle} of Proposition~\ref{Quantitative Maximum Principle}.

More precisely, 
\begin{itemize}
\item if~$0 < \de_{\pa H}(p - \rho \, \nu(p)) \le \rho/2$, we use item~$(ii)$
with~$a$ replaced by~$p - \rho \, \nu(p)$
\item if~$\rho/2 < \de_{\pa H}(p - \rho \, \nu(p)) \le \rho$, 
we use item~$(iii)$
with~$a$ replaced by~$p - \rho \, \nu(p)$ and~$\rho$ replaced by~$\rho/2$,
\item if~$\de_{\pa H}(p - \rho \, \nu(p)) > \rho$, we
use item~$(iii)$
with~$a$ replaced by~$p - \rho \, \nu(p)$.
\end{itemize}
Accordingly, using~\eqref{eq:quantitative maximum principle}
with~$x:=p - t \nu(p)$, we have that
\begin{eqnarray*}
v\big(p - t \nu(p)\big) &\geq& C \| \delta_{\partial H} v\|_{L^1(K)} \, \de_{\pa H}\big(p - t \nu(p)\big) \, \Big( \rho^2 - |p - t \nu(p)-(p - \rho\nu(p))|^2 \Big)_+^s \\
&=& C \| \delta_{\partial H} v\|_{L^1(K)} \, \de_{\pa H}\big(p - t \nu(p)\big) \, 
t^s(2\rho-t)^s \\
&\ge&C \| \delta_{\partial H} v\|_{L^1(K)} \, \de_{\pa H}\big(p - t \nu(p)\big) \, 
t^s\rho^s.
\end{eqnarray*}
Taking the limit as~$t\to0^+$, 
we obtain the desired result in~\eqref{eq:Hopf} and~\eqref{eq:def ol C}.
\end{proof}

\begin{thm}[Quantitative nonlocal Serrin corner point lemma]\thlabel{thm:quantitative corner lemma}
Let~$\Om \in \R^n$ be an open bounded set with~$C^2$ boundary such that the origin~$0\in \pa\Om$. Assume that the hyperplane~$\pa H= \left\lbrace x_1 = 0 \right\rbrace$ is orthogonal to~$\pa\Om$ at~$0$.
Let the assumptions in \thref{Quantitative Maximum Principle} be satisfied with~$U := \Om \cap H $, $a:=(0,\rho,0,\dots,0)\in \pa H$ and~$0 \in \pa U \cap \pa( B_\rho(a)\cap H)$.

Then, setting~$\eta:=(1,1,0\dots,0)$, we have that
\begin{equation}\label{eq:Serrin corner lemma}
\frac{v(t \eta)}{t^{1+s}} \ge \ol{C} \, \| \delta_{\partial H} v\|_{L^1(K)}  \quad \text{ for } 0< t < 
\frac\rho2,
\end{equation}
with~$\ol{C}$ in the same form as~\eqref{eq:def ol C}.
\end{thm}

\begin{proof}
Using~\eqref{eq:quantitative maximum principle} (item~$(i)$) with~$x= t \eta$ (for~$0< t <\rho/2$), we see that
\begin{eqnarray*}
v(t\eta) &\geq& C \| \delta_{\partial H} v\|_{L^1(K)} \, \de_{\pa H}(t\eta) \, \left( \rho^2 - |t\eta-a|^2 \right)_+^s \\
&=& C \| \delta_{\partial H} v\|_{L^1(K)} \, 2^s t^{1+s}( \rho-t)^s\\
&\ge& C \| \delta_{\partial H} v\|_{L^1(K)} \, t^{1+s} \rho^s.
\end{eqnarray*}
The desired result follows. 
\end{proof}

\begin{remark}
{\rm
In the particular case where~$p$ is far from~$\pa H$ (say~$\de_{\pa H}(p) > 2 \rho$), \thref{thm:quantitative Hopf} is a slight improvement of a result that could essentially be deduced from~\cite[Lemma~4.1]{MR4577340}. The key novelty and huge improvement provided by \thref{Quantitative Maximum Principle} is that \thref{thm:quantitative Hopf} remains valid regardless of the closeness of~$p$ to~$\pa H$.
Notice that~\eqref{eq:Hopf} becomes worse as~$p$ becomes closer to~$\pa H$, and, heuristically, becomes qualitatively similar to~\eqref{eq:Serrin corner lemma} for~$\de_{\pa H} (p) \to 0$.
}
\end{remark}

\section{Quantitative moving planes method and proof of Theorem~\ref{thm:Main Theorem}}\label{sec:main result}

Given~\(e\in \Sph^{n-1}\), \(\mu \in \R\), and~$A \subset \R^n$, we will use the following definitions, as in~\cite{DPTVParallelStability}: 
$$
\pi_\mu :=\{ x\in \R^n \text{ : } x\cdot e = \mu \} , \quad  \text{ i.e., the hyperplane orthogonal to }e \text{ and containing } \mu e ,
$$
$$
H_\mu :=\{x\in \R^n  \text{ : } x\cdot e>\mu \} , \quad \text{ i.e., the right-hand half space with respect to } \pi_\mu  ,
$$
$$
A_\mu := A \cap H_\mu , \quad \text{ i.e., the portion of }A \text{ on the right-hand side of } \pi_\mu ,
$$
$$
x_\mu' := x-2(x\cdot e -\mu) e ,  \quad  \text{ i.e., the reflection of } x \text{ across } \pi_\mu ,
$$
$$
A_\mu' := \{x\in \R^n \text{ : } x'_\mu\in A_\mu \} , \quad \text{ i.e., the reflection of } A_\mu \text{ across } \pi_\mu .
$$
It is clear from the above definitions that
$$
H_\mu' =\{x\in \R^n  \text{ : } x\cdot e<\mu \} , \quad \text{ i.e., the left-hand half space with respect to } \pi_\mu .
$$

Let~$\Om$ be smooth and let~$u$ be a solution of~\eqref{eq:Dirichlet problem}. Given~$\mu \in \R$ and setting
\begin{equation*}
v_\mu(x) := u(x) - u( x_\mu' ), \qquad \text{ for }  x\in \R^n ,
\end{equation*}
we have that
\begin{equation*}
(-\Delta)^s v_\mu (x) = f(u(x)) - f(u(x_\mu')) = -c_\mu(x) v_\mu (x) 
\qquad \text{ in } \Omega_\mu' ,
\end{equation*}
where
\begin{equation*}
c_\mu(x) := 
\begin{cases}
-\displaystyle \frac{f(u(x)) - f(u(x_\mu'))}{u(x) - u(x_\mu')} \quad & \text{ if } u(x) \neq u(x_\mu') ,
\\
0 \quad & \text{ if } u(x) = u(x_\mu') .
\end{cases}
\end{equation*}

Thus, $v_\mu$ is an antisymmetric function satisfying
\begin{equation*}
\begin{cases}
(-\Delta)^s v_\mu  +c_\mu v_\mu = 0 \quad &\text{ in } \Omega_\mu ' ,\\
v_\mu = u \quad &\text{ in } (\Omega\cap H_\mu') \setminus \Omega_\mu' ,\\
v_\mu = 0 \quad &\text{ in } H_\mu' \setminus \Omega ,
\end{cases}
\end{equation*}
with~\(c_\mu \in L^\infty (\Omega_\mu ')\). Moreover, we have that
\begin{equation}\label{eq: upper bound infinity norm of c}
\| c_\mu \|_{L^\infty(\Omega_\mu ')} \leq [f]_{C^{0,1}([0,\| u\|_{L^\infty(\Omega)}])}.
\end{equation}

Given a direction~$e \in \Sph^{n-1}$ and defining~$\La=\La_e:= \sup \left\lbrace  x \cdot e \, \text{ : } \, x \in \Om \right\rbrace$, we consider the \emph{critical value}
\begin{equation*}
\la=\la_e:= \inf \Big\{ \ul{\mu} \in \R \, \text{ : } \, \Omega_\mu ' \subset \Om \text{ for all } \mu \in \big( \ul{\mu} , \La \big)  \Big\} ,
\end{equation*}
for which, as usual in the method of the moving planes, either
\begin{itemize}
\item[(i)] $\pa \Om_\la'$ is internally tangent to~$\pa \Om$ at a point~$p \notin \pi_\la$,
\item[(ii)] or the critical plane~$\pi_\la$ is orthogonal to~$\pa\Om$ at a point~$p \in \pi_\la \cap \pa \Om$.
\end{itemize}
\smallskip

We recall that~$v_\mu \ge 0 $ in~$\Om_\mu'$ for all~$\mu \in \left[ \la, \La \right]$: see, e.g., \cite[Lemma~4.1]{DPTVParallelStability}.

\begin{lem} \thlabel{lem:symmetric difference}
Let~$\Om$ be an open bounded set of class~$C^2$ satisfying the uniform interior sphere condition with radius~$r_\Om>0$. 
Let~$f \in C^{0,1}_{\mathrm{loc}}(\R)$ be such that~\(f(0)\geqslant 0\) and define~$R$ as in~\eqref{eq:def R}. Let~$u$ be a weak solution of~\eqref{eq:Dirichlet problem} satisfying~\eqref{eq:higher boundary regularity}.

Then, for any~\(e\in \Sph^{n-1}\), we have that
\begin{equation} \label{eq:lemma symmetric difference}
\int_{(\Omega\cap H'_\lambda) \setminus \Omega_\lambda'} \delta_{\pi_\lambda}(x) u (x) \dd x \leq C  [ \pa_\nu^s u]_{\partial \Om} ,
\end{equation}
where
\begin{equation}\label{eq:constant in lemma symmetric}
C := C(n,s)  \left( \frac{\diam \Om}{R} \right)^{3s} (\diam \Omega)^{n+2-s} . 
\end{equation}
\end{lem}
\begin{proof} 
We use the method of moving planes and, without loss of generality, we take~$e:=-e_1$ and~$\la =0$.

$(i)$ Assume that we are in the first case, that is, the boundary of~$\Om'_\la$ is internally tangent to~$\pa \Om$ at some point~\(p\in (\partial \Om \cap \partial \Om_\lambda') \setminus \{x_1=0\}\).
Notice that, by definition of~$r_\Om$, we have that~$p - \nu(p) \, r_\Om \in H_\la'$.
We now use \thref{thm:quantitative Hopf} with~$H:=H_\la'$, $U:=\Om'_\la$, \(K := (\Om \cap H_\la') \setminus \Om_\la'\), $\rho:=R$ and~$a:= p - R \, \nu(p)$, and we obtain from~\eqref{eq:Hopf} that
\begin{equation}\label{eq:step in lemma}
\int_{(\Omega \cap H_\lambda') \setminus \Omega_\lambda '} y_1 u(y) \dd y
=
\int_{(\Omega \cap H_\lambda') \setminus \Omega_\lambda '} y_1 v_\lambda (y) \dd y 
\le
C \, \frac{ \pa_\nu^s v_\lambda(p)}{p_1} .
\end{equation}
As noticed in Remark~\ref{rem:remark on Faber}, condition~\eqref{eq:condition with eigenvalue}
is satisfied in light of~\eqref{eq:def R} and~\eqref{eq:Faber-Krahn}.
Also notice that, since, for any~$x\in(\Omega\cap H_\lambda') \setminus \Omega_\lambda'$, the point~\(x_\lambda'\) belongs to the reflection of~$\Om$ across~$\{x_1=0\}$, we have that
\eqref{eq:def Sup vecchio M^-1} holds true with~$\sS:= \diam \Om.$
Hence, \thref{Quantitative Maximum Principle}, \eqref{eq: upper bound infinity norm of c} and~\eqref{eq:def R} give that the constant in~\eqref{eq:step in lemma} is as in~\eqref{eq:constant in lemma symmetric}.
Noting that
\begin{equation*}
\frac{ \pa_\nu^s v_\lambda(p)}{p_1} = \frac{ \pa_\nu^s u(p) - \pa_\nu^s u(p_\lambda') }{p_1} = 2 \, \frac{( \pa_\nu^s u(p_\lambda') - \pa_\nu^s u(p) )}{\vert (p_\lambda')_1 - p_1 \vert} \le 2 \, [ \pa_\nu^s u]_{\partial \Om},
\end{equation*}
we obtain~\eqref{eq:lemma symmetric difference} from~\eqref{eq:step in lemma}. 

$(ii)$ In the second case, we can assume (without loss of generality) that~$ 0 = p \in \partial \Om \cap \{x_1=0\}$ and the exterior normal of~$\pa \Om $ at~$ 0 $ agrees with~$e_2$ (which is contained in~\(\{x_1=0\}\)).
We exploit \thref{thm:quantitative corner lemma} with~$H:=H_\la'$, $U:=\Omega'_\lambda$,
$K := (\Om \cap H_\la') \setminus \Om_\la'$, $\rho:=R$ and~$a:= R e_2$,
thus obatining from~\eqref{eq:Serrin corner lemma} that
\begin{equation}\label{eq:preperstimauno}
\int_{(\Omega\cap H_\lambda') \setminus \Omega_\lambda'} x_1 u (x) \dd x \le C \, \frac{ | v_\lambda( t \eta ) | }{t^{1+s}}
\end{equation}
with~$C$ in the same form as in~\eqref{eq:constant in lemma symmetric} and~$\eta := (1,1,0,\dots,0)$. 

Setting~$\de(x)$ to be a smooth function in~$\Om$ that coincides with~$\mathrm{dist}(x, \R^n \setminus \Om)$ in a neighborhood of~$\pa\Om$,
the condition in~\eqref{eq:higher boundary regularity} gives that
$$
u(x) = \de^s(x) \psi (x) \quad \text{ with }\, \psi := \frac{u}{\de^s} \in C^1(\ol{\Om}).
$$
Then, setting~$\eta_* := (-1, 1, 0 \dots, 0)$, we compute that
\begin{equation}\label{eq:preTaylor}
v_\lambda( t \eta ) = u (t \eta) - u (t \eta_* ) = \psi(t \eta) \left[ \de^s(t\eta) - \de^s(t \eta_*) \right] + \de^s (t  \eta_* ) \left[ \psi(t \eta) - \psi(t \eta_* ) \right] .
\end{equation}
As in~\cite[Proof of Lemma~4.3]{MR3395749}, we exploit classical properties of the distance function (see e.g.~\cite{MR1814364}), namely that~$\de(0)=0$, $\na \de (0)= e_2$ and the fact that~$\na^2 \de (0)$ is diagonal, to obtain the Taylor expansions
$$
\de^s(t \eta) = t^s \left( 1+ \frac{s}{2} \langle \na^2 \de(0) \eta, \eta \rangle \, t + o(t) \right) 
$$
and
$$
\de^s(t \eta_* ) = t^s \left( 1+ \frac{s}{2} \langle \na^2 \de(0) \eta_* , \eta_* \rangle \, t + o(t) \right).
$$
As a consequence,
\begin{equation*}
\de^s(t\eta) - \de^s(t \eta_* ) = o(t^{1+s}) .
\end{equation*}
By the continuity of~$\psi$ over~$\ol{\Om}$, this guarantees that
\begin{equation}\label{dewghokjhg09876543w-desfr-34t}
\lim_{t\to 0^+} \frac{\psi(t \eta) \left[ \de^s(t\eta) - \de^s(t \eta_* ) \right]}{t^{1+s}} = 0.
\end{equation}

We now estimate the last summand 
in~\eqref{eq:preTaylor}.
To this aim, a first-order Taylor expansion (for which we use~\eqref{eq:higher boundary regularity}) gives that
$$
\psi (t \eta) = \psi (0) + t \langle \na \psi(0), \eta \rangle + o(t) 
\qquad{\mbox{and}}\qquad
\psi (t \eta_* ) = \psi (0) + t \langle \na \psi(0), \eta_* \rangle + o(t) ,
$$
and therefore
$$
\psi (t \eta) - \psi (t \eta_* ) = 2 t \langle \na \psi(0), e_1 \rangle + o(t) .
$$

Hence, \eqref{eq:preTaylor} and~\eqref{dewghokjhg09876543w-desfr-34t}
give that
$$
v_\lambda( t \eta ) = 2 t^{1+s} \langle \na \psi (0), e_1 \rangle + o(t^{1+s}).
$$
Recalling that~$\psi = u/\de^s$, the definitions in~\eqref{eq:def fracnormal derivative} and~\eqref{eq:def Lipschitz seminorm}, and that~$e_1$ is orthogonal to~$e_2 = \nu (0)$,
we obtain that
$$
\lim_{t \to 0^+} \frac{ | v_\lambda( t \eta ) | }{ t^{1+s} } \le
2 \big|\langle \na \psi (0), e_1 \rangle\big|\le
 2  [\pa_\nu^s u]_{\pa\Om}.
$$
This and~\eqref{eq:preperstimauno} lead to
\begin{equation*}
\int_{(\Omega\cap H_\lambda') \setminus \Omega_\lambda'} x_1 u (x) \dd x  \leqslant C [\pa_\nu^s u]_{\partial \Om} ,
\end{equation*}
and~\eqref{eq:lemma symmetric difference} follows in this case as well. 
\end{proof}

\begin{prop}\thlabel{prop:vecchia prop 4.4}
Let~\(\Omega\) be an open bounded set of class~$C^2$ satisfying the uniform interior sphere condition with radius~\(r_\Omega >0\). Let~$f \in C^{0,1}_{\mathrm{loc}}(\R)$ be such that~\(f(0)\geqslant 0\), and define~$R$ as in~\eqref{eq:def R}. Let~$u$ be a weak solution of~\eqref{eq:Dirichlet problem} satisfying~\eqref{eq:higher boundary regularity}.

For~$e\in \Sph^{n-1}$, denoting with~$\Om'$ the reflection of~$\Om$ with respect to the critical hyperplane~$\pi_\la$, we have that
\begin{equation}\label{eq:vecchia prop 4.4}
\vert \Omega \triangle \Omega' \vert \leqslant C_\star [\pa_\nu^s u]_{\pa\Om}^{\frac1{s+2}}, 
\end{equation}
where
\begin{equation}\label{eq:C star}
C_\star := C(n,s) \left[1+
\frac{ 1 }{ r_\Om^{1-s}  \left( f(0)+\|u\|_{L_s(\R^n)} \right) } \right]
\left( \frac{\diam \Omega}{R} \right)^{2n+2+6s} 
R^{n+2-s} r_\Om^{n-1}.
\end{equation}
\end{prop}
\begin{proof}
The proof runs as the proof of Proposition~4.4 in~\cite{DPTVParallelStability}, but using \thref{lem:symmetric difference} instead of~\cite[Lemma~4.3]{DPTVParallelStability}.
\end{proof}

Notice that, if the Serrin-type overdetermined condition~\eqref{eq:overdetermination} is in force then \([ \pa_\nu^s u]_{\partial \Om}=0\), and hence \thref{prop:vecchia prop 4.4} gives that, for every direction, $\Om$ must be symmetric with respect to the critical hyperplane. This implies that~$\Om$ must be a ball, and hence recovers the main result of~\cite{MR3395749}.

Moreover,
\thref{prop:vecchia prop 4.4} allows us to obtain \thref{thm:Main Theorem} by exploiting tools of the quantitative method of the moving planes that have been developed in~\cite{MR3836150} (see also~\cite{MR4577340, RoleAntisym2022, DPTVParallelStability}).

\begin{proof}[Proof \thref{thm:Main Theorem}]
The proof can be completed by combining \thref{prop:vecchia prop 4.4} with tools already used in~\cite{MR3836150, DPTVParallelStability}.
We can assume that, for every~$i= 1, \dots, n$, the critical plane~$\pi_{e_i}$ with respect to the coordinate direction~$e_i$ coincides with~$\left\lbrace x_i = 0 \right\rbrace $ . Given~$e \in S^{n-1}$, denote by~$\lambda_e$ the critical value associated with~$e$.
Reasoning as in the proof of~\cite[Proposition~4.5]{DPTVParallelStability}, but using \thref{prop:vecchia prop 4.4} instead of~\cite[Proposition~4.4]{DPTVParallelStability}, we get that, if
\begin{equation}\label{eq:assumption for lambda estimate}
[\pa_\nu^su]^{\frac 1{s+2} }_{\partial G} \leq \frac {\vert \Omega \vert } {n C_\star }  ,
\qquad \text{ with } C_\star \text{ as in~\eqref{eq:C star}},  
\end{equation}
then
\begin{equation}\label{eq:lambda estimate frac normal der}
\vert \lambda_e \vert \leqslant C [\pa_\nu^s u]_{\partial \Om}^{\frac 1 {s+2} }
\quad \text{ for all } e\in \Sph^{n-1} ,
\end{equation}
with
\begin{equation*}
C := C(n,s) \left[1+
\frac{ 1 }{ r_\Om^{1-s}  \left( f(0)+\|u\|_{L_s(\R^n)} \right) } \right]
\left( \frac{\diam \Omega}{R} \right)^{2n+2+6s} 
R^{n+2-s} \,
\left( \frac{r_\Om^{n-1} \diam \Om}{|\Om|} \right)
.
\end{equation*}
\thref{thm:Main Theorem} follows by reasoning as in the proof of~\cite[Theorem~1.2]{MR3836150} but using~\eqref{eq:assumption for lambda estimate} and~\eqref{eq:lambda estimate frac normal der} instead of~\cite[Lemma~4.1]{MR3836150}. Notice that, recalling the definition of~$r_\Om$ and~\eqref{eq:def R}, the measure of~$\Omega$ can be easily estimated by means of~$|\Om| \ge |B_1| r_\Om^{n} \ge |B_1| R^n$.
\end{proof}

\section*{Acknowledgements} 

All the authors are members of the Australian Mathematical Society (AustMS).
Giorgio Poggesi is supported by the Australian Research Council (ARC) Discovery Early Career Researcher Award (DECRA) DE230100954 ``Partial Differential Equations: geometric aspects and applications'' and is member of the Istituto Nazionale di Alta Matematica (INdAM)/Gruppo Nazionale Analisi Matematica Probabilit\`a e Applicazioni (GNAMPA).
Jack Thompson is supported by an Australian Government Research Training Program Scholarship.
Enrico Valdinoci is supported by the Australian Laureate Fellowship FL190100081 ``Minimal surfaces, free boundaries and partial differential equations''.

%% file: Part2/Intro-antisym-Harnack.tex
\chapter{Introduction to Part II} 

The Harnack inequality plays a classical role in elliptic and parabolic \textsc{PDE} theory. From the perspective of regularity theory, the Harnack inequality directly implies the oscillation decay leading to H\"older regularity which is essential to the regularity theory for equations in divergence and non-divergence form with bounded, measurable coefficients. Another perspective is that the Harnack inequality is a quantitative analogue of the strong maximum principle which makes it a valuable tool to prove stability, particularly for overdetermined problems. For nonlocal equations, analogues of both the strong maximum principle and the Harnack inequality exist but require the solution to have a sign in all \(\R^n\) (not merely the region in which the solution satisfies the \textsc{PDE}). In this form, the strong maximum principle and Harnack inequality are incompatible with antisymmetric functions since the only antisymmetric function that is nonnegative in \(\R^n\) is the zero function. As such, the stability analysis of nonlocal overdetermined problems requires the development of Harnack inequalities for antisymmetric functions. This is the subject of Part II of this thesis.

\section{Nonlocal maximum principles for antisymmetric functions} 

Let \(\Omega \subset \R^n\) be bounded with sufficiently regular boundary such that \(0\in\Omega\), \(\R^n_+ = \{ x\in \R^n \text{ s.t. } x_1>0\}\), and \(x'=(-x_1,x_2,\dots,x_n)\). It is indicative of the method of moving planes to obtain a set \(\Omega'\subset \Omega\) and a function \(v \in C^2(\Omega)\cap C(\R^n)\) satisfying \(v(x)=-v(x')\) for all \(x\in \R^n\) and \begin{align*}
    \begin{PDE}
-\Delta v &=0, &\text{in }  \Omega' \cap \R^n_+ \\
v&\geq 0, &\text{on } \partial \Omega' \cap \R^n_+ \\
v&=0, &\text{on }\Omega' \cap \partial \R^n_+
    \end{PDE}
\end{align*} in the local case or, in the nonlocal case, \begin{align} \label{YA7P1M6O}
     \begin{PDE}
(-\Delta)^s v &=0, &\text{in }  \Omega' \cap \R^n_+ \\
v&\geq 0, &\text{in } \R^n_+\cap (\Omega \setminus \Omega') \\
v&=0, &\text{in } \R^n_+\setminus \Omega . 
    \end{PDE}
\end{align} For simplicity, we are ignoring the fact that there is usually a zero-th order term present in the above equations. An important step in the method of moving planes is to prove that \(v\geq 0\) in \(\Omega'\cap \R^n_+\). In the local case, this is a simple application of the standard maximum principle. Indeed, \(v\geq 0\) on \(\partial (\Omega'\cap \R^n_+)\) and \(v\) is harmonic in \(\Omega'\cap \R^n_+\), so \(v \geq 0 \) in \(\Omega'\cap \R^n_+\). Furthermore, the only significance the antisymmetry of \(v\) plays (i.e. \(v(x)=-v(x')\)) is to conclude \(v=0\) on \(\Omega' \cap \partial \R^n_+\). 

In the nonlocal regime, however, to apply the usual maximum principle in \(\Omega'\cap \R^n_+\), we require that \(v\geq 0\) in \(\R^n \setminus (\Omega'\cap \R^n_+)\). Since, for any open set \(E\subset \R^n\), the only antisymmetric function \(v:\R^n \to \R\) satisfying \(v\geq 0\) in \(\R^n \setminus E\) is \(v\equiv 0\) in \(\R^n\), the nonlocal case requires a new maximum principle which takes advantage of the symmetry present in \(v\). With this in mind, let \(x\in \Omega'\cap \R^{n+1}\). Then, via the change of variables \(y\to y'\) in the second integral below and using the antisymmetry of \(v\), we obtain \begin{align*}
  (-\Delta)^s v(x) 
    &= \PV \int_{\R^n_+} \frac{v(x)-v(y)}{\vert x- y\vert^{n+2s}} \dd y + \int_{\R^n_-} \frac{v(x)-v(y)}{\vert x- y\vert^{n+2s}} \dd y \\
    &= \PV \int_{\R^n_+} \frac{v(x)-v(y)}{\vert x- y\vert^{n+2s}} \dd y + \int_{\R^n_+} \frac{v(x)+v(y)}{\vert x'- y\vert^{n+2s}} \dd y \\
    &= \PV \int_{\R^n_+} \bigg ( \frac1 {\vert x- y\vert^{n+2s}} -\frac1{\vert x'- y\vert^{n+2s}}\bigg ) (v(x)-v(y)) \dd y \\
& \qquad + 2v(x)\int_{\R^n_+} \frac{\dd y }{\vert x'- y\vert^{n+2s}} .
\end{align*} Since \begin{align*}
     \frac1 {\vert x- y\vert^{n+2s}} -\frac1{\vert x'- y\vert^{n+2s}} &>0 \qquad \text{for all } x\in \R^n_+,
\end{align*} if there exists \(x_0 \in \Omega'\cap \R^n_+\) such that \(v(x_0)=0\) and \(v\) satisfies~\eqref{YA7P1M6O} then \begin{align*}
    (-\Delta)^sv(x_0) &= -\PV \int_{\R^n_+} \bigg ( \frac1 {\vert x_0- y\vert^{n+2s}} -\frac1{\vert x_0'- y\vert^{n+2s}}\bigg ) v(y) \dd y \leq 0
\end{align*} with equality if and only if \(v\equiv 0\) in \(\R^n\). Hence, this implies the strong maximum principle: either \(v>0\) in \(\Omega'\cap \R^n_+\) or \(v\equiv 0\) in \(\R^n\). In particular, we conclude that \(v\geq 0\) in \(\Omega'\cap \R^n_+\). The above argument was first given in \cite{MR3395749} for weak solutions of~\eqref{YA7P1M6O}, see also Proposition \ref{fyoaW}. It is clear from the preceding argument that antisymmetry plays a much more prominent role with nonlocal equations than with local case. 

In the stability analysis of overdetermined problems, conceptually the idea is to replace qualitative arguments with quantitative ones. An example of a qualitative result is the strong maximum principle and its quantitative analogue is the Harnack inequality. Indeed, for a nonnegative harmonic function \(u:\Omega \to \R\) with \(\Omega\) bounded, the strong maximum principle tells us either \(u>0\) in \(\Omega\) or \(u\equiv 0\) in \(\Omega\). On the other hand, the Harnack inequality states that for all connected \(\tilde \Omega\subset\subset \Omega\) then \(\sup_{\tilde \Omega} u \leq C \inf_{\tilde \Omega}u\) with \(C>0\) depending on \(\tilde \Omega\), \(\Omega\), and \(n\). In other words, the existence of a point \(x_0\in \tilde \Omega\) with \(u(x_0)=\varepsilon \) with \(\varepsilon>0\) very small gives the quantitative upper bound \(u\leq C\varepsilon\) in \(\tilde \Omega\). For the Harnack inequality to apply to the method of moving planes and the stability of overdetermined problems, we require that we can apply it to antisymmetric functions. While this is of no significance for local problems, as we describe above, it is significant for nonlocal problems.

In Chapter~\ref{RYsUIqSr}, we prove a Harnack inequality for antisymmetric solutions to linear PDE driven by the fractional Laplacian with zero-th order terms. Such PDE are exactly the equations that arise in the method of moving planes. In its simplest form, our Harnack inequality is the following: let \(c\in L^\infty(B_1^+)\) and suppose that \(u\in C^2(B_1) \cap L^\infty(\R^n)\) is antisymmetric, non-negative in \(\R^n_+\), and satisfies \begin{align}
    (-\Delta)^s u +cu = 0 \qquad \text{in } B_1. \label{TyBhIDPf}
\end{align} Then \begin{align}
    \sup_{B_{1/2}^+} \frac{ u(x)}{x_1} \leq C \inf_{B_{1/2}^+} \frac{u(x)}{x_1} . \label{NDEEpVmI}
\end{align} The constant \(C>0\) depends only on \(n\), \(s\), and \(\|c\|_{L^\infty(B_1^+)}\). Moreover, the quantities \( \inf_{B_{1/2}^+} \frac{u(x)}{x_1}\) and  \(\sup_{B_{1/2}^+} \frac{u(x)}{x_1}\) are both comparable to \begin{align}
    \int_{\R^n_+} \frac{x_1\vert u(x)\vert }{1+\vert x \vert^{n+2s+2}} \dd x . \label{xq8ep7ce}
\end{align} The quantity~\eqref{xq8ep7ce} appears naturally in the proof of the antisymmetric Harnack inequality and plays an identical role to what \begin{align*}
    \| u\|_{L_s(\R^n)} = \int_{\R^n} \frac{\vert u(x) \vert }{1+\vert x \vert^{n+2s}} \dd x 
\end{align*} plays in the usual (non-antisymmetric) nonlocal Harnack inequality. The proof of~\eqref{NDEEpVmI} requires the construction of several antisymmetric barriers which, via viscosity-type arguments (i.e. touching the solutions from above/below with the barriers), implies both an interior boundedness and weak Harnack inequality. Together, these imply the full Harnack inequality. 

In Chapter~\ref{OXXm8GiN}, we continue to explore the antisymmetric Harnack inequality of Chapter~\ref{RYsUIqSr}, but through techniques from harmonic analysis. At the heart of the chapter, is the powerful Bochner's relation: let \(f\) and \(\tilde f\) be two radial Schwartz functions on \(\R^n\) and \(\R^{n+2\ell}\) respectively with the same profile function, that is, \(f(x)=\tilde f(\tilde x)\) whenever \(\vert x \vert = \vert \tilde x \vert \). Moreover, let \(V\) be a solid harmonic polynomial of degree \(\ell\). Then \begin{align*}
    \mathscr F_n \{ Vf\}(\xi) = i^\ell V(\xi) \mathscr F_{n+2\ell}f (\tilde \xi)
\end{align*} for all \(\xi \in \R^n\) and \(\tilde \xi \in \R^{n+2\ell}\) with \(\vert \xi \vert = \vert \tilde \xi \vert\). Here \(\mathscr F_m\) denotes the Fourier transform in \(\R^m\). Morally, Bochner's relation tells us that for radial functions, solid harmonic polynomials commute with the Fourier transform (at the expense of the Fourier transform needing to be evaluated in a higher dimension based on the degree of the polynomial). To connect Bochner's relation to the antisymmetric Harnack inequality, we use two facts. The first is that  \begin{align*}
    \mathscr F_n \{ (-\Delta)^s u\}(\xi) = \vert \xi \vert^{2s}   \mathscr F_n u(\xi),
\end{align*} so the Fourier symbol of \((-\Delta)^{2s}\) (\(\vert \xi\vert^{2s}\)) is a radial function. The second is that, given an antisymmetric function \(u:\R^n \to \R\), the function \(x_1\mapsto \frac{u(x)}{x_1}\) is even (i.e. radial in one dimension) and \(V(x_1)=x_1\) is a homogeneous harmonic polynomial in one dimension. This allows us to prove that \begin{align*}
    (-\Delta)^s \tilde v(\tilde x) &= \frac{(-\Delta)^s u \big ( \sqrt{\tilde x_1^2 +\tilde x_2^2+ \tilde x_3^2 }, \tilde x_4,\dots, \tilde x_{n+2} \big ) }{\sqrt{\tilde x_1^2 +\tilde x_2^2+ \tilde x_3^2 }}
\end{align*} for all \(\tilde x \in \{ \tilde y \in B_1^{n+2}  \text{s.t. } (\tilde y_1,\tilde y_2,\tilde y_3) \neq 0 \}\) where \(\tilde v : \R^{n+2} \to \R\) is given by \begin{align*}
    \tilde v (\tilde x ) &= \frac{u \big ( \sqrt{\tilde x_1^2 +\tilde x_2^2+ \tilde x_3^2 }, \tilde x_4,\dots, \tilde x_{n+2}\big ) }{\sqrt{\tilde x_1^2 +\tilde x_2^2+ \tilde x_3^2 }}.
\end{align*} In particular, if \(u\) satisfies~\eqref{TyBhIDPf} then \(\tilde v\) satisfies \begin{align*}
    (-\Delta)^s \tilde v + \tilde c \tilde v =0 \qquad \text{in } \{ \tilde y \in B_1^{n+2}  \text{s.t. } (\tilde y_1,\tilde y_2,\tilde y_3) \neq 0 \},
\end{align*} so, applying the Harnack inequality for the fractional Laplacian in \(\R^{n+2}\) allows us to conclude the antisymmetric Harnack inequality in \(\R^n\). Not only does this illustrate a beautiful connection between the standard nonlocal Harnack and nonlocal Harnack inequalities over classes of functions with a symmetry present, but the argument's application to the fractional Laplacian only uses that the Fourier symbol of the fractional Laplacian is radial. As such, the arguments directly generalise to nonlocal operators with radial Fourier symbol which we also address in the final section of Chapter~\ref{OXXm8GiN}.

%% file: Part2/Antisymmetric-Harnack-final.tex

\chapter{On the Harnack inequality for antisymmetric \(s\)-harmonic functions} \label{RYsUIqSr}

We prove the Harnack inequality for antisymmetric \(s\)-harmonic functions, and more generally for
solutions of fractional equations with zero-th order terms, in a general domain. This may be used in conjunction with the method of moving planes to obtain quantitative stability results for symmetry and overdetermined problems for semilinear equations driven by the fractional Laplacian.

The proof is split into two parts: an interior Harnack inequality away from the plane of symmetry, and a boundary Harnack inequality close to the plane of symmetry. We prove these results by first
establishing the weak Harnack inequality for super-solutions and local boundedness for sub-solutions in both the interior and boundary case.

En passant, we also obtain a new mean value formula for antisymmetric $s$-harmonic functions.

\section{Introduction and main results}
Since the groundbreaking work \cite{harnack1887basics}, Harnack-type inequalities have become an essential tool in the analysis of partial differential equations (\textsc{PDE}). In \cite{MR1729395,MR3522349,MR3481178}, Harnack inequalities have been applied to antisymmetric functions---functions that are odd with respect to reflections across a given plane, for more detail see Section \ref{LenbLVkk}---in conjunction with the method of moving planes to obtain stability results for local overdetermined problems including \emph{Serrin's overdetermined problem} and the \emph{parallel surface problem}; see also \cite{MR3932952}.

In recent years, there has been consolidated interest in overdetermined problems driven by nonlocal operators, particularly the fractional Laplacian; however, standard Harnack inequalities for such operators are incompatible with antisymmetric functions since they require functions to be non-negative in all of \(\R^n\). In this paper, we will address this problem by proving a Harnack inequality for antisymmetric \(s\)-harmonic functions with zero-th order terms which only requires non-negativity in a halfspace. By allowing zero-th order terms this result is directly applicable to symmetry and overdetermined problems for semilinear equations driven by the fractional Laplacian.

\subsection{Background} Fundamentally the original Harnack inequality for elliptic \textsc{PDE} is a quantitative formulation of the strong maximum principle and directly gives, among other things, Liouville's theorem, the removable singularity theorem, compactness results, and Hölder estimates for weak and viscosity solutions, see~\cite{MR1814364,MR2597943}. This has led it to be extended to a wide variety of other settings. For local \textsc{PDE} these include:
linear parabolic \textsc{PDE}~\cite{MR2597943,MR1465184}; quasilinear and fully nonlinear
elliptic \textsc{PDE}~\cite{MR170096,MR226198,MR1351007}; quasilinear and fully nonlinear
parabolic \textsc{PDE}~\cite{MR244638,MR226168}; and in connection with curvature and geometric
flows such as Ricci flow on Riemannian manifolds~\cite{MR431040,MR834612,MR2251315}.
An extensive survey on the Harnack inequality for local \textsc{PDE} is~\cite{MR2291922}.

For equations arising from jump processes, colloquially known as nonlocal \textsc{PDE}, the first Harnack inequality is due to Bogdan~\cite{MR1438304} who proved the boundary Harnack inequality for the fractional Laplacian. The fractional Laplacian is the prototypical example of a nonlocal operator and is defined by \begin{align}
(-\Delta)^s u(x) &= c_{n,s} \PV \int_{\R^n} \frac{u(x)-u(y)}{\vert x - y \vert^{n+2s}} \dd y  \label{ZdAlT}
\end{align}
where~$s\in(0,1)$, \(c_{n,s}\) is a positive normalisation constant (see~\cite{MR3916700} for more details, particularly Proposition~5.6 there) and \(\PV\) refers to the Cauchy principle value. The result in~\cite{MR1438304} was only valid for Lipschitz domains, but this was later extended to any bounded domain in~\cite{MR1719233}.

Over the proceeding decade there were several papers proving Harnack inequalities for more general jump processes including~\cite{MR1918242,MR3271268}, see also~\cite{MR2365478}. For fully nonlinear nonlocal \textsc{PDE}, a Harnack inequality was established in~\cite{MR2494809}, see also~\cite{MR2831115}. More recently, in~\cite{MR4023466} the boundary Harnack inequality was proved for nonlocal \textsc{PDE} in non-divergence form.

As far as we are aware, the only nonlocal Harnack inequality for antisymmetric functions in the literature is in~\cite{ciraolo2021symmetry} where a boundary Harnack inequality was established for antisymmetric \(s\)-harmonic functions in a small ball centred at the origin. Our results generalise this to arbitrary symmetric domains and to equations with zero-th order terms. 

\subsection{Main results} \label{K4mMv} Let us now describe in detail our main results. First, we will introduce some useful notation. Given a point \(x=(x_1,\dots , x_n) \in \R^n\), we will write \(x=(x_1,x')\) with \(x'=(x_2,\dots , x_n)\) and we denote by \(x_\ast\) the reflection of \(x\) across the plane \(\{x_1=0\}\), that is, \(x_\ast = (-x_1,x')\). Then we call a function \(u:\R^n \to \R\) \emph{antisymmetric} if \begin{align*}
u(x_\ast) = -u(x) \qquad \text{for all } x\in \R^n. 
\end{align*} We will also denote by \(\R^n_+\) the halfspace \(\{x\in \R^n \text{ s.t. } x_1>0\}\) and, given \(A\subset \R^n\), we let \(A^+ := A \cap \R^n_+\). Moreover, we will frequently make use of the
functional space \(\mathscr A_s(\R^n)\) which we define to be the set of all antisymmetric functions \(u\in L^1_{\mathrm {loc}}(\R^n)\) such that \begin{align}
\Anorm{u}:= \int_{\R^n_+} \frac{x_1 \vert u(x) \vert}{1+ \vert x \vert^{n+2s+2}}\dd x \label{QhKBn}
\end{align} is finite. The role that the new space \(\mathscr A_s(\R^n)\) plays will be explained in more detail later in this section.

Our main result establishes the Harnack inequality\footnote{The arguments presented here are quite general and can be suitably extended to other integro-differential operators. For the specific case of the fractional Laplacian, the extension problem can provide alternative ways to prove some of the results presented here. We plan to treat this case extensively in a forthcoming paper, but in Appendix~\ref{APPEEXT:1} here
we present a proof of~\eqref{LA:PAKSM} specific for the fractional Laplacian and~$c:=0$
that relies on extension methods.}
in a general symmetric domain \(\Omega\).

\begin{thm} \thlabel{Hvmln}
Let \(\Omega \subset \R^n\) and \(\tilde \Omega \Subset \Omega\) be bounded domains that are symmetric with respect to \(\{x_1=0\}\), and let \(c\in L^\infty(\Omega^+)\). Suppose that \(u \in C^{2s+\alpha}(\Omega)\cap \mathscr{A}_s(\R^n) \) for some \(\alpha >0\) with \(2s+\alpha\) not an integer, \(u\) is non-negative in \(\R^n_+\), and satisfies \begin{align*}
(-\Delta)^s u +cu&=0 \qquad \text{in }\Omega^+.
\end{align*}

Then there exists \(C>0\) depending only on \(\Omega\), \(\tilde \Omega\), \(\| c\|_{L^\infty(\Omega^+)}\), \(n\), and \(s\) such that \begin{equation}\label{LA:PAKSM}
\sup_{\tilde \Omega^+} \frac{u(x)}{x_1}  \leqslant C \inf_{ \tilde \Omega^+} \frac{u(x)}{x_1} .
\end{equation}  Moreover, \(\inf_{\tilde \Omega^+} \frac{u(x)}{x_1}\) and \(\sup_{\tilde \Omega^+} \frac{u(x)}{x_1}\) are comparable to \(\Anorm{u}\).
\end{thm}

Here and throughout this document we use the convention that if \(\beta>0\) and \(\beta\) is not an integer then \(C^\beta (\Omega)\) denotes the Hölder space \(C^{k,\beta'}(\Omega)\) where \(k\) is the integer part of \(\beta\) and \(\beta'=\beta-k \in (0,1)\).

One can compare the Harnack inequality in Theorem~\ref{Hvmln} here
with previous results in the literature, such as Proposition~6.1 in~\cite{MR4308250},
which can be seen as a boundary Harnack inequality for antisymmetric functions,
in a rather general fractional elliptic setting. We point out that Theorem~\ref{Hvmln} here
holds true without any sign assumption on~$c$ (differently from Proposition~6.1 in~\cite{MR4308250}
in which a local sign assumption\footnote{The sign assumption on \(c\) in Proposition~6.1 of~\cite{MR4308250} is required since the result was for unbounded domains. On the other hand, our result is for bounded domains which is why this assumption is no longer necessary. } on~$c$ was taken), for all~$s\in(0,1)$ (while the analysis of Proposition~6.1 in~\cite{MR4308250}
was focused on the case~$s\in\left[\frac12,1\right)$), and in every dimension~$n$ (while Proposition~6.1 in~\cite{MR4308250} dealt with the case~$n=1$).

We would like to emphasise that, in contrast to the standard Harnack inequality for the fractional Laplacian, \thref{Hvmln} only assumes \(u\) to be non-negative in the halfspace \(\R^n_+\) which is a more natural assumption for antisymmetric functions. Also, note that the assumption that \(\tilde \Omega \Subset \Omega\) allows~\(\tilde \Omega^+\) to touch \(\{x_1=0\}\), but prevents \(\tilde \Omega^+\) from touching the portion of \(\partial\Omega\) that is not on \(\{x_1=0\}\); it is in this sense that we will sometimes refer to \thref{Hvmln} (and later \thref{C35ZH}) as a \emph{boundary} Harnack inequality for antisymmetric functions.

Interestingly, the quantity \(\Anorm{u}\) arises very naturally in the context of antisymmetric functions and plays the same role that \begin{align}
\| u\|_{\mathscr L_s (\R^n)} := \int_{\R^n} \frac{\vert u (x) \vert }{1+\vert x \vert^{n+2s}} \dd x \label{ACNAzPn8}
\end{align} plays in the non-antisymmetric nonlocal Harnack inequality, see for example \cite{MR4023466}. To our knowledge, \(\Anorm{\cdot}\) is new in the literature.

A technical aspect of \thref{Hvmln}, however, is that the fractional Laplacian is in general not defined if we simply have \(u\) is \(C^{2s+\alpha}\) and \(\Anorm{u} < +\infty\). This leads us to introduce the following new definition of the fractional Laplacian for functions in \(\mathscr A_s(\R^n)\) which we will use throughout the remainder of this paper. 

\begin{defn} \thlabel{mkG4iRYH}
Let \(x\in \R^n_+\) and suppose that \(u \in \mathscr A_s(\R^n)\) is \( C^{2s+\alpha}\) in a neighbourhood of \(x\) for some \(\alpha>0\) with \(2s+\alpha\) not an integer. The fractional Laplacian of \(u\) at \(x\), which we denote as usual by \((-\Delta)^su(x)\), is defined by \begin{align}\label{OopaQ0tu} 
(-\Delta)^su(x) &= c_{n,s} \lim_{\varepsilon\to 0^+} \int_{\R^n_+ \setminus B_\varepsilon(x)} \left(
\frac 1 {\vert x - y \vert^{n+2s}} - \frac 1 {\vert x_\ast- y \vert^{n+2s}}\right) \big(u(x)-u(y)\big)\dd y  \\
&\qquad + \frac {c_{1,s}} s  u(x)x_1^{-2s}\nonumber 
\end{align} where \(c_{n,s}\) is the constant from~\eqref{ZdAlT}.
\end{defn}

We will motivate \thref{mkG4iRYH} in Section \ref{LenbLVkk} as well as verify it is well-defined. We also prove in Section \ref{LenbLVkk} that if \(\| u\|_{\mathscr L_s(\R^n)}\) is finite, $u$ is \( C^{2s+\alpha}\) in a neighbourhood of \(x\), and antisymmetric then \thref{mkG4iRYH} agrees with~\eqref{ZdAlT}. 

See also~\cite{MR3453602, MR4108219, MR4030266, MR4308250} where maximum principles in the setting
of antisymmetric solutions of integro-differential equations have been established by exploiting
the antisymmetry of the functions as in~\eqref{OopaQ0tu}.

It is also worth mentioning that the requirement that \(u\) is antisymmetric in \thref{Hvmln} cannot be entirely removed. In Appendix \ref{j4njb}, we construct a sequence of functions that explicitly demonstrates this. Moreover, \thref{Hvmln} is also false if one only assumes \(u\geqslant0\) in \(\Omega^+\) as proven in \cite[Corollary 1.3]{RoleAntisym2022}.

To obtain the full statement of \thref{Hvmln} we divide the proof into two parts: an interior Harnack inequality and a boundary Harnack inequality close to \(\{x_1=0\}\). The interior Harnack inequality is given as follows.

\begin{thm} \thlabel{DYcYH}
Let \(\rho\in(0,1)\) and \(c\in L^\infty(B_\rho(e_1))\). Suppose that \(u \in C^{2s+\alpha}(B_\rho(e_1))\cap \mathscr{A}_s(\R^n) \) for some \(\alpha >0\) with \(2s+\alpha\) not an integer, \(u\) is non-negative in \(\R^n_+\), and satisfies \begin{align*}
(-\Delta)^s u +cu =0 \qquad\text{in } B_{\rho}(e_1).
\end{align*}

Then there exists \(C_\rho>0\) depending only on \(n\), \(s\), \(\| c \|_{L^\infty(B_\rho(e_1))}\), and \(\rho\) such that \begin{align*}
\sup_{B_{\rho/2}(e_1)} u \leqslant C_\rho \inf_{B_{\rho/2}(e_1)} u.
\end{align*} Moreover, both the quantities \(\sup_{B_{\rho/2}(e_1)} u \) and \(\inf_{B_{\rho/2}(e_1)} u \) are comparable to \(\Anorm{u}\).
\end{thm}

For all \(r>0\), we write \(B_r^+ := B_r \cap \R^n_+\). Then the following result is the antisymmetric boundary Harnack inequality.

\begin{thm} \thlabel{C35ZH}
Let \(\rho>0\) and \(c\in L^\infty(B_\rho^+)\). Suppose that \(u \in C^{2s+\alpha}(B_\rho)\cap \mathscr{A}_s(\R^n) \) for some \(\alpha >0\) with \(2s+\alpha\) not an integer, \(u\) is non-negative in \(\R^n_+\), and satisfies \begin{align*}
(-\Delta)^s u +cu&=0 \qquad \text{in } B_\rho^+.
\end{align*}

Then there exists \(C_\rho>0\) depending only on \(n\), \(s\), \(\| c \|_{L^\infty(B_\rho^+)}\), and \(\rho\) such that\begin{align*}
\sup_{x\in B_{\rho/2}^+} \frac{u(x)}{x_1} \leqslant C_\rho \inf_{x\in B_{\rho/2}^+} \frac{u(x)}{x_1}.
\end{align*}  Moreover, both the quantities \(\sup_{x\in B_{\rho/2}^+} \frac{u(x)}{x_1} \) and \(\inf_{x\in B_{\rho/2}^+} \frac{u(x)}{x_1} \) are comparable to \(\Anorm{u}\).
\end{thm}

To our knowledge, Theorems~\ref{DYcYH} and~\ref{C35ZH} are new in the literature. In the particular case \(c\equiv0\), \thref{C35ZH} was proven in \cite{ciraolo2021symmetry}, but the proof relied on the Poisson representation formula for the Dirichlet problem in a ball which is otherwise unavailable in more general settings.

The proofs of Theorems~\ref{DYcYH} and~\ref{C35ZH} are each split into two halves: for \thref{DYcYH} we prove an interior weak Harnack inequality for super-solutions (\thref{MT9uf}) and interior local boundedness of sub-solutions (\thref{guDQ7}). Analogously, for \thref{C35ZH} we prove a boundary weak Harnack inequality for super-solutions (\thref{SwDzJu9i}) and boundary local boundedness of sub-solutions (\thref{EP5Elxbz}). The proofs of Propositions~\ref{MT9uf}, \ref{guDQ7}, \ref{SwDzJu9i}, and~\ref{EP5Elxbz} make use of general barrier methods which take inspiration from \cite{MR4023466,MR2494809,MR2831115}; however, these methods require adjustments which take into account the antisymmetry assumption. 

Once Theorems~\ref{DYcYH} and \ref{C35ZH} have been established, \thref{Hvmln} follows easily from a standard covering argument. For completeness, we have included this in Section \ref{lXIUl}. 

Finally, in Appendix~\ref{4CEly} we provide an alternate elementary proof of \thref{C35ZH} in the particular case \(c\equiv0\). As in \cite{ciraolo2021symmetry}, our proof relies critically on the Poisson representation formula in a ball for the fractional Laplacian, but the overall strategy of our proof is entirely different. This was necessary to show that \(\sup_{x\in B_{\rho/2}^+} \frac{u(x)}{x_1} \) and \(\inf_{x\in B_{\rho/2}^+} \frac{u(x)}{x_1} \) are comparable to \(\Anorm{u}\) which was not proven in \cite{ciraolo2021symmetry}. Our proof makes use of a new mean-value formula for antisymmetric \(s\)-harmonic functions, which we believe to be interesting
in and of itself. 

The usual mean value formula for \(s\)-harmonic functions says that if \(u\) is \(s\)-harmonic in \(B_1\) then we have that \begin{align}
u(0) =  \gamma_{n,s} \int_{\R^n\setminus B_r} \frac{r^{2s} u(y)}{(\vert y \vert^2-r^2)^s\vert y \vert^n}\dd y \qquad \text{for all }r\in(0, 1] \label{K8kw5rpi}
\end{align}where
\begin{equation}\label{sry6yagamma098765} \gamma_{n,s}:= \frac{\sin(\pi s)\Gamma (n/2)}{\pi^{\frac n 2 +1 }}.
\end{equation}

{F}rom antisymmetry, however, both the left-hand side and the right-hand side of~\eqref{K8kw5rpi} are zero irrespective of the fact \(u\) is \(s\)-harmonic. It is precisely this observation that leads to the consideration of \(\partial_1u(0)\) instead of \(u(0)\) which is more appropriate when \(u\) is antisymmetric.

\begin{prop} \thlabel{cqGgE}
Let \(u\in C^{2s+\alpha}(B_1) \cap \mathscr L_s(\R^n)\) with \(\alpha>0\) and \(2s+\alpha\) not an integer. Suppose that~\(u\) is antisymmetric and \(r\in(0, 1]\).  If \((-\Delta)^s u = 0 \) in \(B_1\) then  \begin{align*}
\frac{\partial u}{\partial x_1} (0) &=  2n \gamma_{n,s} \int_{\R^n_+ \setminus B_r^+} \frac{r^{2s}y_1 u(y)}{(\vert y \vert^2 - r^2)^s\vert y \vert^{n+2}} \dd y ,
\end{align*}
with~$\gamma_{n,s}$ given by~\eqref{sry6yagamma098765}.
\end{prop}

\subsection{Organisation of paper}

The paper is organised as follows. In Section \ref{LenbLVkk}, we motivate \thref{mkG4iRYH} and establish some technical lemmata regarding this definition. In Section~\ref{q7hKF}, we
prove \thref{DYcYH} and, in Section~\ref{TZei6Wd4}, we prove \thref{C35ZH}. In Section~\ref{lXIUl}, we prove the main result \thref{Hvmln}. In Appendix~\ref{j4njb}, we construct a sequence that demonstrates the antisymmetric assumption in \thref{Hvmln} cannot be removed. Finally, in Appendix~\ref{4CEly}, we prove
Proposition~\ref{cqGgE}, and with this, we provide an elementary alternate proof of \thref{C35ZH} for the particular case~\(c\equiv0\).

\section{The antisymmetric fractional Laplacian} \label{LenbLVkk}

Let \(n\) be a positive integer and \(s\in (0,1)\). The purpose of this section is to motivate \thref{mkG4iRYH} as well as prove some technical aspects of the definition such as that it is well-defined and that it coincides with~\eqref{ZdAlT} when \(u\) is sufficiently regular. Recall that, given a point \(x=(x_1,\dots , x_n) \in \R^n\), we will write \(x=(x_1,x')\) with \(x'=(x_2,\dots , x_n)\) and we denote by \(x_\ast\) the reflection of \(x\) across the plane \(\{x_1=0\}\), that is, \(x_\ast = (-x_1,x')\). Then we call a function \(u:\R^n \to \R\) \emph{antisymmetric} if \begin{align*}
u(x_\ast) = -u(x) \qquad \text{for all } x\in \R^n. 
\end{align*} It is common in the literature, particularly when dealing with the method of moving planes, to define antisymmetry with respect to a given hyperplane \(T\); however, for our purposes, it is sufficient to take~\(T=\{x_1=0\}\).

For the fractional Laplacian of \(u\), given by~\eqref{ZdAlT}, to be well-defined in a pointwise sense, we need two ingredients: \(u\) needs enough regularity in a neighbourhood of \(x\) to overcome the singularity at \(x\) in the kernel \(\vert x-y\vert^{-n-2s}\), and \(u\) also needs an integrability condition at infinity to account for the integral in~\eqref{ZdAlT} being over \(\R^n\). For example, if \(u\) is \(C^{2s+\alpha}\) in a neighbourhood of \(x\) for some~\(\alpha >0\) with \(2s+\alpha\) not an integer and \(u\in\mathscr L_s(\R^n)\) where \(\mathscr L_s (\R^n)\) is the set of functions \(v\in L^1_{\mathrm{loc}}(\R^n)\) such that \(\| v\|_{\mathscr L_s(\R^n)}  <+\infty\) (recall \(\| \cdot \|_{\mathscr L_s(\R^n)}\) is given by~\eqref{ACNAzPn8}) then \((-\Delta)^su\) is well-defined at \(x\) and is in fact continuous there, \cite[Proposition 2.4]{MR2270163}. 

In the following proposition, we show that if \(u\in \mathscr L_s(\R^n)\) is antisymmetric and \(u\) is \(C^{2s+\alpha}\) in a neighbourhood of \(x\) for some \(\alpha >0\) with \(2s+\alpha\) not an integer then \(u\) satisfies~\eqref{OopaQ0tu}. This simultaneously motivates \thref{mkG4iRYH} and demonstrates this definition does generalise the definition of the fractional Laplacian when the given function is antisymmetric.

\begin{prop} \thlabel{eeBLRjcZ}
Let \(x\in \R^n_+\) and \(u \in \mathscr L_s(\R^n)\) be an antisymmetric function that is \(C^{2s+\alpha}\) in a neighbourhood of \(x\) for some \(\alpha>0\) with \(2s+\alpha\) not an integer. Then~\eqref{ZdAlT} and \thref{mkG4iRYH} coincide. 
\end{prop}

\begin{proof}
Let~$x=(x_1,x')\in \R^n_+$ and take~$\delta\in\left(0,\frac{x_1}2\right)$ such that~$u$
is~$C^{2s+\alpha}$ in~$B_\delta(x)\subset\R^n_+$. Furthermore, let \begin{align}
(-\Delta)^s_\delta u(x):= c_{n,s}\int_{\R^n  \setminus B_\delta(x)} \frac{u(x)-u(y)}{\vert x - y\vert^{n+2s}} \dd y .\label{WUuXb0Hk}
\end{align}
By the regularity assumptions on \(u\), the integral in~\eqref{ZdAlT} is well-defined and \begin{align*}
 \lim_{\delta\to0^+}(-\Delta)^s_\delta u(x) = c_{n,s} \PV \int_{\R^n}& \frac{u(x)-u(y)}{\vert x - y \vert^{n+2s}} \dd y =(-\Delta)^s u(x). 
\end{align*}

Splitting the integral in~\eqref{WUuXb0Hk} into two integrals over \(\R^n_+  \setminus B_\delta(x)\) and \(\R^n_-\) respectively and
then using that \(u\) is antisymmetric in the integral over \(\R^n_-\), we obtain
\begin{align}
(-\Delta)^s_\delta u(x) &= c_{n,s}\int_{\R^n_+ \setminus B_\delta(x)} \frac{u(x)-u(y)}{\vert x - y\vert^{n+2s}} \dd y +c_{n,s} \int_{\R^n_+} \frac{u(x)+u(y)}{\vert x_\ast - y\vert^{n+2s}} \dd y \nonumber \\
&= c_{n,s}\int_{\R^n_+ \setminus B_\delta(x)} \left(
\frac 1 {\vert x - y \vert^{n+2s}} - \frac 1 {\vert x_\ast- y \vert^{n+2s}} \right)\big(u(x)-u(y)\big) \dd y \nonumber \\
&\qquad +2c_{n,s}u(x) \int_{\R^n_+ \setminus B_\delta(x)} \frac{\dd y }{\vert x_\ast - y \vert^{n+2s}} + c_{n,s} \int_{B_\delta(x)} \frac{u(x)+u(y) }{\vert x_\ast - y \vert^{n+2s}}\dd y . \label{cGXQjDIV}
\end{align}

We point out that if~$y\in\R^n_+$ then~$|x_\ast-y|\ge  x_1$,
and therefore
\begin{eqnarray*}
&& \left \vert \int_{B_\delta(x)} \frac{u(x)+u(y) }{\vert x_\ast - y \vert^{n+2s}}\dd y \right \vert\le
\int_{B_\delta(x)} \frac{2|u(x)|+|u(y)-u(x)| }{\vert x_\ast - y \vert^{n+2s}}\dd y\\
&&\qquad\qquad \le \int_{B_\delta(x)} \frac{2|u(x)|+C|x-y|^{\min\{1,2s+\alpha\}} }{\vert x_\ast - y \vert^{n+2s}}\dd y\le  C\big (  \vert u(x) \vert + \delta^{\min\{1,2s+\alpha\}} \big ) \delta^n,
\end{eqnarray*}
for some~$C>0$ depending on~$n$, $s$, $x$, and the norm of~$u$ in a neighbourood of~$x$.

As a consequence, sending \(\delta \to 0^+\) in~\eqref{cGXQjDIV}
and using the Dominated Convergence Theorem,
we obtain \begin{align*}
(-\Delta)^s u(x) &= c_{n,s}\lim_{\delta \to 0^+}\int_{\R^n_+ \setminus B_\delta(x)} \left( \frac 1 {\vert x - y \vert^{n+2s}} - \frac 1 {\vert x_\ast- y \vert^{n+2s}} \right)\big(u(x)-u(y)\big) \dd y  \\
&\qquad +2c_{n,s}u(x) \int_{\R^n_+} \frac{\dd y }{\vert x_\ast - y \vert^{n+2s}}. 
\end{align*}

Also, making the change of variables \(z=(y_1/x_1 , (y'-x')/x_1)\) if \(n>1\) and \(z_1=y_1/x_1\) if \(n=1\), we see that \begin{align*}
 \int_{\R^n_+} \frac{\dd y }{\vert x_\ast - y \vert^{n+2s}}  &= x_1^{-2s}  \int_{\R^n_+}   \frac{\dd z}{\vert e_1+ z \vert^{n+2s}}. 
\end{align*}
Moreover, via a direct computation  \begin{equation}\label{sd0w3dewt95b76 498543qdetr57uy54u}
 c_{n,s}   \int_{\R^n_+}   \frac{\dd z}{\vert e_1+ z \vert^{n+2s}} = \frac {c_{1,s}}{2s} 
\end{equation} see \thref{tuMIrhco} below for details. 

Putting together these considerations, we obtain the desired result.
\end{proof}

\begin{remark} \thlabel{tuMIrhco}
Let \begin{align}
\tilde c_{n,s} := c_{n,s}   \int_{\R^n_+}   \frac{\dd z}{\vert e_1+ z \vert^{n+2s}}. \label{CA7GPGvx}
\end{align}The value of \(\tilde c_{n,s}\) is of no consequence to \thref{mkG4iRYH} and so, for this paper, is irrelevant; however, it is interesting to note that \(\tilde c_{n,s}\) can be explicitly evaluated and is, in fact, equal to \((2s)^{-1}c_{1,s}\) (in particular \(\tilde c_{n,s}\) is independent of \(n\) which is not obvious from~\eqref{CA7GPGvx}). Indeed, if \(n>1\) then \begin{align*}
 \int_{\R^n_+}   \frac{\dd z}{\vert e_1+ z \vert^{n+2s}} &= \int_0^\infty \int_{\R^{n-1}}   \frac{\dd z_1\dd z'}{\big ( (z_1+1)^2 + \vert z' \vert^2\big )^{\frac{n+2s}2}}
\end{align*} so, making the change of variables \(\tilde z'= (z_1+1)^{-1} z'\) in the inner integral, we have that \begin{align*}
 \int_{\R^n_+}   \frac{\dd z}{\vert e_1+ z \vert^{n+2s}} &=  \int_0^\infty   \frac 1 {(z_1+1)^{1+2s}} \left(  \int_{\R^{n-1}}   \frac{\dd \tilde z'}{\big ( 1 + \vert \tilde z' \vert^2\big )^{\frac{n+2s}2}} \right) \dd z_1.
\end{align*} By \cite[Proposition 4.1]{MR3916700}, \begin{align*}
\int_{\R^{n-1}}   \frac{\dd \tilde z'}{\big ( 1 + \vert \tilde z' \vert^2\big )^{\frac{n+2s}2}} &= \frac{\displaystyle
\pi^{\frac{n-1}2} \Gamma \big ( \frac{1+2s}2 \big )}{\displaystyle \Gamma \big ( \frac{n+2s}2 \big )}.
\end{align*} Hence, \begin{align*}
\int_{\R^n_+}   \frac{\dd z}{\vert e_1+ z \vert^{n+2s}} &= \frac{\pi^{\frac{n-1}2} \Gamma \big ( \frac{1+2s}2 \big )}{\Gamma \big ( \frac{n+2s}2 \big )} \int_0^\infty   \frac {\dd z_1} {(z_1+1)^{1+2s}} =  \frac{\pi^{\frac{n-1}2} \Gamma \big ( \frac{1+2s}2 \big )}{2s \Gamma \big ( \frac{n+2s}2 \big )}.
\end{align*} This formula remains valid if \(n=1\). Since \(c_{n,s} = s \pi^{-n/2}4^s \Gamma \big ( \frac{n+2s}2 \big )/ \Gamma (1-s)\), see \cite[Proposition~5.6]{MR3916700}, it follows that \begin{align*}
\tilde c_{n,s} &= \frac{2^{2s-1} \Gamma \big ( \frac{1+2s}2 \big )}{\pi^{1/2} \Gamma(1-s) } = \frac { c_{1,s}}{2s}.
\end{align*}
In particular, this proves formula~\eqref{sd0w3dewt95b76 498543qdetr57uy54u},
as desired.
\end{remark}

The key observation from \thref{eeBLRjcZ} is that the kernel \( {\vert x - y \vert^{-n-2s}} -  {\vert x_\ast- y \vert^{-n-2s}}\) decays quicker at infinity than the kernel \(\vert x-y \vert^{-n-2s}\) which is what allows us to relax the assumption~\(u\in \mathscr L_s (\R^n)\). Indeed, if~$x$, $y\in \R^n_+ $ then \(\vert x_\ast - y \vert \geqslant \vert x-y \vert\) and \begin{align}
 \frac 1 {\vert x - y \vert^{n+2s}} - \frac 1 {\vert x_\ast- y \vert^{n+2s}} &= \frac{n+2s} 2 \int_{\vert x - y \vert^2}^{\vert x_\ast -y \vert^2} t^{-\frac{n+2s+2}2} \dd t \nonumber \\
 &\leqslant \frac{n+2s} 2  \big ( \vert x_\ast -y \vert^2-\vert x -y \vert^2 \big ) \frac 1 {\vert x - y \vert^{n+2s+2}} \nonumber \\
 &= 2(n+2s) \frac{x_1y_1}{\vert x - y \vert^{n+2s+2}}. \label{LxZU6}
\end{align} Similarly, for all~$x$, $y\in \R^n_+$,\begin{align}
\frac 1 {\vert x - y \vert^{n+2s}} - \frac 1 {\vert x_\ast- y \vert^{n+2s}} &\geqslant 2(n+2s)  \frac{x_1y_1}{\vert x_\ast - y \vert^{n+2s+2}}. \label{buKHzlE6}
\end{align} Hence, if \(\vert x\vert <C_0\) and if \(|x-y|>C_1\) for some \(C_0,C_1>0\) independent of \(x,y\), we find that \begin{align}
\frac{C^{-1} x_1y_1}{1+\vert y \vert^{n+2s+2}} \leqslant \frac 1 {\vert x - y \vert^{n+2s}} - \frac 1 {\vert x_\ast- y \vert^{n+2s}} \leqslant \frac{Cx_1y_1}{1+\vert y \vert^{n+2s+2}} \label{Zpwlcwex}
\end{align} which motivates the consideration of \(\mathscr A_s(\R^n)\).

Our final lemma in this section proves that \thref{mkG4iRYH} is well-defined, that is, the assumptions of \thref{mkG4iRYH} are enough to guarantee \eqref{OopaQ0tu} converges.

\begin{lem} \thlabel{zNNKjlJJ}
Let \(x \in \R^n_+\) and \(r\in (0,x_1/2)\) be sufficiently small so that \(B:=B_r(x) \Subset\R^n_+\). If \(u \in C^{2s+\alpha}(B) \cap  \mathscr A_s(\R^n)\) for some \(\alpha>0\) with \(2s+\alpha\) not an integer then \((-\Delta)^su(x)\) given by \eqref{OopaQ0tu} is well-defined and there exists \(C>0\) depending on \(n\), \(s\), \(\alpha\), \(r\), and \(x_1\) such that \begin{align*}
\big \vert (-\Delta)^s u(x) \big  \vert &\leqslant C \big (  \| u\|_{C^{\beta}(B)} + \Anorm{u} \big ) 
\end{align*} where \(\beta = \min \{ 2s+\alpha ,2\}\).
\end{lem}

We remark that the constant \(C\) in \thref{zNNKjlJJ} blows up as \(x_1\to 0^+\). 

\begin{proof}[Proof of Lemma~\ref{zNNKjlJJ}]
Write \( (-\Delta)^s u(x) = I_1 +I_2 + s^{-1} c_{1,s} u(x) x_1^{-2s}\) where  \begin{align*}
I_1 &:=  c_{n,s}\lim_{\varepsilon \to 0^+} \int_{B\setminus B_\varepsilon(x)} \left(
\frac 1 {\vert x - y \vert^{n+2s}} - \frac 1 {\vert x_\ast- y \vert^{n+2s}}\right) \big(u(x)-u(y)\big)\dd y\\
\text{ and } \qquad I_2 &:=  c_{n,s}\int_{\R^n_+\setminus B} \left(\frac 1 {\vert x - y \vert^{n+2s}} - \frac 1 {\vert x_\ast- y \vert^{n+2s}}\right) \left(u(x)-u(y)\right)\dd y.
\end{align*}

\emph{For \(I_1\):} If \(y\in \R^n_+\) then \({\vert x - y \vert^{-n-2s}} -  {\vert x_\ast- y \vert^{-n-2s}} \) is only singular at \(x\), so we may write \begin{align*}
I_1 &= c_{n,s}\lim_{\varepsilon \to 0^+} \int_{B\setminus B_\varepsilon(x)}\frac {u(x)-u(y)} {\vert x - y \vert^{n+2s}}\dd y - \int_B \frac {u(x)-u(y)} {\vert x_\ast- y \vert^{n+2s}}\dd y.
\end{align*}If \(\chi_B\) denotes the characteristic function of \(B\) then \(\bar  u := u\chi_B\in C^{2s+\alpha}(B)\cap L^\infty(\R^n)\), so \((-\Delta)^s\bar u(x)\) is defined and \begin{align*}
(-\Delta)^s\bar u(x) &= c_{n,s}\PV \int_{\R^n} \frac{u(x) - u(y)\chi_B(y)}{\vert x-y\vert^{n+2s}} \dd y \\
&= c_{n,s} \PV \int_B \frac{u(x) - u(y)}{\vert x-y\vert^{n+2s}} \dd y + c_{n,s} u(x)\int_{\R^n\setminus B} \frac{\dd y }{\vert x-y\vert^{n+2s}}\\
&=  c_{n,s}\PV \int_B \frac{u(x) - u(y)}{\vert x-y\vert^{n+2s}} \dd y +  \frac 1{2s}c_{n,s} n \vert B_1 \vert r^{-2s} u(x) .
\end{align*}Hence, \begin{align*}
I_1&= (-\Delta)^s\bar u(x) -\frac 1{2s}c_{n,s} n \vert B_1 \vert r^{-2s} u(x) -c_{n,s} \int_B \frac {u(x)-u(y)} {\vert x_\ast- y \vert^{n+2s}}\dd y 
\end{align*} which gives that \begin{align*}
\vert I_1 \vert &\leqslant  \big \vert (-\Delta)^s\bar u(x) \big \vert +C \| u\|_{L^\infty(B)} +2\| u\|_{L^\infty(B)} \int_B \frac 1 {\vert x_\ast- y \vert^{n+2s}}\dd y  \\
&\leqslant   \big \vert (-\Delta)^s\bar u(x) \big \vert  +C\| u\|_{L^\infty(B)}.
\end{align*}
Furthermore, if \(\alpha <2(1-s)\) then by a Taylor series expansion  \begin{align*}
\big \vert (-\Delta)^s\bar u (x) \big \vert &\leqslant C \int_{\R^n} \frac{\vert 2\bar u(x)-\bar u(x+y) - \bar u(x-y) \vert }{\vert y \vert^{n+2s}} \dd y \\
&\leqslant C [\bar u]_{C^{2s+\alpha}(B)}  \int_{B_1} \frac{\dd y  }{\vert y \vert^{n-\alpha}}  + C \| \bar u \|_{L^\infty(\R^n)} \int_{\R^n\setminus B_1} \frac{\dd y }{\vert y \vert^{n+2s}} \\
&\leqslant C  \big (  \| \bar u\|_{C^{2s+\alpha}(B)} + \| \bar u \|_{L^\infty(\R^n)} \big ) \\
&=C  \big ( \| u\|_{C^{2s+\alpha}(B)} + \| u \|_{L^\infty(B)} \big )  .
\end{align*} If \(\alpha \geqslant 2(1-s)\) then this computation holds with \([\bar u]_{C^{2s+\alpha}(B)} \) replaced with \(\|u\|_{C^2(B)}\).

\emph{For \(I_2\):} By~\eqref{LxZU6}, \begin{align*}
\vert I_2 \vert &\leqslant Cx_1\int_{\R^n_+\setminus B} \frac {y_1\big (\vert u(x) \vert + \vert u(y) \vert  \big ) } {\vert x - y\vert^{n+2s+2}} \dd y \\
&\leqslant C \| u \|_{L^\infty(B)}\int_{\R^n_+\setminus B} \frac {y_1} {\vert x - y\vert^{n+2s+2}} \dd y + C \Anorm{u} \\
&\leqslant C \big (\| u \|_{L^\infty(B)}+ \Anorm{u} \big ) .
\end{align*}

Combining the estimates for \(I_1\) and \(I_2\) immediately gives the result. 
\end{proof}

\section{Interior Harnack inequality and proof of Theorem~\ref{DYcYH}}\label{q7hKF}

The purpose of this section is to prove \thref{DYcYH}. Its proof is split into two parts: the interior weak Harnack inequality for super-solutions (\thref{MT9uf}) and interior local boundedness for sub-solutions (\thref{guDQ7}).

\subsection{The interior weak Harnack inequality}

The interior weak Harnack inequality for super-solutions is given in \thref{MT9uf} below. 

\begin{prop} \thlabel{MT9uf} 
Let \(M\in \R \), \(\rho \in(0,1)\), and \(c\in L^\infty(B_\rho(e_1))\). Suppose that \(u \in C^{2s+\alpha}(B_\rho(e_1))\cap \mathscr{A}_s(\R^n)\) for some \(\alpha>0\) with \(2s+\alpha\) not an integer, \(u\) is non-negative in \(\R^n_+\), and satisfies  \begin{align*}
(-\Delta)^su +cu &\geqslant -M \qquad \text{in } B_\rho(e_1).
\end{align*}

Then there exists \(C_\rho>0\) depending only on \(n\), \(s\), \(\| c \|_{L^\infty(B_\rho(e_1))}\), and \(\rho\) such that \begin{align*}
\Anorm{u} &\leqslant C_\rho \left( \inf_{B_{\rho/2}(e_1)} u + M \right) . 
\end{align*}
\end{prop}

The proof of \thref{MT9uf} is a simple barrier argument that takes inspiration from \cite{MR4023466}. We will begin by proving the following proposition, \thref{YLj1r},  which may be viewed as a rescaled version of \thref{MT9uf}.  We will require \thref{YLj1r} in the second part of the section when we prove the interior local boundedness of sub-solutions (\thref{guDQ7}). 

\begin{prop} \thlabel{YLj1r} 
Let \(M\in \R \), \(\rho \in(0,1)\), and \(c\in L^\infty(B_\rho(e_1))\). Suppose that \(u \in C^{2s+\alpha}(B_\rho(e_1))\cap \mathscr{A}_s(\R^n)\) for some \(\alpha>0\) with \(2s+\alpha\) not an integer, \(u\) is non-negative in \(\R^n_+\), and satisfies  \begin{align*}
(-\Delta)^su +cu &\geqslant -M \qquad \text{in } B_\rho(e_1).
\end{align*}

Then there exists \(C>0\) depending only on \(n\), \(s\), and~\(\rho^{2s}\| c \|_{L^\infty(B_\rho(e_1))}\),
such that  \begin{align*}
 \frac 1 {\rho^n} \int_{B_{\rho/2}(e_1)} u(x) \dd x &\leqslant C \left( \inf_{B_{\rho/2}(e_1)} u + M \rho^{2s}\right) . 
\end{align*} Moreover, the constant \(C\) is of the form \begin{align}
C = C' \left(1+ \rho^{2s} \| c \|_{L^\infty(B_\rho(e_1))}\right) \label{CyJJQHrF}
\end{align} with \(C'>0\) depending only on \(n\) and \(s\). 
\end{prop}

Before we give the proof of \thref{YLj1r}, we introduce some notation.
Given~$ \rho \in(0,1)$, we define
$$  T_\rho:= \left\{x_1=-\frac1{\rho}\right\}\qquad{\mbox{and}}\qquad
H_\rho:=\left\{x_1>-\frac1{\rho}\right\}.$$
We also let~$Q_\rho$ be the reflection across the hyperplane~$T_\rho$, namely
$$Q_\rho(x):=x-2(x_1+1/\rho)e_1\qquad{\mbox{for all }} x\in \R^n. $$
With this, we establish the following lemma.

\begin{lem} \thlabel{rvZJdN88}
Let \(\rho \in(0,1)\) and \(\zeta\) be a smooth cut-off function such that \begin{align*}
\zeta \equiv 1 \text{ in } B_{1/2}, \quad \zeta \equiv 0 \text{ in } \R^n \setminus B_{3/4}, \quad{\mbox{ and }}
\quad
0\leqslant \zeta \leqslant 1.
\end{align*} Define \(\varphi^{(1)} \in C^\infty_0(\R^n)\) by \begin{align*}
\varphi^{(1)}(x) := \zeta (x) - \zeta \big (  Q_\rho(x) \big ) \qquad \text{for all } x\in \R^n.
\end{align*}

Then \(\varphi^{(1)}\) is antisymmetric with respect to \(T_\rho:= \{x_1=-1/\rho\}\) and there exists \(C>0\) depending only on \(n\) and \(s\) (but not on \(\rho\)) such that \(\| (-\Delta)^s \varphi^{(1)} \|_{L^\infty(B_{3/4})} \leqslant C\).
\end{lem}

\begin{proof}
Since \(Q_\rho(x)\) is the reflection of \(x\in \R^n\) across \(T_\rho\), we immediately obtain that \(\varphi^{(1)} \) is antisymmetric with respect to the plane \(T_\rho\). As \(0\leqslant \zeta \circ Q_\rho \leqslant 1\) in \(\R^n\) and \(\zeta \circ Q_\rho=0\) in \(B_{3/4}\), from~\eqref{ZdAlT} we have that \begin{align*}
\vert (-\Delta)^s (\zeta \circ Q_\rho)(x) \vert &=C \int_{B_{3/4}(-2e_1/\rho)} \frac{(\zeta \circ Q_\rho)(y)}{\vert x - y\vert^{n+2s}} \dd y \leqslant C \qquad \text{for all }x \in B_{3/4} 
\end{align*} using also
that \(\vert x- y\vert \geqslant 2(1/\rho -3/4) \geqslant 1/2\). Moreover, \begin{align*}
\|  (-\Delta)^s \zeta \|_{L^\infty(B_{3/4})} &\leqslant C ( \| D^2\zeta \|_{L^\infty(B_{3/4})} + \| \zeta \|_{L^\infty(\R^n)} ) \leqslant C, 
\end{align*} for example, see the computation on p.~9 of \cite{MR3469920}. Thus,
$$\| (-\Delta)^s \varphi^{(1)} \|_{L^\infty(B_{3/4})} \leqslant \| (-\Delta)^s \zeta\|_{L^\infty(B_{3/4})} + \| (-\Delta)^s (\zeta \circ Q_\rho)\|_{L^\infty(B_{3/4})} \leqslant C,$$
which completes the proof. 
\end{proof}

Now we give the proof of \thref{YLj1r}. 

\begin{proof}[Proof of \thref{YLj1r}]
Let \(\tilde u(x):=u (\rho x+e_1)\) and \(\tilde c(x) := \rho^{2s}c(\rho x+e_1)\). Observe that \((-\Delta)^s\tilde u+\tilde c \tilde u \geqslant - M\rho^{2s}\) in \(B_1\) and that \(\tilde u \) is antisymmetric with respect to \(T_\rho\).

Let \(\varphi^{(1)}\) be defined as in \thref{rvZJdN88} and suppose that\footnote{Note that such a \(\tau\) exists. Indeed, let \(U \subset \R^n\), \(u:U \to \R\) be a nonnegative function and let \(\varphi : U \to \R\) be such that there exists \(x_0\in U\) for which \(\varphi(x_0)>0\). Then define \begin{align*}
I := \{ \tau \geqslant 0 \text{ s.t. } u(x) \geqslant \tau \varphi(x) \text{ for all } x\in U \} .
\end{align*} It follows that \(I\) is non-empty since \(0\in I\), so \(\sup I\) exists, but may be \(+\infty\). Moreover, \(\sup I \leqslant u(x_0)/\varphi(x_0) < +\infty\), so \(\sup I\) is also finite.} \(\tau\geqslant 0\) is the largest possible value such that~\(\tilde u  \geqslant \tau \varphi^{(1)} \) in the half space~\(H_\rho\). Since \(\varphi^{(1)} =1\) in \(B_{1/2}\), we immediately obtain that \begin{align}
\tau \leqslant \inf_{B_{1/2}} \tilde u = \inf_{B_{\rho/2}(e_1)} u. \label{mzMEg}
\end{align}

Moreover, by continuity, there exists \(a \in \overline{B_{3/4}}\) such that \(\tilde u(a)=\tau\varphi^{(1)} (a)\). On one hand, using \thref{rvZJdN88}, we have that \begin{align}
(-\Delta)^s(\tilde u-\tau \varphi^{(1)})(a)+\tilde c (a) (\tilde u-\tau \varphi^{(1)})(a) &\geqslant -M\rho^{2s} - \tau \big(  C+ \|\tilde c\|_{L^\infty(B_1)}\big) \nonumber \\
&\geqslant -M\rho^{2s} - C\tau \big(  1+ \rho^{2s}\|c\|_{L^\infty(B_\rho(e_1))}\big) .\label{aDLwG3pX}
\end{align} On the other hand, since \(\tilde u-\tau\varphi^{(1)}\) is antisymmetric with respect to \(T_\rho\), \(\tilde u - \tau \varphi^{(1)} \geqslant 0\) in \(H_\rho\), and~\((\tilde u-\tau \varphi^{(1)})(a)=0\), it follows that \begin{align}
&(-\Delta)^s (\tilde u-\tau \varphi^{(1)})(a)+c_\rho(a)(\tilde u-\tau \varphi^{(1)})(a) \nonumber \\
&\leqslant - C \int_{B_{1/2}} \left(
\frac 1 {\vert a - y \vert^{n+2s}} - \frac 1 {\vert Q_\rho(a) - y \vert^{n+2s}} \right) \big(
\tilde u(y)-\tau \varphi^{(1)}(y)\big) \dd y. \label{wPcy8znV}
\end{align} For all \(y \in B_{1/2}\), we have that \(\vert a - y \vert \leqslant C\) and \(\vert Q_\rho(a) - y \vert \geqslant C\) (the assumption \(\rho <1\) allows to choose this \(C\) independent of \(\rho\)), so \begin{align}
(-\Delta)^s (\tilde u-\tau \varphi^{(1)})(a)+c_\rho(a)(\tilde u-\tau \varphi^{(1)})(a) &\leqslant - C \int_{B_{1/2}}  \big(\tilde u(y)-\tau \varphi^{(1)}(y)\big) \dd y \nonumber \\
&\leqslant -C \left( \frac 1 {\rho^n} \int_{B_{\rho/2}(e_1)}  u(y) \dd y - \tau \right). \label{ETn5BCO5}
\end{align} Rearranging \eqref{aDLwG3pX} and \eqref{ETn5BCO5} then using \eqref{mzMEg}, we obtain \begin{align*}
\frac 1 {\rho^n} \int_{B_{\rho/2}(e_1)}  u(y) \dd y &\leqslant C  \Big( \tau \big(  1+ \rho^{2s}\|c\|_{L^\infty(B_\rho(e_1))}\big)  + M\rho^{2s} \Big) \\
&\leqslant C \big(  1+ \rho^{2s}\|c\|_{L^\infty(B_\rho(e_1))}\big) \left(
\inf_{B_{\rho/2}(e_1)} u + M \rho^{2s}\right)
\end{align*} as required.
\end{proof}

A simple adaptation of the proof of \thref{YLj1r} leads to the proof of \thref{MT9uf} which we now give. 

\begin{proof}[Proof of \thref{MT9uf}]
Follow the proof of \thref{YLj1r} but instead of \eqref{wPcy8znV}, we write \begin{align*}
&(-\Delta)^s (\tilde u-\tau \varphi^{(1)})(a)+c_\rho(a)(\tilde u-\tau \varphi^{(1)})(a)\\
&\qquad= - C \int_{H_\rho} \left(
\frac 1 {\vert a - y \vert^{n+2s}} - \frac 1 {\vert Q_\rho(a) - y \vert^{n+2s}} \right) \big(
\tilde u(y)-\tau \varphi^{(1)}(y)\big) \dd y.
\end{align*} Then, for all \(x,y\in H_\rho\), \begin{align}
\frac 1 {\vert x - y \vert^{n+2s}} - \frac 1 {\vert Q_\rho(x) - y \vert^{n+2s}} &= \frac{n+2s} 2 \int_{\vert x-y\vert^2}^{\vert Q_\rho(x)-y\vert^2} t^{- \frac{n+2s+2}2} \dd t \label{gvKkSL1N}\\
&\geqslant C \Big (\vert Q_\rho(x)-y\vert^2-\vert x-y\vert^2 \Big ) \vert Q_\rho(x) - y \vert^{- (n+2s+2)}  \nonumber\\
&= C \frac{(x_1+1/\rho)(y_1+1/\rho)}{\vert Q_\rho(x) - y \vert^{n+2s+2}}, \nonumber
\end{align} so using that \(\tilde u  - \tau \varphi^{(1)} \geqslant 0\) in \(H_\rho\), we see that \begin{align*}
(-\Delta)^s (\tilde u-\tau \varphi^{(1)})(a)+c_\rho(a)(\tilde u-\tau \varphi^{(1)})(a) 
&\leqslant - C \int_{H_\rho} \frac{(y_1+1/\rho)(\tilde u(y)-\tau \varphi^{(1)}(y))  }{\vert Q_\rho(a) - y \vert^{n+2s+2}}\dd y \\
&\leqslant - C_\rho  \left( \int_{H_\rho} \frac{(y_1+1/\rho)\tilde u(y) }{\vert Q_\rho(a) - y \vert^{n+2s+2}}\dd y + \tau \right) .
\end{align*} Making the change of variables \(z=\rho y +e_1\), we have that \begin{align*}
\int_{H_\rho} \frac{(y_1+1/\rho)\tilde u(y) }{\vert Q_\rho(a) - y \vert^{n+2s+2}}\dd y  &= \rho^{-n-1} \int_{\R^n_+} \frac{z_1 u(z) }{\vert Q_\rho(a) - z/\rho +e_1/\rho \vert^{n+2s+2}}\dd z.
\end{align*} Thus, since \( \vert Q_\rho(a) - z/\rho +e_1/\rho \vert^{n+2s+2} \leqslant C_\rho (1+ \vert z \vert^{n+2s+2} ) \), we conclude that
\begin{align*}
\int_{H_\rho} \frac{(y_1+1/\rho)\tilde u(y) }{\vert Q_\rho(a) - y \vert^{n+2s+2}}\dd y &\geqslant C_\rho \int_{\R^n_+} \frac{z_1 u(z) }{1+ \vert z \vert^{n+2s+2}} \dd z=C_\rho \Anorm{u}.
\end{align*} As a consequence,
\begin{align}
(-\Delta)^s (\tilde u-\tau \varphi^{(1)})(a)+\tilde c(a)(\tilde u-\tau \varphi^{(1)})(a) &\leqslant- C_\rho \Anorm{u}+C_\rho \tau  . \label{Pzf2a}
\end{align} Rearranging~\eqref{aDLwG3pX} and~\eqref{Pzf2a} then using~\eqref{mzMEg} gives \begin{align*}
\Anorm{u} &\leqslant C_\rho ( \tau + M ) \leqslant C_\rho \left( \inf_{B_{\rho/2}(e_1)} u + M \right), 
\end{align*}
as desired.
\end{proof}

\subsection{Interior local boundedness}

The second part of the proof of \thref{DYcYH} is the interior local boundedness of sub-solutions given in \thref{guDQ7} below. 

\begin{prop} \thlabel{guDQ7}
Let \(M \geqslant 0\), \(\rho\in(0,1/2)\), and \(c\in L^\infty(B_\rho(e_1))\). Suppose that \(u \in C^{2s+\alpha}(B_\rho(e_1))\cap \mathscr A_s(\R^n)\) for some \(\alpha>0\) with \(2s+\alpha\) not an integer, and \(u\) satisfies   \begin{equation}\label{sow85bv984dert57nb5}
(-\Delta)^su +cu \leqslant M \qquad \text{in } B_\rho(e_1).
\end{equation} 

Then there exists \(C_\rho>0\) depending only on \(n\), \(s\), \(\| c \|_{L^\infty(B_\rho(e_1))}\), and \(\rho\) such that  \begin{align*}
\sup_{B_{\rho/2}(e_1)} u &\leqslant C_\rho ( \Anorm{u} +M  ) .
\end{align*} 
\end{prop} 

The proof of \thref{guDQ7} uses similar ideas to \cite[Theorem~11.1]{MR2494809} and \cite[Theorem~5.1]{MR2831115}. Before we prove \thref{guDQ7}, we need the following lemma. 

\begin{lem} \thlabel{ltKO2}
Let \(R\in(0,1)\) and \(a\in \overline{ B_2^+}\). Then there exists \(C>0\) depending only on \(n\) and \(s\) such that, for all \(x\in B_{R/2}^+(a)\) and \(y \in \R^n_+ \setminus B_R^+(a)\), \begin{align*}
\frac 1 {\vert x - y \vert^{n+2s}} - \frac 1 {\vert x_\ast - y \vert^{n+2s}} &\leqslant C R^{-n-2s-2} \frac{x_1y_1}{1+\vert y \vert^{n+2s+2}}.
\end{align*} 
\end{lem}

\begin{proof} Let \(\tilde x := (x-a)/R\in B_{1/2}\) and \(\tilde y = (y-a)/ R\in \R^n\setminus B_1\). Clearly,
we have that~\(\vert \tilde y - \tilde x \vert \geqslant 1/2\). Moreover,  since \(\vert \tilde x \vert <1/2 < 1/(2 \vert \tilde y \vert)\), we have that \(\vert \tilde y - \tilde x \vert \geqslant \vert \tilde y \vert -\vert \tilde x \vert  \geqslant (1/2) \vert \tilde y \vert\). Hence, \begin{align*}
\vert \tilde y - \tilde x \vert^{n+2s+2} \geqslant \frac 1 {2^{n+2s+2}}  \max  \big \{ 1 , \vert \tilde y \vert^{n+2s+2} \big \} \geqslant C \big ( 1+  \vert \tilde y \vert^{n+2s+2} \big ) 
\end{align*} for some \(C\) depending only on \(n\) and \(s\). It follows that
\begin{align}
\vert x -y \vert^{n+2s+2} = R^{n+2s+2}  \vert \tilde x - \tilde y  \vert^{n+2s+2} 
\geqslant C R^{n+2s+2} \big ( 1 + R^{-n-2s-2} \vert y-a \vert^{n+2s+2} \big ). \label{TO9DRBuX}
\end{align}

Finally, we claim that there exists \(C\) independent of \(R\) such that \begin{align}
\vert y - a\vert \geqslant CR \vert y \vert \qquad \text{for all } y\in \R^n \setminus B_R(a) . \label{tdNyAxBN}
\end{align} Indeed, if \(y \in \R^n \setminus B_4\) then \begin{align*}
\vert y - a\vert \geqslant \vert y \vert -2 \geqslant \frac12 \vert y \vert>\frac R 2 \vert y \vert,
\end{align*} and if \(y \in (\R^n \setminus B_R(a) ) \cap B_4\) then \begin{align*}
\vert y -a\vert \geqslant R > \frac R 4 \vert y \vert
\end{align*} which proves \eqref{tdNyAxBN}.

Thus, \eqref{TO9DRBuX} and \eqref{tdNyAxBN} give that, for all \( x \in B_{R/2}(a)\) and \( y \in \R^n_+ \setminus B_R(a)\), \begin{align*}
\vert x -y \vert^{n+2s+2} &\geqslant C R^{n+2s+2} \big ( 1 +  \vert y\vert^{n+2s+2} \big ) .
\end{align*} Then the result follows directly from~\eqref{LxZU6}. 
\end{proof}

With this preliminary work, we now focus on the proof of \thref{guDQ7}.

\begin{proof}[Proof of \thref{guDQ7}]
We first observe that,
dividing through by~\( \Anorm{u} +M\), we may
assume that \((-\Delta)^su +cu\leqslant 1\) in \(B_\rho(e_1)\) and~\(\Anorm{u} \leqslant 1\). 

We also point out that if~$u\le0$ in~$B_\rho(e_1)$, then the claim in \thref{guDQ7}
is obviously true. Therefore, we can suppose that
\begin{equation}\label{upos44567890}
\{u>0\}\cap B_\rho(e_1)\ne \varnothing.\end{equation}
Thus, we let~\(\tau\geqslant 0 \) be the smallest possible value such that \begin{align*}
u(x) &\leqslant \tau (\rho-\vert x-e_1 \vert )^{-n-2} \qquad \text{for all } x\in B_\rho(e_1).
\end{align*} Such a \(\tau\) exists in light of~\eqref{upos44567890} by following a similar argument to the one in the footnote at the bottom of p. 11.

To complete the proof, we will show that \begin{align}
\tau &\leqslant C_\rho  \label{svLyO}
\end{align}with \(C_\rho\) depending only on \(n\), \(s\), \(\| c\|_{L^\infty(B_\rho(e_1))}\),
and \(\rho\) (but independent of \(u\)). Since \(u\) is uniformly
continuous in~$B_\rho(e_1)$, there exists \(a\in B_\rho(e_1)\) such that \(u(a) = \tau (\rho-\vert a-e_1\vert)^{-n-2}\).
Notice that~$u(a)>0$ and~\(\tau >0\). 

Let also~\(d=:\rho-\vert a-e_1\vert\) so that \begin{align}
u(a) = \tau d^{-n-2} ,\label{CEnkR}
\end{align}  and let  \begin{align*}
U:= \bigg \{ y \in B_\rho (e_1) \text{ s.t. } u(y) > \frac{u(a)} 2 \bigg \} .
\end{align*}  Since \(\Anorm{u}\leqslant 1\),  if \(r\in(0,d)\) then \begin{align*}
C_\rho &\geqslant  \int_{B_\rho (e_1)} |u(x)| \dd x
\ge \int_{ U\cap B_r (a)} u(x) \dd x \geqslant \frac{u(a)}2 \cdot \vert U \cap B_r (a) \vert  .
\end{align*} Thus, from~\eqref{CEnkR}, it follows that  \begin{align}
 \vert  U \cap B_r (a) \vert \leqslant \frac{C_\rho d^{n+2}}\tau \qquad \text{ for all } r\in(0,d). \label{W5Ar0}
\end{align}

Next, we make the following claim. 

\begin{claim} 
There exists~$\theta_0\in(0,1)$ depending only on \(n\), \(s\), \(\| c\|_{L^\infty(B_\rho(e_1))}\), and \(\rho\)
such that if~\(\theta\in(0,\theta_0]\) there exists~\(C>0\) depending only on \(n\), \(s\), \(\| c\|_{L^\infty(B_\rho(e_1))}\), \(\rho\), and~$\theta$ such that \begin{align*}
\big \vert B_{\theta d /8} (a)  \setminus U \big \vert \leqslant  \frac 1 4 \vert B_{\theta d /8} \vert +C \frac{d^n }{\tau} . 
\end{align*} In particular, neither \(\theta\) nor \(C\) depend on \(\tau\), \(u\), or \(a\).
\end{claim} 

We will withhold the proof of the claim until the end. Assuming the claim is true, we complete the proof of the \thref{guDQ7} as follows. By~\eqref{W5Ar0} (used here with~$r:={\theta_0 d}/8$)
and the claim, we have that \begin{align*}
\frac{C_\rho d^{n+2}}\tau  \geqslant \big \vert B_{\theta_0 d /8}  \big \vert -\big \vert B_{\theta_0 d /8} (a)  \setminus U \big \vert 
\geqslant \frac 3 4 \vert B_{\theta_0 d/8} \vert -C \frac{d^n }{\tau} .
\end{align*} Rearranging gives that \( \tau \leqslant C( d^2 +1 ) \leqslant C\), which proves \eqref{svLyO}. 

Accordingly, to complete the proof of \thref{guDQ7}, it remains to establish
 the Claim. For this, let \(\theta\in(0,1)\) be a small constant to be chosen later. We will prove the claim by applying \thref{YLj1r} to an appropriate auxiliary function. For \(x\in B_{\theta d/2} (a)\), we have that \(\vert x -e_1 \vert \leqslant \vert a -e_1 \vert + \theta d/2=\rho-(1-\theta/2)d \), so, using~\eqref{CEnkR}, \begin{align}
u(x) \leqslant \tau \bigg (1-\frac \theta2\bigg )^{-n-2} d^{-n-2} =  \bigg (1-\frac\theta 2\bigg )^{-n-2} u(a) \qquad \text{for all } x\in B_{\theta d/2}(a). \label{sS0kO}
\end{align}

Let \(\zeta \) be a smooth, antisymmetric function such that~$\zeta \equiv 1$ in~$\{x_1 > 1/2 \}$
and~$0\leqslant \zeta \leqslant 1$ in~$\R^n_+$, 
and consider the antisymmetric function \begin{align*}
v(x) &:=   \left(1-\frac\theta 2\right)^{-n-2} u(a) \zeta(x) -u (x) \qquad \text{for all } x\in \R^n. 
\end{align*} Since \(\zeta \equiv 1\) in \(\{x_1 > 1/2 \} \supset B_{\theta d/2}(a)\) and \(0\leqslant \zeta \leqslant 1\) in \(\R^n_+\), it follows easily from~\eqref{OopaQ0tu} that \((-\Delta)^s \zeta \geqslant 0\) in \(B_{\theta d/2}(a)\). Hence, in \(B_{\theta d/2}(a)\), \begin{align*}
(-\Delta)^s v +cv &\geqslant -(-\Delta)^s u-cu+  c\bigg (1-\frac\theta 2\bigg )^{-n-2} u(a)  \\
&\geqslant -1 -C\| c^-\|_{L^\infty(B_\rho(e_1))}\bigg (1-\frac\theta 2\bigg )^{-n-2}  u(a) . 
\end{align*} Taking \(\theta\) sufficiently small, we obtain \begin{align}
(-\Delta)^s v +cv &\geqslant -C (1 +  u(a) ) \qquad \text{in } B_{\theta d/2}(a). \label{7fUyX}
\end{align}

The function \(v\) is almost the auxiliary function to which we would like to apply \thref{YLj1r}; however, \thref{YLj1r} requires \(v \geqslant 0\) in \(\R^n_+\) but we only have \(v \geqslant 0\) in \(B_{\theta d/2}(a)\) due to \eqref{sS0kO}. To resolve this issue let us instead consider the function \(w\) such that \(w (x) = v^+(x)\) for all \(x\in \R^n_+\) and~\(w(x) = -w(x_\ast)\) for all \(x\in \overline{\R^n_-}\).
We point out that~$w$ coincides with~$v$ in~$B_{\theta d/2}(a)$, thanks to~\eqref{sS0kO},
and therefore it is as regular as~$v$ in~$B_{\theta d/2}(a)$, which allows us to write the fractional Laplacian of~$w$ in~$B_{\theta d/2}(a)$ in a pointwise sense.

Also we observe that we have~\(w\geqslant0\) in \(\R^n_+\) but we no longer have a lower bound for \((-\Delta)^sw+cw\). To obtain this, observe that for all \(x\in \R^n_+\), \begin{align*}
(w-v)(x) 
&= \begin{cases}
0, &\text{for all } x\in \{v>0\} \cap \R^n_+ \\
u(x)- (1-\theta /2 )^{-n-2} u(a) \zeta(x), &\text{for all } x\in \{v \leqslant 0\}\cap \R^n_+. 
\end{cases}
\end{align*} In particular, \(w-v\leqslant |u|\) in \(\R^n_+\). It follows that for all \(x\in B_{\theta d/2}(a)\),
\begin{align*}
(-\Delta)^s (w-v)(x) &\geqslant -C \int_{\R^n_+ \setminus B_{\theta d/2}(a) } \left( \frac 1 {\vert x-y\vert^{n+2s}} - \frac 1 {\vert x_\ast - y \vert^{n+2s}} \right)| u(y)|\dd y.
\end{align*} Moreover, by \thref{ltKO2}, for all~$x\in B_{\theta d/4}(a)$,\begin{align}
(-\Delta)^s (w-v)(x) &\geqslant -C (\theta d)^{-n-2s-2} \Anorm{u}
\geqslant -C (\theta d)^{-n-2s-2}  . \label{Fy8oddTC} 
\end{align}Thus, by~\eqref{7fUyX} and~\eqref{Fy8oddTC}, for all \(x\in  B_{\theta d/4}(a)\), we have that \begin{align}
(-\Delta)^sw(x)+c(x)w(x) &= (-\Delta )^s v (x)+c(x) v(x) + (-\Delta)^s (w-v)(x)
\nonumber \\ 
&\geqslant -C \big (1 +  u(a) +(\theta d)^{-n-2s-2} \big )  \nonumber \\
&\geqslant  -C \big ( (\theta d)^{-n-2s-2}+u(a)  \big ) \label{c1kcbe0C}
\end{align} using that \(\theta d<1\).

Next let us consider the rescaled and translated functions \(\tilde w(x) := w(a_1 x+(0,a'))\) and \(\tilde c (x) := a_1^{2s} c(a_1 x+(0,a'))\) (recall that \(a' = (a_2,\dots, a_n) \in \R^{n-1}\)). By~\eqref{c1kcbe0C} we have that
\begin{align*}
(-\Delta)^s \tilde w+\tilde c\,  \tilde w  \geqslant -Ca_1^{2s} \big ((\theta d)^{-n-2s-2} +  u(a) \big )  \qquad \text{in } B_{\theta d /(4a_1)}(e_1 ).
\end{align*}  On one hand, by \thref{YLj1r}, we obtain \begin{align*}
\left(\frac{\theta d } 8 \right)^{-n} \int_{B_{\theta d /8}(a )}  w( x) \dd x 
&= \left(\frac{\theta d }{8a_1} \right)^{-n} \int_{B_{\theta d/(8a_1)}(e_1 )} \tilde w ( x) \dd x \\
&\leqslant C \Big(  \tilde w(e_1) +  (\theta d)^{-n-2} +u(a) (\theta d )^{2s} \Big)\\
&= C  \left( \left(1-\frac\theta 2\right)^{-n-2} -1 \right)u(a)  +C (\theta d)^{-n-2} +u(a) (\theta d )^{2s} .
\end{align*}We note explicitly that, by~\eqref{CyJJQHrF}, the constant in the above line is given by \begin{align*}
C= C'\big ( 1 +(\theta d)^{2s} \| c\|_{L^\infty(B_{\rho}(e_1))} \big ),
\end{align*} so using that \(\theta d<1\) it may be chosen to depend only on \(n\), \(s\), and \( \| c\|_{L^\infty(B_{\rho}(e_1))}\). On the other hand, \begin{align*}
B_{\theta d /8} (a)  \setminus U  &\subseteq\left\{ w \geqslant \left( \left( 1- \frac \theta  2\right)^{-n-2} -\frac12 \right) u(a) \right\} \cap B_{\theta d /8} (a)  ,
\end{align*}so we have that \begin{align*}
(\theta d )^{-n} \int_{B_{\theta d/8}(a)} w(x) \dd x  &\geqslant (\theta d)^{-n} \bigg ( \bigg ( 1- \frac \theta  2\bigg )^{-n-2} -\frac12 \bigg ) u(a) \cdot  \big \vert B_{\theta d /8} (a)  \setminus U\big   \vert \\
&\geqslant  C (\theta d)^{-n} u(a) \cdot  \big \vert B_{\theta d /8} (a)  \setminus U\big   \vert 
\end{align*} for \(\theta\) sufficiently small. Thus, \begin{align*}
\big \vert B_{\theta d /8} (a)  \setminus U \big   \vert
&\leqslant C  (\theta d)^n \bigg ( \bigg (1-\frac\theta 2\bigg )^{-n-2} -1 \bigg )  +C(u(a))^{-1} (\theta d)^{-2} + (\theta d )^{n+2s} \\
&\leqslant C  (\theta d)^n \bigg ( \bigg (1-\frac\theta 2\bigg )^{-n-2} -1 +\theta^{2s}\bigg )  +C \frac{\theta^{-2} d^n }{\tau} 
\end{align*} using~\eqref{CEnkR} and that \(d^{n+2s}<d^n\) since \(d<1\). At this point we may choose \(\theta\) sufficiently small such that \begin{align*}
 (\theta d)^n \bigg ( \bigg (1-\frac\theta 2\bigg )^{-n-2} -1 +\theta^{2s}\bigg )  \leqslant \frac 1 4 \vert B_{\theta d/8} \vert .
\end{align*} This proves the claim, and thus completes the proof of
Proposition~\ref{guDQ7}.
\end{proof}

\section{Boundary Harnack inequality and proof of \thref{C35ZH}} \label{TZei6Wd4}

In this section, we give the proof of \thref{C35ZH}. Analogous to the proof of \thref{DYcYH}, the proof
of \thref{C35ZH} is divided into the boundary Harnack inequality for super-solutions (\thref{SwDzJu9i}) and 
the boundary local boundedness for sub-solutions (\thref{EP5Elxbz}). Together these two results immediately 
give \thref{C35ZH}. 

\subsection{The boundary weak Harnack inequality}

Our next result is the antisymmetric boundary weak Harnack inequality. 

\begin{prop} \thlabel{SwDzJu9i} 
Let \(M\in \R\), \(\rho>0\), and \(c\in L^\infty(B_\rho^+)\). Suppose that \(u\in C^{2s+\alpha}(B_\rho)\cap \mathscr{A}_s(\R^n)\) for some \(\alpha > 0\) with \(2s+\alpha\) not an integer, \(u\) is non-negative in \(\R^n_+\) and satisfies \begin{align*}
(-\Delta)^s u +cu \geqslant -Mx_1 \qquad \text{in } B_\rho^+.
\end{align*}

Then there exists \(C_\rho>0\) depending only on \(n\), \(s\), \(\| c \|_{L^\infty(B_\rho^+)}\), and \(\rho\) such that \begin{align*}
\Anorm{u}  &\leqslant  C_\rho \left(  \inf_{B_{\rho/2}^+} \frac{u(x)}{x_1} +M \right).
\end{align*}  
\end{prop} 

As with the interior counter-part of \thref{SwDzJu9i}, that is
\thref{MT9uf}, we will prove the following rescaled version of \thref{SwDzJu9i},
namely \thref{g9foAd2c}. This version is essential to the proof of the boundary local boundedness for sub-solutions. Once \thref{g9foAd2c} has been proven, \thref{SwDzJu9i} follows easily with some minor adjustments. 

\begin{prop} \thlabel{g9foAd2c} 
Let \(M\in \R\), \(\rho >0\), and \(c\in L^\infty(B_\rho^+)\). Suppose that \(u\in C^{2s+\alpha}(B_\rho)\cap \mathscr{A}_s(\R^n)\) for some \(\alpha >0\) with~\(2s+\alpha\) not an integer, \(u\) is non-negative in \(\R^n_+\) and satisfies \begin{align*}
(-\Delta)^s u +cu \geqslant -Mx_1\qquad \text{in } B_\rho^+.
\end{align*}

Then there exists \(C>0\) depending only on \(n\), \(s\), and \(\rho^{2s}\| c \|_{L^\infty(B_\rho^+)}\) such that \begin{align*}
\frac 1 {\rho^{n+2}} \int_{B_{\rho/2}^+} y_1 u(y) \dd y &\leqslant C \left(  \inf_{B_{\rho/2}^+} \frac{u(x)}{x_1} + M \rho^{2s}\right). 
\end{align*}  Moreover, the constant \(C\) is of the form \begin{align*}
C=C' \big(1+\rho^{2s} \| c \|_{L^\infty(B_\rho^+)} \big) 
\end{align*} with \(C'\) depending only on \(n\) and \(s\). 
\end{prop}

Before we prove \thref{g9foAd2c}, we require some lemmata. 

\begin{lem} \thlabel{tZVUcYJl}
Let \(M\geqslant0\), \(k \geqslant 0\), and suppose that \(u\in C^{2s+\alpha}(B_1)\cap \mathscr{A}_s(\R^n)\) for some~\(\alpha>0\) with~\(2s+\alpha\) not an integer. \begin{enumerate}[(i)]
\item If \(u\) satisfies \begin{align*}
(-\Delta)^s u +k u \geqslant -Mx_1\qquad \text{in } B_1^+
\end{align*} then for all \(\varepsilon>0\) sufficiently small there exists \(u_\varepsilon \in C^\infty_0(\R^n)\)
antisymmetric and such that \begin{equation}
(-\Delta)^s u_\varepsilon +k u_\varepsilon \geqslant -(M+\varepsilon) x_1\qquad \text{in } B_{7/8}^+. \label{EuoJ3En6}
\end{equation}
\item If \(u\) satisfies \begin{align*}
(-\Delta)^s u +k u \leqslant Mx_1\qquad \text{in } B_1^+
\end{align*} then for all \(\varepsilon>0\) sufficiently small there exists \(u_\varepsilon \in C^\infty_0(\R^n)\)
antisymmetric and such that \begin{align*}
(-\Delta)^s u_\varepsilon +k u_\varepsilon \leqslant (M+\varepsilon) x_1\qquad \text{in } B_{7/8}^+. 
\end{align*}
\end{enumerate} In both cases the sequence \(\{ u_\varepsilon\} \) converges to \(u\) uniformly in \(B_{7/8}\).

Additionally, if \(u \) is non-negative in \(\R^n_+\) then \(u_\varepsilon\) is also non-negative in \(\R^n_+\).
\end{lem}

For the usual fractional Laplacian, \thref{tZVUcYJl} follows immediately by taking a mollification of \(u\) and in principle, this is also the idea here. However, there are a couple of technicalities that need to be addressed. The first is that here the fractional Laplacian is defined according to \thref{mkG4iRYH} and it remains to be verified that this fractional Laplacian commutes with the convolution operation as the usual one does. As a matter of fact, \thref{mkG4iRYH} does not lend itself well to the Fourier transform which makes it difficult to prove such a property. We overcome this issue by first multiplying \(u\) by an appropriate cut-off function which allows us to reduce to the case \((-\Delta)^s\) as given by the usual definition.  

The second issue is that directly using the properties of mollifiers, we can only expect to
control~\(u_\varepsilon\) in some~\(U \Subset B_1^+\) and not up to~\(\{x_1=0\}\). We can relax this thanks to the antisymmetry of \(u\). 

\begin{proof}[Proof of Lemma~\ref{tZVUcYJl}]
Fix \(\varepsilon>0\). Let \(R>1\) and let~\(\zeta\) be a smooth radial cut-off function such that $$
\zeta \equiv 1 \text{ in } B_R, \quad \zeta \equiv 0 \text{ in } \R^n \setminus B_{2R},
\quad{\mbox{and}}\quad 0\leqslant \zeta \leqslant 1 .
$$
Let also \(\bar  u := u \zeta\).

Now let us define a function \(f:B_1\to\R\) as follows: let \(f(x)=(-\Delta)^s u(x) + k u(x)\) for all \(x\in B_1^+\), \(f(x) = 0\) for all \(x\in B_1\cap \{x_1=0\}\), and \(f(x)=-f(x_\ast)\). We also define \(\bar f:B_1\to \R\) analogously with \(u\) replaced with \(\bar u\). By definition, both \(f\) and \(\bar f\) are antisymmetric\footnote{Note that the definition of antisymmetric requires the domain of \(f\) and \(\bar f\) to be \(\R^n\). For simplicity, we will still refer to \(f\) and \(\bar f\) as antisymmetric in this context since this technicality does not affect the proof.}, but note carefully that, \emph{a priori}, there is no reason to expect any regularity of \(f\) and \(\bar f\) across \(\{x_1=0\}\) (we will in fact find that \(\bar f \in C^\alpha(B_1)\)).

We claim that for \(R\) large enough (depending on \(\varepsilon\)), \begin{align}
\vert \bar f(x) - f(x) \vert \leqslant \varepsilon x_1\qquad \text{for all } x\in B_1^+. \label{8LncZIti}
\end{align} Indeed, if \(x\in B_1^+\) then \begin{align*}
\bar f(x) - f(x) =(-\Delta)^s(\bar u - u ) (x) &= C \int_{\R^n_+\setminus B_R} \bigg ( \frac 1 {\vert x - y \vert^{n+2s}} - \frac 1 {\vert x_\ast- y \vert^{n+2s}} \bigg ) ( u -\bar u ) (y) \dd y .
\end{align*} {F}rom~\eqref{LxZU6}, it follows that \begin{align*}
\vert \bar f(x) - f(x) \vert  \leqslant  Cx_1 \int_{\R^n_+\setminus B_R}  \frac{y_1\vert u(y) -\bar  u(y)\vert}{\vert x - y \vert^{n+2s+2}} \dd y \leqslant Cx_1 \int_{\R^n_+\setminus B_R}  \frac{y_1\vert u(y) \vert }{1+ \vert y \vert^{n+2s+2}} \dd y .
\end{align*} Since \(u\in \mathscr A_s(\R^n)\), taking \(R\) large we obtain~\eqref{8LncZIti}.

Next, consider the standard mollifier \(\eta(x) := C_0 \chi_{B_1}(x) e^{-\frac 1 {1-\vert x\vert ^2}}\) with \(C_0>0\) such that \(\int_{\R^n} \eta(x) \dd x =1\) and let~\(\eta_\varepsilon(x) := \varepsilon^{-n} \eta (x/\varepsilon)\). Also, let \(u_\varepsilon := \bar u \ast \eta_\varepsilon\) and \(f_\varepsilon := \bar f \ast \eta_\varepsilon\).

Notice that~$u_\epsilon\in C^\infty_0(\R^n)$ and it is antisymmetric. Additionally, we show that~\eqref{EuoJ3En6}
holds true in case~$(i)$ (case~$(ii)$ being analogous).

To this end, we observe that, since \(\bar u\) has compact support, we have that~\(\bar u \in \mathscr L_s(\R^n)\), so by \thref{eeBLRjcZ}, \((-\Delta)^s \bar u\) can be understood in the usual sense in \(B_1\), that is, by~\eqref{ZdAlT}. Moreover,
by~\cite[Propositions~2.4-2.6]{MR2270163}, we have that~\((-\Delta)^s\bar u\in C^\alpha(B_1)\) which gives that \(\bar f \in C^\alpha(B_1)\) and \begin{align*}
(-\Delta)^s\bar u +k \bar  u &= \bar f \qquad \text{in }B_1.
\end{align*} In particular, we may use standard properties of mollifiers to immediately obtain \begin{align*}
(-\Delta)^s u_\varepsilon +k u_\varepsilon &= f_\varepsilon \qquad \text{in }B_{7/8}.
\end{align*} 

Also, since \(\bar f\) is antisymmetric, it follows that \begin{align*}
f_\varepsilon(x) &= \int_{\R^n }  \bar f (y) \eta_\varepsilon (x-y) \dd y = \int_{\R^n_+ } \bar f (y)\big (  \eta_\varepsilon (x-y) - \eta_\varepsilon (x_\ast-y) \big )\dd y.
\end{align*} Observe that, since \(\eta \) is monotone decreasing in the radial direction and \(\vert x- y \vert \leqslant \vert x_\ast - y \vert\) for all~\(x,y\in \R^n_+\),  \begin{align}
\eta_\varepsilon (x-y) - \eta_\varepsilon (x_\ast-y) \geqslant 0 \qquad \text{for all } x,y\in \R^n_+. \label{DPmhap5t}
\end{align} Moreover, by~\eqref{8LncZIti}, we see that~\(\bar f (x) \geqslant -(M+\varepsilon)x_1\) for all \(x\in B_1^+\), so if \(x\in B_{7/8}^+\) and \(\varepsilon>0\) is sufficiently small (independent of \(x\)) then it follows that \begin{align}\label{fjrehgeruig009887}
f_\varepsilon(x) &=\int_{B_\varepsilon^+(x)} \bar f (y)\big (  \eta_\varepsilon (x-y) - \eta_\varepsilon (x_\ast-y) \big )\dd y \nonumber\\&\geqslant -(M+\varepsilon)\int_{B_\varepsilon^+(x)}  y_1 \big (  \eta_\varepsilon (x-y) - \eta_\varepsilon (x_\ast-y) \big )\dd y.
\end{align}

Next, we claim that \begin{align}
\int_{B_\varepsilon^+(x)}  y_1 \big (  \eta_\varepsilon (x-y) - \eta_\varepsilon (x_\ast-y) \big )\dd y &\leqslant  x_1. \label{Hp0lBCzB}
\end{align} Indeed, \begin{align*}
\int_{B_\varepsilon^+(x)}  y_1 \big (  \eta_\varepsilon (x-y) - \eta_\varepsilon (x_\ast-y) \big )\dd y &= \int_{B_\varepsilon(x)}  y_1   \eta_\varepsilon (x-y) \dd y \\&-\int_{B_\varepsilon^-(x)}  y_1    \eta_\varepsilon (x-y) \dd y - \int_{B_\varepsilon^+(x)}  y_1 \eta_\varepsilon (x_\ast-y) \dd y \\
&= \int_{B_\varepsilon(x)}  y_1   \eta_\varepsilon (x-y) \dd y  \\
&\qquad - \int_{B_\varepsilon^+(x)\setminus B_\varepsilon^+(x_\ast)}  y_1   \eta_\varepsilon (x_\ast-y) \dd y \\
&\leqslant \int_{B_\varepsilon(x)}  y_1   \eta_\varepsilon (x-y) \dd y .
\end{align*} Moreover, using that \(z \mapsto z_1 \eta(z)\) is antisymmetric and \(\int_{B_\varepsilon} \eta(z)\dd z =1\), we obtain that \begin{align*}
 \int_{B_\varepsilon(x)}  y_1   \eta_\varepsilon (x-y) \dd y &=  \int_{B_\varepsilon(x)}  (y_1-x_1)   \eta_\varepsilon (x-y) \dd y +x_1  \int_{B_\varepsilon(x)}    \eta_\varepsilon (x-y) \dd y =x_1
\end{align*}  which gives \eqref{Hp0lBCzB}. 

{F}rom~\eqref{fjrehgeruig009887} and~\eqref{Hp0lBCzB}, we obtain that
$$ f_\epsilon(x) \ge -(M+\epsilon)x_1$$
for all \(x\in B_{7/8}^+\), as soon as~$\epsilon$ is taken sufficiently small.
This is the desired
result in~\eqref{EuoJ3En6}.

Finally, it follows immediately from the properties of mollifiers that \( u_\varepsilon \to u \) uniformly in \(B_{7/8}\) as~\(\varepsilon\to 0^+\). Moreover, if \(u \geqslant 0\) in \(\R^n_+\) then from antisymmetry, \begin{align*}
u_\varepsilon (x) = \int_{\R^n_+ } \bar u(y)\big (  \eta_\varepsilon (x-y) - \eta_\varepsilon (x_\ast-y) \big )\dd y \geqslant 0 \qquad \text{for all } x\in \R^n_+
\end{align*} using~\eqref{DPmhap5t}. 
\end{proof}

Our second lemma is as follows. 

\begin{lem} \thlabel{6fD34E6w}
Suppose that \( v\in C^\infty_0(\R^n)\) is an antisymmetric function satifying \(\partial_1 v (0)=0\). Then \begin{align*}
\lim_{h\to 0 }\frac{(-\Delta)^sv(he_1)} h  = -2c_{n,s}(n+2s) \int_{\R^n_+} \frac{y_1 v(y)}{\vert y \vert^{n+2s+2}} \dd y .
\end{align*}
\end{lem}

Note that since \(v\in C^\infty_0(\R^n)\), the fractional Laplacian is given by the usual definition as per \thref{eeBLRjcZ}. 

\begin{proof} [Proof of Lemma~\ref{6fD34E6w}]
We will begin by proving that \begin{align}
\lim_{h\to 0 }\frac{(-\Delta)^sv(he_1)} h = (-\Delta)^s \partial_1 v(0). \label{JMT019eU}
\end{align} For this, consider the difference quotient \begin{align*}
\partial^h_1v(x) := \frac{v(x+he_1)-v(x)}h 
\end{align*} for all \(x\in \R^n\) and~\(\vert h \vert>0\) small. Since \(v\) is antisymmetric, \(v(0)=0\), so \begin{align*}
\frac{2v(h e_1)-v(he_1+y)-v(he_1-y)}h&= 2\partial^h_1v(0)-\partial^h_1v(y)-\partial^h_1v(-y) - \frac{v(y)+v(-y)}h
\end{align*} for all \(y\in \R^n\). Moreover, the function~\(y \mapsto v(y)+v(-y)\) is odd with
respect to~$y'$, and so \begin{align*}
\int_{\R^n} \frac{v(y)+v(-y)}{\vert y \vert^{n+2s}}\dd y&=0.
\end{align*} It follows that \begin{align*}
\frac{(-\Delta)^sv(he_1)} h &= \frac{c_{n,s}} 2 \int_{\R^n } \bigg ( 2\partial^h_1v(0)-\partial^h_1v(y)-\partial^h_1v(-y) - \frac{v(y)+v(-y)}h \bigg ) \frac{\dd y }{\vert y \vert^{n+2s}}\\
&= (-\Delta)^s \partial^h_1v(0) .
\end{align*} {F}rom these considerations and the computation at the top of p. 9 in \cite{MR3469920}, we have that \begin{align*}
\bigg \vert \frac{(-\Delta)^sv(he_1)} h  -(-\Delta)^s \partial_1 v(0) \bigg \vert &= \big \vert (-\Delta)^s ( \partial^h_1v-\partial_1v)(0) \big \vert \\
&\leqslant C \Big( \|\partial^h_1v-\partial_1v\|_{L^\infty(\R^n)} + \|D^2\partial^h_1v-D^2\partial_1v\|_{L^\infty(\R^n)} \Big) .
\end{align*} Then we obtain~\eqref{JMT019eU} by sending \(h\to 0\), using that \(\partial^h_1v \to \partial_1v\) in \(C^\infty_{\mathrm{loc}}(\R^n)\) as \(h\to 0\).

To complete the proof, we use that \(\partial_1v(0)=0\) and integration by parts to obtain \begin{align*}
(-\Delta)^s \partial_1 v(0) &= -c_{n,s} \int_{\R^n} \frac{\partial_1v(y)}{\vert y \vert^{n+2s}} \dd y \\
&= c_{n,s} \int_{\R^n} v(y) \partial_1\vert y \vert^{-n-2s} \dd y \\
&= -c_{n,s} (n+2s) \int_{\R^n} \frac{y_1v(y)}{\vert y \vert^{n+2s+2}} \dd y \\
&= -2c_{n,s} (n+2s) \int_{\R^n_+} \frac{y_1v(y)}{\vert y \vert^{n+2s+2}} \dd y 
\end{align*} where the last equality follows from antisymmetry of \(v\). 
\end{proof}

We are now able to give the proof of \thref{g9foAd2c}.

\begin{proof}[Proof of \thref{g9foAd2c}]  Since \(u\) is non-negative in \(\R^n_+\),
we have that~\( (-\Delta)^s u + \| c\|_{L^\infty(B_\rho^+)} u \geqslant  -Mx_1\) in \(B_\rho^+\).
Define~\(\tilde u(x):= u(\rho x)\) and note that \begin{align*}
(-\Delta)^s \tilde u +\rho^{2s} \| c\|_{L^\infty(B_\rho^+)} \tilde  u \geqslant -M\rho^{2s+1} x_1 \qquad  \text{in } B_1^+.
\end{align*} By way of \thref{tZVUcYJl} (i), we may take a \(C^\infty_0(\R^n)\) sequence of functions approximating \(\tilde u\) which satisfy the assumptions of \thref{g9foAd2c} with \(M\) replaced with \(M+\varepsilon\), obtain the estimate, then pass to the limit. In this way we may assume \(\tilde u \in C^\infty_0(\R^n)\). 

Let \(\zeta\) be a smooth radially symmetric cut-off function such that\begin{align*}
\zeta \equiv 1 \text{ in } B_{1/2}, \quad \zeta \equiv 0 \text{ in } \R^n \setminus B_{3/4},\quad {\mbox{and}}
\quad 0\leqslant \zeta \leqslant 1,
\end{align*} and define \(\varphi^{(2)} \in C^\infty_0(\R^n)\) by \(\varphi^{(2)}(x):= x_1 \zeta (x)\) for all \(x\in \R^n\). Suppose that \(\tau \geqslant 0\) is the largest possible value such that \(\tilde u  \geqslant \tau \varphi^{(2)}\) in \(\R^n_+\). For more detail on the existance of such a \(\tau\), see the footnote at the bottom of p. 11. Since \(\varphi^{(2)}(x)=x_1\) in \(B_{1/2}\), we have that \(x_1 \tau \leqslant \tilde u(x)\) for all~\(x\in B_{1/2}\), so \begin{align}
\tau \leqslant \inf_{B_{1/2}^+} \frac{\tilde u(x)}{x_1} =  \rho \inf_{B_{\rho/2}^+} \frac{u(x)}{x_1}. \label{pUJA2JZI}
\end{align} Since \(\tilde u\) is \(C^1\) in \(B_1\), there are two possibilities that can occur: either
there exists \(a\in B_{3/4}^+\) such that~\( \tilde u(a) = \tau \varphi^{(2)} (a)\); or there exists \(a \in B_{3/4} \cap \{ x_1=0\}\) such that \(\partial_1 \tilde u(a) = \tau \partial_1 \varphi^{(2)} (a)\). 

First suppose that there exists \(a\in B_{3/4}^+\) such that \( \tilde u(a) = \tau \varphi^{(2)} (a)\). Since \(\varphi^{(2)}\in C^\infty_0(\R^n)\) and is antisymmetric, \((-\Delta)^s\varphi^{(2)}\) is antisymmetric and~$\partial_1 \varphi^{(2)}=0$ in~$\{x_1=0\}$, we can
exploit Lemma~\ref{6fD34E6w} to say that~\((-\Delta)^s\varphi^{(2)}(x)/x_1\) is bounded in \(\R^n\).

On one hand, using that \((\tilde u-\tau \varphi^{(2)})(a)=0\), we have that \begin{align}
(-\Delta)^s(\tilde u-\tau \varphi^{(2)})(a)  &= (-\Delta)^s(\tilde u-\tau \varphi^{(2)})(a) +\rho^{2s} \| c\|_{L^\infty(B_\rho^+)}
(\tilde u-\tau \varphi^{(2)})(a) \nonumber \\
&\geqslant -M\rho^{2s+1}a_1 -\tau \big  (C + \rho^{2s}\| c \|_{L^\infty(B_\rho^+)} \big  ) a_1 \nonumber \\
&\geqslant  -M\rho^{2s+1}a_1 -C \tau  \big (1 + \rho^{2s} \| c \|_{L^\infty(B_\rho^+)} \big ) a_1 . \label{doW9AF3Y}
\end{align} On the other hand, since \(\tilde u-\tau \varphi^{(2)}\) is antisymmetric, non-negative in \(\R^n_+\), and \((\tilde u-\tau \varphi^{(2)})(a) = 0\), we have by \thref{mkG4iRYH} that  \begin{align*}
(-\Delta)^s(\tilde u-\tau \varphi^{(2)})(a) = -C  \int_{\R^n_+} \bigg (\frac 1 {\vert a - y \vert^{n+2s}} - \frac 1 {\vert a_\ast- y \vert^{n+2s}}\bigg )(\tilde u-\tau \varphi^{(2)})(y) \dd y .
\end{align*} It follows from~\eqref{buKHzlE6} that \begin{align}
(-\Delta)^s(\tilde u-\tau \varphi^{(2)})(a) &\leqslant -C a_1 \int_{B_{1/2}^+}  \frac{y_1 (\tilde u-\tau \varphi^{(2)})(y)} {\vert a_\ast- y \vert^{n+2s+2}}\dd y \nonumber \\
&\leqslant -C a_1 \bigg ( \int_{B_{1/2}^+} y_1 \tilde u(y) \dd y -\tau \bigg ) \nonumber  \\
&=-C a_1 \bigg ( \frac 1 {\rho^{n+1}} \int_{B_{\rho/2}^+} y_1 u(y) \dd y -\tau \bigg ). \label{retlxLlR}
\end{align} Rearranging~\eqref{doW9AF3Y} and~\eqref{retlxLlR}, and recalling~\eqref{pUJA2JZI} gives \begin{align}
\frac 1 {\rho^{n+1}} \int_{B_{\rho/2}^+} y_1 u(y) \dd y &\leqslant C \Big(
\tau (1 + \rho^{2s} \| c \|_{L^\infty(B_\rho^+)} )  + \rho^{2s+1}M \Big) \nonumber  \\
&\leqslant C \rho (1 + \rho^{2s} \| c \|_{L^\infty(B_\rho^+)} )  \left(  \inf_{B_{\rho/2}^+} \frac{u(x)}{x_1} + M \rho^{2s}\right), \label{7O4TL1vF}
\end{align}
which gives the desired result in this case.

Now suppose that there exists \(a \in B_{3/4} \cap \{ x_1=0\}\) such that \(\partial_1 \tilde u(a) =\tau \partial_1 \varphi^{(2)} (a)\). Let \(h>0\) be small and set \(a^{(h)}:= a+ he_1\). On one hand, as in~\eqref{doW9AF3Y},  \begin{align*}
(-\Delta)^s(\tilde u-\tau \varphi^{(2)})(a^{(h)}) &\geqslant -M\rho^{2s+1}h -C \tau \big ( 1+ \rho^{2s}\|c\|_{L^\infty(B_\rho^+)}  \big )h .
\end{align*} Dividing both sides by \(h\) and sending \(h\to 0^+\), it follows from \thref{6fD34E6w} (after a translation) that \begin{align*}
-M\rho^{2s+1} -C \tau \big ( 1+ \rho^{2s}\|c\|_{L^\infty(B_\rho^+)}  \big ) &\leqslant  -C \int_{\R^n_+} \frac{y_1 (\tilde u-\tau \varphi^{(2)})(y)}{\vert y-a\vert^{n+2s+2}} \dd y \\
&\leqslant-C \bigg (  \frac 1{\rho^{n+1}} \int_{B_{\rho/2}^+} y_1 u(y) \dd y - \tau \bigg ). 
\end{align*} Rearranging as before gives the desired result. 
\end{proof}

Next, we give the proof of~\thref{SwDzJu9i}. 

\begin{proof}[Proof of~\thref{SwDzJu9i}]
Let \(\varphi^{(2)}\), \(\tau\), and \(a\) be the same as in the proof of \thref{g9foAd2c}. The proof of \thref{SwDzJu9i} is identical to the proof of \thref{g9foAd2c} except for the following changes.

In the case \(a\in B_{3/4}^+\) and \(\tilde u(a) = \tau \varphi^{(2)}(a)\), we use~\eqref{buKHzlE6} to obtain \begin{align*}
(-\Delta)^s (\tilde u - \tau \varphi^{(2)})(a) \leqslant-C a_1  \int_{\R^n_+} \frac{y_1(\tilde u-\tau \varphi^{(2)})(y)}{\vert a_\ast - y \vert^{n+2s+2}} \dd y \leqslant - C_\rho a_1 \big ( \Anorm{u} -\tau \big )
\end{align*} where we have also used that \begin{align}
\vert a_\ast - y \vert^{n+2s+2} \leqslant C (1+\vert y \vert^{n+2s+2} \big ) \qquad \text{for all } y\in \R^n_+. \label{BWqeWX33}
\end{align} Moreover, in the case \(a\in B_{3/4}\cap \{x_1=0\}\) and \(\partial_1\tilde u(a) = \tau \partial_1\varphi^{(2)}(a)\), 
we have that~\(a=a_\ast\) so~\eqref{BWqeWX33} also gives \begin{align*}
\int_{\R^n_+} \frac{y_1 (\tilde u-\tau \varphi^{(2)})(y)}{\vert y-a\vert^{n+2s+2}} \dd y &\geqslant   C_\rho \big (\Anorm{u} -\tau \big ) .
\qedhere
\end{align*}
\end{proof}

\subsection{Boundary local boundedness}

We now prove the boundary local boundedness for sub-solutions. 

\begin{prop} \thlabel{EP5Elxbz}
Let \(M \geqslant 0\), \(\rho\in(0,1)\), and \(c\in L^\infty(B_\rho^+)\). Suppose that \(u \in C^{2s+\alpha}(B_\rho)\cap \mathscr A_s (\R^n)\) for some \(\alpha>0\) with \(2s+\alpha\) not an integer, and \(u\) satisfies  \begin{align*}
(-\Delta)^su +cu &\leqslant M x_1 \qquad \text{in } B_\rho^+.
\end{align*}

Then there exists \(C_\rho>0\) depending only on \(n\), \(s\), \(\| c \|_{L^\infty(B_\rho^+)}\), and \(\rho\) such that  \begin{align*}
\sup_{x\in B_{\rho/2}^+} \frac{ u(x)}{x_1} &\leqslant C_\rho ( \Anorm{u} +M  ) .
\end{align*} 
\end{prop}

Before we prove \thref{EP5Elxbz}, we prove the following lemma. 

\begin{lem} \thlabel{KaMmndyO}
Let \(\varphi \in C^\infty(\R)\) be an odd function such that \(\varphi(t)=1\) if \(t>2\) and \(0\leqslant \varphi(t) \leqslant 1\) for all \(t\geqslant0\). Suppose that \(\varphi^{(3)} \in C^s(\R^n) \cap L^\infty(\R^n)\) is the solution to \begin{align}
\begin{PDE}
(-\Delta)^s \varphi^{(3)} &= 0 &\text{in } B_1, \\
\varphi^{(3)}(x) &= \varphi(x_1)  &\text{in } \R^n \setminus B_1.
\end{PDE} \label{pa7rDaw7}
\end{align}

Then \(\varphi^{(3)}\) is antisymmetric and there exists \(C>1\) depending only on \(n\) and \(s\) such that \begin{align*}
C^{-1}x_1  \leqslant \varphi^{(3)}(x) \leqslant Cx_1
\end{align*} for all \(x\in B_{1/2}^+\).
\end{lem}

\begin{proof}
Via the Poisson kernel representation, see \cite[Section 15]{MR3916700}, and using that \(\varphi\) is an odd function, we may write \begin{align*}
\varphi^{(3)}(x)  &= C \int_{\R^n \setminus B_1} \bigg (  \frac{1-\vert x \vert^2}{\vert y \vert^2-1}\bigg )^s \frac{\varphi(y_1)}{\vert x -y \vert^n} \dd y  \\
&= C \int_{\R^n_+ \setminus B_1^+} \bigg (  \frac{1-\vert x \vert^2}{\vert y \vert^2-1}\bigg )^s \bigg ( \frac 1 {\vert x -y \vert^n} - \frac 1 {\vert x_\ast -y \vert^n}\bigg )  \varphi(y_1)\dd y.
\end{align*} {F}rom this formula, we immediately obtain that \(\varphi^{(3)}\) is antisymmetric (this can also be argued by the uniqueness of solutions to~\eqref{pa7rDaw7}). Then, by an analogous computation to~\eqref{LxZU6} (just replacing~$n+2s$
with~$n$), \begin{align*}
\varphi^{(3)}(x) &\leqslant C  x_1 \int_{\R^n_+ \setminus B_1^+}  \frac{ y_1\varphi(y_1)}{(\vert y \vert^2-1)^s\vert x-y \vert^{n+2}}  \dd y \\
&\leqslant  C  x_1 \left[\int_{B_2^+ \setminus B_1^+}  \frac{ y_1}{(\vert y \vert^2-1)^s }  \dd y
+\int_{\R^n_+ \setminus B_2^+}  \frac{ y_1}{(\vert y \vert^2-1)^s(\vert y \vert - 1)^{n+2}}  \dd y\right] \\
&\leqslant Cx_1
\end{align*} for all \(x\in B_{1/2}^+\). Similarly, using now~\eqref{buKHzlE6} (replacing~$n+2s$
with~$n$), we have that
\begin{align*}
\varphi^{(3)}(x) &\geqslant  C x_1 \int_{\R^n_+ \setminus B_1^+}   \frac{ y_1\varphi(y_1)}{(\vert y \vert^2-1)^s\vert x_\ast -y \vert^{n+2}} \dd y \\
&\geqslant C x_1 \int_{\{y_1>2\}}   \frac{ y_1}{(\vert y \vert^2-1)^s(\vert y \vert +1)^{n+2}} \dd y \\
&\geqslant C x_1
\end{align*} for all \(x\in B_{1/2}^+\). 
\end{proof}

Now we can give the proof of \thref{EP5Elxbz}. 

\begin{proof}[Proof of \thref{EP5Elxbz}]
Dividing through by \(\Anorm{u}+M\), we can also assume that \((-\Delta)^s u + cu \leqslant x_1\) in \(B_\rho^+\) and \(\Anorm{u}\leqslant 1\).  Moreover, as explained at the start of the proof of \thref{g9foAd2c}, via \thref{tZVUcYJl} (ii) (after rescaling), it is not restrictive to assume \(u \in C^\infty_0(\R^n)\) and it is antisymmetric. 

Furthermore, we point out that the claim in \thref{EP5Elxbz} is obviously true if~$u\le0$
in~$B^+_\rho$, hence we suppose that~$\{u>0\}\cap B^+_\rho\ne\varnothing$.

Let \(\varphi^{(3)}\) be as in \thref{KaMmndyO} and let \(\zeta(x) := \varphi^{(3)}(x/(2\rho)) \). Suppose that \(\tau \geqslant 0\) is the smallest value such that \begin{align*}
u(x) &\leqslant \tau \zeta (x) (\rho - \vert x \vert )^{-n-2} \qquad \text{in } B_\rho^+. 
\end{align*}
The existence of such a \(\tau\) follows from a similar argument to the one presented in the footnote at the bottom of p. 11. Notice that~$\tau>0$.
To complete the proof we will show that \(\tau \leqslant C_\rho\) with \(C_\rho\) independent of \(u\).  Since \(u\) is continuously differentiable, two possibilities can occur: \begin{itemize}
\item[Case 1:] There exists \(a \in B_\rho^+\) such that \begin{align*}
u(a) &= \tau \zeta (a) (\rho - \vert a \vert )^{-n-2}.
\end{align*} 
\item[Case 2:] There exists \(a\in B_\rho \cap \{x_1=0\}\) such that \begin{align*}
\partial_1 u(a) = \tau \partial_1 \big \vert_{x=a} \big ( \zeta (x) (\rho - \vert x \vert )^{-n-2} \big )= \tau (\partial_1 \zeta (a)) (\rho - \vert a \vert )^{-n-2}  .
\end{align*} 
\end{itemize}
Let \(d:= \rho - \vert a\vert \) and define \(U\subset B_\rho^+\) as follows: if Case 1 occurs let \begin{align*}
U := \bigg \{ x \in B_\rho^+ \text{ s.t. } \frac{u(x)}{\zeta(x)} > \frac{u(a)}{2\zeta (a)} \bigg \};
\end{align*} otherwise if Case 2 occurs then let \begin{align*}
U:=\bigg \{ x \in B_\rho^+ \text{ s.t. }  \frac{u(x)}{\zeta(x)} > \frac{\partial_1u(a)}{2 \partial_1\zeta (a)} \bigg \}. 
\end{align*} Since  \(u(a) = \tau \zeta(a) d^{-n-2}\) in Case 1 and \(\partial_1u(a) = \tau \partial_1\zeta(a) d^{-n-2}\) in Case 2, we may write \begin{equation}\label{ei395v7b865998cn754mx984zUUU}
U =  \bigg \{ x \in B_\rho^+ \text{ s.t. } u(x)> \frac 1  2 \tau d^{-n-2} \zeta(x) \bigg \}
\end{equation}which is valid in both cases.

Then, we have that, for all \(r\in(0,d)\),
\begin{align*}
C_\rho &\geqslant \int_{B_\rho^+} y_1| u(y)| \dd y \geqslant \frac 1  2 \tau d^{-n-2} \int_{U \cap B_r(a)} y_1 \zeta(y) \dd y .
\end{align*} 
As a consequence, by \thref{KaMmndyO}, we have that \begin{align}
\int_{U \cap B_r(a)} y_1^2 \dd y 
\le C_\rho\int_{U \cap B_r(a)} y_1 \zeta(y) \dd y 
\leqslant \frac{ C_\rho d^{n+2} } \tau  \qquad \text{for all } r\in(0,d). \label{TTWpDkie}
\end{align} Next, we make the following claim. 

\begin{claim}
There exists~$\theta_0\in(0,1)$ depending only on \(n\), \(s\), \(\| c\|_{L^\infty(B_\rho(e_1))}\), and \(\rho\) such that
if~$\theta\in(0,\theta_0]$ there exists~\(C>0\)
depending only on \(n\), \(s\), \(\| c\|_{L^\infty(B_\rho(e_1))}\), \(\rho\), and~$\theta$ such that
\begin{itemize}
\item In Case 1: 
\begin{enumerate}[(i))]
\item If \(a_1 \geqslant   \theta d/16 \) then \begin{align*}
\big  \vert B_{(\theta d)/64}(a) \setminus U  \big \vert   &\leqslant \frac 1 4 \big  \vert B_{(\theta d)/64}  \big \vert  +\frac{C d^n} \tau .
\end{align*}
\item If \(a_1 <   \theta d/16 \) then \begin{align*}
\int_{B_{(\theta d)/64}^+(a) \setminus U } x_1^2 \dd x &\leqslant  \frac 1 4 \int_{B_{(\theta d)/64}^+(a) } x_1^2 \dd x + \frac{C d^{n+2}}\tau .
\end{align*}
\end{enumerate} 
\item In Case 2:  \begin{align*}
\int_{B_{(\theta d)/64}^+(a)\setminus U } x_1^2\dd x &\leqslant\frac 1 4 \int_{B_{(\theta d)/64}^+ (a)} x_1^2 \dd x + \frac{ C d^{n+2}}\tau.
\end{align*}
\end{itemize} In particular, neither~\(\theta\) nor~\(C\) depend on \(\tau\), \(u\), or \(a\).
\end{claim} 

We withhold the proof of the claim until the end. Assuming that
the claim is true, we complete the proof of \thref{EP5Elxbz} as follows.
 
 If Case 1(i) occurs then for all \(y\in B_{(\theta_0 d)/64} (a)\)
 we have that~\(y_1>a_1-( \theta_0 d)/64 \geqslant C d\), and so \begin{align*}
\int_{U \cap B_{(\theta_0 d)/64}(a)} y_1^2 \dd y \geqslant C d^2 \cdot \big \vert U \cap B_{(\theta_0 d)/64}(a) \big \vert .
\end{align*} Hence, from~\eqref{TTWpDkie} (used here with~$r:=(\theta_0d)/64$), we have that \begin{align*}
\vert U \cap B_{( \theta_0 d)/64}(a) \big \vert \leqslant \frac{C_\rho d^n}{\tau}. 
\end{align*}
Then using the claim, we find that \begin{align*}
\frac{C_\rho d^n}{\tau} \geqslant \big  \vert B_{(\theta_0 d)/64} \big \vert  - \big  \vert B_{(\theta_0 d)/64}(a) \setminus U  \big \vert 
\geqslant  \frac 3 4 \big  \vert B_{(\theta_0 d)/64}  \big \vert  -\frac{C d^n} \tau 
\end{align*}which gives that \(\tau \leqslant C_\rho\) in this case.

If Case 1(ii) or Case 2 occurs then from~\eqref{TTWpDkie} (used here with~$r:=(\theta_0d)/64$)
and the claim, we have that \begin{equation}\label{dewiotbv5748976w4598ty}
\frac{C_\rho d^{n+2}}{\tau} \geqslant \int_{B_{(\theta_0 d)/64}^+(a) } x_1^2 \dd x - \int_{B_{( \theta_0 d)/64}^+(a) \setminus U } x_1^2 \dd x 
\geqslant \frac 3 4 \int_{B_{( \theta_0 d)/64}^+(a) } x_1^2 \dd x - \frac{C d^{n+2}}\tau .
\end{equation}
We now observe that, given~$r\in(0,d)$, if~$x\in B_{r/4}\left(a+\frac34 re_1\right)\subset B^+_{r}(a)$
then~$x_1\ge a_1+\frac34 r-\frac{r}4\ge\frac{r}2$,
and thus
\begin{equation}\label{sdwet68980poiuytrkjhgfdmnbvcx23456789}
\int_{B_{r}^+(a) } x_1^2 \dd x\ge
\int_{B_{r/4}(a+(3r)/4 e_1) } x_1^2 \dd x\ge \frac{r^2}4\,|B_{r/4}(a+(3r)/4 e_1) |=C r^{n+2},\end{equation}
for some~$C>0$ depending on~$n$. Exploiting this formula with~$r:=(\theta_0 d)/64$ into~\eqref{dewiotbv5748976w4598ty},
we obtain that
$$ \frac{C_\rho d^{n+2}}{\tau} \geqslant C \left(\frac{ \theta_0 d}{64}\right)^{n+2 }- \frac{C d^{n+2}}\tau,$$
which gives that \(\tau \leqslant C_\rho\) as required.

Hence, we now focus on the proof of the claim. 
Let \(\theta\in(0,1)\) be a small constant to be chosen later. By translating, we may also assume without loss of generality that \(a'=0\) (recall that~\(a'=(a_2,\dots,a_n)\in \R^{n-1}\)).

For each \(x\in B_{\theta d/2}^+(a)\), we have that~\(\vert x \vert \leqslant \vert a \vert +\theta d/2=\rho- (1-\theta/2)d\). Hence, in both Case 1 and Case 2, 
\begin{align}
u(x) \leqslant \tau d^{-n-2}   \bigg (1- \frac \theta 2 \bigg )^{-n-2}\zeta (x) \qquad \text{in }B_{\theta d/2}^+(a). \label{ln9LcJTh}
\end{align} Let \begin{align*}
v(x):=  \tau d^{-n-2}   \bigg (1- \frac \theta 2 \bigg )^{-n-2}\zeta (x)-u(x)  \qquad \text{for all } x\in \R^n. 
\end{align*} We have that \(v\) is antisymmetric and \(v\geqslant 0\) in \(B_{\theta d/2}^+(a) \) due to~\eqref{ln9LcJTh}. Moreover, since \(\zeta\) is \(s\)-harmonic in \(B_\rho^+ \supset B_{\theta d/2}^+(a)\), for all \(x\in B_{\theta d/2}^+(a)\), \begin{align*}
(-\Delta)^s v(x) +c(x)v(x) &= -(-\Delta)^s u(x) - c(x) u(x) +c(x) \tau d^{-n-2} \bigg (1- \frac \theta 2 \bigg )^{-n-2}\zeta (x)\\
&\geqslant -x_1 -C  \tau d^{-n-2} \|c^-\|_{L^\infty(B_\rho^+)} \bigg (1- \frac \theta 2 \bigg )^{-n-2}\zeta (x).
\end{align*} Taking \(\theta\) sufficiently small and using that \(\zeta (x)\leqslant C x_1\) (in light of~\thref{KaMmndyO}),
we obtain \begin{align}
(-\Delta)^s v(x) +c(x)v(x)&\geqslant -C \big ( 1 + \tau d^{-n-2}  \big )x_1  \qquad \text{in } B_{\theta d/2}^+(a). \label{ulzzcUwf}
\end{align}

Next, we define \(w(x):= v^+(x)\) for all \(x\in \R^n_+\) and \(w(x):= -w(x_\ast)\) for all \(x\in \overline{\R^n_-}\). 
We point out that, in light of~\eqref{ln9LcJTh}, $w$ is as regular as~$v$ in~$B_{\theta d/2}^+(a)$, and thus
we can compute the fractional Laplacian of~$w$ in~$B_{\theta d/2}^+(a)$ in a pointwise sense.

We also observe that \begin{align*}
(w-v)(x) (x) &=  \begin{cases}
0 &\text{if } x\in \R^n_+ \cap \{ v\geqslant 0\} ,\\
u(x)-\tau d^{-n-2} \big (1- \frac \theta 2 \big )^{-n-2}\zeta (x),&\text{if } x\in \R^n_+ \cap \{ v< 0\}.
\end{cases}
\end{align*} In particular, \(w-v\leqslant |u|\) in \(\R^n_+\).
Thus, for all \(x\in B_{\theta d/2}^+(a)\), \begin{align*}
(-\Delta)^s (w-v)(x)&\geqslant -C \int_{\R^n_+\setminus B_{\theta d/2}^+(a)} \left(\frac 1 {\vert x -y \vert^{n+2s}} - \frac 1 {\vert x_\ast -y \vert^{n+2s}} \right) |u(y)|\dd y .
\end{align*}  Moreover, by \thref{ltKO2}, for all~$x \in B_{\theta d/4}^+(a)$,
\begin{equation}
(-\Delta)^s (w-v)(x) \geqslant -C(\theta d)^{-n-2s-2} \Anorm{u}
x_1 \geqslant -C(\theta d)^{-n-2s-2}x_1 . \label{y8tE2pf9}
\end{equation} Hence, by~\eqref{ulzzcUwf} and~\eqref{y8tE2pf9}, we obtain  \begin{align}
(-\Delta)^s w +cw &=(-\Delta)^sv +cv +(-\Delta)^s(w-v) \nonumber \\
&\geqslant -C \bigg ( 1 + \tau d^{-n-2}  +(\theta d)^{-n-2s-2}  \bigg )x_1 \nonumber \\
&\geqslant -C \bigg (  (\theta d)^{-n-2s-2} + \tau d^{-n-2}  \bigg )x_1\label{ayFE7nQK}
\end{align} in \( B_{\theta d/4}^+(a)\).

Next, let us consider Case 1 and Case 2 separately.  

\emph{Case 1:} Suppose that \(a\in B_\rho^+\) and let \(\tilde w (x) = w (a_1 x)\) and \(\tilde c (x) = a_1^{2s} c (a_1 x)\). Then from~\eqref{ayFE7nQK}, we have that \begin{align}
(-\Delta)^s \tilde w(x) +\tilde c(x) \tilde w(x) &\geqslant - Ca_1^{2s+1} \bigg (  (\theta d)^{-n-2s-2} + \tau d^{-n-2}  \bigg )x_1 \label{zrghHZhX}
\end{align} for all \( x\in B_{\theta d/(4a_1)}^+(e_1)\). 

As in the proof of \thref{guDQ7}, we wish to apply the rescaled version of the weak Harnack inequality to \(\tilde w\); however, we cannot immediately apply either \thref{YLj1r} or \thref{g9foAd2c} to~\eqref{zrghHZhX}. To resolve this, let us split into a further two cases: (i) \(a_1 \geqslant   \theta d/16 \) and (ii) \(a_1<  \theta d/16\).

\emph{Case 1(i):} If \(a_1 \geqslant   \theta d/16 \) then \( B_{\theta d/(32a_1)}(e_1) \subset B_{\theta d/(4a_1)}^+(e_1)\) and for each \(x\in B_{\theta d/(32a_1)}(e_1)\) we have that~\(x_1<1+\theta d/(32a_1)\le1+1/4=5/4\). Therefore, from~\eqref{zrghHZhX}, we have that \begin{align*}
(-\Delta)^s \tilde w(x) +\tilde c(x) \tilde w(x) &\geqslant - Ca_1^{2s+1}  \big (  (\theta d)^{-n-2s-2} + \tau d^{-n-2}  \big ) 
\end{align*} for all \( x\in B_{\theta d/(32a_1)}(e_1)\).

On one hand, by~\thref{YLj1r} (used here with~$\rho:=\theta d/(32a_1)$), \begin{align*}&
\left( \frac{\theta d}{64} \right)^{-n} \int_{B_{\theta d/64}(a)}  w (x) \dd x \\= &\; \left( \frac{\theta d}{32a_1} \right)^{-n}  \int_{B_{\theta d/(64a_1)}(e_1)} \tilde w (x) \dd x \\
\leqslant&\; C \left(  \tilde w(e_1) +  a_1 (\theta d)^{-n-2} + \tau a_1 \theta^{2s} d^{-n+2s-2} \right) \\
\leqslant& \;C  \tau d^{-n-2} \bigg (  \bigg (1- \frac \theta 2 \bigg )^{-n-2}-1 \bigg ) a_1 +Ca_1(\theta d)^{-n-2} + C\tau a_1 \theta^{2s} d^{-n+2s-2}
\end{align*} using also \thref{KaMmndyO} and that \(u(a) = \tau d^{-n-2} \zeta(a)\). 

On the other hand, by the definition of~$U$ in~\eqref{ei395v7b865998cn754mx984zUUU},
\begin{align}
B_r(a) \setminus U  &\subset \left\{ \frac{w}{\zeta} > \tau d^{-n-2} \left(  \left(1 - \frac \theta 2 \right)^{-n-2} -\frac 1 2 \right)  \right\} \cap B_r(a), \qquad \text{ for all } r\in\left(0,\frac12\theta d\right),\label{I9FvOxCz}
\end{align} and so \begin{align*}
(\theta d)^{-n} \int_{B_{\theta d/64}(a)}  w(x) \dd x &\geqslant \tau \theta^{-n} d^{-2n-2} \bigg ( \bigg (1 - \frac \theta 2 \bigg )^{-n-2} -\frac 1 2 \bigg ) \int_{B_{\theta d/64}(a) \setminus U }  \zeta (x) \dd x \\
&\geqslant  C \tau \theta^{-n} d^{-2n-2} \int_{B_{\theta d/64}(a) \setminus U } \zeta (x) \dd x
\end{align*} for \(\theta\) sufficiently small. Moreover, using that \(x_1>a_1-\theta d/64>Ca_1\) and \thref{KaMmndyO}, we have that\begin{align*}
\int_{B_{\theta d/64}(a) \setminus U } \zeta (x) \dd x \geqslant C\int_{B_{\theta d/64}(a) \setminus U } x_1 \dd x \geqslant C
a_1\cdot \big  \vert B_{\theta d/64}(a) \setminus U  \big \vert . 
\end{align*}
Thus,\begin{align*}
\big  \vert B_{\theta d/64}(a) \setminus U  \big \vert   &\leqslant C  (\theta d)^n \bigg (  \bigg (1- \frac \theta 2 \bigg )^{-n-2}-1 \bigg )  +\frac{C\theta^{-2} d^n} \tau + C (\theta d)^{n+2s} \\
&\leqslant  C  (\theta d)^n \bigg (  \bigg (1- \frac \theta 2 \bigg )^{-n-2}-1 +\theta^{2s} \bigg )  +\frac{C\theta^{-2} d^n} \tau ,
\end{align*} using also that \(d^{n+2s}<d^n\). 

Finally, we can take \(\theta\) sufficiently small so that \begin{align*}
C (\theta d)^n \bigg (  \bigg (1- \frac \theta 2 \bigg )^{-n-2}-1 +\theta^{2s} \bigg ) \leqslant \frac 1 4 \big \vert B_{\theta d/64} \big \vert 
\end{align*}
which gives
$$ \big  \vert B_{\theta d/64}(a) \setminus U  \big \vert\le \frac14 \big \vert B_{\theta d/64} \big \vert
 +\frac{Cd^n} \tau.
$$
This concludes the proof of the claim in Case 1(i). 

\emph{Case 1(ii):} Let \(a_1<  \theta d/16\) and fix \(R:= \frac 1 {2} \big ( \sqrt{ ( \theta d/(4a_1))^2-1} +2 \big )\). Observe that \begin{align*}
2 < R< \sqrt{ \left(\frac{ \theta d}{4a_1}\right)^2-1}.
\end{align*} Hence, \(e_1 \in B_{R/2}^+\). Moreover, if \(x\in B_R^+\) then \begin{align*}
\vert x -e_1\vert^2 <1+R^2< ( \theta d/(4a_1))^2,
\end{align*} so \(B_R^+ \subset B_{\theta d/(4a_1)}^+(e_1)\). 

Thus, applying \thref{g9foAd2c} to the equation in~\eqref{zrghHZhX} in \(B_R^+\), we obtain \begin{align*}&
a_1^{-n-1}\bigg ( \frac R 2\bigg )^{-n-2} \int_{B_{a_1R/2}^+} x_1  w(x) \dd x \\
=&\;  \bigg ( \frac R 2\bigg )^{-n-2} \int_{B_{R/2}^+} x_1 \tilde w(x) \dd x \\
\leqslant&\; C \left( \inf_{B_{R/2}^+} \frac{\tilde w(x)}{x_1} +  a_1^{2s+1}R^{2s}  (\theta d)^{-n-2s-2} + \tau d^{-n-2}a_1^{2s+1}R^{2s}   \right) \\
\leqslant\;& C \tau d^{-n-2}   \left( \left(1- \frac \theta 2 \right)^{-n-2}-1 \right)\zeta(a)
+ C  a_1^{2s+1}R^{2s}  (\theta d)^{-n-2s-2} +C  \tau d^{-n-2}a_1^{2s+1}R^{2s}  
\end{align*} using that \(e_1 \in B_{R/2}^+\).

Since \(R \leqslant C \theta d/a_1\) and \(\zeta(a)\leqslant Ca_1\) by \thref{KaMmndyO}, it follows that
\begin{eqnarray*}&&
a_1^{-1} \bigg ( \frac R {2a_1} \bigg )^{-n} \int_{B_{a_1R/2}^+} x_1  w(x) \dd x \\
&&\qquad\leqslant C \tau d^{-n-2}   \bigg ( \bigg (1- \frac \theta 2 \bigg )^{-n-2}-1 \bigg )a_1  + C  a_1  (\theta d)^{-n-2} +C  \tau a_1 \theta^{2s} d^{-n+2s-2} .
\end{eqnarray*}

On the other hand, we claim that
\begin{equation}\label{skweogtry76t678tutoy4554yb76i78io896}
B_{(\theta d)/64}^+(a) \setminus U \subset B_{a_1R/2}^+ .\end{equation} Indeed,
if~$x\in B_{(\theta d)/64}^+ \setminus U $ then
$$ |x|\le |x-a|+|a|\le \frac{\theta d}{64} +a_1\le a_1\left(\frac{\theta d}{64a_1}+1\right).
$$
Furthermore,
\begin{eqnarray*}
&&R\ge   \frac 1 {2} \sqrt{ \left(\frac{\theta d}{4a_1}\right)^2-\left(\frac{\theta d}{16a_1}\right)^2} +1\ge
\frac{\sqrt{15}\theta d}{32a_1}+1\ge\frac{3\theta d}{32a_1}+1\\&&\qquad
=2\left(\frac{3\theta d}{64a_1}+\frac12\right)=2\left(\frac{\theta d}{64a_1}+\frac{\theta d}{32a_1}+\frac12\right)\ge
2\left(\frac{\theta d}{64a_1}+1\right).
\end{eqnarray*}
{F}rom these observations we obtain that if~$x\in B_{(\theta d)/64}^+ \setminus U $ then~$|x|\le a_1R/2$,
which proves~\eqref{skweogtry76t678tutoy4554yb76i78io896}.

Hence, by~\eqref{I9FvOxCz} (used with~$r:=\theta d/64$), \eqref{skweogtry76t678tutoy4554yb76i78io896}, and \thref{KaMmndyO}, we have that \begin{align*}
&\hspace{-2em}a_1^{-1} \bigg ( \frac R {2a_1} \bigg )^{-n} \int_{B_{a_1R/2}^+} x_1  w(x) \dd x \\ &\geqslant a_1^{-n-1} R^{-n-2} \tau d^{-n-2} \bigg (  \bigg (1 - \frac \theta 2 \bigg )^{-n-2} -\frac 1 2 \bigg ) \int_{B_{ (\theta d)/64}^+(a) \setminus U} x_1  \zeta(x) \dd x \\
&\geqslant C \tau a_1 \theta^{-n-2}  d^{-2n-4} \int_{B_{ (\theta d)/64}^+(a) \setminus U } x_1^2 \dd x
\end{align*} for \(\theta\) sufficiently small. Thus, \begin{align*}
\int_{B_{( \theta d)/64}^+(a) \setminus U } x_1^2 \dd x &\leqslant C   (\theta d)^{n+2}   \bigg ( \bigg (1- \frac \theta 2 \bigg )^{-n-2}-1 \bigg ) + \frac{C d^{n+2}}\tau  +C( \theta  d)^{n+2s+2}\\
&\leqslant  C   (\theta d)^{n+2}   \bigg ( \bigg (1- \frac \theta 2 \bigg )^{-n-2}-1 +\theta^{2s} \bigg ) + \frac{C d^{n+2}}\tau  
\end{align*} using that \(d^{n+2s+2}<d^{n+2}\). 
Recalling formula~\eqref{sdwet68980poiuytrkjhgfdmnbvcx23456789}
and taking~\(\theta\) sufficiently small, we obtain that \begin{align*}
C(\theta d)^{n+2}   \bigg ( \bigg (1- \frac \theta 2 \bigg )^{-n-2}-1 +\theta^{2s} \bigg ) \leqslant \frac 1 4 \int_{B_{( \theta d)/64}^+ (a)} x_1^2 \dd x .
\end{align*} which concludes the proof in Case 1(b). 

\emph{Case 2:} In this case, we can directly apply \thref{g9foAd2c} to~\eqref{ayFE7nQK}. When we do this we find that \begin{align*}&
\bigg ( \frac{\theta d }{64} \bigg )^{-n-2} \int_{B_{\theta d/64}^+(a)} x_1  w (x)\dd x \\
\leqslant&\, C \bigg ( \partial_1  w (0)  +  (r\theta)^{-n-2} + \tau \theta^{2s} d^{-n+2s-2} \bigg )\\
=&\, C\tau d^{-n-2}  \bigg (   \bigg (1- \frac \theta 2 \bigg )^{-n-2}-1 \bigg )\partial_1\zeta(0) + C   (r\theta)^{-n-2} + C\tau \theta^{2s} d^{-n+2s-2}.
\end{align*} On the other hand,~\eqref{I9FvOxCz} is still valid in Case 2 so \begin{align*}
(\theta d )^{-n-2} \int_{B_{\theta d/64}^+(a)} x_1  w (x)\dd x &\geqslant  \tau  \theta^{-n-2} d^{-2n-4} \bigg (  \bigg (1 - \frac \theta 2 \bigg )^{-n-2} -\frac 1 2 \bigg ) \int_{B_{\theta d/64}^+(a)\setminus U } x_1 \zeta (x) \dd y \\
&\geqslant C \tau  \theta^{-n-2} d^{-2n-4} \int_{B_{\theta d/64}^+(a)\setminus U } x_1^2\dd x
\end{align*} using \thref{KaMmndyO} and taking \(\theta\) sufficiently small. Thus,  \begin{align*}
\int_{B_{\theta d/64}^+(a)\setminus U } x_1^2\dd x &\leqslant C (\theta d)^{n+2}  \bigg (   \bigg (1- \frac \theta 2 \bigg )^{-n-2}-1 \bigg ) + \frac{ C d^{n+2}}\tau + C (\theta d)^{n+2s+2}\\
&\leqslant  C (\theta d)^{n+2}  \bigg (   \bigg (1- \frac \theta 2 \bigg )^{-n-2}-1 +\theta^{2s} \bigg ) + \frac{ C d^{n+2}}\tau 
\end{align*} using that \(d^{n+2s+2}<d^{n+2}\). Then, by~\eqref{sdwet68980poiuytrkjhgfdmnbvcx23456789}, we may choose \(\theta\) sufficiently small so that \begin{align*}
 C (\theta d)^{n+2}  \bigg (   \bigg (1- \frac \theta 2 \bigg )^{-n-2}-1 +\theta^{2s} \bigg ) \leqslant \frac 1 4 \int_{B_{\theta d/64}^+(a)} x_1^2\dd x
\end{align*} which concludes the proof in Case 2. 

The proof of \thref{EP5Elxbz}
is thereby complete.
\end{proof}

\section{Proof of \protect\thref{Hvmln}} \label{lXIUl}

In this short section, we give the proof of \thref{Hvmln}. This follows from Theorems~\ref{DYcYH} and~\ref{C35ZH} along with a standard covering argument which we include here for completeness. 

\begin{proof}[Proof of \thref{Hvmln}] 
Recall that~\(\Omega^+=\Omega \cap \R^n_+\) and let \(B_{2R}(y)\) be a ball such that \(B_{2R}(y)\Subset\Omega^+\). We will first prove that there exists a constant \(C=C(n,s,R,y)\) such that \begin{align}
\sup_{B_R(y)} \frac{u(x)}{x_1} &\leqslant C \inf_{B_R(y)} \frac{u(x)}{x_1} .  \label{c4oOr}
\end{align}Indeed, if \(\tilde u (x) := u (y_1x+(0,y'))\) and~\(\tilde c (x):= y_1^{2s} c(y_1x+(0,y'))\) then  \begin{align*}
(-\Delta)^s \tilde u (x)+\tilde c \tilde u = 0, \qquad \text{in } B_{2R/y_1} (e_1). 
\end{align*} By \thref{DYcYH}, \begin{align*}
\sup_{B_R(y)} \frac{u(x)}{x_1}  \leqslant C \sup_{B_R(y)} u = C\sup_{B_{R/y_1}(e_1)} \tilde u  & \leqslant C \inf_{B_{R/y_1}(e_1)} \tilde u =C\inf_{B_R(y)} u \leqslant C\inf_{B_R(y)} \frac{u(x)}{x_1}
\end{align*} using that \(B_R(y) \Subset\R^n_+\).

Next, let \(\{a^{(k)}\}_{k=1}^\infty,\{b^{(k)}\}_{k=1}^\infty \subset \tilde \Omega^+\) be such that \begin{align*}
\frac{u(a^{(k)})}{a^{(k)}_1} \to \sup_{\Omega'}\frac{u(x)}{x_1} \quad \text{and} \quad \frac{u(b^{(k)})}{b^{(k)}_1} \to \inf_{\Omega'}\frac{u(x)}{x_1}
\end{align*} as \(k\to \infty\). After possibly passing to a subsequence, there exist~\(a\), \(b\in \overline{\tilde \Omega^+}\) such that~\(a^{(k)}\to a\) and~\(b^{(k)}\to b\). Let \(\gamma \subset \tilde \Omega^+\) be a curve connecting \(a\) and \(b\). By the Heine-Borel Theorem, there exists a finite collection of balls \(\{B^{(k)}\}_{k=1}^M\) with centres in \(\tilde \Omega \cap \{x_1=0\}\) such that \begin{align*}
\tilde \Omega \cap \{x_1=0\} \subset \bigcup_{k=1}^M B^{(k)}  \Subset\Omega . 
\end{align*} Moreover, if \( \tilde \gamma := \gamma \setminus \bigcup_{k=1}^M B^{(k)}\) then there exists a further collection of balls \(\{B^{(k)}\}_{k=M+1}^N\) with centres in \(\tilde \gamma\) and radii equal to \( \frac 12 \dist(\tilde \gamma, \partial( \Omega^+))\) such that \begin{align*}
\tilde \gamma \subset \bigcup_{k=M+1}^N B^{(k)}  \Subset \Omega^+.
\end{align*} By construction, \(\gamma\) is covered by \(\{B^{(k)}\}_{k=1}^n\). Thus, iteratively applying \thref{C35ZH} (after translating and rescaling) to each \(B^{(k)}\), \(k=1,\dots, M\), and~\eqref{c4oOr} to each \(B^{(k)}\), \(k=M+1,\dots, N\), we obtain the result. 
\end{proof}

\section{Appendix A: A counterexample} \label{j4njb}
In this appendix, we demonstrate that \thref{Hvmln} is in general false if we do not assume antisymmetry. We will do this by constructing a sequence of functions \(\{ u_k\}_{k=1}^\infty \subset C^\infty(B_1(2e_1)) \cap L^\infty(\R^n)\) such that \((-\Delta)^su_k=0\) in \(B_1(2e_1)\), \(u_k \geqslant 0\) in \(\R^n_+\), and \begin{align*}
\frac{\sup_{B_{1/2}(2e_1)}u_k}{\inf_{B_{1/2}(2e_1)}u_k} \to + \infty  \qquad \text{as } k\to \infty.
\end{align*} The proof will rely on the mean value property of \(s\)-harmonic functions. 

Suppose that \(M\geqslant 0\) and \(\zeta_1\), \(\zeta_2\) are smooth functions such that \(0\leqslant \zeta_1,\zeta_2\leqslant 1\) in \(\R^n \setminus B_1(2e_1)\), and \begin{align*}
\zeta_1(x) = 
\begin{cases}
0 &\text{in } \R^n_+ \setminus B_1(2e_1),\\
1 &\text{in } \R^n_- \setminus \{x_1>-1\}
\end{cases} \qquad{\mbox{and}}\qquad \zeta_2(x) =\begin{cases}
0 &\text{in } \R^n_+ \setminus B_1(2e_1) ,\\
1 &\text{in } B_{1/2}(-2e_1).
\end{cases}
\end{align*}Then let \(v\) and \(w_M\) be the solutions to \begin{align*}
\begin{PDE}
(-\Delta)^s v &= 0 &\text{in }B_1(2e_1), \\
v&=\zeta_1 &\text{in } \R^n\setminus  B_1(2e_1)
\end{PDE}
\end{align*} and  \begin{align*}
    \begin{PDE}
(-\Delta)^s w_M &= 0 &\text{in }B_1(2e_1), \\
w_M&=-M\zeta_2 &\text{in } \R^n\setminus  B_1(2e_1),
\end{PDE}
\end{align*} respectively. Both \(v\) and \(w_M\) are in \( C^\infty(B_1(2e_1)) \cap L^\infty(\R^n)\) owing to standard regularity theory. We want to emphasise that \(w_M\) depends on the parameter \(M\) (as indicated by the subscript) but \(v\) does not. Define \(\tilde u_M:= v+w_M\). Since \(w_M \equiv 0 \) when \(M=0\) and, by the strong maximum principle, \(v>0\) in \(B_1(2e_1)\), we have that \begin{align*}
\tilde u_0 >0 \qquad \text{in } B_1 (2e_1).
\end{align*} Hence, we can define \begin{align*}
\bar M := \sup \{ M \geqslant 0 \text{ s.t. } \tilde u_M >0 \text{ in } B_1 (2e_1)\}
\end{align*} though \(\bar M\) may possibly be infinity. We will show that \(\bar M\) is in fact finite and that \begin{align}
\inf_{B_1(2e_1)} \tilde u_{\bar M} = 0. \label{fDMti}
\end{align} Once these two facts have been established, we complete the proof as follows: set \(u_k := \tilde u_{\bar M - 1/k}\). By construction, \(u_k \geqslant 0\) in \(\R^n_+\) and \((-\Delta)^su_k=0\) in \(B_1(2e_1)\). Moreover, by the maximum principle, \(v\leqslant 1\) and \(w_M \leqslant 0\) in \(\R^n\), so \( u_k \leqslant 1\) in \(\R^n\). Hence, \begin{align*}
\frac{\sup_{B_{1/2}(2e_1)}u_k}{\inf_{B_{1/2}(2e_1)}u_k} &\leqslant \frac 1 {\inf_{B_1(2e_1)}u_k} \to + \infty
\end{align*} as \(k \to \infty\). 

Let us show that \(\bar M\) is finite. Since \(w_M\) is harmonic in \(B_1(2e_1)\), the mean value property for \(s\) harmonic functions, for example see \cite[Section 15]{MR3916700}, tells us that\begin{align*}
w_M(2e_1) &= -CM \int_{\R^n \setminus B_1(2e_1)} \frac{\zeta_2(y)}{(\vert y \vert^2 -1 )^s \vert y \vert^n} \dd y .
\end{align*} Then, since \(\zeta_2\equiv 1\) in \(B_{1/2}(-2e_1)\) and \(\zeta_2\geqslant0\), \begin{align*}
\int_{\R^n \setminus B_1(2e_1)} \frac{\zeta_2(y)}{(\vert y \vert^2 -1 )^s \vert y \vert^n} \dd y &\geqslant \int_{B_{1/2}(-2e_1)} \frac{\dd y}{(\vert y \vert^2 -1 )^s \vert y \vert^n}\geqslant C.  
\end{align*} It follows that \(w_M(2e_1) \leqslant -CM\) whence \begin{align*}
\tilde u_M (2e_1) \leqslant 1-CM \leqslant 0
\end{align*} for \(M\) sufficiently large. This gives that \(\bar M\) is finite.

Now we will show \eqref{fDMti}. For the sake of contradiction, suppose that there exists \(a>0\) such that \(\tilde u_{\bar M } \geqslant a \) in \(B_1(2e_1)\). Suppose that \(\varepsilon>0\) is small and let \(h^{(\varepsilon)} := \tilde u_{\bar M+\varepsilon}-\tilde u_{\bar M} = w_{\bar M+\varepsilon}-w_{\bar M}\). Then \(h^{(\varepsilon)}\) is \(s\)-harmonic in \(B_1(2e_1)\) and \begin{align*}
h^{(\varepsilon)}(x) = - \varepsilon \zeta_2(x) \geqslant -\varepsilon \qquad \text{in } \R^n \setminus B_1(2e_1).
\end{align*}Thus, using the maximum principle we conclude that \(h^{(\varepsilon)} \geqslant -\varepsilon\) in \(B_1(2e_1)\),
which in turn gives that \begin{align*}
\tilde u_{\bar M+\varepsilon} =\tilde u +h^{(\varepsilon)}  \geqslant a-\varepsilon >0 \qquad \text{in } B_1(2e_1)
\end{align*} for \(\varepsilon\) sufficiently small. This contradicts the definition of \(\bar M\). 

\section{Appendix B: Alternate proof of \protect\thref{C35ZH} when \(c\equiv0\)} \label{4CEly}

In this appendix, we will provide an alternate elementary proof of \thref{C35ZH} in the particular case \(c\equiv0\) and \(u\in \mathscr L_s(\R^n)\). More precisely, we prove the following.

\begin{thm} \thlabel{oCuv7Zs2}
Let \(u\in C^{2s+\alpha}(B_1) \cap \mathscr L_s(\R^n) \) for some \(\alpha>0\) with \(2s+\alpha\) not an integer. Suppose that \(u\) is antisymmetric, non-negative in \(\R^n_+\), and \(s\)-harmonic in \(B_1^+\).

Then there exists \(C>0\) depending only on \(n\) and \(s\) such that \begin{align*}
\sup_{B_{1/2}^+} \frac{u(x)}{x_1 } &\leqslant C \inf_{B_{1/2}^+} \frac{u(x)}{x_1 }. 
\end{align*} Moreover, \(\inf_{B_{1/2}^+} \frac{u(x)}{x_1 }\) and \(\sup_{B_{1/2}^+} \frac{u(x)}{x_1 }\) are comparable to \(\Anorm{u}\). 
\end{thm} 

Except for the statement  \(\inf_{B_{1/2}^+} \frac{u(x)}{x_1 }\) and \(\sup_{B_{1/2}^+} \frac{u(x)}{x_1 }\) are comparable to \(\Anorm{u}\), Theorem~\ref{oCuv7Zs2} was proven in \cite{ciraolo2021symmetry}. Both the proof presented here and the proof in \cite{ciraolo2021symmetry} rely on the Poisson kernel representation for \(s\)-harmonic functions in a ball. Despite this, our proof is entirely different to the proof in \cite{ciraolo2021symmetry}. This was necessary to prove that~\(\inf_{B_{1/2}^+} \frac{u(x)}{x_1 }\) and \(\sup_{B_{1/2}^+} \frac{u(x)}{x_1 }\) are comparable to \(\Anorm{u}\) which does not readily follow from the proof in \cite{ciraolo2021symmetry}. Our proof of Theorem~\ref{oCuv7Zs2} is a consequence of a new mean-value formula for antisymmetric \(s\)-harmonic functions (Proposition~\ref{cqGgE}) which we believe to be interesting in and of itself.

We first prove an alternate expression of the Poisson kernel representation formula for antisymmetric functions.

\begin{lem} \thlabel{Gku6y}
Let \(u\in C^{2s+\alpha}(B_1) \cap \mathscr L_s (\R^n)\) and \(r\in(0, 1] \).  If \(u\) is antisymmetric and \((-\Delta)^s u = 0 \) in~\(B_1\).

Then \begin{align}
u(x) & =  \gamma_{n,s} \int_{\R^n_+ \setminus B_r^+} \bigg ( \frac{r^2- \vert x \vert^2 }{\vert  y \vert^2 -r^2 } \bigg )^s \bigg ( \frac 1 {\vert x - y \vert^n }-\frac 1 {\vert x_\ast - y \vert^n }  \bigg ) u(y) \dd y \label{BPyAO}
\end{align} for all \(x\in B_1^+\).  Here~$\gamma_{n,s}$ is given in~\eqref{sry6yagamma098765}.
\end{lem}

We remark that
since \(y \mapsto  \frac 1 {\vert x - y \vert^n }-\frac 1 {\vert x_\ast - y \vert^n }  \) is antisymmetric for each \(x\in \R^n_+\), we can rewrite~\eqref{BPyAO} as \begin{align} \label{ZrGnZFNg}
u(x) & = \frac 1 2  \gamma_{n,s} \int_{\R^n \setminus B_r} \bigg ( \frac{r^2- \vert x \vert^2 }{\vert  y \vert^2 -r^2 } \bigg )^s \bigg ( \frac 1 {\vert x - y \vert^n }-\frac 1 {\vert x_\ast - y \vert^n }  \bigg ) u(y) \dd y.
\end{align}

\begin{proof}[Proof of Lemma~\ref{Gku6y}]
The Poisson representation formula \cite{aintegrales} gives\begin{align}
u(x) & =\gamma_{n,s} \int_{\R^n \setminus B_r} \bigg ( \frac{r^2- \vert x \vert^2 }{\vert  y \vert^2 -r^2 } \bigg )^s \frac {u(y)} {\vert x - y \vert^n }\dd y \label{vDlwi}
\end{align} where \(\gamma_{n,s}= \frac{\sin(\pi s)\Gamma (n/2)}{\pi^{\frac n 2 +1 }}\); see also \cite[p.112,p.122]{MR0350027} and \cite[Section 15]{MR3916700}. Splitting the integral in \eqref{vDlwi} into two separate integrals over \(\R^n_+ \setminus B_r^+\) and \(\R^n_- \setminus B_r^-\) respectively, then making the change of variables \(y \to y_\ast\) in the integral over \(\R^n_- \setminus B_r^-\) and using that \(u\) is antisymmetric, we obtain \begin{align*}
u(x) & =\gamma_{n,s} \int_{\R^n_+ \setminus B_r^+} \bigg ( \frac{r^2- \vert x \vert^2 }{\vert  y \vert^2 -r^2 } \bigg )^s \bigg ( \frac 1 {\vert x - y \vert^n }-\frac 1 {\vert x_\ast - y \vert^n }  \bigg ) u(y) \dd y .  \qedhere
\end{align*} 
\end{proof}

Now that we have proven the Poisson kernel formula for antisymmetric functions in Lemma~\ref{Gku6y},
we now establish Proposition~\ref{cqGgE}.

\begin{proof}[Proof of Proposition~\ref{cqGgE}]
Let \(h\in(0,1)\). It follows from Lemma~\ref{Gku6y} that\begin{align}
\frac{u(he_1)} h  & = \frac 1 {h} \gamma_{n,s} \int_{\R^n_+ \setminus B_r^+} \bigg ( \frac{r^2- h^2 }{\vert  y \vert^2 -r^2 } \bigg )^s \bigg ( \frac 1 {\vert he_1 - y \vert^n }-\frac 1 {\vert he_1+ y \vert^n }  \bigg ) u(y) \dd y  \label{CNZLU}
\end{align} for all \(r>h\). Taking a Taylor series in \(h\) about \(0\), we have the pointwise limit \begin{align*}
\lim_{h \to 0 } \frac 1 h \bigg ( \frac 1 {\vert he_1 - y \vert^n }-\frac 1 {\vert he_1+ y \vert^n }  \bigg ) &= \frac{2ny_1 }{\vert y \vert^{n+2}}.
\end{align*} Moreover, by a similar argument to~\eqref{LxZU6},\begin{align*}
\bigg \vert  \frac 1 {(\vert  y \vert^2 -r^2)^s }  \bigg ( \frac 1 {\vert x - y \vert^n }-\frac 1 {\vert x_\ast - y \vert^n }  \bigg )\bigg \vert  &\leqslant   \frac{2n  h \vert  y_1 \vert }{(\vert  y \vert^2 -r^2)^s\vert he_1 - y \vert^{n+2}} \in L^1(\R^n \setminus B_r).
\end{align*} Thus, we obtain the result by taking the limit \(h\to 0\) in \eqref{CNZLU} and applying the Dominated Convergence Theorem to justify swapping the limit and the integral on the right-hand side. 
\end{proof}

{F}rom Proposition~\ref{cqGgE}, we obtain the following corollary. 

\begin{cor} \thlabel{OUqJk}
Let \(u\in C^{2s+\alpha}(B_1) \cap \mathscr L_s(\R^n)\) with \(\alpha>0\) and \(2s+\alpha\) not an integer. Suppose that \(u\) is antisymmetric and \((-\Delta)^s u = 0 \) in \(B_1\).

Then there exists a radially symmetric function \(\psi_s \in C(\R^n) \) satisfying \begin{align}
\frac{C^{-1}}{1+\vert y \vert^{n+2s+2}} \leqslant \psi_s (y) \leqslant \frac{C}{1+\vert y \vert^{n+2s+2}} \qquad \text{for all } y \in \R^n, \label{nhvr2}
\end{align}
for some~$C>1$, such that \begin{align*}
\frac{\partial u}{\partial x_1} (0) &=  \int_{\R^n} y_1\psi_s (y) u(y) \dd y.
\end{align*} In particular, if \(u\) is non-negative in \(\R^n_+\) then \[C^{-1} \Anorm{u} \leqslant \frac{\partial u}{\partial x_1} (0)  \leqslant C \Anorm{u}.\] 
\end{cor}

\begin{proof}
Let \begin{align*}
\psi_s (y) &:=n (n+2)\gamma_{n,s} \int_0^{\min \{ 1/\vert y \vert,1\}} \frac{r^{2s+n+1}}{(1 - r^2)^s}\dd  r. 
\end{align*} It is clear that \(\psi_s\in C(\R^n)\) and that there exists \(C>1\) such that~\eqref{nhvr2} holds. Since \(u\in \mathscr L_s(\R^n)\), we have that~\( \Anorm{u} <+\infty\), so it follows that \begin{align}
\bigg \vert \int_{\R^n}y_1 \psi_s (y) u(y) \dd y \bigg \vert \leqslant C \Anorm{u} < +\infty. \label{9B6Vm}
\end{align} If we multiply \(\partial_1 u (0)\) by \(r^{n+1}\) then integrate from 0 to 1, Proposition~\ref{cqGgE} and~\eqref{ZrGnZFNg} give that
\begin{align}
\frac 1 {n+2} \frac{\partial u}{\partial x_1} (0) = \int_0^1 r^{n+1} \frac{\partial u}{\partial x_1} (0) \dd r = n \gamma_{n,s}\int_0^1 \int_{\R^n \setminus B_r} \frac{r^{2s+n+1}y_1 u(y)}{(\vert y \vert^2 - r^2)^s\vert y \vert^{n+2}} \dd y \dd r .\label{YYJ5o}
\end{align} At this point, let us observe that if we formally swap the integrals in~\eqref{YYJ5o} and then make the change of variables \(r=\vert y \vert \tilde r\), we obtain \begin{align}
 n \gamma_{n,s}  \int_{\R^n \setminus B_1} \int_0^1 &\frac{r^{2s+n+1}y_1 u(y)}{(\vert y \vert^2 - r^2)^s\vert y \vert^{n+2}}\dd r \dd y +  n \gamma_{n,s}\int_{B_1} \int_0^{\vert  y \vert}  \frac{r^{2s+n+1}y_1 u(y)}{(\vert y \vert^2 - r^2)^s\vert y \vert^{n+2}}\dd r \dd y \nonumber\\
&= n \gamma_{n,s}  \int_{\R^n \setminus B_1} \bigg ( \int_0^{1/\vert y \vert} \frac{\tilde r^{2s+n+1}}{(1 -\tilde r^2)^s}\dd \tilde r \bigg ) y_1 u(y)\dd y \nonumber \\
&\qquad \qquad +  n \gamma_{n,s}\int_{B_1}\bigg ( \int_0^1   \frac{\tilde r^{2s+n+1}}{(1 - \tilde r^2)^s}\dd \tilde r\bigg )y_1 u(y) \dd y \nonumber \\
&= \frac 1 {n+2} \int_{\R^n}y_1 \psi_s(y) u(y) \dd y .  \label{LeVoi}
\end{align} By~\eqref{9B6Vm}, equation~\eqref{LeVoi} is finite. Hence, Fubini's theorem justifies changing the order of integration in \eqref{YYJ5o} and that the right-hand side of~\eqref{YYJ5o} is equal to~\eqref{LeVoi} which proves the result. 
\end{proof}

At this point, we can give the proof of Theorem~\ref{oCuv7Zs2}. 

\begin{proof}[Proof of Theorem~\ref{oCuv7Zs2}]
We will begin by proving that \begin{align}
u(x)  \geqslant C x_1 \Anorm{u}  \qquad \text{for all } x\in B_{1/2}^+. \label{p6oBI}
\end{align}
To this end, we observe that since \(\vert x_\ast - y \vert^{n+2} \leqslant C \vert y \vert^{n+2} \) for all \(x \in B_1\) and~\(y\in \R^n\setminus B_1\),~\eqref{buKHzlE6} gives that \begin{align*}
\frac 1 {\vert x - y \vert^n }-\frac 1 {\vert x_\ast - y \vert^n } \geqslant C \frac{  x_1 y_1}{\vert y \vert^{n+2}} \qquad \text{for all } x\in B_1^+ {\mbox{ and }} y\in \R^n_+ \setminus B_1^+.
\end{align*}Hence, by Lemma~\ref{Gku6y}, for all~$x\in B_{1/2}^+$,  \begin{align*}
u(x) &= C\int_{\R^n \setminus B_1} \bigg ( \frac{1- \vert x \vert^2 }{\vert  y \vert^2 -1 } \bigg )^s \bigg ( \frac 1 {\vert x - y \vert^n }-\frac 1 {\vert x_\ast - y \vert^n }  \bigg ) u(y) \dd y  \\
&\geqslant Cx_1 \int_{\R^n \setminus B_1} \frac{ y_1 u(y)}{\big ( \vert  y \vert^2 -1 \big )^s\vert y \vert^{n+2}}    \dd y 
\end{align*} where we used that \((-y_1)u(y_\ast)=y_1u(y)\) and that \(u\geqslant 0\) in \(\R^n_+\). Then
Proposition~\ref{cqGgE} with \(r=1\) gives that \begin{align*}
u(x) &\geqslant Cx_1 \frac{\partial u}{\partial x_1}(0)  \qquad \text{for all } x\in B_{1/2}^+.
\end{align*} Finally, Corollary~\ref{OUqJk} gives~\eqref{p6oBI}. 

Next, we will prove that  \begin{align}
u(x)  \leqslant C x_1 \Anorm{u}  \qquad \text{for all } x\in B_{1/2}^+. \label{3abqf}
\end{align} Similar to above, for all \(x\in B_{1/2}^+\) and~\(y\in  \R^n_+ \setminus B_1^+\), we have that~\(\vert x- y \vert \geqslant \frac 1 2 \vert y \vert\), so~\eqref{LxZU6} gives that\begin{align*}
 \frac 1 {\vert x - y \vert^n }-\frac 1 {\vert x_\ast - y \vert^n }  \leqslant \frac{C  x_1 y_1}{\vert y \vert^{n+2}} \qquad \text{for all } x\in B_{1/2}^+ {\mbox{ and }} y\in  \R^n_+ \setminus B_1^+.
\end{align*} As before, using Lemma~\ref{Gku6y}, we have that, for all~$x\in B_{1/2}^+$,\begin{align*}
u(x) &= C\int_{\R^n \setminus B_1} \bigg ( \frac{1- \vert x \vert^2 }{\vert  y \vert^2 -1 } \bigg )^s \bigg ( \frac 1 {\vert x - y \vert^n }-\frac 1 {\vert x_\ast - y \vert^n }  \bigg ) u(y) \dd y  \\
&\leqslant C x_1 \int_{\R^n \setminus B_1} \frac{ y_1 u(y)}{\big ( \vert  y \vert^2 -1 \big )^s\vert y \vert^{n+2}}    \dd y.
\end{align*}Then Proposition~\ref{cqGgE} and Corollary~\ref{OUqJk} give that \begin{align*}
u(x) &\leqslant Cx_1 \frac{\partial u}{\partial x_1}(0) \leqslant C x_1 \Anorm{u}  \qquad \text{for all } x\in B_{1/2}^+,
\end{align*} which is~\eqref{3abqf}.

{F}rom~\eqref{p6oBI} and~\eqref{3abqf} the result follows easily. 
\end{proof}

\section{Appendix C: A proof of~\eqref{LA:PAKSM} when~$c:=0$
that relies on extension methods}\label{APPEEXT:1}

We consider the extended variables~$X:=(x,y)\in\R^n\times\R$.
Then, a solution~$u$ of~$(-\Delta)^su=0$ in~$\Omega^+$ can be seen as the trace along~$\Omega^+\times\{0\}$ of its $a$-harmonic extension~$U=U(x,y)$ satisfying
$$ {\rm div}_{\!X}\big(|y|^a\nabla_{\!X} U\big)=0\quad{\mbox{ in }}\,\R^{n+1},$$
where~$a:=1-2s$, see Lemma~4.1 in~\cite{MR2354493}.

We observe that the function~$V(x,y):=x_1$ is also a solution of the above equation.
Also, if~$u$ is antisymmetric, then so is~$U$, and consequently~$U=V=0$ on~$\{x_1=0\}$.

As a result, by the boundary Harnack inequality (see~\cite{MR730093}),
\begin{equation}\label{LA:PAKSM:2}
\sup_{\tilde \Omega^+\times(0,1)} \frac{U}{V}  \leqslant C \inf_{ \tilde \Omega^+\times(0,1)} \frac{U}{V} .
\end{equation}

In addition,
$$ \sup_{\tilde \Omega^+\times(0,1)} \frac{U}{V}\ge
\sup_{\tilde \Omega^+\times\{0\}} \frac{U}{V}
=\sup_{\tilde \Omega^+} \frac{u(x)}{x_1}$$
and similarly
$$ \inf_{\tilde \Omega^+\times(0,1)} \frac{U}{V}\le\inf_{\tilde \Omega^+} \frac{u(x)}{x_1}.$$
{F}rom these observations and~\eqref{LA:PAKSM:2} we obtain~\eqref{LA:PAKSM}
in this case.

\section*{Acknowledgements}

All the authors are members of AustMS. 
SD is supported by
the Australian Research Council DECRA DE180100957
``PDEs, free boundaries and applications''.
JT is supported by an Australian Government Research Training Program Scholarship. 
EV is supported by the Australian Laureate Fellowship
FL190100081
``Minimal surfaces, free boundaries and partial differential equations''.

JT would also like to thank David Perrella for his interesting and fruitful conversations.

%% file: Part2/Harnack-Bochner-thesisV2.tex
\chapter{The nonlocal Harnack inequality for antisymmetric functions}  \label{OXXm8GiN}

We revisit a Harnack inequality for antisymmetric functions that has been recently established for the fractional Laplacian and we extend it to more general nonlocal elliptic operators.

The new approach to deal with these problems that we propose in this paper leverages Bochner's relation, allowing one to relate a one-dimensional Fourier transform of an odd function with a three-dimensional Fourier transform of a radial function.

In this way, Harnack inequalities for odd functions, which are essentially Harnack inequalities of boundary type, are reduced to interior Harnack inequalities.

%
%

\section{Introduction}\label{usUMTbSF}

One of the characterizing properties of classical harmonic functions is their ``rigidity'': in spite of the fact that harmonic functions may exhibit, in general, different patterns, a common feature is that if, at a given point, a harmonic function ``bends up'' in a certain direction, then necessarily it has to ``bend down'' in another direction. This observation is typically formalised by the so-called maximum principle. Furthermore, the maximum principle is, in turn, quantified by the Harnack inequality~\cite{HAR} which asserts that the values of a positive harmonic function in a given ball are comparable (see~\cite{MR2291922} for a thorough introduction to the topic).

The growing interest recently surged in 
the study of nonlocal and fractional equations, especially in relation to fractional powers of the Laplace operator,
has stimulated an intense investigation on the possible validity of Harnack-type inequalities in a nonlocal framework.
Several versions of the Harnack inequality have been obtained for the fractional Laplacian as well as for more general nonlocal operators, see~\cite{MR1918242, MR2013738, MR2031452, MR2095633, MR2244602, MR2354493, MR2754080, MR3237774, MR3299862}.\medskip

A striking difference between the classical and the fractional settings is that in the former
a sign assumption on the solution is taken only in the region of interest, while in the
latter such an assumption is known to be insufficient. Indeed, counterexamples to the validity of the nonlocal
Harnack inequality in the absence of a global sign assumptions have been put forth in~\cite{Kassmann2007clas} (see also~\cite[Theorem~3.3.1]{MR3469920} and~\cite[Theorem~2.2]{MR2817382}). In fact,
these counterexamples are just particular cases
of the significant effect that the faraway oscillations may produce
on the local patterns of solutions to nonlocal equations,
see~\cite{MR3626547, MR3935264}
(in case of solutions which change sign, the nonlocal framework however allows
for a Harnack inequality with a suitable integral remainder, see~\cite[Theorem~2.3]{MR2817382}).\medskip

A rather prototypical situation in which the sign condition is violated is that of odd antisymmetric functions. In the nonlocal world, this situation frequently occurs, especially when dealing with moving planes and reflection methods. On these occasions, the lack of sign assumption needs to be replaced by bespoke maximum principles which carefully take into account the additional antisymmetrical structure of the problem under consideration, see~\cite{MR3395749, MR3453602, MR4108219, MR4030266, MR4308250, RoleAntisym2022}.
In particular, in the antisymmetric setting, the following
nonlocal Harnack inequality has been established in~\cite{MR4567494}:

\begin{thm}
\label{thm:odd:harnack}
Suppose that $\Omega$ is an open set of~$\R^n$,
and let~$u$ be a function on $\R^n$ such that $u$ is $C^{2s + \delta}$ in $\Omega$ and antisymmetric,
with
\begin{equation*}
\int_{\R^n}\frac{|x_1|\,|  u(x)|}{(1 + |x|)^{n + 2 +2s}}\, dx<+\infty
\end{equation*}
and~$x_1 u(x) \ge 0$ for every $x \in \R^n$.

Let $c$ be a bounded Borel function on $\Omega$. If
\[
 (-\Delta)^s u + c u = 0
\]
in $\Omega_+ = \{x \in \Omega : x_1 > 0\}$, then for every compact subset $K$ of $\Omega$ we have
\begin{equation}
\label{eq:odd:harnack}
 \sup_{x \in K_0} \frac{u(x)}{x_1} \le C(K, \Omega, \|c\|_{L^\infty(\Omega_+)}) \inf_{x \in K_0} \frac{u(x)}{x_1} ,
\end{equation}
where $K_0 = \{x \in K : x_1 \ne 0\}$.
\end{thm}

In this article, we revisit the nonlocal Harnack inequality under a different perspective,
providing a new, shorter proof of Theorem~\ref{thm:odd:harnack}. This proof relies
on the so-called Bochner's relation (or Hecke--Bochner identity), which links a one-dimensional Fourier transform of an odd function with a three-dimensional Fourier transform of a radial function: roughly speaking, for any function~$ f$ on the positive reals (under a natural integrability condition), if~$F(|Z|)$ denotes the \mbox{3-D} Fourier transform of the 3-D radial function~$f(|X|)$, then the 1-D Fourier transform of the \mbox{1-D} function~$x f(|x|)$ is equal to~$i z F(|z|)$. The complete statement of Bochner's relation is actually even more general (it works in higher dimensions too, and the variable~$x$ can be replaced by a homogeneous harmonic polynomial), see~\cite[page~72]{MR0290095}.

The gist of the argument that we present here is that
Bochner's relation extends to the case of antisymmetric functions dealt with in
Theorem~\ref{thm:odd:harnack}. More specifically,
the Fourier transform of an \(n\)-dimensional antisymmetric function~$ x_1 f(|x_1|, x')$ is equal to~$i z_1 F(|z_1|, z')$, where~$F (|Z|, z')$ is the $(n + 2)$-dimensional Fourier transform of~$f(|X|, x')$, where~$X$ and~$ Z$ are 3-dimensional vectors, $x_1$ and~$z_1$ are real numbers, and~$x'$ and~$z'$ are~$(n - 1)$-dimensional vectors.

Using the above property, one finds that the $n$-dimensional fractional Laplacian applied to $x_1 f(|x_1|, x')$ is equal to~$x_1 g(|x_1|, x')$ if and only if the~$(n + 2)$-dimensional fractional Laplacian applied to~$f(|X|, x')$ is equal to~$ g(|X|, x')$ (in fact, this is true for an arbitrary function of the Laplacian, not just for its fractional powers).

This, in turn, enables one to apply the usual Harnack inequality in dimension~$n + 2$ to deduce a Harnack inequality (or, in a sense, a ``boundary'' Harnack inequality)
in dimension~$ n$ for antisymmetric functions. In this framework,
Theorem~\ref{thm:odd:harnack} follows from a more classical result for positive solutions in dimension~$n + 2$.\medskip

For other applications of
Bochner's relation to antisymmetric functions, see \cite[Theorem~1]{MR2974318},
\cite[Theorem~1.5 and Proposition~3.1]{MR3413864} and~\cite{MR3640641}.\medskip

Besides its intrinsic elegance and the conciseness of the techniques involved, one of the advantages of the methods leveraging Bochner's relation consists of the broad versatility
of the arguments employed. In particular, the methodology employed allows us to
extend the previous result in Theorem~\ref{thm:odd:harnack} to a more general class of nonlocal operators. To this end, 
given a positive integer $n$, we consider an operator $L_n$, acting on an appropriate class of functions $u:\R^n\to \R$, and taking the form
\begin{equation}
 L_n u(x) = \int_{\R^n \setminus \{0\}} \bigl(u(x) - u(x + y) - \nabla u(x) \cdot y \chi_{B_1}(y)\bigr) j_n(|y|) dy ,
\label{5XOlUbY7}
\end{equation}%
where the kernel $j_n$ satisfies the usual integrability condition
\[
 \int_{\R^n \setminus \{0\}} \min\{1, |x|^2\} j_n(|x|) dx < +\infty .
\]
To proceed, we will need the following additional assumption:
\begin{equation}
\label{radialassumption}
 \text{$j_n$ is differentiable and nonincreasing on $(0, \infty)$.}
\end{equation}
This allows us to introduce the operator $L_{n + 2}$, defined by the same formula with kernel
\begin{equation}
\label{nplus2}
 j_{n + 2}(r) = -\frac{j_n'(r)}{2 \pi r} \, .
\end{equation}
In this setting, we establish that:

\begin{thm}\label{THM5.55}
Let~$L_n$ and $L_{n + 2}$ be given by~\eqref{5XOlUbY7}, and suppose that conditions~\eqref{radialassumption} and~\eqref{nplus2} are satisfied. Suppose that $\Omega$ is an open set of~$\R^n$, and let $u$ be a function on~$\R^n$ such that $u\in C^{2}(\Omega)$, $x_1^{-1} u \in L^\infty(\R^n)$, and $u$ is antisymmetric with~$x_1u(x) \geqslant 0$ for all~$x\in \R^n$.

Let $c$ be a bounded Borel function on $\Omega$. If
\[
L_n u + c u = 0
\]
in $\Omega_+ = \{x \in \Omega : x_1 > 0\}$ and the operator $L_{n+2}$ satisfies the Harnack inequality, then for every compact subset $K$ of $\Omega$ we have
\begin{equation}
\label{eq:odd:harnack-general}
 \sup_{x \in K_0} \frac{u(x)}{x_1} \le C(K, \Omega, \|c\|_{L^\infty(\Omega_+)}) \inf_{x \in K_0} \frac{u(x)}{x_1} ,
\end{equation}
where $K_0 = \{x \in K : x_1 \ne 0\}$ and $C(K, \Omega, M)$ is the same as the constant in the Harnack inequality for $L_{n + 2}$.
\end{thm}
For the precise definition of an operator satisfying the Harnack inequality, see Definition~\ref{mTbTJpHm}. The Harnack inequality (with a scale-invariant constant) is known to hold for a large class of operators $L_{n + 2}$, at least when $c = 0$; we refer to Remark~\ref{rem:harnack} in Section~\ref{OJSDLN09ryihfjg9reSTRpsdlmvIJJ} for further discussion.

The rest of the paper is organized as follows. In Section~\ref{SE:2}, we give some heuristic comments about the idea of reducing boundary Harnack inequalities to interior ones, to motivate the main strategy adopted in this paper, while, in Section~\ref{sec:bochner},
we recall Bochner's relation and frame it into a setting convenient for our purposes.

The techniques thus introduced are then applied to the case of the fractional Laplace operator in Section~\ref{sec:flap}, leading to the proof of Theorem~\ref{thm:odd:harnack}.

The case of general nonlocal operators will be explained in detail in Section~\ref{OJSDLN09ryihfjg9reSTRpsdlmvIJJ}.

\section{Reducing boundary estimates to interior estimates: a heuristic discussion}\label{SE:2}

The gist of the arguments that we exploit can be easily explained in the classical case of the Laplacian, in which all computations are straightforward, but already reveal a hidden higher dimensional geometric feature of the problem.

For this, we take coordinates~$y\in\R^k$ and~$z\in\R^{m}$. We use the notation~$\rho=|y|$ and, given a smooth function~$u$ in a given domain of~$\R^{m+1}$, for~$\lambda\in\R$ we set
$$ \tilde v(y,z):=|y|^\lambda u(|y|,z)=\rho^\lambda u(\rho,z).$$
Since, for every~$j\in\{1,\dots,m\}$,
$$ \rho\partial_{y_j} \rho=\frac12\partial_{y_j} \rho^2=\frac12\partial_{y_j}(y_1^2+\dots+y_k^2)=y_j,$$
and therefore
$$ \nabla_{y} \rho=\frac{y}\rho,$$
computing the Laplacian in cylindrical coordinates we find that
\begin{align*}
\Delta\tilde v&=\Delta_{y,z}\rho^\lambda \,u+\rho^\lambda \Delta_{y,z}u +2\nabla_{y,z} \rho^\lambda \cdot\nabla_{y,z} u\\&=
\Delta_y \rho^\lambda\, u+\rho^\lambda \big(\Delta_y u+\Delta_z u\big)+2\nabla_y \rho^\lambda\cdot\nabla_y u\\&=
(\lambda+k-2)\lambda\rho^{\lambda-2}u
+\rho^\lambda \big(\partial^2_\rho u+(k-1)\rho^{-1}\partial_\rho u+\Delta_zu\big)\\&\qquad+2\big(\lambda \rho^{\lambda-1}\nabla_y\rho\big)\cdot\big(\partial_\rho u\nabla_y\rho \big)
\\&=(\lambda+k-2)\lambda\rho^{\lambda-2}u+(2\lambda+k-1) \rho^{\lambda-1}\partial_\rho u+\rho^\lambda \partial^2_\rho u+\rho^\lambda \Delta_zu.
\end{align*}
The choice~$2\lambda+k-1=0$ allows one to cancel the first order term.
The additional choice~$\lambda+k-2=0$ makes the zero order term vanish. All in all, if~$k=3$ and~$\lambda=-1$, we find that
\begin{equation}\label{TBCW} \Delta\tilde v=\frac{\Delta u}\rho,\end{equation}
hence~$\tilde v$ is harmonic if and only if so is~$u$.

For us, the interest of~\eqref{TBCW} is that, for instance, it allows us to recover a boundary Harnack inequality for the Laplace
operator straightforwardly from the interior Harnack inequality.
Specifically, {\em if~$c$ is a bounded Borel function on $B_1^+:=B_1\cap\{x_1>0\}\subset\R^n$
and~$u$ is a nonnegative solution of~$-\Delta u + c u = 0$ in~$B_1^+$, then
\begin{equation}
\label{PRE:eq:odd:harnack}
 \sup_{x \in B_{1/2}^+} \frac{u(x)}{x_1} \le C \inf_{x \in B_{1/2}^+} \frac{u(x)}{x_1} ,
\end{equation}
where $C$ is a positive constant depending only on~$n$ and~$\|c\|_{L^\infty(B_1^+)}$}.

In our setting, the proof of~\eqref{PRE:eq:odd:harnack} can be obtained directly combining~\eqref{TBCW} and the classical Harnack inequality: namely, if we use the notation~$\tilde c(y,z):=c(|y|,z)$, we deduce from~\eqref{TBCW}  that
$$\big(-\Delta \tilde v + \tilde c \tilde v\big)|y| =-\Delta u + c u = 0$$
whenever~$(|y|,z)\in B_1$ and so, in particular, for all~$(y,z)\in\R^3\times\R^{n-1}$ such that~$|(y,z)|<1$.

This and the classical Harnack inequality lead to
$$ \sup_{|(y,z)|<1/2} \tilde v(y,z) \le C \inf_{|(y,z)|<1/2} \tilde v(y,z),$$
from which~\eqref{PRE:eq:odd:harnack} plainly follows.

The strategy that we follow in this note is precisely to adapt this method to more general settings of nonlocal type.
In this situation, additional symmetry structures have to be imposed on the solution, due to the nonlocal features of the operator.
Moreover, the analog of~\eqref{TBCW} requires a series of more subtle arguments, since cylindrical coordinates are typically not easy to handle in a nonlocal framework, due to remote point interactions: this difficulty will be overcome by exploiting 
some classical tools from harmonic analysis, as described in detail in the forthcoming Section~\ref{sec:bochner}.

\section{Bochner's relation and its consequences}
\label{sec:bochner}

\subsection{Notation}

Let~$\ell\in\N$.
As usual, a {\em solid harmonic polynomial} of degree $\ell$ is a homogeneous polynomial $V$ on $\R^n$ of degree $\ell$ (that is, $V(r x) = r^\ell V(x)$) satisfying the Laplace equation $\Delta V = 0$.

Below, we will also use symbols with a tilde to denote elements of $\R^{n + 2 \ell}$ and functions on $\R^{n + 2 \ell}$, while symbols without a tilde for vectors in~$\R^n$ and functions on~$\R^n$. That is, for typographical convenience, we will usually (i.e., unless differently specified) write
$x\in\R^n$ and $\tilde x\in\R^{n + 2 \ell}$, and similarly~$f:\R^n\to\R$ and~$\tilde{f}:\R^{n+2\ell}\to\R$.

We use $\fourier_n$ to denote the $n$-dimensional {\em Fourier transform}: if $f$ is an integrable function on $\R^n$, then
\[
 \fourier_n f(\xi) = \int_{\R^n} f(x) e^{2\pi i \xi\cdot x}\, dx .
\]

When we want to emphasize that we are taking the Fourier transform with respect to the variable~$x$
we also use the notation~$ \fourier_n^{(x)}$. For instance, if~$f:\R^{2n}=\R^n\times\R^n\to\R$, to be written as~$f(x,y)$, we have that, for each~$y\in\R^n$,
$$  \fourier_n^{(x)} f(\xi,y) = \int_{\R^n} f(x,y) e^{2\pi i \xi\cdot x} \,dx .$$
We are following here the
convention on the Fourier transform
on page 28 in~\cite{MR0290095}.

We also consider $L_n$ to be a Fourier multiplier on $\R^n$ with radial symbol $\psi(|\xi|)$, that is
\begin{equation}\label{FOU:SI}
 \fourier_n L_n f(\xi)  = \psi(|\xi|) \fourier_n f(\xi)
\end{equation}
for every Schwartz function $f$.

The case in which $\psi(r) = r^{2 s}$ corresponds to the fractional Laplacian $L_n = (-\Delta)^s$, but the contents of this section work for general $\psi$ with at most polynomial growth.

Let $\mathcal R$ be a rotation on $\R^n$ (that is, an orthogonal transformation of $\R^n$), and for a function $f$ on $\R^n$ denote $\mathcal R f(\xi) = f(\mathcal R \xi)$. By definition, we have the following transformation rule for the Fourier transform:
\[
 \fourier_n (\mathcal R f)(\xi) = \mathcal R (\fourier_n f)(\xi) .
\]
Radial functions are invariant under $\mathcal R$, and so
\[
 \psi(|\xi|) \fourier_n (\mathcal R f)(\xi) = \mathcal R \bigl[\psi(|\xi|) \fourier_n f(\xi)\bigr] .
\]
The inverse Fourier transform satisfies a similar transformation rule, so that
\[
 \fourier_n^{-1} \bigl[\psi(|\xi|) \fourier_n \mathcal R f(\xi)\bigr] = \mathcal R \bigl[ \fourier_n^{-1} \bigl[ \psi(|\xi|) \fourier_n f(\xi) \bigr]\bigr]
\]
whenever $f$ is, say, a Schwartz function. In view of~\eqref{FOU:SI}, this means that
\begin{equation}
\label{eq:commute}
 L_n \mathcal R f(\xi) = \mathcal R L_n f(\xi) ,
\end{equation}
that is, $L_n$ commutes with rotations.

Let us call a function $f$ on $\R^n$ isotropic with respect to the first $k$ coordinates, or briefly \emph{$k$-isotropic}, if $f(x) = f(y)$ whenever
\[
 |(x_1, x_2, \ldots, x_k)| = |(y_1, y_2, \ldots, y_k)| \quad \text{and} \quad (x_{k+1}, x_{k+2}, \ldots, x_n) = (y_{k+1}, y_{k+2}, \ldots, y_n) .
\]
We will only use this notion for $k = 3$: the proof of our main result involves $3$-isotropic functions in $\R^{n+2}$.

A function $f$ on $\R^n$ is said to be \emph{symmetric} (with respect to the first variable) if~$f(-x_1, x_2, \ldots, x_n) = f(x_1, x_2, \ldots, x_n)$ (that is, if~$f$ is even with respect to the first variable; this coincides with the notion of a $1$-isotropic function). Similarly, $f$ is said to be \emph{antisymmetric} (with respect to the first variable) if~$f(-x_1, x_2, \ldots, x_n) = -f(x_1, x_2, \ldots, x_n)$. Clearly, $f$ is symmetric if and only if $x_1 f(x)$ is antisymmetric.

If $f$ is a $k$-isotropic function on $\R^n$, then $f$ is invariant under every rotation $\mathcal R$ which only acts on the first $k$ coordinates and leaves the remaining $n - k$ coordinates unchanged: $\mathcal R f = f$. In view of~\eqref{eq:commute}, we have $\mathcal R(L_n f) = L_n(\mathcal R f) = L_n f$, that is, $L_n f$ is invariant under $\mathcal R$. In other words, $L_n$ maps $k$-isotropic Schwartz functions into $k$-isotropic functions.

\subsection{Fourier transforms of radial functions} Now
we recall the tools from harmonic analysis which will come in handy for the development of the theory.
The gist is the link of the Fourier transform of an $n$-dimensional radial function (multiplied by an appropriate
harmonic polynomial) with that
of its $(n+2\ell)$-dimensional counterpart.

For instance, roughly speaking, for any even function~$f$ on~$\R$, if~$\tilde f:\R^3\to\R$ has the same profile of~$f$ (i.e., $\tilde f(\tilde x)=f(|\tilde x|)$ for each~$\tilde x\in\R^3$), then
the Fourier transform of the one-dimensional function~$x f(x)$ coincides with~$i\xi$ times
the Fourier transform of the three-dimensional radial function~$\tilde f$ at~$\tilde\xi$, as long as~$\tilde\xi\in\R^3$
is such that~$|\xi|=|\tilde\xi|$.

This fact holds true in higher dimensions as well, with suitable modifications,
and goes under the name of Bochner's relation (or Hecke--Bochner identity). This will be our crucial tool and we now recall its precise statement:

\begin{thm}[Bochner's relation; Corollary on page 72 in~\cite{MR0290095}]
\label{thm:bochner}
Let $f$ and $\tilde{f}$ be two radial Schwartz functions on $\R^n$ and $\R^{n + 2 \ell}$, respectively, with the same profile function: $f(x) = \tilde{f}(\tilde{x})$ whenever $|x| = |\tilde x|$.

Let $V$ be a solid harmonic polynomial on $\R^n$ of degree $\ell$.

Then,
\[
 \fourier_n (V f)(\xi) = i^\ell V(\xi) \fourier_{n + 2 \ell} \tilde{f}(\tilde{\xi}) ,
\]
for every~$\xi\in\R^n$ and~$\tilde\xi\in\R^{n+2\ell}$ such that $|\tilde{\xi}| = |\xi|$.
\end{thm}

In the statement of Theorem~\ref{thm:bochner}, the assumption on the profile function of~$f$ and~$\tilde f$ simply means that
there exists a suitable~$f_0:\R\to\R$ such that~$f(x)=f_0(|x|)$ and~$\tilde f(\tilde x)=f_0(|\tilde x|)$ for all~$x\in\R^n$ and~$\tilde x\in\R^{n+2\ell}$ (the function~$f_0$ is called in jargon ``profile function'').

In our setting, we will only need Theorem~\ref{thm:bochner} for $n= \ell = 1$ and $V(x) = x$. In this case Theorem~\ref{thm:bochner} states that if $f$ is an even function on $\R$ and $\tilde{f}(\tilde{x}) = f(|x|)$ is the corresponding radial function on $\R^3$, then
\begin{equation}
\label{eq:bochner:odd}
 \fourier_1[x f(x)](\xi) = i \xi \fourier_3 \tilde{f}(\xi, 0, 0) .
\end{equation}

We have the following immediate extension of Theorem~\ref{thm:bochner}.

\begin{lem}
\label{lem:odd:fourier}
Consider a function~$f:\R^n=\R\times\R^{n-1}\to\R$,
to be written as~$f(x, y)$, where $x \in \R$ and $y \in \R^{n - 1}$.

Assume that~$f$ is a Schwartz function on $\R^n$ which is symmetric with respect to the variable $x$ (i.e., $f(-x, y) = f(x, y)$
for all~$x \in \R$ and $y \in \R^{n - 1}$).

Let~$\tilde f:\R^{n+2}=\R^3\times\R^{n-1}\to\R$,
to be written as~$\tilde{f}(\tilde{x}, y)$, where $\tilde{x} \in \R^3$ and $y \in \R^{n - 1}$, be a $3$-isotropic function given by
\[
 \tilde{f}(\tilde{x}, y) = f(|\tilde{x}|, y) .
\]
Then,
\[
 \fourier_n[x f(x, y)](\xi, \eta) = i \xi \fourier_{n + 2} \tilde{f}((\xi, 0, 0), \eta) . \qedhere
\]
\end{lem}

\begin{proof}
Given $y \in \R^{n - 1}$, we apply Bochner's relation~\eqref{eq:bochner:odd} with respect to $x \in \R$ to find that
\[
 \fourier_1^{(x)}[x f(x, y)](\xi, y) = i \xi \fourier_3^{(\tilde{x})} \tilde{f}(\tilde{\xi}, y) .
\]
It is now sufficient to apply the Fourier transform with respect to $y \in \R^{n - 1}$ for each fixed $\xi \in \R$.
\end{proof}

From Lemma~\ref{lem:odd:fourier} we obtain the following result. 

\begin{lem}
\label{lem:odd:operator}
In the assumptions of Lemma~\ref{lem:odd:fourier}, for every~$x\in\R$ and~$y\in\R^{n-1}$ we have that
\[
 L_n[x f(x, y)] = x L_{n + 2} \tilde{f}((x, 0, 0), y) .
\]
\end{lem}

\begin{proof}
Let us denote
\[
 \tilde{g}(\tilde{x}, y) = L_{n + 2} \tilde{f}(\tilde{x}, y)
\]
and
\[
 g(x, y) = \tilde{g}((x, 0, 0), y) = L_{n +2} \tilde{f}((x, 0, 0), y) .
\]
Since $L_{n + 2}$ maps $3$-invariant functions into $3$-invariant functions, $\tilde{g}$ is $3$-invariant.

Let now~$(\xi,\eta)\in\R\times\R^{n-1}$ and let~$\tilde\xi\in\R^3$ be such that~$|\xi| = |\tilde{\xi}|$. Using, in order, the definition of $L_n$
in~\eqref{FOU:SI}, Lemma~\ref{lem:odd:fourier}, once again the definition of $L_{n + 2}$ in~\eqref{FOU:SI}, the definition of~$\tilde{g}$, and once again Lemma~\ref{lem:odd:fourier}, we find that
\begin{align*}
 \fourier_n L_n[x f(x, y)](\xi, \eta) & = \psi(|(\xi, \eta)|)\;\fourier_n[x f(x, y)](\xi, \eta) \\
 & = \psi(|(\xi, \eta)|) \; i \xi \fourier_{n + 2} \tilde{f}(\tilde{\xi}, \eta) \\
 & = i \xi \; \psi(|(\tilde \xi, \eta)|) \; \fourier_{n + 2} \tilde{f}(\tilde{\xi}, \eta) \\
 & = i \xi \fourier_{n + 2} L_{n + 2} \tilde{f}(\tilde{\xi}, \eta) \\
 & = i \xi \fourier_{n + 2} \tilde{g}(\tilde{\xi}, \eta) \\
 & = \fourier_n[x g(x, y)](\xi, \eta) .
\end{align*}
Thus, by inverting the Fourier transform,
\[ L_n[x f(x, y)]=x g(x, y). \]
Combining this and the definition of $g$, we conclude that
\begin{equation*}
L_n[x f(x, y)]  
=x L_{n + 2} \tilde{f}((x, 0, 0), y),
\end{equation*}
as desired.
\end{proof}

In the same vein, one shows the following more general result (we omit the proof since it can be obtained via the same argument as above, by utilizing the general identity in Theorem~\ref{thm:bochner} instead of the ones specialized for the case~$V(x)=x$).

\begin{lem}
\label{lem:solid:operator}
Let $V$ be a solid harmonic polynomial of degree $\ell$ on $\R^k$. Let $f(x, y)$, where $x \in \R^k$ and $y \in \R^{n - k}$, be a Schwartz function on $\R^n$ which is isotropic with respect to the first variable $x$: $f(x, y) = f(x', y)$ whenever $|x| = |x'|$. Let $\tilde{f}(\tilde{x}, y)$, where $\tilde{x} \in \R^{k + 2 \ell}$ and $y \in \R^{n - k}$, be the corresponding Schwartz function on $\R^{n + 2 \ell}$, given by
\[
 \tilde{f}(\tilde{x}, y) = f(x, y)
\]
whenever $|\tilde{x}| = |x|$.

Then
\[
 L_n[V(x) f(x, y)] = V(x) L_{n + 2 \ell} \tilde{f}((x, 0), y) ,
\]
where $(x, 0) \in \R^{k + 2 \ell}$ stands for the vector $x \in \R^k$ padded with $2 \ell$ zeroes.
\end{lem}

%
%

\section{Fractional Laplacians}
\label{sec:flap}

\subsection{Notation}\label{sec:notation}
Now we
specialize the result of Section~\ref{sec:bochner}
to the case of the fractional Laplacian of order~$s\in(0,1)$, corresponding to the choice~$\psi(r) = r^{2 s}$ in Lemma~\ref{lem:solid:operator}. To this end, we introduce some further notation.

If~$f$ and~$g$ are functions on~$\R^n$, we write
$$ \langle f,g\rangle_{\R^n}=\int_{\R^n}f(x)g(x)\,dx.$$

We use the notation~$f\in C^\alpha$ for non-integer~$\alpha>0$, meaning
that~$f\in C^{k,\beta}$ where~$\alpha=k+\beta$ with~$k\in\N$ and~$\beta\in(0,1)$.

\subsection{Fractional Laplacian identities and a Harnack inequality}\label{subsec:fl}
We now translate the results of Section~\ref{sec:bochner} into the notation that we have just introduced.

\begin{cor}
\label{cor:odd:flap}
Let $f$ be a 
Schwartz function on $\R^n$, symmetric with respect to the first variable. Let $\tilde{f}$ be the corresponding 3-isotropic Schwartz function on $\R^{n + 2}$, given by
\[
 \tilde{f}(\tilde{x}) = f\bigg (\sqrt{\tilde{x}_1^2 + \tilde{x}_2^2 + \tilde{x}_3^2}, \tilde{x}_4, \ldots, \tilde{x}_{n + 2} \bigg ) .
\]
Finally, let $g(x) = x_1 f(x)$. Then
\[
 (-\Delta)^s g(x) = x_1 (-\Delta)^s \tilde{f}((x_1, 0, 0), x_2, \ldots, x_n) .
\]
\end{cor}

\begin{proof} This follows directly from Lemma~\ref{lem:odd:operator}: indeed,
choosing~$\psi(r) = r^{2 s}$ in~\eqref{FOU:SI} gives that~$L_n$ is the fractional Laplacian
acting on functions of~$n$ variables and accordingly, for every~$x=(x_1,x_2,\dots,x_n)\in\R^n$,
\begin{equation*}
(-\Delta)^sg(x)=(-\Delta)^s[x_1 f(x)]=x_1(-\Delta)^s
\tilde f((x_1,0,0),x_2,\dots,x_n)
.\qedhere
\end{equation*}
\end{proof}

\begin{remark}
The identity in Corollary~\ref{cor:odd:flap} can be alternatively proved by a rather straightforward direct calculation. This involves the formula
\begin{align*}
 & r^2 \int_{\sph^2} (\alpha^2 + (r z_1 - \beta)^2 + (r z_2)^2 + (r z_3)^2)^{-1 - \gamma} \, \mathcal H^2 (dz) \\
 & \qquad = \frac{1}{2 \beta \gamma} \, \bigl( (\alpha^2 + (r - \beta)^2)^{-\gamma} - (\alpha^2 + (r + \beta)^2)^{-\gamma} \bigr) ,
\end{align*}
where $\sph^2$ is the unit sphere in $\R^3$ and $\mathcal H^2$ is the surface measure. One substitutes $\alpha = \sqrt{(x_2 - y_2)^2 + \ldots + (x_n - y_n)^2}$, $\beta = x_1$, $r = |y_1|$, and $\gamma = \tfrac{n}{2} + s$ and then integrates with respect to $y \in \R^n$; we omit the details.
\end{remark}

An extension of Corollary~\ref{cor:odd:flap} to arbitrary antisymmetric functions is rather straightforward: in particular, in item~\ref{thm:odd:flap:c} below we consider the extension of the fractional Laplace operator on the class of antisymmetric functions, as described in~\cite{MR4567494}.

\begin{thm}
\label{thm:odd:flap}
Suppose that $\Omega$ is an open set of~$\R^n$ and $\delta>0$. 
Let $u$ be a function on $\R^n$ such that $u$ is~$C^{2s + \delta}$ in~$\Omega$ and
antisymmetric, with
\begin{equation*}
\int_{\R^n}\frac{|x_1|\,|  u(x)|}{(1 + |x|)^{n + 2 +2s}} \,dx<+\infty.
\end{equation*}
Let $\tilde{v}$ be the corresponding function on $\R^{n + 2}$ given by
\[
 \tilde{v}(\tilde{x}) = \frac{u\big (\sqrt{\tilde{x}_1^2 + \tilde{x}_2^2 + \tilde{x}_3^2}, \tilde{x}_4, \ldots, \tilde{x}_{n + 2} \big )}{\sqrt{\tilde{x}_1^2 + \tilde{x}_2^2 + \tilde{x}_3^2}} \, .
\]
Then:
\begin{enumerate}[{(a)}]
\item\label{thm:odd:flap:a} if $\Omega_0 = \{x \in \Omega : x_1 \ne 0\}$, then the function $\tilde{v}$ is $C^{2s + \delta}_{\mathrm{loc}}$ in
\[
 \tilde \Omega_0 = \bigg \{\tilde{x} \in \R^{n + 2} : \Big (\sqrt{\tilde{x}_1^2 + \tilde{x}_2^2 + \tilde{x}_3^2}, \tilde{x}_4, \ldots, \tilde{x}_{n + 2}\Big ) \in \Omega_0\bigg \} ;
\]
\item\label{thm:odd:flap:b} we have that
\begin{equation*}
\int_{\R^{n+2}}\frac{|\tilde{v}(\tilde{x})|}{
(1 + |\tilde{x}|)^{n + 2 + 2s}}\,d\tilde x<+\infty;\end{equation*}
\item\label{thm:odd:flap:c} for every $x \in \Omega_0$, we have that
\begin{equation}
\label{eq:odd:flap}
 (-\Delta)^s u(x) = x_1 (-\Delta)^s \tilde{v}((x_1, 0, 0), x_2, \ldots, x_n) ;
\end{equation}
\item\label{thm:odd:flap:d} for an arbitrary $C_c^\infty$ function $g$ on $\Omega$, we have
\begin{equation}
\label{eq:odd:flap:weak}
 \langle u, (-\Delta)^s g \rangle_{\R^n} = (2 \pi)^{-1} \langle \tilde{v}, (-\Delta)^s \tilde{f} \rangle_{\R^{n + 2}} ,
\end{equation}
where
\[
 \tilde f(\tilde x)= \frac{ g_A\big (\sqrt{\tilde{x}_1^2 + \tilde{x}_2^2 + \tilde{x}_3^2}, \tilde{x}_4, \ldots, \tilde{x}_{n + 2} \big )}{\sqrt{\tilde{x}_1^2 + \tilde{x}_2^2 + \tilde{x}_3^2}}
\]
and
\[
g_A(x) = \frac 12 \big (  g(x_1,\dots,x_n) - g(-x_1,x_2,\dots,x_n) \big ) 
\]
is the anti-symmetric part of $g$;
\item\label{thm:odd:flap:e} for an arbitrary $C_c^\infty$ function $\tilde{f}$ on 
\[\tilde \Omega=\bigg \{\tilde{x} \in \R^{n + 2} : \Big (\sqrt{\tilde{x}_1^2 + \tilde{x}_2^2 + \tilde{x}_3^2}, \tilde{x}_4, \ldots, \tilde{x}_{n + 2}\Big ) \in \Omega\bigg \} ,
\]
we have
\begin{equation}
\label{eq:odd:flap:reverse}
\langle \tilde{v}, (-\Delta)^s \tilde{f} \rangle_{\R^{n + 2}} = 2 \pi \langle u, (-\Delta)^s g \rangle_{\R^n} ,
\end{equation}
where $g(x) = x_1 f(x)$ and $f$ is given in terms of the 3-isotropic symmetrization of $\tilde{f}$ by
\begin{equation}\label{ISODEF}
 f(x) = \frac{1}{4 \pi} \int_{\sph^2} \tilde{f}(|x_1| z, x_2, \ldots, x_n)\, \mathcal H^2(dz) ;
\end{equation}
here $\sph^2$ is the unit sphere in $\R^3$ and $\mathcal H^2$ is the surface measure.
\end{enumerate}
\end{thm}

\begin{remark}
Note that, in~Theorem~\ref{thm:odd:flap}~\ref{thm:odd:flap:a}, the conclusion $\tilde v\in C^{2s+\delta}_{\mathrm{loc}}(\tilde \Omega_0)$ cannot be relaxed to $\tilde v\in C^{2s+\delta}(\tilde \Omega_0)$. Indeed, consider $\Omega = (0,1) \times \R^{n-1}$ and $u(x) = x_1^{2s+\delta}$ for $x_1\geq 0$ extended antisymmetrically to all of $\R^n$. Then $\Omega_0=\Omega$, $\tilde \Omega_0 =\big (  B_1^3\setminus \{0\}\big ) \times \R^{n-1}$ (here $B^3_1$ denotes the unit ball in $\R^3$), and $u\in C^{2s+\delta}(\Omega_0)$, but \begin{align*}
\tilde v(\tilde x ) = \big ( \tilde x_1^2+\tilde x_2^2+\tilde x_3^2 \big ) ^{\frac{2s+\delta-1} 2 } 
\end{align*} which is not in $C^{2s+\delta} (\tilde \Omega_0)$. 
\end{remark}

\begin{proof}[Proof of Theorem~\ref{thm:odd:flap}]
To prove assertion~\ref{thm:odd:flap:a}, observe that, for all $\tilde x \in \tilde \Omega_0$, \begin{align*}
\sqrt{\tilde x_1^2+\tilde x_2^2+\tilde x_3^2} 
\neq 0 .
\end{align*} Hence, the map from $\tilde \Omega_0 \to \R $ defined by $\tilde x \mapsto \frac 1 {\sqrt{\tilde x_1^2+\tilde x_2^2+\tilde x_3^2}}$, is in $C^\infty_{\mathrm{loc} } (\tilde \Omega_0)$. Moreover, 
\[
\tilde x \mapsto u\Big (\sqrt{\tilde{x}_1^2 + \tilde{x}_2^2 + \tilde{x}_3^2}, \tilde{x}_4, \ldots, \tilde{x}_{n + 2}\Big )
\]
is in $C^{2s+\delta}(\tilde \Omega_0)$ which implies $\tilde v\in C^{2s+\delta}_{\mathrm{loc}}(\tilde \Omega_0)$. This proves~\ref{thm:odd:flap:a}.

Additionally,
\begin{equation}\label{4YintRdt1}\begin{split}&{\mathcal{I}}:=
\int_{\R^{n+2}} \frac{|\tilde{v}(\tilde{x})|}{(1 + |\tilde{x}|)^{n + 2 + 2s} }\, d\tilde x=
\int_{\R^{n+2}}
\frac{ \Bigl|u\Bigl(\sqrt{\tilde{x}_1^2 + \tilde{x}_2^2 + \tilde{x}_3^2}, \tilde{x}_4, \ldots, \tilde{x}_{n + 2}\Bigr)\Bigr|}{\sqrt{\tilde{x}_1^2 + \tilde{x}_2^2 + \tilde{x}_3^2} \;\big(1 + |\tilde{x}|\big)^{n + 2 + 2s}} \,d\tilde x\\& \qquad
= \int_{\R^3} \int_{\R^{n-1}} \frac{ \bigl|u\bigl( |a|,b\bigr)\bigr|}{|a|\;\big(1 + \sqrt{|a|^2+|b|^2}\big)^{n + 2 + 2s}} \,da\, db
,\end{split}\end{equation}
Now we observe that if~$\Phi:\R\to\R$ then, integrating in spherical coordinates, we get
\begin{equation}\label{UT:01}
\int_{\R^3} \Phi(|x|)\,dx = 4\pi\int_0^{+\infty}\tau^2 \Phi(\tau) \,d\tau.
\end{equation}
Now, gathering~\eqref{4YintRdt1} and~\eqref{UT:01}, we see that
\[
{\mathcal{I}}
= \int_0^{+\infty} \int_{\R^{n-1}}\frac{\tau \, |u(\tau,b)|}{\big(1 + \sqrt{|\tau|^2+|b|^2}\big)^{n + 2 + 2s}}\, db\, d\tau 
= \int_{ \R^{n}} \frac{|x_1|\, |u(x)|}{ (1 + |x|)^{ n + 2 + 2s}}\, dx ,
\]
from which~\ref{thm:odd:flap:b} plainly follows.

To prove the claim in~\ref{thm:odd:flap:c}, we notice that if $f(x) = x_1^{-1} u(x)$ is a Schwartz function, item~\ref{thm:odd:flap:c} reduces to Corollary~\ref{cor:odd:flap}. The extension to general $u$ follows by mollification and cutting off.

We now prove assertions~\ref{thm:odd:flap:d} and~\ref{thm:odd:flap:e}. For~\ref{thm:odd:flap:d}, since $\tilde f$ is $3$-isotropic, also $(-\Delta)^s \tilde f$ is $3$-isotropic:
\begin{align*}
(-\Delta)^s \tilde f (\tilde x) &= (-\Delta)^s \tilde f \bigg ( \Big (\sqrt{\tilde x_1^2 + \tilde x_2^2 + \tilde x_3^2} , 0 , 0 \Big ),\tilde x_4, \dots ,\tilde x_{n+2} \bigg )
, 
\end{align*} so Corollary~\ref{cor:odd:flap} implies that \begin{align*}
(-\Delta)^s \tilde f (\tilde x) &= \frac{ (-\Delta)^s g_A \big (\sqrt{\tilde x_1^2 + \tilde x_2^2 + \tilde x_3^2}, \tilde x_4, \ldots, \tilde x_{n+2} \big )} {\sqrt{\tilde x_1^2 + \tilde x_2^2 + \tilde x_3^2}} .
\end{align*} Hence, using the definition of \(\tilde v\), we have  \begin{align*}
& \hspace{-1em} \langle \tilde v , (-\Delta)^s \tilde f \rangle_{\R^{n+2}} \\&= \int_{\R^{n+2}} \frac{u \big (\sqrt{\tilde x_1^2 + \tilde x_2^2 + \tilde x_3^2}, \tilde x_4, \ldots, \tilde x_{n+2} \big ) (-\Delta)^s g_A \big (\sqrt{\tilde x_1^2 + \tilde x_2^2 + \tilde x_3^2}, \tilde x_4, \ldots, \tilde x_{n+2} \big )}{\tilde x _1 ^2+\tilde x _2 ^2+\tilde x _3 ^2 }\, d \tilde x \\
& = \int_{\R^3} \int_{\R^{n-1}} \frac{u(|a|, b) (-\Delta)^s g_A(|a|, b)}{|a|^2}\, db \, da.
\end{align*}
Integration in spherical coordinates leads to
\begin{align*}
\langle \tilde v , (-\Delta)^s \tilde f \rangle_{\R^{n+2}}
&= 4 \pi \int_0^{+\infty} \int_{ \R^{n-1}}  u (\tau,b) (-\Delta)^s g_A (\tau,b) \, d b\,d\tau. 
\end{align*} Then, since both \(u\) and \((-\Delta)^s g_A\) are antisymmetric, \begin{align*}
\langle \tilde v , (-\Delta)^s \tilde f \rangle_{\R^{n+2}} &= 2 \pi \int_{-\infty}^{+\infty}\int_{ \R^{n-1} }  u (\tau, b) (-\Delta)^s g_A (\tau, b) \, d b\,d\tau \\
&= 2 \pi \langle u , (-\Delta)^s g \rangle_{\R^n}. 
\end{align*}

For~\ref{thm:odd:flap:e}, we begin by assuming that~\(\tilde f\) is \(3\)-isotropic and deal with the general case later. Since \(\tilde f\) is \(3\)-isotropic, Corollary~\ref{cor:odd:flap} applies to $\tilde f$, $f$ and $g$: we have
\[
 (-\Delta)^s \tilde{f}(\tilde x) = \frac{(-\Delta)^s g\big (\sqrt{\tilde x_1^2 + \tilde x_2^2 + \tilde x_3^2}, \tilde x_4, \ldots, \tilde x_{n+2} \big )}{\sqrt{\tilde x_1^2 + \tilde x_2^2 + \tilde x_3^2}} \, .
\]
Using spherical coordinates,
\begin{align*}
    &\hspace{-2em}\langle \tilde v , (-\Delta)^s \tilde f \rangle_{\R^{n+2}} \\
    &= \int_{\R^{n+2}} \frac{u \big ( \sqrt{\tilde x_1^2+\tilde x_2^2+\tilde x_3^2},\tilde x_4, \ldots, \tilde x_{n+2}\big ) \, (-\Delta)^s g\big (\sqrt{\tilde x_1^2 + \tilde x_2^2 + \tilde x_3^2}, \tilde x_4, \ldots, \tilde x_{n+2} \big ) }{\tilde x_1^2+\tilde x_2^2+\tilde x_3^2} \, d \tilde x \\
    & = \int_{\R^3} \int_{\R^{n-1}} \frac{u(|a|, b) (-\Delta)^s g(|a|, b)}{|a|^2} \, d b\, d a \\
    &= 4\pi \int_0^{+\infty} \int_{\R^{n-1}} u(\tau, b) (-\Delta)^s g(\tau, b)\,  d b\, d\tau.
\end{align*}
Since both $u$ and $(-\Delta)^s g$ are symmetric, we have
\[
 \langle \tilde v, (-\Delta)^s \tilde f \rangle_{\R^{n+2}} = 2 \pi \int_{-\infty}^{+\infty} \int_{\R^{n-1}} u(\tau, b) (-\Delta)^s g(\tau, b) \, db \, d\tau = 2 \pi \langle u, (-\Delta)^s g \rangle_{\R^n} .
\]

Now, if \(\tilde f\) is not \(3\)-isotropic, let
\[
 \tilde F(\tilde x) = \frac{1}{4 \pi} \int_{\sph^2} \tilde{f}\Big (z \sqrt{\tilde x_1^2 + \tilde x_2^2 + \tilde x_3^2}, \tilde x_4, \ldots, \tilde x_{n+2} \Big ) \, \mathcal H^2(dz)
\]
be the $3$-isotropic projection of $\tilde f$. By the definition of $f$,
\[
 \tilde F(\tilde x) = f\Big (\sqrt{\tilde x_1^2 + \tilde x_2^2 + \tilde x_3^2}, \tilde x_4, \ldots, \tilde x_{n+2} \Big ) ,
\]
and so, since we already proved the result for $3$-isotropic functions, we have
\[
 \langle \tilde v, (-\Delta)^s \tilde F \rangle_{\R^{n+2}} = 2 \pi \langle u, (-\Delta)^s g \rangle_{\R^n} .
\]
Therefore, it remains to prove that
\begin{equation}
\label{eq:orthogonal}
 \langle \tilde v, (-\Delta)^s (\tilde f - \tilde F) \rangle_{\R^{n+2}} = 0 .
\end{equation}
Denote $\tilde G = \tilde f - \tilde F$. Then, by the definition of $\tilde F$, 
\[
 \int_{\sph^2} \tilde G(\tau z, \tilde x_4, \ldots, \tilde x_{n+2}) \, \mathcal H^2(dz) = 0
\]
for every $\tau > 0$ and $(\tilde x_4, \ldots, \tilde x_{n+2}) \in \R^{n-1}$. Thus, $\tilde G$ is orthogonal to every $3$-isotropic function: if $\tilde u$ is $3$-isotropic, then integration in spherical coordinates gives
\begin{align*}
 \langle \tilde u, \tilde G\rangle & = \int_{\R^3} \int_{\R^{n-1}} \tilde u(a, b) \tilde G(a, b) \, db \, da \\
 & = \int_{\R^3} \int_{\R^{n-1}} \tilde u((|a|, 0, 0), b) \tilde G(a, b) \, db \, da \\
 & = \int_0^{+\infty} \int_{\sph^2} \int_{\R^{n-1}} \tau^2 \tilde u((\tau, 0, 0), b) \tilde G(\tau z, b) \, db \, \mathcal H^2(dz) \, d\tau \\
 & = \int_0^{+\infty} \int_{\R^{n-1}} \tau^2 \tilde u((\tau, 0, 0), b) \biggr(\int_{\sph^2} \tilde G(\tau z, b) \, \mathcal H^2(dz)\biggr) db \, d\tau = 0 .
\end{align*}
We already know that $(-\Delta)^s$ maps $3$-isotropic Schwartz functions into $3$-isotropic functions. By approximation, the class of $3$-isotropic functions in $L^2(\R^{n+2})$ forms an invariant subspace of $(-\Delta)^s$. Therefore, also its orthogonal complement in $L^2(\R^n)$ is invariant. Since $\tilde G$ is orthogonal to the class of $3$-isotropic functions, we conclude that $(-\Delta)^s \tilde G$ has the same property, and~\eqref{eq:orthogonal} follows.
\end{proof}

In our framework, a natural
application\footnote{When~$s=1$, Theorem~\ref{thm:odd:harnack}
boils down to~\eqref{PRE:eq:odd:harnack}.} of Theorem~\ref{thm:odd:flap} 
leads to the proof of Theorem~\ref{thm:odd:harnack}:

\begin{proof}[Proof of Theorem~\ref{thm:odd:harnack}]
With the notation of Theorem~\ref{thm:odd:flap}, by item~\ref{thm:odd:flap:c} of this result we have
\[
 (-\Delta)^s u(x) = x_1 (-\Delta)^s \tilde{v}((x_1, 0, 0), x_2, \ldots, x_n)
\]
for $x \in \Omega_0$, or equivalently\footnote{One can compare~\eqref{TBCW}
with~\eqref{TBCW2}. Namely, \eqref{TBCW2} recovers~\eqref{TBCW} when~$s=1$ with the notation~$y=(\tilde x_1,\tilde x_2,\tilde x_3)$ and
$z=(\tilde x_4,\dots,\tilde x_{n+2})$.}
\begin{equation}
\label{TBCW2}
 \sqrt{\tilde{x}_1^2 + \tilde{x}_2^2 + \tilde{x}_3^2} \, (-\Delta)^s \tilde{v}(\tilde{x}) = (-\Delta)^s u \Big (\sqrt{\tilde x_1^2 + \tilde x_2^2 + \tilde x_3^2}, \tilde x_4, \ldots, \tilde x_{n+2} \Big ) ,
\end{equation}
for $\tilde{x} \in \tilde \Omega_0$ (recall that $(-\Delta)^3 \tilde v$ is $3$-isotropic). 

Therefore, if $x = \big  (\sqrt{\tilde x_1^2 + \tilde x_2^2 + \tilde x_3^2}, \tilde x_4, \ldots, \tilde x_{n+2} \big )$, we have
\[
 \sqrt{\tilde{x}_1^2 + \tilde{x}_2^2 + \tilde{x}_3^2} \Bigl( (-\Delta)^s \tilde{v}(\tilde{x}) + c(x) \tilde{v}(\tilde{x}) \Bigr) = (-\Delta)^s u(x) + c(x) u(x) = 0 ,
\]
that is, setting $\tilde{c}(\tilde{x}) = c(x)$,
\begin{equation}
\label{eq:schroedinger}
 (-\Delta)^s \tilde{v}(\tilde{x}) + \tilde{c}(\tilde{x}) \tilde{v}(\tilde{x}) = 0
\end{equation}
in $\tilde \Omega_0$. In fact, the above identity holds in the weak sense in $\tilde \Omega$: this follows directly from item~\ref{thm:odd:flap:e} of Theorem~\ref{thm:odd:flap}.

Since $\tilde{v}$ is nonnegative in $\R^{n + 2}$, the standard Harnack inequality for weak solutions of the Schrödinger equation~\eqref{eq:schroedinger} implies that for every compact subset $\tilde{K}$ of $\tilde \Omega$, we have
\[
 \esssup_{\tilde{K}} \tilde{v} \le C(\tilde K, \tilde \Omega, \|\tilde c\|_{L^\infty(\tilde \Omega})) \essinf_{\tilde{K}} \tilde{v} .
\]
Applying this to $\tilde{K} = \big \{\tilde x \in \R^{n+2} : \big (\sqrt{\tilde x_1^2 + \tilde x_2^2 + \tilde x_3^2}, \tilde x_4, \ldots, \tilde x_{n+2} \big ) \in K\big \}$ with $K$ a compact subset of $\Omega$, we obtain
\[
 \esssup_{x \in K} \frac{u(x)}{x_1} \le C(K, \Omega, \|c\|_{L^\infty(\Omega_+)}) \essinf_{x \in K} \frac{u(x)}{x_1} ,
\]
which is equivalent to~\eqref{eq:odd:harnack}.
\end{proof}

\section{More general nonlocal operators}\label{OJSDLN09ryihfjg9reSTRpsdlmvIJJ}
In this section, we generalise the results obtained in the specific setting of the fractional Laplacian to the more general class of operators $L_n$ presented in~\eqref{5XOlUbY7}. 

Given an open set \(\Omega \subset \R^n\) and a function in \(u\in C^2_{{\rm{loc}}}(\Omega) \cap L^\infty(\R^n)\), by a standard computation (see, for example, the computation in \cite[p.~9]{MR3469920}
), \(L_n u(x)\) is defined for all \(x\in \Omega\) and, moreover, $L_nu$ is bounded on compact subsets of $\Omega$
. Furthermore, the Fourier symbol of \(L_n\) is
\begin{align*}
 \Psi_n (\xi) & {} = \int_{\R^n} \big ( 1 - e^{i \xi \cdot x} + \xi \cdot x \chi_{(0,1)}(|x|) \big ) \, j_n(|x|) \, dx \\
 &= \int_0^\infty \int_{\sph^{n-1}} \big ( 1
 - e^{i r \xi \cdot \omega } + i r(\xi \cdot \omega ) \chi_{(0,1)}(r) \big ) j_n(r) r^{n - 1}  \,  \mathcal H^{n-1}(d\omega) \, dr \\
    &= \frac12 \int_0^\infty \int_{\sph^{n-1}} \big ( 2 - e^{i r \xi \cdot \omega } -e^{-i r \xi \cdot \omega }  \big ) j_n(r) r^{n - 1} \, \mathcal H^{n-1}(d\omega) \, dr \\
    &=\int_0^\infty \int_{\sph^{n-1}} \big ( 1 - \cos( r \xi \cdot \omega  ) \big ) j_n(r) r^{n - 1} \,  \mathcal H^{n-1}(d\omega) \, dr \\ 
    &= \mathcal H^{n-2}(\sph^{n-2}) \int_0^\infty \int_{-1}^1 \big ( 1 - \cos( r s |\xi| ) \big ) (1 - s^2)^{\frac{n - 3}{2}} j_n(r) r^{n - 1} \, ds \, dr \\
    & = (2 \pi)^{n/2} \int_0^\infty \biggl(\frac{1}{2^{n/2 - 1} \Gamma(\frac{n}{2})} - (r |\xi|)^{1 - n/2} J_{n/2 - 1}(r |\xi|) \biggr) r^{n - 1} j_n(r) \, dr ,
\end{align*}%
where $J_\nu$ is the Bessel function of the first kind; see~\cite[Chapter 10]{NIST:DLMF}. Hence, it follows that \(\Psi_n\) is rotationally invariant, that is \(\Psi_n(\xi)=\psi(\vert \xi \vert ) \), with $\psi$ given as the Hankel transform
\begin{align}
 \psi(\tau) = (2 \pi)^{n/2} \int_0^\infty \biggl(\frac{1}{2^{n/2 - 1} \Gamma(\frac{n}{2})} - (r \tau)^{1 - n/2} J_{n/2 - 1}(r \tau)\biggr) r^{n - 1} j_n(r) \, dr ,
\label{psidierre000}
\end{align} 
It is no mistake that $\psi$ has no subscript $n$: under our assumptions~\eqref{radialassumption} and~\eqref{nplus2} it turns out that the Fourier symbol $\Psi_{n + 2}$ of the operator $L_{n + 2}$ has the same profile function $\psi$. There are many different ways to prove this fact; for instance, it is a relatively simple consequence of Bochner's relation. A direct proof involving well-known properties of Bessel functions is perhaps even shorter: since $j_{n}(r) = 2 \pi \int_r^\infty t j_{n+2}(t) \, dt$, we have
\begin{align*}
 \psi(\tau) & = (2 \pi)^{n/2 + 1} \int_0^\infty \int_r^\infty \biggl(\frac{1}{2^{n/2 - 1} \Gamma(\frac{n}{2})} - (r \tau)^{1 - n/2} J_{n/2 - 1}(r \tau)\biggr) r^{n - 1} t j_{n + 2}(t) \, dt \, dr \\
 & = (2 \pi)^{n/2 + 1} \int_0^\infty \int_0^t \biggl(\frac{1}{2^{n/2 - 1} \Gamma(\frac{n}{2})} - (r \tau)^{1 - n/2} J_{n/2 - 1}(r \tau)\biggr) r^{n - 1} t j_{n + 2}(t) \, dr \, dt \\
 & = (2 \pi)^{n/2 + 1} \int_0^\infty  \biggl(\frac{1}{2^{n/2} \Gamma(\frac{n}{2} + 1)} - (t \tau)^{-n/2} J_{n/2}(t \tau)\biggr) t^{n + 1} j_{n + 2}(t) \, dt ,
\end{align*}
which is exactly the same expression as~\eqref{psidierre000} with $n$ replaced by $n + 2$; here we have used Fubini's theorem and~\cite[Equation 10.22.1]{NIST:DLMF}.

Therefore, Bochner's relation and its consequences proved in Section~\ref{sec:bochner} apply to this setting. The previously established theory will give us information about \(L_n\) given information about \(L_{n+2}\), namely that if \(L_{n+2}\) admits the Harnack inequality then \(L_n\) admits the antisymmetric Harnack inequality.


\begin{cor}
\label{cor:odd:flapBIS}
Let $f$ be a symmetric Schwartz function on $\R^n$. Let $\tilde{f}$ be the corresponding 3-isotropic Schwartz function on $\R^{n + 2}$, given by
\[
    \tilde{f}(\tilde{x}) = f\bigg (\sqrt{\tilde x_1^2 + \tilde x_2^2 + \tilde x_3^2}, \tilde x_4, \ldots, \tilde x_{n+2}\bigg ) .
\]
Finally, let $g(x) = x_1 f(x)$. Then
\[
    L_n g(x) = x_1 L_{n+2} \tilde{f}((x_1, 0, 0), x_2, \ldots, x_n) .
\]
\end{cor}

\begin{proof} The proof is a consequence of Lemma~\ref{lem:odd:operator}.
Indeed, recall that the symbol~$\psi$ given in~\eqref{psidierre000}
is rotationally invariant, and operators $L_n$ and $L_{n + 2}$ satisfy formula~\eqref{FOU:SI}.
Accordingly, for every~$x=(x_1,x_2,\dots,x_n)\in\R^n$,
\begin{equation*}
L_ng(x)=L_n(x_1 f(x))=x_1L_{n+2}\tilde f((x_1,0,0),x_2,\dots,x_n),
\end{equation*}
as desired.
\end{proof}

Recall that $L_{n+2}$ maps $3$-isotropic functions to $3$-isotropic functions.
We now provide an extension of Corollary~\ref{cor:odd:flapBIS} to antisymmetric functions. This is the counterpart of Theorem~\ref{thm:odd:flap}
for general operators.

\begin{thm}
\label{thm:odd:flapBIS}
Suppose that $\Omega$ is an open set of~$\R^n$. Let $u$ be an antisymmetric function on $\R^n$ such that $u\in C^2_{{\rm{loc}}}(\Omega)$, $x_1^{-1}u \in L^\infty(\R^n)$,  and let $\tilde{v}$ be the corresponding function on $\R^{n + 2}$ given by
\[
 \tilde{v}(\tilde{x}) = \frac{u\big (\sqrt{\tilde x_1^2 + \tilde x_2^2 + \tilde x_3^2}, \tilde x_4, \ldots, \tilde x_{n+2} \big )}{\sqrt{\tilde{x}_1^2 + \tilde{x}_2^2 + \tilde{x}_3^2}} \, .
\]

Then:
\begin{enumerate}[(a)]
\item\label{thm:odd:flap:aBIS} if $\Omega_0 = \{x \in \Omega : x_1 \ne 0\}$ and 
\[
 \tilde \Omega_0 = \bigg \{\tilde{x} \in \R^{n + 2} : \Big (\sqrt{\tilde x_1^2 + \tilde x_2^2 + \tilde x_3^2}, \tilde x_4, \ldots, \tilde x_{n+2} \Big ) \in \Omega_0\bigg \} ;
\]
then the function $\tilde{v}$ is $C^2_{{\rm{loc}}} (\tilde \Omega_0) \cap L^\infty(\R^{n+2}) $. 

\item\label{thm:odd:flap:cBIS} for every $x \in \Omega_0$, we have that
\begin{equation}
\label{eq:odd:flapBIS}
L_nu(x) = x_1 L_{n+2} \tilde{v}((x_1, 0, 0), x_2, \ldots, x_n);
\end{equation}
\item\label{thm:odd:flap:dBIS} for an arbitrary $C_c^\infty$ function $g$ on $\Omega$, we have
\begin{equation}
\label{eq:odd:flap:weakBIS}
 \langle u, L_n g \rangle_{\R^n} = (2 \pi)^{-1} \langle \tilde{v}, L_{n+2} \tilde{f} \rangle_{\R^{n + 2}} ,
\end{equation}
where
\[ \tilde{f}(\tilde{x}) = \frac{g_A\big (\sqrt{\tilde x_1^2 + \tilde x_2^2 + \tilde x_3^2}, \tilde x_4, \ldots, \tilde x_{n+2} \big ) }{\sqrt{\tilde x_1^2 + \tilde x_2^2 + \tilde x_3^2}} \]
and
\[ g_A(x) = \frac 12 \big (  g(x_1,\dots,x_n) - g(-x_1,x_2,\dots,x_n) \big ) ; \]
\item\label{thm:odd:flap:eBIS} for an arbitrary $C_c^\infty$ function $\tilde{f}$ on
\[ \tilde \Omega=\bigg \{\tilde{x} \in \R^{n + 2} : \Big (\sqrt{\tilde x_1^2 + \tilde x_2^2 + \tilde x_3^2}, \tilde x_4, \ldots, \tilde x_{n+2} \Big ) \in \Omega \bigg \}, \]%
we have
\begin{equation}
\label{eq:odd:flap:reverseBIS}
\langle \tilde{v}, L_{n+2}\tilde{f} \rangle_{\R^{n + 2}} = 2 \pi \langle u, L_n g \rangle_{\R^n} ,
\end{equation}
where $g(x) = x_1 f(x)$ and $f$ is given in terms of the 3-isotropic symmetrization of $\tilde{f}$ by
\begin{equation*}
 f(x) = \frac{1}{4 \pi} \int_{\sph^2} \tilde{f}(|x_1| z, x_2, \ldots, x_n)\, \mathcal H^2(dz) ;
\end{equation*}
here $\sph^2$ is the unit sphere in $\R^3$ and $\mathcal H^2$ is the surface measure.
\end{enumerate}
\end{thm}

\begin{proof}
Since the map $\tilde x \mapsto \big (\sqrt{\tilde x_1^2 + \tilde x_2^2 + \tilde x_3^2}, \tilde x_4, \ldots, \tilde x_{n+2} \big )$ is smooth from $\tilde \Omega_0$ to $\Omega_0$ and \(x \mapsto x_1^{-1} u(x)\) is in \(C^2_{{\rm{loc}}}(\Omega_0)\), it follows that \(v\in C^2_{{\rm{loc}}}(\tilde \Omega_0)\). Moreover, the assumption \(x_1^{-1}u \in L^\infty(\R^n)\) immediately implies that \(\tilde v \in L^\infty(\R^{n + 2})\), which gives~(\ref{thm:odd:flap:aBIS}). 

To prove the claim in~(\ref{thm:odd:flap:c}), we notice that if $f(x) = x_1^{-1} u(x)$ is a Schwartz function, item~(\ref{thm:odd:flap:cBIS}) reduces to Corollary~\ref{cor:odd:flapBIS}. The extension to general $u$ follows by mollification and cutting off. 

The proofs of~(\ref{thm:odd:flap:dBIS}) and~(\ref{thm:odd:flap:eBIS}) are analogous to the proofs of Theorem~\ref{thm:odd:flap}~(\ref{thm:odd:flap:d}) and~(\ref{thm:odd:flap:e}) respectively.
\end{proof}

In preparation for the proof of Theorem~\ref{THM5.55}, we now give the definition of an operator satisfying the Harnack inequality. 

\begin{defn} \label{mTbTJpHm}
Let \(\mathcal A\) be the collection of quadruples \(( \Omega,  K,  c,v)\) such that  \( \Omega \subset \R^n\) is open; $K \subset \Omega$ is compact and connected; \( c\in L^\infty(\Omega)\); and \(  v\in C^2_{{\rm{loc}}}( \Omega) \cap L^\infty  (\R^n)\) satisfies, in the weak sense, \begin{align}
\label{PDEBIS}
\begin{PDE}
L_n  v + c   v &= 0, &\text{in } \Omega \\
 v &\geq 0, &\text{in } \R^n.
\end{PDE}
\end{align} We say that an operator \(L_n\) given by~\eqref{5XOlUbY7} admits the \emph{Harnack inequality} if, for all \((\Omega,K,c,v)\in \mathcal A\), there holds that \begin{align*}
\inf_ K  v \leq C \sup_K v
\end{align*} for some \(C=C(\Omega,K,\| c\|_{L^\infty(\Omega}) )>0\).
\end{defn}

\begin{remark}
\label{rem:harnack}
When $c = 0$ in~\eqref{PDEBIS}, that is, when $u$ is a harmonic function in $\Omega$ for the operator $L_n$, then Harnack inquality holds under fairly general conditions on the kernel $j_n$ of $L_n$ (see~\eqref{5XOlUbY7}). Indeed: Theorem~4.1 in~\cite{MR3729529} provides a Harnack inequality (in fact: a boundary Harnack inequality; see Remark~1.10(e) in that paper) when $K$ and $\Omega$ are balls, under the following assumptions: $j_n$ is positive and nonincreasing (so that the kernel of $L_n$ is isotropic and \emph{unimodal}), and there is a constant $C > 0$ such that $j_n(r + 1) \ge C j_n(r)$ for $r \ge 1$ (so that in particular $j_n(r)$ goes to zero as $r \to \infty$ no faster than some exponential).

By Theorem~1.9 in~\cite{MR3729529}, the constant in the Harnack inequality is \emph{scale-invariant} if $j_n$ is positive and decreasing and it satisfies the following scaling condition:
\[ \frac{j_n(r)}{j_n(R)} \le C \biggl(\frac{r}{R}\biggr)^{-n - \alpha} \]
for some constants $C > 0$ and $\alpha \in (0, 2)$.

While it is natural to expect that the results described above also hold for non-zero $c$, such extension does not seem to be available in literature. More restrictive results are given in \cite{KIM2022125746,MR3794384,MR4023466} and references therein.
\end{remark}

We have now developed the machinery to give the proof of Theorem~\ref{THM5.55}.

\begin{proof}[Proof of Theorem~\ref{THM5.55}]
From Theorem~\ref{thm:odd:flapBIS}~(\ref{thm:odd:flap:cBIS}), we have that \begin{align*}
L_nu(x) &= x_1 L_{n+2}\tilde v((x_1, 0, 0), x_2, \ldots, x_n) 
\end{align*} for all \(x\in \Omega_0\); or equivalently, \begin{align*}
\sqrt{\tilde x_1^2 + \tilde x_2^2 + \tilde x_3^2  } \, L_{n+2}\tilde v (\tilde x) &= L_nu \Big (\sqrt{\tilde x_1^2 + \tilde x_2^2 + \tilde x_3^2}, \tilde x_4, \ldots, \tilde x_{n+2} \Big ) 
\end{align*} for all \(\tilde x \in \tilde \Omega_0\) (recall that $L_{n + 2} u$ is $3$-isotropic). It follows that if we denote $x = (\sqrt{\tilde x_1^2 + \tilde x_2^2 + \tilde x_3^2}, \tilde x_4, \ldots, \tilde x_{n+2})$, then \begin{align*}
\sqrt{\tilde x_1^2 + \tilde x_2^2 + \tilde x_3^2  } \, \big ( L_{n+2}\tilde v (\tilde x) + c(x) \tilde v (\tilde x) \big ) &=  L_nu(x)  + c(x) u(x)=0 .
\end{align*} Setting \(\tilde c (\tilde x) = c(x)\), we have that \begin{align*}
L_{n+2}\tilde v + \tilde c \tilde v = 0 \qquad \text{in } \tilde \Omega_0 . 
\end{align*} In fact, by Theorem~\ref{thm:odd:flapBIS}~(\ref{thm:odd:flap:eBIS}), the above equality holds in the weak sense in $\tilde \Omega$. By assumption, \(L_{n+2}\) admits the Harnack inequality, so from the definition applied to the quadruple \((\tilde \Omega, \tilde K , \tilde c, \tilde v)\) with \(\tilde K = \big  \{\tilde x \in \R^{n + 2} : \big (\sqrt{\tilde x_1^2 + \tilde x_2^2 + \tilde x_3^2}, \tilde x_4, \ldots, \tilde x_{n + 2} \big ) \in K \big \}\), we have that \begin{align*}
\sup_{\tilde K } \tilde v \leq C(\tilde \Omega,\tilde K,\|\tilde c\|_{L^\infty(\tilde \Omega)}) \inf_{\tilde K} \tilde v .
\end{align*} We conclude that \begin{align*}
\sup_{ x\in K_0} \frac{u(x)}{x_1} \leq C(\tilde \Omega,\tilde K,\|c\|_{L^\infty(\Omega_+)}) \inf_{x\in K_0}  \frac{u(x) }{x_1} ,
\end{align*} as desired. 
\end{proof}

%% file: Part3/Intro-geometric-identities.tex
\chapter{Introduction to Part III} 

The third part of this thesis explores several topics in the theory of nonlocal minimal surfaces. Since their introduction in \cite{MR2675483}, nonlocal minimal surfaces have been a very active area of research and much is still unknown when compared to their local counterparts. In particular, the regularity of nonlocal minimal surfaces is still, in general, an open problem, unlike the regularity of minimal surfaces which was resolved in the last century collectively in the works of Allard, Almgren, Bombieri, Caccioppoli, De Giorgi, Federer, Fleming, Guisti, Miranda, Reifenberg, Simons, and many other authors. In Chapter~\ref{oyZ0zmYT}, we obtain some nonlocal analogues of identities/estimates that were fundamental to prove of the regularity of minimal surfaces and show, in the limit \(s\to 1^-\), they recover the local result. In Chapter~\ref{1KsVju8j}, we prove density estimates for sets that are stationary with respect to the \(s\)-perimeter. This gives, as a corollary, that the fractional Sobolev inequality holds on hypersurfaces of zero \(s\)-mean curvature.

\section{Regularity of minimal surfaces} 

The concept of perimeter is introduced to the general public from an early age in terms of simple shapes such as circles, squares, triangles, etc. As a mathematical object, perimeter has been of interest for thousands of years. From the works of the Greeks on triangles; the mythological tale of Queen Dido of Carthage whose story was told in the epic of Virgils, written 29-19 BC; and Plateau's experiments with soap films and bubbles in the mid-1800s, many people sought to better understand and apply concepts related to perimeter. Possibly the most studied problem related to perimeter is the one of finding surfaces of least perimeter, otherwise known as minimal surfaces\footnote{It is common, particularly in differential geometry, to refer to submanifolds with zero mean curvature as minimal surfaces. In this thesis, minimal surface always refers to a set that minimises the perimeter according to~\eqref{5BfpzP0M} below and sets of zero mean curvature are referred to as stationary sets, see~\eqref{rkQGJIcf} and the paragraph directly proceeding~\eqref{rkQGJIcf}.}. In the works of Meusnier, Monge, and Lagrange in the 1700s, we can already find some development of minimal surfaces with the derivation of the minimal surface equation and the construction of the first non-trivial examples of surfaces with zero mean curvature such as the catenoid and helicoid. It wasn't until the last century that a more unified theory of minimal surfaces arose which utilised measure-theoretic techniques leading to the development of \emph{geometric measure theory}.

Proving the existence of minimal surfaces by directly working with smooth manifolds is challenging due to the lack of compactness (in short, a sequence of smooth manifolds may become arbitrarily close to an object with no regularity), so the theory of minimal surfaces follows the common principle in analysis of first proving the existence of a weak notion of solution then showing that a weak solution must enjoy higher regularity. In general, geometric measure theory can address surfaces of any co-dimension through concepts such as varifolds and currents; however, for the theory of nonlocal perimeters, it is still an important open problem to even define a concept of perimeter in co-dimension higher than one. For this reason, we will focus on the co-dimension one case which allows us to avoid varifolds and currents, and simply deal with measurable sets. In this context, there is no distinction between a set and its boundary, so we, as is common in the literature, often refer to a set as having a given property instead of its boundary, for example, \(E\subset \R^n\) has zero mean curvature instead of \(\partial E\) has zero mean curvature.

Given a measurable set \(E\subset \R^n\) and an open set \(\Omega \subset \R^n\), the perimeter of \(E\) in \(\Omega\) is defined by
\begin{align*}
    \operatorname{Per}(E;\Omega) &= \sup \bigg \{ \int_E \div T \dd x \text{ s.t. } T\in C^1_0(\Omega ; \R^n), \|T\|_{L^\infty(\Omega)}\leq 1 \bigg \}
\end{align*} and the perimeter of \(E\) is defined by \(\operatorname{Per}(E)=\operatorname{Per}(E;\R^n)\). The set \(E\) is said to be of finite perimeter in \(\overline \Omega\) if \(\operatorname{Per}(E;\overline \Omega)<+\infty\), have finite perimeter if \(\operatorname{Per}(E)<+\infty\), and have locally finite perimeter if \(\operatorname{Per}(E;\overline \Omega)<+\infty\) for all \(\Omega\subset\subset\R^n \). Then \(E\) is \emph{perimeter minimising}, or simply \emph{minimising} in \(\Omega\) if \begin{align}
      \operatorname{Per}(E;\overline\Omega) \leq   \operatorname{Per}(F;\overline\Omega) \label{5BfpzP0M}
\end{align} for all \(F\subset \R^n\) with finite perimeter in \(\Omega\) and \(E\setminus \Omega = F\setminus \Omega\). The collection of sets of finite perimeter in \(\Omega\) enjoy nice compactness properties and the perimeter function also enjoys nice properties such as lower semi-continuity which leads to the existence of a minimising set with some prescribed boundary datum, see, for example, \cite[Chapter 12]{maggi_sets_2012}.

One can also study stationary/critical sets and stable sets with respect to the perimeter functional. Indeed, let \(T\in C^1_0(\Omega ; \R^n)\) and \(\psi_T : \R^n \times (-\varepsilon , \varepsilon) \to \R^n\) be the solution to \begin{align}
    \begin{PDE}
\partial_t \psi_T &= T\circ \psi_T, &\text{in } \R^n \times (-\varepsilon , \varepsilon) \\
\psi_T &= T, &\text{on } \R^n \times \{t=0\}.
    \end{PDE} \label{w4olEwnS}
\end{align} Then the first variation and second variation of the perimeter are given by \begin{align*}
     \delta \operatorname{Per}(E;\Omega)[T] &= \frac{\dd }{\dd t}\bigg \vert_{t=0} \operatorname{Per}(\psi_T(E,t);\Omega) 
\end{align*} and \begin{align*}
     \delta^2 \operatorname{Per}(E;\Omega)[T] &= \frac{\dd^2 }{\dd t^2}\bigg \vert_{t=0} \operatorname{Per}(\psi_T(E,t);\Omega)
\end{align*} respectively. The set \(E\) is stationary in \(\Omega\) if \(\delta \operatorname{Per}(E;\Omega)[T]=0\) for all \(T\in C^1_0(\Omega ; \R^n)\) and is a stationary stable set in \(\Omega\) if \(\delta \operatorname{Per}(E;\Omega)[T]=0\) and \(\delta^2 \operatorname{Per}(E;\Omega)[T]\geq 0 \) for all \(T\in C^1_0(\Omega ; \R^n)\). If \(E\) has smooth boundary, \(\zeta \in C^\infty_0(\partial E \cap \Omega)\), and \(T\vert_{\partial E}=\zeta \nu_E\) then we can write
\begin{equation} 
\begin{split}
    \delta \operatorname{Per}(E;\Omega)[T] &= \int_{\partial E \cap \Omega } \mathrm H \zeta \dd \H^{n-1} \\
    \delta^2 \operatorname{Per}(E;\Omega)[T] &= \int_{\partial E \cap \Omega } \big ( \vert \nabla_{\partial E} \zeta \vert^2 - \big(c^2 -\mathrm H^2 \big ) \zeta^2 \big ) \dd \H^{n-1} 
\end{split} \label{rkQGJIcf}
\end{equation} where \(\mathrm H(x) = \div_{\partial E} \nu_{E}(x)\) is the \emph{mean curvature}, also given by the sum of the principal curvatures of \(\partial E\) at \(x\) and \(c(x)\) is the norm of the second fundamental form defined as the Euclidean norm of the principal curvatures of \(\partial E\) at \(x\). Here \(\nabla_{\partial E}\) and \(\div_{\partial E}\) denote the gradient and divergence of \(\partial E\) respectively. Hence, stationary sets are (or, in the case \(E\) does not have smooth boundary, can be interpreted as) sets whose boundary has zero mean curvature. Moreover, stationary, stable sets (with sufficiently regular boundary) satisfy the stability inequality \begin{align}
    - \Delta_{\partial E} \zeta - c^2 \zeta \geq 0 \quad \text{on } \partial E \cap \Omega \label{uosmVJs7}
\end{align} for all \(\zeta \in C^1_0(\partial E \cap \Omega)\).

Towards a full regularity theory of minimal sets, one first seeks to prove the regularity of a measurable set \(E\) that minimises the perimeter in \(B_1^{n-1}\times \R\) \emph{and} is the graph of a function \(v: B_1^{n-1}\to \R\) in \(B_1\) where \(B_1^{n-1} = \{ x' \in \R^{n-1} \text{ s.t. } \vert x' \vert <1\), \(v:B_1^{n-1}\to \R\). This is already very non-trivial--if one knows that \(v\in W^{1,\infty}(B_1^{n-1})\) then \(v\) satisfies the minimal surface equation \begin{align}
    \div' \bigg ( \frac{\nabla' v}{\sqrt{1+\vert \nabla' v\vert^2 }} \bigg ) =0, \qquad \text{in }B_1^{n-1} \label{pAl4RkAE}
\end{align} in the weak sense where \(\nabla'\) and \(\div'\) are the gradient and divergence on \(\R^{n-1}\). In this case,~\eqref{pAl4RkAE} is a uniformly elliptic equation in divergence form (with the ellipticity constant depending on \(\|\nabla u \|_{L^\infty(B_1^{n-1} )}\)), so De Giorgi-Nash-Moser theory and Schauder theory imply \(v\) is analytic in \(B_{1/2}^{n-1}\). However, there is no reason \emph{a priori} that \(v\) should be \(W^{1,\infty}(B_1^{n-1})\). To establish this, one first proves via the comparison principle that for bounded boundary datum, we have that \( v\in L^\infty(B_1^{n-1})\). Then the gradient estimate of Bombieri, De Giorgi, and Miranda \cite{MR0248647} implies that \(v\in W^{1,\infty}(B_{3/4}^{n-1})\). Hence, \(v\) is a weak solution of~\eqref{pAl4RkAE} (on a smaller ball), so we conclude that \(v\) is analytic in \(B_{1/2}^{n-1}\). 

Now, to prove the regularity of general minimising sets, one hopes to prove that minimising sets are locally the graph of a function, from which one could conclude analyticity from the theory sketched above. A step towards this is the improvement of flatness lemma of De Giorgi \cite{MR0179651} which heuristically states that if \(E\) is \emph{sufficiently flat} about a point \(x\in \partial E\) then it is locally a graph, so the regularity reduces to identifying points about which \(E\) is sufficiently flat. The strategy from here is to perform a so-called `blow-up' at points on the boundary of \(E\) which means we consider the sets \(E_{x,r}= \frac{E-x}{r}\) with \(x\in \partial E\) and \(r\to +\infty\). Geometrically, a blow-up simply amounts to zooming closer and closer into \(x\). It can proven that, after passing to a subsequence, \(E_{x,r}\) convergences to a cone \(C\) in \(L^1_{\mathrm{loc}}(\R^n)\) where a cone \(C\) is a measurable subset of \(\R^n\) such that \(tC=C\) for all \(t>0\). Moreover, from the monotonicity formula for sets of zero mean curvature, it follows that \(C\) is perimeter minimising in \(B_R\) for all \(R>0\). If we can prove that \(C\) is trivial (i.e \(C\) is a halfspace) then the \(L^1_{\mathrm{loc}}(\R^n)\) convergence of \(E_{x,r}\) to \(C\) implies that \(E\) is sufficiently flat about \(x\). Thus, the regularity of minimising sets reduces to proving the non-existence of trivial cones that are locally minimising in \(\R^n\). 

For the classification of locally perimeter minimising cones, we have the following famous theorem. \begin{thm} \thlabel{IPfEs4h0}
Let \(2\leq n \leq 7\). If \(E\) is a locally perimeter minimising cone in \(\R^n\) then \(E\) is a half-space. Moreover, if \(n\geq 8\) then there exist non-trivial locally perimeter minimising cones.
\end{thm}

\thref{IPfEs4h0} is trivial for \(n=2\), and was proved by De Giorgi in 1965 for \(n=3\), Almgren in 1966 for \(n=4\), and Simons in 1968 for \(n \leq 7\). In \cite{MR0233295}, Simons also proved that \begin{align*}
    C = \{ (x,y) \in \R^4 \times \R^4 \text{ s.t. } \vert x\vert < \vert y \vert \}
\end{align*} is a stationary, stable cone. Furthermore, it was proven to be locally perimeter minimising by Bombieri, De Giorgi, and Giusti in \cite{MR0250205}.

Simons' proof of \thref{IPfEs4h0} relies on two inequalities. The first is the stability inequality \begin{align}
    \int_{\partial E} \bigg ( \frac 12 \Delta_{\partial E}c - \vert \nabla_{\partial E}c\vert^2 +c^4 \bigg ) \eta^2 \dd \H^{n-1} \leq \int_{\partial E}c^2 \vert \nabla_{\partial E} \eta \vert^2 \dd \H^n \label{ifBA2379}
\end{align} for all \(\eta \in C^\infty_0(\partial E)\). This follows from~\eqref{uosmVJs7} by choosing \(\zeta = c \eta\) and holds for all stationary stable sets. The second inequality states that, if \(E\) is a cone (centred at 0), then \begin{align}
    \frac 12 \Delta_{\partial E}c - \vert \nabla_{\partial E}c\vert^2 +c^4 \geq \frac{2c^2}{\vert x \vert^2} . \label{GmvebRtF}
\end{align} This inequality relies on the following formula, known as Simons formula, which reads for the case \(E\) has zero mean curvature, \begin{align}
    \frac 1 2 \Delta_M c^2 + c^2 &= \sum_{i,j,k}\vert \nabla_{M,k} h_{ij} \vert^2 \label{q9oyfDh8}
\end{align} where \(h_{ij}\) are the components of the second fundamental form. Combining~\eqref{ifBA2379} and~\eqref{GmvebRtF} then making an intelligent choice for \(\eta\), Simons then obtains a contradiction provided that \(n \leq 7\).

\section{Nonlocal minimal sets}  
In the celebrated paper \cite{MR2675483}, a nonlocal (or fractional) concept of perimeter, known as the \(s\)-perimeter was introduced. For \(s\in (0,1)\), a measurable set \(E\subset \R^n\), and an open set \(\Omega \subset \R^n\), the \(s\)-perimeter of \(E\) in \(\Omega\) is given by \begin{align}
    \operatorname{Per}_s(E;\Omega) = \frac 12 \iint_{\mathcal Q(\Omega)} \frac{\vert \chi_E(x)-\chi_E(y)\vert }{\vert x- y \vert^{n+s}} \dd y \dd x \label{UJuXBCae}
\end{align} where \(\mathcal Q(\Omega) = \R^{2n}\setminus ( \Omega^c\times \Omega^c ) = (\Omega\times \Omega) \cup (\Omega\times \Omega^c)\cup (\Omega^c\times \Omega)\) and \(\Omega^c=\R^n \setminus \Omega\). When \(\Omega = \R^n\), we write \(\operatorname{Per}_s(E)=\operatorname{Per}_s(E;\R^n)\). In this case, we also have \begin{align}
    \operatorname{Per}_s(E) = \frac 12 [ \chi_E]_{W^{s,1}(\R^n)} \label{FGEASHSh}
\end{align} where \begin{align*}
    [ u]_{W^{s,1}(\R^n)} = \int_{\R^n}\int_{\R^n}\frac{\vert u(x) - u(y)\vert}{\vert x-y\vert^{n+s}} \dd y \dd x 
\end{align*} is the Gagliardo semi-norm associated to the fractional Sobolev space \(W^{s,1}(\R^n)\). Since, for suitable functions \(u\), \((1-s) [ u]_{W^{s,1}(\R^n)} \to C(n)\|\nabla u\|_{L^1(\R^n)}\) as \(s\to 1^-\), and the (classical) perimeter of \(E\) can be viewed as \(\|\nabla \chi_E\|_{L^1(\R^n)}\)  in a \(\operatorname{BV}\) sense, we see that the \(s\)-perimeter is a natural generalisation of the classical perimeter. It is common in the literature to express~\eqref{UJuXBCae} and~\eqref{FGEASHSh} in terms of `interaction' terms, that is, if  we define the nonlocal interaction of two disjoint sets \(A\) and \(B\) as \begin{align*}
    I_s(A,B) = \int_A \int_B \frac{\dd y \dd x }{\vert x- y\vert^{n+s}}
\end{align*} then \begin{align*}
    \operatorname{Per}_s(E; \Omega) &= I_s(E\cap \Omega , E^c\cap \Omega)+I_s(E\cap\Omega,E^c\cap \Omega^c) + I_s(E\cap \Omega^c,E^c\cap \Omega)
\end{align*} and \begin{align*}
    \operatorname{Per}_s(E) &= I_s(E,\R^n\setminus E). 
\end{align*} 

It is also possible to consider the first and second variations of the \(s\)-perimeter. If \(T\in C^1(\Omega;\R^n)\), \(\zeta \in C^1_0(\partial E \cap \Omega)\), \(T \vert_{\partial E} = \zeta \nu_E\), and \(\psi_T\) is defined according to~\eqref{w4olEwnS} then \begin{align*}
    \delta \operatorname{Per}_s(E;\Omega)[T] = \frac{\dd }{\dd t} \bigg \vert_{t=0} \operatorname{Per_s}(\psi_T(E,t); \Omega ) = \int_{\partial E \cap \Omega } \mathrm H_s \zeta \dd \mathcal H^{n-1}
\end{align*} and \begin{align*}
    \delta^2 \operatorname{Per}_s(E;\Omega)[T] &= \frac{\dd^2 }{\dd t^2} \bigg \vert_{t=0} \operatorname{Per_s}(\psi_T(E,t); \Omega ) \\ &= \int_{\partial E} \int_{\partial E} \frac{\vert \zeta(x)-\zeta(y)\vert^2}{\vert x - y \vert^{n+s}} \dd \mathcal H^{n-1}_y \dd \mathcal H^{n-1}_x - \int_{\partial E \cap \Omega} (c_s^2-\mathrm H \mathrm H_s ) \zeta^2 \dd \mathcal H^{n-1}
\end{align*} where \begin{align*}
    \mathrm H_s (x) &= \PV \int_{\R^n} \frac{\chi_{\R^n \setminus E}(y) - \chi_ E(y) }{\vert x- y \vert^{n+s}} \dd y
\end{align*} is the \(s\)-mean curvature of \(E\) and \begin{align*}
    c_s^2(x) &= \int_{\partial E} \int_{\partial E} \frac{\vert \nu_E(x)-\nu_E(y)\vert^2 }{\vert x- y \vert^{n+s}} \dd \mathcal H^{n-1}_y \dd \mathcal H^{n-1}_x
\end{align*} is the nonlocal analogue of the norm squared of the second fundamental form, see \cite{MR3322379,davila_nonlocal_2018}. Then, \(E\) is an \(s\)-stationary set if \( \delta \operatorname{Per}_s(E;\Omega)[T] =0\) for all \(T\in C^1(\Omega;\R^n)\) and \(E\) is \(s\)-stationary, stable set if \( \delta \operatorname{Per}_s(E;\Omega)[T] =0\) and \(\delta^2 \operatorname{Per}_s(E;\Omega)[T]  \geq 0\) for all \(T\in C^1(\Omega;\R^n)\).

Unlike the local case, the regularity theory for nonlocal minimal surfaces is unresolved and is one of the most important open problems in the field. In \cite{MR2675483}, in addition to introducing the notion of nonlocal perimeter, several fundamental results were adapted from the local case to the nonlocal case. These include the existence theory, uniform density estimates, the improvement of flatness, and the monotonicity formula. 

For the regularity of \(s\)-minimal graphs, over several papers, it was proven that bounded \(s\)-minimal graphs are analytic. Indeed, consider \(u : B_1^{n-1}\to \R\) with \(u(0)=0\), \(C=B_1^{n-1}\times R\), and \(E\subset \R^n\) such that \begin{align*}
    E \cap C = \{ x\in C \text{ s.t. } x_n < u(x_1,\dots,x_{n-1}) \} 
\end{align*} is \(s\)-minimal in \(C\). In \cite{MR3680376}, it was proven that if \(u\) is Lipschitz then \(u\) is \(C^{1,\alpha}\) for all \(\alpha<s\), and, in \cite{MR3331523}, it was shown that if \(u\) is \(C^{1,\alpha}\) for some \(\alpha>s/2\) then \(u\) is smooth. Furthermore, developing the ideas of the previous two papers, if \(u\) is Lipschitz then it was proven that \(u\) belongs to a Gevrey class (an intermediary space between smooth and analytic) in \cite{MR3334976} and that \(u\) is analytic in \cite{MR4466117}. Finally, in \cite{MR3680376}, a nonlocal analogue of the gradient estimate of Bombieri, De Giorgi, and Miranda was established which implies bounded \(s\)-minimal graphs are Lipschitz. Together, these results form the nonlocal analogues of the De Giorgi-Nash-Moser theory and Schauder theory used in the local regime. 

Towards the regularity of general \(s\)-minimal sets, the theory developed in \cite{MR2675483} implies that regularity of \(s\)-minimal sets in \(\R^n\) follows from the non-existence of non-trivial locally \(s\)-minimal cones in \(\R^n\). From \cite{MR3090533}, it is known that there are no non-trivial locally \(s\)-minimal cones for \(n=2\), but for \(n>2\) the classification of \(s\)-minimal cones is, in general, an open problem. It is known, however, that for \(2\leq n \leq 7\),  there are no non-trivial locally \(s\)-minimal cones provided that \(s\in (s_\star,1)\) for some \(s_\star=s_\star(n)\in (0,1)\) as was proven in \cite{MR2782803}. Morally, this result holds since as \(s\to 1^-\), \(s\)-minimal cones approach minimal cones and we have a full classification of minimal cones from \thref{IPfEs4h0}. Unfortunately, the proof relies on a compactness argument, so it gives no information about the value of \(s_\star\). In fact, in \cite{davila_nonlocal_2018}, a \(s\)-stationary, stable cone was constructed in \(n=7\) for \(s\) close to \(0\) which suggests (but does not prove) that the nonlocal analogue of \thref{IPfEs4h0} may only hold up to \(n=6\) (or even less) when \(s\) is small. This is (morally) believable because the \(s\)-perimeter converges to the volume as \(s\to 0^+\) and minimisers of the volume can be arbitrarily irregular. There are still no known examples in any dimension of non-trivial locally \(s\)-minimal cones. 

Finally, we remark that very little is known about the regularity of \(s\)-stationary, stable sets. In \cite{MR3981295}, it was proven that the perimeter (not \(s\)-perimeter) of \(s\)-stationary, stable sets satisfy universal upper bounds. Interestingly, in the local case, this result is only known for \(n=3\) and is open for \(n\geq 4\). Building upon the estimates of \cite{MR3981295}, it was shown in \cite{MR4116635} that \(s\)-stationary, stable cones in \(\R^3\) are trivial. Furthermore, in \cite{caselli2024yaus,caselli2024fractional}, a nonlocal analogue of Yau's conjecture (i.e. in a closed \(n\)-dimensional Riemannian manifold, there are infinitely many \(s\)-stationary sets) is proven and much of the regularity theory in the nonlocal setting is developed for subsets of Riemannian manifolds (as opposed to subsets of Euclidean space). They also establish the monotonicity formula for \(s\)-stationary sets which was previously only known for \(s\)-minimisers. 

In Chapter \ref{oyZ0zmYT}, inspired by Simons argument for the classification of minimal cones, we establish nonlocal analogues of Simons formula~\eqref{q9oyfDh8} and the stability estimate~\eqref{ifBA2379}. In the limit \(s\to 1^-\), our results recover~\eqref{q9oyfDh8} and~\eqref{ifBA2379}. Unfortunately, we are unable to obtain a nonlocal analogue of~\eqref{GmvebRtF}, so we cannot fully extend his argument. In Chapter \ref{1KsVju8j}, we prove that the perimeter of \(s\)-stationary (not necessarily stable) sets satisfy universal lower bounds, see the following section for more details. This gives, as a corollary, the nonlocal fractional Sobolev inequality on the boundary of \(s\)-stationary sets.

\section{The Sobolev inequality on hypersurfaces}  

The Sobolev inequality is a foundational result in the theory of partial differential equations and functional analysis. In its simplest form, it states that if \(n\geq 1\) is a positive integer, \(1\leq p <n\), and \(p^\ast = \frac{np}{n-p}\) then \begin{align}
    \bigg (\int_{\R^n} \vert u(x) \vert^{p^\ast} \dd x  \bigg )^{\frac 1 {p^\ast} } \leq C \bigg ( \int_{\R^n}\vert \nabla u \vert^p \dd x \bigg )^{\frac1p} \qquad \text{for all }u \in C^\infty_0(\R^n). \label{WQyMQR8L}
\end{align} The Sobolev inequality has many applications, including the Rellich–Kondrachov embedding theorem and energy estimates in \textsc{PDE}s such as in the solution of Hilbert's XIX-th problem, see for example \cite[Proposition 3.11]{MR4560756}. It is also well-known that~\eqref{WQyMQR8L} when \(p=1\) is equivalent to the isoperimetric inequality \begin{align}
    \vert E \vert^{\frac{n-1}n } \leq C \operatorname{Per}(E) \label{Qn6WJebO}
\end{align} for all measurable sets \(E\subset \R^n\) with finite perimeter. Indeed, on one hand, one obtains~\eqref{WQyMQR8L} when \(p=1\) from~\eqref{Qn6WJebO} by writing \(\|u\|_{L^{\frac n {n-1}}(\R^n)}\) in terms of the layer-cake representation formula, applying~\eqref{Qn6WJebO} then using the co-area formula to recover \(\|\nabla u \|_{L^1(\R^n)}\). On the other hand,~\eqref{Qn6WJebO} follows from~\eqref{WQyMQR8L} by applying~\eqref{WQyMQR8L} with \(p=1\) to an \(\varepsilon\)-mollification of \(\chi_E\) (the characteristic function of \(E\)) then sending \(\varepsilon \to 0^+\). Furthermore,~\eqref{WQyMQR8L}, for general \(p\), is easily obtained from the \(p=1\) case by applying the \(p=1\) Sobolev inequality to \(\vert u \vert^p\) and using H\"older's inequality.

Over the last century, a topic that received a lot of attention was, given a sub-Riemannian manifold\footnote{For simplicity, we only mention the co-dimension 1 case, but one can also consider the case of higher co-dimension.} \(M^n\hookrightarrow \R^{n+1}\) satisfying \(\mathrm H_M=0\), does the Sobolev inequality \begin{align}
     \bigg (\int_{M} \vert u(x) \vert^{p^\ast} \dd \H^n_x  \bigg )^{\frac 1 {p^\ast} } \leq C \bigg ( \int_{M}\vert \nabla u \vert^p \dd \H^n_x \bigg )^{\frac1p} \qquad \text{for all }u \in C^\infty_0(M)\label{jwwJTLaz}
\end{align} hold? Here \(\dd \H^n_x\) is the \(n\)-dimensional Hausdorff measure, \(\nabla=\nabla_M\) is the gradient on \(M\), and \(C\) is a positive constant depending only on \(n\) and \(p\). The answer to this question is yes. The affirmative answer was established over many years in different dimensions and under different topological assumptions by several authors including Almgren, Carlemen, Choe, Feinberg, Hsiung, Michael, Osserman, Reid, Schiffer, Schoen, Stone, Simon, Topping and Yau. See the introduction of \cite{brendle_isoperimetric_2021} and the survey \cite{MR2167266}. Furthermore, one might ask if~\eqref{jwwJTLaz} holds on \(M^n\hookrightarrow \R^{n+1}\) with \(M\) not necessarily having zero mean curvature?  In general, the answer to this is no since, for \(M\) closed (say \(M=\Sph^n\) for concreteness), \(u(x)\equiv 1\) is in \(C^\infty_0(M)\), but \(\|u\|_{L^{p^\ast}(M)} = \vert M\vert^{\frac{n-p}{np}} > 0\) and \(\| \nabla u \|_{L^p(M)}=0\); however, if one includes an addition weighted \(L^p\) error term involving the mean curvature on the right-hand side of~\eqref{jwwJTLaz} then an inequality similar to~\eqref{jwwJTLaz} can be obtained. Indeed, the Michael-Simon inequality, also commonly known as the Michael-Simon-Allard inequality or the Michael-Simon Sobolev inequality, states \begin{align}
    \| u \|_{L^{p^\ast}(M)} \leq C \big ( \| \nabla u \|_{L^p(M)}+ \| \mathrm H_M u \|_{L^p(M)}\big ) \qquad \text{for all }u \in C^\infty_0(M). \label{3WKvPG9R}
\end{align} Inequality~\eqref{3WKvPG9R} was first established in the PhD thesis of Leon Simon \cite{Simonthesis} for \(C^2\) manifolds, proved with a simpler method and for a more general class of manifolds by Michael and Simon in \cite{michael_sobolev_1973}, and extended to varifolds by Allard in \cite{allard_first_1972}. A proof using optimal transport methods was also given in \cite{MR2610380}. A nice application of the Michael-Simon inequality is it gives upper bounds for the extinction time of a closed manifold \(M_0\) undergoing mean curvature flow in terms of the initial volume \(\vert M_0\vert\), see \cite[Section F.2]{MR1462701} for more details. Until recently, obtaining the sharp constant in~\eqref{jwwJTLaz} was an open problem (though it was known if \(M\) was area minimising, see \cite{MR0855173}), however, Brendle in \cite{brendle_isoperimetric_2021} proved that if \(M^n\hookrightarrow \R^{n+1}\) is compact (not necessarily with empty boundary) then \begin{align*}
    n \vert B_1^n\vert^{\frac 1 n } \bigg (\int_{M} \vert u \vert^{\frac n {n-1}} \dd \H^n  \bigg )^{\frac{n-1}n } \leq  \int_M \sqrt{\vert \nabla_M u \vert^2 + \mathrm H^2 u^2} \dd \H^n + \int_{\partial M} \vert u \vert \dd \H^{n-1}
\end{align*} for all \(u\in C^\infty(M)\) with equality if and only if \(M=B_1^n\), a ball in \(\R^n\).

Returning to \(\R^n\), it is also well-known that there holds the fractional Sobolev inequality. Let \(s\in (0,1)\), \(1\leq p <n/s\), and \(p^\ast = \frac{np}{n-sp}\). Then \begin{align}
     \bigg (\int_{\R^n} \vert u(x) \vert^{p^\ast} \dd x  \bigg )^{\frac 1 {p^\ast} } \leq C \bigg ( \iint_{\R^n \times \R^n} \frac{\vert u(x) - u(y) \vert^p}{\vert x - y \vert^{n+sp}} \dd y \dd x \bigg )^{\frac 1p } \qquad \text{for all }u \in C^\infty_0(\R^n). \label{QQ6xNqoG}
\end{align} In complete analogy to the local case,~\eqref{QQ6xNqoG} with \(p=1\) is equivalent to the fractional isoperimetric inequality \begin{align*}
    \vert E \vert^{\frac{n-s} n } \leq C \operatorname{Per}_s(E) \qquad \text{for all measurable }E\subset \R^n.
\end{align*} For a proof of~\eqref{QQ6xNqoG}, see \cite[Theorem 6.5]{MR2944369} and, for a proof using rearrangements, see \cite[Chapter 10]{leoni_first_2023}. Furthermore, an elegant, elementary proof of~\eqref{QQ6xNqoG} was given in an unpublished argument attributed to Brezis, see \cite[Theorem 2.2.1]{MR3469920}, which takes advantage of the scale-invariance of \(\R^n\).

It is a current open problem to prove a fractional analogue of the Michael-Simon inequality on hypersurfaces \(M^n \hookrightarrow \R^{n+1}\). A careful examination of the proof of~\eqref{QQ6xNqoG} given by Brezis reveals that if \(M\) satisfies the universal lower bound: \begin{align}
    \vert M \cap B_\rho^{n+1} (x) \vert \geq c_\star \rho^n \qquad \text{for all } \rho >0, x\in M \label{O0QOaKdj}
\end{align} then it can be concluded that \begin{align}
     \bigg (\int_M \vert u(x) \vert^{p^\ast} \dd \H^n_x  \bigg )^{\frac 1 {p^\ast} } \leq C \bigg ( \iint_{M \times M} \frac{\vert u(x) - u(y) \vert^p}{\vert x - y \vert^{n+sp}} \dd \H^n_y \dd \H^n_x \bigg )^{\frac 1p } \qquad \text{for all }u \in C^\infty_0(M). \label{Eb70M9eI}
\end{align} The constant \(C>0\) depends only on \(n\), \(s\), and \(c_\star\). For a proof of this fact, see \cite[Proposition 5.2]{MR3934589}.  The only other result in this direction is when \(M=\partial E\) for some convex \(E\subset \R^{n+1}\). In this case, it was proven in \cite{MR4565417} that if \(\alpha \in (0,1)\) (and all the parameters as above) then \begin{align}
   \| u\|_{L^{p^\ast}(M)} \leq C \big ( [u]_{W^{s,p}(M)} + \| \mathrm H_{\alpha,E}^{\frac \alpha s } u \|_{L^p(M)} \big ) \qquad \text{for all }u \in C^\infty_0(M) \label{LTeitonw}
\end{align} where \begin{align*}
    [u]_{W^{s,p}(M)} =  \bigg (\iint_{M \times M} \frac{\vert u(x) - u(y) \vert^p}{\vert x - y \vert^{n+sp}} \dd \H^n_y \dd \H^n_x \bigg )^{\frac 1 p }. 
\end{align*} In particular, no relationship between \(\alpha\) and \(s\) is assumed. The fact that \(E\) is convex is fundamental to the proof given in \cite{MR4565417} and it is not clear how one can adapt their argument to the non-convex case. For example, an important step in their proof is the pointwise inequality \begin{align*}
    (\operatorname{Per}(E))^{-\frac \alpha n }  \leq C \mathrm H_{\alpha,E}(x) 
\end{align*} for all \(x\in \partial E\) which is easily shown to be false if \(E\) is non-convex. Indeed, it is not yet clear if~\eqref{LTeitonw} is even the correct form of a fractional Michael-Simons inequality in the general (non-convex) case; it is possible the mean curvature term in~\eqref{LTeitonw} must involve both the \(s\)-mean curvature and the classical mean curvature. One motivation for this comes from proving upper bounds for the extinction time of a closed manifold that is evolved according to fractional mean curvature flow (see, \cite{MR2564467,MR2487027}) in terms of the initial volume. If one wishes to replicate the proof in the local case (see \cite[Section 5]{MR4565417}), the estimate \begin{align*}
  \vert M \vert^{\frac{n-1-s} n } \leq C  \int_M \mathrm H_s \mathrm H \dd \H^n,
\end{align*} is required, but this doesn't follow immediately from~\eqref{LTeitonw}.

In Chapter \ref{1KsVju8j} of this thesis, we prove that~\eqref{O0QOaKdj} holds for any set with vanishing \(s\)-mean curvature for some \(c_\star\) depending only \(n\) and \(s\). As explained above, this immediately implies that~\eqref{Eb70M9eI} holds on set with vanishing \(s\)-mean curvature. 


%% file: Part3/Simons-identity-thesis.tex
\chapter{Some nonlocal formulas inspired by an identity of James Simons} \label{oyZ0zmYT}

Inspired by a classical identity proved by James Simons, we establish a new geometric formula
in a nonlocal, possibly fractional, setting.

Our formula also recovers the classical case in the limit, thus providing an approach to Simons' work
that does not heavily rely on differential geometry.

\section{Introduction}

\subsection{Taking inspiration from Simons' work}
A classical identity proved by James Simons in~\cite{MR0233295}
states that at every point of a smooth hypersurface with vanishing mean curvature we have that
\begin{equation}\label{ORIGINAL SIM}
\Delta c^2+2c^4=2\sum_{i,j,k=1}^{n-1} |\delta_{k}h_{ij}|^2.
\end{equation}
Here above, $\delta_{k}$ denotes the tangential derivative
in the $k$-th coordinate direction,
$h_{ij}$ the entries of the second
fundamental form, $c$ the norm of the second fundamental form, and~$\Delta $
the Laplace-Beltrami operator on the hypersurface (see e.g.
formula~(2.16) in~\cite{MR2780140}, or, equivalently, the seventh formula in display on page~123
of~\cite{MR0775682}, for further details on this classical formula).

Simons's Identity is pivotal, since it provides the essential ingredient to establish the regularity of stable minimal surfaces up to dimension~$7$.

In this note we speculate about possible generalizations of Simons' Identity to nonlocal settings.
In particular, we will consider the case of boundary of sets
and of level sets of functions. These cases are motivated, respectively, by the study of nonlocal minimal surfaces and nonlocal phase transition equations. The prototypical case of these problems comes from fractional minimal surfaces, as introduced in~\cite{MR2675483}, and we recall that the full regularity theory of the minimizers of the fractional perimeter is one of the main open problems in the field of nonlocal equations:
up to now, this regularity is only known when the ambient space has dimension~$2$, see~\cite{MR3090533}, or up to dimension~$7$
provided that the fractional parameter is sufficiently close to integer, see~\cite{MR3107529}, or when the surface possesses a graphical structure, see~\cite{MR3934589} (see also~\cite{MR3981295, MR4116635, HAR}
for the case of stable nonlocal minimal surfaces, i.e. for surfaces of vanishing nonlocal mean curvature with
nonnegative definite second variation of the corresponding energy functional).

The problem of nonlocal minimal surfaces can also be considered for more general kernels than the one of purely fractional type, see~\cite{MR3930619, MR3981295}, and it can be recovered as the limit in the $\Gamma$-convergence sense of long-range phase coexistence problems, see~\cite{MR2948285}.
In this regard, the regularity properties of nonlocal minimal surfaces are intimately related to the flatness of nonlocal phase transitions, which is also a problem of utmost importance in the contemporary research:
up to now, these flatness properties have been established in dimension up to~$3$, or up to~$8$ for mildly nonlocal operators
under an additional limit assumption, or in dimension~$4$ for the square root of the Laplace operator, see~\cite{MR2177165, MR2498561, MR2644786, MR3148114, MR3280032, MR3740395, MR3812860, MR3939768, MR4050103, MR4124116}, 
the other cases being widely open.\medskip

In this paper, we will not specifically address these regularity and rigidity problems, but rather focus
on a geometric formula which is closely related to Simons' Identity in the nonlocal scenarios.
The application of this formula for the regularity theory appears to be highly nontrivial, since
careful estimates for the reminder terms are needed (in dimension~$3$, a reminder estimate has been recently put forth in~\cite{HAR}).

An interesting by-product of the formula that we present here is that it recovers the classical Simons' Identity
as a limit case. 
Therefore, our nonlocal formula also provides a new approach towards
the original
Simons' Identity, with a new proof which makes only very limited use
of Riemannian geometry and relies instead on some clever use of the integration by parts.\medskip

Let us now dive into the technical details of our results.

\subsection{The geometric case}\label{PRE1}

Let~$K$ be a kernel satisfying
\begin{equation}\label{ipotesi}
\begin{split}
& K\in C^1(\R^n\setminus\{0\}),\\
& K(x)=K(-x),\\
& |K(x)|\le \frac{C}{|x|^{n+s}}\\
{\mbox{and }}\qquad & |\omega\cdot\nabla K(x)|\le\frac{C\,|\omega\cdot x|}{|x|^{n+s+2}} \qquad{\mbox{for all~$\omega\in S^{n-1}$}},
\end{split}
\end{equation}
for some~$C>0$ and~$s\in(0,1)$.

Given a set~$E$ with smooth boundary, we consider the $K$-mean curvature of~$E$
at~$x\in\partial E$ given by
\begin{equation}\label{defHs}
H_{K,E}(x):=\frac12\,
\int_{\R^n} \big(\chi_{\R^n\setminus E}(y)-\chi_E(y)\big)\,K(x-y)\,dy.\end{equation}
Notice that the above integral is taken in the principal value sense.

The classical mean curvature of~$E$
will be denoted by~$H_E$.
We define
\begin{equation}\label{defc}
{c_{K,E}}(x):=\sqrt{ \frac12\,\int_{\partial E} \big( \nu_{E}(x)-
\nu_{E}(y)\big)^2\,K(x-y)\,d\HH_y },\end{equation}
being~$\nu_E=(\nu_{E,1},\dots,\nu_{E,n})$ the exterior unit normal of~$E$. The quantity~${c_{K,E}}$
plays in our setting the role played by the norm of the second
fundamental form in the classical case, and we can consider it
the $K$-total curvature of~$E$.

We also define the (minus) $K$-Laplace-Beltrami operator along~$\partial E$
of a function~$f$ by
\begin{equation}\label{defop}
{L_{K,E}} f(x):=\int_{\partial E} \big(f(x)-f(y)\big)\,K(x-y)\,d\HH_y.\end{equation} 

As customary, we consider the tangential derivative
\begin{equation}\label{DEL} 
\delta_{E,i} f(x):= \partial_i f(x)-\nu_{E,i}(x)\,\nabla f(x)\cdot \nu_{E}(x)\end{equation}
and we recall that
\begin{equation}\label{SYMMECOR}\delta_{E,i}\nu_{E,j}=\delta_{E,j}\nu_{E,i},\end{equation}
see e.g.
formula~(10.11) in~\cite{MR0775682}.

In this setting, our nonlocal formula inspired by Simons' Identity goes as follows:

\begin{thm}\label{simons}
Let~$K$ be as in~\eqref{ipotesi}.
Let~$E\subset\R^n$ with smooth boundary and~$x\in\partial E$ with~$\nu_{E}(x)=(0,\dots,0,1)$.

Assume that there exist~$R_0>0$ and~$\beta\in[0, n+s)$ such that
for all~$R\ge R_0$ it holds that
\begin{equation}\label{DEA:PA:0}
\int_{\partial E\cap B_R(x)} \big(|H_E(y)|+1\big)\,d{\mathcal{H}}^{n-1}_y\le CR^\beta,
\end{equation}
for some~$C>0$.

Then, for any~$i$, $j\in\{1,\dots,n-1\}$
it holds that
\begin{equation}\label{SIM:FORNEW}\begin{split}
\delta_{E,i} \delta_{E,j} H_{K,E}(x) \,=\,& -{L_{K,E}} \delta_{E,j} \nu_{E,i} (x) +
c^2_{K,E}(x)\, \delta_{E,j} \nu_{E,i}(x) \\&
-\int_{\partial E} \Big( H_E(y)K(x-y)  -\nu_{E}(y) \cdot \nabla K(x-y) \Big)
\nu_{E,i}(y)\, \nu_{E,j}(y)\,d\HH_y .\end{split}
\end{equation}
\end{thm}

The proof of Theorem~\ref{simons} will be given in detail in Section~\ref{PROOFSI}.

It is interesting to remark that the result
of Theorem~\ref{simons} ``passes to the limit efficiently
and localizes'': for instance, if one takes~$\rho\in C^\infty_0([-1,1])$,
$\e>0$ and a kernel of the form~$K_\e(x):=\e^{-n-2}\rho(|x|/\e)$,
then, using Theorem~\ref{simons} and sending~$\e\searrow0$,
one recovers the classical Simons' Identity
in~\cite{MR0233295} (such passage to the limit can be performed
e.g. with the analysis in~\cite{MR3230079} and Appendix~C in~\cite{MR3798717}).

The details\footnote{We also remark that condition~\eqref{DEA:PA:0} is obviously
satisfied with~$\beta: =n-1$ when the set~$E$ is smooth and bounded.
For minimizers, and, more generally, stable critical points, of the nonlocal perimeter functional, one still has perimeter estimates (see formula~(1.16)
in Corollary~1.8 of~\cite{MR3981295}): however, in this general case,
estimating the mean curvature, or, in greater generality, the ``second derivatives'' of the set, may be a demanding task, see~\cite{HAR} for some results in this direction.} on how to reconstruct the classical Simons' Identity in the appropriate
limit are given in Section~\ref{APPENDICE A}.

\subsection{Back to the original Simons' Identity}\label{APPENDICE A}

As mentioned above, our nonlocal
formula~\eqref{SIM:FORNEW}
in Theorem~\ref{simons} recovers, in the limit, the original
Simons' Identity proved in~\cite{MR0233295}. The precise result goes as follows:

\begin{thm}\label{ILTKC}
Let~$E\subset\R^n$ and~$x\in\partial E$.
Assume that there exist~$R_0>0$ and~$\beta\in[0, n+1)$ such that
for all~$R\ge R_0$ it holds that
\begin{equation}\label{DEA:PA}
\int_{\partial E\cap B_R(x)} \big(|H_E(y)|+1\big)\,d{\mathcal{H}}^{n-1}_y\le CR^\beta,
\end{equation}
for some~$C>0$. Then, the identity in~\eqref{ORIGINAL SIM}
holds true as a consequence of formula~\eqref{SIM:FORNEW}.
\end{thm}

The proof of Theorem~\ref{ILTKC} is contained in Section~\ref{OJSDLN-wf}.

We point out that Theorems~\ref{simons} and~\ref{ILTKC} also provide
a new proof of the original Simons' Identity.
Remarkably, our proof relies less on the differential geometry structure
of the hypersurface and it is, in a sense, ``more extrinsic'':
these facts allow us to exploit similar methods also
for the case of 
integrodifferential equations, as will be done in the forthcoming
Section~\ref{IDE2}.

\subsection{The case of integrodifferential equations}\label{IDE2}

The framework that we provide here is
a suitable modification of that given in Section~\ref{PRE1}
for sets. The idea is to ``substitute'' the volume measure~$\chi_E(x)\,dx$
with~$u(x)\,dx$ and the area measure~$\chi_{\partial E}(x)d\HH_x$
with~$|\nabla u(x)|\,dx$. However, one cannot really exploit
the setting of Section~\ref{PRE1} as it is
also for integrodifferential equations, and it is necessary
to ``redo the computation'', so to extrapolate the correct
operators and stability conditions for the solutions.

The technical details go as follows. 
Though more general cases can be considered, for the sake of concreteness, we focus on a kernel~$K$ satisfying
\begin{equation}\label{ipotesi2}
\begin{split}
& K\in C^1(\R^n)\cap L^1(\R^n),\\
&|\nabla K|\in L^1(\R^n),\\
{\mbox{and}}\qquad & K(x)=K(-x).
\end{split}
\end{equation}

Given a function~$u\in W^{1,\infty}(\R^n)$ whose level sets~$\{u=t\}$ are
smooth for a.e.~$t\in\R$,
we define the $K$-mean curvature of~$u$
at~$x\in\R^n$ by
\begin{equation}\label{defHs:2-NN}
H_{K,u}(x):= C_K-
\int_{\R^n} u(y)\,K(x-y)\,dy,\qquad{\mbox{where }}\;
C_K:=\frac12\,\int_{\R^n} K(y)\,dy.
\end{equation}
The setting in~\eqref{defHs:2-NN} has to be compared with~\eqref{defHs}
and especially with the forthcoming formula~\eqref{defHs:2}. The classical mean curvature
of the level sets of~$u$ will be denoted by~$H_u$
(i.e., if~$t_x:=u(x)$,
then~$H_{u}(x)$ is the classical mean curvature of the set~$\{ u>t_x\}$
at~$x$).

We also define the the $K$-total curvature of~$u$ as
\begin{equation}\label{defc-NN}
{c_{K,u}}(x):=\sqrt{ \frac12\,\int_{\R^n} \big( \nu_{u}(x)-
\nu_{u}(y)\big)^2\,K(x-y)\,d\mu_{u,y} },\end{equation}
being~$\nu_{u}(x)$ the exterior unit normal of the level set of~$u$
passing through~$x$ (i.e., if~$t_x:=u(x)$,
then~$\nu_{u}(x)$ is the exterior normal of the set~$\{ u>t_x\}$
at~$x$). In~\eqref{defc-NN}, we also used the notation
\begin{equation}\label{MU}
d\mu_{u,y} :=|\nabla u(y)|\,dy.\end{equation}
Of course, the definition in~\eqref{defc-NN} has to be compared with
that in~\eqref{defc}.
Moreover, by construction we have that
\begin{equation} \label{NUFUP}\nu_{u}(x)=-\frac{\nabla u(x)}{|\nabla u(x)|},\end{equation}
the minus sign coming from the fact that the external derivative
of~$\{u>t_x\}$ points towards points with ``decreasing values'' of~$u$.

We also define the $K$-Laplace-Beltrami operator induced by~$u$
acting on a function~$f$ by
\begin{equation}\label{defop-NN}
{L_{K,u}} f(x):=\int_{\R^n} \big(f(x)-f(y)\big)\,K(x-y)\,d\mu_{u,y}.\end{equation} 
Once again, one can compare~\eqref{defop}
and~\eqref{defop-NN}. Also, we denote by~$\delta_{u,i}$ the tangential
derivatives along the level sets of~$u$ (recall~\eqref{DEL}).
This setting turns out to be
the appropriate one to translate Theorem~\ref{simons}
into a result for solutions of integrodifferential
equations, as will be presented in the forthcoming result:

\begin{thm}\label{simons-NN}
Let~$K$ be as in~\eqref{ipotesi2}.
Let~$u\in W^{1,\infty}(\R^n)$ and assume
that~$\{u=t\}$ is a
smooth hypersurface with bounded mean curvature for a.e.~$t\in\R$.
For any~$x\in\R^n$ with~$\nu_{u}(x)=(0,\dots,0,1)$
and any~$i$, $j\in\{1,\dots,n-1\}$,
it holds that
\begin{equation}\label{duji483765bv9875bv76c98uoixro}\begin{split}
\delta_{u,i} \delta_{u,j} H_{K,u}(x) \,=\,& -{L_{K,u}} \delta_{u,j} \nu_{u,i} (x) +
c_{K,u}^2(x)\, \delta_{u,j} \nu_{u,i}(x) \\&
-\int_{\R^n} \Big( H_u(y)K(x-y)  -\nu_{u}(y) \cdot \nabla K(x-y) \Big)
\nu_{u,j}(y)\, \nu_{u,i}(y)\,d\mu_{u,y} .\end{split}\end{equation}
\end{thm}

The proof of Theorem~\ref{simons-NN} is a careful variation
of that of Theorem~\ref{simons}, but, for the sake
of clarity, we provide full details in Section~\ref{ALSMsd}.
We also observe that the choice~$u:=\chi_E$ would formally allow one to recover
Theorem~\ref{simons} from Theorem~\ref{simons-NN}.

\subsection{Stable sets}

In the study of variational problems, a special role is played by the ``stable'' critical points, i.e.
those critical points at which the second derivative of the energy functional is nonnegative definite, see e.g.~\cite{MR3838575}.

In this spirit, in the study of nonlocal minimal surfaces we say that~$E$ is a stable set in~$\Omega$ if~$H_{K,E}(x)=0$ for any~$x\in\Omega\cap\partial E$
and
\begin{equation}\label{ojwfe034}
\frac12\int_{\partial E}\int_{\partial E}(f(x)-f(y))^2\, K(x-y)\,d{\mathcal{H}}^{n-1}_x\,d{\mathcal{H}}^{n-1}_y-
\int_{\partial E} c_{K,E}^2(x)\,f^2(x)\,d{\mathcal{H}}^{n-1}_x\ge0\end{equation}
for any~$f\in C^\infty_0(\Omega)$.

In connection with this, we set
$$ B_{K,E}(u,v ; x) := \frac1 2 \int_{\partial E}(u(x)-u(y))(v(x)-v(y))\,K(x-y)\,d \mathcal H^{n-1}_y,$$
where the integral is taken in the principal value sense, and
$$ B_{K,E} (u,v):= \int_{\partial E} B_{K,E}(u,v ; x) \,d \mathcal H^{n-1}_x . $$
In this notation, the first term in~\eqref{ojwfe034} takes the form~$B_K(f,f)$.

Also, we consider the integrodifferential operator \(L_{K,E}\) previously introduced in~\eqref{defop}. When~$K(x)=\frac1{|x|^{n+s}}$, this operator reduces to the fractional Laplacian, up to normalizing constants.

With this notation, we have:
\begin{thm}\label{Jfqydwvfbe923ejfn}
Let~$E\subset\R^n$ with smooth boundary. Then, for all \(\eta \in C^\infty_0(\partial E)\), \begin{align*}
-\int_{\partial E} \bigg \{  \frac12 L_{K,E} c_{K,E}^2(x) +B_{K,E}(c_{K,E},c_{K,E};x) - &c_{K,E}^4(x) \bigg \} \eta^2(x) \,d \mathcal H^{n-1}_x \\
&\leqslant \int_{\partial E} c_{K,E}^2(x) B_{K,E}(\eta,\eta ; x ) \,d \mathcal H^{n-1}_x .
\end{align*}
\end{thm}

For the classical counterpart of the above in equality, see e.g.~\cite[equation~(19)]{MR3838575}.
\medskip

The rest of this paper contains the proofs of the results stated above.
Before undertaking the details of the proofs, we mention that the idea of recovering classical results in geometry as a limit of fractional ones, thus providing a unified approach between different disciplines, can offer interesting perspectives (for instance, we will investigate the Ricci curvature from this point of view in the forthcoming article~\cite{DAFA}; see also~\cite{MR4685916} for limit formulas related to trace problems
and~\cite{KAG} for a recovery technique of the Divergence Theorem coming from a nonlocal perspective).

\section*{Acknowledgements}

JT is supported by an Australian Government Research Training Program Scholarship. EV is supported by the Australian Laureate Fellowship FL190100081 {\em Minimal surfaces, free
boundaries, and partial differential equations}. It is a pleasure to thank Joaquim Serra
for many inspiring discussions.

\section{Proof of Theorem~\ref{simons}}\label{PROOFSI}

Up to a translation, we can suppose that~$0\in\partial E$ and prove
Theorem~\ref{simons} at the origin, hence
we can choose coordinates such that
\begin{equation}\label{ZEROZ}
\nu_{E}(0)=(0,\dots,0,1). 
\end{equation}

We point out that assumption~\eqref{DEA:PA:0} guarantees that all the terms in~\eqref{SIM:FORNEW} are finite, see e.g. the forthcoming technical calculation in~\eqref{d9o3285v456746578fdghdshgxcvdheritv}.

Moreover, we take~$K$ to be smooth, compactly supported and nonsingular, so to be able to
take derivatives inside the integral (the general case then
follows by approximation, see e.g.~\cite{MR3322379}).
In this way, we rewrite~\eqref{defHs} as
\begin{equation}\label{defHs:2}
H_{K,E}(x)= C_K-
\int_{E} K(x-y)\,dy,\qquad{\mbox{where }}\;
C_K:=\frac12\,\int_{\R^n} K(y)\,dy.
\end{equation}
Also, this is a good definition for all~$x\in\R^n$ (and not only for~$x\in\partial E$),
so we can consider the full gradient of such an expression.
Moreover, for a fixed~$x\in\R^n$, we use the notation
\begin{equation}\label{NOTA1}
\phi(y):=K(x-y).
\end{equation}
In this way, we have that, for any~$\ell\in\{1,\dots,n\}$,
\begin{equation}\label{NOTA2} \partial_\ell K(x-y)=-\partial_\ell\phi(y).\end{equation}
Exploiting this,~\eqref{defHs:2} and the Gauss-Green
Theorem, we see that, for any~$\ell\in\{1,\dots,n\}$,
\begin{equation*}
\begin{split}
& \partial_\ell H_{K,E}(x)=-\int_{E} \partial_\ell K(x-y)\,dy
=\int_{E} \partial_\ell \phi(y)\,dy
=\int_{E} {\rm div}\big(\phi(y) e_\ell\big)\,dy\\
&\qquad
=\int_{\partial E} \nu_{E}(y)\cdot\big(\phi(y) e_\ell\big)\,d\HH_y
=\int_{\partial E} \nu_{E,\ell}(y)\,K(x-y)\,d\HH_y.
\end{split}\end{equation*}
This gives that, for any~$x\in\partial E$,
\begin{equation}\label{UNO}
\nabla H_{K,E}(x)= \int_{\partial E} \nu_{E}(y)\,K(x-y)\,d\HH_y.
\end{equation}
In addition, from~\eqref{defc},
\begin{equation}\label{UNOprimo}
\begin{split}
c^2_{K,E}(x)\,&=\frac12\,\int_{\partial E} \big( \nu_{E}(x)-\nu_{E}(y)\big)^2\,K(x-y)\,d\HH_y\\&
= \int_{\partial E} K(x-y)\,d\HH_y-
\nu_{E}(x)\cdot\int_{\partial E}\nu_{E}(y)\,K(x-y)\,d\HH_y.\end{split}
\end{equation}

Now, we fix the indices~$i$, $j\in\{1,\dots,n-1\}$
and we make use of~\eqref{DEL} and~\eqref{UNO} to find that
\begin{equation}\label{DUE}
\begin{split}
\delta_{E,i} H_{K,E}(x)&=
\partial_i H_{K,E}(x)-\nu_{E,i}(x)\,\nabla H_{K,E}(x)\cdot \nu_{E}(x)\\
&=\int_{\partial E} \nu_{E,i}(y)\,K(x-y)\,d\HH_y
-\nu_{E,i}(x)\,\nu_{E}(x)\cdot\int_{\partial E} \nu_{E}(y)\,K(x-y)\,d\HH_y.
\end{split}\end{equation}
We take another tangential derivative of~\eqref{DUE}
and evaluate it at the origin, recalling~\eqref{ZEROZ}
(which, in particular, gives that~$\nu_{E,i}(0)=0=\nu_{E,j}(0)$
for any~$i$, $j\in\{ 1,\dots,n-1\}$).
In this way, recalling~\eqref{DEL}, we obtain that
\begin{equation}\label{TRE}
\begin{split}&
\delta_{E,j}\delta_{E,i} H_{K,E}(0)\\=\;&
\partial_j\delta_{E,i} H_{K,E}(0)
\\ =\;& \partial_j \left[
\int_{\partial E} \nu_{E,i}(y)\,K(x-y)\,d\HH_y
-\nu_{E,i}(x)\,\nu_{E}(x)\cdot\int_{\partial E} \nu_{E}(y)\,K(x-y)\,d\HH_y\right]_{x=0}
\\ =\;&
\int_{\partial E} \nu_{E,i}(y)\,\partial_j K(-y)\,d\HH_y
-\partial_j\nu_{E,i}(0)\,\nu_{E}(0)\cdot\int_{\partial E} \nu_{E}(y)\,K(-y)\,d\HH_y.
\end{split}\end{equation}
Also, using the notation in~\eqref{NOTA1} and~\eqref{NOTA2}
with~$x:=0$ and~\eqref{DEL}, we see that
\begin{equation}\label{XTRE}
\begin{split} &
\int_{\partial E} \nu_{E,i}(y)\,\partial_j K(-y)\,d\HH_y\\ 
=\;&-\int_{\partial E} \nu_{E,i}(y)\,\partial_j \phi(y)\,d\HH_y\\
=\;&-\int_{\partial E} \nu_{E,i}(y)\,\delta_{E,j} \phi(y)\,d\HH_y-
\int_{\partial E} \nu_{E,i}(y)\,\nu_{E,j}(y)\,\nabla \phi(y)\cdot \nu_{E}(y)\,d\HH_y.
\end{split}\end{equation}

Now we recall an integration by parts formula for tangential derivatives
(see e.g. the first formula\footnote{We stress that the normal on page~122 of~\cite{MR0775682}
is internal, according to the distance setting on page~120 therein. This causes
in our notation a sign change with respect to the setting in~\cite{MR0775682}.
Also, in the statement of Lemma~10.8 on page~121 in~\cite{MR0775682} there is a typo
(missing a mean curvature inside an integral). We also observe that formula~\eqref{DIV:TH}
can also be seen as a version of the Tangential Divergence Theorem, see e.g.
Appendix~A in~\cite{MR2024995}.}
in display on page~122 of~\cite{MR0775682}),
namely
\begin{equation}\label{DIV:TH}
\int_{\partial E} \delta_{E,j} f(y)\,d\HH_y=
\int_{\partial E} H_E(y)\,\nu_{E,j}(y)\,f(y)\,d\HH_y,
\end{equation}
being~$H_E$ the classical mean curvature of~$\partial E$.
Applying this formula to the product of two functions, we find that
\begin{equation}\label{EQ:6}
\begin{split}
\int_{\partial E} \delta_{E,j} f(y)\,g(y)\,d\HH_y+&
\int_{\partial E} f(y)\,\delta_{E,j} g(y)\,d\HH_y\,\\
&
=
\int_{\partial E} \delta_{E,j} (fg)(y)\,d\HH_y
\\ &=
\int_{\partial E} H_E(y)\,\nu_{E,j}(y)\,f(y)\,g(y)\,d\HH_y.
\end{split}\end{equation}
Using this and~\eqref{NOTA1} (with~$x:=0$ here), we see that
\begin{eqnarray*}&&
-\int_{\partial E} \nu_{E,i}(y)\,\delta_{E,j} \phi(y)\,d\HH_y\\&=&
\int_{\partial E} \delta_{E,j}\nu_{E,i}(y)\,\phi(y)\,d\HH_y
-\int_{\partial E} H_E(y)\,\nu_{E,i}(y)\,\nu_{E,j}(y)\,\phi(y)\,d\HH_y\\
&=&
\int_{\partial E} \delta_{E,j}\nu_{E,i}(y)\,K(-y)\,d\HH_y
-\int_{\partial E} H_E(y)\,\nu_{E,i}(y)\,\nu_{E,j}(y)\,K(-y)\,d\HH_y.\end{eqnarray*}
So, we insert this information into~\eqref{XTRE} and we conclude that
\begin{equation*}
\begin{split}
 \int_{\partial E} \nu_{E,i}(y)& \,\partial_j K(-y)\,d\HH_y\\
&=
\int_{\partial E} \delta_{E,j}\nu_{E,i}(y)\,K(-y)\,d\HH_y
-\int_{\partial E} H_E(y)\,\nu_{E,i}(y)\,\nu_{E,j}(y)\,K(-y)\,d\HH_y\\&\qquad+
\int_{\partial E} \nu_{E,i}(y)\,\nu_{E,j}(y)\,\nabla K(-y)\cdot \nu_{E}(y)\,d\HH_y.
\end{split}\end{equation*}
Plugging this into~\eqref{TRE}, we get that
\begin{equation}\label{X4}\begin{split}
\delta_{E,j}\delta_{E,i} H_{K,E}(0)=&
\int_{\partial E} \delta_{E,j}\nu_{E,i}(y)\,K(-y)\,d\HH_y\\
&\qquad\quad-\int_{\partial E} H_E(y)\,\nu_{E,i}(y)\,\nu_{E,j}(y)\,K(-y)\,d\HH_y\\ &\qquad\quad+
\int_{\partial E} \nu_{E,i}(y)\,\nu_{E,j}(y)\,\nabla K(-y)\cdot \nu_{E}(y)\,d\HH_y\\ &\qquad\quad
-\partial_j\nu_{E,i}(0)\,\nu_{E}(0)\cdot\int_{\partial E} \nu_{E}(y)\,K(-y)\,d\HH_y.
\end{split}\end{equation}

In addition, from~\eqref{UNOprimo},
$$ \partial_j\nu_{E,i}(0)\, c^2_{K,E}(0)
=\int_{\partial E} \partial_j\nu_{E,i}(0)\,K(-y)\,d\HH_y-
\partial_j\nu_{E,i}(0)\,\nu_{E}(0)\cdot\int_{\partial E} \nu_{E}(y)\,K(-y)\,d\HH_y.
$$
Comparing with~\eqref{X4}, we conclude that
\begin{equation*}\begin{split}
\delta_{E,j}\delta_{E,i} H_{K,E}(0)=&
\int_{\partial E} \Big(\delta_{E,j}\nu_{E,i}(y)-\delta_{E,j}\nu_{E,i}(0)\Big)\,K(-y)\,d\HH_y\\ &\qquad
-\int_{\partial E} H_E(y)\,\nu_{E,i}(y)\,\nu_{E,j}(y)\,K(-y)\,d\HH_y\\ &\qquad+
\int_{\partial E} \nu_{E,i}(y)\,\nu_{E,j}(y)\,\nabla K(-y)\cdot \nu_{E}(y)\,d\HH_y
+\partial_j\nu_{E,i}(0)\, c^2_{K,E}(0).
\end{split}\end{equation*}
{F}rom this identity and the definition in~\eqref{defop}, the desired
result plainly follows.\hfill$\Box$

\section{Proof of Theorem~\ref{simons-NN}}\label{ALSMsd}

The proof is similar to that of Theorem~\ref{simons}. Full details are provided for the reader's facility.
Up to a translation, we can prove
Theorem~\ref{simons-NN} at the origin and suppose that
\begin{equation}\label{ZEROZ-NN}
\nu_{u}(0)=(0,\dots,0,1). 
\end{equation}
We observe that our assumptions on the kernel in~\eqref{ipotesi2} yield that all
the terms in~\eqref{duji483765bv9875bv76c98uoixro} are finite.

Using~\eqref{defHs:2-NN},
\eqref{MU} and~\eqref{NUFUP}, we see that, for any~$x\in\R^n$,
\begin{equation}\label{UNO-NN}
\begin{split}
&\nabla H_{K,u}(x)= \nabla\left(C_K-
\int_{\R^n} u(x-y)\,K(y)\,dy\right)
=-\int_{\R^n} \nabla u(x-y)\,K(y)\,dy
\\ &\qquad=-\int_{\R^n} \nabla u(y)\,K(x-y)\,dy
=\int_{\R^n} \nu_{u}(y)\,K(x-y)\,d\mu_{u,y}.
\end{split}\end{equation}
In addition, from~\eqref{defc-NN},
\begin{equation}\label{UNOprimo-NN}
\begin{split}
c^2_{K,u}(x)\,&=
\frac12\,\int_{\R^n} \big( \nu_{u}(x)-\nu_{u}(y)\big)^2\,K(x-y)\,d\mu_{u,y}
\\&
= 
\int_{\R^n} K(x-y)\,d\mu_{u,y}
-\nu_{u}(x)\cdot\int_{\R^n} \nu_{u}(y)\,K(x-y)\,d\mu_{u,y}.\end{split}
\end{equation}
Also, in view of~\eqref{DEL} and~\eqref{UNO-NN},
\begin{equation}\label{DUE-NN}
\begin{split}
\delta_{u,i} H_{K,u}(x)&=
\partial_i H_{K,u}(x)-\nu_{u,i}(x)\,\nabla H_{K,u}(x)\cdot \nu_{u}(x)\\
&=\int_{\R^n} \nu_{u,i}(y)\,K(x-y)\,d\mu_{u,y}
-\nu_{u,i}(x)\,\nu_{u}(x)\cdot\int_{\R^n} \nu_{u}(y)\,K(x-y)\,d\mu_{u,y}.
\end{split}\end{equation}
Consequently, using~\eqref{ZEROZ-NN} and~\eqref{DUE-NN}, for all~$i$, $j\in\{1,\dots,n-1\}$,
\begin{equation}\label{TRE-NN}
\begin{split}&
\delta_{u,j}\delta_{u,i} H_{K,u}(0)\\
=\;&
\partial_j\delta_{u,i} H_{K,u}(0)
\\ =\;& \partial_j \left[
\int_{\R^n} \nu_{u,i}(y)\,K(x-y)\,d\mu_{u,y}
-\nu_{u,i}(x)\,\nu_{u}(x)\cdot\int_{\R^n} \nu_{u}(y)\,K(x-y)\,d\mu_{u,y}
\right]_{x=0}
\\ =\;& 
\int_{\R^n} \nu_{u,i}(y)\,\partial_j K(-y)\,d\mu_{u,y}
-\partial_j\nu_{u,i}(0)\,\nu_{u}(0)\cdot\int_{\R^n} \nu_{u}(y)\,K(-y)\,d\mu_{u,y}
.\end{split}\end{equation}

Now, recalling
the notation in~\eqref{NOTA1} and~\eqref{NOTA2}
with~$x:=0$ and~\eqref{DEL}, we obtain that
\begin{equation}\label{XTRE-NN}
\begin{split}&
\int_{\R^n} \nu_{u,i}(y)\,\partial_j K(-y)\,d\mu_{u,y}\\
=\;&-\int_{\R^n} \nu_{u,i}(y)\,\partial_j \phi(y)\,d\mu_{u,y}
\\
=\;&-\int_{\R^n} \nu_{u,i}(y)\,\delta_{u,j} \phi(y)\,d\mu_{u,y}-
\int_{\R^n} \nu_{u,i}(y)\,\nu_{u,j}(y)\,\nabla \phi(y)\cdot \nu_{u}(y)\,d\mu_{u,y}.
\end{split}\end{equation}
Furthermore,
exploiting the Coarea Formula twice
and the tangential integration
by parts identity in~\eqref{EQ:6}, we obtain that
\begin{eqnarray*}
&&
-\int_{\R^n} \nu_{u,i}(y)\,\delta_{u,j} \phi(y)\,d\mu_{u,y}
= -\int_{\R^n} |\nabla u(y)|\,\nu_{u,i}(y)\,\delta_{u,j} \phi(y)\,dy\\
&&\qquad=
-\int_{\R} \int_{\{u(y)=t\}} 
\nu_{u,i}(y)\,\delta_{u,j} \phi(y)\,d\HH_y\,dt
\\ &&\qquad=
\int_{\R} \int_{\{u(y)=t\}} 
\delta_{u,j}\nu_{u,i}(y)\, \phi(y)\,d\HH_y \,dt\\&&\qquad\qquad\qquad
-
\int_{\R} \int_{\{u(y)=t\}} 
H_u(y)\,\nu_{u,i}(y)\,\nu_{u,j}(y)\,\phi(y)\,d\HH_y \,dt
\\ &&\qquad=
\int_{\R^n}|\nabla u(y)|\,
\delta_{u,j}\nu_{u,i}(y)\, \phi(y)\,dy
-\int_{\R^n}|\nabla u(y)|\,
H_u(y)\,\nu_{u,i}(y)\,\nu_{u,j}(y)\,\phi(y)\,dy
\\ &&\qquad=
\int_{\R^n}
\delta_{u,j}\nu_{u,i}(y)\, \phi(y)\,d\mu_{u,y}
-\int_{\R^n}H_u(y)\,\nu_{u,i}(y)\,\nu_{u,j}(y)\,\phi(y)\,d\mu_{u,y}
\\ &&\qquad=
\int_{\R^n}
\delta_{u,j}\nu_{u,i}(y)\, K(-y)\,d\mu_{u,y}
-\int_{\R^n}H_u(y)\,\nu_{u,i}(y)\,\nu_{u,j}(y)\,K(-y)\,d\mu_{u,y}.\end{eqnarray*}
We can now insert this identity into~\eqref{XTRE-NN} and we get that
\begin{equation*}
\begin{split}
\int_{\R^n} &\nu_{u,i}(y)\,\partial_j K(-y)\,d\mu_{u,y}
\\&=
\int_{\R^n} \delta_{u,j}\nu_{u,i}(y)\,K(-y)\,d\mu_{u,y}
-\int_{\R^n} H_u(y)\,\nu_{u,i}(y)\,\nu_{u,j}(y)\,K(-y)\,d\mu_{u,y} \\&\qquad +
\int_{\R^n} \nu_{u,i}(y)\,\nu_{u,j}(y)\,\nabla K(-y)\cdot \nu_{u}(y)\,d\mu_{u,y}.
\end{split}\end{equation*}
Plugging this into~\eqref{TRE-NN} we get that
\begin{equation}\label{X4-NN}\begin{split}
&\hspace{-2em}\delta_{u,j}\delta_{u,i} H_{K,u}(0) \\
&=
\int_{\R^n} \delta_{u,j}\nu_{u,i}(y)\,K(-y)\,d\mu_{u,y}
-\int_{\R^n} H_u(y)\,\nu_{u,i}(y)\,\nu_{u,j}(y)\,K(-y)\,d\mu_{u,y}\\ &\qquad+
\int_{\R^n} \nu_{u,i}(y)\,\nu_{u,j}(y)\,\nabla K(-y)\cdot \nu_{u}(y)\,d\mu_{u,y}
\\&\qquad -\partial_j\nu_{u,i}(0)\,\nu_{u}(0)\cdot\int_{\R^n} \nu_{u}(y)\,K(-y)\,d\mu_{u,y}.
\end{split}\end{equation}

Also, from~\eqref{UNOprimo-NN}, we have that
$$ \partial_j\nu_{u,i}(0)\, c^2_{K,u}(0)
=\int_{\R^n} \partial_j\nu_{u,i}(0)\,K(-y)\,d\mu_{u,y}-
\partial_j\nu_{u,i}(0)\,\nu_{u}(0)\cdot\int_{\R^n} \nu_{u}(y)\,K(-y)\,d\mu_{u,y}.
$$
Hence, from this and~\eqref{X4-NN}, we conclude that
\begin{equation*}\begin{split}
&\delta_{u,j}\delta_{u,i} H_{K,u}(0)\\ &=
\int_{\R^n} \Big(\delta_{u,j}\nu_{u,i}(y)-\delta_{u,j}\nu_{u,i}(0)\Big)\,K(-y)
\,d\mu_{u,y}
-\int_{\R^n} H_u(y)\,\nu_{u,i}(y)\,\nu_{u,j}(y)\,K(-y)\,d\mu_{u,y}
\\ &\qquad\quad+
\int_{\R^n} \nu_{u,i}(y)\,\nu_{u,j}(y)\,\nabla K(-y)\cdot \nu_{u}(y)\,d\mu_{u,y}
+\partial_j\nu_{u,i}(0)\, c^2_{K,u}(0).
\end{split}\end{equation*}
This and~\eqref{defop-NN} give the desired
result.~\hfill$\Box$

\section{Proof of Theorem~\ref{ILTKC}}\label{OJSDLN-wf}
For clarity, we denote by~$\Delta_{\partial E}$
the Laplace-Beltrami operator on the hypersurface~$\partial E$,
by~$\delta_{k,E}$ the tangential derivative
in the $k$th coordinate direction, by~$\nu_E$ the external derivative and by~$c_E$ the norm of the second fundamental form.

To obtain~\eqref{ORIGINAL SIM} as a limit of~\eqref{SIM:FORNEW},
we focus on a special kernel. Namely,
given~$\e>0$, we let
\begin{equation}\label{KERNTOSIM} K_\e(y):=\frac{\e}{|y|^{n+1-\e}}.\end{equation}

We now recall a simple, explicit calculation:
\begin{eqnarray}
&& \label{0191:0}\int_{B_1} x_1^4\,dx=\frac{3\, {\mathcal{H}}^{n-1}(S^{n-1}) }{n(n+2)(n+4) }\\ {\mbox{and }}\,
\label{0191:1} &&\int_{S^{n-1}} \vartheta_1^4 \,d{\mathcal{H}}^{n-1}_\vartheta= 
\frac{3\, {\mathcal{H}}^{n-1}(S^{n-1}) }{n(n+2) } .
\end{eqnarray}
Not to interrupt the flow of the arguments, we postpone 
the proof of formulas~\eqref{0191:0} and~\eqref{0191:1} to Appendix~\ref{jxiewytdenfrhgfndbvhfdnbvfdefgbhwektheru}.

To complete the proof of Theorem~\ref{ILTKC}, without loss of generality, we assume that~$0=x\in\partial E$
and that~$\partial E\cap B_{r_0}$ is the graph of a function~$f:\R^{n-1}\to\R$
with vertical normal, hence~$f(0)=0$ and~$\partial_i f(0)=0$
for all~$i\in\{1,\dots,n-1\}$.
We can also diagonalize the Hessian matrix of~$f$ at~$0$, and obtain that
the mean curvature~$H_E$ at the origin coincides with the trace\footnote{We stress
that we are not dividing the quantity in~\eqref{DEF H COS} by~$n-1$,
to be consistent with the notation in formula~(10.12)
in~\cite{MR0775682}.}
of such matrix,
namely
\begin{equation}\label{DEF H COS} H_E(0)= -\big(
\partial^2_1 f(0)+\dots+\partial^2_{n-1} f(0) \big).\end{equation}
The sign convention here is inferred by the assumption that~$E$ is locally the subgraph
of~$f$ and the normal is taken to point outwards.
Consequently, for every~$y=(y',f(y'))\in\partial E\cap B_{r_0}$,
\begin{equation}\label{0-120-A5bis}\begin{split}
&f(y')=\frac12\sum_{i=1}^{n-1} \partial^2_i f(0)\,y_i^2+
O(|y'|^3),\\
&\nabla f(y')=(\partial_1 f(y'),\dots,
\partial_{n-1} f(y'))=\big( \partial^2_1 f(0)\,y_1,\dots,
\partial^2_{n-1} f(0)\,y_{n-1}\big)+
O(|y'|^2)\end{split}\end{equation}
and
\begin{equation}\label{SVIL-NU}
\begin{split}&
\nu_E(y) =\frac{(-\nabla f(y'), 1)}{\sqrt{1+|\nabla f(y')|^2}}=(-\nabla f(y'), 1)+O(|y'|^2)
\\&\qquad=
\big( -\partial^2_1 f(0)\,y_1,\dots,
-\partial^2_{n-1} f(0)\,y_{n-1},1\big)+
O(|y'|^2)\big).\end{split}
\end{equation}
Here, the notation \(g = O(h(\vert y' \vert))\) means that \(\vert g \vert \le 
C \vert h(\vert y' \vert) \vert\) for \(\vert y' \vert \) sufficient close to \(0\) with \(C\) independent of \(\varepsilon\), that is, \(g\) is \emph{uniformly in }\(\varepsilon\) big O of \(h\) as \(\vert y'\vert \to 0\).  As a consequence,
for every~$y=(y',f(y'))\in\partial E\cap B_{r_0}$,
\begin{equation} \label{SVIPRO}|y|^2=|y'|^2+|f(y')|^2=|y'|^2+O(|y'|^4)=
|y'|^2\big(1+O(|y'|^2)\big),\end{equation}
and, for any~$i,j\in\{1,\dots,n-1\}$,
\begin{equation} \label{DOPPIO NU}\nu_{E,j}(y)\nu_{E,i}(y)=
\partial^2_{j} f(0)\partial^2_{i} f(0) y_j y_i
+
O(|y'|^3) .\end{equation}
Thus, using~\eqref{SVIPRO}, we see that, for any fixed~$\alpha\in\R$,
\begin{equation} \label{SVIPRO1}|y|^\alpha=
|y'|^\alpha\big(1+O(|y'|^2)\big)^{\alpha/2}=
|y'|^\alpha\big(1+O(|y'|^2)\big).\end{equation}

Then, from~\eqref{SVIL-NU} and~\eqref{SVIPRO1},
we obtain that, for any~$\ell\in\{1,\dots,n-1\}$ and~$y\in\partial E\cap B_{r_0}$,
\begin{eqnarray*}
\nu_{E,\ell}(y)\partial_\ell K_\e(-y)&=&-
\frac{(n+1-\e)\e \partial^2_\ell f(0)\,y_\ell^2}{|y|^{n+3-\e}}+\e O\left( |y'|^{\e-n}\right)\\&=&-
\frac{(n+1-\e)\e \partial^2_\ell f(0)\,y_\ell^2}{|y'|^{n+3-\e}}+\e O\left( |y'|^{\e-n}\right)
\end{eqnarray*}
and also, recalling \eqref{0-120-A5bis},
\begin{eqnarray*}
\nu_{E,n}(y)\partial_n K_\e(-y)&=&
\frac{(n+1-\e)\e y_n(1+O(|y'|^2))}{|y|^{n+3-\e}}\\&=&\frac12\sum_{\ell=1}^{n-1}
\frac{(n+1-\e)\e \partial^2_\ell f(0)\,y_\ell^2}{|y'|^{n+3-\e}}
+\e O(|y'|^{\e-n}).
\end{eqnarray*}
Accordingly, we have that
$$ \nu_{E}(y) \cdot \nabla K_\e(-y) =-\frac12\sum_{\ell=1}^{n-1}
\frac{(n+1-\e)\e \partial^2_\ell f(0)\,y_\ell^2}{|y'|^{n+3-\e}}
+\e O(|y'|^{\e-n}).$$ 
We thereby deduce from the latter identity and~\eqref{DOPPIO NU}
(and exploiting an odd symmetry argument)
that, for any~$r\in(0,r_0]$,
\begin{equation}\label{LAPRIINT}\begin{split}&
-\frac2{(n+1-\e)\e}
\int_{\partial E\cap B_r} \nu_{E}(y) \cdot \nabla K_\e(-y) 
\nu_{E,i}(y)\, \nu_{E,j}(y)\,d\HH_y\\ =\;&
\sum_{\ell=1}^{n-1} \int_{\partial E\cap B_r}\left(
\frac{ \partial^2_\ell f(0)
\,\partial^2_j f(0)\,\partial^2_i f(0)
\,y_\ell^2\,y_j\,y_i}{|y'|^{n+3-\e}}+O(|y'|^{2+\e-n})\right)
\,d\HH_y\\
=\;&
\sum_{\ell=1}^{n-1} \int_{\{|y'|<r\}} \left(
\frac{ \partial^2_\ell f(0)
\,\partial^2_j f(0)\,\partial^2_i f(0)
\,y_\ell^2\,y_j\,y_i}{|y'|^{n+3-\e}}+O(|y'|^{2+\e-n})\right)
\,\sqrt{1+|\nabla f(y')|^2}\,dy'\\
=\;&
\sum_{\ell=1}^{n-1} \int_{\{|y'|<r\}} \left(
\frac{ \partial^2_\ell f(0)
\,\partial^2_j f(0)\,\partial^2_i f(0)
\,y_\ell^2\,y_j\,y_i}{|y'|^{n+3-\e}}+O(|y'|^{2+\e-n})\right)
\,dy'\\
=\;&
\sum_{\ell=1}^{n-1} \int_{\{|y'|<r\}} \left(
\frac{ \partial^2_\ell f(0)
\,(\partial^2_j f(0))^2
\,y_\ell^2\,y_j^2\,\delta_{ji}}{|y'|^{n+3-\e}}+O(|y'|^{2+\e-n})\right)
\,dy'.
\end{split}\end{equation}

Furthermore, exploiting again~\eqref{DOPPIO NU}
and~\eqref{SVIPRO1}, we see that
\begin{equation}\label{LAPRIINT2}
\begin{split}&
\int_{\partial E\cap B_r} H_E(y)K_\e(-y)\,\nu_{E,i}(y)\, \nu_{E,j}(y)\,d\HH_y\\=\;&
\e\int_{\partial E\cap B_r} \left(\frac{H_E(y)\partial^2_{j} f(0)\partial^2_{i} f(0) y_j y_i}{
|y'|^{n+1-\e}}+O(|y'|^{2+\e-n})\right)\,d\HH_y
\\=\;&
\e\int_{\{|y'|<r\}}\left( \frac{H_E(0)\partial^2_{j} f(0)\partial^2_{i} f(0) y_j y_i}{
|y'|^{n+1-\e}}+O(|y'|^{2+\e-n})\right)\,dy'
\\=\;&
\e\int_{\{|y'|<r\}} \left(\frac{H_E(0)\,(\partial^2_{j} f(0))^2 \,y_j^2\,\delta_{ji}}{
|y'|^{n+1-\e}}+O(|y'|^{2+\e-n})\right)\,dy'.
\end{split}
\end{equation}
Now we use polar coordinates in~$\R^{n-1}$ to observe that
\begin{equation} \label{RESTO}
\int_{\{|y'|<r\}} |y'|^{2+\e-n}\,dy'=
{\mathcal{H}}^{n-2}(S^{n-2})\,\int_0^r \rho^{2+\e-n}\,\rho^{n-2}\,d\rho
=\frac{C\,r^{1+\e}}{1+\e},\end{equation}
for some~$C>0$.

Moreover, for any fixed index~$j\in\{1,\dots,n-1\}$,
\begin{equation}\label{BAL}\begin{split}&
\e\int_{\{|y'|<r\}} \frac{y_j^2}{|y'|^{n+1-\e}}\,dy'
=\frac{\e}{n-1}\sum_{k=1}^{n-1}\int_{\{|y'|<r\}} \frac{y_k^2}{|y'|^{n+1-\e}}\,dy'\\&\qquad
=\frac{\e}{n-1}\int_{\{|y'|<r\}} \frac{dy'}{|y'|^{n-1-\e}}
=\frac{\e\,{\mathcal{H}}^{n-2}(S^{n-2})}{n-1}\,\int_0^r
\rho^{\e-1}\,d\rho
=\varpi\,r^\e,\end{split}
\end{equation}
where
\begin{equation}\label{VARPI}
\varpi:=\frac{{\mathcal{H}}^{n-2}(S^{n-2})}{n-1}.
\end{equation}

Now, we compute the term~$
\e\int_{\{|y'|<r\}} \frac{y_\ell^2\,y_j^2}{|y'|^{n+3-\e}}\,dy'$.
For this, first of all we deal with the case~$\ell=j$: in this situation, we have that
\begin{equation}\label{891-304959}
\e\int_{\{|y'|<r\}} \frac{y_j^4}{|y'|^{n+3-\e}}\,dy' =
\e\int_{\{|y'|<r\}} \frac{y_1^4}{|y'|^{n+3-\e}}\,dy'=C_\star\,r^\e,
\end{equation}
where
\begin{equation}\label{0239494i292110202}\begin{split}
&C_\star:= \e\int_{\{|y'|<1\}} \frac{y_1^4}{|y'|^{n+3-\e}}\,dy'
= \e\iint_{(\rho,\vartheta)\in(0,1)\times S^{n-2}}
\rho^{\e-1}\,\vartheta_1^4\,d\rho\,d{\mathcal{H}}^{n-2}_\vartheta\\
&\qquad\qquad= \int_{\vartheta\in S^{n-2}}
\vartheta_1^4 \,d{\mathcal{H}}^{n-2}_\vartheta
=
\frac{3\, {\mathcal{H}}^{n-2}(S^{n-2}) }{(n-1)(n+1) }=
\frac{3\varpi }{n+1}
,\end{split}\end{equation}
thanks to~\eqref{0191:1} (applied here in one dimension less).

Moreover,
the number of different indices~$k$, $m\in\{1,\dots,n-1\}$ is equal to~$(n-1)(n-2)$
and so,
for each~$j\ne\ell\in\{1,\dots,n-1\}$,
\begin{equation*}\begin{split}&
\e\int_{\{|y'|<r\}} \frac{y_\ell^2\,y_j^2}{|y'|^{n+3-\e}}\,dy'\\
&=\e\int_{\{|y'|<r\}} \frac{y_1^2\,y_2^2}{|y'|^{n+3-\e}}\,dy'\\
&=\frac{\e}{(n-1)(n-2)}\sum_{k\neq m=1}^{n-1}\int_{\{|y'|<r\}}
\frac{y_k^2\,y_m^2}{|y'|^{n+3-\e}}\,dy'
\\&=\frac{\e}{(n-1)(n-2)}\left[ \sum_{k, m=1}^{n-1}\int_{\{|y'|<r\}}
\frac{y_k^2\,y_m^2}{|y'|^{n+3-\e}}\,dy' -
\sum_{k=1}^{n-1}\int_{\{|y'|<r\}}
\frac{y_k^4}{|y'|^{n+3-\e}}\,dy'
\right]
\\&
=\frac{\e}{(n-1)(n-2)}\int_{\{|y'|<r\}} \frac{dy'}{|y'|^{n-1-\e}}
-\frac{C_\star\,r^\e}{n-2}=\frac{\varpi\,r^\e}{n-2}-\frac{C_\star\,r^\e}{n-2}
=\frac{\varpi\,r^\e}{n+1} .\end{split}
\end{equation*}
{F}rom this and~\eqref{891-304959},
we obtain that
\begin{equation*} \e\int_{\{|y'|<r\}} \frac{y_\ell^2\,y_j^2}{|y'|^{n+3-\e}}\,dy'=
\frac{(1+2\delta_{\ell j})\,\varpi\,r^\e}{n+1}.\end{equation*} 
Substituting this identity and~\eqref{RESTO} into~\eqref{LAPRIINT},
and recalling also~\eqref{DEF H COS},
we conclude that
\begin{equation}\label{LAPRIINT:BIS}\begin{split}&
-\frac2{n+1-\e}
\int_{\partial E\cap B_r} \nu_{E}(y) \cdot \nabla K_\e(-y) 
\nu_{E,i}(y)\, \nu_{E,j}(y)\,d\HH_y\\ =\;&
\e\sum_{\ell=1}^{n-1} \int_{\{|y'|<r\}} 
\frac{ \partial^2_\ell f(0)
\,(\partial^2_j f(0))^2
\,y_\ell^2\,y_j^2\,\delta_{ji}}{|y'|^{n+3-\e}}
\,dy'
+o(1)
\\ =\;& \frac{\varpi\,r^\e}{n+1}
\sum_{\ell=1}^{n-1} \partial^2_\ell f(0)
\,(\partial^2_j f(0))^2\,(1+2\delta_{\ell j})\,\delta_{ji}
+o(1)
\\ =\;& -\frac{\varpi\,r^\e\,H_E(0)}{n+1}
\,(\partial^2_j f(0))^2\,\delta_{ji}
+ \frac{2\varpi\,r^\e}{n+1}\,(\partial^2_j f(0))^3\,\delta_{ji}
+o(1)
\\ =\;& -\frac{\varpi\,H_E(0)}{n+1}
\,(\partial^2_j f(0))^2\,\delta_{ji}
+ \frac{2\varpi}{n+1}\,(\partial^2_j f(0))^3\,\delta_{ji}
+o(1),
\end{split}\end{equation}
as~$\e\searrow0$.
Similarly,
substituting~\eqref{RESTO} and~\eqref{BAL}
into~\eqref{LAPRIINT2}, we obtain that, as~$\e\searrow0$,
\begin{equation*}
\begin{split}&
\int_{\partial E\cap B_r} H_E(y)K_\e(-y)\,\nu_{E,i}(y)\, \nu_{E,j}(y)\,d\HH_y\\=\;&
\e\int_{\{|y'|<r\}} \frac{H_E(0)\,(\partial^2_{j} f(0))^2 \,y_j^2\,\delta_{ji}}{
|y'|^{n+1-\e}}\,dy'+o(1)\\=\;&
\varpi\,r^\e\,H_E(0)\,(\partial^2_{j} f(0))^2 \,\delta_{ji}+o(1)\\=\;&
\varpi\,H_E(0)\,(\partial^2_{j} f(0))^2 \,\delta_{ji}+o(1).
\end{split}
\end{equation*}
{F}rom this and~\eqref{LAPRIINT:BIS} it follows that
\begin{equation}\label{FI:002:PRE}
\begin{split}&
\lim_{\e\searrow0}\int_{\partial E\cap B_r} \Big( H_E(y)K_\e(-y) 
-\nu_{E}(y) \cdot \nabla K_\e(-y) \Big)
\nu_{E,i}(y)\, \nu_{E,j}(y)\,d\HH_y\\
=\;&\frac{\varpi}{2}\,H_E(0)\,(\partial^2_{j} f(0))^2 \,\delta_{ji}+
\varpi\,(\partial^2_{j} f(0))^3 \,\delta_{ji}.
\end{split}\end{equation}

Now we exploit~\eqref{DEA:PA} and we see that
\begin{equation}\begin{split}\label{d9o3285v456746578fdghdshgxcvdheritv}
&\left|
\int_{\partial E\setminus B_r} \Big( H_E(y)K_\e(-y) 
-\nu_{E}(y) \cdot \nabla K_\e(-y) \Big)
\nu_{E,i}(y)\, \nu_{E,j}(y)\,d\HH_y\right|\\
\le\;& C\e\int_{\partial E\setminus B_r}\left( \frac{|H_E(y)|}{|y|^{n+1-\e}}+
\frac1{|y|^{n+2-\e}}\right)\,d\HH_y\\
\le\;& C\e\int_{\partial E\setminus B_r}\frac{1}{|y|^{n+1-\e}}
\left(  |H_E(y)|+{r}^{-1}\right)\,d\HH_y\\
\le\;& C\,(1+r^{-1})\,\e\int_{\partial E\setminus B_r} \frac{|H_E(y)|+1}{|y|^{n+1-\e}}\,d\HH_y\\
=\;& C\,(1+r^{-1})\,\e\,\sum_{k=0}^{+\infty}
\int_{\partial E\cap(B_{2^{k+1}r}\setminus B_{2^kr})} \frac{|H_E(y)|+1}{|y|^{n+1-\e}}\,d\HH_y\\
\le\;& \frac{C\,(1+r^{-1})\,\e}{r^{n+1-\e}}\,\sum_{k=0}^{+\infty}
\frac1{2^{k(n+1-\e)}}
\int_{\partial E\cap(B_{2^{k+1}r}\setminus B_{2^kr})} \big(|H_E(y)|+1\big)\,d\HH_y
\\ \le\;& \frac{C\,(1+r^{-1})\,\e}{r^{n+1-\e}}\,\sum_{k=0}^{+\infty}
\frac{(2^{k+1}r)^\beta}{2^{k(n+1-\e)}}
\\ \le\;& \frac{2^\beta\,C\,(1+r^{-1})\,\e}{r^{n+1-\e-\beta}}\,\sum_{k=0}^{+\infty}
\frac{1}{2^{k(n+1-\e-\beta)}}
\\ \le\;& \frac{2^\beta\,C\,(1+r^{-1})\,\e}{r^{n+1-\e-\beta}}\,\sum_{k=0}^{+\infty}
\frac{1}{2^{\frac{k(n+1-\beta)}2}}
\\ =\;& o(1),
\end{split}\end{equation}
for small~$\e$,
up to renaming~$C$ line after line,
and consequently
\[\lim_{\e\searrow0}\int_{\partial E\setminus B_r} \Big( H_E(y)K_\e(-y) 
-\nu_{E}(y) \cdot \nabla K_\e(-y) \Big)
\nu_{E,i}(y)\, \nu_{E,j}(y)\,d\HH_y=0.\]
This and~\eqref{FI:002:PRE} give that
\begin{equation}\label{FI:002}
\begin{split}&
\lim_{\e\searrow0}\int_{\partial E} \Big( H_E(y)K_\e(-y) 
-\nu_{E}(y) \cdot \nabla K_\e(-y) \Big)
\nu_{E,i}(y)\, \nu_{E,j}(y)\,d\HH_y\\
=\;&\frac{\varpi}{2}\,H_E(0)\,(\partial^2_{j} f(0))^2 \,\delta_{ji}+
\varpi\,(\partial^2_{j} f(0))^3 \,\delta_{ji}.
\end{split}\end{equation}

In addition, from Lemma A.2 of \cite{MR3798717},
we have that
\begin{equation}\label{FI:003}
\lim_{\e\searrow0} {L_{K_\e,E}}=-\frac{\varpi}{2}\,\Delta_{\partial E},
\end{equation}
where the notation in~\eqref{VARPI} has been used.
Similarly, from Lemma A.4 of \cite{MR3798717},
\begin{equation}\label{FI:004}
\lim_{\e\searrow0} c_{K_\e,E}^2=\frac\varpi2\,
c_E^2,\end{equation}
being~$c_E$ the norm of the second fundamental form of~$\partial E$.

Therefore, using~\eqref{FI:002}, \eqref{FI:003}
and~\eqref{FI:004},
we obtain that
\begin{equation}\label{PRE:SI}\begin{split}
&
\lim_{\e\searrow0}\Big[- {L_{K_\e,E}} \delta_{E,j} \nu_{E,i} (0)+
c_{K_\e,E}^2(0)\, \delta_{E,j} \nu_{E,i}(0) \\&\qquad-
\int_{\R^n} \Big( H_E(y)K_\e(-y)  -\nu_{E}(y) \cdot \nabla K_\e(-y) \Big)
\nu_{E,i}(y)\, \nu_{E,j}(y)\,d\HH_y \Big]\\
=\;& \frac{\varpi}{2}\,\Delta_{\partial E}\delta_{E,j} \nu_{E,i} (0)
+\frac{\varpi}2\,c_E^2(0)\, \delta_{E,j} \nu_{E,i}(0) -
\frac{\varpi}{2}\,H_E(0)\,(\partial^2_{j} f(0))^2 \,\delta_{ji}-
\varpi\,(\partial^2_{j} f(0))^3 \,\delta_{ji}.
\end{split}\end{equation}

Now, given two functions~$\psi$, $\phi$, we exploit~\eqref{EQ:6}
twice to obtain that
\begin{equation}\label{017:PRE}
\begin{split}&
\int_{\partial E} \delta_{E,i} \delta_{E,j}\psi(x)\,\phi(x)\,d\HH_x\\
=&-\int_{\partial E} \delta_{E,j}\psi(x)\,\delta_{E,i} \phi(x)\,d\HH_x+
\int_{\partial E} H_E(x)\,\nu_{E,i}(x)\,\delta_{E,j}\psi(x)\,\phi(x)\,d\HH_x\\
=&
\int_{\partial E} \psi(x)\,\delta_{E,j} \delta_{E,i} \phi(x)\,d\HH_x-
\int_{\partial E} H_E(x)\,\nu_{E,j}(x)\,\psi(x)\,\delta_{E,i} \phi(x)\,d\HH_x
\\&\qquad+
\int_{\partial E} H_E(x)\,\nu_{E,i}(x)\,\delta_{E,j}\psi(x)\,\phi(x)\,d\HH_x.
\end{split}\end{equation}
On the other hand, applying~\eqref{EQ:6}
once again, we see that
\begin{eqnarray*}
&& 
\int_{\partial E} H_E(x)\,\nu_{E,i}(x)\,\delta_{E,j}\psi(x)\,\phi(x)\,d\HH_x
\\&=&
-\int_{\partial E} \delta_{E,j}\big(H_E(x)\,\nu_{E,i}(x)\,\phi(x)\big)\,\psi(x)\,d\HH_x
\\ &&\qquad +
\int_{\partial E} H_E^2(x)\,\nu_{E,i}(x)\,\nu_{E,j}(x)\,\psi(x)\,\phi(x)\,d\HH_x.
\end{eqnarray*}
Plugging this information into~\eqref{017:PRE}, we find that
\begin{equation}\label{017}
\begin{split}&
\int_{\partial E} \delta_{E,i} \delta_{E,j}\psi(x)\,\phi(x)\,d\HH_x\\
=&
\int_{\partial E} \psi(x)\,\delta_{E,j} \delta_{E,i} \phi(x)\,d\HH_x-
\int_{\partial E} H_E(x)\,\nu_{E,j}(x)\,\psi(x)\,\delta_{E,i} \phi(x)\,d\HH_x
\\&\qquad-\int_{\partial E} \delta_{E,j}\big(H_E(x)\,\nu_{E,i}(x)\,\phi(x)\big)\,\psi(x)\,d\HH_x\\&\qquad
+
\int_{\partial E} H_E^2(x)\,\nu_{E,i}(x)\,\nu_{E,j}(x)\,\psi(x)\,\phi(x)\,d\HH_x.\end{split}\end{equation}

Applying~\eqref{017} (twice, at the beginning
with~$\psi:=H_{K_\e,E}(x)$ and at the end with~$\psi:=H_{E}(x)$)
and considering~$\phi$ as a test function,
the convergence of~$H_{K_\e,E}$ to~$\frac{\varpi\,H_E}{2}$
(see Theorem~12 in~\cite{MR3230079}) gives that
\begin{eqnarray*}&&
-\lim_{\e\searrow0}
\int_{\partial E} \delta_{E,i} \delta_{E,j}H_{K_\e,E}(x)\,\phi(x)\,d\HH_x
\\&=&\lim_{\e\searrow0}\left[ -
\int_{\partial E} H_{K_\e,E}(x)\,\delta_{E,j} \delta_{E,i} \phi(x)\,d\HH_x \right .\\
&&\qquad+
\int_{\partial E} H_E(x)\,\nu_{E,j}(x)\,H_{K_\e,E}(x)\,\delta_{E,i} \phi(x)\,d\HH_x
\\&&\qquad+ 
\int_{\partial E} \delta_{E,j}\big(H_E(x)\,\nu_{E,i}(x)\,\phi(x)\big)\,H_{K_\e,E}(x)\,d\HH_x
\\&&\left.\qquad-
\int_{\partial E} H_E^2(x)\,\nu_{E,i}(x)\,\nu_{E,j}(x)\,H_{K_\e,E}(x)\,\phi(x)\,d\HH_x
\right]
\\&=& \frac{\varpi}{2}\,\left[-
\int_{\partial E} H_{E}(x)\,\delta_{E,j} \delta_{E,i} \phi(x)\,d\HH_x+
\int_{\partial E} H_E(x)\,\nu_{E,j}(x)\,H_{E}(x)\,\delta_{E,i} \phi(x)\,d\HH_x\right.
\\&&\qquad+
\int_{\partial E} \delta_{E,j}\big(H_E(x)\,\nu_{E,i}(x)\,\phi(x)\big)\,H_{E}(x)\,d\HH_x
\\&&\left.\qquad-
\int_{\partial E} H_E^2(x)\,\nu_{E,i}(x)\,\nu_{E,j}(x)\,H_{E}(x)\,\phi(x)\,d\HH_x
\right]
\\&=& -\frac{\varpi}{2}\,
\int_{\partial E} \delta_{E,i} \delta_{E,j}H_{E}(x)\,\phi(x)\,d\HH_x.
\end{eqnarray*}
This says that~$ \delta_{E,i} \delta_{E,j}H_{K_\e,E} $ converges 
to~$\frac{\varpi}{2}\delta_{E,i} \delta_{E,j}H_{E}$ in the distributional sense
as~$\e\searrow0$: since, by the Ascoli-Arzel\`a
Theorem, we know that~$ \delta_{E,i} \delta_{E,j}H_{K_\e,E} $
converges strongly up to a subsequence, the uniqueness
of the limit gives that $ \delta_{E,i} \delta_{E,j}H_{K_\e,E} $ converges also
pointwise
to~$\frac{\varpi}{2}\delta_{E,i} \delta_{E,j}H_{E}$.

Combining this with~\eqref{PRE:SI}, we obtain that
\begin{equation}\label{PRE:PUNTO}
\begin{split}
&
\lim_{\e\searrow0}\Big[ \delta_{E,i} \delta_{E,j}H_{K_\e,E}(0)+
{L_{K_\e,E}} \delta_{E,j} \nu_{E,i} (0)-
c_{K_\e,E}^2(0)\, \delta_{E,j} \nu_{E,i}(0) \\&\qquad+
\int_{\R^n} \Big( H_E(y)K_\e(-y)  -\nu_{E}(y) \cdot \nabla K_\e(-y) \Big)
\nu_{E,i}(y)\, \nu_{E,j}(y)\,d\HH_y \Big]\\
=\;& 
\frac{\varpi}{2}\delta_{E,i} \delta_{E,j}H_{E}(0)
-\frac{\varpi}{2}\,\Delta_{\partial E}\delta_{E,j} \nu_{E,i} (0)
-\frac\varpi2\,c_E^2(0)\, \delta_{E,j} \nu_{E,i}(0) \\&\qquad+
\frac{\varpi}{2}\,H_E(0)\,(\partial^2_{j} f(0))^2 \,\delta_{ji}+
\varpi\,(\partial^2_{j} f(0))^3 \,\delta_{ji}.
\end{split}\end{equation}
By formula~\eqref{SIM:FORNEW}, we know that the left hand side
of~\eqref{PRE:PUNTO} is equal to zero.
Therefore, if~$H_E$ also vanishes identically,
we obtain that
$$ 
\Delta_{\partial E} h_{ij}(0)
+ c_E^2(0)\, h_{ij}(0) +
2\,h_{jj}^3(0) \,\delta_{ji}=0.$$
Recall that \(h_{ij}\) are the entries of the second fundamental form. Multiplying by~$h_{ij}$ and summing up over~$i$, $j\in\{1,\dots,n-1\}$,
we infer that
\begin{equation}\label{D234E2-L}
\sum_{i,j=1}^{n-1}h_{ij}(0)\Delta_{\partial E} h_{ij}(0)
+ c_E^2(0)\, \sum_{i,j=1}^{n-1}h_{ij}^2(0) +
2\,\sum_{j=1}^{n-1}h_{jj}^4(0)=0.
\end{equation}

Also, by~\eqref{SVIL-NU}, we have that~$h_{in}(0)=\delta_{E,i}\nu_{E,n}(0)=0$
for all~$i\in\{1,\dots,n-1\}$, and also~$h_{nn}(0)=\delta_{E,n} \nu_{E,n}(0)=0$
by~\eqref{DEL}, and so~\eqref{D234E2-L} becomes
\begin{equation}\label{D234E2-L:2}
\sum_{i,j=1}^{n-1}h_{ij}(0)\Delta_{\partial E} h_{ij}(0)
+c_E^4(0) +
2\,\sum_{j=1}^{n-1}h_{jj}^4(0)=0.
\end{equation}
On the other hand,
$$ \Delta_{\partial E} h_{ij}^2(0)=2 h_{ij}(0)
\Delta_{\partial E} h_{ij}(0)+2\sum_{k=1}^{n-1}|\delta_{E,k} h_{ij}(0)|^2
.$$
Therefore, \eqref{D234E2-L:2} becomes
\begin{equation}\label{D234E2-L:3}
\frac12\sum_{i,j=1}^{n-1}\Delta_{\partial E} h_{ij}^2(0)
=\sum_{i,j,k=1}^{n-1}|\delta_{E,k} h_{ij}(0)|^2
- c_E^4(0) -
2\,\sum_{j=1}^{n-1}h_{jj}^4(0).
\end{equation}

We observe now that, in light of~\eqref{SVIL-NU},
\begin{equation}\label{AHKAQ}
\nu_{E,n}(y)=1-\frac12\,\sum_{j=1}^{n-1} \big(\partial_j^2 f(0)\big)^2\,y_j^2+O(|y'|^3)\end{equation}
and so, by~\eqref{DEL}, \begin{align*}
     h_{nn}(y)=\delta_{E,n}\nu_{E,n}(y)& =\partial_n \nu_{E,n}(y)-\nu_{E,n}(y)\nabla \nu_{E,n}(y)\cdot
\nu_{E}(y) \\ &=-\sum_{j=1}^{n-1} \big(\partial_j^2 f(0)\big)^3\,y_j^2+O(|y'|^3).
\end{align*} This gives that~$h_{nn}^2(y)=O(|y|^4)$ and therefore
\begin{equation}\label{9qe84-129}
\Delta_{\partial E} h_{nn}^2(0)=0. 
\end{equation}
Furthermore, by~\eqref{DEL} and~\eqref{AHKAQ}, for any~$i\in\{1,\dots,n-1\}$,
$$ h_{in}(y)=\delta_{E,i}\nu_{E,n}(y)=
\partial_i \nu_{E,n}(y)-\nu_{E,i}(y)\nabla \nu_{E,n}(y)\cdot
\nu_{E}(y)=-\big(\partial^2_i f(0)\big)^2\,y_i+O(|y|^2)
,$$
which gives that
$$ h_{in}^2(y)=\big(\partial^2_i f(0)\big)^4\,y^2_i+O(|y|^3).$$
As a consequence,
$$ \Delta_{\partial E} h_{in}^2(0)=2\big(\partial^2_i f(0)\big)^4.$$
This and~\eqref{9qe84-129} give that \begin{align*}
    \Delta_{\partial E} c^2_E(0)&=
\sum_{i,j=1}^{n-1}\Delta_{\partial E} h_{ij}^2(0)+2\sum_{i=1}^{n-1}\Delta_{\partial E} h_{in}^2(0)+
\Delta_{\partial E} h_{nn}^2(0) \\&=\sum_{i,j=1}^{n-1}\Delta_{\partial E} h_{ij}^2(0)+
4\sum_{i=1}^{n-1}\big(\partial^2_i f(0)\big)^4.
\end{align*}Plugging this information into~\eqref{D234E2-L:3} we conclude that
\begin{equation*}
\frac12\Delta_{\partial E} c^2_E(0)=
\frac12\sum_{i,j=1}^{n-1}\Delta_{\partial E} h_{ij}^2(0)+
2\sum_{i=1}^{n-1}\big(\partial^2_i f(0)\big)^4
=\sum_{i,j,k=1}^{n-1}|\delta_{E,k} h_{ij}(0)|^2
- c_E^4(0) ,
\end{equation*}
which is~\eqref{ORIGINAL SIM}.
\hfill$\Box$

\section{Proof of Theorem~\ref{Jfqydwvfbe923ejfn}}
Let \(\eta \in C^\infty_0({\partial E})\) be arbitrary. Using \(f: = c_K \eta\) as a test function in~\eqref{ojwfe034}, we have that \begin{eqnarray*}
0&\leqslant& B_K(c_K\eta , c_K \eta ) - \int_{\partial E} c_K^4 \eta^2 \,d \mathcal H^n .
\end{eqnarray*} Next, for all \(x,y\in {\partial E}\), we have that \begin{eqnarray*}
(c_K(x)\eta(x)-c_K(y)\eta(y))^2 &=& \big (c_K(x)(\eta(x)-\eta(y))+ \eta(y) (c_K(x)-c_K(y)) \big )^2\\
&=& c_K^2(x) (\eta(x)-\eta(x))^2 + \eta^2(y)(c_K(x)-c_K(y))^2 \\
&&\qquad + 2c_K(x)\eta(y)(c_K(x)-c_K(y))(\eta(x)-\eta(y)),
\end{eqnarray*}
so it follows that \begin{eqnarray*}
B_K(c_K\eta , c_K \eta ) &=& \int_{\partial E} c_K^2(x) B_K(\eta , \eta ; x) \,d \mathcal H^n_x +  \int_{\partial E} \eta^2(x) B_K(c_K , c_K ; x) \,d \mathcal H^n_x +I  ,
\end{eqnarray*} where \begin{eqnarray*}
I:=  \int_{\partial E} \int_{\partial E} c_K(x)\eta(y)(c_K(x)-c_K(y))(\eta(x)-\eta(y))\,K(x-y)\,d \mathcal H^n_y \,d \mathcal H^n_x.
\end{eqnarray*} Next, by symmetry of \(x\) and \(y\), we have that \begin{eqnarray*}
I= \frac 1 2  \int_{\partial E} \int_{\partial E} (c_K(x)\eta(y)+c_K(y)\eta(x))(c_K(x)-c_K(y))(\eta(x)-\eta(y))\,K(x-y) \,d \mathcal H^n_y \,d \mathcal H^n_x.
\end{eqnarray*}
Moreover, by a simple algebraic manipulation, \begin{eqnarray*}&&
 (c_K(x)\eta(y)+c_K(y)\eta(x))(c_K(x)-c_K(y))(\eta(x)-\eta(y)) \\
&&\qquad= \frac12 (\eta^2(x) -\eta^2(y))(c_K^2(x)-c_K^2(y)) - \frac12 (c_K(x)-c_K(y))^2(\eta(x)-\eta(y))^2 \\
&&\qquad\leqslant \frac12 (\eta^2(x) -\eta^2(y))(c_K^2(x)-c_K^2(y))
\end{eqnarray*}and accordingly \begin{eqnarray*}
I&\leqslant&\frac 1 4  \int_{\partial E} \int_{\partial E} (\eta^2(x) -\eta^2(y))(c_K^2(x)-c_K^2(y))
\,K(x-y)\,d \mathcal H^n_y \,d \mathcal H^n_x \\
&=& \frac 1 2  \int_{\partial E} \int_{\partial E} \eta^2(x)(c_K^2(x)-c_K^2(y))\,K(x-y) \,d \mathcal H^n_y \,d \mathcal H^n_x \\
&=& \frac 12 \int_{\partial E} \eta^2(x) {\mathcal{L}}_K  c_K^2(x) \,d \mathcal H^n_x.
\end{eqnarray*}  Hence, we have that \begin{eqnarray*}
B_K(c_K\eta , c_K \eta ) &\leqslant& \int_{\partial E} c_K^2(x) B_K(\eta , \eta ; x) \,d \mathcal H^n_x \\
&&\qquad+  \int_{\partial E} \bigg \{  B_K(c_K , c_K ; x)+\frac 12 \int_{\partial E} \eta^2(x) {\mathcal{L}}_K  c_K^2(x) \bigg \} \eta^2(x)\,d \mathcal H^n_x 
\end{eqnarray*} and the result follows. 
\hfill$\Box$

\section{Appendix A: Proof of formulas~\eqref{0191:0} and~\eqref{0191:1}}\label{jxiewytdenfrhgfndbvhfdnbvfdefgbhwektheru}

Let $$
Q:=\int_{B_1} x_1^4\,dx \qquad{\mbox{and}}\qquad D:=\int_{B_1} x_1^2x_2^2\,dx
.$$
We consider the isometry~$x\mapsto X\in \R^n$ given by
$$ X_1:=\frac{x_1-x_2}{\sqrt2},\qquad
X_2:=\frac{x_1+x_2}{\sqrt2},\qquad X_i:=x_i \quad{\mbox{ for all }}\; i\in\{3,\dots,n\}.$$
We notice that
$$ 4X_1^2 X_2^2 = (2X_1X_2)^2=\big( (x_1-x_2)(x_1+x_2)\big)^2=
(x_1^2-x_2^2)^2 = x_1^4+x_2^4-2x_1^2 x_2^2$$
and therefore, by symmetry,
$$ 4D=\int_{B_1} 4X_1^2 X_2^2\,dX=\int_{B_1}
\big( x_1^4+x_2^4-2x_1^2 x_2^2\big)\,dx
=2Q-2D,$$
which gives
\begin{equation}\label{24t-d293}
D=\frac{Q}3.
\end{equation}

On the other hand
$$ |x|^4 = (|x|^2)^2=\left( \sum_{i=1}^n x_i^2\right)^2=
\sum_{i,j=1}^n x_i^2 x_j^2
=\sum_{i=1}^n x_i^4+
\sum_{i\ne j=1}^n x_i^2 x_j^2.$$
Therefore, by polar coordinates and symmetry,
\begin{eqnarray*}
&&\frac{ {\mathcal{H}}^{n-1}(S^{n-1}) }{n+4 }=
{\mathcal{H}}^{n-1}(S^{n-1})
\int_0^1 \rho^{n+3}\,d\rho=
\int_{B_1} |x|^4\,dx\\ &&\qquad=
\sum_{i=1}^n \int_{B_1}x_i^4\,dx+
\sum_{i\ne j=1}^n \int_{B_1} x_i^2 x_j^2\,dx=nQ+n(n-1)D.
\end{eqnarray*}
{F}rom this and~\eqref{24t-d293} we deduce that
$$ \frac{ {\mathcal{H}}^{n-1}(S^{n-1}) }{n+4 } = \frac{n(n+2)\,Q}{3},$$
hence
$$ \frac{3\, {\mathcal{H}}^{n-1}(S^{n-1}) }{n(n+2)(n+4) }=Q
=\int_0^1\int_{S^{n-1}} \rho^{n+3}\vartheta_1^4\,d{\mathcal{H}}^{n-1}_\vartheta\,d\rho=
\frac{1}{n+4}\int_{S^{n-1}} \vartheta_1^4 \,d{\mathcal{H}}^{n-1}_\vartheta,
$$
which gives~\eqref{0191:0} and~\eqref{0191:1}, as desired.

%% file: Part3/Density-estimate.tex
\chapter{Density estimates and the fractional Sobolev inequality for sets of zero \(s\)-mean curvature} \label{1KsVju8j}

We prove that measurable sets \(E\subset \R^n\) with locally finite perimeter and zero \(s\)-mean curvature satisfy the surface density estimates: \begin{align*}
    \operatorname{Per} (E; B_R(x)) \geq CR^{n-1}
\end{align*} for all \(R>0\), \(x\in \partial^\ast E\). The constant \(C\) depends only on \(n\) and \(s\), and remains bounded as \(s\to 1^-\). As an application, we prove that the fractional Sobolev inequality holds on the boundary of sets with zero \(s\)-mean curvature.

\section{Introduction and main results}

A fundamental result in the regularity of minimal surfaces is that a measurable set \(E\subset \R^n \) that is is locally perimeter minimising satisfies the thick density estimates: \begin{align}
         \vert E \cap B_R(x) \vert \geq C R^n \text{ and }  \vert B_R(x)  \setminus E \vert \geq C R^n \label{gJJA6X52}
\end{align} for all \(R>0\), \(x\in \partial^\ast E\) where \(B_R(x)\) is a ball in \(\R^n\) centred at \(x\) with radius \(R\) and \(\vert F\vert\) denotes the Lebesgue measure of \(F\) in \(\R^n\), see for example \cite[Section 16.2]{maggi_sets_2012}. Here the constant \(C>0\) depends only on \(n\). Since any sufficiently regular set \(E\) satisfies~\eqref{gJJA6X52} for \(R\) sufficiently small and with the constant depending on the \(C^2\) regularity of \(E\) at \(x\), the non-triviality of~\eqref{gJJA6X52} is that it holds for all \(R>0\) and that the constant is universal. The assumption that \(E\) is locally minimising is essential to the proof of~\eqref{gJJA6X52} since the proof takes advantage of the fact that \(E\cap \partial B_R(x)\) is a competitor for \(\partial^\ast E \cap B_R(x)\).

If one only assumes that \(E\) has zero mean curvature, that is, it is locally stationary with respect to the perimeter functional then one can obtain the surface density estimate: \begin{align}
    \operatorname{Per}(E; B_R(x)) \geq C R^{n-1} \label{jR0GCsky}
\end{align} for all \(R>0\), \(x\in \partial^\ast E\), and with \(C\) depending only on \(n\). Indeed, the monotonicity formula for sets of zero mean curvature immediately implies that \begin{align*}
      \frac{\operatorname{Per}(E; B_R(x))}{R^{n-1}} \geq \lim_{r\to 0^+ }\frac{\operatorname{Per}(E; B_r(x))}{r^{n-1}} =  n\omega_n 
\end{align*} where \(\omega_n = \vert B_1 \vert\), see \cite[Theorem 17.16]{maggi_sets_2012}. Observe that the thick density estimates~\eqref{gJJA6X52} directly imply the surface density estimates~\eqref{jR0GCsky} via the restricted isoperimetric inequality \begin{align*}
    \operatorname{Per}(E;B_R(x) ) \geq C \big ( \min \{ \vert E \cap B_R(x) \vert,\vert B_R(x)  \setminus E \} \big )^{\frac{n-1}n }, 
\end{align*} but sets with zero mean curvature do not necessarily satisfy~\eqref{gJJA6X52}, for example, if \(E = \R^{n-1} \times (0,1)\) then \(\partial E\) has zero mean curvature, but \(\vert E \cap B_R \vert \leq \vert B^{n-1}_R \times (0,1) \vert \leq C R^{n-1} \).

In the celebrated paper \cite{MR2675483}, a nonlocal (or fractional) concept of perimeter, known as the \(s\)-perimeter, was introduced that was a natural generalisation of the classical perimeter. As well as motivating and defining the \(s\)-perimeter, Caffarelli, Roquejoffre, and Savin in \cite{MR2675483} establish several foundational results for the regularity theory of minimisers\footnote{\label{note1}For a precise definition, see Section 2} of the \(s\)-perimeter including the existence of minimisers, improvement of flatness, a monotonicity formula, blow-up limits to \(s\)-minimal cones, and Federer's dimension reduction. In particular, they also established minimisers of the \(s\)-perimeter satisfy the thick density estimates~\eqref{gJJA6X52}. Since their paper, the regularity theory of critical points of the \(s\)-perimeter (stationary, stable, and minimisers) has become an incredibly active area of research. Some important results include: interior higher regularity \cite{MR3331523}; boundary regularity and stickiness \cite{MR3596708,MR4104542}; classification of cones for \(n=2\) \cite{MR3090533}; classification of cones for \(s\) close to \(1\) \cite{MR3107529}; regularity of \(s\)-minimal graphs \cite{MR3934589}; explicit examples \cite{MR4082938,davila_nonlocal_2018}; and Yau's conjecture \cite{caselli2024yaus}. For a nice survey of the current literature, see \cite{MR3824212}.

In analogy to the classical case, if \(E\) is \(s\)-stationary, that is stationary\textsuperscript{\ref{note1}} with respect to the \(s\)-perimeter, and \(\partial E\) is sufficiently smooth then \(E\) satisfies \begin{align*}
    \mathrm H_{s,E}(x) := \lim_{\varepsilon \to 0^+} \int_{\R^n \setminus B_\varepsilon(x)} \frac{\chi_{\R^n \setminus E}(y) - \chi_E (y) }{\vert x - y \vert^{n+s }} \dd y =0
\end{align*} for all \(x\in \partial E\), see \cite{MR3322379}. It is conventional to call \( \mathrm H_{s,E}\) the \(s\)-mean curvature of \(E\). Much of the literature focuses on the regularity of sets that minimise the \(s\)-perimeter and relatively little is known about the regularity of sets with zero of \(s\)-mean curvature. Indeed, although~\eqref{gJJA6X52} was proven in \cite{MR2675483} for minimisers of the \(s\)-perimeter, it is still an open problem to prove that sets of zero \(s\)-mean curvature satisfy the surface density estimate~\eqref{jR0GCsky}. The main result of this note establishes such an estimate.

\begin{thm} \thlabel{jG1dQXin}
Let \(s\in (0,1)\) and \(E \subset \R^n \) be a measurable set with locally finite perimeter satisfying \(\mathrm H_{s,E}=0 \) on \(\partial^\ast E\). Then \begin{align*}
    \operatorname{Per}(E;B_R) \geq C R^{n-1}  \text{ for all }R>0.
\end{align*} The constant \(C>0\) depends only on \(n\) and \(s\), and remains bounded as \(s \to 1^-\).
\end{thm}

The precise definition \(\mathrm H_{s,E}=0 \) on \(\partial^\ast E\) for a set of locally finite perimeter is given in Section 2. Several articles in the literature have explored related density estimates to~\thref{jG1dQXin} including \cite{MR3981295} and \cite{caselli2024yaus,caselli2024fractional}. In \cite{MR3981295}, the authors establish upper surface density estimates of the form \begin{align}
    \operatorname{Per} (E;B_R(x)) \leq C R^{n-1} \label{0YzpLLRv}
\end{align} for all \(x\in \partial^\ast E\), \(R>0\) for stationary, stable sets \(E\) with respect to the \(s\)-perimeter and other nonlocal perimeters with more general kernels. It is interesting to note that in the classical case of stable surfaces with zero mean curvature,~\eqref{0YzpLLRv} is known only for \(n=2\) and is famously open for \(n\geq 3\). In \cite{caselli2024yaus,caselli2024fractional}, the main focus is on establishing a nonlocal Yau's conjecture, that is, in any closed Riemannian manifold there exist infinitely many sets of zero \(s\)-mean curvature. In particular, they prove the estimate~\eqref{0YzpLLRv} for sets with zero \(s\)-mean curvature in closed Riemannian manifolds. We should mention that, though most of the techniques present in this paper are also present in \cite{MR3981295,caselli2024yaus,caselli2024fractional}, particularly \cite{caselli2024yaus,caselli2024fractional}; however, the result \thref{jG1dQXin} is never explicitly stated nor is its connection to the fractional Sobolev inequality on hypersurfaces, see below.

A nice application of \thref{jG1dQXin} is that the fractional Sobolev inequality holds on hypersurfaces in \(\R^n\) with zero \(s\)-mean curvature. The classical Sobolev inequality in \(\R^n\) states that if \(1\leq p < n \) and \(p^\ast := \frac{np}{n-p}\) then \begin{align}
    \| u \|_{L^{p^\ast}(\R^n)} \leq C \| \nabla u \|_{L^p(\R^n)} \label{3U5YSnYf}
\end{align} for all \(u \in C^\infty_0(\R^n)\) with \(C>0\) depending only on \(n\) and \(p\). It is also interesting to ask if the Sobolev inequality can hold on hypersurfaces \(M^n \hookrightarrow \R^{n+1}\). It is easy to see that~\eqref{3U5YSnYf} (with \(\R^n\) replaced with \(M\) and \(\nabla = \nabla_M\) the gradient with respect to the induced metric from \(\R^{n+1}\)) cannot hold verbatim since if \(M\) is compact then \(u=1\in C^\infty_0(M)\); however, a Sobolev-type inequality does hold if one adds an extra \(L^p\) error term involving the mean curvature \(\mathrm H_M\): \begin{align}
    \| u \|_{L^{p^\ast}(M)} \leq C \big ( \| \nabla_M u \|_{L^p(M)} + \|\mathrm H_M u \|_{L^p(M) }\big ) \label{RMdxxBBd}
\end{align} for all \(u \in C^\infty_0(M)\). The inequality~\eqref{RMdxxBBd} is sometimes known as the Michael-Simon and Allard inequality after \cite{allard_first_1972} and \cite{michael_sobolev_1973}. Also, see \cite{brendle_isoperimetric_2021} where the optimal constant in~\eqref{RMdxxBBd} is obtained and \cite{MR4435963} for a simple proof of~\eqref{RMdxxBBd}. 

Naturally, one would like to know if~\eqref{RMdxxBBd} extends to the nonlocal case. If \(M=\R^n\), \(1\leq p < n/s\), and \(p^\ast = \frac{np}{n-sp}\) then it is well known that\begin{align*}
    \| u \|_{L^{p^\ast}(\R^n)} \leq C [u]_{W^{s,p}(\R^n )}
\end{align*} where \begin{align}
    [v]_{W^{s,p}(\R^n)} &= \bigg ( \int_{\R^n } \int_{\R^n}\frac{\vert v(x) - v(y) \vert^p }{\vert x - y \vert^{n+sp}} \dd y \dd x \bigg )^{\frac 1 p }, \label{JS7OQLcc}
\end{align}see \cite{MR2944369} and references therein. If \(E \subset \R^n\) is an open set with locally finite perimeter (or sufficiently smooth boundary) and \begin{align*}
      [v]_{W^{s,p}(\partial^\ast E)} &= \bigg ( \int_{\partial^\ast E } \int_{\partial^\ast E}\frac{\vert v(x) - v(y) \vert^p }{\vert x - y \vert^{n+sp}} \dd \mathcal H^{n-1}_y \dd \mathcal H^{n-1}_x \bigg )^{\frac 1 p }
\end{align*} then it is currently an open problem whether the inequality \begin{align}
     \| u \|_{L^{p^\ast}(\partial^\ast E)} \leq C \big (  [u]_{W^{s,p}(\partial^\ast E)} + \| \mathrm H_{s,E}u\|_{L^p(\partial^\ast E)} \big ) \label{yp5YLcrO}
\end{align} holds for all \(u\in C^\infty_0(\partial^\ast E)\). Recently, it was shown in \cite{MR4565417} that~\eqref{yp5YLcrO} does hold provided that \(E\) is convex. Their argument relies on first establishing a pointwise lower bound on the \(s\)-mean curvature in terms of the perimeter. Unfortunately, as they point out in their paper, this pointwise inequality clearly cannot hold in the non-convex case, so the argument doesn't easily generalise to the case of arbitrary \(E\). Furthermore, in \cite[Proposition 5.2]{MR3934589}, it was proven via an extension of an (unpublished) argument due to Brezis (see \cite[Theorem 2.2.1]{MR3469920} for this argument), that the boundary density estimate~\eqref{jR0GCsky} implies~\eqref{yp5YLcrO} without the term involving the \(s\)-mean curvature. Hence, via \cite[Proposition 5.2]{MR3934589}, \thref{jG1dQXin} immediately implies that~\eqref{yp5YLcrO} holds for sets of zero \(s\)-mean curvature:

\begin{cor}
Let \(s\in (0,1)\), \(1\leq p <n/s\), \(p^\ast = \frac{np}{n-sp}\), and \(E \subset \R^{n+1} \) be a measurable set with locally finite perimeter satisfying \(\mathrm H_{s,E}=0 \) on \(\partial^\ast E\). Then the Sobolev inequality \begin{align*}
    \| u\|_{L^{p^\ast}(\partial^\ast E) } \leq C [u]_{W^{s,p}(\partial^\ast E)}
\end{align*} holds for all \(u\in C^\infty_0(\partial^\ast E)\). The constant \(C>0\) depends only on \(n\) and \(s\). 
\end{cor}

\subsection{Organisation of paper}
The paper is organised as follows. In Section 2, we fix notation and give standard definitions of the classical theory of BV functions and minimal surfaces, and definitions from the theory of nonlocal minimal surfaces. In Section 3, we prove an interpolation inequality between restricted \(s\)-perimeter and the classical perimeter. In Section 4, we survey some results related to the Caffarelli-Silvestre extension in the context of nonlocal minimal surfaces that will be essential to the proof of \thref{jG1dQXin}. Finally, in Section 5, we give the proof of \thref{jG1dQXin}.

\section*{Acknowledgements}
Jack Thompson is supported by an Australian Government Research Training Program Scholarship. JT would like to thank Serena Dipierro, Giovanni Giacomin, and Enrico Valdinoci for their valuable conversations and their comments on the first draft of this note. He would also like to thank Jo{\~a}o Gon\c{c}alves da Silva for the example at the end of the first page and Tommaso Di Ubaldo for his deep insights.

\section{Definitions and notation}
First, we recall some notation and definitions of functions of bounded variation and the classical theory of minimal surfaces. Let \(\Omega \subset \R^n\) be an open set. Given a function \(u\in L^1(\Omega)\) the \emph{total variation} of \(u\) in \(\Omega\) is given by \begin{align*}
    \vert \nabla u \vert(\Omega) = \sup \bigg \{ \int_\Omega u \div \phi \dd x \text{ s.t. } \phi \in C^1_0(\Omega ; \R^n ), \| \phi \|_{L^\infty(\Omega ; \R^n)} \leq 1 \bigg \} . 
\end{align*} The space \(\operatorname{BV}(\Omega) \) is defined to be the set of \(u\in L^1(\Omega)\) such that \(\vert \nabla u \vert (\Omega) <+\infty \) and the space \(\operatorname{BV}_{\mathrm{loc}}(\Omega) \) is the set of functions in \(\operatorname{BV}(\Omega')\)  for all \(\Omega'\subset \subset \Omega\). Moreover, we say a measurable set \(E\subset \R^n\) has \emph{finite perimeter} in \(\Omega\) if \(\chi_E \in \operatorname{BV}(\Omega) \) where \(\chi_A\) denotes the characteristic function of \(A\). In this case, we define the \emph{perimeter of} \(E\) in \(\Omega\) by \begin{align*}
    \operatorname{Per}(E;\Omega) =  \vert \nabla \chi_E \vert(\Omega).
\end{align*} We also say \(E\) has \emph{locally finite perimeter} in \(\Omega\) if \(\chi_E\in \operatorname{BV}_{\mathrm{loc}}(\Omega)\) and simply \(E\) has locally finite perimeter if \(\chi_E\in \operatorname{BV}_{\mathrm{loc}}(\R^n)\). Finally, given \(E\) with finite perimeter in \(\Omega\), the distributional gradient of \(\chi_E\), i.e. \(\nabla \chi_E\), is a vector-valued Radon measure on \(\R^n\). Then we can define the \emph{reduced boundary} of \(E\), denoted \(\partial^\ast E \), as the set \begin{align*}
   \bigg \{ x\in \partial E \text{ s.t. } \vert \nabla \chi_E  \vert(B_r(x))>0 &\text{ for all }r>0, \\ & \text{ and } \lim_{r\to 0^+} \frac{\nabla \chi_E (B_r(x)) }{\vert \nabla \chi_E  \vert(B_r(x)) } \text{ exists and is in } \Sph^{n-1}  \bigg \} .
\end{align*}

Now, let us turn to the nonlocal case. Let \(s\in (0,1)\) and \begin{align*}
    \mathcal I_s(A,B) = \int_A \int_B \frac{\dd y \dd x }{\vert x - y \vert^{n+s}} . 
\end{align*} The \(s\)-perimeter of a measurable set \(E\subset \R^n\) in \(\Omega\) is defined to be \begin{align*}
     \operatorname{Per}_s(E ; \Omega  ) = \mathcal I_s (E\cap \Omega , E^c \cap \Omega) +\mathcal I_s (E\cap \Omega , E^c \setminus \Omega)+\mathcal I_s (E\setminus \Omega, E^c\cap \Omega). 
\end{align*} We will also sometimes refer to \(\operatorname{Per}_s(E ; \Omega  )\) as the \emph{restricted \(s\)-perimeter}. Note, we also have that \begin{align}
    \operatorname{Per}_s(E ; \Omega  ) = \frac12\iint_{\mathcal Q(\Omega) } \frac{\vert \chi_E(x)-\chi_E(y) \vert}{\vert x - y \vert^{n+s} } \dd y \dd x \label{LDtSA29N}
\end{align} where we define \begin{align*}
    \mathcal Q(\Omega) = (\Omega^c \times \Omega^c)^c = (\Omega \times \Omega) \cup (\Omega \times \Omega^c) \cup (\Omega^c\times \Omega) \subset \R^{2n}.
\end{align*}

A measurable set \(E\) is a \emph{minimiser} of the \(s\)-perimeter in \(\Omega\) or \(s\)-minimal in \(\Omega\) if \begin{align*}
    \operatorname{Per}_s(E; \Omega) \leq \operatorname{Per}_s(F;\Omega)
\end{align*} for all measurable set \(F\) such that \(E\setminus\Omega = F\setminus \Omega\). Moreover, a measurable set \(E\) is stationary with respect to the \(s\)-perimeter in \(\Omega\) or \(s\)-stationary if \begin{align}
    \frac{\dd }{\dd t}\bigg \vert_{t=0} \operatorname{Per}_s(E_{t,T }; \Omega)=0 \label{3MedX4AV}
\end{align} for all \(T\in C^\infty_0(\Omega; \R^n )\) where \begin{align*}
    E_{t,T} := \psi_T (E,t)
\end{align*} and \(\psi_T : \R^n \times (-\varepsilon , \varepsilon) \to \R^n\) satisfies \begin{align*}
    \begin{PDE}
\partial_t \psi_T &= T\circ \psi_T, &\text{in } \R^n \times (-\varepsilon , \varepsilon) \\
\psi_T &= T, &\text{on } \R^n \times \{t=0\}.
    \end{PDE}
\end{align*} One can also define \(s\)-stationary stable sets, but we do not require this definition for the current paper.

\section{An interpolation inequality}
In this section, we prove an interpolation inequality which will allow us to estimate the restricted \(s\)-perimeter with the restricted perimeter plus an arbitrarily small error. The precise statement is as follows.

\begin{thm} \thlabel{7qaGfSn3}
Let \(R>0\) and \(u\in \operatorname{BV}_{\mathrm{loc}}(\R^n) \cap L^\infty(\R^n)\). Then, for all \(\varepsilon\in (0,3^{-\frac 1 s } )\), \begin{align*}
   \iint_{\mathcal Q (B_R)} \frac{\vert u(x) - u(y)\vert }{\vert x - y \vert^{n+s}} \dd y \dd x \leq  C \bigg (  \frac{\varepsilon^{-\frac{1-s}s }  R^{1-s}}{1-s} \vert \nabla u \vert ( B_{(1+\varepsilon^{-1/s})R})+ \frac{\varepsilon R^{n-s}} s \| u\|_{L^\infty(\R^n)} \bigg ) .
\end{align*} The constant \(C>0\) depends only on \(n\).
\end{thm}

For similar interpolation inequalities to~\thref{7qaGfSn3} and~\thref{60sDucQ0} below, see \cite[Equation 3.]{MR3981295} and \cite[Lemma 3.15]{caselli2024fractional}. A small, but important difference between \thref{7qaGfSn3} and these results is the inclusion of the \(\varepsilon\) term. This will be essential to the proof of \thref{jG1dQXin} since it will allow us to absorb an error term coming from \thref{7qaGfSn3} into the left-hand side of a chain of inequalities.

Taking \(u = \chi_E\) in \thref{7qaGfSn3}, we immediately obtain the following corollary. 

\begin{cor} \thlabel{60sDucQ0}
Let \(R>0\) and \(E \subset \R^n\) be a measurable set with locally finite perimeter. Then, for all \(\varepsilon\in(0,3^{-\frac 1 s } ) \), \begin{align*}
    \operatorname{Per}_s(E; B_R) \leq C\bigg (  \frac{\varepsilon^{-\frac{1-s}s }  R^{1-s} }{1-s} \operatorname{Per}(E; B_{(1+\varepsilon^{-1/s})R})+ \frac{ \varepsilon R^{n-s} } s \bigg )
\end{align*}  The constant \(C>0\) depends only on \(n\).
\end{cor}

Before we give the proof of \thref{7qaGfSn3}, we require the following lemma.

\begin{lem} \thlabel{aIDAJA5Z}
For all \(x,y\in \R^n\) and \(t\in [0,1]\), \begin{align*}
         \chi_{\mathcal Q(B_1)}(x-ty,x+(1-t)y) \leq \chi_{\{\vert  x \vert \leq \vert y \vert + 1  \}}(x,y) .
    \end{align*}
\end{lem}

\begin{proof}
    Since \(\mathcal Q(B_1) = \{ (x,y ) \in \R^n \text{ s.t. } \vert x \vert <1 \text{ or } \vert y \vert <1 \} \), if \((x-ty,x+(1-t)y) \in \mathcal Q(B_1)\) then either \( \vert x - t y \vert <1 \) or \(\vert x+(1-t) y \vert <1\). If \( \vert x - t y \vert <1 \) then  \begin{align*}
        \vert x \vert \leq t \vert y \vert +1 \leq \vert y \vert +1 .
    \end{align*} Similarly, if \(\vert x+(1-t) y \vert <1\) then \begin{align*}
        \vert x \vert \leq 1 + (1-t) \vert y \vert  \leq 1+ \vert y \vert. 
    \end{align*}
\end{proof}

Now we give the proof of \thref{7qaGfSn3}.

\begin{proof}[Proof of \thref{7qaGfSn3}]

First, assume that \(u\in C^\infty_0(\R^n)\) and \(R=1\). Let \(\rho>0\) to be chosen later and write \begin{align*}
    \iint_{\mathcal Q(B_1)} \frac{\vert u(x) - u (y)\vert }{\vert x - y \vert^{n+s}} \dd y \dd x = I_\rho + J_\rho
\end{align*} where \begin{align*}
    I_\rho = \iint_{\mathcal Q(B_1)\cap \{ \vert x - y \vert <\rho \} } \frac{\vert u(x) - u (y)\vert }{\vert x - y \vert^{n+s}} \dd y \dd x 
\end{align*} and \begin{align*}
     J_\rho = \iint_{\mathcal Q(B_1)\cap \{ \vert x - y \vert \geq \rho\} } \frac{\vert u(x) - u (y)\vert }{\vert x - y \vert^{n+s}} \dd y \dd x.
\end{align*}We will estimate \(I_\rho\) and \(J_\rho\) separately. 

For \(I_\rho\), via the change of variables \(y \to y+x\), we have  \begin{align*}
    I_\rho &= \int_{\R^n} \int_{\R^n} \frac{\vert u(x) - u (y)\vert \chi_{\mathcal Q(B_1) }(x,y) \chi_{\{\vert x-y\vert < \rho \}  }(x,y)}{\vert x - y \vert^{n+s}} \dd y \dd x  \\
    &= \int_{\R^n} \int_{\R^n} \frac{\vert u(x+y) - u (x)\vert \chi_{\mathcal Q(B_1) }(x,x+y) \chi_{ B_\rho  }(y)}{\vert y\vert^{n+s}} \dd y \dd x  .
\end{align*} Since \begin{align*}
        \vert u(x+y) - u(x)\vert &\leq \vert y \vert \int_0^1 \vert \nabla u (x+t y ) \vert \dd t ,
    \end{align*} it follows that \begin{align*}
        I_\rho &\leq \int_{\R^n} \int_{\R^n} \int_0^1 \frac{\vert \nabla u(x+t y) \vert \chi_{\mathcal Q(B_1) }(x,x+y) \chi_{ B_\rho  }(y)}{\vert y\vert^{n+s-1}} \dd t \dd y \dd x 
    \end{align*} which, via the change of variable \(x\to x-ty\), becomes \begin{align*}
        I_\rho &\leq \int_{\R^n} \int_0^1 \int_{\R^n}  \frac{\vert \nabla u(x) \vert \chi_{\mathcal Q(B_1) }(x-ty,x+(1-t)y) \chi_{ B_\rho  }(y)}{\vert y\vert^{n+s-1}}\dd x  \dd t \dd y  .
    \end{align*} Then, from \thref{aIDAJA5Z}, it follows that \begin{align*}
         \chi_{\mathcal Q(B_1)}(x-ty,x+(1-t)y)\chi_{B_\rho}(y) \leq \chi_{\{\vert  x \vert \leq \vert y \vert + 1  \}}(x,y) \chi_{B_\rho}(y) \leq \chi_{B_{\rho +1}}(x) \chi_{ B_\rho }(y),
    \end{align*} so\begin{align*}
        I_\rho &\leq \| \nabla u \|_{L^1(B_{\rho+1})} \int_{B_\rho }  \frac{\dd y }{\vert y\vert^{n+s-1}} \leq \frac {C\rho^{1-s} } {1-s} \| \nabla u \|_{L^1(B_{\rho+1})}
    \end{align*} with \(C\) depending only on \(n\).

    For \(J_\rho\), \begin{align*}
        J_\rho \leq 2 \| u\|_{L^\infty(\R^n) } \iint_{\mathcal Q(B_1)\cap \{ \vert x - y \vert \geq \rho\} } \frac{\dd y \dd x }{\vert x - y \vert^{n+s}} .
    \end{align*} Assuming that \(\rho>3\), \( \mathcal Q(B_1)\cap \{ \vert x - y \vert \geq \rho\} \subset \big ( (B_1 \times B_1^c) \cup (B_1^c \times B_1)\big ) \cap \{ \vert x - y \vert \geq \rho\} \). Indeed, if there exists \((x,y) \in (B_1\times B_1)\cap \{ \vert x - y \vert \geq \rho\} \) then \(3<\rho \leq \vert x- y \vert \leq \vert x \vert + \vert y \vert <2\), a contradiction. Hence, \begin{align*}
         J_\rho \leq 4 \| u\|_{L^\infty(\R^n) } \iint_{ (B_1 \times B_1^c)\cap \{ \vert x - y \vert \geq \rho\} } \frac{\dd y \dd x }{\vert x - y \vert^{n+s}} . 
    \end{align*} Since \(\vert x-y\vert \geq \vert y \vert - \vert x \vert \geq \vert y \vert - 1 \) and \( (B_1 \times B_1^c)\cap \{ \vert x - y \vert \geq \rho\} \subset B_1 \times B_{\rho -1 }^c \), we find that \begin{align*}
         J_\rho &\leq C  \| u\|_{L^\infty(\R^n) } \int_{B_1} \int_{\R^n \setminus B_{\rho -1}} \frac{\dd y \dd x }{(\vert y \vert - 1  )^{n+s}} .
    \end{align*} Since \(\rho>3\), we obtain \begin{align*}
         J_\rho \leq C  \| u\|_{L^\infty(\R^n) } \int_{B_1} \int_{\R^n \setminus B_{\rho -1}} \frac{\dd y \dd x }{\vert y \vert^{n+s}} \leq \frac{ C(\rho -1)^{-s}} s  \| u\|_{L^\infty(\R^n) } \leq \frac{ C\rho ^{-s}} s  \| u\|_{L^\infty(\R^n) }
    \end{align*} with \(C\) depending only on \(n\).

    Thus, combining the above estimates, we have \begin{align*}
        \iint_{\mathcal Q(B_1)} \frac{\vert u(x) - u (y)\vert }{\vert x - y \vert^{n+s}} \dd y \dd x &\leq C \bigg ( \frac {\rho^{1-s} } {1-s} \| \nabla u \|_{L^1(B_{\rho+1})} + \frac{ \rho ^{-s}} s  \| u\|_{L^\infty(\R^n) }\bigg ) 
    \end{align*} for all \(\rho>3\). Choosing \(\rho = \varepsilon^{- \frac 1 s }\), we prove the result. 

    For a general element \(u \in \operatorname{BV}_{\mathrm{loc}}(\R^n) \cap L^\infty (\R^n)\), let \(u_k \in C^\infty_0(\R^n)\), defined as in the proof of Theorem 2 of \S5.2.2 in \cite{evans_measure_1992}. Then \begin{align*}
        \lim_{k \to \infty } \| \nabla u_k \|_{L^1(B_{1+\varepsilon^{-1/s}})} = \vert \nabla u \vert (B_{1+\varepsilon^{-1/s}}) \text{ and } \|u_k\|_{L^\infty(\R^n)} \leq \| u \|_{L^\infty(\R^n)},
    \end{align*} so \begin{align*}
        \iint_{\mathcal Q(B_1)} \frac{\vert u(x) - u (y)\vert }{\vert x - y \vert^{n+s}} \dd y \dd x &\leq \liminf_{k\to +\infty}\iint_{\mathcal Q(B_1)} \frac{\vert u_k(x) - u_k (y)\vert }{\vert x - y \vert^{n+s}} \dd y \dd x\\
        &\leq C \liminf_{k\to +\infty } \bigg (\frac {\varepsilon^{- \frac{1-s} s } } {1-s} \| \nabla u_k \|_{L^1(B_{\rho+1})} + \frac{ \varepsilon} s  \| u_k\|_{L^\infty(\R^n) } \bigg ) \\
        &\leq C \bigg (\frac {\varepsilon^{- \frac{1-s} s } } {1-s}  \vert \nabla u \vert (B_{1+\varepsilon^{-1/s}})  + \frac{\varepsilon } s  \| u\|_{L^\infty(\R^n) } \bigg ). 
    \end{align*}

    Finally, to obtain the case for general \(R>0\), let \(u_R(x)= u(Rx)\). Then \begin{align*}
        \iint_{\mathcal Q(B_R)} \frac{\vert u(x) - u (y)\vert }{\vert x - y \vert^{n+s}} \dd y \dd x &= R^{n-s}\iint_{\mathcal Q(B_1)} \frac{\vert u_R(x) - u_R (y)\vert }{\vert x - y \vert^{n+s}} \dd y \dd x \\
        &\leq C R^{n-s}\bigg (\frac {\varepsilon^{- \frac{1-s} s } } {1-s}  \vert \nabla u_R \vert (B_{1+\varepsilon^{-1/s}})  + \frac{\varepsilon } s  \| u_R\|_{L^\infty(\R^n) } \bigg ) \\
        &= C R^{n-s}\bigg (\frac {\varepsilon^{- \frac{1-s} s }R^{1-n} } {1-s}  \vert \nabla u\vert (B_{ (1+\varepsilon^{-1/s})R})  + \frac{\varepsilon } s  \| u\|_{L^\infty(\R^n) } \bigg ).
    \end{align*}
\end{proof}

\section{The extension problem}

In this section, we record several results regarding the Caffarelli-Silvestre extension, in the context of the theory of nonlocal minimal surfaces, that we require in the proof of~\thref{jG1dQXin}. Throughout this section, we will denote points in \(\R^{n+1}\) with capital letters \(X\), \(Y\), etc and write \(X=(x,x_{n+1})=(x_1,\dots,x_{n+1})\) where \(x=(x_1,\dots,x_n) \in \R^n \) (and analogously \(Y=(y,y_{n+1})\), etc). Moreover, let \begin{align*}
    \R^{n+1}_+ &= \{ X \in \R^{n+1} \text{ s.t. } x_{n+1}>0\}, \\
    \tilde B_R(X) &= \{ Y \in \R^{n+1} \text{ s.t. } \vert Y-X\vert <R\}, \\
    \tilde B_R^+(X) &=   \tilde B_R(X) \cap \R^{n+1}_+,  \\ 
    \tilde B_R &= \tilde B_R(0), \text{ and} \\
    \tilde B_R^+ &= \tilde B_R^+(0) .
\end{align*}

Let \(s\in (0,1)\), \(E\subset \R^n\) be a measurable set, and \begin{align*}
    \tilde \chi _E(x) &= \chi_{\R^n \setminus E}(x) - \chi_E(x)  \qquad x\in \R^n
\end{align*} where  \(\chi_E\) is the characteristic function of \(E\). Now we define \(U_E: \R^{n+1} \to \R \) \begin{align*}
    U_E(X) &= \int_{\R^n} P_{s/2}(X,y) \tilde \chi _E(y) \dd y, \qquad \text{for all } X\in \R^{n+1}_+
\end{align*} where \begin{align}
   P_{s/2}(X,y) =a(n,s) \frac{x_{n+1}^s}{\big ( \vert x-y\vert^2 + x_{n+1}^2 \big )^{\frac{n+s}2} }, \qquad  a(n,s) = \frac{\Gamma \big (\frac {n+s} 2 \big )  }{\pi^{\frac n2} \Gamma \big ( \frac s 2 \big ) }. \label{kzYTNNBU}
\end{align} Since \(\vert  \tilde \chi _E(x) \vert \leq 1 \) and \(\int_{\R^n}P_{s/2}(X,y) \dd y =1 \) for all \(X\in \R^{n+1}_+\), see \cite[Remark 10.2]{MR3916700}, \(U_E\) is well-defined, and, in fact, is a smooth function. Moreover, \(U_E\) satisfies the degenerate PDE \begin{align*}
    \begin{PDE}
\div \big ( x_{n+1}^{1-s} \nabla V \big ) &= 0, &\text{in } \R^{n+1}_+ \\
V &= \tilde \chi_E, &\text{on } \R^n \times \{0\}. 
    \end{PDE} 
\end{align*} where here \(\nabla\) and \(\div\) denotes the gradient and divergence in \(\R^{n+1}\) respectively. This PDE was first studied in the context of the fractional Laplacian in \cite{MR2354493} and in the context of nonlocal minimal surfaces in \cite{MR2675483}. In particular, it was proven that if \(E\) is a minimiser of the \(s\)-perimeter then the function \begin{align}
   \Phi_{E,x}(R) := \frac 1 {R^{n-s}} \int_{\tilde B_R^+((x,0))} y_{n+1}^{1-s} \vert \nabla U_E(Y) \vert^2 \dd Y \label{HvcZ5aGS}
\end{align} is monotone increasing. This function plays an identical role in the theory as the function \( R^{1-n} \operatorname{Per}(E;B_R(x)) \) plays in the classical setting. See also \cite{MR3916700} and references therein for more details. Recently, in \cite{caselli2024fractional}, it was shown that the monotonicity of~\eqref{HvcZ5aGS} also holds for sets that are stationary with respect to the \(s\)-perimeter: 

\begin{prop}[{\cite[Theorem 3.4]{caselli2024fractional}}] \thlabel{uUBAsC54}
Let \(s\in (0,1)\) and \(E\subset \R^n\) be a measurable set that is stationary with respect to the \(s\)-perimeter. Then \(\Phi_{E,x}\) is monotone increasing for all \(x\in \partial^\ast E\).
\end{prop}

Moreover, in \cite{MR2675483} the following ``extension--trace" energy estimate was proven which gives a control on \(\Phi_{E,x}\) in terms of the \(s\)-perimeter of \(E\) in a ball.

\begin{prop}[{\cite[Propostion 7.1(a)]{MR2675483}}] \thlabel{Qu1H6WRf}
    Let \(s\in (0,1)\) and \(E\subset \R^n\) gave locally finite perimeter. Then  \begin{align*}
        \Phi_{E,x}(R) \leq \frac {Cs^{-1}} {R^{n-s}} \operatorname{Per}_s(E; B_{2R}) 
    \end{align*} for all \(R>0\) and \(x\in \partial^\ast E\). The constant \(C>0\) depends only on \(n\).
\end{prop}

The proof of \thref{Qu1H6WRf} is given \cite{MR2675483} (applied to \(\chi_{\R^n \setminus E} - \chi_E\) and given again with more details in \cite[Lemma 3.13]{caselli2024fractional}. We will sketch the proof given \cite[Lemma 3.13]{caselli2024fractional} in our particular case, just to be explicit about how the constant \(C(n)s^{-1/2}\) was obtained.

\begin{proof}
By translating and rescaling, it is sufficient to consider the case \(R=1\) and \(x=0\). We follow the argument given in \cite[Lemma 3.13]{caselli2024fractional} applied to \(u= \chi_{\R^n\setminus E}-\chi_E - \big ( \vert B_2\setminus E\vert - \vert E\cap B_2\vert  \big )\). Note, we have chosen \(u\) such that \(\int_{B_2} u \dd x = 0\). Let \(\zeta \in C^\infty_0(\R^n)\) be such that \(\zeta =1 \) in \(B_{3/2}\), \(\zeta =0 \) in \(\R^n \setminus B_2\), and \(0\leq \zeta \leq 1\). Define \(u_1=\zeta u\), \(u_2= (1-\zeta)u\), and \(U_1\) and \(U_2\) the Caffarelli-Silvestre extensions of \(u_1\) and \(u_2\) respectively. We observe that our choice of normalisation constant implies \begin{align}
    \int_{\R^{n+1}_+} x_{n+1}^{1-s} \vert \nabla U_1\vert^2 \dd X = \frac{s \Gamma \big ( \frac{n+s}2 \big )}{2\pi^{\frac n2} \Gamma \big ( \frac s2 \big ) } [u_1]_{H^{s/2}(\R^n)}^2 \label{ytKVzdar}
\end{align} where \([\cdot ]_{H^{s/2}(\R^n)}\) is given by~\eqref{JS7OQLcc}. Note that the constant in~\eqref{ytKVzdar} is comparable to \(s^2\). Indeed, \(P_{s/2}\) is the same as the one in \cite{MR3916700} (after replacing \(s\) with \(s/2\)), so from \cite[Proposition 10.1]{MR3916700} (see also \cite[Remark 10.5]{MR3916700}), we have via an integration by parts \begin{align*}
    \int_{\R^{n+1}_+} x_{n+1}^{1-s} \vert \nabla U_1\vert^2 \dd X &= - \int_{\R^n} u_1(x) \lim_{t\to 0^+} t^{1-s} \partial_{n+1}U_1(x,t) \dd x \\
    &= \frac{s \Gamma \big ( \frac{n+s}2 \big )}{\pi^{\frac n2} \Gamma \big ( \frac s2 \big ) }  \int_{\R^n}\int_{\R^n} \frac{u_1(x)(u_1(x)-u_1(y)}{\vert x- y \vert^{n+s}} \dd y \dd x \\
    &= \frac{s \Gamma \big ( \frac{n+s}2 \big )}{2\pi^{\frac n2} \Gamma \big ( \frac s2 \big ) }  [u_1]_{H^{s/2}(\R^n)}^2. 
\end{align*} Moreover, in \cite{caselli2024fractional} they apply the fractional Poincaré-Wirtinger inequality \begin{align}
    \| u  \|_{L^2(B_2)}^2 \leq C(n) \int_{B_2}\int_{B_2} \frac{\vert u(x)-u(y)\vert^2}{\vert x-y \vert^{n+s}} \dd y \dd x , \label{rmU2m8rh}
\end{align} see for example\footnote{Note, in \cite[Corollary 2.5]{brasco_characterisation_2021}, they actually prove actually prove \( \| u  \|_{L^2(B_1)} \leq C [u]_{H^{s/2}(\R^n)}\) for \(u\) with \(\int_{B_1}u \dd x =0\). In~\thref{KnSkEdQI}, we briefly outline an extension operator which allows us to obtain~\eqref{rmU2m8rh}. } \cite[Corollary 2.5]{brasco_characterisation_2021}. All together, the argument in \cite{caselli2024fractional} and~\eqref{LDtSA29N} gives \begin{align*}
    \int_{\R^{n+1}_+} x_{n+1}^{1-s}\vert \nabla U_1 \vert^2 \dd X \leq C(n)s^2 \operatorname{Per}_s(E;B_2) \leq C(n) \operatorname{Per}_s(E;B_2).
\end{align*}

Furthermore, rearranging the two inequalities in \cite{caselli2024fractional} at the end of p.41 (using again~\eqref{LDtSA29N}), they prove \begin{align*}
    \int_{\tilde B^+_1} x_{n+1}^{1-s}\vert \nabla U_2\vert^2 \dd X &\leq C(n) \bigg ( \int_{\tilde B^+_1} x_{n+1}^{s-1}\dd X \bigg ) \operatorname{Per}_s(E;B_2)\leq \frac{C(n)} s \operatorname{Per}_s(E;B_2).
\end{align*} Thus, we obtain \begin{align*}
     \int_{\tilde B^+_1} x_{n+1}^{1-s}\vert \nabla U_E \vert^2 \dd X &\leq C \bigg (  \int_{\R^n_+} x_{n+1}^{1-s}\vert \nabla U_1 \vert^2 \dd X+\int_{\tilde B^+_1} x_{n+1}^{1-s}\vert \nabla U_2 \vert^2 \dd X \bigg ) \\
     &\leq C(n) \bigg ( 1 + \frac 1s \bigg ) \operatorname{Per}_s(E;B_2) \\
     &\leq \frac{C(n)} s \operatorname{Per}_s(E;B_2).
\end{align*}
\end{proof}

\begin{remark} \thlabel{KnSkEdQI}
    In \cite[Corollary 2.5]{brasco_characterisation_2021}, they actually prove \( \| u -u_{B_1} \|_{L^p(B_1)} \leq C(n,p) \alpha (1-\alpha)[u]_{W^{\alpha,p}(\R^n)}\) for all \(0<\alpha<1\), \(1\leq p< +\infty\) where \(u_{B_1} = \frac 1 {\omega_n} \int_{B_1}u\dd x \). To obtain~\eqref{rmU2m8rh}, we define \(Eu:\R^n \to \R \) by \begin{align*}
    Eu(x) &= \begin{cases}
u(x), &\text{for } x\in B_1 \\
u (x / \vert x \vert^2 ), &\text{for } x\in \R^n \setminus B_1.
    \end{cases}
\end{align*} Then, using that \( \big \vert \frac{x}{\vert x \vert^2} - \frac{y}{\vert y\vert^2} \big \vert = \frac{\vert x-y\vert}{\vert x \vert \vert y \vert}\) and that the Jacobain of the transformation \(x\to \frac{x}{\vert x \vert^2}\) is \(\frac{1}{\vert x \vert^n}\), we obtain\begin{align*}
    \int_{\R^n \setminus B_1 }\int_{\R^n \setminus B_1 } \frac{\vert Eu(x)-Eu(y)\vert^p}{\vert x - y \vert^{n+\alpha p }} \dd y \dd x 
    &= \int_{ B_1 }\int_{ B_1 } \frac{\vert u(x)-u(y)\vert^p}{\vert x - y \vert^{n+\alpha p }} \vert x \vert^{\alpha p }\vert y \vert^{\alpha p } \dd y \dd x \\ &\leq  \int_{ B_1 }\int_{ B_1 } \frac{\vert u(x)-u(y)\vert^p}{\vert x - y \vert^{n+\alpha p }} \dd y \dd x. 
\end{align*} Similarly, using that \(\big \vert x - \frac{y}{\vert y\vert^2} \big \vert^2 = \frac 1 {\vert y \vert^2} ( \vert x - y \vert^2 +(1-\vert x \vert^2) (1-\vert y \vert^2) \geq \frac{\vert x - y \vert^2}{\vert y \vert^2 }  \) for all \(x,y\in B_1\), we obtain \begin{align*}
    \int_{ B_1 }\int_{\R^n \setminus B_1  } \frac{\vert Eu(x)-Eu(y)\vert^p}{\vert x - y \vert^{n+\alpha p }} \dd y \dd x 
    &\leq  \int_{B_1 }\int_{ B_1 } \frac{\vert u(x)-u(y)\vert^p}{\vert x - y \vert^{n+\alpha p }} \vert y \vert^{\alpha p } \dd y \dd x \\
    &\leq \int_{B_1 }\int_{ B_1 } \frac{\vert u(x)-u(y)\vert^p}{\vert x - y \vert^{n+\alpha p }}  \dd y \dd x. 
\end{align*} Hence, \([Eu]_{W^{\alpha,p}(\R^n)}^p \leq 4 [u]_{W^{\alpha,p}(B_1)}^p\), so applying \cite[Corollary 2.5]{brasco_characterisation_2021} to \(Eu\) gives \begin{align*}
    \| u -u_{B_1} \|_{L^p(B_1)} &= \| Eu -(Eu)_{B_1} \|_{L^p(B_1)} \\
    &\leq C(n,p) \alpha (1-\alpha)[Eu]_{W^{\alpha,p}(\R^n)} \\
    &\leq C(n,p) \alpha (1-\alpha)[u]_{W^{\alpha,p}(B_1)}.
\end{align*}
\end{remark}

Furthermore, we have the following proposition.

\begin{prop} \thlabel{GeFXLo81}
Let \(E\subset \R^n\) be measurable with locally finite perimeter. Then, for all \(x\in \partial^\ast E\), \begin{align}
    \lim_{R\to 0^+} \Phi_{E,x}(R) &= \Phi_{\R^n_+,0}(1)  \label{bpwPOxAA}
\end{align} where \begin{align*}
    \Phi_{\R^n_+,0}(1)  = \frac{2 \pi^{\frac n 2 - 1 } \Gamma \big ( \frac{s+1}2 \big )\Gamma \big ( \frac{1-s}2 \big ) }{\Gamma \big ( \frac s2 \big )\Gamma \big ( \frac{n-s}2+1 \big )}>0.
\end{align*} In particular, \begin{align*}
    \lim_{R\to 0^+} \Phi_{E,x}(R) \geq \frac {Cs} {1-s} 
\end{align*} with \(C\) depending only on \(n\).
\end{prop}

This proposition is the nonlocal analogue of \cite[Equation (15.9) in Corollary 15.8]{maggi_sets_2012} and is essentially proven in \cite[Theorem 9.2]{MR2675483}; however, the reader should be aware that the statement of \cite[Theorem 9.2]{MR2675483} assumes that \(E\) is \(s\)-minimal, but this is assumption is not used in the proof of \thref{GeFXLo81}. Moreover, they do not explicitly say~\eqref{bpwPOxAA} (merely that  \(\lim_{R\to 0^+} \Phi_{E,x}(R) \geq C(n,s)\)); however, if \(E_{x,R} = \frac{E-x}{R}\) then, from the scaling properties of \(\Phi_{E,x}\), \(\Phi_{E,x}(R) = \Phi_{E_{x,R},0}(1)\), which, in conjunction with the convergence properties of \(E_{x,R}\) as \(R\to 0^+\) for \(x\in \partial^\ast E\), see \cite[Theorem 15.5]{maggi_sets_2012}, and \cite[Proposition 9.1]{MR2675483}, immediately implies~\eqref{bpwPOxAA}. The value of \(\Phi_{\R^n_+,0}(1)\) is given the proposition below.

\begin{prop}
    We have that \begin{align*}
        \Phi_{\R^n_+,0}(1) =  \frac{2 \pi^{\frac n 2 - 1 } \Gamma \big ( \frac{s+1}2 \big )\Gamma \big ( \frac{1-s}2 \big ) }{\Gamma \big ( \frac s2 \big )\Gamma \big ( \frac{n-s}2+1 \big )}.
    \end{align*}
\end{prop}

\begin{proof}
First, we claim that \begin{align}
    U_{\R^n_+} (X) &=  -\frac{\Gamma \big ( \frac{s+1}2\big )}{\pi^{\frac12}\Gamma \big ( \frac s2\big )} h \bigg ( \frac{x_n}{x_{n+1}}\bigg ) \label{D9M0UGx2}
\end{align} where \begin{align*}
   h(\tau )&= \int_0^{\tau } \frac {\dd t } { \big ( t^2 +1 \big )^{\frac {1+s} 2} } .
\end{align*} Indeed, if \(X_\ast := (x',-x_n,x_{n+1})\) with \(x'=(x_1,\dots, x_{n-1})\in \R^{n-1}\) then the change of variables \(y \to (y',-y_n)\) in the second integral below gives \begin{align}
    U_{\R^n_+}(X) &= -a(n,s) x_{n+2}^s\bigg ( \int_{\mathbb R^n_+} \frac1{\big ( \vert x - y \vert^2 +x_{n+1}^2 \big )^{ \frac{n+s}2 }} \dd y-\int_{\mathbb R^n_-} \frac1{\big ( \vert x - y \vert^2 +x_{n+1}^2 \big )^{ \frac{n+s}2 }} \dd y \bigg ) \nonumber  \\
    &= -a(n,s) x_{n+2}^s \big ( f(X)-f(X_\ast) \big ) \label{VK6vmGJS}
\end{align} where \begin{align*}
    f(X) := \int_{\mathbb R^n_+} \frac{\dd y}{\big ( \vert x - y \vert^2 +x_{n+1}^2 \big )^{ \frac{n+s}2 }}.
\end{align*} Observe that, via a translation in the first \(n-1\) coordinates, we have \begin{align*}
    f(X) &= \int_0^{+\infty}\int_{\mathbb R^{n-1}} \frac{\dd y'\dd t }{\big ( \vert y '\vert^2 +  ( x_n - t )^2 +x_{n+1}^2 \big )^{ \frac{n+s}2 }}.
\end{align*} Next, the rescaling \(y'\to y'\sqrt{( x_n - t )^2 +x_{n+1}^2}\) in the inner integral gives \begin{align*}
    f(X) &= \bigg ( \int_{\mathbb R^{n-1}} \frac{\dd y' }{\big ( \vert y '\vert^2 + 1 \big )^{ \frac{n+s}2 }}\bigg )   \int_0^{+\infty} \frac{\dd t }{\big ( ( x_n - t )^2 +x_{n+1}^2 \big )^{ \frac{1+s}2 }}
\end{align*} An application of polar coordinates gives \begin{align}
    \int_{\mathbb R^{n-1}} \frac{\dd y' }{\big ( \vert y '\vert^2 + 1 \big )^{ \frac{n+s}2 }} = \frac{\pi^{\frac{n-1}2 \Gamma \big ( \frac{s+1}2\big ) }}{\Gamma \big (  \frac{n+s}2\big ) }, \label{4HylpNJ5}
\end{align} see \cite[Proposition 4.1]{MR3916700} for the explicit computation. Furthermore, the change of variables \(t \to x_{n+1} t \) gives \begin{align}
     \int_0^{+\infty} \frac{\dd t }{\big ( ( x_n - t)^2 +x_{n+1}^2 \big )^{ \frac{1+s}2 }} &=  x_{n+1}\int_0^{+\infty} \frac{\dd t }{\big ( ( x_n - x_{n+1} t )^2 +x_{n+1}^2 \big )^{ \frac{1+s}2 }} \nonumber \\
     &=  x_{n+1}^{-s} \int_0^{+\infty} \frac{\dd t }{\big ( ( x_n/x_{n+1} -  t )^2 +1 \big )^{ \frac{1+s}2 }}.  \label{5yHmdraj}
\end{align} If we define \begin{align*}
   g(\tau) = \int_0^{+\infty} \frac {\dd t } { \big ( (\tau -t  )^2 +1 \big )^{\frac {1+s} 2} }-\int_0^{+\infty} \frac {\dd t } { \big ( (\tau +t  )^2 +1 \big )^{\frac {1+s} 2} }
\end{align*} and \begin{align*}
    \tilde a(n,s) = 2 a(n,s) \frac{\pi^{\frac{n-1}2 \Gamma \big ( \frac{s+1}2\big ) }}{\Gamma \big (  \frac{n+s}2\big ) } = \frac{2\Gamma \big ( \frac{s+1}2\big )}{\pi^{\frac12}\Gamma \big ( \frac s2\big )}
\end{align*} then it follows from~\eqref{VK6vmGJS},~\eqref{4HylpNJ5}, and~\eqref{5yHmdraj} that \begin{align*}
  U_{\R^n_+}(X) &=-\frac12 \tilde a(n,s) \bigg [\int_0^{+\infty} \frac{\dd t }{\big ( ( x_n/x_{n+1} -  t )^2 +1 \big )^{ \frac{1+s}2 }}-\int_0^{+\infty} \frac{\dd t }{\big ( ( x_n/x_{n+1} +  t )^2 +1 \big )^{ \frac{1+s}2 }} \bigg ] \\
  &= -\frac12\tilde a(n,s)  g \bigg ( \frac{x_n}{x_{n+1}} \bigg ) .
\end{align*} Now, \begin{align*}
     g(\tau) &= \int_{-\tau}^{+\infty} \frac {\dd t } { \big ( t^2 +1 \big )^{\frac {1+s} 2} }-\int_{\tau}^{+\infty} \frac {\dd t } { \big ( t^2 +1 \big )^{\frac {1+s} 2} } \\
     &=  \int_{-\tau}^{\tau } \frac {\dd t } { \big ( t^2 +1 \big )^{\frac {1+s} 2} } \\
     &= 2 h(\tau)
\end{align*} which proves~\eqref{D9M0UGx2}. 

From~\eqref{D9M0UGx2}, it follows that\begin{align*}
    \nabla U_{\R^n_+}(X) &= - \tilde a (n,s) \bigg (0,\dots, 0 , \frac 1{x_{n+1}}h '\bigg ( \frac{x_n}{x_{n+1}}\bigg ), - \frac{x_n}{x_{n+1}} h '\bigg ( \frac{x_n}{x_{n+1}}\bigg ) \bigg ) \\
    &=  - \tilde a (n,s) \bigg (0,\dots, 0 , \frac{x_{n+1}^s }{\big ( x_n^2 + x_{n+1}^2 \big )^{\frac {1+s}2 } }, - \frac{x_nx_{n+1}^{s-1} }{\big ( x_n^2 + x_{n+1}^2 \big )^{\frac {1+s}2 } }  \bigg ) , 
\end{align*} so \begin{align*}
    x_{n+1}^{1-s} \vert \nabla U_E\vert^2 &= \tilde a (n,s)^2 \frac{x_{n+1}^{s-1} }{\big ( x_n^2 + x_{n+1}^2 \big )^s}
\end{align*} Using the notation, \(B^{2,+}_1 := \{ (x_n,x_{n+1} ) \in \R^2 \text{ s.t. } x_n^2+x_{n+1}^2<1, x_{n+1}>0 \}\), we have \begin{align*}
    \tilde B_1^+ &= \big \{ X \in \R^{n+1} \text{ s.t. } (x_n,x_{n+1} ) \in B^{2,+}_1, x' \in B^{n-1}_{\sqrt{1-x_n^2-x_{n+1}^2}} \big \},
\end{align*} so we obtain \begin{align*}
    \int_{\tilde B_1^+} x_{n+1}^{1-s} \vert \nabla U_{\R^n_+} \vert^2 \dd X &= \tilde a (n,s)^2 \int_{B^{2,+}_1}\int_{ B^{n-1}_{\sqrt{1-x_n^2-x_{n+1}^2}}} \frac{x_{n+1}^{s-1} }{\big ( x_n^2 + x_{n+1}^2 \big )^s} \dd x'\dd x_n\dd x_{n+1} \\
    &= \tilde a (n,s)^2 \omega_{n-1} \int_{B^{2,+}_1} \frac{x_{n+1}^{s-1}\big (1- x_n^2 - x_{n+1}^2 \big )^{\frac{n-1}2} }{\big ( x_n^2 + x_{n+1}^2 \big )^s} \dd x_n\dd x_{n+1}.
\end{align*} Hence, polar coordinates \((x_n,x_{n+1})=(r\cos \theta, r\sin \theta)\) gives \begin{align*}
      \int_{\tilde B_1^+} x_{n+1}^{1-s} \vert \nabla U_{\R^n_+} \vert^2 \dd X &=\tilde a (n,s)^2 \omega_{n-1} \bigg ( \int_0^1 r^{-s} (1-r^2)^{\frac{n-1} 2 } \dd r \bigg ) \bigg (\int_0^{\pi } (\sin \theta)^{s-1} \dd \theta \bigg ) .
\end{align*} One can show that \begin{align*}
\int_0^1  r^{-s} \big (1-r^2 \big )^{ \frac{n-1}2 }     \dd r =\frac{\Gamma \left(\frac{n+1}{2}\right) \Gamma \left(\frac{1-s}{2}\right)}{2 \Gamma \left(\frac {n-s} 2 + 1 \right)} \text{ and }  \int_0^\pi  (\sin \theta )^{s-1}  \dd \theta  = \frac{\pi^{\frac 12 }  \Gamma \left(\frac{s}{2}\right)}{\Gamma \left(\frac{s+1}{2}\right)}
\end{align*} Thus, \begin{align*}
    \int_{\tilde B_1^+} x_{n+1}^{1-s} \vert \nabla U_{\R^n_+} \vert^2 \dd X &= \tilde a (n,s)^2 \omega_{n-1} \bigg ( \frac{\Gamma \left(\frac{n+1}{2}\right) \Gamma \left(\frac{1-s}{2}\right)}{2 \Gamma \left(\frac {n-s} 2 + 1 \right)} \bigg ) \bigg ( \frac{\pi^{\frac 12 }  \Gamma \left(\frac{s}{2}\right)}{\Gamma \left(\frac{s+1}{2}\right)}\bigg ) \\
    &= \frac{2 \pi^{\frac n 2 - 1 } \Gamma \big ( \frac{s+1}2 \big )\Gamma \big ( \frac{1-s}2 \big ) }{\Gamma \big ( \frac s2 \big )\Gamma \big ( \frac{n-s}2+1 \big )}
\end{align*} as required. 
\end{proof}

\section{Proof of main theorem}

In this section, we give the proof of \thref{jG1dQXin}.

\begin{proof}[Proof of \thref{jG1dQXin}] 
   By translating, we may assume, without loss of generality, that \(x=0\in \partial^\ast E\). Let \(U_E\) be given by~\eqref{kzYTNNBU}. From \thref{uUBAsC54} and \thref{GeFXLo81}, we have that \begin{align*}
        \int_{\tilde B_R^+ } x_{n+1}^{1-s} \vert \nabla U_E \vert^2 \dd X \geq  R^{n-s} \lim_{r\to 0^+} \Phi_{E,0}(r) \geq \frac{CsR^{n-s}}{1-s}
    \end{align*} for all \(R>0\). Moreover, since \begin{align*}
        \vert \tilde \chi_E(x) - \tilde \chi_E(y)\vert^2 =4\vert  \chi_E(x) - \chi_E(y)\vert \qquad \text{for all} x,y\in \R^n
    \end{align*} by~\thref{Qu1H6WRf} and \thref{60sDucQ0}, we obtain \begin{align*}
         \frac{sR^{n-s}}{1-s} &\leq C(n)\iint_{\mathcal Q (B_{2R})} \frac{\vert \tilde \chi_E(x) - \tilde \chi_E(y)\vert^2 }{\vert x - y \vert^{n+s}} \dd y \dd x  \\
        &= \frac{C(n)} s   \operatorname{Per}_s(E; B_{2R} ) \\
        &\leq \frac{C(n)} s   \bigg (  \frac{\varepsilon^{-\frac{1-s}s }  R^{1-s} }{1-s} \operatorname{Per}(E; B_{\delta R})+ \frac{ \varepsilon R^{n-s} } s \big )
    \end{align*} with \(\delta = 1+\varepsilon^{-1/s}\) provided \(\varepsilon < 3^{-1/s}\). Choosing \(\varepsilon = \frac 12 \min \big \{ \frac {s^3} {C(1-s)}, 3^{- \frac 1 s } \big \}\), we find \begin{align*}
         R^{n-1} \leq  \frac{C(n)}{s^2} \varepsilon^{-\frac{1-s}s }  \operatorname{Per}(E; B_{\delta R}),
    \end{align*} which, after rescaling, further gives \begin{align*}
        R^{n-1} \leq  \frac{C(n)}{s^2} \delta^{n-1} \varepsilon^{-\frac{1-s}s } \operatorname{Per}(E; B_{ R}) \qquad \text{for all } R>0.
    \end{align*} Finally, note \(s^{-2}\delta^{n-1} \varepsilon^{-\frac{1-s}s } \leq C(n) s^{-2} 3^{ \frac {n-s} {s^2} } \) which remains bounded away from \(0\) as \(s\to 1^-\). 
\end{proof}